\documentclass{ulect-l}

\usepackage{amssymb}
\usepackage[all]{xy} \CompileMatrices
\usepackage[pagebackref]{hyperref}
\usepackage{slashed}
\usepackage{tikz-cd}

\makeatletter
\@namedef{subjclassname@2020}{\emph{2020 Mathematics Subject Classification}}
\makeatother

\numberwithin{equation}{chapter}

\newtheorem{prop}{Proposition}
\newtheorem{lemma}[prop]{Lemma}

\newtheorem{thm}[prop]{Theorem}
\newtheorem{cor}[prop]{Corollary}
\newtheorem{conj}[prop]{Conjecture}
\numberwithin{prop}{chapter}

\numberwithin{section}{chapter}

\theoremstyle{definition}
\newtheorem{defn}[prop]{Definition}
\newtheorem{ex}[prop]{Example}
\newtheorem{rmk}[prop]{Remark}

\newtheorem{question}[prop]{Question}

\renewcommand{\bar}[1]{\overline{#1}}
\newcommand{\dth}{e_{\theta}}
\newcommand{\gRc}{\mathcal{R}c}
\newcommand{\la}{\langle}
\newcommand{\ra}{\rangle}
\newcommand{\cD}{\mathcal D}
\newcommand{\del}{\partial}
\newcommand{\delb}{\bar{\partial}}\newcommand{\dt}{\frac{\partial}{\partial t}}
\newcommand{\brs}[1]{\left| #1 \right|}

\newcommand{\gG}{\Gamma}
\renewcommand{\gg}{\gamma}
\newcommand{\bga}{\bar{\alpha}}
\newcommand{\bgb}{\bar{\beta}}
\newcommand{\bl}{\bar{l}}

\newcommand{\bs}{\bar{s}}

\newcommand{\bA}{\bar{A}}
\newcommand{\bB}{\bar{B}}
\newcommand{\bC}{\bar{C}}
\newcommand{\bD}{\bar{D}}
\newcommand{\bL}{\bar{L}}
\newcommand{\bn}{\bar{n}}
\newcommand{\bi}{\bar{i}}

\newcommand{\bgU}{\bar{\Upsilon}}

\newcommand{\gD}{\Delta}
\newcommand{\gd}{\delta}
\newcommand{\gs}{\sigma}

\newcommand{\gl}{\lambda}
\newcommand{\gk}{\kappa}
\newcommand{\gL}{\Lambda}

\newcommand{\gw}{\omega}
\newcommand{\ga}{\alpha}
\newcommand{\gb}{\beta}
\renewcommand{\ge}{\epsilon}
\newcommand{\N}{\nabla}
\newcommand{\OO}{\mathcal O}
\newcommand{\FF}{\mathcal F}
\newcommand{\GG}{\mathcal G}
\newcommand{\gU}{\Upsilon}
\newcommand{\WW}{\mathcal W}
\newcommand{\NN}{\mathcal N}

\newcommand{\JJ}{\mathcal J}

\newcommand{\DD}{\mathcal D}
\newcommand{\RR}{\mathcal R}
\newcommand{\KK}{\mathcal K}
\newcommand{\PP}{\mathcal P}
\newcommand{\til}[1]{\widetilde{#1}}

\newcommand{\ohat}[1]{\overset{\circ}{#1}}
\newcommand{\bw}{\bar{w}}
\newcommand{\bj}{\bar{j}}
\newcommand{\bp}{\bar{p}}
\newcommand{\bq}{\bar{q}}
\newcommand{\bk}{\bar{k}}

\newcommand{\TT}{\mathcal T}
\newcommand{\CB}[1]{[#1]_c}

\newcommand{\bz}{\bar{z}}
\newcommand{\bmu}{\bar{\mu}}

\renewcommand{\i}{\sqrt{-1}}

\newcommand{\IP}[1]{\left<#1\right>}

\newcommand{\hook}{\lrcorner}

\DeclareMathOperator{\Sym}{Sym}
\DeclareMathOperator{\End}{End}
\DeclareMathOperator{\spn}{span}
\DeclareMathOperator{\Rc}{Rc}
\DeclareMathOperator{\Rm}{Rm}
\DeclareMathOperator{\inj}{inj}
\DeclareMathOperator{\tr}{tr}
\DeclareMathOperator{\Ker}{Ker}

\DeclareMathOperator{\Id}{Id}
\DeclareMathOperator{\divg}{div}

\DeclareMathOperator{\Vol}{Vol}

\DeclareMathOperator{\grapho}{graph}
\DeclareMathOperator{\Ann}{Ann}
\DeclareMathOperator{\codim}{codim}
\DeclareMathOperator{\Diff}{Diff}
\DeclareMathOperator{\rank}{rank}
\DeclareMathOperator{\CL}{CL}
\DeclareMathOperator{\Spin}{Spin}
\DeclareMathOperator{\Isom}{Isom}
\DeclareMathOperator{\Aut}{Aut}
\DeclareMathOperator{\SU}{SU}
\DeclareMathOperator{\divop}{div}

\begin{document}

\frontmatter

\date{\today}

\title[Generalized Ricci Flow]{Generalized Ricci Flow}
\author{Mario Garcia-Fernandez and Jeffrey Streets}
\address{Universidad Aut\'onoma de Madrid, and Instituto de Ciencias Matem\'aticas, Cantoblanco, 28049 Madrid, Spain}
\email{\href{mailto:mario.garcia@icmat.es}{mario.garcia@icmat.es}}
\address{University of California, Irvine, CA 92617}
\email{\href{mailto:jstreets@uci.edu}{jstreets@uci.edu}}

\subjclass[2020]{53C55, 53D18, 53E20, 53E30}

\maketitle

\tableofcontents

\mainmatter

\chapter{Introduction}

The Ricci flow is a nonlinear parabolic partial differential equation for a Riemannian metric that has yielded new results and insights in topology, Riemannian
geometry, complex geometry, partial differential equations, and mathematical
physics.  The equation was introduced by Hamilton \cite{Hamilton3folds}, who 
used
it to classify compact Riemannian $3$-manifolds with positive Ricci curvature. 
In the ensuing decades Hamilton went on to prove various results and formulate
precise conjectures on the Ricci flow in low dimensions.  The importance of
Ricci flow was then permanently enshrined by Perelman, in his spectacular
resolution of the Poincar\`e Conjecture and Thurston Geometrization Conjecture 
\cite{Perelman1, Perelman3,Perelman2}.

Parallel to the story of Ricci flow in Riemannian geometry is that of
K\"ahler-Ricci flow in complex geometry.  The fundamental result of Cao
\cite{CaoKRF} gave a proof of the classical celebrated
theorems of Aubin-Yau \cite{Aubin, YauCC} and Yau \cite{YauCC} using Ricci flow.  Since then the K\"ahler-Ricci
flow has yielded many further results in complex geometry.  In recent
years the ``Analytic minimal
model program'' of Song-Tian \cite{SongTian} has attracted much attention, aiming at 
fundamental
breakthroughs in complex and algebraic geometry.  A third parallel thread informs the study of Ricci flow, coming from
mathematical physics.  In the thesis of Friedan \cite{Friedan}, it was observed
that the Ricci flow equation arises in the context of renormalization group flow
for nonlinear sigma models, and this connection was strengthened by Perelman, who noted \cite{Perelman1} that this physical point of view loosely
suggests some of his fundamental monotone quantities.

More recently a natural extension of the Ricci flow equation has appeared in the context of \emph{generalized geometry}, a relatively new subject drawing inspiration from Poisson geometry, complex geometry, and mathematical physics.  An early appearance of this subject was the discovery of generalized K\"ahler structure in the seminal work of Gates-Hull-Ro\v{c}ek \cite{GHR} in mathematical physics.  Twenty years later these structures were rederived by Gualtieri \cite{GualtieriThesis} within the general framework of Hitchin's generalized geometry program.  Like generalized geometry itself, the generalized Ricci flow equation also has several different origins coming from considerations in classical geometry, complex geometry, and mathematical physics.  The relationships between these points of view, and the realization that a priori different extensions of the Ricci flow equation are unified by a single equation, the generalized Ricci flow, have only been achieved recently.  Thus,

\vskip 0.1in
\begin{quotation}
\textbf{The primary purpose of this book is to provide an introduction to the
fundamental geometric, algebraic, topological, and analytic aspects of the
generalized Ricci flow equation.}
\end{quotation}
\vskip 0.1in

\noindent Our focus is mostly on foundational results, although we will discuss
some more technical global existence and convergence results in the penultimate chapter.  As the 
mathematical study of the
generalized Ricci flow equation is more nascent, the results described
inevitably fall short of the depth of results attained for the Ricci flow. 
However it seems apparent that such depth of results are waiting to
be attained in the context of generalized geometry, beyond tautologically
recapturing what is known for Ricci flow.  Thus,
\vskip 0.1in
\begin{quotation}
\textbf{The secondary purpose of this book is to formulate questions and
conjectures about the generalized Ricci flow as an invitation to the reader.}
\end{quotation}
\vskip 0.1in
These questions and conjectures will be interspersed througout the text.

\section{Outline}

We begin in Chapter \ref{GRGchapter} by introducing fundamental concepts of generalized Riemannian geometry.  The starting point is to extend the smooth structure of a manifold, in the form of the Lie bracket on the tangent bundle, to new structures on the direct sum of the tangent and cotangent bundle, specifically the Dorfman and Courant brackets.  These brackets come equipped with a symmetry group that includes the natural action of diffeomorphisms, but is enlarged by the action of $B$-fields.  We then introduce generalized Riemannian metrics from the point of view of reducing the structure group of $T \oplus T^*$ to a maximal compact subgroup.  We show the equivalence between a generalized metric and a pair $(g, b)$ consisting of a classical Riemannian metric and skew-symmetric two-form $b$.

Having introduced the fundamental properties of generalized geometry, in Chapter \ref{c:GCC} we introduce and analyze special generalized metrics.  We begin with the theory of generalized connections, leading eventually to a derivation of natural curvature quantities associated to a generalized metric that naturally involve the curvature of the classical Bismut connection.  Given this, we seek to define a class of canonical geometries through a variational principle generalizing the derivation of the classical Einstein equation.  Specifically we generalize the Einstein-Hilbert action using a scalar curvature quantity natural to generalized geometry.  Critical points of this functional satisfy a natural coupling of the classical Einstein equation with the equations for a harmonic three-form, again expressed naturally in terms of the Ricci curvature of the Bismut connection.  
With this in mind, we give the classification of Bismut-flat connections.  This is a special case of a classical result of Cartan-Schouten, and we show that all examples are covered by semisimple Lie groups with bi-invariant metrics and torsion tensor determined by the Lie bracket.  

In Chapter \ref{GRFchapter} we introduce the generalized Ricci flow equation as a tool for constructing canonical generalized geometric structures, specifically generalized Einstein structures.  We first introduce the equation, as expressed in terms of both classical and generalized objects.  We explain the invariance properties of the equation, and introduce natural classes of solutions, illustrated with basic examples.    We discuss self-similar solutions of generalized Ricci flow, called solitons, and derive some of their basic properties.  We end by discussing special properties satisfied by generalized Ricci flow in low dimensions, and special classes of solutions that illustrate connections to further extensions of the Ricci flow equation.

We prove the fundamental analytic properties of generalized Ricci flow in Chapter \ref{LEchapter}.  We first show that the equation is well-posed for arbitrary initial data on compact manifolds.  We then exhibit evolution equations for the Riemannian and Bismut curvature tensors, and use these to establish estimates on derivatives of curvature in the presence of bounds on the Riemannian curvature.  This leads to the basic fact that the Riemannian curvature tensor must blow up at any finite time singularity of the flow.  We end by recording compactness results for solutions to generalized Ricci flow.  In Chapter \ref{energychapter} we show that generalized Ricci flow is the gradient flow for a natural energy functional.  The construction leads furthermore to natural entropy functionals which are monotone along the generalized Ricci flow.  As consequences of this we establish a noncollapsing result for broad classes of solutions to generalized Ricci flow, as well as results on the limiting behavior of certain nonsingular solutions of the flow.

The remainder of the text focuses on the role of generalized Ricci flow in generalized complex geometry.  We start in Chapter \ref{c:GCG}, using generalized geometry to motivate and introduce generalized complex structures, natural objects which unify and extend classical complex structures as well as symplectic structures.  After recording their most basic properties we describe other constructions in complex geometry and show their relationship to generalized geometry.  In particular, we describe pluriclosed structures (also known as strong K\"ahler with torsion structures) using generalized complex geometry.  Lastly we define generalized K\"ahler geometry and discuss its basic properties and local structure.

Having established the fundamental geometric properties of generalized complex structures, in Chapter \ref{c:CMGCG} we discuss canonical metrics in the setting of generalized complex geometry, extending the discussion of Chapter \ref{c:GCC}.  We
recall basic aspects of different Hermitian connections relevant in complex, non-K\"ahler geometry, and express the generalized Einstein equations in this setting in terms of their curvatures.  In the setting of generalized K\"ahler geometry these equations can be reduced to certain scalar PDE using an extension of the classical transgression formula for the first Chern class.  We end the chapter with a discussion of examples of and rigidity results for canonical metrics in this setting.

In Chapter \ref{c:GFCG} we establish fundamental properties of the relationship between generalized Ricci flow and generalized complex geometry.  We begin with the classical point of view, introducing the K\"ahler-Ricci flow, then introduce the pluriclosed flow, an extension of the K\"ahler-Ricci flow to complex, non-K\"ahler geometry.  We show that the pluriclosed flow is gauge-equivalent to generalized Ricci flow, and furthermore preserves generalized K\"ahler structure, yielding the generalized K\"ahler-Ricci flow system.  We describe all of these flows directly in the language of generalized geometry, and describe natural deformation classes of generalized K\"ahler structure which are preserved along the flow.  We prove the short-time existence of solutions to all these equations and formulate conjectures on the sharp existence time.

In the second half of this chapter we prove global existence and convergence results for the generalized Ricci flow in complex geometry.  To begin we exhibit natural reductions of the generalized Ricci flow in the setting of complex geometry, generalizing the reduction of K\"ahler-Ricci flow to a scalar equation modeled on the parabolic complex Monge-Amp\`ere equation.  Next we establish higher regularity of the flow in the presence of uniform parabolicity estimates, generalizing the classic estimates of Evans-Krylov and Calabi/Yau for the complex parabolic Monge-Amp\`ere equation to the system of equations determined by the pluriclosed flow.  We then give geometric settings where these uniform parabolicity estimates can be verified, leading to sharp global existence and convergence results for the flow.  First we prove the theorem of Tian-Zhang on the sharp existence time for K\"ahler-Ricci flow.  In the case of pluriclosed flow we prove global existence on complex manifolds admitting Hermitian metrics of nonpositive bisectional curvature, and prove exponential convergence to a flat, K\"ahler, metric on complex tori.  This result includes a complete description of the generalized K\"ahler-Ricci flow, and yields a classification result for generalized K\"ahler structures on tori.

We finish in Chapter \ref{c:Tdual} by describing the relationship between the generalized Ricci flow and the physical phenomenon of T-duality.  While T-duality first arose in physics, we recall the relatively recent, purely mathematical, formulation of this concept, focusing on the most fundamental topological and geometric aspects.  We then show that T-duality naturally transforms solutions to the generalized Ricci flow.  As corollaries we immediately obtain some global existence and convergence results for the generalized Ricci flow by applying T-duality to Ricci flow solutions whose behavior is known.

\section{On pedagogy}

This book assumes a working familiarity with fundamental concepts of Riemannian
geometry, i.e. smooth manifolds, tangent/cotangent/tensor bundles, differential
forms, Lie groups, Riemannian metrics, connections, and curvature.  We direct 
the reader to
e.g. \cite{LeeRiemann, Leesmooth,Petersen,Warner} for background in these areas.
We also assume some mild familiarity with the language and basic results of 
ordinary and partial differential
equations, and direct the reader to e.g. \cite{Evans} for general background in 
the area.  We also assume familiarity with complex geometry, although we briefly review some of this material at a more basic level within.  We refer to \cite{Chernbook, GriffithsHarris, Kodbook} for further background.

Notably, we will \emph{not} assume familiarity with the Ricci flow equation
itself, and instead will introduce the fundamental concepts of generalized Ricci
flow whole cloth, without an independent discussion of the classical Ricci flow
concepts inevitably carried within.  This choice is made in part to save space,
but also since by now there are many excellent introductory texts on the Ricci
flow e.g. \cite{AndrewsBook, BrendleBook, KnopfBook, CLNBook, MorganTian}.  Having said all of this, we will at times refer back to and
discuss results specific to the Ricci flow when it aids in exposition. 
Moreover, prior familiarity with the Ricci flow equation would of course help in
reading this text, as would reading a Ricci flow text in parallel.  Ultimately 
we aim to make the text self-contained given the stated prerequisites, which 
will of course entail some overlap with existing texts.

Generalized geometry employs many structures (e.g. Dirac structures, Courant
algebroids) which have a rich history and outlook of their own, independent of
the role they play in generalized geometry.  Moreover, generalized geometry as 
a subject by itself appears related to a variety of questions in algebraic 
geometry, Poisson geometry, mathematical physics, etc.  We will make no attempt 
at an
exhaustive exploration of these various avenues, rather focusing on the aspects
which are most relevant to understanding the analysis of generalized
Ricci flow.  One imagines that there are likely fruitful further generalizations
of Ricci flow incorporating yet more structure from generalized geometry.  We 
point out some
possibilities for this as well as some preliminary results in the literature.

\section{Acknowledgements}

The first named author is grateful to his collaborators Luis \'Alvarez C\'onsul, Ra\'ul Gonz\'alez
Molina, Roberto Rubio, Carl Tipler, and Carlos Shahbazi for useful conversations and insight on different aspects of this text. He would like to specially thank Nigel Hitchin and Marco Gualtieri for introducing him into the subject of generalized geometry.  He would like also to thank Gil Cavalcanti, Pedram Hekmati, Pavol \v Severa, Rafael Torres, Fridrich Valach, and Dan Waldram for helpful and inspiring conversations. He acknowledges support from the Spanish MICINN via the ICMAT Severo Ochoa project No. SEV-2015-0554, and under grants No. PID2019-109339GA-C32 and No. MTM2016-81048-P.

The second named author is grateful to Vestislav Apostolov, Matthew Gibson, Joshua Jordan, Man-Chun Lee, Yury Ustinovskiy, and Micah Warren, for collaborations underpinning many results in this text.  Furthermore, his understanding of this subject has benefitted greatly from conversations with Jess Boling, Ryushi Goto, Paul Gauduchon, Marco Gualtieri, Nigel Hitchin, Chris Hull, and Martin Ro\u{c}ek.  He would like to thank David Streets for a careful proofreading.  Finally, he would like to offer greatest thanks to Mark Stern and Gang Tian for their support, encouragement, guidance, and collaboration regarding this subject.  He gratefully acknowledges support from the National Science Foundation via DMS-1454854, DMS-1301864, DMS-1006505, and from the Alfred P. Sloan Foundation.

\chapter{Generalized Riemannian Geometry}\label{GRGchapter}

On the road to understanding Riemannian geometry from the modern point of view, one begins with the notion of a smooth manifold, and derives various objects canonically associated to the smooth structure, such as the tangent and cotangent bundles, and the Lie bracket.  Geometric structure enters by assigning a Riemannian metric, which measures the lengths of vectors in the tangent bundle. This way, one can associate an energy to a smooth curve, thought of as the trajectory of a point particle moving along the manifold, which leads to geodesics and Jacobi fields.

Generalized geometry changes this story in a fundamental way.  Whereas vectors and the tangent bundle have primacy in Riemannian geometry, the starting point of generalized geometry is to put vectors and covectors on equal footing, and to treat sections of $T \oplus T^*$ as the fundamental object of geometry.  This leads to a new bracket structure on such sections, the Dorfman bracket, which involves a delicate combination of the classical operators of differential geometry.  This point of view also inspires the definition of a generalized Riemannian metric, which measures sections of $T \oplus T^*$. Such generalized metrics are equivalent to a choice of a classical Riemannian metric $g$ together with a skew-symmetric two-form $b$. Similar to Riemannian metrics, a generalized metric can be used to define an energy for a surface mapping into the manifold, thought of as the trajectory of a `string' moving along the target space.

\section{Courant algebroids} \label{s:Courant}

The fundamental object underlying generalized Riemannian geometry is a
Courant algebroid. These first arose in the work of Courant \cite{CourantDirac} and Dorfman \cite{Dorfman}, partly in an effort to describe the geometry of equivariant moment map level sets, and was later axiomatized by Liu, Weinstein, and Xu \cite{LWX}. Before giving the general definition of Courant algebroid we begin by describing the main example, with further examples to follow.  The linear algebra of $T \oplus T^*$ plays a central role in the subject, and we begin with some basic definitions.

\begin{defn} Given $V$ a vector space of dimension $n$, define an inner product on $V \oplus V^*$
via
\begin{gather} \label{Diracpairings}
\begin{split}
\IP{ X + \xi, Y + \eta} := \tfrac{1}{2} \left( \xi(Y) + \eta(X) \right).
\end{split}
\end{gather}
The inner product $\IP{,}$ is obviously symmetric and has signature $(n,n)$, and thus determines a copy of $O(n,n)$ consisting of endomorphisms preserving this product.
\end{defn}

Note that for an $n$-dimensional vector space $V$ there is a natural decomposition
\begin{align*}
\Lambda^{2n} (V \oplus V^*) = \Lambda^n V \otimes \Lambda^n V^*.
\end{align*}
As there is the natural determinant pairing between $\Lambda^n V$ and $\Lambda^n V^*$, this induces a canonical identification of $\Lambda^{2n}(V \oplus V^*)$ with $\mathbb R$, and hence the positive real numbers correspond to a canonical orientation.  The Lie group preserving both the symmetric product and this orientation will be isomorphic to $SO(n,n)$.

\begin{defn} Given $M$ a smooth manifold, we denote $T \oplus T^* := TM \oplus T^*M$. The \emph{Dorfman bracket} on sections of $T\oplus T^*$ is defined by
\begin{gather} \label{Dorfmanbracket}
[ X + \xi, Y + \eta] = [X,Y] + L_{X} \eta - i_{Y} d\xi.
\end{gather}
Note that this bracket restricts to the Lie bracket when acting on tangent vectors.  However, some of the fundamental properties of the Lie bracket do not extend to the Dorfman bracket. We first observe that the familiar Leibniz rule still holds for the first-order differential operator $[X + \xi,\cdot]$ acting on sections of $T \oplus T^*$.  Before stating this we define the natural projection
\begin{gather*} %\label{anchor}
\begin{split}
\pi \colon&\ T \oplus T^* \to T, \qquad X + \xi \mapsto X,
\end{split}
\end{gather*}
which will be called the \emph{anchor map}.
\end{defn}

\begin{lemma}\label{l:DorfmanLeibniz} Given $a,b \in \Gamma(T \oplus T^*)$ and $f \in
C^{\infty}(M)$, the Dorfman bracket satisfies
\begin{gather*}%\label{DorfmanLeibniz}
[a, fb] = f [a, b] + \pi(a)f b.
\end{gather*}
\begin{proof} Setting $a = X + \xi, b = Y + \eta$, we directly compute using the properties of Lie derivatives,
\begin{gather*}
\begin{split}
[ X + \xi, f(Y + \eta)] =&\ [X, f Y] + L_{X} f \eta - i_{fY} d\xi \\
=&\ f [X + \xi,Y + \eta] + (X f) Y + (Xf) \eta\\
=&\ f [X + \xi, Y + \eta] + (Xf) (Y + \eta),
\end{split}
\end{gather*}
as required.
\end{proof}
\end{lemma}

The differential operator $[X + \xi,\cdot]$ is furthermore compatible with the inner product $\IP{,}$, in the following sense:

\begin{lemma}\label{l:Dorfmanorthogonal} Given $a,b,c \in \Gamma(T \oplus T^*)$, the Dorfman bracket satisfies
\begin{gather*}
\pi(a)\IP{b,c} = \IP{[a, b],c} + \IP{b,[a,c]}.
\end{gather*}
\begin{proof}
Setting $a = X + \xi, b = Y + \eta, b = Z + \zeta$ and using the identity $i_{[X,Y]} = [L_X, i_Y]$ we compute
\begin{align*}
\IP{[a, b],c} + \IP{b,[a,c]} =&\ \tfrac{1}{2} (i_{[X,Y]}\zeta + i_Z(L_X \eta - i_Y d\xi) + i_{[X,Z]}\eta + i_Y(L_X \zeta - i_Z d\xi))\\
=& \ \tfrac{1}{2} (L_X i_{Y}\zeta + L_X i_{Z}\eta) \\
=& \ \tfrac{1}{2} i_X d (\zeta(Y) + \eta(Z) ), 
\end{align*}
as required.
\end{proof}
\end{lemma}

The main structural property of the Dorfman bracket is the Jacobi identity:

\begin{lemma}\label{l:DorfmanJacobi} Given $a,b,c \in \Gamma(T \oplus T^*)$ one has
\begin{align*}
[a,[b,c]] = [[a,b],c] + [b,[a,c]].
\end{align*}
\begin{proof} Setting $a = X + \xi,\ b = Y + \eta,\ c = Z + \zeta$, one has, using the identities $i_{[X,Y]} = [L_X, i_Y]$ and $L_{[X,Y]} = L_X L_Y - L_Y L_X$,
\begin{align*}
[[a,b],c]\ +&\ [b,[a,c]]\\
=&\ [[X,Y] + L_X \eta - i_Y d \xi,Z + \zeta] + [Y + \eta,[X,Z] + L_X \zeta - i_Z d \xi]\\
=&\ [[X,Y],Z] + L_{[X,Y]} \zeta - i_Z d (L_X \eta - i_Y d \xi)\\
&\ + [Y,[X,Z]] + L_Y (L_X \zeta - i_Z d \xi) - i_{[X,Z]} d \eta\\
=&\ [X,[Y,Z]] + L_X L_Y \zeta - L_X i_Z d \eta - L_Y i_Z d \xi + i_Z d i_Y d \xi\\
=&\ [X,[Y,Z]] + L_X (L_Y \zeta - i_Z d \eta) - i_{[Y,Z]} d \xi\\
=&\ [a,[b,c]],
\end{align*}
as claimed.
\end{proof}
\end{lemma}

Despite these favorable properties, unlike the Lie bracket of vector fields, the Dorfman bracket is not skew-symmetric.  In fact, one can easily show that
\begin{gather}\label{eq:symmetricDorfman}
[a,b] + [b,a] = 2d \IP{a, b}.
\end{gather}
The skew-symmetrization of the Dorfman bracket, called the Courant bracket, will be useful for some calculations.

\begin{defn} The \emph{Courant bracket} on sections of $T\oplus T^*$ is defined by
\begin{gather*}
\CB{a, b} = \tfrac{1}{2}([a,b] - [b,a]).
\end{gather*}
\end{defn}
\noindent Notice that
\begin{align} \label{f:CDrelation}
\CB{a,b} = [a,b] - d \IP{a, b}.
\end{align}
Furthermore, a straightforward calculation leads to the following explicit formula:
\begin{gather*}
\CB{X + \xi, Y + \eta} = [X,Y] + L_{X} \eta - L_{Y} \xi + \tfrac{1}{2}  d \left( \xi(Y) - \eta(X) \right).
\end{gather*}
We summarize the main properties of the Courant bracket in the following lemma.

\begin{lemma}\label{l:Courant} The following hold:
\begin{enumerate}
\item Given $a,b \in \gG(T\oplus T^*)$ and $f \in C^{\infty}(M)$, the Courant bracket satisfies
\begin{gather*}
\CB{ a, fb} = f \CB{a, b} + \pi(a)f b - \IP{a, b} df.
\end{gather*}
\item Given $a,b,c \in \gG(T\oplus T^*)$, one has 
\begin{gather*}
\begin{split}
\CB{a,\CB{b,c}} +&\ \CB{c,\CB{a,b}} + \CB{b,\CB{c,a}}\\
=&\ \tfrac{1}{3} d \left( \IP{\CB{a,b},c} + \IP{\CB{b,c},a} + \IP{\CB{c,a},b} \right)
\end{split}
\end{gather*}
\end{enumerate}

\begin{proof} We set $a = X + \xi,\ b = Y + \eta,\ c = Z + \zeta$. To prove $(1)$, we directly compute as in Lemma \ref{l:DorfmanLeibniz}
\begin{gather*}
\begin{split}
\CB{ X + \xi, f(Y + \eta)} =&\ [X, f Y] + L_{X} f \eta - L_{fY} \xi + \frac{1}{2}d \left( f\xi(Y) - f\eta(X) \right)\\
=&\ f \CB{X + \xi,Y + \eta} + (X f) Y + (Xf) \eta - \xi(Y) df + \tfrac{1}{2} \left(
\xi(Y) - \eta(X) \right) df\\
=&\ f \CB{X + \xi, Y + \eta} + (Xf) (Y + \eta) - \tfrac{1}{2} \left( \xi(Y) +
\eta(X) \right) df\\
=&\ f \CB{X + \xi, Y + \eta} + (Xf) (Y + \eta) - \IP{X + \xi, Y + \eta} df,
\end{split}
\end{gather*}
as required.

As for (2), using Lemma \ref{l:DorfmanLeibniz} and the fact that $[\xi,c] = 0$ for any closed one-form $\xi$, we obtain
\begin{align*}
\CB{\CB{a,b},c} =&\ [\CB{a,b}, c] - d \IP{\CB{a,b},c}\\
=&\ [([a,b] - d \IP{a,b}), c]  - d \IP{\CB{a,b},c}\\
=&\ [[a,b],c] - d \IP{\CB{a,b},c}.
\end{align*}
Using these properties we compute
\begin{align*}
\CB{\CB{a,b},c} &+ \CB{\CB{c,a},b} + \CB{\CB{b,c},a}\\
=&\ \tfrac{1}{4} \left( [[a,b],c] - [c,[a,b]] - [[b,a],c] + [c,[b,a]] + \mbox{cyclic} \right)\\
=&\ \tfrac{1}{4} \left( [a,[b, c]] - [b,[a,c]] - [c,[a,b]] \right.\\
&\ \qquad \left. - [b,[a,c]] + [a,[b,c]] + [c,[b,a]] + \mbox{cyclic} \right)\\
=&\ \tfrac{1}{4} \left( [a,[b,c]] - [b,[a,c]] + \mbox{cyclic} \right)\\
=&\ \tfrac{1}{4} \left( [[a,b], c] + \mbox{cyclic} \right)\\
=&\ \tfrac{1}{4} \left( \CB{\CB{a,b]}c} + d \IP{\CB{a,b},c} + \mbox{cyclic} \right)\\
=&\ \tfrac{1}{4} \left( \CB{\CB{a,b},c} + \CB{\CB{c,a},b} + \CB{\CB{b,c},a} \right.\\
&\ \qquad \left. + d \left( \IP{\CB{a,b},c} + \IP{\CB{b,c},a} + \IP{\CB{c,a},b} \right) \right),
\end{align*}
from which the claim follows.
\end{proof}
\end{lemma}

The structure $(T \oplus T^*,\IP{,},[,],\pi)$ was formalized into a general definition, that of a Courant algebroid, first given in 
\cite{LWX}.

\begin{defn}\label{d:CA} A \emph{Courant algebroid} \footnote{We retain the term Courant algebroid despite the fact that all brackets appearing in this definition are Dorfman brackets, not Courant brackets.} is a vector bundle $E \to 
M$ together 
with a nondegenerate symmetric bilinear form $\IP{,}$ a bracket 
$[,]$ on $\gG(E)$, and a bundle map $\pi : E \to TM$ such that, given $a,b,c \in \Gamma(E)$ and $f \in C^{\infty}(M)$, one has
\begin{enumerate}
\item $[a,[b,c]] = [[a,b],c] + [b,[a,c]]$
\item $\pi[a,b] = [\pi a,\pi b]$
\item $[a,fb] = f [a,b] + \pi(a) f b$
\item $\pi(a)\IP{b,c} =\IP{[a,b],c} + \IP{b,[a,c]}$
\item $[a,b] + [b,a] = \DD\IP{{a,b}}$,
\end{enumerate}
where $\DD : C^{\infty}(M) \to \gG(E)$ denotes the composition of three maps: the exterior differential $d$ acting on functions, the natural map $\pi^* \colon T^* \to E^*$, and the isomorphism $E^* \to E$ provided by the symmetric product.
\end{defn}

In the sequel, we will abuse notation and denote by $\pi^*$ the composition of $\pi^* \colon T^* \to E^*$ with the natural isomorphism $E^* \to E$ provided by the symmetric product. The fact that our explicit structure $(T \oplus T^*,\IP{,},[,],\pi)$ fulfills the conditions of a Courant algebroid follows from Lemma \ref{l:DorfmanLeibniz}, Lemma \ref{l:Dorfmanorthogonal}, Lemma \ref{l:DorfmanJacobi}, equation \eqref{eq:symmetricDorfman}, and the definition of $\pi$. Note that, in this explicit situation, we have that
$$
\DD f = 2df.
$$
We leave as an \textbf{exercise} to check from the previous axioms that the image of $\pi^*$ is isotropic. To finish this section, we unravel Definition \ref{d:CA} for the particular class of Courant algebroids which is of primary interest for the present book.

\begin{defn}\label{d:CAex} 
A Courant algebroid is \emph{exact} if it fits into an exact sequence of vector bundles
\begin{align*}
0 \longrightarrow T^* \overset{\pi^*}{\longrightarrow} E \overset{\pi}{\longrightarrow} T \longrightarrow 0.
\end{align*}
\end{defn}

For exact Courant algebroids notice that, by definition, the image of $\pi^*$ coincides with the kernel of $\pi$. Further consider an isotropic splitting $\sigma \colon T \to E$ of $\pi$. Such a splitting can be constructed by choosing an arbitrary splitting $\sigma_0 \colon T \to E$, and setting
\begin{equation}\label{eq:sigma0sigma}
\sigma = \sigma_0 - \tfrac{1}{2}\pi^* \tau,
\end{equation}
where $\tau \in \Sym^2 T^*$ is defined by $\tau(X,Y) = \IP{\sigma_0X,\sigma_0Y}$.  Using these structures we can give a classification result for exact Courant algebroids, incorporating the first appearance of a closed three-form $H$, which features prominently in generalized geometry.

\begin{prop}\label{p:exactCourant} Given an exact Courant algebroid $E$ with isotropic splitting $\gs$, the map $F \colon T \oplus T^* \to E$ defined by
$$
F(X + \xi) = \sigma X + \tfrac{1}{2}\pi^* \xi
$$
is an isomorphism of orthogonal bundles, for the symmetric product \eqref{Diracpairings}. Via this isomorphism the anchor map is given by $\pi(X + \xi) = X$ and the Dorfman bracket is
\begin{align} \label{f:Hbracket}
[X + \xi, Y + \eta]_H := [X + \xi, Y + \eta] + i_Y i_X H,
\end{align}
where $H \in \Lambda^3T^*$ is a closed three-form defined by
\begin{align} \label{f:Hsigma}
H(X,Y,Z) = 2\IP{[\gs X, \gs Y], \gs Z}.
\end{align}
\begin{proof} We first note that
\begin{align*}
\IP{\sigma X + \tfrac{1}{2} \pi^* \xi, \sigma Y + \tfrac{1}{2} \pi^* \eta} = \tfrac{1}{2}(\xi(\pi \sigma Y) + \eta(\pi \sigma X)) = \tfrac{1}{2}(\xi(Y) + \eta(X)),
\end{align*}
noting that $\IP{\sigma X, \sigma Y} = 0$, which proves that $F$ is an isometry. Next we transport the Courant algebroid structure from $E$ to $T \oplus T^*$ using $F$. It follows from the definition of $F$ that the induced anchor map on $T \oplus T^*$ is $\pi_E(X + \xi) := \pi F(X + \eta) = X$ and that $\mathcal D_E = \pi^* d = 2d$.  The induced bracket, denoted $[,]_E$, is defined via
\begin{gather} \label{f:exactclass10}
\begin{split}
F [X + \xi, Y + \eta]_{E} :=&\ [F(X + \xi), F(Y + \eta)]\\
=&\ [\sigma X, \sigma Y] + \tfrac{1}{2}[\sigma X, \pi^* \eta] + \tfrac{1}{2}[\pi^* \xi, \sigma Y] + \tfrac{1}{4}[\pi^* \xi, \pi^* \eta].
\end{split}
\end{gather}
It is an \textbf{exercise} using the axioms of a Courant algebroid to show that $[\pi^* \xi, \pi^* \eta] = 0$ for any $\xi, \eta$.  For the second term on the right hand side we first observe that
\begin{align*}
\pi[\sigma X, \pi^* \eta] = [\pi \sigma X, \pi \pi^* \eta] = [\pi \sigma X, 0] = 0.
\end{align*}
Thus $[X, \eta]_E \in T^*$, and we can compute, using axiom (5),
\begin{align*}
\IP{[\sigma X,\pi^*\eta],\sigma Z} =&\ \IP{\sigma X, \DD \IP{\pi^*\eta,\sigma Z}} - \IP{\pi^*\eta,[\sigma X,\sigma Z]}\\
=&\ X\IP{\pi^* \eta,Z} - \eta(\pi[\sigma X,\sigma Z])\\
=&\ X\eta(Z) - \eta([X,Z])\\
=&\ d \eta(X,Z) + Z\eta(X).
\end{align*}
We conclude that
\begin{align*}
[X,\eta]_E = L_X \eta.
\end{align*}
The third term of (\ref{f:exactclass10}) is similar, where using axiom (5) we conclude
\begin{align*}
[\xi, Y]_E = - [Y,\xi]_E + 2d\IP{\xi,Y} = - L_Y \xi + d \xi(Y) = - i_Y d\xi.
\end{align*}
Finally, we decompose the first term on the right hand side of (\ref{f:exactclass10}) into tangent and cotangent pieces with respect to $F$. From axiom (2) it follows that $\pi[\sigma X, \sigma Y] = [X,Y]$.  We define a tensor $H \in (T^*)^{\otimes 3}$ via
\begin{align*}
H(X,Y,Z) = 2 \IP{[\sigma X , \sigma Y],\sigma Z},
\end{align*}
and then it follows that
$$
[\sigma X , \sigma Y] = \sigma [X,Y] + \tfrac{1}{2}\pi^* H(X,Y)
$$
It is an \textbf{exercise} to show that $H$ is indeed a tensor, and moreover is totally skew-symmetric.  Putting the above observations back into (\ref{f:exactclass10}) shows that
\begin{align*}
[X + \xi, Y + \eta]_{E} = [X, Y] + L_X \eta - i_Y d\xi + i_Y i_X H,
\end{align*}
as claimed.  The Jacobi identity of axiom (1) for $[X + \xi, Y + \eta]_{E}$ ensures that $d H = 0$.
\end{proof}
\end{prop}

\section{Symmetries of the Dorfman bracket} \label{ss:Courantsymmetries}

A classic fact in smooth manifold theory is that the only automorphisms of the 
tangent bundle preserving the Lie bracket structure are given by the tangent 
maps associated to diffeomorphisms of the underlying manifold.  In this 
section we prove an analogous result for $T \oplus T^*$ equipped with the symmetric product \eqref{Diracpairings} and the Dorfman bracket, 
and use this to obtain \v Severa's classification of exact Courant algebroids \cite{Severa}. In particular the symmetry 
group of the Dorfman bracket is enlarged to include so-called $B$-field 
transformations. These extra symmetries play a central role throughout 
generalized geometry.  With the group of \emph{generalized diffeomorphisms} at hand, we provide a different interpretation of this group by means of a two-dimensional variational problem.  This leads to a first-principles derivation of the Dorfman bracket and a conceptual explanation for some of its most basic features collected in Definition \ref{d:CA}.

Recall that a bundle automorphism of $T \oplus T^*$ is given by a pair $(f,F)$, where $f \in \Diff(M)$ is diffeomorphism of $M$ and $F \colon T \oplus T^* \to T \oplus T^*$ is a vector bundle automorphism, which fit into a commutative diagram
\begin{equation*}
\begin{gathered}
    \xymatrix{
   	T \oplus T^* \ \ar[d] \ar[r]^{F}  & T \oplus T^* \ar[d] & \\
    M \ar[r]^{f} & M &
  }
  \end{gathered}
\end{equation*}
Here, the vertical arrows denote the bundle projections. 

\begin{defn}
Given $M$ a smooth manifold, a bundle automorphism $(f,F)$ of $T \oplus T^*$ is said to be a \emph{Courant automorphism} if $F$ is orthogonal and the natural action of $(f,F)$ on sections of $T \oplus T^*$ preserves the Dorfman bracket.
\end{defn}

Given a diffeomorphism $f$ of $M$, we can define an orthogonal bundle 
automorphism via $(f,\bar{f})$, where
\begin{align}\label{eq:fbar}
\bar{f} := \left(
\begin{matrix}
f_* & 0\\
0 & (f^*)^{-1}
\end{matrix} \right).
\end{align}
It follows from naturality of the various operations in formula \eqref{Dorfmanbracket} that $\bar{f}$ preserves the Dorfman bracket. We next discuss a natural class of Courant automorphisms, given by (closed) $B$-field transformations, which are not associated to diffeomorphisms of the manifold.

\begin{defn} Given $M$ a smooth manifold and $B \in \Lambda^2 T^*$, define the 
associated \emph{B-field transformation} via
\begin{align*}
e^B (X + \xi) =&\ \left( 
\begin{matrix}
1 & 0\\
B & 1
\end{matrix} \right) 
\left( \begin{matrix}
X\\
\xi \end{matrix} \right) = X + \xi + i_X B.
\end{align*}
\end{defn}

\begin{lemma} \label{l:Bfldorth} Given $M$ a smooth manifold and $B \in 
\Lambda^2 T^*$, the map 
$e^B$ is orthogonal with respect to $\IP{,}$.
\begin{proof} Given $X + \xi, Y + \eta \in T \oplus T^*$ we directly compute
\begin{align*}
\IP{ e^B(X + \xi), e^B(Y + \eta)} =&\ \IP{ X + \xi + i_X B, Y + \eta + i_Y B}\\
=&\ \tfrac{1}{2} \left( \eta(X) + i_X i_Y B + \xi(Y) + i_Y i_X B \right)\\
=&\ \tfrac{1}{2} \left( \eta(X) + \xi(Y) \right)\\
=&\ \IP{X + \xi, Y + \eta}.
\end{align*}
\end{proof}
\end{lemma}

\begin{prop}\label{p:bfieldbracket} Given $M$ a smooth manifold and $B \in 
\Lambda^2 T^*$, we have
\begin{align}\label{eq:Bshift}
[e^B(X + \xi), e^B (Y + \eta)] = e^B [X + \xi, Y + \eta] + i_Y i_X dB,
\end{align}
for any $X + \xi, Y + \eta \in \gG(T \oplus T^*)$. Consequently, the map $e^B$ is an automorphism of the Dorfman bracket if and only if $dB = 0$.
\begin{proof} Fix $X + \xi, Y + \eta \in \gG(T \oplus T^*)$ and $B \in 
\gG(\Lambda^2 T^*)$, and then compute using the Cartan formula
\begin{align*}
[e^B(X + \xi), e^B (Y + \eta)] =&\ [X + \xi + i_X B, Y + \eta + i_Y B]\\
=&\ [X + \xi, Y + \eta] + [X,i_Y B] + [i_X B, Y]\\
=&\ [X + \xi, Y + \eta] + L_X i_Y B - i_Y d (i_X B)\\
=&\ [X + \xi, Y + \eta] + L_X i_Y B - i_Y L_X B + i_Y i_X d B\\
=&\ [X + \xi, Y + \eta] + i_{[X,Y]} B + i_Y i_X dB\\
=&\ e^B [X + \xi, Y + \eta] + i_Y i_X dB.
\end{align*}
The result follows.
\end{proof}
\end{prop}

Having exhibited that closed $B$-field transformations are Courant automorphisms of the form $(\Id,e^B)$ in Lemma \ref{l:Bfldorth} and Proposition \ref{p:bfieldbracket}, we now show that these, together with the differentials of diffeomorphisms $(f,\overline{f})$, give all possible Courant automorphisms.

\begin{prop} \label{p:Courantsymmetries} Let $M$ be a smooth manifold and 
suppose $(f,F)$ is a Courant automorphism. Then $F = \overline{f} \circ e^B$, that is, $F$ can be expressed as a composition of a 
diffeomorphism and a closed $B$-field transformation.
\begin{proof} Since the given $f$ is a diffeomorphism, we can define a bundle 
automorphism via $(f,\bar{f})$, as in \eqref{eq:fbar}.  Setting $\Phi = \bar{f}^{-1} \circ F$, we obtain that 
$(\Id, \Phi)$ is also a bundle automorphism.  Fixing sections $a, b \in \gG(T \oplus T^*)$ and a smooth 
function $p$, we obtain, using Lemma \ref{l:DorfmanLeibniz}, that
\begin{align*}
\Phi([p a, b]) =&\ \Phi \left( p[a,b] - \left( (\pi b) p \right) 
a\right)\\
=&\ p \Phi ([a,b]) - \left( (\pi b) p \right) \Phi(a).
\end{align*}
On the other hand, again using Lemma \ref{l:DorfmanLeibniz}, one has
\begin{align*}
[ \Phi(p a), \Phi(b)] = p [\Phi(a), \Phi(b)] - ( \pi \Phi(b) p) 
\Phi(a).
\end{align*}
As $\Phi$ preserves the Courant bracket, the two quantities above are equal, 
and hence $\pi \circ \Phi = \pi$ since $a,b$ and $p$ are arbitrary. On the other hand, the equality
$$
2 \IP{b, \Phi^{-1}dp} = 2 \IP{\Phi(b), dp} = \pi \Phi(b) p = \pi b p = 2 \IP{b, dp}
$$
implies that $\Phi^{-1}$ (and hence $\Phi$) acts as the identity on $T^*$. Together these 
facts imply that
\begin{align*}
\Phi = \left(\begin{matrix} \Id & 0\\
B & \Id
\end{matrix} \right),
\end{align*}
where $B \in (T^*)^{\otimes 2}$.  Since $\Phi$ is also orthogonal, it follows 
that $B \in \Lambda^2 T^*$.  Lastly, since $\Phi$ preserves the Dorfman 
bracket, Proposition \ref{p:bfieldbracket} implies that $B$ is closed.  Thus $F = 
\bar{f} \circ e^B$, as claimed.
\end{proof}
\end{prop}

Notice that the proof of Proposition \ref{p:Courantsymmetries} implies that a Courant automorphism $(f,F)$ is automatically compatible with the anchor map, in the sense that 
$$
\pi \circ F = \overline{f} \circ \pi.
$$
Proposition \ref{p:exactCourant} and Proposition \ref{p:bfieldbracket} make clear an elementary way to construct new Courant algebroids, by `twisting' the standard Dorfman bracket by a three-form $H$. This is key to \v Severa's classification of exact Courant algebroids in Theorem \ref{p:gendiff} below.

\begin{defn}  Given $M$ a smooth manifold and $H \in \Lambda^3T^*$, we define the $H$-twisted Dorfman bracket on sections of $T \oplus T^*$ via
\begin{align*}
[X + \xi, Y + \eta]_H := [X + \xi, Y + \eta] + i_Y i_X H.
\end{align*}
\end{defn}

\begin{prop} \label{p:twistedCAprop} Given $M$ a smooth manifold and $H \in \Lambda^3 T^*$ such that $dH = 0$, the triple $(T\oplus T^*, \IP{,}, [,]_H,\pi)$ defines a Courant algebroid.
\begin{proof} Most of the axioms in the definition of Courant algebroid are routine to check.  The main difficulty is to verify axiom (1).  Since $H$ is closed and the computation is local, we can assume $H = dB$.  We note that we can rephrase equation \eqref{eq:Bshift} as
\begin{align*}
e^{-B} [e^B(X + \xi), e^B(Y + \eta)] =&\ [X + \xi, Y + \eta] + e^{-B} i_Y i_X dB\\
=&\ [X + \xi, Y + \eta] + i_Y i_X dB\\
=&\ [X + \xi, Y + \eta]_H.
\end{align*}
Using this, Lemma \ref{l:DorfmanJacobi}, and Lemma \ref{l:Bfldorth} we compute for any $a,b,c \in \Gamma(T \oplus T^*)$
\begin{align*}
[[a,b]_H&, c]_H + [[c,a]_H, b]_H + [[b,c]_H, a]_H\\
=&\ e^{-B} [ e^B [a,b]_H, e^{B} c] + e^{-B} [ e^B [c,a]_H, e^{B} b] + e^{-B} [ e^B [b,c]_H, e^{B} a]\\
=&\ e^{-B} [ [ e^B a, e^B b], e^B c] + e^{-B} [ [e^B c, e^B a], e^{B} b] + e^{-B} [ [e^B b, e^B c], e^{B} a] = 0.
\end{align*}
The proposition follows.
\end{proof}
\end{prop}

Our next goal is to provide a classification of exact Courant algebroids due to \v Severa \cite{Severa}, which follows from Proposition \ref{p:exactCourant} and Proposition \ref{p:twistedCAprop}. We first give the following abstract definition.

\begin{defn}  
Given $M$ a smooth manifold and $E$ and $E'$ Courant algebroids over $M$, an \emph{isomorphism} between $E$ and $E'$ is a pair $(f,F)$, where $f$ is a diffeomorphism of $M$ and $F \colon E \to E'$ is an orthogonal Courant automorphism, which fit into a commutative diagram
\begin{equation*}
\begin{gathered}
    \xymatrix{
   	E \ \ar[d] \ar[r]^{F}  & E' \ar[d] & \\
    M \ar[r]^{f} & M &
  }
  \end{gathered}
\end{equation*}
We say that $(f,F)$ is \emph{small} if $f$ lies in the identity component of $\Diff(M)$. We say that $E$ and $E'$ are in the same \emph{small isomorphism class} if $E$ and $E'$ are isomorphic via a small isomorphism.
\end{defn}

\begin{thm} \label{p:Severaclass} Given $M$ a smooth manifold, the small isomorphism classes of exact Courant algebroids on $M$ are in one-to-one correspondence with the cohomology group $H^3(M,\mathbb{R})$.
\begin{proof} 
Let $E$ be an exact Courant algebroid over $M$. By Proposition \ref{p:exactCourant} we can choose an isotropic splitting $\sigma$ of $E$, such that $E$ is isomorphic via a small isomorphism to $(T\oplus T^*, \IP{,}, [,]_{H_\sigma},\pi)$, for $H_\sigma$ the closed three-form on $M$ given by \eqref{f:Hsigma}. To the small isomorphism class of $E$ we want to associate the \emph{\v Severa class}
$$
[H_\sigma] \in H^3(M,\mathbb{R}).
$$
Let us check that this is well-defined. First, if $\sigma'$ is a different choice of isotropic splitting, via the orthogonal automorphism $E \cong T \oplus T^*$ induced by $\sigma$ we can identify
$$
\sigma'(X) = e^B(X)
$$ 
for some $B \in \Lambda^2 T^*$. Therefore, from \eqref{eq:Bshift}
$$
H_{\sigma'} = H_\sigma + dB,
$$
and the class $[H_\sigma]$ is independent of the choice of isotropic splitting.

Let $E'$ be another exact Courant algebroid over $M$, that we can identify with $(T\oplus T^*, \IP{,}, [,]_{H'},\pi)$ for a closed three-form $H'$ on $M$. If $E'$ is isomorphic to $E$ there exists $(f,F) \colon E \to E'$ which exchanges the brackets on $E$ and $E'$. Without loss of generality, we identify $E$ with $(T\oplus T^*, \IP{,}, [,]_{H},\pi)$, where we set $H = H_\sigma$. We denote $\Phi = \overline{f}^{-1} \circ F$, and arguing as in the proof of Proposition \ref{p:Courantsymmetries} we have that $\pi \circ \Phi  = \pi$ and that $\Phi$ acts as the identity on $T^*$. Thus, $\Phi = (\Id,e^B)$ for some $B \in \Lambda^2 T^*$. Since $(f,F)$ exchanges the brackets we obtain
$$
e^B [a,b]_{H} = \overline{f}^{-1} [ \overline{f} e^B a , \overline{f} e^B b]_{H'}.
$$
Combined with
$$
\overline{f}^{-1}[ \overline{f} a , \overline{f} b]_{H'} = [a,b]_{f^*H'} 
$$
for any $a,b \in \Gamma(T \oplus T^*)$, we are led to
\begin{equation}\label{eq:HtoH}
f^* H ' = H - dB.
\end{equation}
If $E'$ is in the same small isomorphism class as $E$, then $f$ is in the identity component of $\Diff(M)$ and by definition there exists a one-parameter group of diffeomorphisms $f_t$ generated by $X \in \Gamma(T)$ such that $f = f_1$. Thus, using that $H$ is closed
$$
f^*H - H = \int_0^1 \frac{d}{dt}f_t^* H dt = \int_0^1 f_t^* (d i_X H) dt = d b,
$$
where $b = \int_0^1 \frac{d}{dt}f_t^* (i_X H) dt \in \Lambda^2 T^*$, and we conclude that
$$
H' = H - d B',
$$
for $B' = f_*(B + b)$, which implies $[H'] = [H]$.

By Proposition \ref{p:twistedCAprop} the map that associates to any small isomorphism class of exact Courant algebroids its \v Severa class is surjective. Finally, if $E$ and $E'$ yield the same cohomology class, taking isotropic splittings as before we obtain $H = H' + dB$ and therefore by Proposition \ref{p:bfieldbracket} we have that $E$ and $E'$ are isomorphic via the $B$-field transformation $e^B$, and consequently their small isomorphism classes coincide.

\end{proof}
\end{thm}

As an immediate consequence of the proof of the previous theorem we obtain the following complete classification of exact Courant algebroids on a smooth manifold.

\begin{cor} \label{p:BigSeveraclass} Given $M$ a smooth manifold, denote by $\Diff_0(M)$ the component of the identity in $\Diff(M)$ and set $\Gamma_M = \Diff(M)/\Diff_0(M)$ the corresponding mapping class group. Then, the isomorphism classes of exact Courant algebroids on $M$ are in one-to-one correspondence with the quotient $H^3(M,\mathbb{R})/\Gamma_M$.
\begin{proof} 
By the proof of Theorem \ref{p:Severaclass}, the isomorphism classes of exact Courant algebroids on $M$ are in one-to-one correspondence with $H^3(M,\mathbb{R})/\Diff(M)$. Since the action of $\Diff_0(M)$ on $H^3(M,\mathbb{R})$ is trivial, the proof follows.
\end{proof}
\end{cor}

Another direct consequence of the proof of Theorem \ref{p:Severaclass} is an explicit characterization of the group of automorphisms $(f,F) \colon E \to E$ of an exact Courant algebroid $E$ over a smooth manifold $M$. By a choice of an isotropic splitting of $E$, we can identify $E$ with the twisted Courant algebroid $(T\oplus T^*, \IP{,}, [,]_H,\pi)$ in Proposition \ref{p:twistedCAprop}. To state the result, we abuse notation and denote a pair $(f,F)$ simply by $F$.

\begin{prop} \label{p:gendiff} 
Given $M$ a smooth manifold and $H \in \Lambda^3T^*$ a closed three-form, the group of automorphisms $F \colon E \to E$ of the twisted Courant algebroid $(T\oplus T^*, \IP{,}, [,]_H,\pi)$ is given by
\begin{align*}
\operatorname{Aut}(E) = \{ \overline{f} \circ e^B : f \in \Diff(M), B \in \Lambda^2 T^* \; \textrm{such that} \;  f^*H = H -dB \},
\end{align*}
together with the product given, for $F = \overline{f} \circ e^B$ and $F'= \overline{f'} \circ e^{B'}$, by
$$
F \circ F' = \overline{f \circ f'} \circ e^{B' + f'^* B}.
$$
\begin{proof} 
The characterization of the group follows from \eqref{eq:HtoH} and the proof of Theorem \ref{p:Severaclass}. Following the notation in the statement, the composition law can be obtained from
$$
H = f_*(f'_*(H-dB') - dB).
$$
\end{proof}
\end{prop}

Using the previous result, we can also give an explicit description of the Lie algebra of infinitesimal automorphisms of an exact Courant algebroid $E$.

\begin{cor} \label{c:Liegendiff} 
Given $M$ a smooth manifold and $H \in \Lambda^3T^*$ a closed three-form, the Lie algebra of the group of automorphisms of the twisted Courant algebroid $(T\oplus T^*, \IP{,}, [,]_H,\pi)$ is given by
\begin{align*}
\operatorname{Lie} \operatorname{Aut}(E) = \{ X + B \ |\ X \in T,\ B \in \Lambda^2 T^*,\ L_X H = -dB \},
\end{align*}
together with the Lie bracket given by
\begin{equation}\label{eq:LiebracketLieAut}
[X + B,X'+ B'] = [X,X'] + L_X B' - L_{X'} B.
\end{equation}
\begin{proof}
Let $F_t = \overline{f}_t \circ e^{B_t}$ be a one-parameter family in $\operatorname{Aut}(E)$ with $F_0 = \Id_E$. Taking derivatives in $ f_t^*H = H -dB_t$ at $t = 0$, it follows that 
\begin{equation*}
X := \Bigg(\dt f_t\Bigg)_{|t=0},  \qquad B = \Bigg(\dt B_t\Bigg)_{|t=0}
\end{equation*}
satisfies $L_X H = -dB$. Conversely, given $X + B \in T \oplus \Lambda^2 T^*$ satisfying $L_X H = -dB$, we define 
$$
F_t = \overline{f}_t \circ e^{\overline{B}_t} \in \Aut(E)
$$
where $f_t$ is the one-parameter family of diffeomorphisms generated by $X$ and 
$$
\overline{B}_t = \int_0^t f_s^*B_s ds.
$$
Notice that
$$
d\overline{B}_t = \int_0^t f_s^*dB_s ds =  - \int_0^t f_s^*L_{X_s}H ds = H - f_t^*H
$$
and hence $F_t$ is well-defined.

To calculate the Lie bracket, by Proposition \ref{p:gendiff} we have
$$
F^{-1} \circ F_t \circ F = \overline{f^{-1} \circ f_t \circ f} \circ e^{B + f_*B_t - (f^{-1}f_t f)^* B },
$$
and hence the adjoint action is
$$
F^{-1} (X' + B') = f^*X' + f^*B' - L_{f^*X'}B.
$$
Setting $F = F_t$ and taking derivatives in this expression we obtain \eqref{eq:LiebracketLieAut}. Observe that we use the convention of right invariant vector fields for the Lie bracket.
\end{proof}
\end{cor}

We next provide a different view on the group $\Aut(E)$ in Proposition \ref{p:gendiff}. As we will see, this is an infinite-dimensional Lie group which arises naturally from a $2$-dimensional variational problem \cite{Severa}. Here, we will avoid entering into any details about the theory of infinite-dimensional manifolds and Lie groups, but rather focus on formal aspects of the construction. Let $M$ be a smooth manifold. Let $\Sigma$ be a smooth compact surface (without boundary, say). To be consistent with our notation, we will denote the tangent and cotangent bundles of $M$ by $T$ and $T^*$, while those of $\Sigma$ will be denoted $T \Sigma$ and $T^*\Sigma$. Consider the infinite-dimensional space of smooth maps from $\Sigma$ to $M$, denoted $C^\infty(\Sigma,M)$.  Given a closed three-form $H \in \Lambda^3 T^*$ on $M$ representing an integral cohomology class $[H] \in H^3(M,\mathbb{Z})$, we can define a functional
\begin{gather}\label{eq:SH}
\begin{split}
S_H &\ \colon C^\infty(\Sigma,M) \to \mathbb{R}/\mathbb{Z}, \qquad \varphi \mapsto \int_{Y} \overline \varphi^*H,
\end{split}
\end{gather}
where $Y$ is a three-manifold with boundary $\Sigma$ and $\overline \varphi \colon Y \to M$ is any smooth extension of $\varphi$ to $Y$. Since $H$ is assumed to have integral periods over the integer homology $H_3(M,\mathbb{Z})$ of $M$, this functional is well-defined. Notice also that when $H$ is exact, by Stokes' Theorem the functional reduces to
$$
S_{dB}(\varphi) = \int_\Sigma \varphi^*B.
$$
This functional is known in the mathematical physics literature as the \emph{Wess-Zumino term} in the action of a two-dimensional $\Sigma$-model (see e.g. \cite{Gawedzki}), and lies at the core of the relation between generalized geometry and string theory. 

Using that $H$ is closed and applying Stokes' Theorem again, the variation of $S_H$ is given by
$$
\delta S_H(X) = \int_Y \overline \varphi^* L_{\overline X} H = \int_\Sigma \varphi^* i_X H,
$$
where $X \in \Gamma(\varphi^*T)$ is identified with a vector in the tangent space of $C^\infty(\Sigma,M)$ at $\varphi$, and $\overline X \in \Gamma(\overline \varphi^*T)$ denotes any extension of $X$ along $\overline \varphi$. Observe that the formula for $\delta S_H$ makes sense even if $H$ does not have integral periods. Solutions of the variational problem, that is, critical points of $S_H$, are given by maps $\varphi$ such that for all  $X \in \Gamma(\varphi^*T)$ there exists a one-form $\xi \in \Gamma(T^*\Sigma)$ such that
$$
\varphi^*(i_X H) = d \xi.
$$
In particular, if $X + \xi \in \Gamma(T \oplus T^*)$ is a global section satisfying
\begin{equation}\label{eq:iXHdxi}
i_X H = d \xi,
\end{equation}
by pulling-back to $\Sigma$ via $\varphi$ we obtain tangent directions on the space of smooth maps along which $S_H$ is constant.

A natural question is whether these `flat directions' of $S_H$ arise from some symmetries of the functional. As is customary in variational problems, we will try to provide an answer by looking at the space of \emph{Lagrangian densities}, given in this case by the space of closed three-forms on $M$
$$
\Omega^3_{cl}(M) := \{ H \in \Lambda^3 T^* : dH = 0\}.
$$
Note the group of diffeomorphisms $\Diff(M)$ acts on $C^\infty(\Sigma,M)$ via $\varphi \mapsto f \circ \varphi$, for $f \in \Diff(M)$ and we have
$$
S_H(f \circ \varphi) = S_{f^*H}(\varphi).
$$
This suggests a left action of $\Diff(M)$ on $\Omega^3_{cl}(M)$ given by push-forward. 
Consider now the semi-direct product of $\Diff(M)$ by the space of two forms on $M$
$$
\GG = \Diff(M) \ltimes \Gamma(\Lambda^2 T^*),
$$
with group structure
$$
(f,B) \cdot (f',B') = (f \circ f', B' + f'^*B).
$$
Then, the left action of $\Diff(M)$ on $\Omega^3_{cl}(M)$ extends to a $\GG$-action defined by
$$
(f,B) \cdot H = f_*(H - dB).
$$
In terms of the corresponding functionals, setting $H' = f_*(H - dB)$ we have
$$
S_{H'}(\varphi) = S_{H}(f^{-1}\circ \varphi) - \int_\Sigma \varphi^*(f_*B).
$$
Notice that our formula for the $\GG$-action reproduces the change of a closed three-form under an isomorphism of Courant algebroids in \eqref{eq:HtoH} and, in fact, we have the following.

\begin{prop} \label{p:Autrev} Given $M$ a smooth manifold and a closed three-form $H$, the isotropy group $\GG_H \leq \GG$ of $H$ is isomorphic to the group of automorphisms of the twisted Courant algebroid $(T\oplus T^*, \IP{,}, [,]_H,\pi)$.
\begin{proof} 
The proof is straightforward from Proposition \ref{p:gendiff} and the definition of the $\GG$-action on $\Omega^3_{cl}$.
\end{proof}
\end{prop}

By the previous result, a Lagrangian density for our variational problem, given by a closed three-form $H$, determines a geometric gadget, the exact Courant algebroid with twisted bracket, such that symmetries of $S_H$ are realized as automorphisms of $(T\oplus T^*, \IP{,}, [,]_H,\pi)$. In order to interpret condition \eqref{eq:iXHdxi} in the present setup, we need to regard elements in $\Gamma(T \oplus T^*)$ as symmetries. For this, it is better to think in terms of the Lie algebra
$$
\operatorname{Lie} \GG = \Gamma (T \oplus \Lambda^2 T^*)
$$
%Notice that the (right) adjoint action of $\GG$ is given by
%$$
%f^*X + f^*(B - L_X B'),
%$$
%for $X + B \in \operatorname{Lie} \GG$ and $(f,B') \in \GG$, and therefore the 
with Lie bracket (see Corollary \ref{c:Liegendiff})
$$
[X_1 + B_1, X_2 + B_2] = [X_1,X_2] + L_{X_1}B_2 - L_{X_2}B_1.
$$
Given $H \in \Omega^3_{cl}$, the Lie algebra of its isotropy group is
$$
\operatorname{Lie} \GG_H =  \{X + B : d(i_X H + B) = 0\}.
$$
\begin{prop} \label{p:NormalLie} Given $M$ a smooth manifold and a closed three-form $H$, the following defines a normal Lie subalgebra of $\operatorname{Lie} \GG_H$:
$$
\operatorname{Lie} \GG_H^{0} =  \{X + d\xi - i_X H : \xi \in T^*\}.
$$
\begin{proof} 
Using the identity $i_{[X,Y]} = [L_X, i_Y]$ we calculate
\begin{gather*}
\begin{split}
[X + d\xi - i_X H,Y+B] = &\ [X,Y] + L_X B - L_Y(d\xi - i_X H) \\
=&\ [X,Y] + d i_XB + i_X d(B + i_Y H) - i_{[X,Y]}H - d(L_Y\xi) \\
=&\ [X,Y] + d (i_XB - L_Y\xi)  - i_{[X,Y]}H
\end{split}
\end{gather*}
for any $Y + B \in \operatorname{Lie} \GG_H$ and $X + d\xi - i_X H \in \operatorname{Lie} \GG_H^{0}$, as required.
\end{proof}
\end{prop}

To finish, we note that the normal Lie subalgebra $\operatorname{Lie} \GG_H^{0}\ \triangleleft\ \operatorname{Lie} \GG_H$ is parametrized by elements in $\Gamma(T \oplus T^*)$. This vector space provides a natural representation for $\GG$, via the action
$$
(f,B)\cdot (X + \xi) = f_*(X + \xi + i_X B)
$$
and there is a (right) $\GG_H$-equivariant map
\begin{equation}\label{eq:Psimap}
\Psi \colon \Gamma(T \oplus T^*) \to \operatorname{Lie} \GG_H^{0} \colon X + \xi \mapsto X + d\xi - i_X H.
\end{equation}
We have now that the twisted Dorfman bracket is recovered via the infinitesimal action
$$
\Psi(X + \xi) \cdot (Y + \eta) = [X,Y] + L_X \eta - i_X d\xi + i_Yi_X H,
$$
and the Jacobi identity in Lemma \ref{l:DorfmanJacobi} is a direct consequence of the $\GG_H$-equivariance of the map $\Psi$. The map $\Psi$ determines special elements $X + \xi \in \Gamma(T \oplus T^*)$ such that $\Psi(X + \xi) = X$, that is, $i_X H = d \xi$. The geometric content of this condition is that the classical infinitesimal action of the vector field $X$ on $T \oplus T^*$ via the Lie derivative is realized through the twisted Dorfman bracket via the differential operator $[X + \xi,\cdot]_H$. Going back to our original variational problem, elements $X + \xi \in \Gamma(T \oplus T^*)$ such that $i_X H = d \xi$ provide tangent directions in the space of smooth maps $C^\infty(\Sigma,M)$ along which $S_H$ is constant. The condition $i_X H = d \xi$ appears also naturally in the context of reduction of Courant algebroids by the action of a Lie group in \cite{BursztynCavalcantiGualtieri}.

\section{Generalized metrics} \label{s:genmetric}

Recall the definition of a classical Riemannian metric, which is a smooth 
choice of positive definite inner product on the tangent spaces of a smooth 
manifold.  The geometric kernel of this fundamental object is clear, as it 
allows one to assign the size of a vector tangent to a manifold, and hence by 
integration along a curve one obtains concepts of length, and eventually 
notions of differentiation, curvature, etc.  The forthcoming definition of a 
generalized metric plays a similar role, when we substitute curves--regarded as the trajectories of point particles--by maps from a surface into the manifold--regarded as the trajectories of strings.

To start, we introduce generalized metrics seen through the lens of structure groups.  In particular, recall that a 
classical Riemannian metric is equivalent to a reduction of the structure group 
of the tangent bundle to $O(n)$.  If we now take the bundle $T \oplus T^*$ with its neutral inner product instead of the tangent bundle, then reductions of its structure group should render interesting geometries.  As discussed in \S \ref{s:Courant}, the neutral inner product 
already reduces the structure group to $O(n,n)$, and as it turns out a further reduction to $O(n) \times O(n)$ will yield the right notion of ``generalized metric''.  Unwinding this reduction in terms of tensor fields (see Lemma \ref{l:Gmetricreduction}) 
finally yields the fundamental notion.

\subsection{Linear algebra}\label{subsec:LinearGmetrics}

We shall first consider the linear algebra of generalized metrics. Let $V$ and $W$ be real vector spaces which fit into a short exact sequence
\begin{equation}
\label{eq:linearsurjection}
\begin{gathered}
\xymatrix{0 \ar[r] & V^* \ar[r] &  W \ar[r]^-{\pi} & V \ar[r] & 0.}
\end{gathered}
\end{equation}
We assume that $W$ is endowed with a metric $\langle,\rangle$ of split signature $(n,n)$ such that the image of $\pi^* : V^* \to W^*$ is isotropic. Furthermore, we assume that the first arrow is given by the composition of $\frac{1}{2}\pi^*\colon V^* \to W^*$ with the isomorphism $W \cong W^*$ given by the metric. We denote by $O(n,n)$ the group of orthogonal transformations of $W$. Of course, \eqref{eq:linearsurjection} corresponds to the exact sequence in Definition \ref{d:CAex} for a fixed point on the base manifold.

\begin{defn}\label{d:GGdefnlinear} A \emph{generalized metric} on $W$ is an endomorphism $\GG$ of $W$ 
satisfying
\begin{enumerate}
 \item $\IP{\GG a, \GG b} = \IP{a,b}$,
 \item $\IP{\GG a, b } = \IP{a, \GG b}$,
 \item The bilinear pairing $\IP{\GG a, b}$ is symmetric and 
positive definite.
\end{enumerate}
As a basic consequence of the definitions, it follows that $\GG^2 = \Id$.
\end{defn}

\begin{lemma} \label{l:Gmetricreduction0} A generalized metric on $W$ is equivalent to either
\begin{enumerate}
 \item a choice of maximal compact subgroup $O(n) \times O(n) \leq O(n,n)$,
 \item an orthogonal decomposition $W = V_+ \oplus V_-$ such that the restriction of $\IP{,}$ to $V_+$ is positive definite.
 \end{enumerate}
\begin{proof} A choice of maximal compact subgroup $O(n) \times O(n)$ inside $O(n,n)$ requires a choice of plane on which the neutral inner product is positive definite.  Call such a choice $V_+$, and let $V_-$ denote the neutral orthogonal complement of $V_+$, and note that the neutral inner product is negative definite on $V_-$. Thus (1) implies (2). Given the orthogonal decomposition $W = V_+ \oplus V_-$ as in (2), we can define a positive definite inner product on $W$ via
\begin{align*}
G(a,b) = \IP{a_+,b_+}_{V_+} - \IP{a_-, b_-}_{V_-},
\end{align*}
where $a = a_+ + a_- \in V_+ \oplus V_-$ and similarly for $b$. Using the neutral inner product we can identify $W$ with its dual, and so identify $G$ with an endomorphism $\GG$ of $W$, which is easily seen to satisfy the axioms of a generalized metric.  

Conversely, given a generalized metric, since $\GG$ is real and $\GG^2 = \Id$ it follows that the only possible eigenvalues for $\GG$ are $\pm 1$, with eigenspaces $V_{\pm}$.  It follows that $V_+$ is a maximal plane on which the neutral inner product is positive definite, yielding a maximal compact $O(n) \times O(n)$ given by orthogonal transformations which commute with $\GG$. Note that of course the two constructions are inverses, with the spaces $V_{\pm}$ playing the same role in each.
\end{proof}
\end{lemma}

A generalized metric determines a series of useful isomorphisms and 
projections which we will employ in the sequel, and which clarify the 
relationship between the generalized metric and the classical data to be 
determined below.

\begin{defn} \label{d:metricprojections} Given $\GG$ a generalized metric on $W$, define projection maps $\pi_{\pm} : W \to V_{\pm}$ via
\begin{align*}
\pi_{\pm} a :=&\ \tfrac{1}{2} \left( a \pm \GG a \right).
\end{align*}
\end{defn}

Since the pairing $\IP{,}$ vanishes on $\pi^* V^*$ it follows that $V_\pm 
\cap V^* = \{0\}$, and hence the 
map $\pi : V_\pm \to V$ has no kernel.  Since both are vector spaces of dimension 
$n$ it follows that $\pi : V_\pm \to V$ is an isomorphism.

\begin{defn} \label{d:genmetricsections} Given $\GG$ a 
generalized metric on $W$, define isomorphisms
\begin{align} \label{f:genmetricsections10}
\sigma_{\pm} := \pi_{|V_\pm}^{-1} : V \to V_{\pm}.
\end{align}
\end{defn}

Using now either $\sigma_+$ or $\sigma_-$, we can define an isotropic splitting of \eqref{eq:linearsurjection} (see  \eqref{eq:sigma0sigma})
\begin{align} \label{f:genmetricsections20}
\sigma = \sigma_\pm - \tfrac{1}{2}\pi^* \tau_\pm,
\end{align}
where $\tau_\pm \in \Sym^2 V^*$ is defined by $\tau_\pm(X,Y) = \IP{\sigma_\pm X,\sigma_\pm Y}$. Recall that, similar to Proposition \ref{p:exactCourant}, the isotropic splitting $\gs$ determines an isometry $F \colon V \oplus V^* \to W$ defined by
$$
F(X + \xi) = \sigma X + \tfrac{1}{2}\pi^* \xi
$$
for the symmetric product \eqref{Diracpairings}.

\begin{lemma}\label{l:Gmetricabs} A generalized metric $\GG$ on $W$ is equivalent to pair $(g,\sigma)$, where $g$ is a positive definite metric on $V$ and $\sigma: V \to W$ is an isotropic splitting of \eqref{eq:linearsurjection}, such that
\begin{align} \label{Gformula0}
F^{-1} \GG F =&\ \left( \begin{matrix}
                     0 & g^{-1}\\
                     g & 0
                    \end{matrix} \right), \qquad F^{-1}(V_\pm) = \{X \pm g(X): X \in V\}.
\end{align}
We refer to the generalized metric determined by $(g, \sigma)$ as $\GG(g,\sigma)$.
\begin{proof} 
Given a generalized metric $\GG$, we can construct $\sigma$ as in (\ref{f:genmetricsections20}). Notice that 
$$
\IP{\sigma_+(X) - \tfrac{1}{2}\pi^* \tau_+(X) -  \sigma_- + \tfrac{1}{2}\pi^* \tau_-,a} = 0
$$
for all $a \in W$, and therefore $\sigma$ is independent of the choice of $\gs_{\pm}$. We also have
$$
0 < \IP{\GG \sigma_+ X,\sigma_+ X} = \IP{\sigma_+ X,\sigma_+ X} = \tau_+(X,X)
$$
for $X \neq 0$, and therefore $g = \tau_+$ is a positive definite metric on $V$. Thus,
$$
F(X + g(X)) = \sigma X + \tfrac{1}{2}\pi^*g(X) = \sigma_+ X
$$
and therefore $F^{-1}(V_+) = \{X + g(X): X \in V\}$. Finally, since $F$ is an isometry and $V_-$ is orthogonal to $V_+$, \eqref{Gformula0} follows. Conversely, given a pair $(g,\sigma)$, where $g$ is a positive definite metric on $V$ and $\sigma: V \to W$ is an isotropic splitting, it is easy to verify that $\GG$ defined by \eqref{Gformula0} is a generalized metric on $W$.
\end{proof}
\end{lemma}

Our next goal is to make our previous discussion more explicit, writing a generalized metric in terms of tensors on the vector space $V$ in \eqref{eq:linearsurjection}. This will help us to determine the classical data which is equivalent to the specification of a generalized metric, in the geometric situation below. We fix an isotropic splitting $\sigma_0: V \to W$, inducing an identification
$$
W = V \oplus V^*,
$$
so that $\sigma_0(X) = X$ and the neutral metric on $W$ is given by \eqref{Diracpairings}. Given another isotropic splitting $\sigma: V \to V \oplus V^*$, it follows that $b = \sigma - \sigma_0 \colon  V \to V^*$, which can be identified with an element $b \in V^* \otimes V^*$. The isotropic condition implies further that $b \in \Lambda^2 V^*$ and therefore
$$
\sigma(X) = X + b(X) = e^b(X).
$$
Thus, we can identify the isometry $F \colon V \oplus V^* \to V \oplus V^*$ with $e^b$, that is,
\begin{align*}
F(X + \xi) = e^b (X + \xi) =&\ \left( 
\begin{matrix}
1 & 0\\
b & 1
\end{matrix} \right) 
\left( \begin{matrix}
X\\
\xi \end{matrix} \right) = X + \xi + i_X b.
\end{align*}
The next proposition follows now as an immediate consequence of Lemma \ref{l:Gmetricabs}.

\begin{prop}\label{p:genmetriclinearexp} A generalized metric $\GG$ on $V \oplus V^*$ is equivalent to a pair $(g,b)$ consisting of a positive definite metric $g$ on $V$ and $b \in \Lambda^2 T^*$, such that
\begin{align} \label{Gformula}
 \GG =&\  e^b \left( \begin{matrix}
                     0 & g^{-1}\\
                     g & 0
                    \end{matrix} \right) e^{-b} = \left(
                    \begin{matrix}
                    - g^{-1} b & g^{-1}\\
                    g - b g^{-1} b & b g^{-1}
                    \end{matrix} \right).
\end{align}
We refer to the generalized metric determined by $(g, b)$ as $\GG(g,b)$.
Furthermore, given $k \in \Lambda^2(V^*)$, we obtain a new generalized metric 
via $\GG_k := e^k \GG e^{-k}$, with associated pair $(g,b + k)$.
\end{prop}

In terms of the previous explicit description of a generalized metric $\GG = \GG(g,b)$ on $V \oplus V^*$, we have that 
$$
V_\pm = \{X + (b \pm g)X, \; X \in V\},
$$
that the projections $\pi_\pm \colon V \oplus V^* \to V_\pm$ in Definition \ref{d:metricprojections} are given by
$$
\pi_\pm (X + \xi) = \tfrac{1}{2}(X + (b \pm g)X) + \tfrac{1}{2}(\xi \pm bg^{-1}\xi \pm g^{-1}\xi),
$$
and that the linear sections $\sigma_\pm \colon V \to V_\pm$ in Definition \ref{d:genmetricsections} are
$$
\sigma_\pm(X) = X + (b \pm g)X.
$$

To conclude our study of the linear theory, we next characterize the tangent space to the space of generalized metrics. It will convenient to go back to the abstract setup at the beginning of this section (see \eqref{eq:linearsurjection}). We denote
$$
\GG(W) \subset \End(W)
$$
the space of generalized metrics on $W$.

\begin{lemma} \label{l:tangentGG1}
The tangent space at $\GG$ to the space of generalized metrics $\GG(W)$ is given by 
\begin{align*}
\left\{ \GG \Phi \in \End(W)\ |\ \Phi \in \Lambda^2 W, \quad \Phi \GG + \GG \Phi = 0 \right\}.
\end{align*}
\begin{proof} Let $\GG_t$ be a one-parameter family of generalized metrics on $W$, with $\GG_0 = \GG$. Denote
$$
\dot \GG = \left. \frac{d}{dt} \GG_t \right|_{t=0} \in \End (W).
$$
Setting $\Phi = \GG\dot \GG$, by Definition \ref{d:GGdefnlinear} it follows that
\begin{equation}\label{eq:tangentGG}
\Phi \in \Lambda^2 W, \qquad \Phi \GG + \GG \Phi = 0.
\end{equation}
It remains to show that an element $\Phi \in \End(W)$ satisfying \eqref{eq:tangentGG} can be realized through a curve of generalized metrics. For this, consider $\Phi_t = e^{-\tfrac{t}{2}\Phi} \in O(W)$ and define
$$
\GG_t = \Phi_t \GG \Phi_{-t}.
$$
It is easy to see that $\GG_t$ defines a path of generalized metrics and
$$
\dot \GG = \left. \frac{d}{dt} \GG_t \right|_{t=0} = - \tfrac{1}{2}\Phi \GG + \tfrac{1}{2}\GG \Phi = \GG \Phi,
$$
as required.
\end{proof}
\end{lemma}

By the second condition in \eqref{eq:tangentGG}, a tangent vector $\Phi$ at $\GG$ induces well-defined maps
$$
\Phi_{|V_-} \colon V_- \to V_+, \qquad \Phi_{|V_+} \colon V_+ \to V_-.
$$
Furthermore, the first condition in \eqref{eq:tangentGG} implies that $\Phi_{|V_+}$ is uniquely determined by $\Phi_{|V_-}$, thus leading us to the following result. 

\begin{lemma} \label{l:tangentGG2}
The map $\Phi \to \Phi_{|V_-}$ defines a linear isomorphism between the tangent space at $\GG$ and $\operatorname{Hom}(V_-,V_+)$.
\begin{proof} 
By Lemma \ref{l:tangentGG1}, the linear map $\Phi \to \Phi_{|V_-}$ is well-defined and injective. Given $R_+ \in \operatorname{Hom}(V_-,V_+)$, we can define $R_- \in \operatorname{Hom}(V_+,V_-)$ by
$$
\IP{R_- a_+,b_-} := - \IP{a_+,R_+ b_-}
$$
for any $a_+ \in V_+$ and $b_- \in V_-$. Then it follows that $\Phi = R_+ + R_- \in \End(W)$ is tangent to $\GG$, and furthermore $\Phi_{|V_-} = R_+$.
\end{proof}
\end{lemma}

Using Lemma \ref{l:Gmetricabs} and Proposition \ref{p:genmetriclinearexp}, we unravel our previous discussion in terms of the vector space $V$ in \eqref{eq:linearsurjection}. By Lemma \ref{l:Gmetricabs}, the space of generalized metrics $\GG(W)$ is in one-to-one correspondence with
$$
\mathcal{M}(V) \times \mathcal{S}(W),
$$
where $\mathcal{M}(V)$ denotes the space of positive definite metrics on $V$ and $\mathcal{S}(W)$ denotes the space of isotropic splittings $\sigma \colon V \to W$ of the sequence \eqref{eq:linearsurjection}. Upon a choice of an element $(g,\sigma) \in \mathcal{M}(V) \times \mathcal{S}(W)$, the space $\GG(W)$ is bijective to
$$
\operatorname{GL}(V)/O(V) \times \Lambda^2 V^*,
$$
where $O(V)$ denotes the space of orthogonal transformations of $(V,g)$. Thus, the tangent space at $\GG = \GG(g,\sigma)$ can be canonically identified with
$$
\Sym^2V^* \oplus \Lambda^2 V^* \cong V^* \otimes V^*.
$$
\begin{lemma} \label{l:symmetryaction}
Let $\GG = \GG(g,\sigma)$ be a generalized metric, and consider the induced identification $W = V \oplus V^*$ as in Lemma \ref{l:Gmetricabs}. Given a one-parameter family of generalized metrics $\GG_t = \GG(g_t,\sigma_t)$ with $\GG_0 = \GG$,
\begin{equation}\label{eq:dotGGexp}
\dot \GG = \left. \frac{d}{dt} \GG_t \right|_{t=0} = [{\bf 1} - e^{\dot b}, \GG_0] + \left(
\begin{matrix}
0 & - g^{-1} \dot g g^{-1}\\
\dot g & 0
\end{matrix} \right) = \left(\begin{matrix}
                    - g^{-1} \dot b & - g^{-1} \dot g g^{-1}\\
                    \dot g & \dot b g^{-1}
                    \end{matrix} \right),
\end{equation}
where
\begin{align*}
\left. \frac{d}{dt} g_t \right|_{t=0} = \dot g \in \Sym^2 V^*, \qquad \left. \frac{d}{dt} \sigma_t \right|_{t=0} = \dot b \in \Lambda^2 V^*.
\end{align*}
In particular, the associated tangent vector to the space of generalized metrics in Lemma \ref{l:tangentGG1} is
\begin{equation}\label{eq:Phiexp}
\Phi = \GG\dot \GG = \left(\begin{matrix}
                    g^{-1} \dot g & g^{-1} \dot b g^{-1}\\
                    - \dot b & - \dot g g^{-1}
                    \end{matrix} \right).
\end{equation}
\begin{proof}
Using the identification $W = V \oplus V^*$ induced by $\sigma$, following the notation in Lemma \ref{p:genmetriclinearexp} we can write $\GG_t = \GG(g_t,b_t)$ for 
$$
(g_t,b_t) \in \mathcal{M}(V) \times \Lambda^2V^*
$$
and $b_0 = 0$. Then, \eqref{eq:dotGGexp} follows easily taking derivatives in expression \eqref{Gformula}. We leave \eqref{eq:Phiexp} to the reader.
\end{proof}
\end{lemma}

\begin{rmk}
Given $\Phi \in \End(W)$ satisfying \eqref{eq:tangentGG} it is an \textbf{exercise} to verify that there exists $(\dot g, \dot b) \in \Sym^2 V^* \oplus \Lambda^2V^*$ such that \eqref{eq:Phiexp} holds. 
\end{rmk}

Given a generalized metric $\GG$ and a tangent vector $\Phi$, using the isomorphism $\sigma_- \colon V \to V_-$ in \eqref{f:genmetricsections10} combined with the map $\pi \colon V_+ \to V$ we can regard the homomorphism $ \Phi_{|V_-}$ in Lemma \ref{l:tangentGG2} as an element
$$
\pi \circ \Phi_{|V_-} \circ \sigma_- \in \End (V).
$$
Then, using \eqref{eq:Phiexp} we have the following.

\begin{lemma} \label{l:R-exp}
Let $\GG = \GG(g,\sigma)$ be a generalized metric. Given a tangent vector $\Phi$ as in \eqref{eq:Phiexp}, we have
\begin{equation}
\pi \circ \Phi_{|V_-} \circ \sigma_-(X) =  g^{-1}(\dot g(X) - \dot b(X)).
\end{equation}
\begin{proof}
Using the identification $W = V \oplus V^*$ induced by $\sigma$ we have $\sigma_-(X) = X - g(X)$, and the proof follows from \eqref{eq:Phiexp}.
\end{proof}
\end{lemma}

\subsection{Generalized metrics on manifolds}

We next turn to the geometry, by introducing the notion of generalized metric on a smooth manifold $M$. Consider the exact Courant algebroid $(T\oplus T^*, \IP{,}, [,],\pi)$ studied in \S \ref{s:Courant}.

\begin{defn} \label{d:GGdefn} Given $M$ a smooth manifold, a 
\emph{generalized metric} on $T \oplus T^*$ is an endomorphism $\GG \in \Gamma(\End T \oplus T^*)$ satisfying
\begin{enumerate}
 \item $\IP{\GG (X + \xi), \GG(Y + \eta)} = \IP{X + \xi, Y + \eta}$,
 \item $\IP{\GG (X + \xi), Y + \eta } = \IP{X + \xi, \GG(Y + \eta)}$,
 \item The bilinear pairing $\IP{\GG(X + \xi), Y + \eta}$ is symmetric and 
positive definite.
 \end{enumerate}
\end{defn}

Relying on Lemma \ref{l:Gmetricreduction0}, a generalized metric provides a reduction of the $O(n,n)$-bundle of frames of $T \oplus T^*$ to a maximal compact $O(n) \times O(n)$, thus showing that generalized metrics are natural analogues of ordinary metrics (regarded as reductions of the frame bundle of $M$ to the orthogonal group $O(n)$). The details of the following lemma are left to the reader.

\begin{lemma} \label{l:Gmetricreduction} A generalized metric on $T \oplus T^*$ is equivalent to either
\begin{enumerate}
 \item a reduction of the bundle of frames of $T \oplus T^*$ to a maximal compact subgroup $O(n) \times O(n) \leq O(n,n)$,
 \item an orthogonal bundle decomposition $T \oplus T^* = V_+ \oplus V_-$ such that the restriction of $\IP{,}$ to $V_+$ is positive definite.
 \end{enumerate}
\end{lemma}

We will use the same notation for the bundle maps $\pi_\pm, \sigma_\pm, \sigma$ as in the linear theory. The next proposition is fundamental to the study of generalized geometry,  indicating the classical data which is equivalent to the specification of a  generalized metric.  In particular, a generalized metric is seen equivalent to  a  classical Riemannian metric $g$, together with a $b$-field.  We have already  seen closed $b$-fields playing a role as symmetries of the Courant bracket in  \S \ref{ss:Courantsymmetries}, whereas here the associated $b$-field is not closed, and plays a nontrivial role in determining relevant associated geometric quantities. The proof of the following is straightforward from Proposition \ref{p:genmetriclinearexp}.

\begin{prop}\label{p:genmetricreduction} Given $M$ a smooth manifold, a generalized metric $\GG$ on $T \oplus T^*$ uniquely determines a pair $(g,b)$ consisting  of a Riemannian metric on $M$ and two-form $b \in \Gamma(\Lambda^2 T^*)$. Conversely, such a pair $(g,b)$ determines $\GG$ via
\begin{align*}
 \GG =&\ e^b \left( \begin{matrix}
                     0 & g^{-1}\\
                     g & 0
                    \end{matrix} \right) e^{-b} = \left(
                    \begin{matrix}
                    - g^{-1} b & g^{-1}\\
                    g - b g^{-1} b & b g^{-1}
                    \end{matrix} \right).
\end{align*}
We refer to the generalized metric determined by $(g, b)$ as $\GG(g,b)$.
Furthermore, given $k \in \Gamma(\Lambda^2T^*)$, we obtain a new generalized metric via $\GG_k := e^k \GG e^{-k}$, with associated pair $(g,b + k)$.
\end{prop}

By Proposition \ref{l:Gmetricabs}, a generalized metric $\GG$ determines an isotropic splitting on $T \oplus T^*$ given by
$$
\sigma(X) = e^b (X) = X + b(X), 
$$
and therefore a preferred presentation of the Dorfman bracket on $\Gamma(T \oplus T^*)$:
$$
e^{-b}[e^b(X + \xi), e^b (Y + \eta)] = [X,Y] + L_{X} \eta - i_{Y} d\xi + i_Y i_X db.
$$
We also obtain a preferred representative $db$ of the (trivial) \v Severa class of the exact Courant algebroid $(T\oplus T^*, \IP{,}, [,],\pi)$. This situation extends naturally to the case of generalized metrics on general exact Courant algebroids, as considered in Definition \ref{d:CAex}.

\begin{defn} \label{d:GGdefnECA} Given $M$ a smooth manifold and an exact Courant algebroid $E$ over $M$, a \emph{generalized metric} on $E$ is a bundle endomorphism $\GG \in \Gamma(\End E)$ satisfying
\begin{enumerate}
 \item $\IP{\GG a, \GG b} = \IP{a,b}$,
 \item $\IP{\GG a, b} = \IP{a, \GG b}$,
 \item The bilinear pairing $\IP{\GG a, b}$ is symmetric and 
positive definite.
 \end{enumerate}
\end{defn}

The analogue of Proposition \ref{p:genmetricreduction} for the case of a general exact Courant algebroid is provided by the following proposition. The proof follows directly from Lemma \ref{l:Gmetricabs}. Recall that an isotropic splitting $\gs \colon T \to E$ determines an isometry $F \colon T \oplus T^* \to E$ defined by
$$
F(X + \xi) = \sigma X + \tfrac{1}{2}\pi^* \xi
$$
for the symmetric product \eqref{Diracpairings}.

\begin{prop}\label{l:Gmetricabsgeom} Given $M$ a smooth manifold and an exact Courant algebroid $E$ over $M$, a generalized metric $\GG$ is equivalent to pair $(g,\sigma)$, where $g$ is a positive definite metric on $M$ and $\sigma: T \to E$ is an isotropic splitting of \eqref{eq:linearsurjection}, such that
\begin{align*}
F^{-1} \GG F =&\ \left( \begin{matrix}
                     0 & g^{-1}\\
                     g & 0
                    \end{matrix} \right), \qquad F^{-1}(V_\pm) = \{X \pm g(X): X \in T\}.
\end{align*}
Consequently, a generalized metric determines a preferred presentation of the Dorfman bracket on $\Gamma(E)$,
$$
F^{-1}[F(X + \xi),F(Y + \eta)] = [X,Y] + L_{X} \eta - i_{Y} d\xi + i_Y i_X H,
$$
where $H$ is a preferred representative of the \v Severa class of $E$, given by
$$
H(X,Y,Z) = 2\IP{[\gs X, \gs Y], \gs Z}.
$$
\end{prop}  

Given a generalized metric $\GG$ on an exact Courant algebroid $E$ over $M$, our next goal is to characterize the group of \emph{generalized isometries}, that is, 
the group of Courant automorphisms $F \colon E \to E$ which preserve the generalized metric. Using the isotropic splitting induced by $\GG$, $E$ is isomorphic to the twisted Courant algebroid $(T\oplus T^*, \IP{,}, [,]_H,\pi)$, whose automorphism group has been characterized in Proposition \ref{p:gendiff}.

\begin{prop} \label{p:gendisom} 
Let $\GG = \GG(g,\sigma)$ be a generalized metric on an exact Courant algebroid $E$ over $M$. The \emph{group of generalized isometries} of $\GG$, defined as the group of Courant automorphisms $F \colon E \to E$ which preserve the generalized metric, is given by
\begin{align*}
   \{ \overline{f} : f \in \Diff(M), \; \textrm{such that} \;  f^*H = H, \; f^*g = g \} \subset \operatorname{Aut}(E)
\end{align*}
where $H \in \Omega^3_{cl}(M)$ is the preferred representative of the \v Severa class of $E$ determined by $\GG$.
\begin{proof} 
Consider the isomorphism $E \cong (T\oplus T^*, \IP{,}, [,]_H,\pi)$ given by the generalized metric $\GG$. With this identification, by Lemma \ref{l:Gmetricabs}) we have
\begin{align*} 
\GG =&\ \left( \begin{matrix}
                     0 & g^{-1}\\
                     g & 0
                    \end{matrix} \right).
\end{align*}
By Proposition \ref{p:gendiff}, a general element $\overline{f} \circ e^B \in \operatorname{Aut}(E)$ preserves $\GG$, that is,
$$
\overline{f} \circ e^B \circ \GG = \GG \circ \overline{f} \circ e^B,
$$
if and only if
$$
(f^*g)^{-1}(B(X) + \xi) = g^{-1}\xi, \qquad (f^*g)(X) = B(g^{-1}\xi) + g(X)
$$
for all $X + \xi \in \Gamma(T \oplus T^*)$. Setting e.g. $\xi = 0$ this implies $f^*g = g$, and consequently $B = 0$, as claimed. Returning again to Proposition \ref{p:gendiff} it follows that $f^* H = H$.
\end{proof}
\end{prop}

\begin{rmk}\label{r:infisometry}
With the notation in Proposition \ref{p:genmetricreduction}, the action of $F = \overline{f} \circ e^B$ on $\GG$ is
$$
F \circ \GG \circ F^{-1} = \GG(f_*g,f_* B).
$$ 
In particular, by Lemma \ref{l:symmetryaction} the infinitesimal action of $X + B \in \operatorname{Lie} \operatorname{Aut}(E)$ is given by
\begin{equation*}
(X + B) \cdot \GG = \left(\begin{matrix}
                    - g^{-1} B & g^{-1} (L_X g) g^{-1}\\
                    -L_X g & B g^{-1}
                    \end{matrix} \right).
\end{equation*}
\end{rmk}

To finish this section, we show that a generalized metric $\GG$ determines an energy functional $S_\GG$ on the space of smooth maps $C^\infty(\Sigma,M)$ from a surface $\Sigma$ into our target manifold $M$. This functional can be thought of as an analogue of the energy functional which associates the length squared to a parametrized curve on $M$, leading to the notion of distance and geodesics in Riemannian geometry. Morally, $S_\GG$ can be similarly used to define a pseudo-metric on the space of loops $C^\infty(S^1,M)$ of $M$, by extremizing the energy with prescribed boundary components. We will follow our discussion at the end of \S \ref{ss:Courantsymmetries}. 

Let $M$ be an oriented smooth manifold. Let $\Sigma$ be a smooth compact surface (oriented and without boundary, say). Consider the infinite-dimensional space of smooth maps $C^\infty(\Sigma,M)$. Let $(E,\GG)$ be an exact Courant algebroid over $M$ endowed with a generalized metric $\GG$. By Lemma \ref{l:Gmetricabsgeom}, $\GG$ determines a Riemannian metric $g$ and a closed three-form $H \in \Lambda^3 T^*$ on $M$. We will assume that $[H] \in H^3(M,\mathbb{Z})$. We fix a Riemannian metric $\eta$ on $\Sigma$ with volume $dV_\eta$. Then, we can define a functional
\begin{equation}\label{eq:S_GGWZW}
S_\GG \colon C^\infty(\Sigma,M) \to \mathbb{R}/\mathbb{Z}, \qquad \varphi \mapsto \frac{1}{2}\int_\Sigma |d\varphi|^2 dV_\eta + \int_{Y} \overline \varphi^*H.
\end{equation}
Here, $|d\varphi|^2$ denotes the norm squared of $d\varphi \in \Gamma(\Sigma,T^*\Sigma \otimes \varphi^* T)$ with respect to the metric $\eta^* \otimes \varphi^* g$ and $Y$ is a three-manifold with boundary $\Sigma$ and $\overline \varphi \colon Y \to M$ is any smooth extension of $\varphi$ to $Y$. This functional is known in the mathematical physics literature as the \emph{Polyakov action with Wess-Zumino term} in the action of a two-dimensional $\Sigma$-model (see e.g. \cite{Gawedzki}). Note that the second summand in the definition of $S_\GG$ corresponds to the functional $S_H$ in \S \ref{ss:Courantsymmetries}. Our next result shows that the critical points of $S_\GG$ provide an interesting generalization of the harmonicity condition for maps $\varphi \in C^\infty(\Sigma,M)$.

\begin{prop} \label{p:SGGcritical}
The critical points $\varphi \in C^\infty(\Sigma,M)$ of the functional $S_\GG$ are the solutions of the equation
$$
\nabla^*\nabla \varphi + \IP{\varphi^* g^{-1}H,dV_\eta}_\eta = 0,
$$
where $\nabla^*\nabla$ is the rough Laplacian for the product of Levi-Civita connections of $\eta$ and $g$, and $\IP{\varphi^* g^{-1}H,dV_\eta}_\eta \in \Gamma(\varphi^*T)$ denotes the product of $\varphi^* g^{-1}H \in \Gamma(\varphi^*T \otimes \Lambda^2 T^*\Sigma)$ with $dV_\eta$ with respect to $\eta$.
\begin{proof} 
A calculation in coordinates shows that the functional $S_\GG$ can be alternatively written as
$$
S_\GG(\varphi) = \frac{1}{2}\int_\Sigma |\nabla\varphi|^2 dV_\eta + \int_{Y} \overline \varphi^*H.
$$
This follows simply from the fact that action of the Levi-Civita connection on functions coincides with the exterior differential. Then, the variation of $S_\GG$ along $X \in \Gamma(\varphi^*T)$, identified with a vector in the tangent space of $C^\infty(\Sigma,M)$ at $\varphi$, is given by
\begin{align*}
\delta S_\GG(X)& = \int_\Sigma \IP{\nabla\varphi,\nabla X} dV_\eta  + \int_\Sigma \varphi^* i_X H \\
& = \int_\Sigma \IP{\nabla^* \nabla\varphi,X} dV_\eta  + \int_\Sigma \IP{\varphi^* g^{-1}H,X \otimes dV_\eta} dV_\eta,
\end{align*}
as claimed.
\end{proof}
\end{prop}

\section{Divergence operators} \label{s:divops}

As we have seen in the previous section, natural notions in generalized geometry are described in terms of the $O(n,n)$-bundle of frames of $T \oplus T^*$. Starting with its definition as a reduction to a maximal compact subgroup, a generalized metric $\GG$ is \emph{a posteriori} seen to be equivalent to a classical Riemannian metric $g$ together with a $b$-field. Since the neutral metric on $T \oplus T^*$ is fixed in this geometry, how does generalized geometry keep track of conformal rescalings of the classical metric associated to $\GG$? In this section we answer this question, by means of the following device introduced in \cite{GF19}.

\begin{defn}\label{def:divergenceop}
Let $E$ be an exact Courant algebroid over a smooth manifold $M$. A \emph{divergence operator} on $E$ is a differential operator $\divop \colon \Gamma(E) \to C^\infty(M)$ satisfying
\begin{equation}\label{eq:Leibnizdiv}
\divop(f e) =  \pi(e)(f) + f \divop(e),
\end{equation}
for all $e \in \Gamma(E)$ and $f \in C^\infty(M)$.
\end{defn}

As a consequence of the Leibniz rule \eqref{eq:Leibnizdiv}, the divergence operators on $E$ form an affine space modeled on $\Gamma(E)$, since if we fix a divergence operator $\divop$, any other divergence operator $\divop'$ is given by
$$
\divop' = \divop + \IP{e,\cdot}
$$
for some $e \in \Gamma(E)$. The following result shows that this space is always non-empty.

\begin{lemma} \label{l:divergenceop}
Let $E$ be an exact Courant algebroid over $M$. Given a connection $\nabla$ on $T$, one obtains a divergence operator on $E$ via
$$
\divop^\nabla(e) = \tr \nabla \pi(e),
$$
where $\pi \colon E \to T$ is the anchor map.
\begin{proof}
This follows from a simple calculation:
$$
\divop^\nabla(f e) = f \divop^\nabla(e) + \tr df \otimes \pi(e) = f \divop^\nabla(e) + \pi(e)(f).
$$
\end{proof}
\end{lemma}

More explicitly, consider an isotropic splitting $\sigma \colon T \to E$ and the induced identification $E = T \oplus T^*$, with anchor map
$$
\pi(X + \xi) = X.
$$
Then, it follows that, for $e = X + \xi$,
$$
\divop^\nabla(e) = \tr \nabla X.
$$
Assuming now that $\N$ is torsion-free, so that $\tr \nabla X = \tr \left( \nabla_X - L_X \right)$, and that there exists a density $\mu \in \Gamma(|\det T^*|)$ preserved by $\nabla$, we obtain
$$
\divop^\nabla(e) \mu = L_X \mu,
$$
which recovers the standard notion of divergence for the vector field $\pi(e) = X$. In particular, for any choice of metric $g$ on $M$ we can choose its Levi-Civita connection $\nabla$ and the associated Riemannian density $\mu^g$, obtaining a divergence operator
$$
\divop^g(e) \mu^g = L_X \mu^g.
$$

\begin{defn}\label{d:GRiemnniandiv} Given a generalized metric $\GG$ on an exact Courant algebroid $E$, define its \emph{Riemannian divergence} as
\begin{equation*}
\divop^\GG(e) = \frac{L_{\pi(e)} \mu^g}{\mu^g},
\end{equation*}
where $g$ is the Riemannian metric associated to $\GG$ in Lemma \ref{l:Gmetricabsgeom}, and $\mu^g$ is the associated Riemannian density.
\end{defn}

The relation between divergence operators and conformal geometry appears via the notion of \emph{Weyl structure} \cite{Folland}. Recall that a conformal structure $\mathcal{C}$ on a smooth manifold $M$ is a set of Riemannian metrics, such that if $g,g' \in \mathcal{C}$ there exists a smooth function $f \in C^\infty(M)$ such that
$$
g' = e^f g.
$$

\begin{defn}\label{d:Weylstr} 
Let $\mathcal{C}$ be a conformal structure on a smooth manifold $M$. A \emph{Weyl structure} is a map
$$
W \colon \mathcal{C} \to \Gamma(T^*)
$$
such that $W(e^f g) = W(g) - df$ for all $g \in \mathcal{C}$ and $f \in C^\infty(M)$.
\end{defn}

Let us fix an exact Courant algebroid $E$ over $M$, and denote by $\GG(E)$ the space of generalized metrics on $E$. By Lemma \ref{l:Gmetricabsgeom}, there is a bijection 
$$
\GG(E) \cong \mathcal{M}(T) \times \mathcal{S}(E),
$$
where $\mathcal{M}(T)$ and $\mathcal{S}(E)$ denote the space of Riemannian metrics on $M$ and the space of isotropic splittings of $E$, respectively. Given $(g,\sigma) \in \mathcal{M}(T) \times \mathcal{S}(E)$, we denote $\GG(g,\sigma)$ the associated generalized metric. Our next result shows that a divergence operator provides a natural analogue of a Weyl structure in the context of generalized geometry. 

\begin{lemma} \label{l:Weyl}
Let $\divop$ be a divergence operator on $E$. Define
$$
W \colon \GG(E) \to \Gamma(E) \colon \GG \mapsto \divop - \divop^{\GG},
$$
where $\divop^{\GG}$ is the Riemannian divergence of $\GG$ (see Definition \ref{d:GRiemnniandiv}). Then,
$$
W(\GG(e^{f}g,\sigma)) = W(\GG(g,\sigma)) - n \IP{df,}
$$
for any $(g,\sigma) \in \mathcal{M}(T) \times \mathcal{S}(E)$ and any $f \in C^\infty(M)$.
\begin{proof}
Consider the identification $E = T \oplus T^*$ given by $\GG = \GG(g,\sigma)$. Then, setting $\GG' = \GG(e^{f}g,\sigma)$ we have $\mu^{e^{f}g}= e^{\tfrac{n}{2}f}\mu^g$ and therefore
$$
\divop^{\GG'}(X + \xi) = \frac{L_X(e^{\tfrac{n}{2}f}\mu^g)}{e^{\tfrac{n}{2}f}\mu^g} = \divop^{\GG}(X + \xi) + \tfrac{n}{2}df(X) = \divop^{\GG}(X + \xi) + n\IP{df,X}.
$$
\end{proof}
%Notice that, via the identification $E \cong E^*$ using $\IP{,}$, we have $\pi^*df = 2df$.
\end{lemma}

Our next goal is to understand a natural compatibility condition between a generalized metric and a divergence operator, which plays an important role in the definition of the generalized Ricci flow. Given a generalized metric $\GG$ on $E$ and $e \in \Gamma(E)$ we define an endomorphism
$$
[e,\GG] \in \End(E),
$$
by
$$
[e,\GG](e') := [e,\GG e'] - \GG[e,e'].
$$
It is an easy \textbf{exercise}, using axiom (3) of the Dorfman bracket in Definition \ref{d:CA}, to check that $[e,\GG]$ indeed defines a tensor. 

\begin{defn}\label{d:infisometry}
Given a generalized metric $\GG$ on $E$, we say that $e \in \Gamma(E)$ is an \emph{infinitesimal isometry} if $[e,\GG] = 0$.
\end{defn}

\begin{lemma} \label{l:infisometry}
Let $\GG$ be a generalized metric on $E$. Then, $e \in \Gamma(E)$ is an 
infinitesimal isometry if and only if
$$
L_Xg = 0, \qquad d\xi = i_X H,
$$
where $e = X + \xi$ via the identification $E = T \oplus T^*$ given by $\GG$.
\begin{proof}
Consider the identification $E = T \oplus T^*$ given by $\GG = \GG(g,\sigma)$, so that $e = X + \xi$. Using the notation in \S \ref{ss:Courantsymmetries}, we have (see \eqref{eq:Psimap} and Remark \ref{r:infisometry})
$$
[e,\GG] = \Psi(X + \xi) \cdot \GG = \left(\begin{matrix}
                    - g^{-1} (d\xi - i_X H) & g^{-1} (L_X g) g^{-1}\\
                    -L_X g & (d\xi - i_X H) g^{-1}
                    \end{matrix} \right),
$$
and the result follows.
\end{proof}
\end{lemma}

\begin{rmk}\label{r:infisometryflat}
Following the discussion in \S \ref{ss:Courantsymmetries}, an infinitesimal isometry provides a flat direction for the Wess-Zumino functional $S_H$ in \eqref{eq:SH}.
\end{rmk}

We introduce the following compatibility condition.

\begin{defn}\label{d:GGdivergencecomp}
Let $E$ be an exact Courant algebroid over $M$. A pair $(\GG,\divop)$ is \emph{compatible} if $W(\GG)$ is an infinitesimal isometry of $\GG$, that is,
$$
[W(\GG),\GG] = 0.
$$ 

\end{defn}

As a straightforward consequence of Lemma \ref{l:Gmetricabsgeom} and Lemma \ref{l:infisometry} we obtain the following characterization of compatible pairs $(\GG,\divop)$.

\begin{prop} \label{p:GGdivergencecomp}
Let $E$ be an exact Courant algebroid over $M$. Then, $(\GG,\divop)$ is a compatible pair if and only if
$$
\divop = \divop^{\GG} + \IP{\sigma(X) + \tfrac{1}{2}\pi^*\xi, \cdot}
$$
where $\GG = \GG(g,\sigma)$ and 
$$
L_X g = 0, \qquad d\xi = i_X H.
$$
\end{prop}

An interesting class of infinitesimal isometries, and hence of compatible pairs, is provided by closed $1$-forms, that is, given by $X = 0$ and $d \xi = 0$. More invariantly, this condition can be expressed in terms of the anchor map as 
$$
[W(\GG),\GG] = 0, \qquad \pi(W(\GG)) = 0.
$$
This class of compatible pairs, associated to closed one-forms on the manifold, will be formally introduced in Definition \ref{def:closed} (see also Definition \ref{def:exact}), and further studied in \S \ref{sec:ggdirac} and \S \ref{s:EHF}.

\chapter{Generalized Connections and Curvature}\label{c:GCC}

Having established the fundamental properties of generalized metrics, we turn to the question of connections and curvature.  
The fundamental theorem of Riemannian geometry states that for each Riemannian metric 
there is a unique linear connection which is compatible with the metric and 
torsion-free. Here the story is more complicated, once one observes that the condition of having zero torsion does not determine a metric-compatible generalized connection uniquely.  The relevant connections in generalized Riemannian geometry will be related to certain classical connections on the tangent bundle, which in most cases will have nontrivial torsion. We analyze the curvature of these connections from the generalized and classical points of view, eventually leading to the definition of the generalized Ricci and scalar curvatures. In particular, we will study in detail the classical Bismut connection, a metric compatible connection with totally skew-symmetric torsion, deriving formulas for the curvature of this connection as well as Bianchi identities.

With these fundamental concepts at hand, we turn to the task of defining canonical generalized geometries. 
For this, we seek to extend the classical Einstein equation to the setting of generalized geometry.  By coupling the classical Einstein-Hilbert action to the torsion we motivate an Einstein-like equation from a variational point of view, and give nontrivial examples of generalized Einstein structures. In particular, we will prove that all Bismut-flat connections are associated to bi-invariant metrics on semisimple Lie groups.

\section{Generalized connections}\label{s:connection}

We start with the study of connections in the setting of generalized geometry.  In the spirit of treating a Courant algebroid $E$ as the generalized tangent bundle, we want to not only differentiate sections of $E$, but we want to differentiate them in the tangent as well as cotangent `directions'. 
Parallel to the notion of orthogonal affine connection in Riemannian geometry, we introduce next the concept of generalized connection, which allows us to differentiate sections of $E$ along directions in $E$, in a way compatible with the inner product $\la,\ra$. The study of generalized connections was initiated in \cite{AXu}, and largely expanded in \cite{GF14,GF19,GualtieriBranes,jurco2016courant}.

\begin{defn}\label{d:GConn}
A \emph{generalized connection} on an exact Courant algebroid $E$ is a first-order differential operator
$$
D \colon \Gamma(E) \to \Gamma(E^* \otimes E)
$$
which satisfies the Leibniz rule and is compatible with $\la,\ra$, i.e.
\begin{equation}\label{eq:generalizedconn}
\begin{split}
D_{a}(fb) &= \pi(a)(f)b + fD_{a}b,\\
\pi(a)\langle b,c \rangle & = \langle D_{a} b,c \rangle + \langle b,D_{a} c \rangle.
\end{split}
\end{equation}
We will use the notation $D_{a}b = \la a, Db\ra$ for the natural pairing between elements in $E^*$ and elements in $E$.
\end{defn}

\begin{lemma} \label{l:genconns} Given $E$ an exact Courant algebroid, the space of generalized connections is a nonempty affine space modeled on the vector space
$$
\Gamma(E^* \otimes \mathfrak{o}(E)),
$$
where
$$
\mathfrak{o}(E) = \{\Phi \in \End(E) : \langle \Phi ,\ra = - \langle , \Phi \ra \}.
$$
\begin{proof}
Fix a standard orthogonal connection $\nabla^E$ on $E$, and then one can check that the formula 
\begin{equation}\label{eq:generalizedconnexp}
D_{a} b = \nabla^E_{\pi(a)}b
\end{equation}
defines a generalized connection on $E$.  The proof that the difference of two generalized connections lies in $\gG(E^* \otimes \mathfrak{o}(E))$ is a straightforward \textbf{exercise} using the axioms (\ref{eq:generalizedconn}).
\end{proof}
\end{lemma}

\begin{ex}\label{ex:nablanabla*}
 Using Proposition \ref{p:exactCourant}, choose an isotropic splitting $\sigma \colon T \to E$, with corresponding isomorphism $F \colon E \to T \oplus T^*$.  Fix an affine connection $\nabla$ on $T$, with connection $\N^*$ naturally induced on $T^*$ by enforcing
$$
X(\ga(Y)) = \ga(\N_X Y) + (\N_X^* \ga)(Y).
$$
Then it is straightforward to check that the connection
$$
\nabla^E = \nabla \oplus \nabla^*
$$
on $T \oplus T^*$ is orthogonal with respect to $\IP{,}$, and thus will define a generalized connection via \eqref{eq:generalizedconnexp}.  Finally, pulling-back $D$ via $F$, we obtain the desired generalized connection on $E$.
\end{ex}

\begin{defn}\label{def:torsion} Given $E$ an exact Courant algebroid and $D$ a generalized connection on $E$, the \emph{torsion} $T_D$ of $D$ is defined by
\begin{equation}\label{eq:torsionG}
T_D(a,b,c) = \la D_{a}b - D_{b}a - [a,b],c \ra + \la D_{c}a,b \ra,
\end{equation}
for any $a,b,c \in \Gamma(E)$.
\end{defn}

Using the axiomatic definition of Courant algebroid and the Leibniz rule for a generalized connection (see Definition \ref{d:CA} and Definition \ref{d:GConn}), one can easily prove that formula \eqref{eq:torsionG} defines a three-tensor $T_D \in \Gamma(E^{\otimes 3})$.  Another elementary calculation using that $D$ is compatible with $\la,\ra$ shows that, indeed,
\begin{equation}\label{eq:torsionskw}
T_D \in \Gamma(\Lambda^3 E^*).
\end{equation}
Further computations show that this definition is equivalent to other formulations which have appeared (eg. \cite{AXu, GualtieriBranes}).  Rather than giving an abstract proof of the symmetries of the generalized torsion, we derive these by means of an explicit formula in the next result. We will use $\gs(a,b,c)$ to denote sum over cyclic pertumutations on $a,b,c$.

\begin{lemma} \label{l:torsionformula} Given $M$ a smooth manifold let $E = T \oplus T^*$ with $H$-twisted Dorfman bracket \eqref{f:Hbracket}.  Let $\nabla$ be an affine connection on $M$, and let $\N^E = \N \oplus \N^*$.  Given $\chi \in \Gamma(E^* \otimes \mathfrak{o}(E))$, the torsion of $D = \nabla^E_{\pi(\cdot)} + \chi$ is given by 
\begin{equation}\label{eq:torsionG2}
\begin{split}
T_D(a,b,c) = \sum_{\gs(a,b,c)} \left\{ \la T_{\nabla}(\pi a,\pi b),c\ra + \la \chi_a b,c\ra - \tfrac{1}{6}H(\pi a,\pi b, \pi c) \right\}.
\end{split}
\end{equation}
\end{lemma}

\begin{proof}
Denoting $D^0 = \nabla^E_{\pi(\cdot)}$, we have that
\begin{align*}
T_D(a,b,c) =&\ T_{D^0}(a,b,c) + \la \chi_{a}b - \chi_{b}a,c \ra + \la \chi_{c}a,b \ra\\
=&\  T_{D^0}(a,b,c) + \sum_{\gs(a,b,c)} \la \chi_a b,c\ra,
\end{align*}
where we have used that $\la \chi_{b}a,c \ra = - \la \chi_{b}c,a \ra$. Setting now $a = X+\xi$, $b = Y + \eta$, $c = Z + \zeta$, we define
\begin{align*}
N :=&\ D^0_a b - D^0_ba - [a,b]\\
=&\ T_\nabla(X,Y) + \nabla^*_X\eta  - \nabla^*_Y\xi - L_X \eta + i_Y d\xi - i_Yi_X H.
\end{align*}
It follows that
\begin{align*}
T_{D^0}(a,b,c) & = \la N ,c \ra + \la D_c a, b \ra  \\
& = \la N ,c \ra + \la \nabla^E_Z a, b\ra\\
& = \la N ,c \ra + \tfrac{1}{2}\eta(\nabla_Z X) + \tfrac{1}{2}i_Y(\nabla_Z^* \xi) \\
& = \la N ,c \ra + \tfrac{1}{2}\eta(\nabla_Z X) - \tfrac{1}{2}\xi(\nabla_Z Y) +  \tfrac{1}{2}i_Z d(\xi(Y)) \\
& = - \tfrac{1}{2} H(X,Y,Z) + \la T_{\nabla}(\pi a,\pi b),c\ra + \la \nabla^*_X\eta  - \nabla^*_Y\xi - L_X \eta + i_Y d\xi, c\ra\\
& \qquad + \tfrac{1}{2}\eta(\nabla_Z X) - \tfrac{1}{2}\xi(\nabla_Z Y) +  \tfrac{1}{2}i_Z d(\xi(Y)) \\
& = - \tfrac{1}{2} H(X,Y,Z) + \la T_{\nabla}(\pi a,\pi b),c\ra + \la \nabla_ZX - \nabla_XZ,b\ra + \la \nabla_YZ - \nabla_ZY,a\ra \\
& \qquad + \tfrac{1}{2}i_X  d(\eta(Z)) +  \tfrac{1}{2}i_Z d(\xi(Y)) - \tfrac{1}{2}i_Y d(\xi(Z)) - \tfrac{1}{2}i_ZL_X\eta + \tfrac{1}{2}i_Zi_Y d\xi\\
& = - \tfrac{1}{2} H(X,Y,Z) + \la T_{\nabla}(\pi a,\pi b),c\ra + \la \nabla_ZX - \nabla_XZ,b\ra + \la \nabla_YZ - \nabla_ZY,a\ra \\
& \qquad + \tfrac{1}{2}(L_X i_Z - i_Z L_X) \eta + \tfrac{1}{2}(i_Z L_Y - L_Y i_Z) \xi\\
& = - \tfrac{1}{2} H(X,Y,Z) + \la T_{\nabla}(\pi a,\pi b),c\ra + \la T_{\nabla}(\pi b,\pi c),a\ra + \la T_{\nabla}(\pi c,\pi a),b\ra,
\end{align*}
as claimed. The last part of the statement follows from the standard symmetry $T_\nabla(X,Y) = - T_\nabla(Y,X)$ for the torsion of an affine connection $\nabla$ on $T$.
\end{proof}

With the previous result at hand, we can provide some Lie theoretic motivation for Definition \ref{def:torsion}. Suppose for a moment that we have an abstract Courant algebroid $E$ and assume that $M$ is just one point. Then, the axioms in Definition \ref{d:CA} tell us that $E$ is a \emph{quadratic Lie algebra}, that is, a Lie algebra endowed with an invariant bilinear form $\la,\ra$. In this setup, a generalized connection is simply an element $\chi \in E^* \otimes \mathfrak{o}(E)$ and the torsion
$$
T_\chi(a,b,c) = \sum_{\gs(a,b,c)} \la \chi_a b,c\ra - \tfrac{1}{3}\la [a,b],c \ra
$$
measures the difference in $\Lambda^3 E^*$ between the total skew-symmetrization of $\chi$ and the canonical Cartan three-form on the quadratic Lie algebra $E$.

In order to study torsion-free generalized connection compatible with a generalized metric in the next section, we finish this section with a detailed analysis of the space of generalized connections with fixed torsion, and its relation with divergence operators in Definition \ref{def:divergenceop}. Let $T$ be an element in $\Gamma(\Lambda^3 E^*)$, and consider the set $\cD^T$ of generalized connections on $E$ with fixed torsion $T$
$$
\cD^T \subset \cD.
$$

\begin{lemma} Given $E$ a Courant algebroid and $D$ a connection with torsion $T$.  The space $\cD^T$ is an affine space modeled on
\begin{equation}\label{eq:sigma}
\Sigma = \left \{ \chi \in \Gamma(E^{\otimes 3}): \chi(a,b,c) = - \chi(a,c,b), \ \sum_{\gs(a,b,c)} \ \chi(a,b,c) = 0 \right\}.
\end{equation}
\begin{proof} \textbf{Exercise}.
\end{proof}
\end{lemma}

\begin{rmk} The space $\Sigma$ admits a canonical splitting
\begin{equation}\label{eq:splittingyoung}
\Sigma = \Sigma_0 \oplus \Gamma(E),
\end{equation}
where $e \in \Gamma(E)$ corresponds to the mixed symmetric tensor $\chi^e$ defined by
$$
\chi^e(a,b,c) = \la a,b \ra \la e,c\ra - \la e,b \ra \la a,c \ra
$$
and the orthogonal complement of $\Gamma(E)$ is given by
\begin{equation}\label{eq:sigma0}
\Sigma_0 = \Big{\{}\chi\in \Sigma : \sum_{i=1}^{2n} \chi(e_i,\til e_i,\cdot) = 0\Big{\}}.
\end{equation}
Here, $2n$ is the rank of $E$ (twice the dimension $n$ of the smooth manifold $M$), $\{e_i\}$ is an orthogonal local frame for $E$ and the $\{\til e_i\}$ are sections of $E$ defined so that $\la e_i, \til e_j \ra = \delta_{ij}$. 

More explicitly, given an arbitrary element $\chi \in \Gamma(E^* \otimes \mathfrak{o}(E))$ we have a unique decomposition
\begin{equation}\label{eq:decompositionchi}
\chi = \chi_0 + \chi^e,
\end{equation}
where
$$
\chi^e_{a}b = \la a,b \ra e - \la e,b \ra a
$$
with 
$$
e = \frac{1}{2n -1}\sum_{i=1}^{2n} \chi_{e_i}\til e_i.
$$
It follows by construction that $\chi - \chi^e \in \Sigma_0$.
\end{rmk}

The $E^*$-valued skew-symmetric endomorphism $\chi^e$ in the decomposition \eqref{eq:decompositionchi} is reminiscent of the `$1$-form valued Weyl endomorphisms' in conformal geometry, which appear in the variation of a metric connection with fixed torsion upon a conformal change of the metric. Similarly, our $E^*$-valued \emph{Weyl endomorphisms} $\chi^e$ enable us to deform a generalized connection $D$ with fixed torsion $T$ inside the space $\cD^T$. For later applications in this work, it is important to study the interaction between the space $\mathcal{D}^T$ and the notion of divergence operator. This will allow us to  `gauge-fix' the Weyl degrees of freedom in the space of generalized connections $\cD^T$ corresponding to $\Gamma(E)$ in \eqref{eq:splittingyoung}. The basic observation is that any generalized connection determines a divergence operator as in \S \ref{s:divops}.  Given this, one can show that the divergence constrains the Weyl degrees of freedom in the space generalized connections with fixed torsion, as required.

\begin{defn} \label{def:divergence}
The divergence operator of a generalized connection $D$ on $E$ is
\begin{equation*}
\divop_D(a) = \tr Da \in C^\infty(M).
\end{equation*}
\end{defn}

\begin{lemma}\label{lem:weylfixed} Given $E$ an exact Courant algebroid fix
 $T \in \Gamma(\Lambda^3E^*)$ and $\divop \colon \Gamma(E) \to C^\infty(M)$ a divergence operator on $E$. Then
$$
\cD^T(\divop) := \{D \in \cD \; | \; T_D = T, \; \divop_D = \divop \} \subset \cD^T
$$
is an affine space modeled on $\Sigma_0$, as defined in \eqref{eq:sigma0}.
\end{lemma}

\begin{proof}
Given $D', D \in \cD^T$, we denote $D' = D + \chi$ with $\chi \in \Sigma$ (see \eqref{eq:sigma}). Then, we have
\begin{equation*}
\divop_{D'}(a) = \divop_{D}(a) - \sum_{i=1}^{2n}\la \chi_{\til e_i} e_i,a\ra.
\end{equation*}
Decomposing $\chi = \chi_0 + \chi^e$ as in \eqref{eq:decompositionchi}, for $e \in \Gamma(E)$, we obtain
$$
\sum_{i=1}^{2n}\chi_{\til e_i} e_i = (2n - 1)e.
$$
Imposing now $\divop_{D'} = \divop_D = \divop$, implies $e = 0$, which concludes the proof.
\end{proof}

\section{Metric compatible connections}\label{sec:ggmetrics}

In this section we prove an analogue in generalized geometry of the Fundamental Lemma of Riemannian geometry, that is, the existence of a unique torsion-free connection compatible with a given metric.  In generalized geometry there is always the freedom to choose a compatible torsion-free generalized connection, that we will constrain via the Weyl gauge fixing mechanism introduced in Lemma \ref{lem:weylfixed}. This will lead us to the definition of canonical operators in the next sections, providing a weak analogue of the Fundamental Lemma.  Throughout this section we fix an exact Courant algebroid $E$ over a smooth manifold $M$ and a generalized metric $\GG$. Ideas in this direction appeared originally in \cite{GualtieriBranes,HitchinBrackets}. Our discussion here follows closely \cite{GF19}.

\begin{defn} \label{d:compatibility}
A generalized connection $D$ on $E$ is \emph{compatible with $\GG$} if
\begin{align*}
D \GG \equiv 0.
\end{align*}
Equivalently, this condition will hold if
$$
D (\Gamma(V_\pm)) \subset \Gamma(E^* \otimes V_\pm).
$$
Thus, the space of $\GG$-compatible generalized connections on $E$, which we denote $\cD(\GG)$, is an affine space modeled on
\begin{equation*}
\Gamma(E^* \otimes \mathfrak{o}(V_+)) \oplus \Gamma(E^* \otimes \mathfrak{o}(V_-)).
\end{equation*}
Such a connection is thus determined by four first-order differential operators
\begin{align*}
D^+_- & \colon \Gamma(V_+) \to \Gamma(V_-^* \otimes V_+), \qquad D^-_+ \colon \Gamma(V_-) \to \Gamma(V_+^* \otimes V_-),\\
D^+_+ & \colon \Gamma(V_+) \to \Gamma(V_+^* \otimes V_+), \qquad D^-_- \colon \Gamma(V_-) \to \Gamma(V_-^* \otimes V_-),
\end{align*}
satisfying the Leibniz rule (formulated in terms of the projections $\pi_\pm$). The operators $D_{\pm}^{\pm}$ are called \emph{pure type}, whereas the operators $D^{\pm}_{\mp}$ are called \emph{mixed type}.  We furthermore say that the torsion of such a connection $D$ is of \emph{pure type} if
$$
T_D \in  \Gamma(\Lambda^3 V_+^* \oplus \Lambda^3 V_-^*).
$$
\end{defn}

The next lemma shows that the \emph{mixed-type operators} $D^\pm_\mp$ are fixed, if we vary a generalized connection $D$ inside $\cD(\GG)$ while preserving the torsion $T_D$.  And furthermore when the torsion is of pure-type then $D^\pm_\mp$ are uniquely determined by the generalized metric and the bracket on $E$.

\begin{lemma}\label{l:mixedfixed}
Let $D \in \cD(\GG)$ with torsion $T \in \Gamma(\Lambda^3 E^*)$. 
\begin{enumerate}

\item If $D' \in \cD(\GG)$ and $T_{D'} = T$, then $(D')^\pm_\mp = D^\pm_\mp$.

\item Furthermore, $T$ is of pure-type if and only if the mixed-type operators $D^\pm_\mp$ are
\begin{equation}\label{eq:mixedcanonical}
D_{a_-}b_+ = [a_-,b_+]_+, \qquad D_{a_+}b_- = [a_+,b_-]_-.
\end{equation}
\end{enumerate}
\end{lemma}
\begin{proof}
Given $a,b,c \in \Gamma(E)$, we have
\begin{equation}\label{eq:TDmixed}
T(a_-,b_+,c_+) = (D_{a_-}b_+ - [a_-,b_+],c_+),
\end{equation}
and therefore $(2)$ follows. To prove $(1)$, we subtract from \eqref{eq:TDmixed} the analogous expression for $T_{D'}(a_-,b_+,c_+)$.
\end{proof}

The second part of the previous lemma provides a weak analogue of the Koszul formula in Riemannian geometry. The canonical mixed-type operators \eqref{eq:mixedcanonical} were first considered in \cite{HitchinBrackets} and further studied in \cite{GualtieriBranes}, and relate to the following classical connections with skew-symmetric torsion.  

\begin{defn} \label{d:Bismutconn} Let $(M^n, g, H)$ be a Riemannian manifold 
and $H \in \Lambda^3T^*$.  The \emph{Bismut connections} associated to $(g,H)$ are defined via
\begin{align} \label{f:Bismutconn}
\IP{\N^{\pm}_X Y, Z} = \IP{\N_X Y, Z} {\pm} \tfrac{1}{2} H(X,Y,Z),
\end{align}
where $\nabla$ denotes the Levi-Civita connection of $g$. These are the unique metric-compatible connections with torsion $\pm H$.
\end{defn}

Our next goal is to show that the canonical mixed type operators in Lemma \ref{l:mixedfixed} are determined by the classical Bismut connections associated to the generalized metric via Proposition \ref{l:Gmetricabsgeom}, first observed in \cite[Theorem 2]{HitchinBrackets}.  Thus, the Bismut connection plays a fundamental role in generalized Riemannian geometry.  To see this, we will work with the explicit presentation of the Courant algebroid $E$ given by $\GG$. Recall that $\GG$ determines an isometry $F \colon T \oplus T^* \to E$ (Proposition \ref{l:Gmetricabsgeom}), and for any $X \in T$ we have isomorphisms $\sigma_{\pm}: T \to V_{\pm}$ given by (see Definition \ref{d:genmetricsections})
$$
F(X \pm g X) = \sigma_\pm X.
$$
We need the following basic identity.

\begin{lemma} \label{l:splitbracket} Given a generalized metric $\GG$ on $E$, consider the associated closed three-form $H \in \Lambda^3T^*$ and isometry $F \colon T \oplus T^* \to E$ in Proposition \ref{l:Gmetricabsgeom}. Then, one has the equality
\begin{gather} \label{f:genmetbracket}
F^{-1}[\sigma_- X, \sigma_+ Y] = [X,Y] + i_Y i_X H + L_X i_Y g + L_Y i_X g - d i_X 
i_Y g.
\end{gather}
\begin{proof} We directly compute, using (\ref{f:genmetricsections10}) and the definition of the twisted Courant bracket,
\begin{align*}
F^{-1}[\sigma_- X,\sigma_+ Y] =&\ [X - gX, Y + g Y]_H\\
=&\ [X,Y] + [X,i_Y g] - [ i_X g, Y] + 
i_Y i_X H\\
=&\ [X,Y] + L_X i_Y g - \tfrac{1}{2} d i_X i_Yg + L_Y i_X g  -
\tfrac{1}{2} d i_Y i_X g + i_Y i_X H\\
=&\ [X,Y] + \left( L_X i_Y g + L_Y i_X g - d i_X i_Y g \right) + 
i_Y i_X H.
\end{align*}
\end{proof}
\end{lemma}

\begin{prop} \label{p:HitchinBismut} 
Given a generalized metric $\GG$ on $E$, consider the associated closed three-form $H \in \Lambda^3T^*$ with associated Bismut connections $\N^{\pm}$. Then one has the equality
\begin{gather*}
\begin{split}
\N^{\pm}_X Y =&\ \pi \circ [\sigma_{\mp} X,\sigma_{\pm} Y]_{\pm}.
\end{split}
\end{gather*}
\begin{proof} Let $\mathcal N(X,Y) = \pi \circ [\sigma_- X,\sigma_+ Y]_+$.  We first 
compute
\begin{align*}
\NN(fX,Y) =&\ \pi \circ \pi_+ [f \sigma_- X,\sigma_+ Y]\\
=&\ \pi \circ \pi_+ \left( f [\sigma_- X,\sigma_+ Y] + Yf \sigma_-X - \IP{\sigma_- X,\sigma_+ Y} \DD f \right)\\
=&\ f \pi \circ \pi_+ [\sigma_- X, \sigma_+ Y]\\
=&\ f \NN(X,Y).
\end{align*}
Next we obtain
\begin{align*}
\NN(X,fY) =&\ \pi \circ \pi_+ [\sigma_- X,f \sigma_+ Y]\\
=&\ \pi \circ \pi_+ \left( f [\sigma_- X,\sigma_+ Y] +  ((\pi \sigma_- X) f) \sigma_+ Y - \IP{\sigma_- X,\sigma_+ Y} \DD 
f \right)\\
=&\ f \pi \circ \pi_+ [\sigma_- X, \sigma_+ Y] + (X f) \pi \circ \pi_+ \sigma_+ Y\\
=&\ f \NN(X,Y) + (X f) Y.
\end{align*}
This implies that $\NN$ does indeed define a linear connection on $T$.  We next 
show that this connection is metric compatible.  First observe using 
Proposition \ref{l:Gmetricabsgeom} that
\begin{align} \label{f:HB20}
g(\NN(X,Y), Z) = \IP{ \sigma_+\NN(X,Y), \sigma_+ Z} = \IP{ \pi_+ [\sigma_- X,\sigma_+ Y], \sigma_+ Z}.
\end{align}
Then we can compute, using Proposition \ref{l:Gmetricabsgeom} and also 
Property (5) of Definition \ref{d:CA},
\begin{align*}
\IP{\pi_+ [\sigma_- X,\sigma_+ Y],\sigma_+ Z} +&  \IP{\pi_+[\sigma_- X,\sigma_+ Z],\sigma_+Y}\\
=&\ \IP{[ \sigma_-X,\sigma_+ Y],\sigma_+ Z} + 
\IP{[\sigma_- X,\sigma_+ Z],\sigma_+Y}\\
=&\ \pi(\sigma_-X) \IP{\sigma_+Y,\sigma_+Z} - \IP{ \DD \IP{\sigma_-X,\sigma_+Z}, \sigma_+ Y}\\
& - \IP{\sigma_+ Z, \DD 
\IP{\sigma_- X,\sigma_+ Y}}\\
=&\ X \IP{\sigma_+ Y,\sigma_+ Z}\\
=&\ X g(Y,Z),
\end{align*}
which yields the claim comparing against (\ref{f:HB20}).  It remains to compute 
the torsion tensor.  First, arguing as in (\ref{f:HB20}) we have
\begin{gather} \label{f:HB30}
\begin{split}
T(X,Y,Z) =&\ g( \NN(X,Y) - \NN(Y,X) - [X,Y], Z)\\
=&\ \IP{ \pi_+ \left( [\sigma_- X,\sigma_+ Y] - [\sigma_- Y,\sigma_+ X] \right),\sigma_+Z} - g([X,Y],Z)\\
=&\ \IP{ [\sigma_- X,\sigma_+ Y] - [\sigma_- Y,\sigma_+ X],\sigma_+Z} - g([X,Y],Z).
\end{split}
\end{gather}
Now, using Lemma \ref{l:splitbracket}, and observing that the final three terms 
in (\ref{f:genmetbracket}) are symmetric, we obtain
\begin{gather} \label{f:HB40}
\begin{split}
\IP{[\sigma_- X,\sigma_+ Y] - [\sigma_- Y,\sigma_+ X],\sigma_+ Z} =&\ 2 \IP{ [X,Y] + i_Y i_X H , \sigma_+ Z}\\
=&\ 2 \IP{ \sigma_-[X,Y] + i_{[X,Y]} g + i_Y i_X H, \sigma_+Z}\\
=&\ g([X,Y],Z) + H(X,Y,Z).
\end{split}
\end{gather}
Using (\ref{f:HB40}) in (\ref{f:HB30}) yields the claim $T = H$, and so the 
result follows.
\end{proof}
\end{prop}

To proceed further, we must analyze the space of torsion-free $\GG$-compatible generalized connections, which we denote
$$
\cD^0(\GG) = \cD(\GG) \cap \cD^0.
$$

\begin{prop}\label{prop:existence}
Given a generalized metric $\GG$ on a Courant algebroid $E$, there exists a torsion-free generalized connection on $E$ compatible with $\GG$.
\end{prop}
\begin{proof}
We construct first a $\GG$-compatible generalized connection $D$ on $E$ which has torsion of pure-type. For this, we define the mixed-type operators $D^\pm_\mp$ by \eqref{eq:mixedcanonical}. To define the \emph{pure-type operators}, we choose arbitrary metric connections $\til{\nabla}^+$ on $V_+$ and $\til{\nabla}^-$ on $V_-$, and set
$$
D_{a_\pm}b_\pm := \til{\nabla}^\pm_{\pi(a_\pm)}b_\pm.
$$
Finally, we define our torsion-free generalized connections $D^0$ by explicitly removing the torsion, that is
$$
D^0 = D - \frac{1}{3}T_D,
$$
where we use the metric $\la\cdot,\cdot\ra$ to regard
$$
T_{D} \in \Gamma(V_+^* \otimes \mathfrak{o}(V_+)) \oplus \Gamma(V_-^* \otimes \mathfrak{o}(V_-)).
$$
Note that the pure-type condition on $T_D$ implies that $D^0$ is still $\GG$-compatible.
\end{proof}

Given a generalized metric, the torsion-free condition does not determine a compatible generalized connection uniquely. Unlike in Riemannian geometry, the \emph{generalized Levi-Civita connections} form an affine space, modeled on the pure-type mixed symmetric $3$-tensors
$$
\Sigma^+ \oplus \Sigma^-,
$$
where
$$
\Sigma^\pm =  \Gamma(V_\pm^{\otimes 3}) \cap \Sigma,
$$
and $\Sigma$ is as in \eqref{eq:sigma}. There are canonical splittings
\begin{equation}\label{eq:splittingyoungpure}
\Sigma^\pm = \Sigma_0^\pm \oplus \Gamma(V_\pm),
\end{equation}
where the first summand corresponds to `trace-free' elements, in analogy with \eqref{eq:splittingyoung}. To see this, fix $\{e_i^+\}$ an orthonormal local frame for $V_+$. Given $\chi^+ \in \Gamma(V_+^* \otimes \mathfrak{o}(V_+))$, set
$$
e_+ := \frac{1}{n -1}\sum_{i=1}^{n} \chi_{e_i^+} e_i^+.
$$
Then, on $V_+$ the splitting \eqref{eq:splittingyoungpure} corresponds to the following direct sum decomposition
\begin{equation}\label{eq:decompositionchipm}
\chi^+ = \chi^+_0 + \chi_+^{e_+},
\end{equation}
where $\chi_0^+$ is such that
$$
\sum_{i=1}^{n} (\chi_0^+)_{e_i^+} e_i^+ = 0,
$$
and
\begin{equation}\label{eq:chipm}
(\chi^{e_+}_+)_{a_+}b_+ = \la a_+,b_+ \ra e_+ - \la e_+,b_+ \ra a_+.
\end{equation}

\begin{ex}\label{ex:D0} 
Using the above discussion we can produce finally an explicit example of a torsion-free generalized connection compatible with a generalized metric. Let $\GG$ be a generalized metric on an exact Courant algebroid $E$ over a manifold $M$ of dimension $n$. The generalized metric induces an isomorphism $E \cong T \oplus T^*$, where the twisted bracket on $T \oplus T^*$ is given by a closed $3$-form $H$ in the \v Severa class of $E$ in $H^3(M,\mathbb{R})$. Via this isomorphism, the generalized metric takes the simple form:
$$
V_\pm = \{X \pm g(X): X \in T\},
$$
where $g$ is the Riemannian metric induced by the isomorphism $\pi_{|V_+} \colon V_+ \to T$.

Using Lemma \ref{l:mixedfixed} and Proposition \ref{p:HitchinBismut}, we see that the initial step of the construction is forced upon us, where we must use the Bismut connections $\N^{\pm}$ to define the mixed type operators.  As in the proof of Proposition \ref{prop:existence}, we must choose metric compatible connections to define the pure-type operators, and it is natural to still use the Bismut connections.  Doing so we obtain the \emph{Gualtieri-Bismut connection} $D^B$, introduced in \cite{GualtieriBranes}, with pure-type operators
\begin{gather*}
\begin{split}
\pi D^B_{\sigma_+X}\sigma_+Y &\ = \nabla^+_XY,\\
\pi D^B_{\sigma_- X}\sigma_-Y &\ = \nabla^-_XY.
\end{split}
\end{gather*}
By construction this connection is compatible with $\GG$, and has torsion $T_{D^B} = \pi_{|V_+}^*H + \pi_{|V_-}^*H$.  Following the proof of Proposition \ref{prop:existence}, we construct a torsion-free generalized connection $D^0 \in \cD^0(V_+)$ from $D^B$ by
$$
D^0 = D^B - \frac{1}{3}T_{D^B}.
$$
A computation shows that this has pure-type operators
\begin{gather*}
\begin{split}
\pi D^0_{\sigma_+X}\sigma_+Y & = \nabla^{+1/3}_XY,\\
\pi D^0_{\sigma_-X}\sigma_-Y & = \nabla^{-1/3}_XY,
\end{split}
\end{gather*}
where $\nabla^{\pm1/3}$ denotes the metric connection with skew-symmetric torsion (cf. \eqref{f:Bismutconn})
\begin{align*}
\IP{\N^{\pm1/3}_X Y, Z} = \IP{\N_X Y, Z} {\pm} \tfrac{1}{6} H(X,Y,Z).
\end{align*}
\end{ex}

To finish this section, we introduce the space of torsion-free metric connections compatible with a fixed divergence operator $\divop$, that is,
\begin{equation}\label{eq:DVdiv}
\cD^0(\GG,\divop) = \{D \in \cD(\GG) \; | \; T_D = 0, \; \divop_D = \divop\}.
\end{equation}
This space of connections plays an important role in the definition of the generalized Ricci tensor. By Lemma \ref{lem:weylfixed}, $\cD^0(\GG,\divop)$ is an affine space modeled on 
$$
\Sigma^+_0 \oplus \Sigma^-_0.
$$
In the next result we show that $\cD^0(\GG,\divop)$ is non-empty. Recall from Definition \ref{d:GRiemnniandiv} that a generalized metric has an associated Riemannian divergence $\divop^\GG$.

\begin{lemma}\label{lem:Ddiv}
Let $(\GG,\divop)$ be a pair given a generalized metric $\GG$ and a divergence operator $\divop$ on $E$. Denote $\divop^\GG - \divop = \langle e_+ ,  \cdot \rangle + \langle e_- , \cdot \rangle$, for $e_\pm \in V_\pm$. Then,
\begin{enumerate}

\item $D^0 \in \cD^0(\GG,\divop^\GG)$, where $D^0$ is defined as in Example \ref{ex:D0}.

\item $D \in \cD^0(\GG,\divop)$, where
\begin{equation}\label{eq:LCe}
D = D^0 + \frac{1}{n-1}\Big{(}\chi_+^{e_+} + \chi_-^{e_-}\Big{)},
\end{equation}
and $\chi_\pm^{e_\pm}$ are defined as in \eqref{eq:chipm}.
\end{enumerate}

\begin{proof} 
To prove $(1)$, consider the Levi-Civita connection $\nabla$ of the Riemannian metric $g$ induced by $\GG$ and the generalized connection $D' = \nabla \oplus \nabla^*$ as in Example \ref{ex:nablanabla*}. Then, a simple calculation shows that
$$
\divop_{D'}(e) = \tr \nabla X = \frac{L_X \mu^g}{\mu^g}.
$$
Applying now \cite[Proposition 3.6]{GF14} it follows that $D'$ and $D^0$ differ by a totally skew-symmetric element in $\Gamma(E \otimes \mathfrak{o}(E))$ and therefore $\divop_{D'} = \divop_{D^0}$.

As for $(2)$, we choose an orthogonal frame $\{e_i\}$ for $E$ with dual frame $\{\til e_i\}$ and calculate
\begin{align*}
\divop_D(e') & = \sum_i \langle  \til e_i, D_{e_i} \rangle \\
& = \divop_{D^0}(e') + \frac{1}{n-1}\sum_i \langle  \til e_i, (\chi_+^{e_+})_{e_i}e' \rangle + \frac{1}{n-1}\sum_i \langle  \til e_i, (\chi_-^{e_-})_{e_i}e' \rangle\\
& = \divop^\GG(e') + \frac{1}{n-1}\sum_i \langle  \til e_i, e_+ \rangle \langle \pi_+ e_i, e' \rangle - \langle  e_+,  e' \rangle \langle  \til e_i, \pi_+ e_i \rangle\\
& + \frac{1}{n-1}\sum_i \langle  \til e_i, e_- \rangle \langle \pi_- e_i, e' \rangle - \langle  e_-, e' \rangle \langle  \til e_i, \pi_- e_i \rangle\\
& = \divop^\GG(e') - \langle e_+, e' \rangle - \langle e_-, e' \rangle\\
& = \divop(e').
\end{align*}
The proof follows from the decomposition \eqref{eq:splittingyoungpure}, which shows that $D$ is torsion-free and metric compatible.
\end{proof}
\end{lemma}

\section{The classical Bismut connection} \label{s:CBC}

Proposition \ref{p:HitchinBismut} shows that the Bismut connection, as presented in Definition \ref{d:Bismutconn}, plays a fundamental role in generalized Riemannian geometry. In this section we pause for a moment in the development of the general theory, in order to present some basic features of this classical object.  It is interesting to observe that 
Bismut naturally came across this connection in the context of 
index theory problems in complex non-K\"ahler geometry.  We note here that 
historically this family of connections was considered by Yano \cite{Yano} in 
his study of manifolds with an almost product structure.  Remarkably, 
Cartan-Schouten \cite{CartanSchouten,CartanSchouten2} also encountered these 
connections in their study of Riemannian manifolds admitting an ``absolute 
parallelism'', i.e. a flat metric compatible connection. In particular they 
showed that, beyond classical flat Riemannian metrics, the only examples are 
semisimple Lie groups with bi-invariant metric and torsion determined 
canonically by the Lie algebra structure, as well as an exceptional example on 
$S^7$.  In the Lie group case the torsion is skew-symmetric and closed, thus 
fitting naturally into our setting, whereas the exceptional $S^7$ example has nonclosed torsion. We will go back to this interesting class of examples in \S \ref{s:EHF}. 

Our first result gives a formula for the curvatures of the Bismut connection in terms of the Riemann curvature tensor. Throughout this section $\nabla$ will denote the Levi-Civita connection of a fixed Riemannian metric $g$. Given $H \in \Lambda^3T^*$, we will denote by $\nabla^+$ the associated Bismut connection as in Definition \ref{d:Bismutconn}. We will denote by $\Rm^+$, $\Rc^+$, and $R^+$, the curvature tensor, the Ricci tensor, and the scalar curvature of $\nabla^+$, respectively. Similarly, we will use the notation $\Rm$, $\Rc$, and $R$ for the curvatures associated to the metric via $\nabla$. We denote by $d^*$ the adjoint of the exterior differential acting on differential forms with respect to the metric $g$.

\begin{prop} \label{p:Bismutcurvature} Let $(M^n, g, H)$ be a Riemannian 
manifold with $H \in \Lambda^3T^*$, 
$dH = 0$.  Then
\begin{gather} \label{Bismutcurvature}
\begin{split}
\Rm^+ (X,Y,Z,W) =&\ \Rm(X,Y,Z,W)\\
&\ + \tfrac{1}{2} \N_X H(Y,Z,W) - 
\tfrac{1}{2} \N_Y H(X,Z,W)\\
&\ - \tfrac{1}{4} \IP{ H(X,W), H(Y,Z)} + \tfrac{1}{4} \IP{ H(Y,W), H(X,Z)},\\
\Rc^+ =&\ \Rc - \tfrac{1}{4} H^2 - \tfrac{1}{2} d^* H,\\
R^+ =&\ R - \tfrac{1}{4} \brs{H}^2,
\end{split}
\end{gather}
where
\begin{gather} \label{f:H^2def}
H^2(X,Y) := \IP{ X \lrcorner H, Y \lrcorner H} = \sum_{i,j} H(X,e_i,e_j) 
H(Y,e_i,e_j).
\end{gather}
\begin{proof} It suffices to establish the first formula for local normal 
coordinate vectors, i.e. $\N_X Y(p) = 0$ and $[X,Y] = 0$ etc.  Using the 
definition of the Bismut connection we thus directly compute
\begin{align*}
&\Rm^+(X,Y,Z,W)\\
&\ \quad = \IP{ \N^+_X \N^+_Y Z - \N^+_Y \N^+_X Z - \N^+_{[X,Y]} Z, 
W}\\
&\ \quad = \IP{ \N_X (\N^+_Y Z),W} + \tfrac{1}{2} H(X,\N^+_Y Z, W)\\
&\ \quad \qquad - \IP{ \N_Y(\N^+_X 
Z),W} - \tfrac{1}{2} H(Y,\N^+_X Z,W)\\
&\ \quad = \IP{ \N_X (\N_Y Z + \tfrac{1}{2} g^{-1} H(Y,Z,\cdot)), W} + \tfrac{1}{2} 
H(X, \N_Y Z + \tfrac{1}{2} g^{-1} H(Y,Z,\cdot),W)\\
&\ \quad \qquad - \IP{ \N_Y (\N_X Z + \tfrac{1}{2} g^{-1} H(X,Z,\cdot)) W} - \tfrac{1}{2} 
H(Y, \N_X Z + \tfrac{1}{2} g^{-1} H(X,Z,\cdot),W)\\
&\ \quad = \Rm(X,Y,Z,W) + \tfrac{1}{2} \N_X H(Y,Z,W) - \tfrac{1}{2} \N_Y H(X,Z,W)\\
&\ \quad \qquad - \tfrac{1}{4} \IP{ H(X,W), H(Y,Z)} + \tfrac{1}{4} \IP{ H(Y,W), H(X,Z)},
\end{align*}
as required.
Taking the trace of this with respect to the first and fourth indices, we obtain
\begin{align*}
\Rc^+(Y,Z) =&\ \Rm^+ (e_i, Y, Z, e_i)\\
=&\ \Rm(e_i, Y, Z, e_i) + \tfrac{1}{2} \N_{e_i} H(Y,Z,e_i) - \tfrac{1}{2} 
\N_Y H(e_i,Z,e_i)\\
&\ - \tfrac{1}{4} \IP{ H(e_i,e_i), H(Y,Z)} + \tfrac{1}{4} \IP{ H(Y,e_i), 
H(e_i,Z)}\\
=&\ \Rc(Y,Z) - \tfrac{1}{4} H^2(Y,Z) - \tfrac{1}{2} d^* H(Y,Z),
\end{align*}
as required.  A final trace, using that the trace of $d^* H$ is zero, yields 
the formula for the scalar curvature.
\end{proof}
\end{prop}

The classical Bianchi identities in Riemannian geometry reflect the 
diffeomorphism invariance of the curvature tensor.  More general forms of these 
identities hold for connections with torsion as well (cf. \cite[Theorem 5.3]{KN1}), and play a fundamental 
role in the study of the generalized Ricci flow.  We begin with an identity on 
the divergence operator acting on the tensor $H^2$, and then prove the Bianchi 
identity for connections with torsion.

\begin{lemma} \label{l:divHlemma}  Given $(M^n, g)$ a Riemannian manifold and
$H \in \Lambda^3 T^*$, $dH = 0$, one has
\begin{align*}
\left(\divg H^2 \right)_i =&\ \tfrac{1}{6} \N_i \brs{H}^2 - \left(d^* H
\right)^{mn} H_{imn}.
\end{align*}
\begin{proof}
We choose local 
normal coordinates for $g$ at some point and 
compute, suppressing some appearances of the metric tensor and using that $H$ 
is 
closed
\begin{align*}
\left( \divg H^2 \right)_j =&\ \N_i H^2_{ij}\\
=&\ \N_i \left( H_{imn} H_{jmn} \right)\\
=&\ - \left(d^* H \right)_{mn} H_{jmn} + H_{imn} \N_i H_{jmn}\\
=&\ - \left( d^* H \right)_{mn} H_{jmn} + H_{imn} \left( \N_j H_{imn}
- \N_m H_{ijn} + \N_n H_{ijm} \right)\\
=&\ - \left( d^* H \right)_{mn} H_{jmn} + \tfrac{1}{2} \N_j \brs{H}^2 - 2
H_{imn} \N_i H_{jmn},
\end{align*}
where in the last line we rearranged indices.  Comparing the third and fifth 
lines above yields that $H_{imn} \N_i H_{jmn} = \tfrac{1}{6} \N_j \brs{H}^2$, 
and so the lemma follows.
\end{proof}
\end{lemma}

\begin{prop}\label{p:Bianchi} Given $(M^n, g)$ a Riemannian manifold and
$H \in \Lambda^3 T^*$, $d H = 0$, one has
\begin{align} \label{f:Bianchi1}
\sum_{\gs(X,Y,Z)} R^+(X,Y)Z =&\ \sum_{\gs(X,Y,Z)} \left\{ H(H(X,Y),Z) + (\N^+_X 
H)(Y,Z) \right\},
\end{align}
and
\begin{align} \label{f:Bianchi2}
 \sum_{\gs(X,Y,Z)} \left\{ (\N^+_X R)(Y,Z) + R^+(H(X,Y),Z) \right\} = 0,
\end{align}
and
\begin{gather} \label{f:Bianchi3}
\begin{split}
0 =&\ \sum_{\gs(X,Y,Z)} \left\{ (\N^+_X H)(Y,Z,U) + 2 g (H(X,Y),H(Z,U)) \right\}\\
&\ \qquad - (\N^+_U H)(X,Y,Z),
\end{split}
\end{gather}
and
\begin{gather} \label{f:Bianchi4}
\begin{split}
R^+(X,Y,Z,U) - R^+(Z,U,X,Y) =&\ - \tfrac{1}{2} \sum_{\gs(X,Y,Z,U)} \N^+_U H(X,Y,Z).
\end{split}
\end{gather}

\begin{proof} The first two formulas are \textbf{exercises}. The third formula follows from
\begin{gather*}
\begin{split}
dH(X,Y,Z,U) =&\ \sum_{\gs(X,Y,Z)} \left\{ (\N^+_X H)(Y,Z,U) + 2 g (H(X,Y),H(Z,U)) \right\} - (\N^+_U H)(X,Y,Z),
\end{split}
\end{gather*}
and $d H = 0$. To show the fourth we first rewrite (\ref{f:Bianchi1}) using (\ref{f:Bianchi3})
\begin{align*}
0 =&\ \sum_{\gs(X,Y,Z)} \left\{ R^+(X,Y,Z,U) -  g(H(H(X,Y),Z),U) - \N^+_X H (Y,Z,U) \right\}\\
=&\ \sum_{\gs(X,Y,Z)} \left\{ R^+(X,Y,Z,U) -  g(H(X,Y),H(Z,U)) - \N^+_X H (Y,Z,U) \right\}\\
=&\ \sum_{\gs(X,Y,Z)} \left\{ R^+(X,Y,Z,U) +  g(H(X,Y),H(Z,U)) \right\} - \N^+_U H(X,Y,Z)
\end{align*}
Summing this over cyclic permutation of the entries yields
\begin{align*}
0 =&\ R^+(X,Y,Z,U) + R^+ (Z,X,Y,U) + R^+(Y,Z,X,U) - \N^+_U H(X,Y,Z)\\
&\ + R^+(U,X,Y,Z) + R^+(Y,U,X,Z) + R^+(X,Y,U,Z) - \N^+_Z H(U,X,Y)\\
&\ + R^+(Z,U,X,Y) + R^+(X,Z,U,Y) + R^+(U,X,Z,Y) - \N^+_Y H(Z,U,X)\\
&\ + R^+(Y,Z,U,X) + R^+(U,Y,Z,X) + R^+(Z,U,Y,X) - \N^+_X H(Y,Z,U)\\
&\ +\sum_{\gs(X,Y,Z)} g(H(X,Y), H(Z,U)) + \sum_{\gs(U,X,Y)} g(H(U,X), H(Y,Z))\\
&\ + \sum_{\gs(Z,U,X)} g(H(Z,U), H(X,Y)) + \sum_{\gs(Y,Z,U)} g(H(Y,Z), H(U,X)).
\end{align*}
One observes that the first and third columns of curvature terms collectively cancel, as do the terms of type $H^2$, yielding
\begin{align*}
0 =&\ 2 R^+(X,Z,U,Y) - 2 R^+ (U,Y,X,Z)\\
&\ - \N^+_U H(X,Y,Z) - \N^+_Z H(U,X,Y) - \N^+_Y H(Z,U,X) - \N^+_X H(Y,Z,U).
\end{align*}
The result follows.
\end{proof}
\end{prop}

Our next result (cf. \cite{Bismut} Theorem 1.6) gives an analogue of the first Bianchi identity for Levi-Civita connection, which relates the curvature tensors of the Bismut connections $\nabla^+$ and $\nabla^-$.

\begin{prop} \label{p:Bismutpair} Let $(M^n, g)$ be a Riemannian manifold and fix $H \in \Lambda^3 T^*$, $d H = 0$.  Then
\begin{align*}
R^+(X,Y,Z,W) = R^-(Z,W,X,Y).
\end{align*}
\begin{proof} Using the formula for the Bismut curvature in Proposition \ref{p:Bismutcurvature} separately for $\N^{\pm}$ and the symmetry of the Riemann curvature tensor we see
\begin{align*}
R^+(X,Y,Z,W) & - R^-(Z,W,X,Y)\\
=&\ \tfrac{1}{2} \left( \N_X H(Y,Z,W) - \N_Y H(X,Z,W) + \N_Z H(W,X,Y) - \N_W H(Z,X,Y) \right)\\
&\ + \tfrac{1}{4} \left( - \IP{H(X,W),H(Y,Z)} + \IP{H(Y,W),H(X,Z)} \right.\\
&\ \qquad \left. + \IP{H(Z,Y),H(W,X)} - \IP{H(W,Y),H(Z,X)} \right)\\
=&\ \tfrac{1}{2} dH(X,Y,Z,W)\\
=&\ 0,
\end{align*}
as required.
\end{proof}
\end{prop}

\section{Curvature and the First Bianchi Identity}\label{sec:curvature}
We turn next to the study of curvature quantities in generalized geometry. We fix an exact Courant algebroid $E$ over a smooth manifold $M$.

\begin{defn} \label{d:gencurvature} 
Let $\GG$ be a generalized metric on $E$, and $D$ a metric compatible generalized connection. Associated to $D$ there are the two \emph{generalized curvature operators}
\begin{align*}
\mathcal R_D^\pm &\in \Gamma(V_\pm^* \otimes V_\mp^* \otimes \mathfrak{o}(V_\pm)),
\end{align*}
defined by
\begin{equation*}
\begin{split}
\mathcal R_D^+(a_+,b_-)c_+ = D_{a_+}D_{b_-}c_+ - D_{b_-}D_{a_+} c_+ - D_{[a_+,b_-]}c_+,\\
\mathcal R_D^-(a_-,b_+)c_- = D_{a_-}D_{b_+}c_- - D_{b_+}D_{a_-} c_- - D_{[a_-,b_+]}c_-.
\end{split}
\end{equation*}
\end{defn}

In general, the curvatures of a torsion-free generalized connection are not an invariant of the generalized metric $\GG$, but rather depend on $D$ via the choice of pure-type operators.  We illustrate this with the following example.

\begin{ex}\label{ex:DR+}
Consider the torsion-free metric compatible connection $D^0$ constructed in Example \ref{ex:D0}. We choose a one-form $\varphi$ on $M$ and use this to deform $D^0$ to a different torsion-free metric connection, via deformation of the associated divergence operator (see Definition \ref{def:divergence}). We set $\varphi_\pm = \pi_\pm\varphi$ using the orthogonal projections $\pi_\pm \colon E \to V_\pm$ and define
\begin{equation*}
D = D^0 + \frac{1}{n-1}\Big{(}\chi_+^{\varphi_+} + \chi_-^{\varphi_-}\Big{)},
\end{equation*}
where $\chi_\pm^{\varphi_\pm}$ are defined as in \eqref{eq:chipm}. By Lemma \ref{lem:Ddiv}, $D$ is torsion-free and metric compatible, and the associated divergence is 
$$
\divop_D(e) = \divop_{D^0}(e) - \la \varphi, e \ra = \frac{L_X \mu^g}{\mu^g} - \tfrac{1}{2}\varphi(X)
$$
for $e = X + \xi \in T \oplus T^*$. Being torsion-free, by Lemma \ref{l:mixedfixed} the generalized connection $D$ is completely determined by the corresponding pure-type operators, given in this case by
\begin{equation}\label{eq:LCvarphipure}
\begin{split}
\pi D_{\sigma_+ X}\sigma_+ Y & = \nabla^{+1/3}_XY + \frac{1}{2(n-1)}(g(X,Y)g^{-1}\varphi - \varphi(Y)X),\\
\pi D_{\sigma_+ X}\sigma_-Y & = \nabla^{-1/3}_XY + \frac{1}{2(n-1)}(g(X,Y)g^{-1}\varphi - \varphi(Y)X).
\end{split}
\end{equation}
We are now ready to calculate the curvature of $D$ following \cite{GF14} (e.g. for $\mathcal R_D^+$). First, we have
\begin{align*}
D_{\sigma_+ X}D_{\sigma_- Y}\sigma_+ Z & = \sigma_+\Big{(}\nabla_X^{1/3}\nabla^+_YZ + \frac{1}{2(n-1)}g(X,\nabla^+_YZ)g^{-1}\varphi - \frac{1}{2(n-1)}\varphi(\nabla^+_YZ ) X\Big{)},\\
D_{\sigma_- Y}D_{\sigma_+ X}\sigma_+ Z & = \sigma_+\Bigg{(}\nabla_Y^+\nabla^{1/3}_XZ - \frac{1}{2(n-1)}\varphi(Z)\nabla^+_YX + \frac{1}{2(n-1)}g(X,Z)g^{-1}\nabla^+_Y\varphi \\
& - \frac{1}{2(n-1)}i_Yd(\varphi(Z))X + \frac{1}{2(n-1)}i_Yd(g(X,Z))g^{-1}\varphi\Bigg{)}.
\end{align*}
Using the equality
\begin{align*}
[\sigma_+ X,\sigma_- Y] & = [X,Y] - g(\nabla_XY + \nabla_YX,\cdot) + H(X,Y,\cdot)
\end{align*}
we also obtain
\begin{align*}
D_{[\sigma_+ X,\sigma_- Y]}\sigma_+ Z & = \sigma_+\Bigg{(}\nabla_{[X,Y]}Z + \frac{1}{6}g^{-1}H(2[X,Y] + \nabla_XY + \nabla_YX,Z,\cdot)\\
&  - \frac{1}{6}g^{-1}H(g^{-1}H(X,Y,\cdot),Z,\cdot)\\
& - \frac{1}{4(n-1)}\varphi(Z)([X,Y] - \nabla_Y^gX - \nabla_XY + g^{-1}H(X,Y,\cdot))\\
& + \frac{1}{4(n-1)}(g(Z,[X,Y] - \nabla_Y X - \nabla_X Y) + H(X,Y,Z))\varphi\Bigg{)}.
\end{align*}
Finally, the identity
\begin{align*}
(\nabla_X H)(Y,Z,\cdot) & =  \nabla_X(H(Y,Z,\cdot)) - H(\nabla_X Y,Z,\cdot) - H(Y,\nabla Z,\cdot)
\end{align*}
leads us to the following expression for the curvature
\begin{gather}\label{eq:gencurvatureexp}
\begin{split}
\pi R_{D}^+& (\sigma_+X,\sigma_-Y)\sigma_+Z\\
=&\ \til R(X,Y)Z -\frac{1}{2(n-1)}(g(X,Z)g^{-1}\nabla^{+}_Y\varphi - i_Z(\nabla^+_Y\varphi)X)
\end{split}
\end{gather}
where
\begin{align*}
\til R(X,Y)Z & = R(X,Y)Z + g^{-1}\Bigg{(}\frac{1}{2}(\nabla_XH)(Y,Z,\cdot) - \frac{1}{6}(\nabla_YH)(X,Z,\cdot)\\
&\quad + \frac{1}{12}H(X,g^{-1}H(Y,Z,\cdot),\cdot) - \frac{1}{12}H(Y,g^{-1}H(X,Z,\cdot),\cdot)\\
&\quad - \frac{1}{6}H(Z,g^{-1}H(X,Y,\cdot),\cdot)\Bigg{)}.
\end{align*}
The tensor $\til R$ is a hybrid of the second covariant derivatives for $\nabla^+,\nabla^{1/3}$, and $\nabla^-$, and can be identified with the generalized curvature of $D^0$. More invariantly, it can be written as
\begin{equation*}
\begin{split}
\til R(X,Y)Z & = \nabla^{1/3}_X\nabla^+_YZ - \nabla^+_Y\nabla^{1/3}_XZ + \nabla^{1/3}_{\nabla^+_YX}Z - \nabla^+_{\nabla^-_XY}Z.
\end{split}
\end{equation*}
From the previous formulae, it is transparent that the generalized curvature operators $\mathcal R_D^\pm$ depend on the choice of torsion-free compatible connection $D$.
\end{ex}

Our next goal is to define curvature quantities which only depend on the generalized metric $\GG$. Before we address this question in \S \ref{subsec:Ricci}, we prove an algebraic property of the curvatures of a torsion-free generalized connection, that we will use later in Lemma \ref{lem:Ricci}.

\begin{prop}[First Bianchi identity]\label{propo:bianchi}
Let $D$ be a torsion-free generalized connection compatible with $V_+$. Then, for any $e_\pm, a_\pm, b_\pm, c_\pm \in \Gamma(V_\pm)$ we have
\begin{equation}\label{eq:bianchi}
\sum_{\gs(a,b,c)} \la \mathcal R_D^\pm(a_\pm,e_\mp)b_\pm,c_\pm\ra  = 0.
\end{equation}
\end{prop}
\begin{proof}
We argue for $\mathcal R^+_D$, as the other case is symmetric. We decompose
$$
\sum_{\gs(a,b,c)} \la \mathcal R_D^+(a_+,e_-)b_+,c_+\ra = I + J + K,
$$
where
\begin{align*}
I & = \sum_{\gs(a,b,c)} \la D_{a_+}([e_-,b_+]_+),c_+\ra,\\
J & = - \sum_{\gs(a,b,c)} \la [e_-,D_{a_+} b_+],c_+\ra,\\
K & = - \sum_{\gs(a,b,c)} \la D_{[a_+,e_-]}b_+,c_+\ra.
\end{align*}
Using that $T_D = 0$ and formula \eqref{eq:torsionG} we have 
\begin{align*}
K =&\ \sum_{\gs(a,b,c)}  - \la D_{b_+}[a_+,e_-],c_+\ra - \la [[a_+,e_-],b_+],c_+\ra + \la D_{c_+}[a_+,e_-],b_+\ra\\
=&\ \sum_{\gs(a,b,c)}  \la[a_+,e_-],D_{b_+} c_+ - D_{c_+} b_+ \ra - \la [[a_+,e_-],b_+],c_+\ra\\
&\ \qquad  - \pi(b_+)(\la[a_+,e_-],c_+\ra) + \pi(c_+)(\la [a_+,e_-],b_+\ra),
\end{align*}
where in the last equality we used the compatibility between $D$ and $\la\cdot,\cdot\ra$. Using again this last property, we also have
\begin{align*}
I & = \sum_{\gs(a,b,c)} \left( \pi(a_+)(\la[e_-,b_+],c_+\ra) - \la[e_-,b_+],D_{a_+}c_+\ra \right),\\
J & = \sum_{\gs(a,b,c)} \left(- \pi(e_-)(\la[a_+,b_+],c_+\ra) + \la[e_-,c_+],D_{a_+}b_+\ra \right),
\end{align*}
where in the last equality we used property (C4) of the bracket (see Definition \ref{d:CA})). From this, using property (C5) of the bracket, which implies $[e_-,a_+] = - [a_+,e_-]$, we obtain
\begin{align*}
I + J + K & =  - \la [[a_+,e_-],b_+], c_+\ra + \la [[b_+,e_-],a_+], c_+\ra - \la [[c_+,e_-],a_+], b_+\ra\\
& + \pi(b_+)(\la [c_+,e_-],a_+\ra) - \pi(a_+)(\la [c_+,e_-],b_+\ra) - \pi(c_+)(\la [b_+,e_-],a_+\ra)\\
& - \pi(e_-)(T_D(a_+,b_+,c_+)) - \pi(e_-)(\la [a_+,b_+],c_+\ra).
\end{align*}
Finally, using again (C5) we have an equality
$$
\pi(c_+)(\la [e_-,b_+],a_+\ra) = \la c_+, [a_+,[e_-,b_+]] + [[e_-,b_+],a_+]\ra,
$$
and from $T_D = 0$ we conclude
$$
I + J + K = \la [[e_-,a_+],b_+] + [a_+,[e_-,b_+]] - [e_-,[a_+,b_+]], c_+\ra,
$$
which vanishes by the Jacobi identity (property (C1)).
\end{proof}

The \emph{algebraic Bianchi identity} \eqref{eq:bianchi} for torsion-free generalized connections is a remarkable property, as their construction typically involves standard connections with skew-torsion in the tangent bundle of $M$. Recall from Proposition \ref{p:Bianchi} that a torsionful connection on a manifold satisfies complicated Bianchi identities which involve first-order derivatives of its torsion, as well as quadratic terms in $H$. To illustrate this, we compare now the generalized curvature operator in Example \ref{ex:DR+} with the curvature of the torsionful generalized connection $D^B$ in Example \ref{ex:D0}.

\begin{prop}\label{prop:genCurvature} 
Given $\GG$ a generalized metric on $E$, let $D^B$ denote the associated Gualtieri-Bismut connection in Example \ref{ex:D0}. One has
$$
\pi \mathcal R_{D^B}^\pm(\sigma_\pm X,\sigma_\mp Y)\sigma_\pm Z \cong \Rm^\pm(X,Y)Z,
$$
where $\Rm^\pm$ denotes the curvature of the standard Bismut connection in \eqref{Bismutcurvature}.
\begin{proof}
The proof follows as an \textbf{exercise} from Example \ref{ex:D0}.
\end{proof}
\end{prop}

\section{Generalized Ricci curvature}\label{subsec:Ricci}

Different approaches to the definition of the generalized Ricci tensor and scalar curvature can be found in \cite{boulanger2015toric,CSW,GF19,Goto_2020,jurco2016courant,SeveraValach1,SeveraValach2} (see also references therein). In \cite{CSW,GF14,Jur_o_2018,SeveraValach2} the authors gave a treatment of generalized connections and curvature and their relationship to supergravity theories.  To proceed with our approach to the Ricci tensor in generalized geometry following \cite{GF19}, we introduce next the definition of Ricci tensor for a metric-compatible generalized connection.  

\begin{defn}\label{def:RicciD}
Let $\GG$ be a generalized metric on $E$. The \emph{Ricci tensors} $\gRc_D^\pm \in \Gamma(V_\mp^* \otimes V_\pm^*)$ of a generalized metric-compatible connection $D \in \cD(\GG)$ are defined by
\begin{equation*}
\gRc_D^\pm(a_\mp,b_\pm) = \tr (e_\pm \to \mathcal R_D^\pm(e_\pm,a_\mp)b_\pm),
\end{equation*}
for $e_\pm \in \Gamma(V_\pm)$.
\end{defn}

Our next result investigates the variation of the Ricci tensors when we deform a torsion-free generalized connection in $\cD(\GG)$, while preserving the pure-type condition of the torsion and the divergence (see Definition \ref{def:divergenceop}). 

\begin{prop}\label{propo:Riccitorsion}
Let $(\GG,\divop)$ be a pair given by a generalized metric and a divergence operator on $E$. Let $D, D' \in \cD(\GG)$ with divergence $\divop_D = \divop = \divop_{D'}$. If $T_D = 0$ and $T_{D'}$ is of pure-type then $\gRc_D^\pm = \gRc^\pm_{D'}$. In particular, $\gRc_D^\pm = \gRc^\pm_{D'}$ for any pair $D,D' \in \cD^0(\GG,\divop)$ (see \eqref{eq:DVdiv}).
\end{prop}
\begin{proof}
Without loss of generality we will argue for $\gRc^+_D$, since the case $\mathcal R_D^-$ is analogous. Setting $\chi = D' - D$, the condition $\divop_D = \divop_{D'}$ combined with $T_{D'}$ being of pure-type imply that
$$
\chi = \chi^+ + \chi^-,
$$
where (see \eqref{eq:splittingyoungpure})
$$
\chi^\pm = \chi^\pm_0 + \sigma^\pm \in \Sigma_0^\pm \oplus \Lambda^3(V_\pm).
$$
Let $\{e_i^+\}_{i=1}^n$ an orthonormal frame for $V_+$. Given $e_\pm \in \Gamma(V_\pm)$ we calculate
\begin{equation}\label{eq:RicDD}
\begin{split}
\gRc_{D'}^+(a_-,b_+) & = \sum_{i=1}^{n} \la  e_i^+, R_{D'}^+(e_i^+,a_-)b_+\ra\\
& = \gRc_{D}^+(a_-,b_+) + \sum_{i=1}^{n} \la e_i^+, \chi^+_{e_i^+}(D_{a_-}b_+) - D_{a_-} (\chi^+_{e_i^+}b_+) - \chi^+_{[e_i^+,a_-]_+}b_+\ra\\
& = \gRc_{D}^+(a_-,b_+) + \sum_{i=1}^{n} \Big{(} - \la \chi^+_{e_i^+} e_i^+,D_{a_-}b_+\ra + \la \chi^+_{e_i^+}  b_+,D_{a_-} e_i^+\ra\\
& \qquad + \iota_{\pi(a_-)} d\la \chi^+_{e_i^+} e_i^+,b_+\ra - \la \chi^+_{[a_-,e_i^+]_+}  e_i^+,b_+\ra \Big{)}\\
& = \gRc_{D}^+(a_-,b_+) + \sum_{i=1}^{n} \Big{(} \la \chi^+_{e_i^+} b_+,[a_-, e_i^+]\ra  - \la \chi^+_{[a_-,e_i^+]_+}  e_i^+,b_+\ra \Big{)},
\\
& = \gRc_{D}^+(a_-,b_+) + \sum_{i=1}^{n} \Big{(} \la (\chi^+_0)_{e_i^+} b_+,[a_-, e_i^+]_+\ra  - \la (\chi^+_0)_{[a_-,e_i^+]_+} e_i^+,b_+\ra \Big{)}
\end{split}
\end{equation}
where have used $\sum_{i=1}^{n} \chi^+_{e_i^+}e_i^+ = 0$, Lemma \ref{l:mixedfixed}, and that $\sigma^\pm \in \Lambda^3 V_\pm$. To conclude, given $x \in M$ and $a_- \in V_{-|x}$, using the previous formula we shall prove that 
$$
\gRc_{D'}^+(a_-,b_+) = \gRc_{D}^+(a_-,b_+),
$$
and hence the statement will follow since $\gRc_{D'}^+$ and $\gRc_{D}^+$ are bilinear. For this, we construct a special orthonormal frame in order to evaluate formula \eqref{eq:RicDD}. Choose an orthonormal frame $\{e_{i}^+\}$ around $x$ for $V_+$ which satisfies $D_{a_-}e_{i}^+ = 0$ at $x$. This frame can constructed by smooth extension of the parallel transport for the connection $D^+_-$ along a curve starting at $x$ with initial velocity $\pi(a_-)$. The proof now follows from the previous formula, using that $[a_-,e_{i}^+]_+ = D_{a_-}e_{i}^+$. The last part of the statement follows from Lemma \ref{lem:weylfixed}.
\end{proof}

\begin{defn}\label{def:Ricci}
Let $(\GG,\divop)$ be a generalized metric $\GG$ and a divergence operator $\divop$ on the exact Courant algebroid $E$. The \emph{Ricci tensors} 
$$
\gRc^\pm \in \Gamma(V_\mp^* \otimes V_\pm^*)
$$ 
of the pair $(\GG,\divop)$ are defined by $\gRc^\pm = \gRc_D^\pm$ for any choice of generalized connection $D \in \cD(\GG,\divop)$.  This is well-defined by Proposition \ref{propo:Riccitorsion}.
\end{defn}

In the next result we calculate the variation of the Ricci tensors $\gRc^\pm$ when we fix the generalized metric $\GG$ and vary the divergence operator $\divop$. Recall from Definition \ref{d:GRiemnniandiv} that a generalized metric has an associated Riemannian divergence $\divop^\GG$.

\begin{lemma}\label{lem:Riccivardiv}
Let $(\GG,\divop)$ be a pair as in Definition \ref{def:Ricci}. Denote $\divop^\GG - \divop = \langle e , \cdot \rangle$, for $e \in \Gamma(E)$. Then, setting $e_\pm = \pi_\pm e$, we have
\begin{equation*}
\begin{split}
\gRc^+(\GG,\divop)(a_-,b_+) & = \gRc^+(\GG,\divop^\GG)(a_-,b_+) - \la [e_+,a_-],b_+\ra,\\ 
\gRc^-(\GG,\divop)(a_+,b_-) & = \gRc^-(\GG,\divop^\GG)(a_+,b_-) - \la [e_-,a_+],b_-\ra.
\end{split}
\end{equation*}
\begin{proof} 
We argue for $\gRc^+$, as the other case is symmetric. Let $D^0 \in \cD(\GG,\divop^\GG)$ and $D \in \cD(\GG,\divop)$ as in Lemma \ref{lem:Ddiv}. Let $\{e_i^+\}_{i=1}^n$ be an orthonormal frame for $V_+$. Then,
\begin{equation*}
\begin{split}
\gRc^+& (\GG,\divop)(a_-,b_+)\\
=&\ \gRc^+_D(a_-,b_+)\\
%=&\ \sum_{i=1}^{n} \la  e_i^+, \mathcal R_{D}^+(e_i^+,a_-)b_+\ra\\
=&\ \gRc_{D^0}^+(a_-,b_+)\\
& \qquad + \frac{1}{n-1}\sum_{i=1}^{n} \la e_i^+, (\chi_+^{e_+})_{e_i^+}(D^0_{a_-}b_+) - D^0_{a_-} ((\chi_+^{e_+})_{e_i^+}b_+) - (\chi_+^{e_+})_{[e_i^+,a_-]_+}b_+\ra\\
=&\ \gRc^+(\GG,\divop^\GG)(a_-,b_+) + \frac{1}{n-1}\sum_{i=1}^{n} \Big{(} - \la \chi^+_{e_i^+} e_i^+,[a_-,b_+]_+\ra \\
& \qquad + \la (\chi_+^{e_+})_{e_i^+}  b_+,[a_-,e_i^+]_+\ra + \iota_{\pi(a_-)} d\la (\chi_+^{e_+})_{e_i^+} e_i^+,b_+\ra + \la (\chi_+^{e_+})_{[a_-,e_i^+]_+}  e_i^+,b_+\ra \Big{)}\\
=&\ \gRc^+(\GG,\divop^\GG)(a_-,b_+) - \langle e_+,[a_-,b_+] \rangle  + \iota_{\pi(a_-)} d\la e_+,b_+\ra + C\\
%=&\ \gRc^+(\GG,\divop^\GG)(a_-,b_+) - \langle e_+,[a_-,b_+] \rangle + \la [a_-,e_+],b_+\ra + \la e_+,[a_-,b_+]\ra + C\\
=&\ \gRc^+(\GG,\divop^\GG)(a_-,b_+) + \la [a_-,e_+],b_+\ra + C,
\end{split}
\end{equation*}
where
\begin{align*}
C & = \frac{1}{n-1}\sum_{i=1}^{n} \Big{(} \la (\chi_+^{e_+})_{e_i^+}  b_+,[a_-,e_i^+]_+\ra  + \la (\chi_+^{e_+})_{[a_-,e_i^+]_+}  e_i^+,b_+\ra \Big{)}.
\end{align*}
Arguing as in the proof of Proposition \ref{propo:Riccitorsion}, we can now choose a point $x \in M$ and $\{e_{i}^+\}$ which satisfies $[a_-,e_i^+]_+ = 0$ at $x$. Thus, $C = 0$ and the statement follows from axiom $(5)$ in Definition \ref{d:CA}.
\end{proof}
\end{lemma}

Our next result provides an explicit formula for the Ricci tensors $\gRc^\pm$ associated to a pair $(\GG,\divop)$. We follow the notation in Definition \ref{d:GRiemnniandiv} and Lemma \ref{lem:Riccivardiv}.

\begin{prop}\label{prop:Ricciexplicit}
Let $(\GG,\divop)$ be a pair as in Definition \ref{def:Ricci}. Denote $\divop^\GG - \divop = \la e, \cdot \ra $, and set $\varphi_\pm = g(\pi e_\pm, \cdot) \in T^*$. Then, one has
\begin{equation}\label{eq:Riccipmexact}
\begin{split}
\gRc^+(\sigma_- X, \sigma_+ Y) = \Big{(} \Rc - \frac{1}{4} H^2 - \frac{1}{2}d^*H + \nabla^+\varphi_+\Big{)}(X,Y),\\
\gRc^-(\sigma_+ X, \sigma_- Y) = \Big{(}\Rc - \frac{1}{4} H^2 + \frac{1}{2}d^*H - \nabla^-\varphi_-\Big{)}(X,Y).
\end{split}
\end{equation}
\begin{proof}
First, observe that $\gRc^\pm(\GG,\divop^\GG) = \gRc_{D^B}^\pm$ by Proposition \ref{propo:Riccitorsion}. Now, by Proposition \ref{prop:genCurvature} we have that $\gRc_{D^B}^\pm$ can be identified with the Ricci tensor for the classical Bismut connection in \eqref{Bismutcurvature}. More precisely, we have
$$
\gRc^\pm_{D^B}(\sigma_\mp X, \sigma_\pm Y) =  \Big{(} \Rc - \frac{1}{4} H^2 \mp \frac{1}{2}d^*H\Big{)}(X,Y).
$$
Applying now Lemma \ref{lem:Riccivardiv}, we obtain
\begin{align*}
\gRc^\pm(\sigma_\mp X, \sigma_\pm Y) & :=  \gRc^\pm(\GG,\divop)(\sigma_\mp X, \sigma_\pm Y) \\
 & = \Big{(} \Rc - \frac{1}{4} H^2 \mp \frac{1}{2}d^*H\Big{)}(X,Y) - \la [e_\pm,\sigma_\mp X],\sigma_\pm Y\ra.
\end{align*}
The statement follows from
\begin{align*}
\la [e_\pm,\sigma_\mp X],\sigma_\pm Y\ra =&\ - \la \sigma_\pm \nabla^\pm_X (g^{-1}\varphi_\pm),\sigma_\pm Y\ra\\
=&\ \mp g(\nabla^\pm_X (g^{-1}\varphi_\pm) ,Y) = \mp (\nabla^\pm \varphi_\pm)(X,Y).
\end{align*}
We leave as an \textbf{exercise} to recover formula \eqref{eq:Riccipmexact} by taking traces in \eqref{eq:gencurvatureexp}, when $e \in T^*$.
\end{proof}
\end{prop}

The Ricci tensors $\gRc^\pm$ of a pair $(\GG,\divop)$ can be assembled naturally into an endomorphism of $E$, which we will call the \emph{generalized Ricci tensor}. 

\begin{defn}\label{d:generalizedRicci} Given a pair $(\GG,\divop)$ as in Definition \ref{def:Ricci}, define the \emph{generalized Ricci tensor} via
\begin{align*}
\gRc(\GG,\divop) =&\ \gRc^+ - \gRc^- \in \End(E),
\end{align*}
where $\gRc^\pm = \gRc^\pm(\GG,\divop)$ are regarded as maps $\gRc^\pm \colon V_\mp \to V_{\pm}^* \cong V_{\pm}$ using the induced metric on $V_{\pm}$.
\end{defn}

Recall the compatibility condition for pairs $(\GG,\divop)$ introduced in Definition \ref{d:GGdivergencecomp}, which states that $e \in \Gamma(E)$, defined by $\la e,\ra = \divop^\GG - \divop$, is an infinitesimal isometry of $\GG$, that is,
$$
[e,\GG] = 0.
$$ 
Our next result proves that the generalized Ricci tensor is an element of $\mathfrak{so}(E) \cong \Lambda^2 E \subset \End(E)$ if and only if $(\GG,\divop)$ is a compatible pair. This structural property of the Ricci tensor in generalized geometry is strongly reminiscent of the symmetric property of the Ricci tensor in Riemannian geometry. In particular, it implies that the two Ricci tensors $\gRc^\pm$ contain the same information and, furthermore, that we can regard $\gRc(\GG,\divop)$ as a tangent vector to the space of generalized metrics (see Lemma \ref{l:tangentGG1}). This is the basic principle which will allow us to define the \emph{generalized Ricci flow} in the next chapter.

\begin{lemma}\label{lem:Ricciskew}
Let $(\GG,\divop)$ be as in Definition \ref{def:Ricci}. Then, we have
$$
\GG \circ \gRc(\GG,\divop) = - \gRc(\GG,\divop) \circ \GG.
$$
Furthermore,
$$
\gRc(\GG,\divop) \in \mathfrak{so}(E)
$$
if and only if $(\GG,\divop)$ is a compatible pair in the sense of Definition \ref{d:GGdivergencecomp}.
\begin{proof}
The first part of the statement follows directly by Definition \ref{d:generalizedRicci} and Lemma \ref{l:tangentGG2}. As for the second part, identify $V_\pm \cong T$ via the anchor map $\pi$. Regard $\gRc^\pm \in T^* \otimes T^*$ and consider the decomposition $\gRc^\pm = h^\pm + k^\pm$ into symmetric and skew-symmetric parts. Then, 
%%MGF: I comment this. It may be useful later
%in the isotropic splitting defined by $\GG$ we have (cf. \eqref{eq:Phiexp})
%\begin{equation}\label{eq:Phiexp}
%\gRc(\GG,\divop) = \left(\begin{matrix}
%                    g^{-1} h^- & g^{-1} k^- g^{-1}\\
%                    - k^- & - h^- g^{-1}
%                    \end{matrix} \right) \circ \pi_+ + \left(\begin{matrix}
%                    g^{-1} h^+ & g^{-1} k^+ g^{-1}\\
%                    - k^+ & - h^+ g^{-1}
%                    \end{matrix} \right) \circ \pi_-,
%\end{equation} 
%and 
the property $\gRc(\GG,\divop) \in \mathfrak{so}(E)$ is equivalent to 
\begin{equation}\label{eq:Riccisymmetry}
h^+ = h^-, \qquad k^+ = -k^-.
\end{equation}
Using formula \eqref{eq:Riccipmexact} it follows immediately that $\gRc(\GG,\divop^\GG) \in \mathfrak{so}(E)$. Finally, setting $\gRc := \gRc(\GG,\divop)$, Lemma \ref{lem:Riccivardiv} implies that
\begin{align*}
\la \gRc (a),b \ra +& \la a,\gRc (b) \ra \\
=&\  \gRc^+(a_-,b_+) + \gRc^+(b_-,a_+) - \gRc^-(a_+,b_-) - \gRc^-(b_+,a_-)\\
=&\ - \la [e_+,a_-],b_+\ra - \la [e_+,b_-],a_+\ra + \la [e_-,a_+],b_-\ra + \la [e_-,b_+],a_-\ra \\
=&\ - \la [e_+,a_-],b_+\ra - \la [e_+,b_-],a_+\ra + \la [e,a_+],b_-\ra + \la [e,b_+],a_-\ra\\
&\ - \la [e_+,a_+],b_-\ra - \la [e_+,b_+],a_-\ra\\
=&\ - \la [e_+,a_-],b_+\ra - \la [e_+,b_-],a_+\ra + \la [e,a_+],b_-\ra + \la [e,b_+],a_-\ra\\
&\ + \la a_+,[e_+,b_-]\ra + \la b_+,[e_+,a_-]\ra\\
=&\ \la [e,a_+],b_-\ra + \la [e,b_+],a_-\ra.
\end{align*}
and therefore $\gRc \in \mathfrak{so}(E)$ if and only if $[e,\GG] = 0$.
\end{proof}
\end{lemma}

\begin{rmk}
By Lemma \ref{lem:Ricciskew}, the generalized Ricci tensor $\gRc(\GG,\divop^\GG)$ associated naturally to a generalized metric $\GG$ is an element in $\mathfrak{so}(E)$. From the classical point of view, this tensor incorporates the Ricci curvatures of \emph{both} Bismut connections (associated to $\pm H$). The more general quantity $\gRc(\GG,\divop)$ considered here, with varying divergence operator $\divop$, plays an important role in the interplay between the generalized Ricci flow and pluriclosed flow in Chapter \ref{c:GFCG}, as well as T-duality in Chapter \ref{c:Tdual}.
\end{rmk}

\begin{rmk}\label{rmk:Riccibeta}
When the divergence $\divop$ in Proposition \ref{prop:Ricciexplicit} is such that $e = 4df$, for some smooth function $f$ (so that $\varphi_+ = - \varphi_- = 2df$), the pair $(\GG,\divop)$ is automatically compatible (see Proposition \ref{p:GGdivergencecomp}). In this case, the two natural quantities $h$ and $k$ which arise from the generalized Ricci tensor agree exactly with the leading order term in $\alpha'$ expansion of the so-called \emph{$\beta$-functions} for the graviton and antisymmetric tensor fields in the sigma model approach to type II string theory \cite{Friedanetal}. This establishes an identification between the $1$-loop renormalization group flow of this physical theory and the generalized Ricci flow in the next chapter.
\end{rmk}

\section{Generalized scalar curvature}\label{sec:ggdirac}

In this section we define a notion of scalar curvature in generalized geometry, following the approach from  \cite{CSW,GF19}, using Dirac operators and their associated Lichnerowicz formulae.  Throughout this section we will assume that the smooth manifold $M$ is endowed with a spin structure. This assumption can be removed either working locally, or considering the approach to the generalized scalar curvature in \cite{SeveraValach2} via the Schr\"odinger-type operator introduced in \cite{OSW} (see \eqref{eq:Schrodinger}). Let $\GG$ be a generalized metric on an exact Courant algebroid $E$.

Let $\CL(V_+)$ and $\CL(V_-)$ denote the bundles of complex Clifford algebras of $V_+$ and $V_-$, respectively, defined by the relation
\begin{equation}\label{eq:Cliffordrel}
v \cdot v = \la v,v \ra 
\end{equation}
for $v \in V_\pm$, and $\la ,\ra$ given by the restriction of the neutral pairing to $V_\pm$. Via the isometries $\pi : V_\pm \to (T,\pm g)$, where $g$ is the Riemannian metric associated to $\GG$, we can fix complex spinor bundles $S_\pm$ for the orthogonal bundles $V_\pm$. We assume that $S_+$ (resp. $S_-$) is associated to the fixed spin principal bundle given by the spin structure on $M$, for a choice of irreducible representation of the complex Clifford algebra on $n$-dimensional Euclidean space $(\mathbb{R}^n,g_0)$ (resp. $(\mathbb{R}^n,-g_0)$) (see \cite{MicLaw}, and note the different convention for the relation \eqref{eq:Cliffordrel} in the Clifford algebra).

In order to define the generalized scalar curvature, we need to further study the Koszul-type formula in Lemma \ref{l:mixedfixed}. The freedom in the construction of a torsion-free generalized connection compatible with the generalized metric $\GG$ corresponds to the choice of pure-type operators
\begin{equation*}
D^+_+ \colon \Gamma(V_+) \to \Gamma(V_+^* \otimes V_+), \qquad D^-_- \colon \Gamma(V_-) \to \Gamma(V_-^* \otimes V_-).
\end{equation*}
We define next a pair of Dirac-type operators which are independent of these choices, once we fix a divergence operator $\divop$ on $E$ (see Lemma \ref{lem:weylfixed}). For any choice of such connection, and via the isometries $\pi : V_\pm \to (T,g)$, the operators $D^\pm_\pm$ can be identified with standard metric connections on $(T,g)$ and induce canonically spin connections
\begin{equation}\label{eq:LCspinpure}
D^{S_+}_+ \colon \Gamma(S_+) \to \Gamma(V_+^* \otimes S_+), \qquad D^{S_-}_- \colon \Gamma(S_-) \to \Gamma(V_-^* \otimes S_-).
\end{equation}
Recall that we denote by $\cD^0(\GG)$ the space of torsion-free generalized connections compatible with $\GG$, while $ \cD^0(\GG,\divop) \subset  \cD^0(\GG)$ denotes the subspace compatible with a given divergence operator $\divop$.

\begin{defn} \label{d:scalardirac} With the assumptions above, given a generalized connection $D \in \cD^0(\GG)$, define a pair of Dirac-type operators
\begin{align*}
\slashed D^+ & \colon \Gamma(S_+) \to \Gamma(S_+), \qquad \slashed D^- \colon \Gamma(S_-) \to \Gamma(S_-),
\end{align*}
given explicitly by
$$
\slashed D^\pm \alpha = \sum_{i=1}^{n} e_i^\pm \cdot D^{S_\pm}_{e_i^\pm} \alpha,
$$
for $\alpha \in \Gamma(S_\pm)$, and a choice of orthonormal basis $\{e_i^\pm\}_{i=1}^n$ of $V_\pm$, i.e. such that $\IP{e_i^\pm,e_j^{\pm}} = \pm \delta_{ij}$.
\end{defn}

\begin{lemma}\label{lem:dDpm}
The Dirac operators $\slashed D^+$ and $\slashed D^-$ are independent of the choice of generalized connection $D \in \cD^0(\GG,\divop)$.
\end{lemma}

\begin{proof}
Let $D \in \cD(\GG,\divop)$ and consider a torsion-free $\GG$-compatible generalized connection $D' = D + \chi \in \cD^0(\GG)$ with
$$
\chi = \chi^+ + \chi^- \in \Sigma^+ \oplus \Sigma^-.
$$
Using the decomposition in \eqref{eq:decompositionchipm} there exists $e_\pm \in \Gamma(V_\pm)$ such that
$$
\chi^\pm = \chi^\pm_0 + \chi_\pm^{e_\pm}.
$$
Then, we obtain
$$
\divop_{D'} = \divop - (n - 1)\la e_+,\cdot \ra - (n - 1) \la e_-, \cdot \ra
$$
and therefore $\divop_{D'} = \divop$ implies $e_+ = e_- = 0$. Define $\sigma \in \Gamma(\Lambda^3 E^*)$ by
$$
\sigma(e_1,e_2,e_3) = \la \chi_{e_1}e_2,e_3\ra,
$$
and note that the total skew-symmetrization $c.p. (\sigma)$ vanishes, since $D'$ and $D$ are torsion-free. Decomposing $\sigma = \sigma^+ + \sigma^-$ in pure-types, the result follows from
$$
\slashed D'^\pm = \slashed D^\pm - \tfrac{1}{2} c.p.(\sigma^\pm) = \slashed D^\pm.
$$

\end{proof}

\begin{rmk}
In \cite{GF19}, the canonical Dirac operators $\slashed D^\pm$ are defined using twisted spinor bundles (for a choice of root of $\det S_\pm^*$). This twist is required to ensure a canonical choice of the spin connections $D^{S_\pm}_\pm$ on more general Courant algebroids, but it is not necessary in our setting.
\end{rmk}

Consider the metric connections $\nabla^\pm$ and $\nabla^{\pm1/3}$ on $(T,\pm g)$ associated to the generalized metric $\GG$, as in Example \ref{ex:D0}.  We will abuse notation and let $\nabla^{\pm}$ and $\nabla^{\pm1/3}$ also denote the spin connections induced on a given spinor bundle for $\pm g$.  Furthermore, by $\slashed \nabla^{\pm1/3}$ we mean the \emph{cubic Dirac operator}, that is, the Dirac operator associated to $\nabla^{\pm1/3}$. We will use the following Lichnerowicz-type formula for the cubic Dirac operator due to Bismut \cite{Bismut} (see also \cite[Theorem 6.2]{AgricolaFriedrich04}). Note the different sign convention for the Clifford algebra relations \eqref{eq:Cliffordrel}. 

\begin{lemma}\label{lem:scalarcurv}
The spinor Laplacian $(\nabla^{\pm})^*\nabla^{\pm}$ and the square of the cubic Dirac operator $\Big{(}\slashed \nabla^{\pm 1/3}\Big{)}^2$ are related by 
\begin{equation*}
\Big{(}\slashed \nabla^{\pm1/3}\Big{)}^2 - (\nabla^{\pm})^*\nabla^{\pm} = \mp \frac{1}{4} R \pm \frac{1}{48} |H|^2.
\end{equation*}
\end{lemma}

The following lemma provides the structural property for the definition of the generalized scalar curvatures. Observe that the canonical mixed-type operators associated to $\GG$ in Lemma \ref{l:mixedfixed} induce canonical spin connections as in \eqref{eq:LCspinpure}
\begin{equation*}
D^{S_+}_- \colon \Gamma(S_+) \to \Gamma(V_-^* \otimes S_+), \qquad D^{S_-}_+ \colon \Gamma(S_-) \to \Gamma(V_+^* \otimes S_-).
\end{equation*}
We follow the notation in Definition \ref{d:GRiemnniandiv}, Lemma \ref{lem:Riccivardiv}, and Lemma \ref{lem:dDpm}.

\begin{prop}\label{prop:scalarpmexplicit}
Let $(\GG,\divop)$ be a pair given by a generalized metric $\GG$ and a divergence operator $\divop$ on the exact Courant algebroid $E$. Denote $\divop^\GG - \divop = \la e, \cdot \ra $, and set $\varphi_\pm = g(\pi e_\pm, \cdot) \in T^*$. Then, for any spinor $\eta \in \Gamma(S_\pm)$ one has
\begin{gather}\label{eq:scalarpmexplicit}
\begin{split}
\Big{(}\slashed D^{S_\pm}_\pm\Big{)}^2 \eta & - \Big{(}D^{S_\pm}_\mp\Big{)}^*D^{S_\pm}_\mp \eta\\
=&\ \mp \frac{1}{4} \Bigg{(} R - \frac{1}{12} |H|^2 \mp  2 d^*\varphi_\pm - |\varphi_\pm|^2 + 2 d\varphi_\pm \pm 4 \nabla^\pm_{g^{-1}\varphi_\pm}\Bigg{)}\eta.
\end{split}
\end{gather}
\begin{proof}
We derive a formula for 
$$
\til{\mathcal{S}}^+ := \Big{(}\slashed D^{S_+}_+\Big{)}^2 - \Big{(}D^{S_+}_-\Big{)}^*D^{S_+}_-,
$$ 
using Lemma \ref{lem:scalarcurv}. A similar formula for $\til{\mathcal{S}}^-$, defined as $\mathcal{S}^+$ replacing $+$ by $-$, follows from the analogue of Lemma \ref{lem:scalarcurv} for the opposite Clifford algebra relations \eqref{eq:Cliffordrel}. Via the isommetry $\pi_{|V_+} \colon V_+ \to (T,g)$, we have
$$
\til{\mathcal{S}}^+ = \Big{(}\til{\slashed{\nabla}} \Big{)}^2 - \Big{(}\nabla^+\Big{)}^*\nabla^+
$$
acting on sections of the fixed spinor bundle for $g$, where $\til{\nabla}$ is as in \eqref{eq:LCvarphipure}: 
$$
\til{\nabla}_X Y =  \nabla^{+1/3}_XY + \frac{1}{n-1}(g(X,Y)g^{-1}\varphi_+ - \varphi_+(Y)X).
$$
Let $\{e_1,\ldots,e_n\}$ be a local orthonormal frame for $(T,g)$. An endomorphism $A\in \End(T)$ satisfies
$$ A = \sum_{i,j=1}^n g(Ae_i,e_j)e^i\otimes e_j,$$
for $\{e^j\}$ the dual frame of $\{e_j\}$. Since $e^i\otimes e_j - e^j\otimes
e_i\in \mathfrak{so}(T)$ embeds as $\frac{1}{2}e_j e_i$ in the Clifford
bundle $Cl(T)$, an endomorphism $A\in \mathfrak{so}(T)$ corresponds to
\begin{equation*}
  A=\frac{1}{4}\sum_{i,j} g(Ae_i,e_j) e_j e_i \in Cl(T).
\end{equation*}
Then, given a local spinor $\eta$, we have
\begin{align*}
\til \nabla \eta &{} =  \nabla^{+1/3} \eta + \frac{1}{4(n-1)} \sum_{i,j} \varphi_+(e_j) e_j e_i \eta \otimes e^i - \varphi_+(e_i) e_j e_i \eta \otimes e^j
\end{align*}
and hence, 
\begin{align*}
\til{\slashed{\nabla}} \eta & =  \slashed{\nabla}^{+1/3}\eta + \frac{1}{4(n-1)}\sum_{i, j} \varphi_+(e_j) e_i e_j e_i \eta - \varphi_+(e_i) e_j e_j e_i \eta\\
& = \slashed{\nabla}^{+1/3}\eta + \frac{1}{4(n-1)}\sum_{i,j} \varphi_+(e_j) (2 \delta_{ij} - e_je_i) e_i \eta - \varphi_+(e_i) e_i \eta \\
& = \slashed{\nabla}^{+1/3}\eta - \frac{1}{2}\varphi_+ \eta.
\end{align*}
From this, we calculate
\begin{align*}
\Big{(}\til{\slashed{\nabla}} \Big{)}^2 \eta & = \Big{(}\slashed{\nabla}^{+1/3} \Big{)}^2 \eta - \frac{1}{2} \sum_j e_j \nabla^{+1/3}_{e_j}(\varphi_+ \eta) - \frac{1}{2} \sum_j \varphi_+ e_j\nabla^{+1/3}_{e_j}\eta + \frac{1}{4}|\varphi_+|^2 \eta  \\
 & = \Big{(}\slashed{\nabla}^{+1/3} \Big{)}^2 \eta - \frac{1}{2} \sum_j e_j (\nabla^{+1/3}_{e_j}\varphi_+) \eta - \frac{1}{2} \sum_j (e_j \varphi_+ +\varphi_+ e_j)\nabla^{+1/3}_{e_j}\eta + \frac{1}{4}|\varphi_+|^2 \eta\\
  & = \Big{(}\slashed{\nabla}^{+1/3} \Big{)}^2 \eta - \frac{1}{2} \sum_{j,k}  (i_{e_k}\nabla^{+1/3}_{e_j}\varphi_+)e_je_k \eta -  \nabla^{+1/3}_{g^{-1}\varphi_+}\eta + \frac{1}{4}|\varphi_+|^2 \eta,\\
  & = \Big{(}\slashed{\nabla}^{+1/3} \Big{)}^2 \eta + \frac{1}{2} (d^*\varphi_+) \eta - \frac{1}{2}(d\varphi_+) \eta + \frac{1}{4}|\varphi_+|^2 \eta  \\
   &\qquad - \frac{1}{12} \sum_{j,k} H(e_j,g^{-1}\varphi_+,e_k) e_je_k \eta - \nabla_{g^{-1}\varphi_+}\eta\\
   &\qquad - \frac{1}{24} \sum_{j,k} H(g^{-1}\varphi_+,e_j,e_k) e_ke_j \eta\\
     & = \Big{(}\slashed{\nabla}^{+1/3} \Big{)}^2 \eta + \frac{1}{2} (d^*\varphi_+) \eta - \frac{1}{2} (d\varphi_+) \eta + \frac{1}{4}|\varphi_+|^2 \eta -  \nabla^+_{g^{-1}\varphi_+}\eta,
\end{align*}
where we used that $e^j \alpha + \alpha e^j = 2 \alpha(e_j)$ for any $\alpha \in T^*$, combined with the standard formulae
\begin{align*}
d^*\alpha & = - \sum_j i_{e_j}  \nabla_{e_j}\alpha, \\
\nabla \alpha (e_j,e_k)  & = \tfrac{1}{2} d\alpha (e_j,e_k) + \tfrac{1}{2}(i_{e_k}\nabla_{e_j} \alpha + i_{e_j}\nabla_{e_k} \alpha).
\end{align*}
The proof follows now from Lemma \ref{lem:scalarcurv}.
\end{proof}
\end{prop}

Observe that the right hand side of \eqref{eq:scalarpmexplicit} defines a differential operator of order $1$ acting on spinors. The degree zero part of this operator has a scalar component and a component of degree two, the latter given by $\mp d \varphi_\pm$. If we insist in defining a \emph{scalar} out of \eqref{eq:scalarpmexplicit}, there are two natural conditions forced upon us. The first condition is that $\mp \nabla^\pm_{g^{-1}\varphi_\pm}$ is a natural operator, depending only on $(\GG,\divop)$, so that it can be absorbed in the left hand side of \eqref{eq:scalarpmexplicit}. This is achieved provided that $\pi e= 0$, which implies $\varphi_+ = - \varphi_-$ and (see Lemma \ref{l:mixedfixed} and Proposition \ref{p:HitchinBismut})
$$
\nabla^\pm_{g^{-1}\varphi_\pm} = - \nabla^\pm_{g^{-1}\varphi_\mp} \equiv - D^{S_\pm}_{e_\mp}.
$$
Given this, the second condition is that $e = \varphi \in T^*$ satisfies $d\varphi = 0$. As observed in \cite{GF19}, this pair of conditions relate naturally to the theory of \emph{Dirac generating operators} on Courant algebroids introduced in \cite{AXu,Severa}. 

\begin{defn}\label{def:closed}
Let $(\GG,\divop)$ be a pair given by a generalized metric $\GG$ and a divergence operator $\divop$ on the exact Courant algebroid $E$. We say that $(\GG,\divop)$ is \emph{closed} if $e \in \Gamma(E)$, defined by $\divop^\GG - \divop = \la e, \cdot \ra $, satisfies $\pi e = 0$ and $d e = 0$.
\end{defn}

By Proposition \ref{p:GGdivergencecomp}, it follows immediately that if $(\GG,\divop)$ is closed, then it is a compatible pair in the sense of Definition \ref{d:GGdivergencecomp}. We are ready to introduce our definition of the scalar curvatures for a closed pair $(\GG,\divop)$.

\begin{defn}\label{def:scalar}
Let $(\GG,\divop)$ be a closed pair on the exact Courant algebroid $E$. The \emph{generalized scalar curvature} 
$$
\mathcal{S} = \mathcal{S}(\GG,\divop) \in C^\infty(M)
$$ 
of the pair $(\GG,\divop)$ is defined by
$$
\mathcal{S}: = \mathcal{S}^+ - \mathcal{S}^-
$$
where $\mathcal{S}_\pm \in C^\infty(M)$ are defined by the following Lichnerowicz type formulae
\begin{equation}\label{eq:scalardef}
\begin{split}
-\frac{1}{2} \mathcal{S}^+ & := \Big{(}\slashed D^{S_+}_+\Big{)}^2 - \Big{(}D^{S_+}_-\Big{)}^*D^{S_+}_- -  D^{S_+}_{e_-},\\
-\frac{1}{2} \mathcal{S}^- & := \Big{(}\slashed D^{S_-}_-\Big{)}^2 - \Big{(}D^{S_-}_+\Big{)}^*D^{S_-}_+ - D^{S_-}_{e_+}.
\end{split}
\end{equation}
This is well-defined by Lemma \ref{l:mixedfixed},  Lemma \ref{lem:dDpm}, and Proposition \ref{prop:scalarpmexplicit}
\end{defn}

\begin{rmk}
More explicitly, if $(\GG,\divop)$ is a closed pair and we denote $\divop^\GG - \divop = \la \varphi, \ra $, for a closed one-form $\varphi \in \Gamma(T^*)$, one has
\begin{equation}\label{eq:scalarexplicit}
\mathcal{S} = R - \frac{1}{12} |H|^2 - d^*\varphi - \frac{1}{4}|\varphi|^2.
\end{equation}
An important fact that we will use later in \S \ref{s:EHF} is that $\mathcal{S}(\GG,\divop)$ does \emph{not} coincide with the trace of the Ricci tensors $\gRc^\pm$ in \eqref{eq:Riccipmexact}.  Note also that we can regard \eqref{eq:scalardef} as a local formula on $M$, so that there is no obstruction for the existence of the spinor bundles. Therefore we can define the generalized scalar curvature for exact Courant algebroids over any smooth manifold.
\end{rmk}

\begin{rmk}\label{rem:TypeIIscalar}
When the divergence $\divop$ in Proposition \ref{prop:Ricciexplicit} is such that $e = 4df$, for some smooth function $f$ (so that $\varphi_+ = - \varphi_- = 2df$), the generalized scalar curvature is given by
\begin{equation}\label{eq:scalarexplicitstring}
\mathcal{S} = R - \frac{1}{12} |H|^2 + 4\Delta f - 4|\nabla f|^2,
\end{equation}
where $\Delta f = - d^* df = \tr_g \nabla^2 f$. Formula \eqref{eq:scalarexplicitstring} agrees exactly with the leading order term in $\alpha'$ expansion of the so-called \emph{$\beta$-function} for the dilaton field in the sigma model approach to type II string theory \cite{Friedanetal} (cf. Remark \ref{rmk:Riccibeta}).
\end{rmk}

We finish this section with an alternative point of view on the Ricci tensor, using the canonical Dirac operators as above (formula \eqref{eq:Ricci+-op} below corrects a factor of two in the statement of \cite[Lemma 4.7]{GF19}, taking into account the different normalization for the Dirac operators).

\begin{lemma}\label{lem:Ricci}
Assume that $M$ is positive-dimensional. Let $V_+ \subset E$ be a generalized metric and let $div$ be a divergence operator on $E$. Then the generalized Ricci tensors in Definition \ref{def:Ricci} can be calculated by
\begin{equation}\label{eq:Ricci+-op}
\pm \frac{1}{2}\iota_{a_\mp} \gRc^\pm \cdot \eta = \Big{(}\slashed D^\pm D^{S_\pm}_{a_\mp} - D^{S_\pm}_{a_\mp}\slashed D^\pm - \sum_{i=1}^{r_\pm} e_i^\pm \cdot D^{S_\pm}_{[e_i^\pm,a_\mp]_\mp}\Big{)}\eta,
\end{equation}
where $\eta \in \Gamma(S_\pm)$ and $a_\mp \in \Gamma(V_\mp)$.
\end{lemma}

\begin{proof}
We give a complete proof for $\gRc^+$, as the other case is symmetric. It is easy to see that the right hand side of \eqref{eq:Ricci+-op} is tensorial in $a_\mp \in \Gamma(V_\mp)$. To evaluate \eqref{eq:Ricci+-op}, we choose an orthogonal frame $\{e_i^+\}$ for $V_+$ around $x \in M$ which satisfies $D_{a_-}e_i^+ = 0$ at the point $x$. This frame can constructed as in the proof of Proposition \ref{propo:Riccitorsion}, by smooth extension of the parallel transport for the $V_-$-connection $D^+_-$ along a curve starting at $x$ with initial velocity $\pi(a_-)$. With the conventions above we have
\begin{align*}
\Big{(}[\slashed D^\pm,D_{a_-}]\eta & - \sum_{i=1}^{r_\pm} e_i^\pm \cdot D^{S_\pm}_{[e_i^\pm,a_\mp]_\mp}\eta\Big{)}_{|x}\\
 &= \sum_{i=1}^{n} e_i^+ \cdot (D_{e_i^+}D_{a_-} - D_{a_-} D_{e_i^+} - D_{[e_i^+,a_-]}) \eta\\
& = \sum_{i=1}^{n} e_i^+ \cdot \RR^+_D(e_i^+,a_-) \cdot \eta\\
& = \tfrac{1}{2}\tr (d^+ \to \RR^+_D(d^+,a_-))\cdot \eta \\
& \qquad - \tfrac{1}{2} \sum_{i < j < k} (c.p._{ijk} \la \RR^+_D(e_i^+,a_-)e_j^+,e_k^+\ra) e_i^+ \cdot e_j^+ \cdot e_k^+ \cdot \eta,
\end{align*}
where in the first equality we have used that $[a_-,e_i^+]_+ = D_{a_-}e_i^+$ by Lemma \ref{l:mixedfixed}, and $(D_{a_-}e_i^+)_{|x} = 0$ by the choice of frame. The proof follows from the algebraic Bianchi identity in Proposition \ref{propo:bianchi}.
\end{proof}

\begin{rmk}\label{rem:RicciWaldram}
The formula \eqref{eq:Ricci+-op} is an interpretation of a formula for the Ricci tensor of a generalized metric in the physics literature \cite{CSW}, claimed without proof and later proved in \cite{GF19}. The present proof relies on the Bianchi identity \eqref{propo:bianchi} for the generalized curvature of a torsion-free generalized connection.
\end{rmk}

Since the right hand side of \eqref{eq:Ricci+-op} is tensorial in $e_\mp$ and the operators $\slashed D^+$ and $D_-^+$
only depend on $(\GG,\divop)$ (see Lemma \ref{l:mixedfixed} and Lemma \ref{lem:dDpm}), formula \eqref{eq:Ricci+-op} can be taken as an alternative definition of the Ricci tensors, without relying on Proposition \ref{propo:Riccitorsion}. For this, we can regard \eqref{eq:Ricci+-op} as a local formula, and therefore there is no obstruction for the existence of the spinor bundles.

\section{Generalized Einstein-Hilbert functional} \label{s:EHF}

Having constructed the generalized Ricci tensor and scalar curvature in the previous sections, it is natural to seek a notion of canonical metric on generalized geometry which can be expressed in terms of these curvature quantities. Taking the point of view of the 
calculus of variations, we will consider canonical metrics given as critical points of a 
functional defined on a natural class of compatible pairs $(\GG,\divop)$ (see Definition \ref{d:GGdivergencecomp}). In Riemannian and Lorentzian 
geometry the Einstein-Hilbert functional, defined as
\begin{align*}
EH(g) = \int_M R_g\ dV_g,
\end{align*}
is a second order diffeomorphism invariant functional, whose critical points in 
the Riemannian setting, $n \geq 3$, satisfy $\Rc \equiv 0$, and are dubbed 
Einstein metrics.  This functional and the attendant Einstein metrics were 
discovered in the early 20th century and have played a central role in 
the development of Riemannian geometry and general relativity. By means of our analysis in generalized geometry we will introduce a more general, flexible class of metrics which can yield useful geometric structure on a wider class of manifolds.  To begin we consider a certain generalization of the Einstein-Hilbert functional for generalized metrics.

\begin{defn} \label{d:genEH}
Let $E$ be an exact Courant algebroid on an oriented smooth manifold $M$.  Define the 
\emph{generalized Einstein-Hilbert functional} for a generalized metric $\GG$ via
\begin{align*}
\mathcal{EH}(\GG) := \int_M \mathcal{S}(\GG) dV_g =&\ \int_M \left(R_g - \tfrac{1}{12} |H|^2_g\right) dV_g.
\end{align*}
\end{defn}

Here we abuse of the notation and denote $\mathcal{S}(\GG) := \mathcal{S}(\GG,\divop^\GG)$ (see Definition \ref{def:scalar}). As we  show in the next proposition, the critical point equation is naturally expressed in terms of the generalized scalar curvature $\mathcal{S}(\GG)$ and the generalized Ricci tensor $\gRc(\GG) := \gRc(\GG,\divop^\GG)$ (see Definition \ref{d:generalizedRicci}).

\begin{prop} \label{p:genEHvar2}
Let $E$ be an exact Courant algebroid over an oriented smooth manifold $M^n$, with $n \geq3$. A generalized metric $\GG$ on $E$ is a 
critical point for the generalized Einstein-Hilbert functional if and only if
\begin{align*}
\GG\gRc(\GG) = \tfrac{1}{2} \mathcal{S}(\GG)\GG.
\end{align*}
Equivalently, in terms of the associated pair $(g,H)$, the critical points of the functional are given by solutions of the equation
\begin{align*}
\Rc^+ - \tfrac{1}{2} R g - \tfrac{1}{12} \brs{H}^2 g = 0.
\end{align*}
\begin{proof} The proof involves computing variation formulas for the different pieces of the integrand, and these are carried out in \S \ref{s:variation} below.  Specifically, using Lemmas \ref{l:volumevariation}, \ref{l:Henergyvar} and 
\ref{l:curvaturevariation}, and applying the divergence theorem one derives
\begin{align*}
\left. \frac{d}{ds} \right|_{s=0} \mathcal{EH}(g_s,b_s) =&\ \int_M \left\{ 
\left(  - \gD \tr_g h + \divg \divg h - \IP{h, \Rc} \right) \right.\\
&\ \left.  \qquad - \tfrac{1}{12} \left( 2 \IP{H, dK} -3 \IP{h, H^2} \right) + 
\left(R - \tfrac{1}{12} \brs{H}^2 \right) \left( \tfrac{1}{2} \tr_g h \right) 
\right\} dV_g\\
=&\ \int_M \left\{ \IP{ h, - \Rc + \tfrac{1}{4} H^2 + \tfrac{1}{2} \left(R 
- \tfrac{1}{12} \brs{H}^2\right) g} - \tfrac{3}{2} \IP{d^* H, K} \right\} dV_g.
\end{align*}
Thus this variation vanishes for all choices of $(h,K)$ if and only if
\begin{align*}
\Rc - \tfrac{1}{4} H^2 - \tfrac{1}{2} \left( R - \tfrac{1}{12} \brs{H}^2 
\right) g = 0, \qquad d^* H = 0.
\end{align*}
The first part of the statement follows combining Proposition \ref{prop:Ricciexplicit}, equation \eqref{eq:scalarexplicit}, and Lemma \ref{l:symmetryaction}. As for the second part, by Proposition \ref{p:Bismutcurvature}, the vanishing of $d^* H$ is equivalent to vanishing of the skew-symmetric part of 
$\Rc^+$.  We reinterpret the scalar piece of the first equation in terms of the Bismut scalar curvature to obtain the equivalent equation
\begin{align*}
\Rc^+ - \tfrac{1}{2} R g - \tfrac{1}{12} \brs{H}^2 g = 0,
\end{align*}
as required.
\end{proof}
\end{prop}

Observe the strong analogy between the critical point equation for this functional, formulated in the generalized geometry language, and the classical Einstein equation. While these \emph{generalized Einstein metrics} are of interest, it turns out to be more useful to adopt a more flexible functional, where we allow for the divergence operator to vary independent of the generalized metric.  Following the definition of the scalar curvature in Definition \ref{def:scalar}, we introduce the class of compatible pairs  $(\GG,\divop)$ of main interest.

\begin{defn}\label{def:exact}
Let $(\GG,\divop)$ be a pair given by a generalized metric $\GG$ and a divergence operator $\divop$ on an exact Courant algebroid $E$. We say that $(\GG,\divop)$ is \emph{exact} if $e \in \Gamma(E)$, defined by $\divop^\GG - \divop = \la e, \cdot \ra$, satisfies $\pi e = 0$ and $e = 2 d f$ for some smooth function $f \in C^\infty(M)$.
\end{defn}

Being a special case of the closed pairs in Definition \ref{def:closed}, an exact pair $(\GG,\divop)$ is automatically compatible in the sense of Definition \ref{d:GGdivergencecomp}. Exact pairs are closely related to the notion of \emph{density} in the smooth manifold $M$. For our purposes it will be sufficient to assume that the manifold $M$ is orientable and deal with the case of volume forms. Given a volume form $\mu$ on $M$, there is an associated divergence operator on the exact Courant algebroid $E$ defined by (cf. \ref{d:GRiemnniandiv})
$$
\divop^\mu(e) = \frac{L_{\pi e} \mu}{\mu}.
$$
Given now a generalized metric $\GG$ on $E$ we can write
$$
\mu = e^{-f} dV_g\
$$
for some smooth function $f \in C^\infty(M)$, and therefore we obtain
$$
\divop^\GG(e') - \divop^\mu(e') = \pi(e')f = \la 2df , e' \ra
$$
for any $e' \in \Gamma(E)$. Therefore $(\GG,\divop^\mu)$ is exact. Conversely, any exact pair $(\GG,\divop)$ is of the form $\divop = \divop^\mu$, for some volume form $\mu$ uniquely determined up to a positive multiplicative constant.

\begin{defn} \label{d:genEHdiv} 
Let $E$ be an exact Courant algebroid on an oriented smooth manifold $M$.  Define the 
\emph{generalized Einstein-Hilbert functional} for exact pairs $(\GG,\divop^\mu)$ via
\begin{gather}\label{GHFdiv}
\begin{split}
\mathcal{EH}(\GG,\divop^\mu) =&\ \int_M \mathcal{S}(\GG,\divop^\mu) \mu\\
=&\ \int_M \left(R - \frac{1}{12} |H|^2_g + 2\Delta_g f - |\nabla f|^2_g\right)  e^{-f} dV_g.
\end{split}
\end{gather}
\end{defn}

The last explicit formula in \eqref{GHFdiv} follows directly from \eqref{eq:scalarexplicit} taking $\varphi = 2 df$. Using the identity 
$$
\Delta_g(e^{-f}) = e^{-f}|\nabla f|_g^2 - e^{-f} \Delta_g f
$$
and integration by parts, we obtain the more amenable formula
\begin{align}\label{stringaction}
\mathcal{EH}(\GG,\divop^\mu) =&\ \int_M \left(R - \frac{1}{12} |H|^2 + |\nabla f|^2_g\right)  e^{- f} dV_g.
\end{align}
This functional is known in the mathematical physics literature as the \emph{low energy string effective action} in the sigma model approach to type II string theory \cite{Friedanetal}. It will play a key role as an energy functional for the generalized Ricci flow in Chapter \ref{energychapter}.

\begin{thm} \label{t:genEHvar} Given $E$ an exact Courant algebroid over a smooth manifold $M$ of dimension $n \geq 3$. Then, an exact pair $(\GG,\divop^\mu)$ is a critical point for the generalized Einstein-Hilbert functional if and only if
\begin{align}\label{eq:RicciflatScalarflat}
\gRc(\GG, \divop^{\mu}) = 0, \qquad \mathcal{S}(\GG, \divop^{\mu}) = 0.
\end{align}
\begin{proof} 
We postpone the proof of the variation until Lemma \ref{l:energyvar}, and apply formula \eqref{Ffirstvar} directly. Let $(\GG_s,\divop^{\mu_s})$ be a curve of exact pairs. In terms of the classical data $(g_s,b_s,\mu_s)$ given by $(\GG_s,\divop^{\mu_s})$ (see Lemma \ref{l:symmetryaction} and Proposition \ref{p:genmetricreduction}) we have
\begin{align*}
\left. \frac{\del}{\del s} \right|_{s = 0} g_s =&\ h, \qquad 
\left. \frac{\del}{\del s} \right|_{s = 0} b_s = k, \qquad 
\left. \frac{\del}{\del s} \right|_{s = 0} \mu_s  = (\tfrac{1}{2}\tr_g h -\phi)\mu.
\end{align*}
Thus, the tangent vector to $(\GG_s,\divop^{\mu_s})$ at $s = 0$ is given by (see Lemma \ref{l:tangentGG1})
\begin{align*}
\left. \frac{\del}{\del s} \right|_{s = 0} (\GG_s,\divop^{\mu_s}) =&\ (\GG\Phi, -d(2\phi - \tr_g h)).
\end{align*}
Applying formula \eqref{Ffirstvar} and formula \eqref{eq:Riccipmexact},
\begin{gather*}
\begin{split}
\left. \frac{d}{ds} \right|_{s=0} \mathcal{EH}(\GG_s,\divop^{\mu_s}) =&\ \int_M \left[ \IP{ - \Rc 
+ \tfrac{1}{4} H^2 - \N^2 f, h} + \IP{ \tfrac{1}{2} \left( - d^* H - \N f 
\lrcorner H \right), k} \right.\\
&\ \left. + \left( R - \tfrac{1}{12} \brs{H}^2 + 2 \gD f - \brs{\N f}^2 \right) 
\left( \tfrac{1}{2} \tr_g h - \phi \right) \right] e^{-f} dV_g\\
=&\ \int_M  \IP{\gRc, \GG\Phi}e^{-f}dV_g + \int_M \mathcal{S} 
\left( \tfrac{1}{2} \tr_g h - \phi \right)  e^{-f} dV_g,
\end{split}
\end{gather*}
and hence the statement follows.
\end{proof}
\end{thm}

It is not difficult to see that any critical point of the generalized Einstein-Hilbert functional on a compact manifold has necessarily vanishing three-form $H = 0$ and constant $f$. This follows taking the trace in the Ricci tensor $\Rc^+$ and comparing with the formula for the generalized scalar curvature (see the proof of \cite[Proposition 5.8]{GF19}). Thus, Ricci flat and scalar flat pairs $(\GG,\divop^\mu)$ reduce to the classical Einstein metrics on a compact manifold. The situation is more interesting if we consider solutions of \eqref{eq:RicciflatScalarflat} for closed pairs $(\GG,\divop)$, as in Definition \ref{def:closed}, rather than exact, as illustrated in the next example.

\begin{ex}\label{ex:GRicciflatHopf}
Consider the compact four-dimensional manifold $M = S^3 \times S^1$. We use the Lie group structure given by identifying
$$
M \cong SU(2) \times U(1).
$$
Consider generators for the Lie algebra
$$
\mathfrak{su}(2) \oplus \mathbb{R} = \langle e_1, e_2, e_3, e_4 \rangle, \qquad  
$$
such that
$$
de^1 = e^{23}, \quad de^2 = e^{31}, \quad de^3 = e^{12}, \quad de^4 = 0,
$$
for $\{e^j\}$ the dual basis, satisfying $e^j(e_k) = \delta_{jk}$, and the notation $e^{ij}=e^i\wedge e^j$ and similarly. We take $k \in \mathbb{R}$ and define a left-invariant three-form
$$
H = - k e^{123}. 
$$
This corresponds to a constant multiple of the pull-back of the Cartan three-form on the $SU(2)$ factor, and hence it is bi-invariant and closed. Thus, it defines an equivariant exact Courant algebroid on $E = T \oplus T^*$, with Dorfman bracket twisted by $H$. To define a closed pair $(\GG,\divop)$, we fix $x \in \mathbb{R}$ a non-zero constant, and define the bi-invariant metric
$$
g = k (e^1 \otimes e^1 + e^2 \otimes e^2 + e^3 \otimes e^3 + x^2 e^4 \otimes e^4),
$$ 
and the closed form
$$
\varphi = - x e^4.
$$
Then we define our pair $(\GG,\divop)$ taking the generalized metric $\GG$ determined by $g$ and the standard isotropic splitting $T \subset E$, and the divergence
$$
\divop = \divop^\GG - \la \varphi, \ra.
$$
We leave as an \textbf{exercise} to check that this pair satisfies the equations \eqref{eq:RicciflatScalarflat}. As a hint, observe that $g$ is the product of invariant metrics on the simple group $SU(2)$ and the Abelian group $U(1)$. Then, arguing as in Proposition \ref{p:Bismutflat} below it follows that $\N^{\pm}$ are flat. Thus, \eqref{eq:RicciflatScalarflat} reduces to check that (see \eqref{eq:Riccipmexact} and \eqref{eq:scalarexplicit})
$$
\N^{\pm} \varphi = 0, \qquad \frac{1}{6}|H|^2 - d^*\varphi - \frac{1}{4}|\varphi|^2 = 0.
$$
\end{ex}

As we will see in Chapter \ref{energychapter}, by considering variations of $\mathcal{EH}$ which fix a background volume element, the critical point equation no longer includes the scalar curvature vanishing, and we obtain an Euler equation determined only by the vanishing of the generalized Ricci tensor.  We thus adopt the following general definition:

\begin{defn} \label{d:generalizedEinstein} Given an exact Courant algebroid $E$ over a smooth manifold $M^n$, we say that a generalized metric $\GG$, with compatible divergence operator $\divop$, is \emph{generalized Einstein} if
\begin{align*}
\gRc(\GG, \divop) = 0.
\end{align*}
It is elementary to see that a special case of such metrics occurs for $(\GG, \divop^{\GG})$, where the associated pair $(g, H)$ has vanishing classical Bismut Ricci tensor.
\end{defn}

To finish this section we give examples of generalized Einstein metrics.  A basic source of Einstein metrics in Riemannian geometry arises from the famous uniformization theorem. This classical result classifies complete metrics with constant sectional curvature, showing they are all quotients of the three classical model spaces of the round sphere $(S^n, g_{S^n})$, the flat Euclidean plane $(\mathbb R^n, g_{\mbox{\tiny{Eucl}}})$ and hyperbolic space 
$(\mathbb H^n, g_{\mbox{\tiny{Hyp}}})$. In generalized geometry, the presence of torsion gives a wider class of examples.  In the examples below we will always have $\divop = \divop^{\GG}$, so that the generalized Einstein equation reduces to asking for the classical Bismut Ricci tensor to vanish, which will in particular be satisfied when the Bismut connection is flat.  We begin with the explicit construction of flat generalized metrics on simple Lie groups.

\begin{prop} \label{p:Bismutflat}  Let $G$ be a Lie group with 
bi-invariant metric $g$.  
Define connections $\N^{\pm}$ as the unique metric connections with torsion 
$T(X,Y) = \pm [X,Y]$, acting on left-invariant vector fields.  Then $\N^{\pm}$ are flat.
\begin{proof} Recall that left-invariant vector fields on a Lie group with 
bi-invariant metric $g$ satisfy
\begin{align*}
g ([X,Y], Z) = g(X,[Y,Z]). 
\end{align*}
Using this, a direct calculation shows that the Levi-Civita connection takes 
the 
form
\begin{align*}
 \N_X Y = \tfrac{1}{2} [X,Y].
\end{align*}
Furthermore, it follows that the tensor $H(X,Y,Z) := 
g(T(X,Y),Z)$ is skew-symmetric, since it is obviously skew-symmetric in $X$ and 
$Y$, while
\begin{align*}
 H(X,Z,Y) =&\ g(T(X,Z),Y) = \pm g( [X,Z],Y)\\
 =&\ \mp g([Z,X],Y) = \mp g(Z,[X,Y]) = - 
H(X,Y,Z).
\end{align*}
Thus we can express the connection with torsion 
$T$ 
as
\begin{align*}
 \IP{\N^{\pm}_X Y,Z} = \IP{\N_X Y, Z} + \tfrac{1}{2} H(X,Y,Z) = \tfrac{1}{2} 
\IP{[X,Y],Z} \pm \tfrac{1}{2} \IP{[X,Y],Z}.
\end{align*}
It is clear that $\IP{\N^-_X Y,Z}$ vanishes identically, and is thus flat.  For 
$\N^+$, this formula easily yields
\begin{align*}
 R(X,Y)Z =&\ \N_X(\N_Y Z) - \N_Y (\N_X Z) - \N_{[X,Y]} Z\\
 =&\ \N_X [Y,Z] - \N_Y [X,Z] - [[X,Y],Z]\\
 =&\ [X,[Y,Z]] - [Y,[X,Z]] - [[X,Y],Z]\\
 =&\ 0,
\end{align*}
 by the Jacobi identity.
\end{proof}
\end{prop}

We next give the classification of generalized Riemannian metrics with 
vanishing Bismut curvature.  As discussed above, this result was originally 
established in the work of Cartan-Schouten \cite{CartanSchouten, 
CartanSchouten2}, who proved a more general statement classifying 
Riemannian manifolds admitting an absolute parallelism.  More recently a short 
proof of this theorem which does not rely on the classification of symmetric 
spaces was given by Agricola-Friedrich \cite{AgricolaFriedrich}.  We 
provide an adaptation of their proof in our setting which is significantly
simplified due to the fact that we already assume the torsion is skew-symmetric 
and closed.

\begin{thm} \label{t:flatclass} Let $(M^n, g, H)$ be a complete simply-connected Riemannian manifold with $d H = 0$ such that the 
Bismut connection associated to $(g,H)$ is flat.  Then $(M^n, g)$ is isometric to a product of simple Lie groups with bi-invariant metrics $g$, and $g^{-1} H (X,Y) = \pm [X,Y]$ on left-invariant vector fields.
\begin{proof} 
For convenience we set
\begin{align*}
(\gs_H)(X,Y,Z,U) =&\ \sum_{\gs(X,Y,Z)} \IP{H(X,Y),H(Z,U)}.
\end{align*}
First note that since the connection is flat, combining (\ref{f:Bianchi1}) and (\ref{f:Bianchi3}) and using that $dH = 0$ we obtain
\begin{align} \label{f:flatclass10}
0 = d H (X,Y,Z,V) = \gs_H (X,Y,Z,U) - (\N^+_U H)(X,Y,Z).
\end{align}
Inserting this identity back into (\ref{f:Bianchi3}) yields $\gs_H = 0$ (\textbf{Exercise}).  Comparing against (\ref{f:flatclass10}) we see that $\N^+ H = 0$.  A further elementary computation shows that $\N^+ H - \N H = \tfrac{1}{2} \gs_H = 0$, thus $\N H = 0$.  Having shown $\N^+ H = 0$, Proposition \ref{p:Bismutcurvature} yields
\begin{align*}
\Rm(X,Y,Z,W) = \tfrac{1}{4} \left( \IP{H(X,W),H(Y,Z)} - \IP{H(Y,W),H(X,Z)} \right)
\end{align*}
In particular, the sectional curvatures satisfy
\begin{align*}
K(X,Y) =&\ \frac{1}{4} \frac{\brs{H(X,Y)}^2}{\brs{X}^2 \brs{Y}^2 - \IP{X,Y}^2} \geq 0.
\end{align*}
This implies a splitting of the torsion tensor in the presence of a metric splitting.  In particular, if the metric splits as a product $M = M_1 \times M_2$ $g = g_1 \oplus g_2$, then for $X_i \in TM_i$ unit we obtain $0 = K(X_1,X_2) = \brs{H(X_1,X_2)}^2$, hence $H = H_1 \oplus H_2$.  Thus we may reduce to analyzing the case that $(M, g)$ is irreducible.

Now fix an orthonormal basis $\{e_i\}$ for some point $p \in M$.  Since $\N$ is flat, by parallel transport we can extend this to a global orthonormal frame on $M$.  By the identity
\begin{align*}
[e_i,e_j] = - T_{\N^+}(e_i,e_j).
\end{align*}
it follows directly that $\{e_i\}$ are Killing fields and also
\begin{align*}
0 = \N_{e_i} H(e_j, e_k, e_l) = - e_i (\IP{[e_j,e_k],e_l}),
\end{align*}
in other words the functions $\IP{[e_i,e_j],e_l}$ are constant.  Thus $\{e_i\}$ are the basis for a Lie algebra.  The corresponding simply-connected Lie group is a simple, compact Lie group isometric  to $(M, g)$.
\end{proof}
\end{thm}

We next turn to low dimensional examples, and in the next proposition give a classification of solutions 
in dimension $n=3$.

\begin{prop} \label{p:GE3d} Let $(M^3, g, H)$ be a generalized Einstein 
three-manifold.  Then there exists $\gl \in \mathbb R$ such that $H = \gl 
dV_g$, and $(M^3, g)$ is isometric to a space form with either zero sectional curvature if $\gl = 0$, and positive sectional curvature otherwise.
\begin{proof} From the generalized Einstein equation we know that $H$ is 
harmonic.  But since $H \in \Lambda^3$ and $M$ is three-dimensional this means 
that
\begin{align*}
0 =&\ d * H = d \left( \tfrac{H}{dV_g} \right),
\end{align*}
hence $\tfrac{H}{dV_g} \equiv \gl$, as claimed.  Given this it follows that 
$H^2 = 2 \gl^2 g$, and hence by the vanishing of $\Rc^+$ we obtain
\begin{align*}
\Rc =&\ \Rc^+ + \tfrac{1}{4} H^2 = \tfrac{\gl^2}{2} g.
\end{align*}
Hence $g$ is an Einstein metric.  Expressing the Ricci curvature in terms of an 
orthonormal frame, one observes that $R_{1212} + R_{1313} = R_{2121} + R_{2323} 
= R_{3131} + R_{3232}$, which implies that $R_{1212} = R_{1313} = R_{2323}$.  
Since this holds for an arbitrary orthonormal frame it follows that the metric 
has constant sectional curvature, and so $g$ is a space form, as claimed.
\end{proof}
\end{prop}

\begin{ex} \label{e:S3S1E} Note that in the previous proposition that a generalized Einstein 
structure can be associated to a classical Einstein metric with nonzero Einstein constant.  In particular, consider the 
example of $(S^3, g_{S^3})$, the three sphere with the round metric of constant 
sectional curvature $1$.  This metric is Einstein, in particular $\Rc^{g_{S^3}} 
= 2 g_{S^3}$.  Now set $H = 2dV_g$.  This is a parallel three-form, 
and a direct calculation shows that $H^2 = 8 g$, and hence it 
follows that $\Rc^+ \equiv 0$, and so the structure $(g_{S^3}, H)$ is 
generalized Einstein.  This example is actually Bismut flat, and is left-invariant on the Lie group $S^3$, in line with Theorem \ref{t:flatclass}.
\end{ex}

\begin{ex} \label{e:Einsteinproducts} Given $(M_i^{n_i}, g_i, H_i)$, $i=1,2$ 
two generalized Einstein structures, we 
form a product from these structures as follows.  First set $M = M_1 \times 
M_2$, with canonical projection maps $\pi_i : M \to M_i$.  We furthermore set $g = \pi_1^* g_1 + \pi_2^* g_2$ and $H = \pi_1^* H_1 + 
\pi_2^* H_2$.  It is easily checked that this gives another generalized Einstein structure.  This allows us to form our first nontrivial example, namely we take $(S^3, 
g_{S^3}, \tfrac{1}{2} dV_{g_{S^3}})$ and $(S^1, g_{S^1}, 0)$, both generalized 
Einstein (see Example \ref{e:S3S1E}), 
and so the resulting product structure is a nontrivial generalized Einstein 
structure.  This example plays a central role in 
the study of generalized Einstein structures in complex geometry, as we will 
see in Chapter \ref{c:CMGCG}.  Note also that not only is the product metric $\pi_1^* 
g_{S^3} + \pi_2^* g_{S^1}$ not an Einstein metric, but $S^3 \times S^1$ cannot 
admit any Einstein metric by the Hitchin Thorpe inequality (cf. \cite{HitchinHT, ThorpeHT}).
\end{ex}

\begin{question} A classical result of Alekseevskii-Kimelfeld \cite{Alekseevski} establishes that a homogeneous Riemannian manifold $M = G / K$ with zero Ricci curvature is locally Euclidean, and so isometric to a product of a Euclidean space with a flat torus.  An elementary proof of this fact was given by B\'erard-Bergery (\cite{BerardBergery}) which we briefly recount here.  Using Cheeger-Gromoll splitting \cite[Theorem 5]{CheegerGromoll}, the space splits as a product of a compact homogeneous Ricci flat space with a flat Euclidean space.  Using the Bochner argument on the compact piece, it follows that all Killing fields are parallel.  Thus there is a global parallel frame and so the compact piece is flat.  Does this result extend to generalized setting?  That is, given $(M^n, g, H)$ a homogeneous Riemannian manifold with $H$ invariant and zero Bismut-Ricci curvature, is the associated Bismut connection flat?
\end{question}

\chapter{Fundamentals of Generalized Ricci Flow} \label{GRFchapter}

Having established fundamental properties of generalized geometry, and
some basic notions and examples of canonical structures in generalized geometry, it is 
natural to ask the question,
\begin{quotation}
\textbf{Which manifolds admit canonical generalized 
geometries, and how can we construct them?}
\end{quotation}
As stated this question is quite broad, and we cannot expect satisfying answers 
without further restrictions.  In the context of Riemannian geometry, Hamilton 
introduced a powerful tool, the Ricci flow, for approaching these questions 
in a broad, principled, way.  Ricci flow is an evolution equation for 
Riemannian metrics taking the form
\begin{align*}
 \dt g =&\ -2 \Rc.
\end{align*}
The basic philosophy underpinning the Ricci flow equation is at once elementary 
and profound.  This equation is meant to be interpreted as a heat equation
for Riemannian metrics, where one expects the curvature to smooth out and 
become in a certain sense uniform, eventually yielding a canonical Riemannian metric, say an 
Einstein metric, as a limit.  The generalized Ricci flow seeks to adapt this same philosophy to the 
setting of generalized geometry, and can be expressed in the language of generalized geometry as
\begin{align*}
\GG^{-1} \dt \GG =&\ -2 \gRc(\GG).
\end{align*}
\noindent  As we will see, this is a kind of heat equation which 
deforms generalized metrics, at least in principle, towards a canonical 
structure, say a generalized Einstein metric.  From this basic starting point 
we will show in the remainder of the text that the generalized Ricci flow 
equation is rich with geometric structure, broad enough to encompass meaningful 
geometric and topological questions, and yet tame enough to yield global 
existence results.

\section{The equation and its motivation}

In this section we give various definitions and equivalent formulations of 
generalized Ricci flow, of increasing conceptual complexity.  To begin we give 
the equations first from the point of view of classical objects in Riemannian geometry.

\subsection{Classical formulation}

\begin{defn} Given $M^n$ a smooth manifold, $H_0 \in \Lambda^3 T^*$, $d H_0 = 
0$, we say that a one-parameter family $(g_t, b_t)$ 
of pairs of Riemannian metrics and two forms is a solution of 
\emph{generalized Ricci flow} if, setting $H = H_0 + db_t$, one has
\begin{gather} \label{f:GRF}
\begin{split}
\dt g =&\ -2 \Rc + \tfrac{1}{2} H^2,\\
\dt b =&\ - d^*_g H.
\end{split}
\end{gather}
\end{defn}

As discussed in Chapter \ref{c:GCC}, this system first appeared in physics literature from the study of renormalization group flow \cite{Friedanetal}, with the first mathematical results appearing in \cite{OSW,Streetsexpent}.  From a naive point of view this flow is a natural coupling of the Ricci flow 
equation with a degenerate heat equation for a two-form.  To get a feeling for the generalized Ricci flow, we recall some basic aspects of parabolic equations.  Consider the
initial value problem for a section $u$ of some smooth vector bundle $E$ over $M$ with connection $\N$,
\begin{gather} \label{generalparabolic}
\begin{split}
\frac{\del u}{\del t} =&\ a^{ij}(t, x, u, \N u) \N_i \N_j u + f(t, x, u, \N u),\\
u(0) =&\ u_0.
\end{split}
\end{gather}
Here of course the right hand side should be defined by some coordinate
invariant operators, and the equation (\ref{generalparabolic}) is to be interpreted in
some coordinate chart.

\begin{defn} The equation (\ref{generalparabolic}) is \emph{parabolic}
at $(u, x, t)$ if 
\begin{align*}
a^{ij}(x, t, u, \N u) > 0
\end{align*}
in the sense of bilinear forms.
\end{defn}

The standard theory of parabolic partial diffferential equations gives that for any initial data $u_0$ on a compact manifold such that (\ref{generalparabolic}) is parabolic at all points $(u_0,x,0)$, there exists $\ge > 0$ and a solution to (\ref{generalparabolic}) on $[0, \ge)$.  For more context and further results in these areas 
we direct the reader to \cite{Evans, GT, Lieberman}.  Due to the issue of gauge invariance, the generalized Ricci flow is not a strictly parabolic equation.  However, the resemblance of generalized Ricci flow to a heat equation can be seen in the next lemma.

\begin{lemma} \label{l:GRFasheat} Fix $M$ a smooth manifold, $H_0 \in \Lambda^3 T^*$, $d H_0 = 
0$, and $(g_t, b_t)$ a solution of generalized Ricci flow.  Assuming $d^*_g b = 0$, in harmonic coordinates centered at a point $(p_0, t_0)$, the generalized Ricci flow is expressed as
\begin{align*}
\dt g_{ij} =&\ g^{kl} \frac{\del^2 g_{ij}}{\del x^k \del x^l} + Q_1(g, \del g, H_0, \del b),\\
\dt b_{ij} =&\ g^{kl} \frac{\del^2 b_{ij}}{\del x^k \del x^l} + Q_2(g, \del g, H_0, \del b) + Q_3(g, b, \del^2 g),
\end{align*}
where the quantities $Q_i$ denote quadratic expressions in the arguments.
\begin{proof} By \cite[Lemma 49]{Petersen} we know that at the center of a harmonic coordinate chart one has
\begin{align*}
\Rc_{ij} =&\ -\tfrac{1}{2} g^{kl} \frac{\del^2 g_{ij}}{\del x^k \del x^l} + Q(g, \del g)
\end{align*}
for some quadratic expression $Q(g, \del g)$ in first derivatives of $g$.  As the term $H^2$ in the definition of generalized Ricci flow is expressed as a quadratic in $H_0$ and first derivatives of $b$, the first claimed equation follows.  As we have assumed $d^*_g b = 0$, we can furthermore express
\begin{align*}
\dt b_{ij} =&\ - d^*_g \left(H_0 + db \right)\\
=&\ \gD_d b + Q(g, \del g, H_0)\\
=&\ g^{kl} \frac{\del^2 b_{ij}}{\del x^k \del x^l} + Q_2(g, \del g, H_0, \del b) + Q_3(g, b, \del^2 g),
\end{align*}
as claimed.
\end{proof}
\end{lemma}

Thus we see that in special coordinates for $g$, together with the divergence-free condition for $b$, the generalized Ricci flow appears to be a strictly parabolic equation. The choice of coordinates and imposition of $d^*_g b = 0$ relate to the gauge-invariance of the flow, explored more in \S \ref{s:invgroup}, and exploited in the proof of short-time existence of the flow in Chapter \ref{LEchapter}. For instance, the condition $d^*_g b = 0$ corresponds to a choice of gauge-fixing rendering $b$ $L^2$-orthogonal to the infinitesimal action of inner $B$-field symmetries, given by the image of $T^*$ via \eqref{eq:Psimap}, in the space of generalized metrics (see Remark \ref{r:infisometry}).

In a slight abuse of terminology, the one-parameter family of pairs $(g_t, H_t = H_0 + d b_t)$ will also be referred to as a solution of generalized Ricci flow.  In the next lemma we record the evolution equation for $H$.

\begin{lemma} \label{l:Hevolution}  Given $M^n$ a smooth manifold, $H_0 \in 
\Lambda^3 T^*$, $d H_0 = 
0$, and $(g_t,b_t)$ a solution of generalized Ricci flow, one has
\begin{align*}
 \left( \dt - \gD_{d,g_t} \right) H =&\ 0.
\end{align*}
\begin{proof} Using the flow equations of (\ref{f:GRF}) and the fact that $H_0$ 
is 
closed we compute
\begin{align*}
\dt H =&\ \dt \left(H_0 + d b \right) = - d d^*_{g_t} H = \gD_{d,g_t} H,
\end{align*}
as required.
\end{proof}
\end{lemma}

\begin{rmk} \label{r:supersolutions} The tensor $H^2$ is positive semidefinite, thus it follows that a solution to generalized Ricci flow is automatically a \emph{supersolution} of Ricci flow, i.e.
\begin{align*}
\dt g \geq&\ -2 \Rc.
\end{align*}
Many results have appeared in recent years concerning Ricci flow supersolutions (cf. for instance \cite{HaslhoferNaber,KopferSturm, McCannTopping, SturmSRF}), and thus all such results automatically apply for generalized Ricci flow.
\end{rmk}

\subsection{Generalized formulation}

It is useful to express the generalized Ricci flow directly in terms of the language of generalized geometry.  To begin we give an elementary lemma which expresses the generalized Ricci flow in terms of the 
curvature of the Bismut connection.

\begin{lemma} \label{l:BismutGRF} Fix $M^n$ a smooth manifold and $H_0 \in \Lambda^2 T^*$, $d H_0 = 0$. Then $(g_t, b_t)$ is a solution of generalized Ricci flow if and only if
\begin{align*}
\dt (g - b) (X,Y)=&\ - 2 \Rc^+(X,Y),
\end{align*}
where $\Rc^+$ denotes the Ricci curvature tensor associated to 
the Bismut connection $\N^+ = \N + \tfrac{1}{2} g^{-1} H$.
\begin{proof} This follows directly from the evolution equations (\ref{f:GRF}) 
and 
the formulas for the symmetric and skew-symmetric parts of the Bismut Ricci 
curvature exhibited in Proposition \ref{p:Bismutcurvature}.
\end{proof}
\end{lemma}

These formulations make it clear that the generalized Ricci flow equations 
involve precisely the data of a generalized metric on an exact Courant 
algebroid, namely a Riemannian metric and a skew-symmetric 
two-form.  Furthermore, it is clear from Lemma \ref{l:BismutGRF} that the fixed points of the generalized Ricci flow are precisely Bismut Ricci-flat structures.  It is thus natural to express the flow equations directly in these terms, which is carried out next.  

We fix a generalized metric $\GG_0$ on an exact Courant algebroid $E$ over the smooth manifold $M$. Using the splitting determined by $\GG_0$ (see Proposition \ref{l:Gmetricabsgeom}), we can identify $E = T \oplus T^*$ endowed with the $H_0$-twisted Courant bracket. Given now a one-parameter family of generalized metrics $\GG_t$ on $E$, via Proposition \ref{p:genmetricreduction} we can identify $\GG_t$ with a one-parameter family of pairs $(g_t,b_t)$. Recall from Lemma \ref{l:symmetryaction} that we have the following explicit formula for the derivative of $\GG_t$ along the family 
\begin{equation}\label{eq:Phiexpbis}
\GG^{-1}\dt \GG = \left(\begin{matrix}
                    g^{-1} h & g^{-1} k g^{-1}\\
                    - k & - h g^{-1}
                    \end{matrix} \right),
\end{equation}
where
\begin{align*}
\dt g =&\ h, \qquad \dt b = k.
\end{align*}
Notice that $\GG^{-1} = \GG$ by definition of generalized metric. It is also important to observe that \eqref{eq:Phiexpbis} is written in the isotropic splitting induced by the metric $\GG_t$.

To state the next result, we denote $\gRc(\GG) = \gRc(\GG,\divop^\GG)$ the generalized Ricci tensor associated to a generalized metric $\GG$ via its Riemannian divergence (see Definition \ref{d:generalizedRicci} and Definition \ref{d:GRiemnniandiv}).

\begin{prop} \label{p:GRFonCA} Given $(E,[,], \IP{,}) \to M$ an exact Courant algebroid, a 
one-parameter family $\GG_t$ of generalized metrics satisfies
\begin{align} \label{f:CourantGRF}
\GG^{-1} \dt \GG =&\ -2 \gRc(\GG)
\end{align}
if and only if the one-parameter family of pairs $(g_t,b_t)$ determined by $\GG_t$ and the initial splitting associated to $\GG_0$ satisfies generalized Ricci flow.
\begin{proof}
Notice that $\gRc(\GG)$ defines a tangent vector to the space of generalized metrics at $\GG$ by Lemma \ref{l:tangentGG1} and Lemma \ref{lem:Ricciskew}. Then, \eqref{f:CourantGRF} follows directly from Lemma \ref{l:R-exp}, Lemma \ref{l:BismutGRF}, and Proposition \ref{prop:Ricciexplicit}. 
\end{proof}
\end{prop}

\begin{rmk}
The condition $\GG \gRc(\GG) = - \gRc(\GG)\GG$ in Proposition \ref{prop:Ricciexplicit} shows that we can alternatively write the generalized Ricci flow as (see \cite{GF19,StreetsTdual})
$$
\dt \GG = [\gRc(\GG),\GG].
$$
\end{rmk}

\section{Examples}

Before delving into the analysis of generalized Ricci flow, we record some 
fundamental examples which give context to the analysis to follow and also 
indicate the role the torsion tensor $H$ plays in affecting the qualitative behavior of 
the flow.  We first record some terminology for certain kinds of entire solutions.

\begin{defn} \label{d:ancientsolutions} We say that a solution $(g_t, H_t)$ to generalized Ricci flow is
\begin{enumerate}
\item \emph{ancient} if it is defined on $M \times (-\infty, 0)$,
\item \emph{immortal} if it is defined on $M \times (0,\infty)$,
\item \emph{eternal} if it is defined on $M \times (-\infty, \infty)$.
\end{enumerate}
\end{defn}

\begin{ex} \label{e:shrinkingsphere} Consider $(S^3, g_{S^3})$ the 
three-dimensional sphere with round metric of constant sectional curvature $1$. 
 As noted in Example \ref{e:S3S1E}, this metric satisfies $\Rc^{g_{S^3}} = 2 
g_{S^3}$.  What is more, due to the scale invariance of the Ricci tensor, for 
any $\gl > 0$ one has $\Rc^{\gl g_{S^3}} = 2 g_{S^3}$.  Fix now a constant 
$\eta \in \mathbb R$ and let $H = \eta dV_{g_{S^3}}$.  Note that for any pair 
of the form $(g,H) = (\gl g_{S^3}, \eta dV_{g_{S^3}})$ one has that $d^*_g H = 
d^*_{\gl g_{S^3}} \left( \eta dV_{g_{S^3}} \right) = 0$, and $H^2 = 2 
\tfrac{\eta^2}{\gl^2} g_{S^3}$, hence
\begin{align} \label{f:shrinkingsphere10}
\Rc^B = \Rc^{g} - \tfrac{1}{4} H^2 - \tfrac{1}{2} d^*_{g} H = \left(2 - 
\tfrac{1}{2} \tfrac{\eta^2}{\gl^2} \right) g_{S^3}.
\end{align}
We now construct solutions to the generalized Ricci flow within this ansatz, 
i.e. $(g_t, H_t) = (\gl_t g_{S^3}, \eta_t dV_{g_{S^3}})$.  Using 
(\ref{f:shrinkingsphere10}), we see that it will always be the case that $\dt 
\eta = 0$, and so our flow reduces to the evolution of $\gl$, which we obtain 
from (\ref{f:GRF}) and (\ref{f:shrinkingsphere10}) via
\begin{align*}
\left( \dt \gl \right) g_{S^3} = \dt \left( \gl g_{S^3} \right) = - 2 \Rc + 
\tfrac{1}{2} H^2 = \left(-4 + \frac{\eta_0^2}{\gl^2} \right) g_{S^3},
\end{align*}
which implies
\begin{align} \label{f:shrinkingsphere20}
\dt \gl = -4 + \frac{\eta_0^2}{\gl^2}.
\end{align}

Here the analysis splits into two cases, whose behavior is markedly different.  
In particular, in the case $\eta_0 = 0$, we simply obtain $\dt \gl = -4$, and 
so the solution to the flow exists on the time interval $[0,\frac{\gl_0}{4})$, with 
explicit solution $(g_t, H_t) = \left( (\gl_0 - 4 t) g_{S^3}, 0 \right)$.  Note 
that, as $H = 0$ along the flow, the generalized Ricci flow has reduced to the 
classical Ricci flow equation, and this example indicates one of the basic 
qualitative behaviors of the Ricci flow, namely that the round sphere shrinks 
homothetically at a constant rate until its diameter goes to zero at the finite 
time $\frac{\gl_0}{4}$.  This solution in fact can be extended for all negative times, giving a basic example of an ancient solution.

In the case $\eta_0 \neq 0$, i..e when the $H$ field is present, the 
generalized Ricci flow behaves quite differently.  In particular, we see in 
this case that there is a unique fixed point of (\ref{f:shrinkingsphere20}), 
namely when $\gl = \frac{\brs{\eta_0}}{2}$.    Furthermore, this fixed point is 
attractive in the sense that for $\gl < \frac{\brs{\eta_0}}{2}$ one has $\dt 
\gl > 0$ and for $\gl > \frac{\brs{\eta_0}}{2}$ one has $\dt \gl < 0$.  
Elementary calculus arguments show that for any choice of $\gl_0$,
the solution to (\ref{f:shrinkingsphere20}) exists on $[0,\infty)$, and $\gl_t$ 
converges as $t \to \infty$ to the value $\frac{\brs{\eta_0}}{2}$.  Thus we see 
a striking difference in the qualitative behavior in this example depending on 
the presence of $H$, where in some rough sense the presence of $H$ 
attenuates the otherwise singular behavior of the flow, turning what was 
once a finite time singularity into a flow with infinite existence time and 
convergence at infinity.  One may argue that, in the case when $H = 0$, a 
simple rescaling in space and time to normalize the volume along the flow 
suffices to recover global existence and convergence at infinity to $g_{S^3}$, 
but nonetheless this example gives a sense of what can be expected 
conjecturally in more general cases, where the overall behavior is not as 
elementary.  We note that, if $\gl - \frac{\brs{\eta_0}}{2} > 0$ initially, the solution extends backwards in time to be an eternal solution, otherwise it will only be immortal.
\end{ex}

\begin{ex} \label{e:hyperbolic} Consider $(M^3, g_{\mbox{\tiny{Hyp}}})$ a 
compact hyperbolic $3$-manifold with constant sectional curvature $-1$.  Then 
$\Rc^{g_{\mbox{\tiny{Hyp}}}} = -2 g_{\mbox{\tiny{Hyp}}}$.  Fix a constant $\eta 
\in \mathbb R$ and let $H = \eta dV_{g_{\mbox{\tiny{Hyp}}}}$.  Note that for 
any pair of the form $(g,H) = (\gl g_{S^3}, \eta dV_{g_{S^3}})$ one has that 
$d^*_g H = d^*_{\gl g_{\mbox{\tiny{Hyp}}}} \left( \eta 
dV_{g_{\mbox{\tiny{Hyp}}}} \right) = 0$, and $H^2 = 2 \tfrac{\eta^2}{\gl^2} 
g_{\mbox{\tiny{Hyp}}}$, hence
\begin{align} \label{f:hyperbolic10}
\Rc^B = \Rc^{g} - \tfrac{1}{4} H^2 - \tfrac{1}{2} d^*_{g} H = \left(-2 - 
\tfrac{1}{2} \tfrac{\eta^2}{\gl^2} \right) g_{\mbox{\tiny{Hyp}}}.
\end{align}
We now construct solutions to the generalized Ricci flow within this ansatz, 
i.e. $(g_t, H_t) = (\gl_t g_{\mbox{\tiny{Hyp}}}, \eta_t 
dV_{g_{\mbox{\tiny{Hyp}}}})$.  Using (\ref{f:hyperbolic10}), we see that it 
will always be the case that $\dt \eta = 0$, and so our flow reduces to the 
evolution of $\gl$, which we obtain from (\ref{f:GRF}) via
\begin{align*}
\left( \dt \gl \right) g_{\mbox{\tiny{Hyp}}} = \dt \left( \gl 
g_{\mbox{\tiny{Hyp}}} \right) = - 2 \Rc + \tfrac{1}{2} H^2 = \left(4 + 
\frac{\eta_0^2}{\gl^2} \right) g_{\mbox{\tiny{Hyp}}},
\end{align*}
which implies
\begin{align*}
\dt \gl = 4 + \frac{\eta_0^2}{\gl^2}.
\end{align*}
It follows from elementary arguments that for arbitrary $\gl_0 > 0$ this 
solution exists for $[0,\infty)$, and $\gl_t$ is asymptotic to $4t$ in the 
sense that $\lim_{t \to \infty} \frac{\gl_t}{t} = 4$.  These solutions extend backwards to be immortal, but are not eternal solutions.

We now address the issue of how to appropriately renormalize this flow to 
obtain some form of convergence.  In particular, we set
\begin{align*}
\til{g}_t = \frac{g_t}{t}, \qquad \til{H}_t = \frac{H_t}{t},
\end{align*}
and define a new time parameter via $s = \log t$.  Then we compute
\begin{align*}
\frac{\del}{\del s} \til{g} =&\ t \frac{\del}{\del t} \frac{g_t}{t} = \frac{\del g}{\del t} - \frac{g_t}{t} = -2 \Rc + \tfrac{1}{2} H^2 - \til{g} = -2 \Rc^{\til{g}} + \tfrac{1}{2} \til{H}^2 - \til{g},
\end{align*}
where the last equality follows since the Ricci tensor is scale invariant, as is the quantity $H^2$.  A similar calculation yields
\begin{align*}
\frac{\del}{\del s} \til{H} = \til{\gD}_d \til{H} - \til{H}.
\end{align*}
Thus this scaling has the effect of adding normalizing terms to the original equations.  Since $H$ was fixed under the original flow, it follows easily that $\til{H} \to 0$, whereas since $\frac{\gl_t}{t}$ approaches $4$, we will obtain $\til{g} \to 4 g_{\mbox{\tiny{Hyp}}}$.
\end{ex}

\begin{ex} \label{e:sphereproduct} The presence of the torsion term $H$ will of 
course not remove all singular 
behavior.  One only expects the term $H$ to prevent the collapse of the length 
of vectors $v$ for which $v \lrcorner H \neq 0$, a condition which is difficult 
to verify in general along a solution to the flow, but which we can illustrate 
with a simple example.  In particular, as in Example \ref{e:S3S1E}, on $S^3$, 
the structure $(g,H) = (g_{S^3},\tfrac{1}{2} dV_{g_{S^3}})$ is generalized 
Einstein, and so a fixed point of generalized Ricci flow.  However, consider $M 
= S^3 \times S^3$, with $g = \pi_1^* g_{S^3} + \pi_2^* g_{S^3}$, and $H = 
\tfrac{1}{2} \pi_1^*(dV_{g_{S^3}})$.  Elementary arguments show that the 
generalized Ricci flow with this initial condition will behave as the product 
of the solutions on the factors of $S^3$.  In particular, the first factor will 
remain fixed as it is generalized Einstein, and the second factor will 
homothetically shrink to a point in finite time as described in Example 
\ref{e:shrinkingsphere}.
\end{ex}

\begin{ex} \label{e:neckpinch} A classical singularity model of the Ricci flow is the shrinking cylinder, which models regions of Ricci flow experiencing a neckpinch singularity.  Let $M = S^n \times \mathbb R$, and consider the one-parameter family of metrics
\begin{align*}
g_t = - 2(n-1) t g_{S^n} \oplus ds^2,
\end{align*}
defined for $t < 0$.  Since $\Rc_{g_{S^n}} = (n-1) g_{S^n}$ and the metric $ds^2$ is flat, it follows that $g_t$ is a solution of Ricci flow.  Thus this solution is ancient, and is meant to model a collapsing neck.

The dynamics of necks under generalized Ricci flow can be quite different depending on the dimension and the choice of $H$.  First consider $S^2 \times \mathbb R$, let $H \equiv dV_{S^2} \wedge ds$, fix constants $\phi, \psi > 0$ and consider a metric $g$ of the form
\begin{align*}
g =&\ \phi g_{S^2} \oplus \psi ds^2.
\end{align*}
Note that the Ricci tensor for any such metric is
\begin{align*}
\Rc = g_{S^2}.
\end{align*}
Further direct computation shows that for these choices of $g$ and $H$ one has
\begin{align*}
H^2 =&\ \phi^{-1} \psi^{-1} g_{S^2} \oplus \phi^{-2} ds^2.
\end{align*}
As $H$ is harmonic with respect to any metric in the given ansatz, the generalized Ricci flow can thus be reduced to the system of ODE
\begin{align*}
\dt \phi =&\ -2 + \tfrac{1}{2} \phi^{-1} \psi^{-1},\\
\dt \psi =&\ \tfrac{1}{2} \phi^{-2}.
\end{align*}
We leave as an \textbf{exercise} to show that, for any positive initial values of $\phi$ and $\psi$, the solution to this system of ODEs exists on a time interval of the form $[0, \infty)$, where
\begin{align*}
\lim_{t \to \infty} \phi = 0, \qquad \lim_{t \to \infty} \psi = \infty.
\end{align*}
In other words, the neck can still collapse, although it takes infinitely long to do so.

On the other hand if we consider $S^3 \times \mathbb R$, with $H = dV_{S^3}$ and $g = \gl g_{S^3} \oplus ds^2$, then the solution to generalized Ricci flow will split according to the product structure, and the dynamics of the $S^3$ portion are the same as those described in the case of $S^3$ in Example \ref{e:shrinkingsphere}, and thus the neck stabilizes, yielding convergence to a cylindrical metric.
\end{ex}

To finish this section we analyze the generalized Ricci flow of all left-invariant metrics on $S^3 \cong \SU(2)$.  Before doing this we set up some general remarks on the generalized Ricci flow of left-invariant structures on compact semisimple Lie groups.  First, as the unique left-invariant representative of any cohomology class is harmonic with respect to any left-invariant metric, it follows that $H$ will remain fixed along any such solution.  Furthermore, as indicated by the example of round $S^3$ above, the initial cohomology class chosen for $H$ can have a significant effect on the qualitative behavior of the generalized Ricci flow.  In fact, even for the case of Ricci flow ($H \equiv 0$), there is not a general picture of the behavior of Ricci flow in this setting.  However, as described in Proposition \ref{p:Bismutflat}, there are canonical Bismut-flat structures given by the unique bi-invariant metric $g$ together with torsion $H = \pm g ([\cdot,\cdot],\cdot)$.

\begin{conj} \label{c:ssconj} Let $G$ denote a compact semisimple Lie group with bi-invariant metric $g_{\infty}$.  Given $g_0$ a left-invariant metric and
\begin{align*}
H_0 = \pm g_{\infty}( [\cdot,\cdot],\cdot),
\end{align*}
the generalized Ricci flow with initial condition $(g_0, H_0)$ exists on $[0,\infty)$ and converges to the Bismut-flat structure $(g_{\infty}, H_{0})$.
\end{conj}

\begin{ex} We prove Conjecture \ref{c:ssconj} in the case of $S^3 \cong \SU(2)$.  Recall we constructed generalized Ricci flat structures on $\SU(2) \times S^1$ in Example \ref{ex:GRicciflatHopf}, using a left-invariant frame.  Here we choose a slightly different frame, in particular we choose left-invariant vector fields $X_1, X_2, X_3$ such that
\begin{align*}
[X_1, X_2] = - 2 X_3, \qquad [X_2, X_3] =&\ - 2 X_1, \qquad [X_3, X_1] = - 2 X_2.
\end{align*}
Such a basis is called a Milnor frame, and we refer to \cite{Milnor} for further information on the geometry of left-invariant metrics on Lie groups.  This frame induces a global coframe $\{\mu^i\}$, and we can consider left-invariant metrics of the form
\begin{align*}
g =&\ A \mu^1 \otimes \mu^1 + B \mu^2 \otimes \mu^2 + C \mu^3 \otimes \mu^3.
\end{align*}
The bi-invariant metrics occur exactly when $A = B = C$.  An \textbf{exercise} shows that the only nonvanishing components of the Ricci tensor are
\begin{align*}
\Rc(X_1, X_1) =&\ \frac{2}{BC} \left( A^2 - (B - C)^2 \right),\\
\Rc(X_2, X_2) =&\ \frac{2}{CA} \left( B^2 - (C - A)^2 \right),\\
\Rc(X_3, X_3) =&\ \frac{2}{AB} \left( C^2 - (A - B)^2 \right).
\end{align*}
Furthermore we can compute, up to a positive multiple,
\begin{align*}
H_{0} =&\ \mu^1 \wedge \mu^2 \wedge \mu^3,
\end{align*}
and hence associated to $g$ and $H_{0}$ we have that the only nonvanishing components of $H_0^2$ are
\begin{align*}
H_0^2(X_1, X_1) =&\ \frac{2}{BC}, \qquad H_0^2(X_2, X_2) = \frac{2}{CA}, \qquad H_0^2(X_3, X_3) = \frac{2}{AB}.
\end{align*}
As remarked above, the generalized Ricci flow with initial condition $(g_0 = g, H_0)$ will preserve $H_t = H_0$ for all times.  It follows from the above computations that the flow reduces to the following system of ODE for the positive numbers $A, B, C$:
\begin{align*}
\dot{A} =&\ \frac{1}{BC} \left( - 4 A^2 + 4 (B - C)^2 + 1 \right),\\
\dot{B} =&\ \frac{1}{CA} \left( - 4 B^2 + 4 (C - A)^2 + 1 \right),\\
\dot{C} =&\ \frac{1}{AB} \left( - 4 C^2 + 4 (A - B)^2 + 1 \right).
\end{align*}
Without loss of generality we can asssume $A_0 \leq B_0 \leq C_0$.  A computation shows that
\begin{align*}
\frac{d}{dt} \left( C - A \right) =&\ 4 \frac{C - A}{ABC} \left( B^2 - (A + C)^2 + 1 \right).
\end{align*}
This can be used to show that $A_t \leq B_t \leq C_t$ for all times such that the solution is defined.  Furthermore, we observe that
\begin{align*}
\dot{A} \geq B^{-1} C^{-1} \left( -4 A^2 + 1 \right),
\end{align*}
so a uniform positive lower bound for $A$ is preserved for all times.  We can also compute
\begin{align} \label{f:S3conv}
\frac{d}{dt} \left( \frac{C - A}{A} \right) =&\ - 8 A^{-2} B^{-1} \left(C - A \right) \left( C + A - B \right) \leq 0.
\end{align}
It follows that there is a uniform constant $\gl > 0$ such that $C \leq \gl A$.  Using this estimate we furthermore obtain
\begin{align*}
\dot{C} \leq&\ - 8 + 4 \frac{A}{B} + \frac{1}{AB} \leq - 4 + \frac{\gl^2}{C^2}.
\end{align*}
It follows that there is a uniform upper bound for $C$.  With the uniform lower bound for $A$ and upper bound for $C$ in place, it follows that the solution to the ODE is defined for all $t > 0$.  Returning to (\ref{f:S3conv}), it follows that there is a uniform constant $\gd > 0$ so that
\begin{align*}
\frac{d}{dt} \left(\frac{C - A}{A} \right) \leq&\ - \gd \left( \frac{C - A}{A} \right),
\end{align*}
and the solution converges exponentially to $A_{\infty} = B_{\infty}= C_{\infty}$, a multiple of the bi-invariant metric.
\end{ex}

It would be interesting to carry out a similar analysis for left-invariant solutions of generalized Ricci flow in other classes of Lie groups, which help us to understand qualitative properties of the equations. The case of nilpotent Lie groups has been studied in \cite{paradiso2020generalized} (see also \cite{Boling, FinoFlow} in the setting of pluriclosed flow in Chapter \ref{c:GFCG}).

\section{Maximum principles}

Moving beyond homogeneous examples, the most basic tool in analyzing second order elliptic and parabolic partial differential equations in general is the maximum principle, and these tools are crucial in understanding generalized Ricci flow.  In this section we record statements and some proofs of elliptic and parabolic maximum principles on manifolds.

\begin{prop} \label{p:strongmax} Let $(M^n, g)$ be a compact connected Riemannian manifold and suppose $u \in C^{\infty}(M)$ satisfies
\begin{align*}
L u := \gD u + X \cdot u \leq 0,
\end{align*}
where $X \in \gG(T)$ is a vector field.  Then $u$ is constant.
\end{prop}

\begin{prop} \label{p:maxprinc1} Let $(M^n, g_t)$ be a compact manifold with a 
one-parameter family of Riemannian metrics.  Suppose $u \in C^{\infty}(M \times 
[0,T))$ satisfies
\begin{align*}
\left(\dt - \gD_{g_t} \right) u \leq 0.
\end{align*}
Then
\begin{align*}
\sup_{M \times \{t\}} u \leq \sup_{M \times \{0\}} u
\end{align*}
for all $t \in [0,T)$.
\begin{proof} Fix $\ge > 0$, and set $u_{\ge} := u - \ge(1 + t)$.  One 
obviously 
has that $\sup_{M \times \{0\}} u_{\ge} < \left( \sup_{M \times \{0\}} u 
\right) 
=: \lambda$.  We will show that $\sup_{M \times \{t\}} u_{\ge} < \gl$, which 
will finish the proposition by sending $\ge$ to zero.  If this claim was false, 
there exists a point such that $u_{\ge} \geq \gl$.  Since $M$ is compact, we 
can 
choose $(x,t) \in M \times (0,T)$ such that $u_{\ge}(x,t) = \gl$, and 
$u_{\ge}(y,s) \leq \gl$ for all $y \in M$, $0 \leq s \leq t$.  Hence $(x,t)$ is 
a spacetime maximum for $u_{\ge}$ on $M \times [0,t]$, and so it follows that 
$\gD_{g_t} u_{\ge} \leq 0$ and $\dt u_{\ge} \geq 0$ at $(x,t)$.  We conclude 
that, at $(x,t)$ we have
\begin{align*}
0 \leq&\ \dt u_{\ge} \leq \gD_{g_t} u_{\ge} - \ge < 0,
\end{align*}
a contradiction.  Hence $\sup_{M \times \{t\}} u_{\ge} < \gl$ for all $t \in 
[0,T)$, and the proposition follows.
\end{proof}
\end{prop}

For various applications one needs to consider more general reaction-diffusion equations 
beyond sub/supersolutions of the heat equation.  We next formalize a 
general 
class of equations which suffices for our purposes.

\begin{defn} \label{d:semilinearheat} Given $(M^n, g_t)$ a smooth manifold with 
a smooth one-parameter family of Riemannian metrics, $X_t \in \gG(T)$ a smooth one-parameter 
family of vector fields and $F : \mathbb R \times [0,T) \to \mathbb R$ which is
continuous in $t$ and locally Lipschitz in $x$, define the associated 
\emph{semi-linear heat operator} via
\begin{align*}
L u := \dt u - \gD_{g_t} u - X_t u - F(u,t).
\end{align*}
\end{defn}

\begin{prop} \label{p:maxprinc2} Let $(M^n, g_t)$ be a compact manifold with a 
one-parameter family of Riemannian metrics.  Suppose $L$ is a semi-linear heat 
operator as in Definition \ref{d:semilinearheat}, and suppose $u,v \in C^2(M 
\times [0,T))$ satisfy
\begin{align*}
L v \leq L u, \qquad \sup_{M \times \{0\}} \left(v - u \right) \leq 0.
\end{align*}
Then
\begin{align*}
\sup_{M \times \{t\}} \left(v - u \right) \leq 0
\end{align*}
for all $t \in [0,T)$.
\begin{proof} We set $w(x,t) = u(x,t) - v(x,t)$, and aim to show that $w \geq 0$ at all points.  Fix a time $\tau \in [0,T)$.  Since the function $F$ is locally Lipschitz in space and the region $M \times [0,\tau]$ is compact, there exists a constant $K$ such that
\begin{align*}
\brs{F(u(x,t),t) - F(v(x,t),t)} \leq K \brs{u(x,t) - v(x,t)}.
\end{align*}
for all $(x,t) \in M \times [0,\tau]$.  Now fix $\ge > 0$ and let
\begin{align*}
w_{\ge}(x,t) = u(x,t) - v(x,t) + \ge e^{2 K t}.
\end{align*}
Using the given differential inequalities we compute
\begin{align*}
\frac{\del w_{\ge}}{\del t} \geq&\ \gD w_{\ge} + \IP{X, \N w_{\ge}} - F(u,t) + F(v,t) + 2 K \ge e^{2Kt}\\
\geq&\ \gD w_{\ge} + \IP{X, \N w_{\ge}} - K \brs{w} + 2 K \ge e^{2 K t}.
\end{align*}
If ever $w_{\ge} = 0$, then arguing as in Proposition \ref{p:maxprinc1} we can choose a point $(x_0, t_0)$ such that $w_{\ge}(x_0,t_0) = 0$ and $(x_0,t_0)$ realizes the minimum for $w_{\ge}$ on $M \times [0,t_0]$.  Note then that at the point $(x_0,t_0)$, $w = - \ge e^{2 K t_0}$, $\N w_{\ge} = 0, \gD w_{\ge} \geq 0, \frac{\del w_{\ge}}{\del t} \leq 0$.  Thus we conclude
\begin{align*}
0 \geq \frac{\del w_{\ge}}{\del t} \geq - K \ge e^{2 K t_0} + 2 K \ge e^{2 K t_0} \geq K \ge e^{2 K t_0} > 0,
\end{align*}
a contradiction.  Thus $w_{\ge} > 0$ for all $\ge > 0$ and thus $w \geq 0$ on $M \times [0,\tau]$, finishing the proof.
\end{proof}
\end{prop}

We end this section with the fundamental observation that if $H$ vanishes at the initial time of a solution to generalized Ricci flow, then it vanishes for all time $t > 0$, and the one-parameter family 
of metrics evolves by the Ricci flow.  This justitifes the terminology of 
``generalized Ricci flow,'' and yields the important philosophical point that, 
without special assumptions on $H$, in the most general setting the results must be extensions of those for Ricci flow.  However as we have already seen there are 
geometrically natural settings where the presence of $H$ has important 
qualitative effects on the behavior of the flow.  

\begin{prop} \label{p:Hzeropres} Let $M^n$ be a smooth compact manifold 
and suppose $(g_t, H_t)$ is a solution to generalized Ricci flow satisfying 
$H_0 \equiv 0$.  Then for all $t$ such that the flow is defined, $H_t \equiv 
0$, and
\begin{align*}
\dt g =&\ -2 \Rc,
\end{align*}
that is, $g_t$ is a solution to Ricci flow.
\begin{proof} Suppose the solution exists on $[0,T)$, fix some $\tau < T$ and 
let
\begin{align*}
K := \sup_{M \times [0,\tau]} \brs{\Rm_g}.
\end{align*}
Using Lemma \ref{l:Hevolution} and the Bochner formula (cf. \ref{p:curvtotorsion} below)
we obtain the differential inequality
\begin{align*}
\left(\dt - \gD \right) \brs{H}^2 \leq&\ C K \brs{H}^2.
\end{align*}
In particular, if we define $\Phi := e^{-CKt} \brs{H}^2$, we easily obtain 
$\left(\dt - \gD \right) \Phi \leq 0$.  Applying the maximum principle 
(Proposition \ref{p:maxprinc1}) we see that for all $t \in [0,\tau]$ one has
\begin{align*}
\sup_{M \times \{t\}} e^{-CKt} \brs{H}^2 \leq \sup_{M \times \{0\}} e^{-CKt} 
\brs{H}^2 = 0,
\end{align*}
and so $H$ vanishes on $[0,\tau]$.  This argument applies for arbitrary $\tau < 
T$ and so $H$ vanishes identically wherever the flow is defined.  The fact that 
$g_t$ is a solution to Ricci flow now follows directly from (\ref{f:GRF}).
\end{proof}
\end{prop}

\section{Invariance group and solitons} \label{s:invgroup}

A fundamental feature of many evolution equations in geometry is the presence 
of an infinite dimensional invariance group.  For instance, the Ricci tensor of a Riemannian 
metric is diffeomorphism covariant, and so diffeomorphisms can act on solutions 
to the Ricci flow, preserving the equation.  More generally, one can pull a 
solution to Ricci flow back by a family of time-dependent diffeomorphisms.  It 
is natural to consider this as equivalent to the original flow, as the 
corresponding time slices are isometric.  However, the resulting flow equation 
now includes a Lie derivative term corresponding to the variation of the family 
of diffeomorphisms.  This circle of ideas is relevant to the generalized Ricci 
flow as well.  As discussed in \S \ref{ss:Courantsymmetries}, the relevant 
symmetry group includes $B$-field transformations.  By letting a time-dependent 
family of symmetries act on a solution to generalized Ricci flow one produces a 
more general flow equation which includes modification by Lie derivative terms 
as well as a closed two-form.

\subsection{Gauge-fixed generalized Ricci flow}

To begin we recall the derivation of the gauge-fixed Ricci flow equation from the Ricci flow. We say that $g = g_t$ satisfies the \emph{gauge-fixed Ricci flow} if there exists a one-parameter family of vector fields $X_t \in \Gamma(T)$ such that $\til g_t = (\phi_t)_*g_t$ is a solution of the Ricci flow, where $\phi_t$ is the one-parameter family of diffeomorphisms generated by $X_t$
$$
X_t \circ \phi_t = \dt \phi_t.
$$
Using the naturality of the of the Ricci tensor under diffeomorphisms, the gauge-fixed Ricci flow is equivalent to
$$
\dt \til g = {\phi_t}_*\Bigg(\dt  g - L_{X_t}g\Bigg) = -2 {\phi_t}_* \Rc(g) = -2 \Rc(\til g)
$$
and we obtain the equation
\begin{equation}\label{eq:gaugefixed}
\dt  g = - 2 \Rc  + L_{X_t}g.
\end{equation}
Following this line of thought, we can define a gauge-fixed version of generalized Ricci flow.

\begin{defn} Given $E$ an exact Courant algebroid, a one-parameter family of generalized metrics $\GG_t$ satisfies the \emph{gauge-fixed generalized Ricci flow} if there exists a one-parameter family of generalized diffeomorphisms $F_t \in \Aut(E)$ such that
$$
\til \GG = (F_t)_*\GG = F_t \GG_t F_t^{-1}
$$ 
is a solution of the generalized Ricci flow. 
\end{defn}

\begin{prop} \label{p:GFF} 
Given $H_0 \in \Lambda^3 T^*$ such that $dH_0 = 0$, consider the exact Courant algebroid $E = (T\oplus T^*, \IP{,}, [,]_{H_0},\pi)$. Then, a one-parameter family of generalized metrics $\GG_t$ satisfies the gauge-fixed generalized Ricci flow if and only if there exists a family of vector fields $X_t$ and $B_t \in \Lambda^2$ satisfying
\begin{align*}
d \left(B_t + i_{X_t} H_0 \right) = 0,
\end{align*}
such that the associated one-parameter family of pairs $(g_t, b_t)$ satisfies
\begin{gather} \label{f:GFGRFdecomp}
\begin{split}
\dt g =&\ -2 \Rc + \tfrac{1}{2} H^2 + L_{X} g\\
\dt b =&\ - d^*_g H - B + L_X b.
\end{split}
\end{gather}
\begin{proof} Using the naturality of the generalized Ricci tensor $\gRc$, we first observe that the generalized Ricci flow equation is equivalent to
$$
\til \GG_t^{-1}\dt \til \GG = {F_t}_*\Bigg(\GG^{-1}\dt  \GG + \GG^{-1}\zeta_t\cdot \GG \Bigg) = -2 {F_t}_* \gRc(\GG) = -2 \gRc(\til \GG)
$$
where $\zeta_t \cdot \GG$ denotes the infinitesimal left action of the Lie algebra element
$$
\zeta_t = F_t^{-1}\dt F_t \in \operatorname{Lie} \Aut(E).
$$
Thus, $\GG = \GG_t$ is a solution of the gauge-fixed generalized Ricci flow if and only if there exists a one-parameter family of infinitesimal automorphisms $\zeta_t \in \operatorname{Lie} \Aut(E)$ such that
\begin{equation}\label{eq:RFgaugefixedabs}
\GG^{-1}\dt  \GG = - 2 \gRc(\GG) - \GG^{-1}\zeta_t\cdot \GG.
\end{equation}
We use now $E = (T\oplus T^*, \IP{,}, [,]_{H_0},\pi)$ to write this condition explicitly. We have (see Corollary \ref{c:Liegendiff})
$$
\operatorname{Lie} \Aut(E) = \{ X + B \; | \; L_XH_0 = - dB \} = \{ X + B \; | \; d(i_XH_0 + B) = 0 \}.
$$
A one-parameter family $\zeta_t = X_ t + B_t \in \operatorname{Lie} \Aut(E)$ can be integrated to 
$$
F_t = \overline{\phi}_t \circ e^{\overline{B}_t} \in \Aut(E)
$$
where $\phi_t$ is the one-parameter family of diffeomorphisms generated by $X_t$ and 
$$
\overline{B}_t = \int_0^t \phi_s^*B_s ds.
$$
Now, if $\GG = \GG(g,b)$ as in Proposition \ref{p:genmetricreduction}, it follows that
$$
{F_t}_*\GG(g,b) = \GG({\phi_t}_*g,{\phi_t}_*(b + \overline B_t)),
$$
and therefore
\begin{equation*}
\GG^{-1}\zeta_t \cdot \GG  = \left(\begin{matrix}
                    -g^{-1} L_{X_t}g & - g^{-1} (L_{X_t}b - B_t) g^{-1}\\
                    L_{X_t}b - B_t & L_{X_t}g g^{-1}
                    \end{matrix} \right),
\end{equation*}
where the previous formula is written in the isotropic splitting given by the metric $\GG_t$. From this, we conclude that the gauge-fixed generalized Ricci flow is given by
\begin{gather} \label{f:GRFgauge}
\begin{split}
\dt g =&\ -2 \Rc + \tfrac{1}{2} H^2 + L_{X_t}g,\\
\dt b =&\ - d^*_g H - B_t + L_{X_t}b,
\end{split}
\end{gather}
where $H = H_0 + db$, as required.
\end{proof}
\end{prop}

\begin{rmk} \label{r:GFGRF} Note that in the formulation above, $B_t$ is not necessarily closed.  It is possible to rephrase the above gauge-fixed equation in terms of the action of a closed $B$-field.  Namely set
$$
k_t = - B_t - i_{X_t}H_0 + d(i_{X_t}b).
$$
Then
$$
- B_t + L_{X_t}b = k_t + i_{X_t}H_0 + i_{X_t}db = k_t + i_{X_t}H,
$$
and we obtain the following alternative description of the flow:
\begin{gather}\label{f:gaugefixedgrf}
\begin{split}
\dt g =&\ -2 \Rc + \tfrac{1}{2} H^2 + L_{X_t}g,\\
\dt b =&\ - d^*_g H + i_{X_t}H + k_t.
\end{split}
\end{gather}
We will at times refer to this as \emph{$(X_t,k_t)$-gauge-fixed generalized Ricci flow}.
Observe that in either formulation the evolution of the three-form is 
$$
\dt H = \dt (H_0 + db) = - d d^*_g H + L_{X_t}H.
$$
\end{rmk}

It turns out that it is possible to interpret the gauge-fixed flow naturally in terms of the generalized Ricci tensor $\gRc^+$ associated to a particular choice of divergence operator.

\begin{prop} \label{p:GFGRFinGG+} Let $E$ be an exact Courant algebroid over a smooth manifold $M$. Let $X_t$ be a one-parameter family of vector fields $X_t$ on $M$. Then, a one-parameter family $\GG_t$ of generalized metrics satisfies
$$
\GG^{-1}\dt \GG \circ \pi_-  = - 2 \gRc^+(\GG,\divop) \circ \pi_-,
$$
where $\divop = \divop^{\GG} + \IP{2i_{X_t} g_t,}$, if and only if the one-parameter family of pairs $(g_t,b_t)$ determined by $\GG_t$ and the initial splitting associated to $\GG_0$ satisfies the $(X_t,k_t)$-gauge-fixed generalized Ricci flow \eqref{f:gaugefixedgrf} with $k_t = - d(i_{X_t} g_t)$.
\end{prop}

\begin{proof}
Let $\varphi \in T^*$ a one-form. From Proposition \ref{prop:Ricciexplicit}, the generalized Ricci tensor $\gRc^+$ with divergence $\divop = \divop^{\GG} - \IP{2\varphi,}$ is given by 
\begin{equation*}
\begin{split}
\gRc^+(\sigma_- X, \sigma_+ Y) = \Big{(} \Rc - \tfrac{1}{4} H^2 + \tfrac{1}{2}L_{\varphi^\sharp}g - \tfrac{1}{2}d^*H + \tfrac{1}{2}d \varphi - \tfrac{1}{2} i_{\varphi^\sharp} H \Big{)}(X,Y),
\end{split}
\end{equation*}
where we have used the identity
$$
\nabla^+ \varphi = \tfrac{1}{2}L_{\varphi^\sharp}g + \tfrac{1}{2} d \varphi - \tfrac{1}{2} i_{\varphi^\sharp}H.
$$
Setting $\varphi = - i_{X_t} g_t$ and identifying the generalized Ricci tensor $\gRc^+$ with a map $\gRc^+ \colon V_- \to V_+$, the statement follows combining equation \eqref{f:gaugefixedgrf} with Lemma \ref{l:R-exp}.
\end{proof}

For the case of a gradient vector field the divergence $\divop = \divop^{\GG} + \IP{2i_{X_t} g_t,}$ is exact (see Definition \ref{def:exact}) and we can express the gauge-fixed generalized Ricci flow \eqref{f:gaugefixedgrf} as generalized Ricci flow with divergence operator.

\begin{prop} \label{p:GFGRFinGG} 
Let $E$ be an exact Courant algebroid over a smooth manifold $M$. Let $f_t$ be a one-parameter family of functions on $M$. Then, a one-parameter family $\GG_t$ of generalized metrics satisfies
$$
\GG^{-1}\dt \GG  = - 2 \gRc(\GG,\divop^{\mu_t}),
$$
where $\mu_t = e^{-f_t} dV_{g_t}$, if and only if the one-parameter family of pairs $(g_t,b_t)$ determined by $\GG_t$ and the initial splitting associated to $\GG_0$ satisfies the $(- \N f_t,0)$-gauge-fixed generalized Ricci flow \eqref{f:gaugefixedgrf}.
\end{prop}

\begin{proof}
Recall from Definition \ref{def:exact} that $\divop^{\mu_t} = \divop^\GG - \IP{2df,}$. Thus, by Proposition \ref{p:GFGRFinGG+} and Lemma \ref{lem:Ricciskew}, $\GG_t$ satisfies generalized Ricci flow with divergence operator $\divop^{\mu_t}$ as above if and only if $(g_t,b_t)$ satisfies \eqref{f:gaugefixedgrf} with $(X_t,k_t) = (- \N f_t,0)$.
\end{proof}

\subsection{Generalized Ricci solitons}

In understanding geometric evolution equations, a key role is played by the fixed points of the equation in determining the possible existence and regularity behavior of the flow.  Initially we would say that the fixed points of the equation must satisfy $\gRc \equiv 0$, and so are generalized Ricci flat structures.  More generally though, it is possible for the flow to evolve along a one-parameter family of Courant symmetries, so that while the pairs $(g, H)$ are indeed evolving, at each time slice they are geometrically indistinguishable, and these solutions are called solitons.

To begin we recall the derivation of the Ricci soliton equations.  In particular, Ricci solitons arise as special solutions of the gauge-fixed Ricci flow, upon imposing two conditions: the first is that $g_t$ evolves in a self-similar way, that is, by homotheties via the \emph{canonical dilation} of a fixed metric
$$
g_t = (1 - 2 t\lambda)g.
$$
Here $\lambda \in \{-1,0,1\}$ is a constant which determines the character of the soliton. Second is that
$$
X_t = \frac{-1}{1 - 2\lambda t} X
$$ 
for a vector field $X$ which is constant in time.  From these conditions and \eqref{eq:gaugefixed}, we get the \emph{Ricci soliton} equation
$$
\Rc + \tfrac{1}{2} L_X g = \lambda g,
$$
which is \emph{expanding, steady}, or \emph{shrinking} depending on whether $\gl = -1,0,1$, respectively.

To adapt these ideas to generalized Ricci flow, we first observe that the space of generalized metrics is not invariant under scaling.  However, we can define a canonical variation which corresponds to scaling the underlying metric.

\begin{defn}\label{d:dilation}
Let $\lambda \in \{-1,0,1\}$. The canonical variation of a generalized metric $\GG$ on $E$ is given by
\begin{equation*}
\GG_{t,\lambda} = \left(\begin{matrix}
                    0 & (1 - 2 t\lambda)^{-1}g^{-1}\\
                    (1 - 2 t\lambda)g & 0
                    \end{matrix} \right)
\end{equation*}
in the splitting determined by $\GG$.  Note that
$$
 \GG\GG_{t,\lambda} = \left(\begin{matrix}
                    1 - 2 t\lambda & 0\\
                   0 &  (1 - 2 t\lambda)^{-1}
                    \end{matrix} \right)
$$
is the canonical dilation in the splitting determined by $\GG$.
\end{defn}

\begin{defn} \label{d:GRSdef1} Given $(E, [,], \IP{,})$ an exact Courant algebroid, we say that $\GG$ is a \emph{generalized Ricci soliton} if there exists  $\lambda \in \{-1,0,1\}$ such that the canonical variation $\GG_{t,\lambda}$ solves the gauge-fixed generalized Ricci flow \eqref{eq:RFgaugefixedabs}
with
$$
\zeta_t = \frac{-1}{1 - 2\lambda t} (X + B),
$$ 
where $X + B \in \operatorname{Lie} \Aut(E)$ is constant in time.  We again call the soliton \emph{expanding, steady}, or \emph{shrinking} depending on whether $\gl = -1,0,1$, respectively.
\end{defn}

\begin{prop} \label{p:GRSprop} Given $(E, [,], \IP{,})$ an exact Courant algebroid, suppose $\GG$ is a generalized Ricci soliton.  If $\gl \neq 0$, then $H \equiv 0$ and the associated metric $g$ is an expanding or shrinking Ricci soliton.  If $\gl = 0$ then one has the equivalent equations
\begin{gather} \label{f:solitonexp}
\begin{split}
0 =&\ \Rc - \tfrac{1}{4} H^2 + \tfrac{1}{2} L_{X}g,\\
0 =&\ \tfrac{1}{2} d^*_g H - \tfrac{1}{2} B,
\end{split}
\end{gather}
for $B$ satisfying
$$
d(B + i_X H) = 0.
$$
Thus, in particular,
$$
0 = - \tfrac{1}{2} \Delta_d H + \tfrac{1}{2} L_X H.
$$
\begin{proof} Applying the equations \eqref{f:GRFgauge}, we see that $\GG$ is a generalized Ricci soliton if there exists $\lambda \in \{-1,0,1\}$ such that
\begin{gather*}
\begin{split}
\lambda g =&\ \Rc - \frac{1}{4(1 - 2\lambda t)^{2}} H^2 + \tfrac{1}{2} L_{X}g,\\
0 =&\ \frac{1}{1 - 2 \gl t} \left(d^*_g H - B \right),
\end{split}
\end{gather*}
for all $t$, where $X + B \in  \operatorname{Lie} \Aut(E)$.  If $\lambda \neq 0$ then we obtain
$$
\lambda g = \Rc + \tfrac{1}{2} L_{X}g, \quad H= 0,
$$
that is, it is a Ricci soliton in the usual sense.  The claimed equations (\ref{f:solitonexp}) in the case $\lambda = 0$ follow directly, and by taking the exterior derivative of the second equation of (\ref{f:solitonexp}) we obtain the final claimed equation.
\end{proof}
\end{prop}

\begin{rmk} Thus we see that in our formulation, shrinking and expanding generalized Ricci solitons automatically have vanishing torsion.  This arises in part because there is no obviously natural way to introduce a scaling of the $b$-field associated to a generalized metric.  Nonetheless, in taking certain rescalings or reparameterizations of generalized Ricci flow, it is reasonable to consider flows for which the pairs $(g, H)$ might both scale homothetically (cf. Definition \ref{d:nonsingular} below for nonsingular solutions).  Although, we are only aware of rigidity results showing that solitons for such flows automatically have $H \equiv 0$.
\end{rmk}

\begin{rmk}
The dilation $\GG\GG_{t,\lambda}$ in Definition \ref{d:dilation} has appeared in \cite{corts2017twist} under the name of \emph{conformal change}. Here we make the observation that, in order to be invariant, this transformation depends upon a choice of isotropic splitting on $E$.
\end{rmk}

%\begin{ex} On $\mathbb R^n$, the Euclidean metric $g_E$ can be interpreted as a shrinking or expanding Ricci soliton by choosing $X = \pm \tfrac{1}{2} \N \brs{x}^2$.  On $\mathbb R^3$ we can construct a steady generalized Ricci soliton.  Let $g = g_E$ be the Euclidean metric, let $H = dV_{g_E} = dx \wedge dy \wedge dz$, and set $X = \tfrac{1}{2} \N \brs{x}^2$.  It follows that
%\begin{align*}
%\Rc - \tfrac{1}{4} H^2 + \tfrac{1}{2} L_X g =&\ 0 - \tfrac{1}{2} g + \tfrac{1}{2} g = 0.
%\end{align*}
%Furthermore, we set 
%\begin{align*}
%b =&\ \tfrac{1}{3} X \hook H = \tfrac{1}{3} x dy \wedge dz - y dx \wedge dz + z dx \wedge dy,\\
%B =&\ - 3b.
%\end{align*}
%Since $d^*_{g_E} dV_{g_E} = 0$ and $db = H$, we conclude
%\begin{align*}
%d^*_g H + B + L_X b = -3b + i_X db = 0,
%\end{align*}
%as required.
%\end{ex}

By adapting the point of view of Remark \ref{r:GFGRF}, it is possible to give a slightly different formulation for the generalized Ricci soliton equations which involve the action of a closed $B$-field instead.  We state this as an equivalent definition, as it is a consequence of the above discussion.

\begin{defn} \label{d:GRSdef2}Given $(E, [,], \IP{,})$ an exact Courant algebroid, a generalized metric $\GG$ is a \emph{steady generalized Ricci soliton} if there exists a vector field $X$ and a closed two-form $k$ such that
\begin{gather} \label{f:solitonexp2}
\begin{split}
0 =&\ \Rc - \tfrac{1}{4} H^2 + \tfrac{1}{2} L_{X}g,\\
0 =&\ \tfrac{1}{2} d^*_g H + \tfrac{1}{2} i_X H + k.
\end{split}
\end{gather}
We furthermore say that the soliton is \emph{gradient} if there exists a smooth function $f$ such that $X = \N f$.
\end{defn} 

\begin{question} Bryant has shown the existence of steady Ricci solitons with rotational symmetry \cite{Bryant}.  Are there examples of steady generalized Ricci solitons with rotational symmetry and nontrivial torsion $H$?
\end{question}

One fundamental consequence of Perelman's energy formula for Ricci flow is that compact steady solitons for Ricci flow are automatically gradient.  As we will see below in Chapter \ref{energychapter}, by adapting these energy functionals to generalized Ricci flow, we can show that steady generalized Ricci solitons on compact manifolds are automatically gradient, and moreover satisfy $k = 0$.  It is an interesting question to probe the possibilities for $X$ and $k$ in the noncompact setting.  We end this section with some further structural results on solitons.  First, we show that generalized Ricci solitons satisfy fundamental differential equations which 
generalize the fact that Einstein manifolds automatically have constant scalar 
curvature using the Bianchi identity.

\begin{prop} \label{p:solitonbasics} Let $(M^n, g, H, f)$ be a gradient generalized Ricci soliton with $k = 0$.  Then
\begin{enumerate}
\item $R - \tfrac{1}{4} \brs{H}^2 + \gD f = 0$,
\item $\N \left(R - \tfrac{5}{12} \brs{H}^2 + \brs{\N f}^2 \right) = 0$,
\item $\N \left( R - \tfrac{1}{12} \brs{H}^2 + 2 \gD f - \brs{\N f}^2 \right) = 0$.
\end{enumerate}
\begin{proof} The first equation follows by tracing the first equation of
(\ref{f:solitonexp2}) with $g$.  For the second we begin by differentiating and 
applying (\ref{f:solitonexp2}) to yield
\begin{gather} \label{f:solitonbasics10}
\begin{split}
\N_i \left( R - \tfrac{5}{12} \brs{H}^2 + \brs{\N f}^2 \right) =&\ \N_i R - 
\tfrac{5}{12} \N_i \brs{H}^2 + 2 \N_i \N_j f \N_j f\\
=&\ \N_i R - \tfrac{5}{12} \N_i \brs{H}^2 + \left( \tfrac{1}{2} H^2 - 2 \Rc
\right)_{ij} \N_j f.
\end{split}
\end{gather}
But, using the Bianchi identity and property (1) we obtain
\begin{gather} \label{f:solitonbasics20}
\begin{split}
\N_i R =&\ \tfrac{1}{4} \N_i \brs{H}^2 - \N_i \gD f\\
=&\ \tfrac{1}{4} \N_i \brs{H}^2 - \N_j \N_i \N_j f + \Rc_{ij} \N_j f\\
=&\ \tfrac{1}{4} \N_i \brs{H}^2 + \N_j \left( \Rc - \tfrac{1}{4} H^2
\right)_{ij} + \Rc_{ij} \N_j f\\
=&\ \tfrac{1}{4} \N_i \brs{H}^2 + \tfrac{1}{2} \N_i R - \tfrac{1}{4} (\divg 
H^2)_i + \Rc_{ij} \N_j f\\
=&\ \tfrac{1}{2} \N_i \brs{H}^2  - \tfrac{1}{2} (\divg H^2)_i + 2 \Rc_{ij} \N_j 
f,
\end{split}
\end{gather}
where the last line follows by combining the $\N R$ terms and multiplying by 
$2$.  Also, using Lemma \ref{l:divHlemma} and the skew-symmetric piece of the soliton equation we obtain
\begin{gather} \label{f:solitonbasics30}
\begin{split}
- \tfrac{1}{2} (\divg H^2)_i =&\ - \tfrac{1}{12} \N_i \brs{H}^2 + \tfrac{1}{2} 
(d^*H)^{mn} H_{imn}\\
=&\ - \tfrac{1}{12} \N_i \brs{H}^2 - \tfrac{1}{2} H^2_{ij} 
\N_j f.
\end{split}
\end{gather}
Plugging (\ref{f:solitonbasics20}) and (\ref{f:solitonbasics30}) into 
(\ref{f:solitonbasics10}) yields the second claim.  Taking twice the gradient of (1) and 
subtracting (2) gives claim (3).
\end{proof}
\end{prop}

\begin{rmk} For the third equation in Proposition \ref{p:solitonbasics} it is 
possible to give a more conceptual argument using the $\FF$ functional to be introduced in Chapter \ref{energychapter}.  In particular, we know that a generalized steady soliton is a critical point for $\gl$, but by definition the relevant $f$ extremizes the $\FF$-functional, thus by deforming $f$ by $R - \tfrac{1}{12} \brs{H}^2 + 2 \gD f - \brs{\N f}^2$ and perturbing the metric by scaling to preserve the unit volume condition, the claim of Proposition \ref{p:solitonbasics} (3) follows from (\ref{Ffirstvar}).
\end{rmk}

In the next result we prove that compact steady gradient Ricci solitons are Einstein. Recall that, as a consequence of Perelman's energy formula for Ricci flow, compact steady solitons for Ricci flow are automatically gradient. Thus any compact Ricci soliton is Einstein. However, there are compact steady generalized Ricci solitons with nontrivial $H$ and nontrivial $f$ (cf. Remark \ref{r:steadysolitons}). This shows that the class of steady generalized Ricci solitons is strictly larger than the class of generalized Ricci flat structures.

\begin{prop} \label{p:compactRiccisolitons} Compact steady gradient Ricci solitons are Einstein.
\begin{proof} In the case of $H \equiv 0$ we have the simplified soliton equation
\begin{align*}
\Rc + \N^2 f =&\ 0.
\end{align*}
Taking the divergence of this equation and using the Bianchi identity yields
\begin{align*}
0 =&\ \tfrac{1}{2} \N_j R + \N_i \left( \N_j \N_i f \right)\\
=&\ \tfrac{1}{2} \N_j R + \N_j \gD f - R_{iji}^p \N_p f\\
=&\ \tfrac{1}{2} \N_j R + \N_j \gD f + \Rc_j^p\N_p f.
\end{align*}
Using the first item of Proposition \ref{p:solitonbasics} yields
\begin{align*}
0 =&\ \tfrac{1}{2} \N_j R - \N_j  R + \Rc_j^p\N_p f.
\end{align*}
Rearranging and taking another divergence yields
\begin{align*}
\gD R =&\ 2 \N_j \left[ \Rc_j^p \N_p f \right]\\
=&\ \IP{\N R, \N f} + 2 \IP{ \Rc, \N^2 f}\\
=&\ \IP{\N R, \N f} - 2 \brs{\Rc}^2\\
\leq&\ \IP{\N R, \N f}.
\end{align*}
Since $M$ is compact, $R$ is constant by the strong maximum principle (Proposition \ref{p:strongmax}).
\end{proof}
\end{prop}

\section{Low dimensional structure}

In this section we consider the generalized Ricci flow in dimensions $n=2,3$, where the low dimensionality gives extra structure which renders the flow easier to study.  In dimension $n=2$ there are no nontrivial three-forms, and thus we expect that the flow must reduce to Ricci flow.  We formalize this in the next proposition.

\begin{prop} \label{p:GRFsurfaces}  Let $M^2$ be a Riemann surface, fix $0 = H_0 \in \Lambda^3 T^* = \{0\}$, and let $(g_t, b_t)$ be a solution to generalized Ricci flow on $M$.  Then $(g_t, b_t)$ satisfies
\begin{gather}
\begin{split}
\dt g =&\ -2 \Rc = - R g, \qquad \dt b = 0.
\end{split}
\end{gather}
In other words, $g_t$ is a solution of Ricci flow and $b_t \equiv b_0$.
\begin{proof} Since the manifold is two dimensional, $\Lambda^3 T^* = \{0\}$, and so $H_t \equiv 0$ for all $t$.  Comparing against the generalized Ricci flow system we see that $\dt b = 0$ and $\dt g = -2 \Rc$.  Also in two dimensions $\Rc = \tfrac{1}{2} R g$, finishing the lemma.
\end{proof}
\end{prop}

\begin{rmk} A complete picture of the Ricci flow on compact Riemann surfaces was first completed by Chow/Hamilton \cite{ChowRFSurfaces, HamiltonRFSurfaces}.  In all cases, the volume-normalized flow converges exponentially fast to a constant scalar curvature metric $g_{\infty}$.  Later a different proof using energy methods from conformal geometry was given by Struwe \cite{StruweFlows}.  Using this result, and being slightly overpedantic, we could observe that the \emph{generalized} Ricci flow converges up to a generalized gauge transformation to $(g_{\infty}, 0)$, where by an overall $b$-field action we can set $b$ to zero.
\end{rmk}

In dimension $n=3$, all three-forms are closed, and $\Lambda^3 T^*$ is a rank $1$ bundle, thus we can reduce the tensor $H$ to an equivalent scalar function.  We formalize this and derive the relevant evolution equations  in the next proposition.

\begin{prop} \label{p:GRFthreefolds}  Let $M^3$ be a three-manifold, fix $H_0 \in \Lambda^3 T^*$ and let $(g_t, b_t)$ be a solution of generalized Ricci flow on $M$.  Define $\phi_t \in C^{\infty}(M)$ by
\begin{align*}
\phi_t =&\ \frac{H_t}{dV_{g_t}}.
\end{align*}
Then
\begin{gather*}
\begin{split}
\dt g =&\ -2 \Rc + \phi^2 g,\\
\dt \phi =&\ \gD \phi + R \phi - \tfrac{3}{2} \phi^3.
\end{split}
\end{gather*}
\begin{proof} By direct computations in normal coordinates and expressing $dV_g = dx^1 \wedge dx^2 \wedge dx^3$ one obtains
\begin{align*}
H^2_{ij} =&\ H_{i p q} H_{j r s} g^{pr} g^{qs} = 2 \phi^2 g_{ij}, \qquad \brs{H}^2 = 6 \phi^2.
\end{align*}
Furthermore we can differentiate, using the generalized Ricci flow equation and Lemma \ref{l:volumevariation},
\begin{align*}
\dt \phi =&\ \dt \frac{H}{dV_{g}}\\
=&\ \frac{\gD_d H}{dV_g} - \tfrac{1}{2} \tr_g \left(\dt g \right) \frac{H}{dV_g}\\
=&\ \gD \phi + \left(R - \tfrac{1}{4} \brs{H}^2 \right) \phi\\
=&\ \gD \phi + R \phi - \tfrac{3}{2} \phi^3,
\end{align*}
where the penultimate line follows since, using that $\star dV_g = 1$,
\begin{align*}
\gD_d H = - (d \star d \star) (\phi dV_g) = - d \star d \phi = - \star d^* d \phi = \gD \phi dV_g.
\end{align*}
\end{proof}
\end{prop}

If we impose some symmetries on the underlying structure, it turns out that we can realize other naturally occurring geometric evolution equations as special cases of generalized Ricci flow.  

\begin{prop} \label{p:RYMandGRF} Let $S^1 \to P \to M^2$ be a principal $S^1$ bundle, and suppose $(g_t, \theta_t)$ is a solution to
\begin{gather} \label{f:RYM}
\begin{split}
\dt g =&\ -2 \Rc + 2 F^2,\\
\dt \theta =&\ - d^*_g F.
\end{split}
\end{gather}
Let $\bar{g}_t = g_t \oplus \theta_t \otimes \theta_t$, and $H_t = F_t \wedge \theta_t$. Then $(\bar{g}_t, H_t)$ is a solution of generalized Ricci flow on $P$.
\begin{proof} The proof relies on several basic computations, which are left as \textbf{exercises}.  First, the Ricci tensor of a metric $\bar{g} = g \oplus \theta \otimes \theta$ on $P$ is expressed in natural block diagonal form as
\begin{align*}
\bar{\Rc} = \left(
\begin{matrix}
\Rc - \tfrac{1}{2} F^2 & \tfrac{1}{2} d^* F\\
\tfrac{1}{2} d^* F & \tfrac{1}{4} \brs{F}^2
\end{matrix} \right).
\end{align*}
Furthermore, for $H = F \wedge \theta$, we obtain
\begin{align*}
H^2 = \left(
\begin{matrix}
2 F^2 & 0\\
0 & \brs{F}^2\\
\end{matrix} \right).
\end{align*}
Putting these computations together and using the definition of $\bar{g}$ shows that $\bar{g}$ satisfies the metric component of generalized Ricci flow.

A further computation shows that $d^*_{\bar{g}} \left(F \wedge \theta \right) = d_g^* F \wedge \theta$.  Since $M$ is two dimensional we thus obtain
\begin{align*}
\dt H =&\ \dt F \wedge \theta = - d^*_g F \wedge \theta + F \wedge (- d^*_g F) = - d^*_g F \wedge \theta = - d^*_{\bar{g}} H,
\end{align*}
as required.
\end{proof}
\end{prop}

The system of equations (\ref{f:RYM}) is a natural coupling of the Ricci flow with the Yang-Mills flow for a $U(1)$ principal connection.  This system is thus called Ricci Yang-Mills flow, and was introduced in \cite{StreetsThesis, YoungThesis} (up to a scalar factor on the $F^2$ term which can be fixed by rescaling the fiber metric).  We see here that in the case of $U(1)$ gauge group (or $U(1) \times \dots \times U(1)$), the flow is a special case of generalized Ricci flow.  This point of view is useful in understanding the relationship of generalized Ricci flow to $T$-duality in Chapter \ref{c:Tdual}.  Further results on the global existence and convergence properties of (\ref{f:RYM}) appeared in \cite{StreetsRGflow, StreetsRYMsurfaces}.  Some related equations motivated by considerations in mathematical physics have appeared in \cite{fei2018parabolic, HeLiuBS}

Yet a different viewpoint on the equations (\ref{f:RYM}) is as generalized Ricci flow on a \emph{transitive Courant algebroid} on $M^2$ obtained by reduction of the (twisted) exact Courant algebroid $TP \oplus T^*P$, as defined in \cite{GF19}. This follows directly from the fact that any three-form on $M^2$ identically vanishes, and suggests that generalized Ricci flow is a robust geometric flow which is preserved under natural geometric operations on Courant algebroids such as generalized reduction \cite{BursztynCavalcantiGualtieri} or T-duality (see Chapter \ref{c:Tdual}). For stationary points of the flow, the preservation of the Ricci flat condition under generalized reduction has been proved in \cite{BarHek}. We leave the general case, concerning generalized Ricci flow, as an open question.

\chapter{Local Existence and Regularity} \label{LEchapter}

In this chapter we address basic aspects of the existence and regularity of generalized Ricci flow.  We saw in Lemma \ref{l:GRFasheat} that, with special gauge choices made, the generalized Ricci flow appears to be a nonlinear parabolic system of equations.  Our first goal is to rigorously prove short-time existence of generalized Ricci flow with arbitrary smooth initial data on compact manifolds, using the method of gauge fixing.  Given this we approach the question of how long the solution can be extended, and what the obstructions are to global existence.  We begin by deriving evolution equations for the Riemannian and Bismut curvatures, which lead to smoothing estimates in the presence of a bound on the Riemannian curvature tensor.  Using this we can show that the Riemannian curvature must blow up at any finite time singularity of generalized Ricci flow, first established in \cite{Streetsexpent}.  We end the chapter with a general discussion of the theory of compactness of sequences of solutions of generalized Ricci flow, omitting most of the technical detail.

\section{Variational Formulas} \label{s:variation}

We record here elementary variational formulas related to Riemannian metrics, the Riemannian curvature tensor, and the Bismut curvature tensor, necessary for what follows.

\begin{rmk} We note here a notational convention we will use in several places.  In particular, given tensors $A$ and $B$ on a Riemannian manifold $(M^n, g)$, we let $A \star B$ denote any tensor derived from $A \otimes B$ by raising and lowering indices using $g$ and taking contractions.
\end{rmk}

\begin{lemma} \label{l:volumevariation} Given $(M^n, g)$ a Riemannian manifold 
and $g_s$ a one-parameter family of metrics such that
\begin{align*}
\left. \frac{\del}{\del s} \right|_{s=0} g_s = h, \qquad g_0 = g,
\end{align*}
then
\begin{align*}
\left. \frac{\del}{\del s} \right|_{s=0} dV_{g_s} =&\ \tfrac{1}{2} \left( \tr_g 
h \right) dV_g.
\end{align*}
\begin{proof} Recall the local coordinate formula $dV_{g_s} = \sqrt{\det 
g_{ij}(s)} dx_1 \wedge \dots \wedge dx_n$.  Also, for a one-parameter family of 
positive definite symmetric matrices $A_s$ with variation $\left. 
\tfrac{\del}{\del s} \right|_{s=0} A_s = B$, one has $\left. \tfrac{\del}{\del 
s} \right|_{s=0} \det (A_s) = \det (A_0) \tr \left( A_0^{-1} B \right)$.  
Combining these facts yields
\begin{align*}
\left. \frac{\del}{\del s} \right|_{s=0} dV_{g_s} =&\ \left. \frac{\del}{\del 
s} 
\right|_{s=0} \sqrt{\det g_{ij}(s)} dx_1 \wedge \dots \wedge dx_n\\
=&\ \tfrac{1}{2 \sqrt{\det g_{ij}}} \left( \det g_{ij} \tr \left( g^{-1} h 
\right) \right) dx_1 \wedge \dots \wedge dx_n\\
=&\ \tfrac{1}{2} \left( \tr_g h \right) dV_g,
\end{align*}
as required.
\end{proof}
\end{lemma}

\begin{lemma} \label{l:Henergyvar} Given $(M^n, g)$ a Riemannian manifold 
and $g_s$ a one-parameter family of metrics such that
\begin{align*}
\left. \frac{\del}{\del s} \right|_{s=0} g_s =&\ h, \qquad g_0 = g,\\
\left. \frac{\del}{\del s} \right|_{s = 0} H =&\ d K, \qquad H_0 = H,
\end{align*}
then
\begin{align*}
\left. \frac{\del}{\del s} \right|_{s = 0} \brs{H}^2 =&\ - 3 \IP{h, H^2} + 2 
\IP{d K, H}.
\end{align*}
\begin{proof} Computing in local coordinates and using the symmetries of $H$ 
and 
$g$ we see that
\begin{align*}
\left. \frac{\del}{\del s} \right|_{s = 0} \brs{H}^2 =&\ \left. 
\frac{\del}{\del 
s} \right|_{s = 0} g^{i_1 j_1} g^{i_2 j_2} g^{i_3 j_3} H_{i_1 i_2 i_3} H_{j_1 
j_2 j_3}\\
=&\ - 3 g^{i_1 k} h_{kl} g^{l j_1} g^{i_2 j_2} g^{i_3 j_3} H_{i_1 i_2 i_3} 
H_{j_1 j_2 j_3} + 2 g^{i_1 j_1} g^{i_2 j_2} g^{i_3 j_3} \left( dK \right)_{i_1 
i_2 i_3} H_{j_1 j_2 j_3}\\
=&\ - 3 \IP{h, H^2} + 2 \IP{dK, H},
\end{align*}
as required.
\end{proof}
\end{lemma}

\begin{lemma} \label{l:curvaturevariation} Given $(M^n, g)$ a Riemannian 
manifold 
and $g_s$ a one-parameter family of metrics such that
\begin{align*}
\left. \frac{\del}{\del s} \right|_{s=0} g_s = h, \qquad g_0 = g,
\end{align*}
then
\begin{align*}
\left. \frac{\del}{\del s} \right|_{s=0} (\N)_{ij}^k =&\ \tfrac{1}{2} 
g^{kl} \left( \N_i h_{jl} + \N_j h_{il} - \N_l h_{ij} \right),\\
\left. \frac{\del}{\del s} \right|_{s=0} R_{ijk}^l =&\ \tfrac{1}{2} g^{ql} 
\left[ \N_i \N_k h_{jq} - \N_i \N_q 
h_{jk} - \N_j \N_k h_{iq} + \N_j \N_q h_{ik} - R_{ijk}^p h_{pq} - R_{ijq}^p 
h_{kp} \right],\\
\left. \frac{\del}{\del s} \right|_{s=0} R_{jk} =&\ \tfrac{1}{2} \left( g^{qi} 
(\N_i \N_j h_{kq} + \N_i \N_k h_{jq}) - \gD h_{jk} - \N_j \N_k \tr_g h \right)\\
\left. \frac{\del}{\del s} \right|_{s=0} R =&\ - \gD \tr_g h + \divg \divg h - 
\IP{h, \Rc}.
\end{align*}
\begin{proof} We fix a point $p \in M$ and choose normal coordinates for $g$ at 
$p$.  Using that $\del_i g_{jk}(p) = \gG_{ij}^k(p) = 0$, we directly obtain
\begin{align*}
\left. \frac{\del}{\del s} \right|_{s=0} (\N_{g_s})_{ij}^k =&\ \left. 
\frac{\del}{\del s} \right|_{s=0} \tfrac{1}{2} g^{kl} \left(\del_i g_{jl} + 
\del_j g_{il} - \del_l g_{ij} \right)\\
=&\ \tfrac{1}{2} g^{kl} \left(\del_i h_{jl} + \del_j h_{il} - \del_l h_{ij} 
\right)\\
=&\ \tfrac{1}{2} g^{kl} \left(\N_i h_{jl} + \N_j h_{il} - \N_l h_{ij} \right),
\end{align*}
as required.

For the variation of the curvature tensor we recall the coordinate formula and 
differentiate, again using the vanishing of $\del_i g_{jk}(p)$ and 
$\gG_{ij}^k(p)$, to yield
\begin{align*}
\left. \frac{\del}{\del s} \right|_{s=0} R_{ijk}^l =&\ \left. \frac{\del}{\del 
s} \right|_{s=0} \left( \del_i \gG_{jk}^l - \del_j \gG_{ik}^l + \gG_{jk}^p 
\gG_{ip}^l - \gG_{ik}^p \gG_{jp}^l \right)\\
=&\ \tfrac{1}{2} \left[ \del_i \left( g^{ql}( \N_{j} h_{kq} + \N_k h_{jq} - 
\N_q h_{jk}) \right) - \del_j \left( g^{ql} (\N_i h_{kq} + \N_k h_{iq} - \N_q 
h_{ik} )\right) \right]\\
=&\ \tfrac{1}{2} g^{ql} \left[ \N_i \N_j h_{kq} + \N_i \N_k h_{jq} - \N_i \N_q 
h_{jk} - \N_j \N_i h_{kq} - \N_j \N_k h_{iq} + \N_j \N_q h_{ik} \right]\\
=&\ \tfrac{1}{2} g^{ql} \left[ \N_i \N_k h_{jq} - \N_i \N_q 
h_{jk} - \N_j \N_k h_{iq} + \N_j \N_q h_{ik} - R_{ijk}^p h_{pq} - R_{ijq}^p 
h_{kp} \right],
\end{align*}
where the last line follows using the definition of the curvature tensor.

Given this, the variational formula for the Ricci tensor follows directly by 
contracting.  Finally we use this formula to obtain the variation of scalar 
curvature via
\begin{align*}
\left. \frac{\del}{\del s} \right|_{s=0} R =&\ \left. \frac{\del}{\del s} 
\right|_{s=0} \left( g^{jk} R_{jk} \right)\\
=&\ - g^{jp} h_{pq} g^{qk} R_{jk} + \tfrac{1}{2} g^{jk} \left( g^{qi} (\N_i 
\N_j h_{kq} + \N_i \N_k h_{jq}) - \gD h_{jk} - \N_j \N_k \tr_g h \right)\\
=&\ - \gD \tr_g h + \divg \divg h - \IP{h, \Rc},
\end{align*}
as required.
\end{proof}
\end{lemma}

\begin{lemma} \label{l:Bcurvaturevariation} Given $(M^n, g)$ a Riemannian 
manifold 
and $(g_s, H_s)$ a one-parameter family such that
\begin{align*}
\left. \frac{\del}{\del s} \right|_{s=0} g_s = h, \qquad g_0 = g, \qquad \left. 
\frac{\del}{\del s} \right|_{s=0} H_s = d K, \qquad H_0 = H,
\end{align*}
Furthermore set $A = h + K$.  Then
\begin{align*}
\left. \frac{\del}{\del s} \right|_{s=0} \left(\N^+\right)_{ij}^k =&\ \tfrac{1}{2} g^{kl} \left( \N^+_i A_{jl} + \N^+_j A_{li} - \N^+_l A_{ji} \right) + \tfrac{1}{2} g^{kl} \left( - H_{ij}^p A_{lp} + H_{il}^p A_{jp} + H_{jl}^p A_{pi} \right)\\
\left. \frac{\del}{\del s} \right|_{s=0} \left( R^+ \right)_{ijk}^l =&\ \N^+_i 
\left(\dot{\N}^+ \right)_{jk}^l - \N^+_j \left(\dot{\N}^+ \right)_{ik}^l + 
H_{ij}^p \left(\dot{\N}^+ \right)_{pk}^l,\\
\left. \frac{\del}{\del s} \right|_{s=0} R^+_{jk} =&\ \N^+_i \left(\dot{\N}^+ 
\right)_{jk}^i - \N^+_j \left(\dot{\N}^+ \right)_{ik}^i + H_{ij}^p 
\left(\dot{\N}^+ \right)_{pk}^i,\\
\left. \frac{\del}{\del s} \right|_{s=0} R^+ =&\ - \IP{h, \Rc^+} + g^{jk} 
\left\{ \N^+_i \left(\dot{\N}^+ \right)_{jk}^i - \N^+_j \left(\dot{\N}^+ 
\right)_{ik}^i + H_{ij}^p \left(\dot{\N}^+ \right)_{pk}^i\right\}.
\end{align*}
\begin{proof} We use the formula from Definition \ref{d:Bismutconn}, expressed 
in components as $(\gG^+)_{ij}^k = \gG_{ij}^k + \tfrac{1}{2} H_{ijl} g^{kl}$, 
then apply the result of Lemma \ref{l:curvaturevariation} to obtain
\begin{gather} \label{f:Biscurv10}
\begin{split}
\left. \frac{\del}{\del s} \right|_{s=0} \N^+_s =&\ \tfrac{1}{2} 
g^{kl} \left( \N_i h_{jl} + \N_j h_{il} - \N_l h_{ij} \right) + \left. 
\frac{\del}{\del s} \right|_{s=0} \left[ \tfrac{1}{2} H_{ijl} g^{kl} \right]\\
=&\ \tfrac{1}{2} 
g^{kl} \left( \N_i h_{jl} + \N_j h_{il} - \N_l h_{ij} \right) + \tfrac{1}{2} 
dK_{ijl} g^{kl} - \tfrac{1}{2} H_{ijl} g^{kp} h_{pq} g^{ql},
\end{split}
\end{gather}
Using the definition of 
the Bismut connection, we see that
\begin{align*}
\N_i h_{jk} =&\ \del_i h_{jk} - \gG_{ij}^l h_{lk} - \gG_{ik}^l h_{jl}\\
=&\ \del_i h_{jk} - \left( (\gG^+)_{ij}^l - \tfrac{1}{2} H_{ijp} g^{pl} \right) 
h_{lk} - \left( (\gG^+)_{ik}^l - \tfrac{1}{2} H_{ikp} g^{pl} \right) h_{jl}\\
=&\ \N^+_i h_{jk} + \tfrac{1}{2} H_{ij}^p h_{pk} + \tfrac{1}{2} H_{ik}^p 
h_{jp}.
\end{align*}
Hence
\begin{align*}
\N_i h_{jl} + \N_j h_{il} - \N_l h_{ij} =&\ \N^+_i h_{jl} + \N^+_j h_{il} - 
\N^+_l h_{ij}\\
&\ + \tfrac{1}{2} g^{pq} \left( H_{ijp} h_{ql} + H_{ilp} h_{jq} +  H_{j i p} 
h_{ql} + H_{jlp} h_{iq} - H_{lip} h_{qj} - H_{ljp} h_{iq} \right)\\
=&\ \N^+_i h_{jl} + \N^+_j h_{il} - \N^+_l h_{ij} + g^{pq} \left( H_{ilp} 
h_{jq} + H_{jlp} h_{iq} \right).
\end{align*}
Similarly we have
\begin{align*}
(d K)_{ijl} =&\ \N_i K_{jl} + \N_l K_{ij} + \N_j K_{li}\\
=&\ \N^+_i K_{jl} + \N^+_l K_{ij} + \N^+_j K_{li}\\
&\ + \tfrac{1}{2} \left( H_{ij}^p K_{pl} + H_{il}^p K_{jp} + H_{li}^p K_{pj} + H_{lj}^p K_{ip} + H_{jl}^p K_{pi} + H_{ji}^p K_{lp} \right)\\
=&\ \N^+_i K_{jl} + \N^+_l K_{ij} + \N^+_j K_{li} + H_{ij}^p K_{pl} + H_{il}^p K_{jp} + H_{lj}^p K_{ip}
\end{align*}
Plugging these computations into (\ref{f:Biscurv10}) yields the first equality.

Next, using the definition of the curvature tensor, we differentiate the 
general coordinate formula to yield
\begin{gather}
\begin{split}
\left. \frac{\del}{\del s} \right|_{s=0} \left( R^+ \right)_{ijk}^l =&\ \left. 
\frac{\del}{\del s} \right|_{s=0} \left\{ \del_i (\gG^+)_{jk}^l - \del_j 
(\gG^+)_{ik}^l + (\gG^+)_{ip}^l (\gG^+)_{jk}^p - (\gG^+)_{j p}^l (\gG^+)_{ik}^p 
\right\}\\
=&\ \del_i (\dot{\gG}^+)_{jk}^l - \del_j (\dot{\gG}^+)_{ik}^l + 
(\dot{\gG}^+)_{ip}^l (\gG^+)_{jk}^p + ({\gG}^B)_{ip}^l (\dot{\gG}^+)_{jk}^p\\
&\ - (\dot{\gG}^+)_{j p}^l (\gG^+)_{ik}^p - (\gG^+)_{j p}^l 
(\dot{\gG}^+)_{ik}^p\\
=&\ \N^+_i (\dot{\gG}^+)_{jk}^l + (\gG^+)_{ij}^p (\dot{\gG}^+)_{pk}^l - \N^+_j 
(\dot{\gG}^+)_{ik}^l - (\gG^+)_{ji}^p (\dot{\gG}^+)_{pk}^l\\
=&\ \N^+_i (\dot{\gG}^+)_{jk}^l - \N^+_j (\dot{\gG}^+)_{ik}^l + H_{ij}^p 
(\dot{\gG}^+)_{pk}^l,
\end{split}
\end{gather}
as required. Taking the trace over $i$ and $l$ yields the variation for 
$R^+_{jk}$.

Lastly, using that $R^+ = g^{jk} R^+_{jk}$, we obtain
\begin{align*}
\left. \frac{\del}{\del s} \right|_{s=0} R^+ =&\ \left. \frac{\del}{\del s} 
\right|_{s=0} g^{jk} R^+_{jk}\\
=&\ - g^{jp} h_{pq} g^{qk} R^+_{jk} + g^{jk} \left( \left. \frac{\del}{\del s} 
\right|_{s=0} R^+_{jk} \right)\\
=&\ - \IP{h, \Rc^+} + g^{jk} \left\{ \N^+_i \left(\dot{\N}^+ \right)_{jk}^i - 
\N^+_j \left(\dot{\N}^+ \right)_{ik}^i + H_{ij}^p \left(\dot{\N}^+ 
\right)_{pk}^i\right\},
\end{align*}
as required.
\end{proof}
\end{lemma}

\section{Short time existence}

In this section we establish the fundamental fact that for smooth initial data 
on a compact manifold, there exists a smooth solution to the generalized Ricci 
flow on some small time interval depending on the initial data.  As we will 
show 
in this section, the generalized Ricci flow equation is a degenerate parabolic 
equation, and so the classical theorems concerning short-time existence of 
nonlinear parabolic evolution equations do 
not directly 
apply.  Many geometric evolution equations encounter this obstacle, with the 
degeneracies of the equation arising from an infinite dimensional group of 
symmetries acting on solutions to the flow.  

The resolution of this issue is via the ``DeTurck method'', where the 
large symmetry group is broken via the explicit introduction of extra terms 
tangent to the group action which render the new equation strictly parabolic.  
As such, one can now apply classical results to obtain a short-time solution to 
this flow.  By construction, the extra terms can be removed via an application 
of the group action, producing finally a solution to the desired evolution 
equation.  In the case of Ricci flow, the relevant group is the diffeomorphism 
group of the underlying manifold.  As we have seen in \S 
\ref{ss:Courantsymmetries}, the symmetry 
group in generalized geometry is enlarged to include $b$-field transformations, 
which indeed preserve the generalized Ricci flow.  Thus our proof exploits this 
enlarged group, but structurally follows the outline described above precisely.

\begin{thm} \label{t:STE} 
Let $(E,[,],\IP{,}) \to M$ be an exact 
Courant algebroid over a compact manifold $M$.  Given $\GG_0$ a generalized metric on $E$, there
exists $\ge > 0$ and a unique smooth solution $\GG_t$ to generalized Ricci flow with
initial condition $\GG_0$ on $[0,\ge)$.

\begin{proof} \textbf{Step 1: Short time existence of gauge-fixed flow} 
\vskip 0.1in
The first step is to construct a solution of the gauge-fixed generalized Ricci flow, specifically the metric equation of (\ref{f:gaugefixedgrf}) and the induced equation on $H$, for a particular choice of $X$ and $k = 0$.  Let $g_0$ denote the Riemannian metric determined by $\GG_0$. We will work in the splitting of E determined by $\GG_0$. Given a metric $g$ we define the vector field
\begin{align*}
X := X(g,g_0) = \tr_g \left( \N^g - \N^{g_0} \right),
\end{align*}
which is expressed in components as
\begin{align*}
X^i = g^{pq} \left( \left(\gG^g \right)_{pq}^i - \left(\gG^{g_0} \right)_{pq}^i 
\right).
\end{align*}
This vector field is precisely the one used in showing short-time existence of solutions to Ricci flow, but we note here that it admits a natural interpretation in generalized geometry as the difference of divergence operators associated to $\GG$ and $\GG_0$ (cf. \S \ref{s:divops}).  Now define the differential operator
\begin{align*}
\OO_1(g,H) := -2 \Rc + \tfrac{1}{2} H^2 + L_{X(g,g_0)} g.
\end{align*}
We proceed to compute the symbol of $\OO_1$ in phases.  We fix one-parameter 
families $(g_t,b_t)$ such that
\begin{align*}
 \left. \tfrac{d}{ds} \right|_{s=0} g_s = h, \qquad \left. \tfrac{d}{ds} 
\right|_{s=0} H = d K.
\end{align*}
We first observe that the variation of $-2 
\Rc$ is already contained in Lemma \ref{l:curvaturevariation}.  Also, the 
dependence of the term $H^2$ on $g$ is algebraic, and thus its linearization is 
a zeroth order operator which we can ignore. To address the Lie derivative term, we fix 
normal 
coordinates for $g$ at some point $p \in M$ we compute
\begin{align*}
\left. \tfrac{d}{ds} \right|_{s=0} (L_{X(g,g_0)} g) =&\ \left. \tfrac{d}{ds} 
\right|_{s=0} \left( 
\N_i X^k g_{kj} + \N_j X^k g_{ki} \right)\\
=&\ \left. \tfrac{d}{ds} \right|_{s=0} \left( \del_i X^k g_{kj} + \gG(g) \star X + 
\del_j X^k g_{ki} + \gG(g) * X \right)\\
=&\ \del_i \left( \tfrac{1}{2} g^{pq} \left( \N_p h_{qj} + \N_q h_{pj} - \N_j 
h_{pq} \right) \right) + \del_j \left( \tfrac{1}{2} g^{pq} \left( \N_p h_{qi} + 
\N_q h_{pi} - \N_i h_{pq} \right) \right)\\
=&\ \tfrac{1}{2} g^{pq} \left( \N_i \N_p h_{qj} + \N_i \N_q h_{pj} + \N_j \N_p 
h_{qi} + \N_j \N_q h_{pi} - \N_i \N_j h_{pq} - \N_j \N_i h_{pq} \right)\\
=&\ \N_i \divg h_j + \N_j \divg h_i - \N_i \N_i \tr_g h + \mbox{l.o.t},
\end{align*}
where ``$\mbox{l.o.t}$'' denotes any term involving at most one derivative of 
$h$ or $K$.  Combining this calculation with Lemma \ref{l:curvaturevariation}, 
we obtain that
\begin{align} \label{f:ste10}
\left. \tfrac{d}{ds} \right|_{s=0} \OO_1(g_s,H) = \gD_d h + \mbox{l.o.t}.
\end{align}
On the other hand, examining $\OO_1$ it is clear that the only dependence on 
$b$ occurs in the term $H^2$, which is first-order in $b$, and hence one 
immediately obtains
\begin{align} \label{f:ste20}
\left. \tfrac{d}{ds} \right|_{s=0} \OO_1(g,H_s) = \mbox{l.o.t}.
\end{align}

Next, let
\begin{align*}
 \OO_2(g,H) :=&\ \gD_d H + L_{X(g,g_0)} H.
\end{align*}
We compute, using the Bochner formula,
\begin{gather} \label{f:ste30}
 \begin{split}
  \left. \tfrac{d}{ds} \right|_{s=0} \OO_2(g,H_s) =&\ \gD K + \mbox{l.o.t}.
 \end{split}
\end{gather}
Also, note that the variation of $\OO_2$ with respect to $g$ will be a 
certain second-order differential operator, due to the second derivative terms 
appearing in $\gD_d H$ and $L_{X(g,g_0)} H$.  However, the precise form of 
this operator is not relevant to the argument.  In 
particular, combining this observation with (\ref{f:ste10})-(\ref{f:ste30}) we obtain
\begin{align*}
 \gs \left\{ D_{(g,H)} \left(\mathcal O_1, \mathcal O_2 \right) \right\} \left( 
\begin{matrix} h\\ K \end{matrix} \right) =&\ \left( \begin{matrix} \gD & 0\\ 
\star & \gD \end{matrix} \right) \left( \begin{matrix} 
h\\ K \end{matrix} \right).
\end{align*}
This shows that the system of equations
\begin{gather} \label{f:stegff}
\begin{split}
\dt g =&\ \OO_1(g,H) = -2 \Rc + \tfrac{1}{2} H^2 + L_{X(g,g_0)} g\\
\dt H =&\ \OO_2(g,H) = \gD_d H + L_{X(g,g_0)} H
\end{split}
\end{gather}
is strictly parabolic.  Applying the general result on existence of short-time solutions to strictly parabolic evolution systems on compact manifolds, we conclude that there exists $\ge > 0$ and a unique solution to (\ref{f:stegff}) on $[0,\ge)$.

\begin{rmk}
An alternative proof of the short-time existence of generalized Ricci flow could possibly be obtained directly from the gauge-fixed flow \eqref{f:GFGRFdecomp}, for $X$ as above and $B = d\xi_t -  i_{X_t} H_0$ for a suitable one-parameter family of one-forms $\xi_t \in T^*$. This gauge fixing by \emph{inner symmetries} of the exact Courant algebroid is very natural, and we expect that it would provide a conceptual explanation of Lemma \ref{l:GRFasheat}.
\end{rmk}

\vskip 0.1in

\textbf{Step 2: Gauge modification to solve generalized Ricci flow}

\vskip 0.1in

Fix $G_0$ a generalized metric, with associated pair $(g_0,H_0)$.  By the arguments of Step 1, there exists $\ge > 0$ and a solution to (\ref{f:stegff}) on $[0,\ge)$.  We must perform a diffeomorphism gauge transformation to properly gauge-fix the system.  In particular, we define a one-parameter family of diffeomorphisms $\phi_t$ via
\begin{gather} \label{f:ste40}
\begin{split}
\dt \phi_t =&\ - X(g_t,g_0) \circ \phi_t\\
\phi_0 =&\ \Id.
\end{split}
\end{gather}
Now let
\begin{align*}
\bar{g}_t :=&\ \phi_t^* g_t\\
\bar{H}_t =&\ \phi_t^* H.
\end{align*}
We first compute
\begin{gather*}
\begin{split}
\dt \bar{g}_t =&\ \phi_t^* \left(\dt g_t \right) + \left. \frac{\del}{\del s}\right|_{s=0} (\phi_{t+s}^* (g_t))\\
=&\ \phi_t^* \left(-2 \Rc + \tfrac{1}{2} H^2 + L_{X(g_t,g_0)} g\right) - L_{(\phi_t^{-1})_* X(g_t,g_0)} (\phi_t^* g_t)\\
=&\ - 2 \Rc_{\bar{g}} + \tfrac{1}{2} \bar{H}^2,
\end{split}
\end{gather*}
where the quantity $\bar{H}^2$ naturally employs the metric $\bar{g}$.  Next we compute
\begin{gather*}
\begin{split}
\dt \bar{H}_t =&\ \phi_t^* \left(\dt H_t \right) + \left.\frac{\del}{\del s} \right|_{s=0} \left( \phi_{t+s}^* H_t \right)\\
=&\ \phi_t^* \left( \gD_{d,g} H + L_{X(g,g_0)} H \right) - L_{(\phi_t^{-1})_* X(g_t,g_0)} \phi_t^* (H_t)\\
=&\ \gD_{d,\bar{g}} \bar{H}.
\end{split}
\end{gather*}
We can now recover the $2$-form potential $b$ by direct integration.  That is, we are given $b_0$ and define $b_t$ via
\begin{align*}
\dt \bar{b}_t =&\ - d^*_{\bar{g}_t} \bar{H}.
\end{align*}
It follows directly by comparing the evolution equations that $H_0 + d \bar{b}_t = \bar{H}_t$ for all $t$.  Thus we have proved that $(\bar{g}_t, \bar{b}_t)$ is a solution to generalized Ricci flow, as claimed.
\vskip 0.1in
\textbf{Step 3: Uniqueness via coupled harmonic map heat flow}
\vskip 0.1in
A beautiful trick for establishing uniqueness of the Ricci flow equation reverses the above procedure, by connecting solutions to the Ricci flow to the appropriately gauge-fixed Ricci flow.  The subtlety is to discover the appropriate gauge transformation, which arises by expressing the ODE defining $\phi_t$ in terms of the metric $\bar{g}$, after which it follows that $\phi_t$ is a solution to the harmonic map heat flow.

In particular, given $(\bar{g}_t, \bar{b_t})$ a solution to generalized Ricci flow, let $\phi_t$ be the one-parameter family of diffeomorphisms which is the unique solution to 
\begin{gather} \label{f:ste50}
\begin{split}
\dt \phi_t =&\ \gD_{\bar{g}_t, \bar{g}_0} \phi_t\\
\phi_0 =&\ \Id.
\end{split}
\end{gather}
Here $\gD_{\bar{g}_t, \bar{g}_0}$ denotes the harmonic map Laplacian with source metric $\bar{g}_t$ and target metric $\bar{g}_0$.  As this is a strictly parabolic equation and $M$ is compact, solvability of this equation for some short-time $\ge > 0$ follows.  A fundamental computation (cf. \cite{CLNBook} Remark 2.50) shows that, letting $\Id$ denote the identity map of $M$,
\begin{gather} \label{f:ste60}
\begin{split}
\gD_{\bar{g}_t, \bar{g}_0} \phi_t = \gD_{(\phi_t^{-1})^* \bar{g}_t, \bar{g}_0} \Id = - X((\phi_t^{-1})^* \bar{g}_t, \bar{g}_0).
\end{split}
\end{gather}
Now set 
\begin{align*}
g_t =&\ (\phi_t^{-1})^* \bar{g}_t\\
H_t =&\ (\phi_t^{-1})^* \bar{H}_t
\end{align*}
A standard computation using the generalized Ricci flow equation and (\ref{f:ste60}) shows that $g_t$ satisfies the gauge-fixed generalized Ricci flow of (\ref{f:stegff}).

To finish the uniqueness argument, we suppose $(\bar{g}^i_t, \bar{b}^i_t)$, $i = 1,2$, are two solutions to generalized Ricci flow with the same initial data.  These have induced torsions $\bar{H}^i_t$.  By employing the gauge-fixing procedure described above, we produce two solutions $(g^i_t, H^i_t)$ of the gauge-fixed system (\ref{f:stegff}) with the same initial data.  This system is strictly parabolic, and enjoys uniqueness of solutions, thus $g^1_t = g^2_t, H^1_t = H^2_t$ for all times where either is defined.  As discussed above, we can now interpret the relevant diffeomorphisms used to gauge-fix, $\phi_t^i$, as solutions to the ODE (\ref{f:ste40}).  Since the two metrics $g^i_t$ agree, it follows that $\phi_t^1 = \phi_t^2$ for all relevant times, and thus $\bar{g}^1_t = \bar{g}^2_t, \bar{H}^1_t = \bar{H}^2_t$ for all times $t$.  It follows then that $\dt \bar{b}^1_t = \dt \bar{b}^2_t$ for all times, and thus $\bar{b}^1_t = \bar{b}^2_t$ for all $t$.
\end{proof}
\end{thm}

\begin{rmk} For various applications in geometry and mathematical physics it 
would be desirable to study the generalized Ricci flow on complete, noncompact 
spaces.  In this setting even short-time existence is a delicate issue, and 
examples such as a cylinder with arbitrarily small necks pinched at infinity 
suggest that a general short-time existence result should not be attainable.  
With some form of control over the initial metric, say bounded curvature, it 
seems likely that the flow should admit short-time solutions, as is known in 
the case of Ricci flow \cite{Shi}.  Uniqueness of these solutions within the class of solutions with bounded curvature was established in \cite{ChenZhu}, (cf. \cite{KotschwarUniqueness}).
\end{rmk}

\section{Curvature evolution equations}

Fundamental to the study of generalized Ricci flow is to understand the 
evolution of quantities determined by the associated curvature tensors.  Though 
the Bismut connection is in many ways the most relevant connection, the 
evolution of the Riemannian curvature is relevant for some applications.  As we will see, while the key aspect of the proof of Theorem \ref{t:STE} was overcoming the lack of strict parabolicity of the equation, all gauge-invariant quantities such as curvature, torsion, and their derivatives, will obey strictly parabolic equations.  The reason for this difference is because of the extra differential identity satisfied by curvature, namely the Bianchi identity, a manifestation of its gauge-invariance.

\subsection{Riemannian curvature evolution}

\begin{lemma} \label{l:Riemannvariation} 
Let $(E,[,],\IP{,}) \to M$ be an exact 
Courant algebroid.  Given $\GG_t$ a solution to generalized Ricci flow, one has
\begin{align*}
 \left( \dt - \gD \right) \Rm_{ijkl} =&\ 2 \left( Q_{jikl} - Q_{ijkl} - Q_{ikjl} + Q_{iljk} \right)\\
 &\ - R_{qjkl} \Rc_{iq} - R_{iqkl} \Rc_{jq} - R_{ijql} \Rc_{kq} - R_{ijkq} \Rc_{lq}\\
 &\ + \tfrac{1}{4} \left( \N_i \N_k H^2_{jl} - \N_i \N_l H^2_{jk} - \N_j \N_k H^2_{il} + \N_j \N_l H^2_{ik} + R_{ijkq} H^2_{ql} + R_{ijq l} H^2_{kq} \right).
\end{align*}
where
\begin{align} \label{f:curvQdef}
Q_{ijkl} := R_{p i j q} R_{p k l q}.
\end{align}
\begin{proof} We insert the evolution equation $\dt g = -2 \Rc + \tfrac{1}{2} 
H^2$ into the variational formula of Lemma \ref{l:curvaturevariation}, and 
compute the two terms separately.  Formulaically this takes the form
\begin{gather} \label{l:Riemannvar10}
\begin{split}
\dt \Rm_{ijkl} =&\ \left(\dt \Rm \right)_{ijk}^p g_{pl} + \Rm_{ijk}^p \left(\dt g \right)_{pl}\\
=&\ \left\{ \left( \N^2 \Rc + \N^2 H^2 \right)_{ijk}^p \right\}g_{pl} + \Rm_{ijk}^p \left( -2 \Rc_{pl} + \tfrac{1}{2} H^2_{pl} \right).
\end{split}
\end{gather}
First, the contribution of the Ricci tensor yields
\begin{gather} \label{l:Riemannvar20}
\begin{split}
\left(\N^2 \Rc \right)_{ijk}^p g_{pl} =&\ - \left[ \N_i \N_k \Rc_{jl} - \N_i 
\N_l 
\Rc_{jk} - \N_j \N_k \Rc_{il} + \N_j \N_l \Rc_{ik} - R_{ijk}^p \Rc_{pl} - 
R_{ijl}^p \Rc_{kp} \right].
\end{split}
\end{gather}
To see that the second derivative terms above are really $\gD \Rm$ in disguise, we need to use the Bianchi identity.  In particular, using several applications of the Bianchi identity, commuting derivatives and noting the definition of $Q$ (\ref{f:curvQdef}) we see
\begin{gather}
\begin{split}
\gD R_{ijkl} =&\ \N_p \N_p R_{ijkl}\\
=&\ - \N_p \N_j R_{pikl} - \N_p \N_i R_{jpkl}\\
=&\ - \N_j \N_p R_{pikl} - R_{jppq} R_{qikl} - R_{jpiq} R_{pqkl} - R_{jpkq} R_{piql} - R_{jplq} R_{pikq}\\
&\ - \N_i \N_p R_{jpkl} - R_{ipjq} R_{qpkl} - R_{ippq} R_{jqkl} - R_{ipkq} R_{jpql} - R_{iplq} R_{jpkq}\\
=&\ - \N_j \N_p R_{klpi}  - \Rc_{jq} R_{qikl} + R_{jpiq} (R_{plqk} + R_{pklq}) - Q_{jkil} + Q_{jlik}\\
&\ - \N_i \N_p R_{kljp} + R_{ipjq} (R_{qlpk} + R_{qklp}) - \Rc_{iq} R_{jqkl} + Q_{ikjl} - Q_{iljk}\\
=&\ \N_j \N_l R_{pkpi} + \N_j \N_k R_{lppi} - \Rc_{jq} R_{qikl} + Q_{jilk} - Q_{jikl} - Q_{jkil} + Q_{jlik}\\
&\ + \N_i \N_l R_{pkjp} + \N_i \N_k R_{lpjp} - \Rc_{iq} R_{jqkl} + Q_{ijkl} - Q_{ijlk} + Q_{ikjl} - Q_{iljk}\\
=&\ - \N_j \N_l \Rc_{ki} + \N_j \N_k \Rc_{li} + \N_i \N_l \Rc_{kj} - \N_i \N_k \Rc_{lj} + \Rc_{jq} R_{iqkl} + \Rc_{iq} R_{qjkl}\\
&\ + 2 \left( Q_{ijkl} - Q_{jikl} + Q_{ikjl} - Q_{iljk} \right).
\end{split}
\end{gather}
Next, plugging directly into the result of Lemma \ref{l:curvaturevariation} we 
simply record
\begin{gather} \label{l:Riemannvar30}
\begin{split}
\left( \N^2 H^2 \right)_{ijk}^p g_{pl} =&\ \tfrac{1}{4} \left[ \N_i \N_k 
H^2_{jl} - \N_i \N_l 
H^2_{jk} - \N_j \N_k H^2_{il} + \N_j \N_l H^2_{ik} \right.\\
&\ \qquad \left. - R_{ijkq} H^2_{ql} - 
R_{ijlq} H^2_{kq} \right].
\end{split}
\end{gather}
Combining (\ref{l:Riemannvar10})-(\ref{l:Riemannvar30}) yields the claim.
\end{proof}
\end{lemma}

\begin{lemma} \label{l:Riccivar} Let $(E,[,],\IP{,}) \to M$ be an exact 
Courant algebroid.  Given $\GG_t$ a solution to generalized Ricci flow, one has
\begin{align*}
 \left( \dt - \gD \right) \Rc =&\ 2 \Rm \circ \left( \Rc - \tfrac{1}{4} H^2 \right) - \Rc \circ \left( \Rc - \tfrac{1}{4} H^2 \right)\\
 &\ \qquad  + \tfrac{1}{4} \left[ \gD H^2 - \N^2 \brs{H}^2 + \N \circ \divg H^2 \right].
\end{align*}
\begin{proof} We return to Lemma \ref{l:curvaturevariation} to directly compute
\begin{gather} \label{l:Riccivar10}
\begin{split}
\dt \Rc_{jk} =&\ \dt R_{ijk}^i\\
=&\ \left( \N^2 \Rc + \N^2 H^2 \right)_{ijk}^i\\
&\ \qquad  - \tfrac{1}{2} g^{qi} \left( R_{ijk}^p \left(-2 \Rc + \tfrac{1}{2}H^2 \right)_{pq} + R_{ijq}^p \left(-2 \Rc + \tfrac{1}{2} H^2 \right)_{kp} \right)\\
=&\ \left( \N^2 \Rc + \N^2 H^2 \right)_{ijk}^i + R_{ijk}^p \Rc_p^i - \Rc_j^p \Rc_{pk}\\
&\ \qquad - \tfrac{1}{4} R_{ijk}^p (H^2)_{p}^i + \tfrac{1}{4} \Rc_{j}^p (H^2)_{kp}.
\end{split}
\end{gather}
Next we simplify, commuting derivatives and applying the second Bianchi identity,
\begin{gather} \label{l:Riccivar20}
\begin{split}
\left( \N^2 \Rc \right)_{ijk}^i =&\ g^{qi} \left[ -\N_i \N_k \Rc_{jq} + \N_i \N_q \Rc_{jk} + \N_j \N_k \Rc_{iq} - \N_j \N_q \Rc_{ik} \right]\\
=&\ \gD \Rc_{jk} + \N_j \N_k R - \N_j \left( \divg \Rc \right)_{k} - g^{qi} \N_i \N_k \Rc_{jq}\\
=&\ \gD \Rc_{jk} + \tfrac{1}{2} \N_j \N_k R - g^{qi} \N_k \N_i \Rc_{qj} - g^{qi} R_{kij}^p \Rc_{pq} - g^{qi} R_{kiq}^p \Rc_{jp}\\
=&\ \gD \Rc_{jk} + \tfrac{1}{2} \N_j \N_k R - \N_k (\divg \Rc)_{j} - R_{kij}^p \Rc_p^i - \Rc_k^p \Rc_{pj}\\
=&\ \gD \Rc_{jk} - R_{kij}^p \Rc_p^i - \Rc_k^p \Rc_{pj}.
\end{split}
\end{gather}
Also we simplify by commuting derivatives,
\begin{gather} \label{l:Riccivar30}
\begin{split}
\left( \N^2 H^2 \right)_{ijk}^i =&\ \tfrac{1}{4} g^{qi} \left[ \N_i \N_k H^2_{jq} - \N_i \N_q H^2_{jk} - \N_j \N_k H^2_{iq} + \N_j \N_q H^2_{ik} \right]\\
=&\ \tfrac{1}{4} \gD H^2_{jk} - \tfrac{1}{4} \N_j \N_k \brs{H}^2 + \tfrac{1}{4} \N_j (\divg H^2)_k + \tfrac{1}{4} g^{qi} \N_i \N_k H_{jq}^2\\
=&\ \tfrac{1}{4} \gD H^2_{jk} - \tfrac{1}{4} \N_j \N_k \brs{H}^2 + \tfrac{1}{4} \N_j (\divg H^2)_k + \tfrac{1}{4} g^{qi}  \N_k \N_i H_{jq}^2\\
&\ \qquad + \tfrac{1}{4} g^{qi} R_{kij}^p (H^2)_{pq} + \tfrac{1}{4} g^{qi} R_{kiq}^p H^2_{jp}\\
=&\ \tfrac{1}{4} \left[ \gD H^2_{jk} - \N_j \N_k \brs{H}^2 + \N_j (\divg H^2)_k + \N_k (\divg H^2)_j \right.\\
&\ \left. \qquad + g^{qi} R_{kij}^p (H^2)_{pq} + g^{qi} R_{kiq}^p H^2_{jp} \right].
\end{split}
\end{gather}
Combining (\ref{l:Riccivar10})-(\ref{l:Riccivar30}) yields the result.
\end{proof}
\end{lemma}

\begin{lemma} \label{l:riemscalarevolution} Let $(E,[,],\IP{,}) \to M$ be an exact 
Courant algebroid.  Given $\GG_t$ a solution to generalized Ricci flow, one has
\begin{align*}
 \left( \dt - \gD \right) R =&\ \gD R - \tfrac{1}{2} \gD \brs{H}^2 + \tfrac{1}{2} \divg \divg H^2 + 2 \IP{\Rc, \Rc - \tfrac{1}{4} H^2}.
\end{align*} 
\begin{proof} We leave as an \textbf{exercise} to derive this from Lemma \ref{l:Riccivar}.  Here we derive the result directly from the variation formula for scalar curvature in Lemma \ref{l:curvaturevariation}.  Specifically, plugging the generalized Ricci flow equation into the variation formula and applying Bianchi identities and Lemma \ref{l:divHlemma} yields
\begin{align*}
\dt R =&\ - \gD \tr_g \left( - 2 \Rc + \tfrac{1}{2} H^2 \right) + \divg \divg \left( - 2 \Rc + \tfrac{1}{2} H^2 \right) - \IP{\Rc, \left( - 2 \Rc + \tfrac{1}{2} H^2 \right)}\\
=&\ \gD R - \tfrac{1}{2} \gD \brs{H}^2 + \tfrac{1}{2} \divg \divg H^2 + 2 \IP{\Rc, \Rc - \tfrac{1}{4} H^2},
\end{align*}
as claimed.
\end{proof}
\end{lemma}

\begin{lemma} \label{l:gradkRmev} Let $(E,[,],\IP{,}) \to M$ be an exact 
Courant algebroid.  Given $\GG_t$ a solution to generalized Ricci flow, one has
\begin{align*}
 \left( \dt - \gD \right) \N^k \Rm =&\ \N^{k+2} H^2 + \sum_{j=0}^k \N^j \left( \Rm + H^2 \right) \star \N^{k-j} \Rm,\\
 \left(\dt - \gD \right) \brs{\N^k \Rm}_{g_t}^2 =&\ - 2 \brs{\N^{k+1} \Rm}^2_{g_t} + \N^{k+2} H^2 \star \N^k \Rm\\
 &\ + \sum_{j=0}^k \N^j \left( \Rm + H^2 \right) \star \N^{k-j} \Rm \star \N^k \Rm.
\end{align*}
\begin{proof} We compute, using the result of Lemma \ref{l:Riemannvariation},
\begin{align*}
\dt \N^k \Rm =&\ \dt \left\{ \left( \del + \gG \right) \dots \left( \del + \gG \right) \Rm \right\}\\
=&\ \sum_{j=0}^{k-1} \N^j \left\{ \dt \gG \star \N^{k-1-j} \Rm \right\} + \N^k \left( \dt \Rm \right)\\
=&\ \sum_{j=0}^{k-1} \N^j \left\{ \N \left(\Rc + H^2 \right) \star \N^{k-1-j} \Rm \right\} + \N^k \left( \gD \Rm + \Rm^{\star 2} + \Rm \star H^2 + \N^2 H^2 \right)\\
=&\ \N^k \gD \Rm + \N^{k+2} H^2 + \sum_{j=0}^k \N^j \left( \Rm + H^2 \right) \star \N^{k-j} \Rm\\
=&\ \gD \N^k \Rm + \N^{k+2} H^2 + \sum_{j=0}^k \N^j \left( \Rm + H^2 \right) \star \N^{k-j} \Rm,
\end{align*}
where the last line follows by commuting derivatives.  This yields the first equation, and the second follows as an elementary \textbf{exercise}, noting that the terms arising from the variation of the inner product yield more terms of the form $j=0$ in the final sum.
\end{proof}
\end{lemma}

\begin{lemma} \label{l:gradkHev} Let $(E,[,],\IP{,}) \to M$ be an exact 
Courant algebroid.  Given $\GG_t$ a solution to generalized Ricci flow, one has
\begin{align*}
 \left( \dt - \gD \right) \N^k H =&\ \sum_{j=0}^k \N^j \Rm \star \N^{k-j} H + \sum_{j=1}^k \N^j H^2 \star \N^{k-j} H,\\
 \left( \dt - \gD \right) \brs{\N^k H}^2_{g_t} =&\ - 2 \brs{\N^{k+1} H}^2 + \sum_{j=0}^k \N^j \Rm \star \N^{k-j} H \star \N^k H\\
 &\ + \sum_{j=1}^k \N^j H^2 \star \N^{k-j} H \star \N^k H.
\end{align*}
\begin{proof} We compute, using the result of Lemma \ref{l:Hevolution} and the Bochner formula,
\begin{align*}
\dt \N^k H =&\ \dt \left\{ \left( \del + \gG \right) \dots \left( \del + \gG \right) H \right\}\\
=&\ \sum_{j=0}^{k-1} \N^j \left\{ \dt \gG \star \N^{k-1-j} H \right\} + \N^k \left( \dt H \right)\\
=&\ \sum_{j=0}^{k-1} \N^j \left\{ \N \left(\Rc + H^2 \right) \star \N^{k-1-j} H \right\} + \N^k \left( \gD H + \Rm \star H \right)\\
=&\ \N^k \gD H + \sum_{j=0}^{k-1} \N^{j+1} \left( \Rm + H^2 \right) \star \N^{k-1-j} H + \sum_{j=0}^k \N^j \Rm \star \N^{k-j} H\\
=&\ \gD \N^k H + \sum_{j=0}^k \N^j \Rm \star \N^{k-j} H + \sum_{j=1}^k \N^j H^2 \star \N^{k-j} H,
\end{align*}
where in the last line we have reindexed one sum and commuted derivatives.  This proves the first formula, and again we leave the derivation of the second equation as an \textbf{exercise}.
\end{proof}
\end{lemma}

\subsection{Bismut curvature evolution}

\begin{lemma} \label{l:BismutcurvatureLaplacian}
Let $(M^n, g)$ be a smooth manifold with closed three-form $H$.  Then
\begin{align*}
\gD^+ R^+_{ijkl} =&\ - \N^+_j \N^+_l \Rc^+_{ik} + \N^+_j \N^+_k \Rc^+_{il} + \N^+_i \N^+_l \Rc^+_{jk} - \N^+_i \N^+_k \Rc^+_{jl}\\
&\  - \Rc^+_{jq} R^+_{qikl} - R^+_{jpiq} R^+_{pqkl} - Q^+_{jkil} + Q^+_{jlik}\\
&\ - R^+_{ipjq} R^+_{qpkl} - \Rc^+_{iq} R^+_{jqkl} + Q^+_{ikjl} - Q^+_{iljk}\\
&\ + \TT_1(\N^+ \N^+ (H^2)) + \TT_2(\N^+(H \star R^+)),
\end{align*}
where
\begin{align*}
& Q^+_{ijkl} = R^+_{p i j q} R^+_{p k l q},\\
& \TT_1(\N^+ \N^+ H^2)_{ijkl} = \N^+_i \N^+_p \sum_{\gs(kpl)} H_{kpq} H_{ljq} - \N^+_j \N^+_p \sum_{\gs(kpl)} H_{kpq} H_{liq},\\
& \TT_2 (\N^+ (H \star R^+))_{ijkl} = \\
&\ \qquad - \N^+_p \left( H_{ij}^q R^+_{qpkl} + H_{pi}^q R^+_{qjkl} + H_{jp}^q R^+_{qikl} \right)\\
&\ \qquad + \N^+_j \left( H_{pk}^q R^+_{qlpi} + H_{lp}^q R^+_{qkpi} + H_{kl}^q \Rc^+_{qi} - \Rc^+_{kq} H_{qli} + R^+_{pklq} H_{pqi} + R^+_{pkiq} H_{plq} \right.\\
&\ \qquad \qquad \left. + R^+_{pliq} H_{qkp} + R^+_{plkq} H_{iqp} - \Rc^+_{lq} H_{ikq} \right)\\
&\ \qquad + \N^+_i \left( H_{pk}^q R^+_{qljp} + H_{lp}^q R^+_{qkjp} - H_{kl}^q \Rc^+_{qj} + \Rc^+_{kq} H_{qlj} - R^+_{pklq} H_{pqj} - R^+_{pkjq} H_{plq} \right.\\
&\ \qquad \qquad \left. - R^+_{pljq} H_{qkp} - R^+_{plkq} H_{jqp} + \Rc^+_{lq} H_{jkq} \right).
\end{align*}
\begin{proof} We work in local coordinates.  Also, we drop the superscript on the connection and curvature for notational simplicity, with all objects in this proof associated to the Bismut connection.  First we note that, using that $H$ is closed we can simplify formula (\ref{f:Bianchi4}) to 
\begin{gather} \label{f:BCL10}
\begin{split}
R_{ijkl} - R_{klij} =&\ \tfrac{1}{2} \left( \N_k H_{ilj} + \N_j H_{kil} + \N_l H_{jki} + \N_i H_{ljk} \right)\\
=&\  \N_k H_{ilj} + \N_l H_{jki}  - \tfrac{1}{2} (d^{\N} H)_{kilj}\\
=&\  \N_k H_{ilj} + \N_l H_{jki}  + \sum_{\gs(kil)} H_{kip} H_{ljp}\\
=&\ \N_j H_{kil} + \N_i H_{ljk} + \sum_{\gs(jki)} H_{jkp} H_{ilp}.
\end{split}
\end{gather}

Using Proposition \ref{p:Bianchi} and (\ref{f:BCL10}) we obtain, with all objects associated to the Bismut connection,
\begin{gather} \label{f:BCL20}
\begin{split}
\gD R_{ijkl} =&\ \N_p \N_p R_{ijkl}\\
=&\ \N_p \left( - \N_j R_{pikl} - \N_i R_{jpkl} - H_{ij}^q R_{qpkl} - H_{pi}^q R_{qjkl} - H_{jp}^q R_{qikl}  \right)\\
=&\ - \N_j \N_p R_{pikl} - R_{jppq} R_{qikl} - R_{jpiq} R_{pqkl} - R_{jpkq} R_{piql} - R_{jplq} R_{pikq}\\
&\ - \N_i \N_p R_{jpkl} - R_{ipjq} R_{qpkl} - R_{ippq} R_{jqkl} - R_{ipkq} R_{jpql} - R_{iplq} R_{jpkq}\\
&\ - \N_p \left( H_{ij}^q R_{qpkl} + H_{pi}^q R_{qjkl} + H_{jp}^q R_{qikl} \right)\\
=&\ - \N_j \N_p \left( R_{klpi} + \N_k H_{pli} + \N_l H_{ikp}  + \sum_{\gs(kpl)} H_{kpq} H_{liq} \right)\\
&\ - \N_i \N_p \left( R_{kljp} - \N_k H_{plj} - \N_l H_{jkp} - \sum_{\gs(kpl)} H_{kpq} H_{ljq} \right)\\
&\  - \Rc_{jq} R_{qikl} - R_{jpiq} R_{pqkl} - Q_{jkil} + Q_{jlik}\\
&\ - R_{ipjq} R_{qpkl} - \Rc_{iq} R_{jqkl} + Q_{ikjl} - Q_{iljk}\\
&\ - \N_p \left( H_{ij}^q R_{qpkl} + H_{pi}^q R_{qjkl} + H_{jp}^q R_{qikl} \right).
\end{split}
\end{gather}
We next analyze the highest order terms appearing in the first line of the final equality above.  First,
\begin{align*}
\N_p R_{klpi} =&\ - \N_l R_{pkpi} - \N_k R_{lppi} - H_{pk}^q R_{qlpi} - H_{lp}^q R_{qkpi} - H_{kl}^q R_{qppi}\\
=&\ \N_l \Rc_{ki} - \N_k \Rc_{li} - H_{pk}^q R_{qlpi} - H_{lp}^q R_{qkpi} - H_{kl}^q \Rc_{qi}.
\end{align*}
Next
\begin{align*}
\N_p R_{kljp} =&\ - \N_l R_{pkjp} - \N_k R_{lpjp} - H_{pk}^q R_{qljp} - H_{lp}^q R_{qkjp} - H_{kl}^q R_{qpjp}\\
=&\ - \N_l \Rc_{kj} + \N_k \Rc_{lj} - H_{pk}^q R_{qljp} - H_{lp}^q R_{qkjp} + H_{kl}^q \Rc_{qj}.
\end{align*}
Next
\begin{align*}
\N_p \N_k H_{pli} =&\ \N_k \N_p H_{pli} - R_{pkpq} H_{qli} - R_{pklq} H_{pqi} - R_{pkiq} H_{plq}\\
=&\ - \N_k d^* H_{li} - R_{pkpq} H_{qli} - R_{pklq} H_{pqi} - R_{pkiq} H_{plq}\\
=&\ \N_k \left( \Rc_{li} - \Rc_{il} \right) - R_{pkpq} H_{qli} - R_{pklq} H_{pqi} - R_{pkiq} H_{plq}.
\end{align*}
Next
\begin{align*}
\N_p \N_l H_{ikp} =&\ \N_l \N_p H_{ikp} - R_{pliq} H_{qkp} - R_{plkq} H_{iqp} - R_{plpq} H_{ikq}\\
=&\ - \N_l d^* H_{ik} - R_{pliq} H_{qkp} - R_{plkq} H_{iqp} - R_{plpq} H_{ikq}\\
=&\ \N_l \left( \Rc_{ik} - \Rc_{ki} \right) - R_{pliq} H_{qkp} - R_{plkq} H_{iqp} - R_{plpq} H_{ikq}.
\end{align*}
Next
\begin{align*}
\N_p \N_k H_{plj} =&\ \N_k \N_p H_{plj} - R_{pkpq} H_{qlj} - R_{pklq} H_{pqj} - R_{pkjq} H_{plq}\\
=&\ - \N_k d^* H_{lj} - R_{pkpq} H_{qlj} - R_{pklq} H_{pqj} - R_{pkjq} H_{plq}\\
=&\ \N_k \left( \Rc_{lj} - \Rc_{jl} \right) - R_{pkpq} H_{qlj} - R_{pklq} H_{pqj} - R_{pkjq} H_{plq}.
\end{align*}
Lastly
\begin{align*}
\N_p \N_l H_{jkp} =&\ \N_l \N_p H_{jkp} - R_{pljq} H_{qkp} - R_{plkq} H_{jqp} - R_{plpq} H_{jkq}\\
=&\ - \N_l d^* H_{jk} - R_{pljq} H_{qkp} - R_{plkq} H_{jqp} - R_{plpq} H_{jkq}\\
=&\ \N_l \left( \Rc_{jk} - \Rc_{kj} \right) - R_{pljq} H_{qkp} - R_{plkq} H_{jqp} - R_{plpq} H_{jkq}.
\end{align*}
Inserting the above computations into (\ref{f:BCL20}) gives the result.
\end{proof}
\end{lemma}

\begin{lemma} \label{l:Bismutconnectionev}
Let $(E,[,],\IP{,}) \to M$ be an exact 
Courant algebroid.  Given $\GG_t$ a solution to generalized Ricci flow, one has
\begin{align*}
\dt \left(\N^+\right)_{ij}^k =&\ - g^{kl} \left( \N^+_i \Rc^+_{lj} + \N^+_j \Rc^+_{il} - \N^+_l \Rc^+_{ij} \right) + g^{kl} \left( H_{ij}^p \Rc^+_{pl} - H_{il}^p \Rc^+_{pj} - H_{jl}^p \Rc^+_{ip} \right).
\end{align*}
\begin{proof} Using the result of Lemma \ref{l:BismutGRF} expressed in terms of the notation of Lemma \ref{l:Bcurvaturevariation}, we have
\begin{align*}
A_{ij} =&\ -2 (\Rc^+)^{\mbox{sym}}_{ij} + 2 (\Rc^+)^{\mbox{skew}}_{ij}\\
=&\ - \Rc^+_{ij} - \Rc^+_{ji} + \Rc^+_{ij} - \Rc^+_{ji}\\
=&\ - 2 \Rc^+_{ji}
\end{align*}
Plugging this into the result of Lemma \ref{l:Bcurvaturevariation} gives the result.
\end{proof}
\end{lemma}

\begin{prop} \label{p:Bismutcurvatureev} Let $(E,[,],\IP{,}) \to M$ be an exact 
Courant algebroid.  Given $\GG_t$ a solution to generalized Ricci flow, one has
\begin{align*}
 \left( \dt - \gD^+ \right) R^+_{ijkl} =&\ R^+ \star R^+ + \N^+ \N^+ (H^2) + \N^+ (H \star R^+).
\end{align*}
\begin{proof} Using Lemma \ref{l:Bcurvaturevariation} and the result of Lemma \ref{l:Bismutconnectionev} we obtain
\begin{align*}
\dt \left( R^+ \right)_{ijkl} =&\ \dt \left( (R^+)_{ijk}^q g_{ql} \right)\\
=&\ \left(\N^+_i \left(\dot{\N}^+ \right)_{jk}^q - \N^+_j \left(\dot{\N}^+ \right)_{ik}^q + 
H_{ij}^p \left(\dot{\N}^+ \right)_{pk}^q \right)  g_{ql} - 2 (R^+)_{ijk}^q (\Rc^{\Sym})_{ql}\\
=&\ \N_i^+ \left( - \left( \N^+_j \Rc^+_{lk} + \N^+_k \Rc^+_{jl} - \N^+_l \Rc^+_{jk} \right) + \left( H_{jk}^q \Rc^+_{lq} - H_{jl}^q \Rc^+_{kq} - H_{kl}^q \Rc^+_{qj} \right) \right)\\
&\ - \N_j^+ \left( - \left( \N^+_i \Rc^+_{lk} + \N^+_k \Rc^+_{il} - \N^+_l \Rc^+_{ik} \right) + \left( H_{ik}^q \Rc^+_{lq} - H_{il}^q \Rc^+_{kq} - H_{kl}^q \Rc^+_{qi} \right) \right)\\
&\ + H_{ij}^p \left(- \left( \N^+_p \Rc^+_{kl} + \N^+_k \Rc^+_{lp} - \N^+_l \Rc^+_{kp} \right) + \left( H_{pk}^r \Rc^+_{lr} - H_{pl}^r \Rc^+_{kr} - H_{kl}^r \Rc^+_{rp} \right) \right)\\
&\ - 2 (R^+)_{ijk}^q (\Rc^{\Sym})_{ql}\\
=:&\ \sum_{i=1}^{19} A_i.
\end{align*}
We first simplify
\begin{align*}
A_1 + A_7 =&\ g^{lp} \left( \N_j^+ \N_i^+ - \N_i^+ \N_j^+ \right) \Rc^+_{kp} = R^+_{ijkq} \Rc^+_{qp} + R^+_{ijpq} \Rc^+_{kq}.
\end{align*}
Next using Lemma \ref{l:BismutcurvatureLaplacian} we have
\begin{align*}
A_2 + A_3 + A_8 + A_9 =&\ \gD^+ R^+_{ijkl} + \Rc^+_{jq} R^+_{qikl} + R^+_{jpiq} R^+_{pqkl} + Q^+_{jkil} - Q^+_{jlik}\\
&\ + R^+_{ipjq} R^+_{qpkl} + \Rc^+_{iq} R^+_{jqkl} - Q^+_{ikjl} + Q^+_{iljk}\\
&\ - \TT_1(\N^2 (H^2)) - \TT_2(\N(H \star R)).
\end{align*}
The claimed formula now follows, with the precise forms of the different expressions implicit by the formulas above.
\end{proof}
\end{prop}

\begin{question} 
A fundamental property of the Ricci flow is that the holonomy of the Levi-Civita connection is preserved along the flow \cite{HamiltonSing}.  While the holonomy can reduce in singular or infinite time limits, it will not reduce in finite time for smooth solutions \cite{KotschwarHolonomy}.  It is natural to ask if there is a similar behavior for the generalized Ricci flow, namely, if there is a natural notion of holonomy in generalized geometry which is preserved along the flow.  Due to the lack of a canonical Levi-Civita connection in generalized geometry (see Chapter \ref{c:GCC}), it is not even completely clear what the correct notion of holonomy should be.  A possible candidate is given by the holonomy of the classical Bismut connection.
\end{question}

\section{Smoothing estimates}

A fundamental property of the heat equation on $\mathbb R^n$ is the concept of 
``smoothing,'' namely that a solution with low regularity becomes 
instantaneously smooth and moreover obtains uniform estimates on all 
derivatives in terms of the elapsed time.  This is a very general feature of 
parabolic equations, where typically in the nonlinear setting some kind of a 
priori control is needed before the smoothing effect can take over.  As a 
nonlinear degenerate parabolic equation, the generalized Ricci flow exhibits 
this behavior, where a bound on the curvature tensor implies a smoothing effect 
for all derivatives of the curvature.  This result sets the stage for obtaining 
long time existence results for the flow, indicating that a bound on the 
curvature tensor suffices to completely control the flow, which we prove in the 
next section.

\begin{thm} \label{t:smoothing} Let $(E, [,], \IP{,}) \to M$ be an exact 
Courant algebroid with $M$ compact.  Let $\GG_t$ be a solution to generalized Ricci flow on $[0,T_0], 0 < T_0 \leq \frac{\ga}{K}$ satisfying
\begin{align} \label{globalsmoothinghyp}
\sup_{M \times [0,T_0]} \left\{ \brs{\Rm} + \brs{\N H} + \brs{H}^2 \right\} \leq K.
\end{align}
Given $k \in \mathbb N$, there exists $C = C(n,k,\ga)$ such that
\begin{align} \label{globalsmoothingconc}
\sup_{M \times [0,T_0]} \left( \brs{\N^k \Rm} + \brs{\N^{k+1} H} + \brs{H} \brs{\N^k H} \right) \leq \frac{ C K}{t^{k/2}}.
\end{align}
\begin{proof} We begin with the case $k=1$.  Consider the test function
\begin{align*}
F = t \left( \brs{\N \Rm}^2 + \brs{\N^2 H}^2 \right) + P \brs{\Rm}^2 + Q \brs{\N H}^2,
\end{align*}
where $P \geq 1$ is to be chosen below.  Combining Lemmas \ref{l:gradkRmev}, \ref{l:gradkHev}, and the Cauchy-Schwarz inequality yields
\begin{align*}
\left(\dt - \gD \right) F =&\ t \left(- 2 \brs{\N^{2} \Rm}^2 + \N^{3} H^2 \star \N \Rm + \sum_{j=0}^1 \N^j \left( \Rm + H^2 \right) \star \N^{1-j} \Rm \star \N \Rm \right)\\
&\ + t \left( - 2 \brs{\N^{3} H}^2 + \sum_{j=0}^2 \N^j \Rm \star \N^{2-j} H \star \N^2 H+ \sum_{j=1}^2 \N^j H^2 \star \N^{2-j} H \star \N^2 H \right)\\
&\ +\left(1 - 2 P \right) \brs{\N \Rm}^2 + \left(1 - 2 P \right) \brs{\N^2 H}^2\\
&\ + P \left( \N^2 H^2 \star \Rm + \Rm \star \Rm \star (\Rm + H^2) \right)\\
&\ + P\left( \sum_{j=0}^1 \N^j \Rm \star \N^{1-j} H \star \N H + \N H^2 \star H \star \N H \right)\\
\leq&\ t \left( C \brs{\N^3 H} \brs{H} \brs{\N \Rm} + C \left( \brs{\N \Rm}^2 \brs{\Rm} + \sum_{j=0}^1 \brs{\N^j H^2} \brs{\N^{1-j} \Rm} \brs{\N \Rm} \right. \right.\\
&\ \left. \left. + \sum_{j=0}^2 \brs{\N^j \Rm} \brs{\N^{2-j} H} \brs{\N^2 H} + \sum_{j=1}^2 \brs{\N^j H^2} \brs{\N^{2-j} H} \brs{\N^2 H} \right) -2 \brs{\N^3 H}^2 \right)\\
&\ - P \brs{\N \Rm}^2 - P \brs{\N^2 H}^2 + P \left( \brs{\N^2 H^2} \brs{\Rm} + \brs{\Rm}^3 + \brs{\Rm}^2 \brs{H}^2 \right)\\
&\ + P \left( \brs{\Rm} \brs{\N H}^2 + \brs{\N \Rm} \brs{H} \brs{\N H} + \brs{\N H^2} \brs{H} \brs{\N H} \right).
\end{align*}
Applying \ref{globalsmoothinghyp} and the arithmetic-geometric mean we obtain
\begin{align*}
\left(\dt - \gD \right) F \leq&\ \brs{\N \Rm}^2 \left( C K t - P\right) + \brs{\N^2 H}^2 \left(C K t - \frac{P}{2} \right) + C P K^3 (1 + tK).
\end{align*}
Using that $t \leq \frac{\ga}{K}$, we can choose $P$ large with respect to universal constants and $\ga$ to obtain
\begin{align*}
\left(\dt - \gD \right) F \leq&\ C K^3.
\end{align*}
Applying the maximum principle on this time interval we obtain, for $t \leq \frac{\ga}{K}$,
\begin{align*}
\sup_{M \times \{t\}} F \leq&\ \sup_{M \times \{0\}} F + C K^3 t \leq C K^2 (1 + \ga).
\end{align*}
This finishes the proof for the case $k = 1$.

We now prove the inductive step.  Assuming the result holds for all values $p \leq k-1$, define
\begin{align*}
F = t^k \left( \brs{\N^k \Rm}^2 + \brs{\N^{k+1} H}^2 \right) + \sum_{j=0}^{k-1} t^j P_j \left( \brs{\N^j \Rm}^2 + \brs{\N^{j+1} H}^2 \right).
\end{align*}
We will derive the relevant differential inequality for $F$ in stages.  First of all we apply Lemmas \ref{l:gradkRmev} and \ref{l:gradkHev} to obtain, for times $t > 0$,
\begin{gather} \label{f:globsmooth10}
\begin{split}
\left(\dt - \gD \right) & \left( \brs{\N^k \Rm}^2 + \brs{\N^{k+1} H}^2 \right)\\
\leq&\ - 2 \brs{\N^{k+1} \Rm}^2 + \N^{k+2} H^2 \star \N^k \Rm + \sum_{j=0}^k \N^j \left( \Rm + H^2 \right) \star \N^{k-j} \Rm \star \N^k \Rm\\
 &\ - 2 \brs{\N^{k+2} H}^2 + \sum_{j=0}^{k+1} \N^j \Rm \star \N^{k+1-j} H \star \N^{k+1} H+ \sum_{j=1}^{k+1} \N^j H^2 \star \N^{k+1-j} H \star \N^{k+1} H.
\end{split}
\end{gather}
We now estimate some terms in (\ref{l:gradkRmev}), using the induction hypothesis.  First
\begin{gather} \label{f:globsmooth20}
\begin{split}
\brs{\N^{k+2} H^2 \star \N^k \Rm} \leq&\ C \left(\brs{\N^{k+2} H} \brs{H} + \brs{\N^{k+1} H} \brs{\N H} + \sum_{j=2}^{k} \brs{\N^j H} \brs{\N^{k+2-j} H} \right) \brs{\N^k \Rm}\\
\leq&\ \frac{1}{2} \brs{\N^{k+2} H}^2 + C \brs{H}^2 \brs{\N^k \Rm}^2 + C K \left( \brs{\N^{k+1} H}^2 + \brs{\N^k \Rm}^2 \right)\\
&\ + C \frac{K}{t^{(j-1)/2}} \frac{K}{t^{(k+1-j)/2}} \brs{\N^k \Rm}\\
\leq&\ \frac{1}{2} \brs{\N^{k+2} H}^2 + C K \left( \brs{\N^{k+1} H}^2 + \brs{\N^k \Rm}^2 \right) + C \frac{K^3}{t^k}.
\end{split}
\end{gather}
Similarly we estimate
\begin{gather} \label{f:globsmooth30}
\begin{split}
& \brs{\sum_{j=0}^{k+1} \N^j \Rm \star \N^{k+1-j} H \star \N^{k+1} H}\\
& \qquad \leq C \left( \brs{\N^{k+1} \Rm} \brs{H}+ \brs{\N^k \Rm} \brs{\N H} + \sum_{j=0}^{k-1} \brs{\N^j \Rm} \brs{\N^{k+1-j} H} \right)  \brs{\N^{k+1} H}\\
& \qquad \leq \tfrac{1}{2} \brs{\N^{k+1} \Rm}^2 + C \brs{H}^2 \brs{\N^{k+1} H}^2 + C K \left( \brs{\N^k \Rm}^2 + \brs{\N^{k+1} H}^2 \right)\\
& \qquad \qquad+ C \frac{K}{t^{j/2}} \frac{K}{t^{(k-j)/2}} \brs{\N^{k+1} H}^2\\
& \qquad \leq \tfrac{1}{2} \brs{\N^{k+1} \Rm}^2 + C K \left( \brs{\N^k \Rm}^2 + \brs{\N^{k+1} H}^2 \right) + C \frac{K^3}{t^k}.
\end{split}
\end{gather}
We leave as an elementary \textbf{exercise} to estimate the two remaining terms of (\ref{f:globsmooth10}) using the induction hypothesis via
\begin{gather} \label{f:globsmooth40}
\begin{split}
\brs{ \sum_{j=0}^k \N^j \left( \Rm + H^2 \right) \star \N^{k-j} \Rm \star \N^k \Rm} \leq&\ C K \left( \brs{\N^k \Rm}^2 + \brs{\N^{k+1} H}^2 \right) + C \frac{K^3}{t^k},\\
\brs{\sum_{j=1}^{k+1} \N^j H^2 \star \N^{k+1-j} H \star \N^{k+1} H} \leq&\ C K \left( \brs{\N^k \Rm}^2 + \brs{\N^{k+1} H}^2 \right) + C \frac{K^3}{t^k}.
\end{split}
\end{gather}
Plugging (\ref{f:globsmooth20}) - (\ref{f:globsmooth40}) into (\ref{f:globsmooth10}) yields the differential inequality
\begin{gather} \label{f:globsmooth50}
\begin{split}
\left(\dt - \gD \right) & \left( \brs{\N^k \Rm}^2 + \brs{\N^{k+1} H}^2 \right)\\
\leq&\ - \brs{\N^{k+1} \Rm}^2 - \brs{\N^{k+2} H}^2 + C K \left( \brs{\N^{k+1} H}^2 + \brs{\N^k \Rm}^2 \right) + C \frac{K^3}{t^k}.
\end{split}
\end{gather}
Using this differential inequality, we can choose the constants $P_j$ by an elementary induction argument to arrive at the differential inequality
\begin{align*}
\left(\dt - \gD \right) F \leq&\ C K^3,
\end{align*}
for all times $t \leq \frac{\ga}{K}$.  Applying the maximum principle on the time interval $[0,\frac{\ga}{K}]$ yields the estimate
\begin{align*}
\sup_{M \times \{t\}} F \leq \sup_{M \times \{0\}} F + C K^3 t \leq C K^2(1 + \ga),
\end{align*}
which after rearranging yields the required estimate.
\end{proof}
\end{thm}

\section{Results on maximal existence time} \label{s:MET}

Having established the general short-time existence of solutions to generalized 
Ricci flow on compact manifolds, the natural follow-up question is, how far can 
this solution be extended?  As we have seen in Example \ref{e:shrinkingsphere}, 
the maximal interval of existence can certainly be finite.  In this section we 
provide a basic criterion for determining the singular time of a solution to 
generalized Ricci flow, namely showing that at a singular time for generalized 
Ricci flow, the norm of the Riemann curvature tensor must blow up.  Verifying 
global existence of a solution to generalized Ricci flow thus reduces to 
establishing a bound on the Riemann curvature tensor.  

\begin{lemma} \label{l:H4estimate} Let $M^n$ be a smooth manifold.  Given $H 
\in 
\Lambda^3 T^*$ and $g$ a Riemannian metric on $M$, one has
\begin{align*}
\brs{H^2}^2_g \geq&\ \tfrac{1}{n} \brs{H}^4_g.
\end{align*}
\begin{proof} From the definition (\ref{f:H^2def}) it is clear that $H^2$ is a 
positive 
semidefinite two tensor, and that one has
$\tr_g H^2 = \brs{H}^2$.  Using the orthogonal decomposition $H^2 = \ohat{H^2} 
+ 
\tfrac{1}{n} \tr_g H^2 g$ into tracefree and diagonal tensors, we obtain
\begin{align*}
\brs{H^2}^2_g =&\ \brs{ \ohat{H^2} + \tfrac{1}{n} \tr_g H^2 g}^2\\
=&\ \brs{\ohat{H^2}}_g^2 + \tfrac{1}{n^2} \left( \tr_g H^2 \right)^2 
\brs{g}_g^2\\
=&\ \brs{\ohat{H^2}}_g^2 + \tfrac{1}{n} \brs{H}^4_g\\
\geq&\ \tfrac{1}{n} \brs{H}_g^4,
\end{align*}
as required.
\end{proof}
\end{lemma}

\begin{prop} \label{p:curvtotorsion} Let $(E,[,],\IP{,}) \to M$ be an exact 
Courant algebroid.  Let $\GG_t$ be a solution to generalized Ricci flow on $[0,\tfrac{\ga}{K}]$ such that
\begin{align*}
 \sup_{M \times [0,\tfrac{\ga}{K}]} \brs{\Rm} \leq K.
\end{align*}
Then there exists a constant $C = C(n)$ such that for all $t \in 
[0,\tfrac{\ga}{K}]$,
\begin{align*}
 \sup_{M \times \{t\}} t \brs{H}^2 \leq C \max \{\ga,1\}.
\end{align*}
\begin{proof} Combining Lemma \ref{l:Hevolution} with the Bochner formula 
and the estimate of Lemma \ref{l:H4estimate} 
yields 
the differential inequality
\begin{align*}
\dt \brs{H}^2 =&\ 2 \IP{ \gD_d H, H} - 3 \IP{ \left( \dt g \right) , H^2}\\
=&\ 2 \IP{ \gD H + \Rm \star H, H} - 3 \IP{ \left(-2 \Rc + \tfrac{1}{2} H^2 
\right), H^2}\\
=&\ \gD \brs{H}^2 - 2 \brs{\N H}^2 + \Rm \star H^{\star 2} - \tfrac{3}{2} \brs{H^2}^2\\
\leq&\ \gD \brs{H}^2 - 2 \brs{\N H}^2 + C K \brs{H}^2 - \tfrac{3}{2 n} 
\brs{H}^4,
\end{align*}
where $C$ depends only on $n$.  As a direct consequence we obtain, using that 
$t 
\leq \tfrac{\ga}{K}$,
\begin{align*}
\left(\dt - \gD \right) t \brs{H}^2 \leq&\ \left(1 + C t K \right) \brs{H}^2 - 
\tfrac{3}{2 n} t \brs{H}^4\\
\leq&\ C \max\{\ga,1\} \brs{H}^2 - \tfrac{3}{2n} t \brs{H}^4\\
\leq&\ \brs{H}^2 \left( C \max \{\ga,1\} - \tfrac{3}{2n} t \brs{H}^2 \right).
\end{align*}
We note that at any point where $t \brs{H}^2 \geq \frac{2 C \max \{\ga,1\}}{3 
n} 
=: C_1$, the right hand side is nonpositive.  Since $\sup_{M \times \{0\}} t 
\brs{H}^2$ is obviously zero, it follows from the maximum principle that the 
estimate $\sup_{M \times \{t\}} \leq C_1$ is preserved for $t \in 
[0,\tfrac{\ga}{K}]$.
\end{proof}
\end{prop}

\begin{prop} \label{p:curvtonablatorsion} Let $(E,[,],\IP{,}) \to M$ be an exact 
Courant algebroid with $M$ compact.  Suppose $\GG_t$ is a solution to generalized Ricci
flow on $[0,\tfrac{\ga}{K}]$ such that
\begin{align*}
 \sup_{M \times [0,\tfrac{\ga}{K}]} \brs{\Rm} \leq K.
\end{align*}
Then there exists a constant $C = C(n)$ such that for all $t \in 
[0,\tfrac{\ga}{K}]$,
\begin{align*}
 \sup_{M \times \{t\}} t \brs{\N H} \leq C \max \{\ga,1\}.
\end{align*}
\begin{proof} Combining Lemma \ref{l:Hevolution} with the Bochner formula 
we obtain
\begin{align*}
\dt \N H =&\ \left( \dt \N \right) H + \N \left(\dt H \right)\\
=&\ \N \left(\Rc - \tfrac{1}{4} H^2 \right) \star H + \N \left( \gD H + \Rm \star H \right)\\
=&\ \gD \N H + \Rm  \star \N H + \left( \N \Rm + \N H \star H \right) \star H.
\end{align*}
A further calculation then yields, using Lemma \ref{l:H4estimate},
\begin{align*}
\dt \brs{\N H}^2 =&\ 2 \IP{\dt \N H, \N H}\\
&\ - \left\{ \left(-2 \Rc + \tfrac{1}{2} H^2 \right)^{ij}  g^{kl} g^{pq} g^{rs} + 3 \left(-2 \Rc + \tfrac{1}{2} H^2 \right)^{kl} g^{ij} g^{pq} g^{rs} \right\} \N_i H_{kpr} \N_j H_{lqs}\\
=&\ 2 \IP{ \gD \N H + \Rm  \star \N H + \left( \N \Rm + \N H \star H \right) \star H, \N H}\\
&\ + \Rm \star \N H^{\star 2} - \tfrac{1}{2} \left\{ (H^2)^{ij} g^{kl} g^{pq} g^{rs} + 3 (H^2)^{kl} g^{ij} g^{pq} g^{rs} \right\} \N_i H_{kpr} \N_j H_{lqs}\\
\leq&\ \gD \brs{\N H}^2 - 2 \brs{\N^2 H}^2 + C \brs{\Rm} \brs{\N H}^2 + \brs{\N \Rm} \brs{H} \brs{\N H} + C \brs{H}^2 \brs{\N H}^2
\end{align*}
where $C$ depends only on $n$.  As a direct consequence we obtain, using that 
$t 
\leq \tfrac{\ga}{K}$, and applying Proposition \ref{p:curvtotorsion},
\begin{align*}
\left(\dt - \gD \right) t \brs{\N H}^2 \leq&\ \left(1 + C t K + C t \brs{H}^2 \right) \brs{\N H}^2 + t \brs{\N \Rm}^2\\
\leq&\ C \max\{\ga,1\} \brs{\N H}^2 + t \brs{\N \Rm}^2.
\end{align*}
We note that at any point where $t \brs{H}^2 \geq \frac{2 C \max \{\ga,1\}}{3 
n} 
=: C_1$, the right hand side is nonpositive.  Since $\sup_{M \times \{0\}} t 
\brs{H}^2$ is obviously zero, it follows from the maximum principle that the 
estimate $\sup_{M \times \{t\}} t \brs{H}^2 \leq C_1$ is preserved for $t \in 
[0,\tfrac{\ga}{K}]$.
\end{proof}
\end{prop}

\begin{lemma} \label{l:smoothfamilies} Let $M^n$ be a smooth manifold, and suppose $(g_t, H_t)$ is a smooth one-parameter family of pairs on $[0,T)$, $T < \infty$, such that for any $m_1,m_2 \in \mathbb N$, 
\begin{align*}
\sup_{M \times [0,T)} \left\{ \brs{\N^{m_1} \frac{\del^{m_2}}{\del t^{m_2}} g} + \brs{\N^{m_1} \frac{\del^{m_2}}{\del t^{m_2}} H} \right\} < \infty.
\end{align*}
Then there exists a smooth pair $(g_T, H_T)$ such that
\begin{align*}
(g_t, H_t) \rightarrow (g_T, H_T)
\end{align*}
in the $C^{\infty}$ topology.
\begin{proof} First we establish the existence of the limiting metric $g_T$.  Fix a vector $V \in T_p M$, $0 \leq t_1 \leq t_2 < T$ and observe
\begin{align*}
\brs{\log \frac{g(p,t_2)(V,V)}{g(x,t_2)(V,V)}} =&\ \brs{ \int_{t_1}^{t_2}\frac{\del g}{\del t}_{(p,t)} \left( \frac{V}{\brs{V}}, \frac{V}{\brs{V}} \right) dt}\\
\leq&\ \int_{t_1}^{T} \brs{ \frac{\del g}{\del t}_{(p,t)}}_{g(t)} dt\\
\leq&\ C(T - t_1).
\end{align*}
This shows that $\lim_{t \to T} g_t(V,V)$ exists, and by polarization we obtain $g_T = \lim_{t \to T} g_t$.  

With this uniform equivalence of $g$ established we can obtain the limit $H_T$.  In particular we fix $V,X,Y \in T_p M$, $0 \leq t_1 \leq t_2 < t$ and estimate
\begin{align*}
\brs{H_{(p,t_2)}(V,X,Y) - H_{(p,t_1)}(V,X,Y)} =&\ \brs{ \int_{t_1}^{t_2} \frac{\del H}{\del t}_{(p,t)} (V,X,Y) dt}\\
\leq&\ \int_{t_1}^{t_2} \brs{\frac{\del H}{\del t}_{(p,t)}}_{g(t)} \brs{V}_{g(t)} \brs{X}_{g(t)} \brs{Y}_{g(t)} dt\\
\leq&\ C(T - t_1).
\end{align*}
This shows $\lim_{t \to T} H(V,X,Y)$ exists, defining a continuous $3$-form $H_T$.

Showing strong convergence to and regularity of $(g_T, H_T)$ requires inductively proving stronger estimates on the derivatives of $g$ and $H$ with respect to some fixed background connection.  These technical details are identical to that for Ricci flow, and we refer the reader to (\cite{ChowKnopf} Proposition 6.48) for the proof.
\end{proof}
\end{lemma}

\begin{thm} \label{t:Rmblowup} Let $(E,[,],\IP{,}) \to M$ be an exact 
Courant algebroid.  Given $\GG_0$ a generalized metric, let $T \in \mathbb 
R_{>0} \cup \{\infty\}$ denote the maximal extended real number such that the 
solution to generalized Ricci flow with initial condition $\GG_0$ exists smoothly 
on $[0,T)$.  If $T < \infty$, then
\begin{align*}
\limsup_{t \to T} \sup_{M \times \{t\}} \brs{\Rm} = \infty.
\end{align*}
\begin{proof} We argue by contradiction, and suppose that
\begin{align} \label{f:LTE10}
\limsup_{t \to T} \sup_{M \times \{t\}} \brs{\Rm} = K < \infty.
\end{align}
Fix some $0 < \gd < T$.  As the flow is smooth on $[0,\gd]$, for all $m\in \mathbb N$ there exists a constant $C_m$ such that
\begin{align*}
\sup_{M \times [0,\gd]} \left( \brs{\N^m \Rm} + \brs{\N^{m+1} H} + \brs{H} \brs{\N^m H} \right) \leq C_m.
\end{align*}
Given $t \in [\gd, T]$, by applying Theorem \ref{t:smoothing} on $[t-\gd,t]$ one obtains for a different constant $C_m'$,
\begin{align*}
\sup_{M \times \{t\}} \left( \brs{\N^m \Rm} + \brs{\N^{m+1} H} + \brs{H} \brs{\N^m H} \right) \leq C_m' K \gd^{-m/2}.
\end{align*}
Thus there is a uniform bound on the curvature, torsion, and all their covariant derivatives on $[0,T)$.  With these uniform estimates in place, we can apply Lemma \ref{l:smoothfamilies} to obtain the existence of a pair $(g_T, H_T)$ such that
\begin{align*}
\lim_{t \to T} (g_t, H_t) = (g_T, H_T).
\end{align*}
As $M$ is compact, we may apply Theorem \ref{t:STE} to claim that there exists $\ge > 0$ and a smooth solution to generalized Ricci flow with initial condition $(g_T, H_T)$ on $[0,\ge)$.  Concatenating this solution with the original solution yields a smooth solution to generalized Ricci flow on $[0,T+\ge)$, contradicting the maximality of $T$.  Thus (\ref{f:LTE10}) cannot hold, finishing the proof.
\end{proof}
\end{thm}

\section{Compactness results for generalized metrics} \label{s:compactness}

In the study of generalized Ricci flow, as in any partial differential 
equation, it is very useful to have a compactness principle for taking limits 
of sequences of solutions.  In the context of Riemannian geometry, the relevant 
limits are taken with respect to the Cheeger-Gromov topology (cf. \cite{CheegerThesis}, \cite{Petersen} Chapter 10).  
Hamilton later extended this definition to taking limits of one-parameter 
families of metrics satisfying certain geometric bounds 
\cite{HamiltonCompactness}, with the goal of applying it to solutions of Ricci 
flow.  In this section we will briefly recall these notions, and sketch a proof 
of an elementary extension of Hamilton's result to one-parameter families of generalized 
metrics, which will be instrumental in taking limits of solutions of 
generalized Ricci flow in various contexts.  We begin with a fundamental definition:

\begin{defn} \label{d:CGlimit} We say that a sequence $\{(M_k, g_k, O_k)\}$ of complete pointed Riemannian manifolds \emph{converges in the 
$C^{\infty}$ Cheeger-Gromov topology} to a limit $(M_{\infty}, g_{\infty}, O_{\infty})$ if there exists 
an exhaustion of $M_{\infty}$ by open sets $U_k$ containing $O_{\infty}$ and a sequence of 
diffeomorphisms $\phi_k : U_k \to V_k \subset M_k$ such that $\phi_k(O_{\infty}) = O_k$, 
and $\phi_k^* g_k \to g_{\infty}$ uniformly on compact subsets in 
the $C^{\infty}$ topology.
\end{defn}

This definition highlights a fundamental point, which is that if we want to 
take limits of metrics satisfying natural geometric conditions, i.e. conditions 
invariant under the diffeomorphism group, then we must take the diffeomorphism 
invariance explicitly into account and choose an appropriate diffeomorphism 
gauge before the limits of the metric tensors can be taken in a classic 
sense.  In view of our discussion in Proposition \ref{p:Courantsymmetries} of the enlarged gauge group 
relevant to generalized geometry, the following definition is inevitable.

\begin{defn} \label{d:GCGlimit} We say that a sequence $\{(E_k, M_k, \GG_k, O_k)\}$ of complete pointed generalized Riemannian manifolds \emph{converges in 
the $C^{\infty}$ generalized Cheeger-Gromov topology} to a limit $(E_{\infty}, M_{\infty}, \GG_{\infty}, O_{\infty})$ 
if there exists an exhaustion of $M_{\infty}$ by open sets $U_k$ containing $O_{\infty}$ and a 
sequence of diffeomorphisms $\phi_k : U_k \to V_k \subset M_k$ covered by Courant algebroid automorphisms $F_k$ such that 
\begin{enumerate}
\item $\phi_k(O_{\infty}) = O_k$,
\item $F_k^{-1} \GG_k F_k \to \GG_{\infty}$ uniformly on compact subsets in the $C^{\infty}$ topology.
\end{enumerate}
We observe that if $g_k$ and $H_k$ denote the metric and three-forms canonically associated to $(E_k, \GG_k)$ by Proposition \ref{l:Gmetricabsgeom}, with $g_{\infty}, H_{\infty}$ similarly defined, then
\begin{enumerate}
\item $\phi_k^* g_k \to g_{\infty}$ uniformly on compact subsets in the $C^{\infty}$ topology,
\item $\phi_k^* H_k \to H_{\infty}$ uniformly on compact subsets in the $C^{\infty}$ topology.
\end{enumerate}
\end{defn}

\begin{rmk} A basic fact of Cheeger-Gromov convergence is invariance under the diffeomorphism group.  It follows from Propositions \ref{p:bfieldbracket} and \ref{l:Gmetricabsgeom} that the limiting process is invariant under the action of the enlarged symmetry group, in particular under the action of closed $B$-fields.  In particular, if $\{(E_k, M_k, \GG_k, O_k)\}$ converges to $(E_{\infty}, M_{\infty}, \GG_{\infty}, O_{\infty})$, and $F_k$ denotes a Courant automorphism of $E_k$, then the sequence $\{(E_k, M_k, F_k^{-1} \GG_k F_k, O_k)\}$ also converges to $(E_{\infty}, M_{\infty}, \GG_{\infty}, O_{\infty})$.
\end{rmk}

Definitions \ref{d:CGlimit} and \ref{d:GCGlimit} give a notion of sequential 
convergence, which by a standard method can be bootstrapped to define an 
actual topology on the space of generalized Riemannian manifolds.  We point to \cite{RubioTipler} for recent further progress on understanding the moduli space of generalized metrics.  However, 
the specifics of this topology are not as directly useful here as the notion of 
convergence and precompactness in this topology, which arises via the question of when a sequence of generalized 
metrics admits a convergent subsequence.  In the context of Riemannian metrics, 
the fundamental question was answered by Cheeger \cite{CheegerThesis}.  We state here a 
simplified version which suffices for our purposes, proved by Hamilton \cite{HamiltonCompactness}.

\begin{thm} \label{t:CGcompactness} (\cite{HamiltonCompactness} Theorem 2.3) Given a sequence $\{(M_k, g_k, O_k)\}$ 
of complete pointed Riemannian manifolds such that for all $j \in \mathbb N$ 
there exists $C_j < \infty$ such that
\begin{align*}
\sup_{M_k} \brs{\N^j_{g_k} \Rm_{g_k}}_{g_k} \leq C_j,
\end{align*}
and such that there exists $\gd > 0$ so that
\begin{align*}
\inj_{g_k} (O_k) \geq \gd,
\end{align*}
then there exists a convergent subsequence.
\end{thm}

Building upon this fundamental theorem, we prove a related theorem on 
compactness for sequences of generalized metrics satisfying certain natural 
bounds.

\begin{cor} \label{t:GCGcompactness} Given a sequence $\{(E_k, M_k, \GG_k, O_k\}$ of complete pointed generalized Riemannian manifolds such that for all 
$j \in \mathbb N$ there exists $C_j < \infty$ such that
\begin{align*}
\sup_{M_k} \left\{ \brs{\N^j_{g_k} \Rm_{g_k}}_{g_k} + \brs{\N^j H_k}_{g_k} 
\right\}\leq C_j,
\end{align*}
and such that there exists $\gd > 0$ so that
\begin{align*}
\inj_{g_k} (O_k) \geq \gd,
\end{align*}
then there exists a convergent subsequence.
\begin{proof} As the hypotheses of Theorem \ref{t:CGcompactness} are satisfied for the sequence $\{g_k\}$, we can choose a subsequence such that the underlying Riemannian manifolds converge in the $C^{\infty}$ Cheeger-Gromov topology to a limiting Riemannian manifold $(M_{\infty}, g_{\infty})$.  Using the derivative estimates for $H$, it follows from the Arzela-Ascoli theorem that the sequence of pullbacks $\{\phi_k^* H_k\}$ converges uniformly on compact subsets of $M_{\infty}$ to a limiting three-form $H_{\infty}$.  As discussed in \S \ref{s:Courant}, this tensor $H_{\infty}$ defines a twisted Dorfman bracket on $M_{\infty}$, which determines the exact Courant algebroid $E_{\infty}$.  If we define $\GG_{\infty}$ on $E_{\infty}$ using $g_{\infty}$ and the canonical splitting of $T M_{\infty} \oplus T^* M_{\infty}$, then $\{(E_k, \GG_k, O_k)\}$ converges to $(E_{\infty}, \GG_{\infty}, O_{\infty})$.
\end{proof}
\end{cor}

\begin{rmk} In Corollary \ref{t:GCGcompactness}, in addition to the topology of the limiting space changing, the topology of the Courant bracket structure can also change.  In particular, fix $M$ any smooth manifold which admits $H$ a closed three-form such that $[H] \neq 0 \in H^3(M, \mathbb R)$.  We can define generalized metrics $\GG$ associated to $g$ and the canonical splitting on the twisted Courant algebroid defined by $\gl H$ for any $\gl \in \mathbb R$.  Sending $\gl \to 0$ gives a sequence of generalized metrics with nontrivial \v Severa classes converging to a generalized metric on a Courant algebroid with trivial \v Severa class.
\end{rmk}

What is most relevant for studying generalized Ricci flow is to take limits not 
just of generalized metrics, but of one-parameter families of generalized 
metrics.  The relevant notion from Riemannian geometry was introduced by 
Hamilton \cite{HamiltonCompactness}.  Given the clear relationship between 
Definitions \ref{d:CGlimit} and \ref{d:GCGlimit}, in the interest of space we 
only state the definition in the context of generalized metrics.

\begin{defn} \label{d:HCGlimit} 
We say that a sequence $\{(E_k, M_k, \GG^t_k, O_k)\}$ of one-parameter families of complete pointed generalized Riemannian manifolds \emph{converges in 
the $C^{\infty}$ generalized Cheeger-Gromov topology} to a limiting one-parameter family $(E_{\infty}, M_{\infty}, \GG^t_{\infty}, O_{\infty})$, defined for $t \in (\ga, \gw)$,
if there exists an exhaustion of $M_{\infty}$ by open sets $U_k$ containing $O_{\infty}$ and a 
sequence of diffeomorphisms $\phi_k : U_k \to V_k \subset M_k$ covered by Courant automorphisms $F_k$ such that 
\begin{enumerate}
\item $\phi_k(O_{\infty}) = O_k$,
\item $F_k^{-1} (\GG_k)^t F_k \to \GG^t_{\infty}$ uniformly on compact subsets of $M_{\infty} \times (\ga, \gw)$ in the $C^{\infty}$ topology.
\end{enumerate}
As before, if $g_k^t$ and $H_k^t$ denote the families of metrics and three-forms canonically associated to $(E_k, \GG^t_k)$ by Proposition \ref{l:Gmetricabsgeom}, with $g^t_{\infty}, H^t_{\infty}$ similarly defined, then
\begin{enumerate}
\item $\phi_k^* g_k^t \to g^t_{\infty}$ uniformly on compact subsets of $M_{\infty} \times (\ga, \gw)$ in the $C^{\infty}$ topology,
\item $\phi_k^* H^t_k \to H^t_{\infty}$ uniformly on compact subsets of $M_{\infty} \times (\ga, \gw)$ in the $C^{\infty}$ topology.
\end{enumerate}
\end{defn}

\begin{thm} \label{t:flowcompactness} Let $\{(E_k, M_k, \GG^t_k)\}$ be a sequence of complete solutions to generalized Ricci flow defined on time intervals $I_k = (\ga_k, \gw_k) \subset \mathbb R$.  Suppose $\ga_k \to \ga, \gw_k \to \gw$, and for all 
$j \in \mathbb N$ there exists $C_j < \infty$ such that
\begin{align*}
\sup_{M_k \times I_k} \left\{ \brs{\N^j_{g_k} \Rm_{g_k}}_{g_k} + \brs{\N^j H_k}_{g_k} 
\right\}\leq C_j.
\end{align*}
Furthermore, suppose there exists $O_k \in M_k$, $t_0 \in (\ga, \gw)$, and $\gd > 0$ so that
\begin{align*}
\inj_{g_k(t_0)} (O_k) \geq \gd,
\end{align*}
for all sufficiently large $k$.  Then there exists a subsequence of $\{(E_k, M_k, \GG^t_k, O_k)\}$ converging to a pointed solution of generalized Ricci flow defined on the time interval $(\ga,\gb)$.
\begin{proof} We give a brief sketch of the proof, leaving technical arguments as lengthy \textbf{exercises}.  First, the sequence of generalized Riemannian manifolds $\{(E_k, M_k, \GG^{t_0}_k, O_k)\}$ satisfies the hypotheses of Corollary \ref{t:GCGcompactness}, and thus there is a subsequence converging to a limiting generalized Riemannian manifold $(E_{\infty}, M_{\infty}, \GG_{\infty}, O_{\infty})$.  Having established convergence of the given time slices, one must use the generalized Ricci flow equation together with the uniform derivative estimates to show convergence for all times.  The key point is to obtain uniform estimates for $g$ and $H$ on compact subsets of spacetime.  These estimates follow as in the proof of Lemma \ref{l:smoothfamilies}.  The uniform equivalence for the metrics is proved in Lemma \ref{l:smoothfamilies}, whereas for the higher regularity estimates we refer to (\cite{ChowKnopf} Proposition 6.48).
\end{proof}
\end{thm}

\chapter{Energy and Entropy Functionals} \label{energychapter}

In the previous two chapters we motivated the generalized Ricci flow by showing that 
it is a well-posed evolution equation with generalized Einstein metrics as 
fixed points, and establishing fundamental analytic structure
along the flow stemming from its parabolic structure.  For deeper applications 
it is natural to seek monotone quantities along the flow which provide further 
control of the metric, as well as convergence results for smooth long-time 
solutions.  For the Ricci flow equation the discovery of such quantities was 
among the fundamental breakthroughs of Perelman \cite{Perelman1}.  By the 
explicit incorporation of extra terms involving $H$, it is possible to modify 
these constructions to yield energy and entropy functionals for the generalized 
Ricci flow.

\section{Generalized Ricci flow as a gradient flow} \label{s:GRFasgradient}
\subsection{Energy Functional} \label{ss:energy}

As was discovered in \cite{OSW} following 
Perelman's construction, the generalized Ricci flow is the gradient 
flow for the first eigenvalue of a natural Schr\"odinger operator, directly generalizing Perelman's energy 
monotonicity for Ricci flow.  The energy functional below already appeared in Chapter \ref{c:GCC} as the generalized Einstein-Hilbert functional (cf. Definition \ref{d:genEHdiv}).

\begin{defn}  \label{d:Penergy} Let $M^n$ be a smooth manifold.  Given $g$ a Riemannian metric, $H 
\in \Lambda^3 T^*, 
dH = 0$, and $f \in C^{\infty}$, let
\begin{align*}
\FF(g,H,f) = \int_M \left( R - \tfrac{1}{12} \brs{H}^2 + \brs{\N f}^2 \right) 
e^{-f} dV.
\end{align*}
Furthermore, let
\begin{align*}
\gl(g,H) = \inf_{\{f\ |\ \int_M e^{-f} dV = 1 \}} \FF(g,H,f).
\end{align*}
\end{defn}

\begin{rmk} Note that we have chosen to define the functional $\FF$ using the 
three-form $H$ directly, as opposed to freezing a background $H_0$ and defining 
the functional using $b \in \Lambda^2 T^*$, with the convention that $H = H_0 + 
db$, as is done for instance in \cite{OSW}.
\end{rmk}

\begin{lemma} \label{l:energyminimizerexists} Given $(M^n, g, H)$ a smooth 
compact Riemannian manifold with $H \in 
\Lambda^3 T^*, dH = 0$, the infimum defining $\gl(g,H)$ is uniquely achieved, 
and is the lowest eigenvalue of the Schr\"odinger operator
\begin{equation}\label{eq:Schrodinger}
-4 \gD + R - \tfrac{1}{12} \brs{H}^2.
\end{equation}
\begin{proof} We note that we may reexpress the functional $\FF$ in terms of $u 
= e^{-\tfrac{f}{2}}$ as
\begin{align*}
\FF(g,H,u) = \int_M \left[ \left( R - \tfrac{1}{12} \brs{H}^2 \right) u^2 + 4 
\brs{\N u}^2 \right] dV.
\end{align*}
If we minimize this functional over $u$ satisfying $\int_M u^2 dV = 1$, a direct 
variational computation with Langrange multipliers shows that $u$ must satisfy
\begin{align*}
\left(-4 \gD + R - \tfrac{1}{12} \brs{H}^2 \right) u = \gl u.
\end{align*}
This 
minimizing 
function exists and is unique by (\cite{ReedSimon} XIII.12).
\end{proof}
\end{lemma}

\begin{rmk}
The Schr\"odinger operator \eqref{eq:Schrodinger}, introduced in \cite{OSW}, has been recently interpreted in \cite{SeveraValach2} as an operator on half-densities to provide an alternative definition of the entropy functional $\FF$ in Definition \ref{d:Penergy}.
\end{rmk}

\subsection{Conjugate heat equation} \label{ss:CHE}

Note that the definition of $\FF$ requires the extra parameter $f$, whereas the 
infimum in the definition of $\gl$ removes this parameter.  Obtaining evolution 
equations for $\gl$ is challenging to achieve directly due to this infimum, 
since the function $f$ achieving this infimum can change in a discontinuous way 
as $g$ and $H$ change.  For this reason it is more effective to work with $\FF$ 
directly.  It is natural then to couple an evolution equation for the parameter 
$f$ to the generalized Ricci flow, the most natural choice being the conjugate 
heat equation, which will preserve the unit volume condition.

\begin{defn} \label{d:CHE} Let $(M^n, g_t, H_t)$ be a solution to generalized Ricci flow.  A 
one-parameter family $u_t \in C^{\infty}(M)$ is a solution to the 
\emph{conjugate heat equation} if
\begin{align} \label{f:conjugateheat}
\left(\dt + \gD_{g_t} \right) u = R u - \tfrac{1}{4} \brs{H}^2 u.
\end{align}
We summarize this equation in terms of the \emph{conjugate heat operator} 
$\square^*$, where
\begin{align} \label{f:boxstar}
\square^* := - \dt - \gD_{g_t} + R - \tfrac{1}{4} \brs{H}^2.
\end{align}
In particular, a function $u$ solves the conjugate heat equation if and only if 
$\square^* u = 0$.  The operator $\square^*$ is the $L^2$ adjoint of the time-dependent heat operator (cf. Lemma \ref{l:conjugateheat} below)
\begin{align*}
\square := \dt - \gD_{g_t}.
\end{align*}
It is convenient to phrase many results below in terms of the generalized Ricci flow in a general gauge (cf. Remark \ref{r:GFGRF}).  Given a solution to $(X_t, k_t)$-gauge-fixed generalized Ricci flow, the relevant conjugate heat equation becomes
\begin{align} \label{f:gffconjugateheat}
\left(\dt + \gD_{g_t} - X \right) u = R u - \tfrac{1}{4} \brs{H}^2 u.
\end{align}
\end{defn}

Observe that the sign on the left hand side of (\ref{f:conjugateheat}) indicates that, as a heat equation, 
it is evolving in \emph{reverse} time of the generalized Ricci flow.  This is 
natural from the point of view of understanding monotonicity for $\FF$, as one 
wants to construct a test function at some forward time, and pass it in a 
natural way to an earlier time, where one has some geometric understanding. 
 The next lemma shows that the conjugate heat equation preserves the evolving 
$f$-weighted volume, and is the $L^2$ adjoint of the heat equation with respect 
to the evolving metric.

\begin{lemma} \label{l:conjugateheat} Let $(M^n, g_t, H_t)$ be a solution to 
generalized Ricci flow.  
\begin{enumerate}
\item If $\phi_t$ is a solution of $\square^* \phi = 0$, then $\tfrac{d}{dt} 
\int_M \phi dV_{g} = 0$.
\item Given $\phi_t, \psi_t \in C^{\infty}(M)$ such that $\phi_0 = \phi_T = 0$, 
one has
\begin{align}
\int_0^T \int_M \left(\square \phi \right) \psi dV = \int_0^T \int_M \phi 
\left(\square^* \psi \right) dV.
\end{align}
\end{enumerate}
\begin{proof} For statement (1), differentiate directly using 
(\ref{f:conjugateheat}) and Lemma \ref{l:volumevariation} to see
\begin{align*}
\tfrac{d}{dt} \int_M \phi dV =&\ \int_M \left[ \dt \phi + \phi \left( - R + 
\tfrac{1}{4} \brs{H}^2 \right) \right] dV\\
=&\ \int_M (- \gD \phi) dV\\
=&\ 0.
\end{align*}
Similarly we compute for general $\phi, \psi$ using the self-adjointness of the 
Laplacian,
\begin{align*}
\frac{d}{dt} \int_M \phi \psi dV =&\ \int_M \left[ \left( \dt \phi \right) 
\psi + \phi \left( \dt \psi \right) + \phi \psi \left( - R + \tfrac{1}{4} 
\brs{H}^2 \right) \right] dV\\
=&\ \int_M \left\{ \left[ \left( \dt - \gD \right) \phi \right] \psi - \left[ 
\phi \left( - \dt - \gD + R - \tfrac{1}{4} \brs{H}^2 \right) \psi \right] 
\right\} dV\\
=&\ \int_M \left( \square \phi \right) \psi dV - \int_M \phi \left( \square^* 
\psi \right) dV.
\end{align*}
Integrating this result over $[0,T]$ and applying the boundary condition yields 
the result.
\end{proof}
\end{lemma}

\subsection{Monotonicity} \label{ss:steadymonotonicity}

We are now in a position to establish monotonicity of the $\FF$-functional and also $\gl$.  We first require a general variational formula for $\FF$.

\begin{lemma} \label{l:energyvar} Given $(M^n, g, H)$ a smooth manifold with $H 
\in \Lambda^3 T^*, 
dH = 0$, suppose $g_s, H_s, f_s$ are one-parameter families such that
\begin{align*}
\left. \frac{\del}{\del s} \right|_{s = 0} g =&\ h, \qquad g_0 = g, \qquad 
\left. \frac{\del}{\del s} \right|_{s = 0} H = dK, \qquad H_0 = H\\
\left. \frac{\del}{\del s} \right|_{s = 0} f =&\ \phi, \qquad f_0 = f.
\end{align*}
Then
\begin{gather} \label{Ffirstvar}
\begin{split}
\left. \frac{d}{ds} \right|_{s=0} \FF(g_s,H_s,f_s) =&\ \int_M \left[ \IP{ - \Rc 
+ \tfrac{1}{4} H^2 - \N^2 f, h} + \IP{ \tfrac{1}{2} \left( - d^* H - \N f 
\lrcorner H \right), K} \right.\\
&\ \left. + \left( R - \tfrac{1}{12} \brs{H}^2 + 2 \gD f - \brs{\N f}^2 \right) 
\left( \tfrac{1}{2} \tr_g h - \phi \right) \right] e^{-f} dV.
\end{split}
\end{gather}
\begin{proof} Combining Lemmas \ref{l:volumevariation}, \ref{l:Henergyvar} 
and \ref{l:curvaturevariation} we obtain
\begin{gather} \label{f:energyvar10}
\begin{split}
\left. \frac{d}{ds} \right|_{s=0} \FF(g_s,H_s,f_s) =&\ \int_M \left( - \gD 
\tr_g 
h + \divg \divg h - \IP{h,\Rc} \right.\\
&\ - \tfrac{1}{6} \IP{d K, H} + \tfrac{1}{4} \IP{H^2, h} + 2 \IP{\N \phi, \N f} 
- \IP{h, \N f \otimes \N f}\\
&\ \left. + \left(R - \tfrac{1}{12} \brs{H}^2 + \brs{\N f}^2 \right) \left( 
\tfrac{1}{2} \tr_g h - \phi \right) \right) e^{-f} dV.
\end{split}
\end{gather}
Integrating by parts yields the formulas
\begin{gather} \label{f:energyvar20}
\begin{split}
\int_M \left(- \gD \tr_g h + \divg \divg h \right) e^{-f} dV =&\ \int_M \left( 
\tr_g h (- \gD e^{-f}) + \IP{h, \N^2 e^{-f}} \right) dV\\
=&\ \int_M \left( \tr_g h (\gD f - \brs{\N f}^2) + \IP{h, \N f \otimes \N f - 
\N^2 f} \right) e^{-f} dV.
\end{split}
\end{gather}
Also
\begin{gather} \label{f:energyvar30}
\begin{split}
- \tfrac{1}{6} \int_M \IP{d K, H} e^{-f} dV =&\ - \tfrac{1}{2} \int_M \N_i 
K_{jk} H_{ijk} e^{-f} dV\\
=&\ \tfrac{1}{2} \int_M K_{jk} \N_i \left( H_{ijk} e^{-f} \right) dV\\
=&\ \tfrac{1}{2} \int_M K_{jk} \left( - d^* H_{jk} - (\N f \lrcorner H)_{jk} 
\right) e^{-f} dV\\
=&\ \int_M \IP{K, \tfrac{1}{2} (- d^* H - \N f \lrcorner H)} e^{-f} dV.
\end{split}
\end{gather}
Lastly we have
\begin{gather} \label{f:energyvar40}
\begin{split}
\int_M 2 \IP{\N \phi, \N f} e^{-f} dV = \int_M \phi \left( - 2 \gD f + 2 
\brs{\N 
f}^2 \right) e^{-f} dV.
\end{split}
\end{gather}
Inserting (\ref{f:energyvar20})-(\ref{f:energyvar40}) into 
(\ref{f:energyvar10}) 
and simplifying yields the result.
\end{proof}
\end{lemma}

Finally we are ready to establish the monotonicity of $\FF$.  The proof exploits the diffeomorphism invariance of $\FF$ in an essential way, and we refer back to Remark \ref{r:GFGRF} for the notion of the general $(X_t, k_t)$-gauge-fixed generalized Ricci flow.

\begin{prop} \label{p:genFFmonotonicity} Let $(M^n, g_t, H_t)$ be a solution to 
the $(X_t, k_t)$-gauge-fixed generalized Ricci flow equations.  Let $u_t$ denote a 
solution to the 
$(X_t, k_t)$-gauge-fixed conjugate heat equation, and let $f_t = - \log u_t$.  Then
\begin{align*}
\frac{d}{dt} \FF(g_t, H_t, f_t) =&\ \int_M \left[ 2 \brs{\Rc - \tfrac{1}{4} H^2 
+ \N^2 f}^2 + \tfrac{1}{2} \brs{d^* H + \N f \lrcorner H}^2 \right] e^{-f} dV.
\end{align*}
\begin{proof} Since the 
functional $\FF$ is gauge invariant, it suffices to prove the formula for a specific choice of gauge.  In particular, given $u_t$ and $f_t = - \log u_t$ as in the statement,
we can pull back $g, H$ and $f$ by time-dependent diffeomorphisms to produce a solution of the $(- \N f, 0)$-gauge-fixed generalized Ricci flow.  Using equations (\ref{f:gaugefixedgrf}) and (\ref{f:conjugateheat}) we see that this system of equations takes the form
\begin{gather}\label{eq:gradfixflow}
\begin{split}
\dt g =&\ -2 \left( \Rc - \tfrac{1}{2} H^2 + \N^2 f \right),\\
\dt H =&\ - d \left( d^*_g H + \N f \lrcorner H \right),\\
\dt f =&\ - \gD f - R + \tfrac{1}{4} \brs{H}^2.
\end{split}
\end{gather}
Observing that, for this variation, $\tfrac{1}{2} \tr_g \dt g - \dt f = 0$, Lemma \ref{l:energyvar} 
implies the result.
\end{proof}
\end{prop}

\begin{rmk}
As a curiosity, observe that if $g_t, H_t, f_t$ is a solution of \eqref{eq:gradfixflow}, then the volume form $\mu_t = e^{-f_t}dV_{g_t}$ is constant along the flow (cf. \S \ref{s:EHF}). Instead, one could consider that $\mu_t$ evolves by the generalized scalar curvature in Definition \ref{def:scalar} and then by Lemma \ref{l:energyvar} we still have monotonicity of $\FF$ along the flow. 
\end{rmk}

\begin{cor} \label{c:genllmonotonicity} Suppose $M$ is compact and let $(M, g_t, H_t)$ be a solution to 
the $(X_t, k_t)$-gauge-fixed generalized Ricci flow equations.  Then for all $t_1 < 
t_2$, one has
\begin{align*}
\gl(g_{t_1},H_{t_1}) \leq \gl(g_{t_2}, H_{t_2}).
\end{align*}
The case of equality holds if and only if $(g_{t_1},H_{t_1})$ is a steady generalized Ricci soliton with $f$ 
realizing the infimum in the definition of $\gl$.  Moreover, the flow on 
$[t_1,t_2]$ evolves by diffeomorphism pullback by the family generated by $\N 
f$.
\begin{proof} Choose $f_{t_2}$  such that $\int_M e^{-f_{t_2}} dV_{g_{t_2}} = 
1$, and $\FF(g_{t_2},H_{t_2},f_{t_2}) = \gl(g_{t_2},H_{t_2})$, which exists by 
Lemma \ref{l:energyminimizerexists}.  As $M$ is compact, there exists a unique solution to the conjugate heat equation on 
$[t_1,t_2]$ with boundary condition $u_{t_2} = -e^{f_{t_2}}$.  Using 
Proposition \ref{p:genFFmonotonicity} we obtain
\begin{align*}
\gl(g_{t_1},H_{t_1}) \leq&\ \FF(g_{t_1},H_{t_1},f_{t_1})\\
=&\ \FF(g_{t_2},H_{t_2},f_{t_2})\\
&\ \qquad  - \int_{t_1}^{t_2} \int_M \left[ 2 \brs{\Rc - 
\tfrac{1}{4} H^2 
+ \N^2 f}^2 + \tfrac{1}{2} \brs{d^* H + \N f \lrcorner H}^2 \right] e^{-f} dV\\
=&\ \gl(g_{t_2},H_{t_2})\\
&\ \qquad - \int_{t_1}^{t_2} \int_M \left[ 2 \brs{\Rc - 
\tfrac{1}{4} H^2 
+ \N^2 f}^2 + \tfrac{1}{2} \brs{d^* H + \N f \lrcorner H}^2 \right] e^{-f} dV.
\end{align*}
The final term appearing above is manifestly nonpositive, so the claimed 
inequality follows, with equality if and only if this term vanishes, which 
implies the soliton equations hold.
\end{proof}
\end{cor}

\begin{cor} \label{c:steadysolitongradient} Given $(M^n, g, H)$ a compact steady generalized Ricci soliton with vector field $X$ and two-form $k$, then $k = 0$ and there exists $f \in C^{\infty}(M)$ such that $X = \N f$, i.e it is a gradient soliton.
\begin{proof} As a steady soliton, the solution to generalized Ricci flow with initial condition $(g, H)$ evolves by pullback by a one-parameter family of diffeomorphisms. Thus $\gl(g_t, H_t)$ is constant, and the result
follows from Corollary \ref{c:genllmonotonicity}.
\end{proof}
\end{cor}

\subsection{Steady Harnack} \label{ss:steadyHarnack}

The classical Harnack estimate for the heat equation yields a definite improvement in the oscillation of the solution over certain domains in spacetime, and is a crucial point in establishing the regularity of this equation.  In the setting of Ricci flow Harnack-type estimates were discovered by 
Perelman as pointwise versions of his energy and entropy monotonicity formulas, and again play a crucial role in understanding the singularity formation of Ricci flow.    We extend these here to the setting of generalized Ricci flow.

\begin{thm} \label{t:steadyharnack} Let $(M^n, g_t, H_t)$ be a solution to
generalized Ricci flow and suppose $u_t$ satisfies the conjugate heat equation. 
Let $f = - \log u$ and let $v = \mathcal{S} u$ where (cf. \eqref{eq:scalarexplicitstring})
\begin{align*}
\mathcal{S} = 2 \gD f - \brs{\N f}^2 + R - \frac{1}{12} \brs{H}^2.
\end{align*}
Then
\begin{align*}
\square^* v =&\ - \left[ 2 \brs{\Rc -
\frac{1}{4} H^2 + \N^2 f}^2 + \frac{1}{2} \brs{d^* H - \N f \lrcorner H}^2 
\right] u.
\end{align*}
\begin{proof} Observe that $f$ satisfies the equation
\begin{align} \label{f:steadyHarnack5}
\dt f =&\ - \gD f + \brs{\N f}^2 - R + \frac{1}{4} \brs{H}^2.
\end{align}
We compute
 \begin{align*}
  \left( \dt + \gD \right) 2 \gD f =&\ 2 \left[ \left<2 \Rc - \frac{1}{2} H^2,
\N^2 f \right> + \gD \left( - \gD f + \brs{\N f}^2 - R + \frac{1}{4} \brs{H}^2
\right) \right. \\
&\ \left. \qquad - \left< \frac{1}{2} \divg H^2 - \frac{1}{4} \N \brs{H}^2, \N f
\right> + \gD \gD f \right]\\
=&\ \left< 4 \Rc - H^2, \N^2 f \right> + 2 \gD \brs{\N f}^2 - 2 \gD R +
\frac{1}{2} \gD \brs{H}^2 + \left< \frac{1}{2} \N \brs{H}^2 - \divg H^2, \N f
\right>.
 \end{align*}
Next
\begin{align*}
 \left( \dt + \gD \right) \left(- \brs{\N f}^2 \right) =&\ - \left[ 2 \left< \N
(- \gD f + \brs{\N f}^2 - R + \frac{1}{4} \brs{H}^2), \N f \right> \right.\\
&\ \left. \qquad + \left< 2
\Rc - \frac{1}{2} H^2, \N f \otimes \N f \right> + \gD \brs{\N f}^2 \right]\\
=&\ 2 \left< \N\left(\gD f - \brs{\N f}^2 + R - \frac{1}{4} \brs{H}^2 \right),
\N f \right>\\
&\ \qquad + \left< \frac{1}{2} H^2 - 2 \Rc, \N f \otimes \N f \right> - \gD
\brs{\N f}^2.
\end{align*}
Then we have
\begin{align*}
 \left( \dt + \gD \right) R =&\ 2 \gD R + 2 \brs{\Rc}^2 - \frac{1}{2} \gD
\brs{H}^2 + \frac{1}{2} \divg \divg H^2 - \frac{1}{2} \left< \Rc, H^2 \right>.
\end{align*}
Next
\begin{align*}
 \left( \dt + \gD \right) \left( - \frac{1}{12} \brs{H}^2 \right) =&\ -
\frac{1}{12} \left[ \left<6 \Rc - \frac{3}{2} H^2, H^2 \right> + 2 \left< \gD_d
H, H \right> + \gD \brs{H}^2 \right]\\
=&\ \left< \frac{1}{8} H^2 - \frac{1}{2} \Rc, H^2 \right> - \frac{1}{6} \left<
\gD_d H, H \right> - \frac{1}{12} \gD \brs{H}^2.
\end{align*}
Combining these yields
\begin{align*}
 \left( \dt + \gD \right) \mathcal{S} =&\ \gD \brs{\N f}^2 - \frac{1}{12} \gD \brs{H}^2 -
\frac{1}{6} \left< \gD_d H, H \right> + \frac{1}{2} \divg \divg H^2\\
&\ + \left< \frac{1}{2} \N \brs{H}^2 - \divg H^2, \N f \right> + \left<
\frac{1}{2} H^2 - 2 \Rc, \N f \otimes \N f \right>\\
&\ + 2\brs{\Rc - \frac{1}{4} H^2 + \N^2 f}^2 - 2\brs{\N^2 f}^2\\
&\ + 2 \left< \N\left( \gD f - \brs{\N f}^2 + R - \frac{1}{4} \brs{H}^2 \right),
\N f \right>.
\end{align*}
Using Lemma \ref{l:divHlemma} one has
\begin{align*}
 \frac{1}{2} \divg \divg H^2 - \frac{1}{6} \left< \gD_d H, H \right> -
\frac{1}{12} \gD \brs{H}^2 = \frac{1}{2} \brs{d^* H}^2,
\end{align*}
and
\begin{align*}
 \left< \frac{1}{2} \N \brs{H}^2 - \divg H^2, \N f \right> = \frac{1}{3} \left<
\N \brs{H}^2, \N f \right> - \left< d^* H, \N f \lrcorner H \right>.
\end{align*}
Also one has
\begin{align*}
 \gD \brs{\N f}^2 - 2 \brs{\N^2 f}^2 - 2 \left< \Rc, \N f \otimes \N f \right> =
2 \left<\N \gD f, \N f \right>.
\end{align*}
Therefore
\begin{align} \label{f:steadyHarnack10}
 \left( \dt + \gD \right) \mathcal{S} =&\ 2 \brs{\Rc - \frac{1}{4} H^2 + \N^2 f}^2 +
\frac{1}{2} \brs{d^* H - \N f \lrcorner H}^2 + 2 \left< \N \mathcal{S}, \N f \right>.
\end{align}
It follows that
\begin{align*}
 \left( \dt + \gD \right) v =&\ \left( \dt + \gD \right) (\mathcal{S} u)\\
=&\ \left[\left( \dt + \gD \right)\mathcal{S} \right]u + \mathcal{S} \left( \dt + \gD \right) u + 2
\left< \N \mathcal{S}, \N u \right>\\
=&\ \left[ 2 \brs{\Rc - \frac{1}{4} H^2 + \N^2 f}^2 + \frac{1}{2} \brs{d^* H -
\N f \lrcorner H}^2 \right] u\\
&\ + \mathcal{S} \left( \gD f - \brs{\N f}^2 + R - \frac{1}{4} \brs{H}^2 \right) u + \mathcal{S}
\left( - \gD f + \brs{\N f}^2 \right)u\\
=&\ \left[ 2 \brs{\Rc - \frac{1}{4} H^2 + \N^2 f}^2 + \frac{1}{2} \brs{d^* H -
\N f \lrcorner H}^2 \right] u + v \left( R - \frac{1}{4} \brs{H}^2 \right),
\end{align*}
as claimed.
\end{proof}
\end{thm}

One application of this fundamental calculation is a pointwise estimate for 
$\mathcal S$ which can be interpreted as a kind of Harnack estimate for generalized 
Ricci flow.

\begin{cor} \label{c:steadyharnack} Let $(M^n, g_t, H_t)$ be a solution to 
generalized Ricci flow on a compact
manifold.  With $u$ and $f$ defined as above, one has that $\sup \mathcal{S}$ 
is a nondecreasing function of $t$.
\begin{proof} We let $\tau = -t$ denote a backwards time parameter.  Using Theorem \ref{t:steadyharnack} we compute
\begin{align*}
 \frac{\del}{\del \tau} \mathcal{S} =&\ - \frac{\del}{\del t} \mathcal{S}\\
=&\ - 2 \brs{\Rc - \frac{1}{4} H^2 + \N^2 f}^2 - \frac{1}{2} \brs{d^* H - \N f
\lrcorner H}^2 - R \mathcal{S} + \frac{1}{4} \brs{H}^2 \mathcal{S}\\
&\ + \frac{\gD v}{u} + \frac{\mathcal{S}}{u} \left( \gD f - \brs{\N f}^2 + R -
\frac{1}{4} \brs{H}^2 \right) u\\
=&\ \gD \mathcal{S} + 2 \frac{\left< \N v, \N u \right>}{u^2} - 2
\frac{v\brs{\N u}^2}{u^3} - 2 \brs{\Rc - \frac{1}{4} H^2 + \N^2 f}^2\\
&\ -
\frac{1}{2} \brs{d^* H - \N f \lrcorner H}^2\\
=&\ \gD \mathcal{S} + 2 \frac{\left< \N u, \N \mathcal{S}
\right>}{u} - 2 \brs{\Rc - \frac{1}{4} H^2 + \N^2 f}^2 - \frac{1}{2} \brs{d^* H
- \N f \lrcorner H}^2.
\end{align*}
Thus, in terms of the backwards time parameter $\tau$, the function 
$\mathcal{S}$ is a subsolution of a parabolic equation, and so by the maximum 
principle (cf. Proposition \ref{p:maxprinc2}), the supremum of $\mathcal{S}$ is 
nonincreasing in $\tau$, and is thus nondecreasing in $t$.
\end{proof}
\end{cor}

\begin{rmk} We observe that Proposition \ref{p:genFFmonotonicity} follows 
immediately from Theorem \ref{t:steadyharnack}.  In particular, since $\dt dV 
= - R + \tfrac{1}{4} \brs{H}^2$ along a solution to generalized Ricci flow, it 
follows from Theorem \ref{t:steadyharnack} that
\begin{align*}
\dt \left( \mathcal{S} u dV \right) =&\ \left\{ - \gD (\mathcal{S} u) + \left[ 2 \brs{\Rc -
\frac{1}{4} H^2 + \N^2 f}^2 + \frac{1}{2} \brs{d^* H - \N f \lrcorner H}^2 
\right] u \right\} dV.
\end{align*}
Integrating this result over $M$ yields Proposition \ref{p:genFFmonotonicity}.
\end{rmk}

\section{Expander entropy and Harnack estimate}

The quantities and results of \S \ref{s:GRFasgradient} are designed to capture 
the behavior of steady solitons, which are flow lines which evolve purely by 
diffeomorphisms.  By appropriately modifying the solution to the conjugate heat equation, we can obtain an entropy 
functional and Harnack inequality as in \S \ref{s:GRFasgradient} which is fixed on expanding solitons.  We 
simplify the discussion however by deriving the main differential inequality 
and observing the monotonicity and Harnack inequalities as corollaries.  These quantities were originally discovered in \cite{Streetsexpent}, generalizing Feldman-Ilmanen-Ni's adaptation 
\cite{FIN} of Perelman's quantities to the expanding setting.

\begin{thm} \label{t:expandingharnack} Let $(M^n, g_t, H_t)$ be a solution to 
generalized Ricci 
flow on $[T, T']$
and suppose $u_t$ is a solution to the conjugate heat equation.  Let $f_+$ be
defined by $u = \frac{e^{-f_+}}{(4\pi(t - T))^{\frac{n}{2}}}$, and let
 \begin{align*}
  v_+ = \left[ (t - T) \left( 2 \gD {f_+} - \brs{\N f_+}^2 + R - \frac{1}{12}
\brs{H}^2 \right) - f_+ + n \right] u.
 \end{align*}
Then
\begin{align*}
\square^* v_+ =&\ - 2(t - T) \left( \brs{\Rc - \frac{1}{4} H^2 + \N^2 f_+ + \frac{g}{2
(t - T)}}^2 + \frac{1}{4} \brs{d^* H - \N f_+ \lrcorner H}^2 \right) u\\
&\  - 
\frac{1}{6}
\brs{H}^2 u.
\end{align*}

\begin{proof} We assume without loss of generality that $T = 0$.  Let $V := 2
\gD f_+ - \brs{\N f_+}^2 + R - \frac{1}{12} \brs{H}^2$.  To compute the 
relevant evolution equation for $V$ we first observe that $f_+$ satisfies
\begin{gather} \label{f:fplusev}
\frac{\del f_+}{\del t} = - \gD f_+ + \brs{\N f_+}^2 - R + \frac{1}{4}
\brs{H}^2 - \frac{n}{2t}.
\end{gather}
Note that this differs from the variation of $f$ in the steady setting only by 
the addition of the constant term (see (\ref{f:steadyHarnack5}), the 
computation of $(\dt + \gD) V$ is identical to that in Theorem 
\ref{t:steadyharnack}, thus we may jump to line (\ref{f:steadyHarnack10}) to 
yield
\begin{gather} \label{f:expHarnack5}
\left( \dt + \gD \right) V = 2 \brs{\Rc - \frac{1}{4} H^2 + \N^2 f_+}^2 +
\frac{1}{2} \brs{d^* H - \N f_+ \hook \tau}^2 + 2 \left< \N V, \N f_+
\right>.
\end{gather}
Now let $W := tV - f_+ + n$.  Using (\ref{f:expHarnack5}) we get
\begin{align*}
\left( \dt + \gD \right) W =&\ V + 2 t \brs{\Rc - \frac{1}{4} H^2 + \N^2 f_+}^2
+ \frac{t}{2} \brs{d^* H - \N f_+ \hook H}^2\\
&\ + 2 t \left< \N V, \N f_+ \right> - \brs{\N f_+}^2 + R - \frac{1}{4}
\brs{H}^2 - \frac{n}{2 t}\\
=&\ 2 \gD f_+ + 2 R - \frac{1}{3} \brs{H}^2 + \frac{n}{2 t} + 2 \left<\N W,
\N f_+ \right>\\
&\ 2 t \brs{\Rc - \frac{1}{4} H^2 + \N^2 f_+}^2 + \frac{t}{2} \brs{d^* H - \N
f_+ \hook H}^2.\\
=&\ 2 t \brs{\Rc - \frac{1}{4} H^2 + \N^2 f_+ + \frac{g}{2 t}}^2 + \frac{t}{2}
\brs{d^* H - \N f_+ \hook H}^2\\
&\ + \frac{1}{6} \brs{H}^2 + 2 \left< \N W, \N f_+ \right>.
\end{align*}
Finally we can compute
\begin{align*}
\left(\dt + \gD \right) v_+ =&\ \left( \dt + \gD \right) \left(W u \right)\\
=&\ \left(\left( \dt + \gD \right) W \right) u + W \left(\dt + \gD \right) u + 2
\left<\N W, \N u \right>\\
=&\ \left( 2 t \brs{\Rc - \frac{1}{4} H^2 + \N^2 f_+ + \frac{g}{2 t}}^2 +
\frac{t}{2} \brs{d^* H - \N f_+ \hook H}^2 \right.\\
&\ \left. + \frac{1}{6} \brs{H}^2 + 2 \left< \N W, \N f_+ \right> \right) u +
W \left( R u - \frac{1}{4} \brs{H}^2 u \right) - 2 \left< \N W, \N f_+
\right> u\\
=&\ 2 t \brs{\Rc - \frac{1}{4} H^2 + \N^2 f_+ + \frac{g}{2 t}}^2 u + \frac{t}{2}
\brs{d^* H - \N f_+ \hook H}^2 u + \frac{1}{6} \brs{H}^2 u\\
&\ + R v_+ - \frac{1}{4} \brs{H}^2 v_+
\end{align*}
and the result follows.
\end{proof}
\end{thm}

\begin{defn} \label{d:expandingentropy} Let $(M^n, g, H)$ be a smooth manifold with $H 
\in \Lambda^3 T^*, 
dH = 0$.  Fix $u \in C^{\infty}$, $u > 0$, and define $f_+$ via $u = \frac{e^{-f_+}}{(4\pi(t - T))^{\frac{n}{2}}}$.  The 
\emph{expanding entropy} associated to this data is
 \begin{align*}
\WW_+(g,H,f_+,\gs)  := \int_M \left[ \gs \left( \brs{\N f_+}^2 + R - 
\frac{1}{12}
\brs{H}^2 \right) - f_+ + n \right] u dV.
\end{align*}
Furthermore, let
\begin{gather*} %\label{f:muplusdef}
\mu_+ \left(g, H, \gs \right) := \inf_{ \left\{ f_+ | \int_M e^{-f_+}{(4 \pi \gs)^{-\frac{n}{2}}} dV = 1 
\right\}} \WW_+
\left(g, H, f_+, \gs \right),
\end{gather*}
and
\begin{gather*} %\label{f:nuplusdef}
\nu_+(g, H) := \sup_{\gs > 0} \mu_+ \left(g,H,\gs \right).
\end{gather*}
\end{defn}

\begin{cor} \label{c:expandingentropymonotonicity} Let $(M^n, g_t, H_t)$ be a 
solution to 
the $(X_t, k_t)$-gauge-fixed generalized Ricci flow equations.  Let $u_t$ denote a 
solution to the 
$(X_t, k_t)$-gauge-fixed conjugate heat equation, and define $f_+$ via $u = 
\frac{e^{-f_+}}{(4\pi(t - T))^{\frac{n}{2}}}$.  Then
\begin{gather*}
\begin{split}
\frac{d}{dt}& \WW_+(g, H, f_+, t - T)\\
=&\ \int_M \left[ (t - T) \left( 2 \brs{\Rc - 
\tfrac{1}{4} H^2 
+ \N^2 f_+ + \frac{g}{2(t-T)}}^2 + \frac{1}{2} \brs{d^* H + \N f_+ \lrcorner H}^2 \right) + 
\frac{1}{6}
\brs{H}^2 \right] u dV.
\end{split}
\end{gather*}
\begin{proof} As in the proof of Proposition \ref{p:genFFmonotonicity}, it suffices to consider the flow in any specific gauge, and in this case we pick the standard gauge.  Since $\dt dV = R - \tfrac{1}{4} \brs{H}^2$ along a solution to 
generalized Ricci flow, it follows from Theorem \ref{t:expandingharnack} that
\begin{align*}
\dt \left( v_+ dV \right) =&\ \left[ (t - T) \left( 2 \brs{\Rc - \tfrac{1}{4} 
H^2 
+ \N^2 f_+ + \frac{g}{2(t-T)}}^2 + \frac{1}{2} \brs{d^* H + \N f_+ \lrcorner H}^2 \right) + 
\frac{1}{6}
\brs{H}^2 \right] u dV.
\end{align*}
Integrating this result over $M$ yields the result.
\end{proof}
\end{cor}

\begin{cor} \label{c:expandingharnack} Let $(M^n, g_t, H_t)$ be a solution to 
generalized Ricci flow on a compact
manifold, $t \in [T, T']$.  Let $u_t$ be a solution to the conjugate heat equation and define $v_+$ as in Theorem 
\ref{t:expandingharnack}.  Then 
$\sup \frac{v_+}{u}$ is nondecreasing in $t$.
\begin{proof} Without loss of generality assume $T = 0$.  We compute with respect to the backwards time parameter $\tau = -t$,
\begin{align*}
 \frac{\del}{\del \tau} \frac{v_+}{u} =&\ - \frac{\del}{\del t} \frac{v_+}{u}\\
=&\ - 2 t\brs{\Rc - \frac{1}{4} H^2 + \N^2 f_+ + \frac{g}{2 t}}^2 - \frac{1}{2}
t \brs{d^* H - \N f_+
\hook H}^2 -\frac{1}{6} \brs{H}^2\\
&\ - R \frac{v_+}{u} + \frac{1}{4} \brs{H}^2 \frac{v_+}{u} + \frac{\gD v_+}{u} +
\frac{v_+}{u^2} \left( \gD u + R u - \frac{1}{4} \brs{H}^2 u \right)\\
=&\ \gD \left( \frac{v_+}{u} \right) + 2 \frac{\left< \N v_+, \N u \right>}{u^2}
- 2
\frac{v_+\brs{\N u}^2}{u^3}\\
&\ - 2 t \brs{\Rc - \frac{1}{4} H^2 + \N^2 f_+ + \frac{g}{2 t}}^2 -
\frac{1}{2} t \brs{d^* H - \N f_+ \hook H}^2 - \frac{1}{6} \brs{H}^2\\
=&\ \gD \left( \frac{v_+}{u} \right) + 2 \frac{\left< \N u, \N \frac{v_+}{u}
\right>}{u} - 2 t\brs{\Rc - \frac{1}{4} H^2 + \N^2 f_+ + \frac{g}{2t} }^2\\
&\ -
\frac{1}{2} t \brs{d^* H
- \N f_+ \hook H}^2 - \frac{1}{6} \brs{H}^2.
\end{align*}
The result follows by the maximum principle as in Corollary \ref{c:steadyharnack}.
\end{proof}
\end{cor}

\begin{prop} \label{p:rigidityofgenexp} Let $(M^n, g_t, H_t)$ be a solution to generalized Ricci flow
on a compact manifold.
The infimum in the definition of $\mu_+$ is attained by a unique function $f$.  Furthermore $\mu_+(g_t, H_t,t-T)$ is
monotonically nondecreasing along generalized Ricci flow, and is constant only on a generalized Ricci expander, i.e. $H
\equiv 0$ and $g$ is a Ricci expander.
\begin{proof} Using the relationship $u = \frac{e^{- f_+}}{(4 \pi (t - T)^{\frac{n}{2}}}$ and then setting $w = u^{\tfrac{1}{2}}$, we can express, up to a constant shift,
\begin{align*}
\WW_+(g,H,w,\gs) =&\ \int_M \left[ \gs \left( 4 \brs{\N w}^2 + R w^2 - \frac{1}{12} \brs{H}^2 w^2\right) + w^2 \log w^2 \right] dV.
\end{align*}
This functional is coercive for the Sobolev space $W^{1,2}$ and convex.  By \cite{Rothaus} this functional has a unique smooth minimizer.  Thus at any time $t_0$ we can choose the smooth minimizer $u_{t_0}$ defining $\mu_+$, and construct a solution to the conjugate heat equation on $[0,t_0]$ with this final value.  It follows using Corollary \ref{c:expandingentropymonotonicity} that, at time $t_0$, one has
\begin{align*}
\frac{d}{dt} & \mu_+(g_t, H_t, t - T)\\
=&\ \frac{d}{dt} \WW_+( g_t, H_t, u_t, t - T)\\
=&\ \int_M 2 \left[ (t-T) \left( \brs{\Rc -
\frac{1}{4} H^2 + \N^2 f_+ + \frac{g}{2 (t-T)}}^2 \right. \right.\\
&\ \qquad \qquad \qquad \qquad \left. \left. + \frac{1}{4} \brs{d^* H - \N f_+ \hook H}^2 \right) + \frac{1}{12}
\brs{H}^2 \right] u dV.
\end{align*}
The monotonicity and rigidity statements follow easily from this formula, which holds for the corresponding choice of $u$ at all times $t$.
\end{proof}
\end{prop}

\section{Shrinking Entropy and local collapsing} \label{s:shrinker}

Having dealt with the steady and expanding soliton cases, it is natural to seek a monotone entropy quantity which is fixed  along a shrinking soliton.  Following the ideas of the previous sections, a natural entropy functional is suggested which is a generalization of Perelman's original entropy which includes $H$.  However, as we will see below, this entropy is not monotone along generalized Ricci flow.  Nonetheless, by assuming the existence of a certain kind of subsolution to the heat equation along the solution, which exists in many applications, it is possible to further modify to obtain a monotone quantity.  In this setting we adapt Perelman's ideas to obtain a $\gk$-noncollapsing result for solutions to generalized Ricci flow.

\begin{thm} \label{t:shrinkingharnack} Let $(M^n, g_t, H_t)$ be a solution to 
generalized Ricci 
flow on $[0,T]$
and suppose $u_t$ is a solution to the conjugate heat equation.  Let $f_-$ be
defined by $u = \frac{e^{-f_-}}{(4\pi(T - t))^{\frac{n}{2}}}$, and let
 \begin{align*}
  v_- = \left[ (T - t) \left( 2 \gD {f_-} - \brs{\N f_-}^2 + R - \frac{1}{12}
\brs{H}^2 \right) + f_- - n \right] u.
 \end{align*}
Then
\begin{align*}
\square^* v_- =&\ 2(T - t) \left( \brs{\Rc - \frac{1}{4} H^2 + \N^2 f_- - \frac{g}{2
(T - t)}}^2 + \frac{1}{4} \brs{d^* H - \N f_- \lrcorner H}^2 \right) u\\
&\ -
\frac{1}{6}
\brs{H}^2 u.
\end{align*}
\begin{proof} This is an elementary modification of Theorem \ref{t:expandingharnack}, and we leave the details as an \textbf{exercise}.
\end{proof}
\end{thm}

\begin{rmk}The presence of the final term $-\frac{1}{6} \brs{H}^2 u$ above makes it more difficult to use Theorem \ref{t:shrinkingharnack} without further hypotheses.  Nonetheless we will be able to modify this quantity further to obtain a monotone entropy in some settings.  The discovery of a monotone entropy quantity fixed on shrinking solitons for generalized Ricci flow in full generality remains an important open problem.
\end{rmk}

\begin{defn} \label{d:shrinkingentropy} Let $(M^n, g, H)$ be a smooth manifold with $H 
\in \Lambda^3 T^*, 
dH = 0$.  Fix $u \in C^{\infty}$, $u > 0$, and define $f_-$ via $u = \frac{e^{-f_-}}{(4\pi(T - t))^{\frac{n}{2}}}$.  The 
\emph{shrinking entropy} associated to this data is
 \begin{align*} 
\WW_-(g,H,f_-,\tau)  := \int_M \left[ \tau \left( \brs{\N f_-}^2 + R - 
\frac{1}{12}
\brs{H}^2 \right) + f_- - n \right] u dV.
\end{align*}
Furthermore, let
\begin{gather*}
\mu_- \left(g, H, \tau \right) := \inf_{\left\{ f_- | \int_M e^{-f_-}{(4 \pi \tau)^{-\frac{n}{2}}} dV = 1 
\right\} } \WW_-
\left(g, H, u, \tau \right),
\end{gather*}
and
\begin{gather*}
\nu_-(g, H) := \sup_{\tau > 0} \mu_- \left(g,H,\tau \right).
\end{gather*}
\end{defn}

The evolution equation for $\WW_-$ follows as an immediate application of Theorem \ref{t:shrinkingharnack}, and will not be monotone in general.  To modify this to obtain a monotone energy we make a further definition.

\begin{defn} \label{tbsdef} Let $M^n$ be a compact manifold, and suppose $(g_t,
H_t)$ is a
solution to generalized Ricci flow.  We say that a one-parameter family of functions $\phi$ is a \emph{torsion-bounding
subsolution} if $\phi \geq 0$ and
\begin{align*}
 \dt \phi \leq&\ \gD \phi - \brs{H}^2.
\end{align*}
\end{defn}

The following basic lemma indicates the usefulness of a torsion-bounding subsolution.

\begin{lemma} \label{tbsmono} Let $M^n$ be a compact manifold.  Given
$(g_t,H_t)$ a solution
to generalized Ricci flow, $u_t$ a solution of the conjugate heat equation, and $\phi_t$ a solution to
\begin{align*}
 \dt \phi =&\ \gD \phi + \psi,
\end{align*}
one has
\begin{align*}
 \dt \int_M \phi u dV =&\ \int_M \psi u dV. 
\end{align*}
\begin{proof} We compute
\begin{align*}
\dt \int_M \phi u dV =&\ \int_M \left[ \dot{\phi} u + \phi \dot{u} + \phi u \left(- R + \tfrac{1}{4} \brs{H}^2 \right) \right] dV\\
=&\ \int_M \left[ \left(\gD \phi  + \psi \right) u + \phi \gD u \right]\\
=&\ \int_M \psi u dV,
\end{align*}
as required.
\end{proof}
\end{lemma}

\begin{defn} Let $(M^n, g_t, H_t)$ be a solution to generalized Ricci flow on a compact manifold, and suppose furthermore that $\phi_t$ is a smooth
one-parameter family of functions.  Let the
\emph{augmented entropy} be
\begin{align*}
 \WW_-& (g,H,\phi,f_-,\tau)\\
 :=&\ \int_M \left[ \tau \left(  
\brs{\N f_-}^2 + R - \frac{1}{12} \brs{H}^2 \right) - \phi + f_- - n \right] \left( 4 \pi \tau
\right)^{-\frac{n}{2}} e^{-f_-} dV.
\end{align*}
Moreover, let
\begin{align*}
\mu_-(g,H,\phi,\tau) = \inf_{\{f_- | \int_M (4 \pi \tau)^{-\frac{n}{2}} e^{-f_-} = 1\}}
\WW_-(g,H,\phi,f_-,\tau)
\end{align*}
\end{defn}

\begin{prop} \label{modifiedmono} Let $(M^n, g_t, H_t)$ be a solution to generalized Ricci flow on a compact manifold, and suppose that $\phi_t$ is a torsion-bounding 
subsolution.  Then, setting $\tau = T - t$ for some time $T$ we have
\begin{align*}
\dt & \WW_-(g,H,\phi,f_-,\tau)\\
=&\ \int_M \left[ 2 \tau \brs{\Rc - \frac{1}{4} H^2
+ \N^2 f_- - \frac{1}{2 \tau}
g}^2 +\frac{\tau}{6} \brs{d^* H + i_{\N f_-} H}^2 + \frac{5}{6} \brs{H}^2
\right]
(4 \pi \tau)^{-\frac{n}{2}} e^{-f_-} dV.
\end{align*}
\begin{proof} This follows immediately by combining the results of Theorem
\ref{t:shrinkingharnack} and Lemma \ref{tbsmono}.
\end{proof}
\end{prop}

\begin{prop} \label{p:rigidityofgenshr} Let $(M^n, g_t, H_t)$ be a solution to generalized Ricci flow
on a compact manifold admitting a torsion-bounding subsolution $\phi_t$.  The infimum in the definition of $\mu_-$ is attained by a unique function $f$.  Furthermore $\mu_-(g_t, H_t,T - t)$ is
monotonically nondecreasing along generalized Ricci flow, and is constant only on a generalized shrinking Ricci soliton, i.e. $H
\equiv 0$ and $g$ is a shrinking Ricci soliton.
\begin{proof} This proof is directly analogous to Proposition \ref{p:rigidityofgenexp}.
\end{proof}
\end{prop}

Thus in the presence of a torsion-bounding subsolution, we have constructed a monotone entropy functional.  To end this section we give one key consequence of this result, namely that in the presence of a torsion-bounding subsolution, solutions to generalized Ricci flow are not locally collapsing.

\begin{defn} \label{d:localcollapse} A solution $(g_t, H_t)$ to generalized Ricci flow is \emph{locally
collapsing at $T$} if there is a sequence $t_k \to T$ and a sequence of metric
balls $B_k = B(p_k, r_k, g_{t_k})$ such that $r_k^2 / t_k$ is bounded,
$\brs{\Rm}(g_{t_k}) \leq r_k^{-2}$ in $B_k$, and
\begin{align*}
 \lim_{k \to \infty} r_k^{-n} \Vol(B_k) = 0.
\end{align*}
\end{defn}

This property, roughly speaking, says that if we rescale around the given sequence so that the curvature is unit size, then the volume of unit balls around the basepoint in the blowup sequence approaches zero, indicating that the manifold is converging to a lower-dimensional space.  If we want to apply the compactness theorems of \S \ref{s:compactness} to a blowup sequence of generalized Ricci flows, then we need to rule this behavior out.  The augmented entropy monotonicity allows us to establish this key estimate.  We indicate in Chapter \ref{c:GFCG} (see Remark \ref{r:torsionremark}) some natural situations where a torsion-bounding subsolution can be found and hence this estimate can be applied.

\begin{thm} \label{t:kappanoncollapse} Let $(M^n, g_t, H_t)$ be a solution to generalized Ricci flow on $[0,
T)$, $T < \infty$.  Suppose there exists a torsion-bounding subsolution $\phi_t$ on $[0,T)$. 
Then $(g_t, H_t)$ is not locally collapsing at $T$.
\begin{proof} Suppose there exists a sequence $\{(p_k, t_k, r_k)\}$ which exhibits local collapsing as in Definition \ref{d:localcollapse}.  We fix $\eta : [0,\infty) \to [0,1]$ a nonincreasing function such that $\eta(s) = 1$ for $s \in [0,\frac{1}{2}]$, $\eta(s) = 0$ for $s \geq 1$, and $\brs{\eta'(s)} \leq 4$.  Then let
\begin{align*}
U_k(x) := e^{- \gl_k/2} \eta \left( d_{g_{t_k}}(p_k, x) / r_k \right).
\end{align*}
These functions $U_k$ are scaled cutoff functions for the sequence of collapsing balls.  We aim to use these as test functions in $\WW_-$.  Specifically, we define $f_k$ via $e^{- f_k/2} = U_k$, and will show that $\WW_-(g_{t_k}, H_{t_k}, f_k, r_k^2) \to-\infty$.  Given this, by the monotonicity of $\mu_-(g_t, H_t, \phi_t, T - t)$, it follows that $\mu_-(g_0, H_0, t_k + r_k^2) \to -\infty$, which contradicts that $\mu(g_0, H_0, \tau)$ is finite for all $\tau$, yielding a contradiction.

For $f_k$ to be a valid test function for $\mu$, we need to ensure the unit volume condition, which is determined by the scaling factor $\gl_k$.  We observe that the necessary condition implies
\begin{align*}
1 =&\ \int_M U_k^2 \left(4 \pi r_k^2 \right)^{-\frac{n}{2}} dV\\
=&\ e^{-\gl_k} \left(4 \pi r_k^2 \right)^{-\frac{n}{2}} \int_M \eta^2 \left( d_{g_{t_k}}(p_k, x)/r_k \right) dV\\
\leq&\ e^{-\gl_k} \left(4 \pi r_k^2 \right)^{-\frac{n}{2}} \Vol(B_k).
\end{align*}
Thus, by the hypothesis of local collapsing, we have that $\gl_k \to - \infty$.  To estimate $\WW_-$ from above, we first observe using the assumed curvature bound, the normalization of $f_k$ and dropping some negative terms,
\begin{align*}
\WW_-&(g_{t_k}, H_{t_k}, \phi_{t_k}, f_k, r_k^2)\\
\leq&\ C + \left(4 \pi r_k^2 \right)^{-\frac{n}{2}} \int_M \left[ r_k^2 \brs{\N f_k}^2 + f_k \right] e^{- f_k} dV.
\end{align*}
By using the lower bound on the Ricci curvature on $B_k$, the estimates on $\eta$, and Bishop-Gromov relative volume comparison one has, setting $\eta_k = \eta(d_{g_{t_k}}(p_k, \cdot)/r_k)$, 
\begin{align*}
\left(4 \pi r_k^2 \right)^{-\frac{n}{2}} & \int_M \left[ r_k^2 \brs{\N f_k}^2 + f_k \right] e^{- f_k} dV\\
=&\ \left(4 \pi r_k^2 \right)^{-\frac{n}{2}} e^{-\gl_k} \int_M \left[ r_k^2 \brs{\N \eta_k}^2 - 2 \eta_k^2 \ln \eta_k \right] dV + \gl_k\\
\leq&\ C \frac{\Vol(B(p_k,r_k)) - \Vol(B(p_k, r_k/2))}{\Vol(B(p_k, r_k/2))} + \gl_k\\
\leq&\ C + \gl_k.
\end{align*}
Thus $\WW_-(g_{t_k}, H_{t_k}, \phi_{t_k}, f_k, r_k^2) \to -\infty$, as required.
\end{proof}
\end{thm}

\section{Corollaries on nonsingular solutions}

The energy and entropy functionals above play an important role in understanding singularities of generalized Ricci flow.  The simplest application is to understanding global solutions with some asymptotic control at infinite time.

\begin{defn} \label{d:nonsingular} Given $(M^n, g_t, H_t)$ a solution to generalized Ricci flow, we say that the solution is \emph{nonsingular} if there exists $\gL \in \{-1,0,1\}$ such that the corresponding rescaled solution of
\begin{gather*}
\begin{split}
\dt g =&\ -2 \Rc + \frac{1}{2} H^2 + \gL g,\\
\dt H =&\ \gD_d H + \gL H,
\end{split}
\end{gather*}
exists on $[0,\infty)$ and satisfies
\begin{align*}
\sup_{M \times [0,\infty)} \brs{\Rm_g}_g < \infty.
\end{align*}
\end{defn}

\begin{rmk}
\begin{enumerate}
\item This definition differs slightly from the classical definition of nonsingular solution which asks for bounded curvature along the normalization which fixes the volume.  In  many cases nonsingular solutions in this sense will approach constant scalar curvature, rendering them nonsingular in our sense.  However, in general the two notions are different.
\item Solitons are examples of nonsingular solutions.
\item From the point of view of analysis, it is natural to modify the flow of $H$ as we have done, so that the resulting family $(g_t, H_t)$ is a spacetime scaling of the corresponding unnormalized flow.  Note however that this operation will of course change the cohomology class of $H$, thus altering the underlying Courant algebroid structure.
\end{enumerate}
\end{rmk}

\begin{thm} \label{t:nonsingularclassification} Suppose $(M^n, g_t, H_t)$ is a nonsingular solution of generalized Ricci flow with $\gL = 0, -1$.  Then exactly one of the following two statements holds:
\begin{enumerate}
\item $\limsup_{t \to \infty} \inj_{g_t} = 0$.
\item There exists a sequence $(p_j, t_j)$, $t_j \to \infty$ such that $(M^n, g_{t_j}, H_{t_j}, p_j)$ converges in the generalized Cheeger-Gromov topology to a generalized Ricci soliton $(M^n_{\infty}, g_{\infty}, H_{\infty}, p_{\infty})$.  In the case $\gL = -1$, $\GG_{\infty}$ is a generalized expanding soliton, i.e. $H_{\infty} \equiv 0$ and $g_{\infty}$ is an expanding Ricci soliton with negative scalar curvature.
\end{enumerate}
\begin{proof} If the first case does not hold, there exists $\gd > 0$ and a sequence of points $(p_j, t_j)$, $t_j \to \infty$ such that $\inj_{g_{t_j}}(p_j) \geq \gd$.  Our goal is to modify this sequence using the energy and entropy monotonicities to obtain the claimed limit.  At this point we restrict the proof to the case $\gL = 0$ and apply the monotonicity of the $\FF$-functional.  In the case $\gL = - 1$ one must use the $\WW_{+}$ functional, with the vanishing of $H$ following from the extra $\brs{H}^2$ term in the monotonicity formula of Corollary \ref{c:expandingentropymonotonicity}.  We leave the details to the reader.  We first observe that by choosing the function $f_t \equiv \log \Vol(g_t)$, it follows that
\begin{align*}
\gl(g_t, H_t) \leq \FF(g_t,H_t,f_t) = \Vol(g_t)^{-1} \int_M \left( R - \tfrac{1}{12} \brs{H}^2 \right) dV \leq C,
\end{align*}
where the last line follows since the curvature is uniformly bounded along the flow.  As $\gl$ is now monotonically nondecreasing and bounded above, we expect convergence to a critical point provided we can establish regularity of the flow.

Given one of the times $t_j$ above, as in the proof of Corollary \ref{c:genllmonotonicity} we construct a solution $f_t$ to the conjugate heat equation where
\begin{align*}
\int_M e^{- f_{t_j+1}} dV = 1, \qquad \FF(g_{t_j + 1}, H_{t_j + 1}, f_{t_j + 1}) = \gl (g_{t_j + 1}, H_{t_j + 1}).
\end{align*}
Applying the argument of Corollary \ref{c:genllmonotonicity} then yields
\begin{align*}
\gl(g_{t_j},H_{t_j}) \leq&\ \gl(g_{t_j + 1},H_{t_j + 1})\\
&\ \qquad - \int_{t_j}^{t_j + 1} \int_M \left[ 2 \brs{\Rc - 
\tfrac{1}{4} H^2 
+ \N^2 f}^2 + \tfrac{1}{2} \brs{d^* H + \N f \lrcorner H}^2 \right] e^{-f} dV.
\end{align*}
Without loss of generality we can assume $t_{j+1} \geq t_j + 1$, and then by summing the above inequalities we obtain, taking care with the definition of $f$ on each time interval and using that $\gl$ is bounded above,
\begin{align*}
\sum_{j=1}^{\infty} \int_{t_j}^{t_j + 1} \int_M \left[ 2 \brs{\Rc - 
\tfrac{1}{4} H^2 
+ \N^2 f}^2 + \tfrac{1}{2} \brs{d^* H + \N f \lrcorner H}^2 \right] e^{-f} dV < \infty.
\end{align*}
It follows that the terms in the sum must go to zero, and then in particular there exists $t_j \leq t_j' \leq t_j + 1$ such that
\begin{align} \label{f:nonsingular10}
\lim_{j \to \infty} \int_M \left[ 2 \brs{\Rc - 
\tfrac{1}{4} H^2 
+ \N^2 f}^2 + \tfrac{1}{2} \brs{d^* H + \N f \lrcorner H}^2 \right] e^{-f} dV = 0.
\end{align}

It remains to take a smooth limit of $(M^n, g_{t_j}, H_{t_j}, f_{t_j}, p_j)$.  To obtain this, first observe that for all $t \geq 1$, we can apply Theorem \ref{t:smoothing} on $[t- 1, t]$ to conclude uniform estimates on all derivatives of curvature and of $H$.  Since we have a uniform estimate on $\frac{\del g}{\del t}$, the first part of the argument of Lemma \ref{l:smoothfamilies} applies to yield uniform equivalence of all metrics $g_t$, $t_j \leq t \leq t_j + 1$.  Since $\inj_{g_{t_j}}(p_j) \geq \gd$, it follows that there exists $\gd' > 0$ so that $\inj_{g_{t_j'}}(p_j) \geq \gd'$.  We next sketch an argument that the functions $f_{t_{j}'}$ satisfy estimates of the form
\begin{align*}
\sup_{B_R(p_j, g_{t_j'})} \brs{\N^k_{g_{t_j'}} f_{t_j'}}_{g_{t_j'}} \leq C(k,R).
\end{align*}
First, given the uniform curvature and derivative estimates, at any point we can work on the universal cover of the unit ball, which is then uniformly equivalent to a Euclidean ball by Jacobi field estimates.  Since $u_{t_j + 1} = e^{- f_{t_j + 1}}$ by construction realizes the infimum of $\gl$, it satisfies the elliptic equation (cf. Lemma \ref{l:energyminimizerexists})
\begin{align*}
\left(- 4 \gD + R - \tfrac{1}{12} \brs{H}^2 - \gl \right) u_{t_j+1} = 0
\end{align*}
Since all derivatives of curvature and $H$ are bounded and $\gD$ is uniformly elliptic, we can apply elliptic regularity to obtain estimates on all derivatives of $u_{t_j+1}$ on any unit ball as described above.  Moreover, the conjugate heat equation $\square^* u = 0$ is uniformly parabolic, and the linear term $R - \tfrac{1}{4} \brs{H}^2$ has uniform estimates, so we can invoke parabolic regularity results to obtain uniform control over all derivatives of $u_t$, $t_j \leq t \leq t_j + 1$.  We can now apply Theorem \ref{t:flowcompactness} to conclude that the sequence of pointed solutions to generalized Ricci flow determined by $\{ (M^n, g_t, H_t, p_j), t \in [t_j, t_{j+1}] \}$ admits a convergent subsequence, and that the associated functions $f$ as defined above converge as well.  It follows from (\ref{f:nonsingular10}) that this limit is a generalized Ricci soliton.
\end{proof}
\end{thm}

\begin{rmk}
\begin{enumerate}
\item In the first case above where the injectivity radius goes to zero everywhere, the manifolds are collapsing with bounded curvature.  Convergence results for the Ricci flow which yield geometric limits in this setting were proved by Lott \cite{Lott1, LottDR}.  Further rigidity results for the generalized Ricci flow appeared in \cite{GindiStreets}.
\item The limit space $M_{\infty}$ need not be diffeomorphic to $M$.  An elementary example occurs for hyperbolic three-manifolds (cf. Example \ref{e:hyperbolic}), where the limit of the unrescaled flow will be flat $\mathbb R^3$ with $H\equiv 0$ (noting that while $H$ is fixed in time as a tensor its $g_t$-norm goes to zero along the flow).  As described there the appropriately rescaled flow will converge back to the hyperbolic metric with $H \equiv 0$.  Taking products of this example with a Bismut-flat structure on a semisimple Lie group yields examples with a nontrivial limit for the unrenormalized flow, where the topology of the base manifold has changed.
\item A similar result can be obtained in the case $\gL = 1$, under the further assumption of the existence of a torsion-bounding subsolution, using the shrinking entropy.
\end{enumerate}
\end{rmk}

\begin{rmk} Perelman's energy and entropy quantities are also related to certain formal metric constructions on the whole spacetime of a solution to Ricci flow.  A related spacetime geometry approach for generalized Ricci flow was taken in \cite{Chencanonicalsolitons}, yielding certain canonical solitons associated to any solution.
\end{rmk}

\chapter{Generalized Complex Geometry} \label{c:GCG}

 In this chapter we introduce foundational concepts of generalized complex 
geometry, following fundamental works of Hitchin \cite{HitchinGCY} and Gualtieri \cite{GualtieriThesis}. 
 A core conceit in this subject is the unification of complex and symplectic 
structures into a single, natural object most clearly defined using the 
language of Courant algebroids.  We begin with the basic definitions and 
indicate how complex and symplectic structures fit in as examples of 
generalized 
complex structures, and also illustrate nontrivial examples.  We also express 
the geometry of pluriclosed Hermitian structures in this language, and 
discuss their relation to Courant algebroids in the holomorphic category.  We then 
turn to generalized K\"ahler geometry, first motivating and defining 
these structures in terms of generalized geometry, then discussing the 
equivalence to the classical description, which first arose in work of Gates-Hull-Ro\v{c}ek on supersymmetric $\sigma$-models \cite{GHR}.  We provide examples and a 
discussion of the associated Poisson geometry and conclude by exhibiting 
natural classes of variations of generalized K\"ahler structure.

\section{Linear generalized complex structures}

Before beginning our study of generalized complex structures we briefly recall the conceptual buildup of complex structures on manifolds.  Given $M^{2n}$ a smooth manifold, an \emph{almost complex
structure} on $M$ is an endomorphism $J$ of the tangent bundle covering the
identity map satisfying
\begin{align*}
J^2 = - \Id.
\end{align*}
The pair $(M^{2n}, J)$ is called an \emph{almost complex manifold}.  The restriction to an even dimensional manifold is
necessary for the existence of an endomorphism which squares to $- \Id$.  The endomorphism $J$ gives a splitting of the complexified tangent bundle into $\pm \i$ eigenbundles, and it is natural to ask when these subbundles are closed under the Lie bracket.  This is measured precisely by the vanishing of the \emph{Nijenhuis
tensor of $J$}, defined by
\begin{align*}
N(X, Y) = [JX, JY] - [X, Y] - J[JX, Y] - J[X, JY].
\end{align*}
We say that $J$ is \emph{integrable} if $N \equiv 0$.  In this case we say
that $(M^{2n}, J)$ is a \emph{complex manifold}.  It follows from the Newlander-Nirenberg Theorem \cite{Newlander} that the vanishing
of the Nijenhuis tensor is equivalent to the existence of a complex coordinate
atlas, i.e. complex coordinate charts covering the manifold with biholomorphic
transition maps.  The story of generalized complex structures follows a similar path, first defining almost complex structures on $T \oplus T^*$, then analyzing the geometric implications of the integrability of the associated splitting under the Dorfman bracket.

\subsection{Definition and examples}

To begin we give the definition of a linear generalized complex structure, which is a natural extension of the notion of complex structure to the generalized tangent bundle, further preserving the natural symmetric inner product.  As for generalized metrics, we can also motivate this concept from the point of view of structure groups.  Whereas a classical complex structure reduces the structure group of the tangent bundle to $U(n)$, a generalized complex structure will be a reduction of the structure group of $T \oplus T^*$ to $U(n,n)$.

Following our discussion in \S \ref{subsec:LinearGmetrics}, we fix a pair of real vector spaces $V$ and $W$ which fit into a short exact sequence
\begin{equation}
\label{eq:linearsurjection2}
\begin{gathered}
\xymatrix{0 \ar[r] & V^* \ar[r] &  W \ar[r]^-{\pi} & V \ar[r] & 0.}
\end{gathered}
\end{equation}
We assume that $W$ is endowed with a metric $\langle,\rangle$ of split signature $(n,n)$ such that the image of the first arrow in \eqref{eq:linearsurjection2}, given by $\tfrac{1}{2}\pi^* : V^* \to W^* \cong W$, is isotropic. 

\begin{defn} A \emph{generalized complex structure on $W$} is an endomorphism $\JJ$ of $W$ such that
\begin{enumerate}
 \item $\JJ^2 = - \Id$.
 \item $\JJ$ is orthogonal with respect to the symmetric inner product
$\IP{,}$ on $W$.
\end{enumerate}
When $W = V \oplus V^*$ and $\IP{,}$ is given by \eqref{Diracpairings}, we will simply say that $\JJ$ is a generalized complex structure on $V$.
\end{defn}

Of course, any generalized complex structure on $W$ defines a generalized complex structure on $V$ upon a choice isotropic splitting of the sequence \eqref{eq:linearsurjection2}. Nonetheless, in our discussion we will often find generalized complex structures interacting with generalized metrics, which determine a preferred isotropic splitting, and the abstract definition above will be useful. 

 We next discuss two central examples which will inform the entire discussion, 
illustrating that classical complex structures as well as symplectic structures 
admit interpretations as generalized complex structures.

\begin{ex} \label{e:compGC}  Let $W = V \oplus V^*$ and $\IP{,}$ given by \eqref{Diracpairings}. Let $J : V \to V$ be a complex structure, and define
\begin{align*}
 \JJ_J := \left(
 \begin{matrix}
  - J & 0\\
  0 & J^*
 \end{matrix}
\right),
\end{align*}
where the block diagonal structure comes from the natural decomposition of $V
\oplus V^*$ and $J^* \ga$ is the natural induced action of $J$ on $V^*$, i.e.
$J^* \ga(v) = \ga(Jv)$.  The equation $\JJ_J^2 = -\Id$ is immediate, and
we can furthermore compute
\begin{align*}
 \IP{\JJ_J (X + \xi), \JJ_J (Y + \eta)} =&\ \IP{- JX + J^* \xi, - J Y + J^*
\eta}\\
 =&\ \tfrac{1}{2} \left( - J^* \xi (J Y) - J^* \eta (JX) \right)\\
 =&\ \tfrac{1}{2} \left( \xi(Y) + \eta(X) \right)\\
 =&\ \IP{ X + \xi, Y + \eta }.
\end{align*}
Thus $\JJ_J$ defines a generalized complex structure on $V$.
\end{ex}

\begin{ex} \label{e:sympGC} Let $\gw \in \Lambda^2 V^*$ be a symplectic 
structure on $V$.  Observe that
$\gw$ defines a canonical map $\gw: V \to V^*$ using interior product, i.e
$\gw(v) := i_v \gw$.  By assumption this map is invertible, and we denote the
inverse map as $\gw^{-1} : V^* \to V$.  Now set
\begin{align*}
 \JJ_{\gw} := \left(
 \begin{matrix}
  0 & - \gw^{-1}\\
  \gw & 0
 \end{matrix}
\right).
\end{align*}
The equation $\JJ_{\gw}^2 = -\Id$ is immediate, and we can furthermore
compute
\begin{align*}
 \IP{\JJ_{\gw} (X + \xi), \JJ_{\gw} (Y + \eta)} =&\ \IP{- \gw^{-1} \xi +
\gw(X), - \gw^{-1} \eta + \gw(Y)}\\
 =&\ \tfrac{1}{2} \left( - \gw(X, \gw^{-1} \eta) - \gw(Y, \gw^{-1} \xi)
\right)\\
 =&\ \tfrac{1}{2} \left( \eta(X) + \xi(Y) \right)\\
 =&\ \IP{ X + \xi, Y + \eta }.
\end{align*}
Thus $\JJ_{\gw}$ defines a generalized complex structure on $V$.
\end{ex}

\begin{ex}  Given vector spaces $V_i$ with generalized complex structures 
$\JJ_i$, $i=1,2$, one defines the direct product $\JJ = \JJ_1 \oplus \JJ_2$.  A 
simple check shows that this defines a generalized complex structure on $V_1 
\oplus V_2$.  Using this construction in conjunction with the two previous 
examples, one immediately obtains generalized complex structures which are not 
determined solely by a complex or symplectic structure.
\end{ex}

\begin{ex} \label{e:GCSbtrans}  An important generalization of these basic examples exploits $b$-field symmetries. Let $V$ be a vector space with generalized 
complex structure $\JJ$.  Given $b \in \Lambda^2(V^*)$, let $\JJ_b := e^b \JJ 
e^{-b}$.  An elementary calculation shows that $\JJ_b^2 = - \Id$.  Moreover, 
since $b$-field transformations are orthogonal for $\IP{,}$ by Lemma 
\ref{l:Bfldorth}, as is $\JJ$, it follows immediately that $\JJ_b$ is 
orthogonal with respect to $\IP{,}$, showing that $\JJ_b$ defines a new generalized complex structure. Observe that the different choices of $b \in \Lambda^2(V^*)$ can be regarded as different isotropic splittings of the sequence \eqref{eq:linearsurjection2} and therefore $\JJ_b$ is, in a sense, equivalent to $\JJ$.
\end{ex}

\begin{rmk} \label{r:GCevendim}  Since every even-dimensional vector space $V$ admits both complex and symplectic structures, it follows from the above examples that they admit generalized complex structures as well.  In fact, if $V$ admits a generalized complex structure $\JJ$, then $V$ must be even dimensional.  To see this, fix $X + \xi \in V \oplus V^*$ a vector which is null for the neutral inner product.  As $\JJ$ is a complex structure on $V \oplus V^*$ orthogonal with respect to the neutral inner product, it follows that $\JJ (X + \xi)$ is null, and orthogonal to $X + \xi$.  These two vectors span a $\JJ$-invariant isotropic plane $S$, and so we can choose a new vector $X' + \xi'$ orthogonal to this plane and obtain that $\JJ (X' + \xi')$ is null, and orthogonal to $S$ and $X' + \xi'$.  Repeating this inductively we eventually construct a maximal isotropic subspace, which has even dimension.  Since the inner product has signature $(n,n)$ with $n=\dim V$, it follows that $V$ is even dimensional.
\end{rmk}

Using this fact we can give the characterization of generalized complex structures in terms of reduction of structure groups. Given $W$ as in \eqref{eq:linearsurjection2} we denote by $O(2n,2n) := O(W)$ the group of linear isometries of $W$.

\begin{lemma} \label{l:gencompstructuregroup}  
If $W$ as in \eqref{eq:linearsurjection2} admits a generalized complex structure, then $\dim V = 2n$. Furthermore, a generalized complex structure on $W$ is equivalent to  a choice of maximal compact subgroup $U(n,n)  \leq O(2n,2n)$.
\begin{proof} 
The first part follows from Remark \ref{r:GCevendim} by identifying $W = V \oplus V^*$ via a choice of isotropic splitting of \eqref{eq:linearsurjection2}. As for the second part, endomorphisms which preserve $\JJ$ will lie in $GL(2n, \mathbb C)$, and so a choice of generalized complex structure on $W$ determines a maximal compact $O(2n,2n) \cap GL(2n,\mathbb C) = U(n,n)$.
\end{proof}
\end{lemma}

\subsection{Linear Dirac structures} \label{ss:linear_dirac}

To expand upon the examples above, we will show that a generalized complex 
structure is characterized in terms of Dirac structures. For simplicity, we will assume that $W = V \oplus V^*$ with $\IP{,}$ given by \eqref{Diracpairings}.

\begin{defn} Given $V$ a vector space, a \emph{Dirac structure} on $V$ is a 
subspace $L \subset V \oplus V^*$ which is maximally isotropic for the pairing 
$\IP{,}$.
\end{defn}

\begin{ex} Both $V$ and $V^*$ are clearly isotropic for $\IP{,}$, and have 
the 
maximal possible dimension for an isotropic subspace, $\dim V$, and hence 
define 
Dirac structures.
\end{ex}

\begin{ex} Fix $U \subset V$ a subspace, and consider $L = U \oplus \Ann(U) 
\subset V \oplus V^*$, where $\Ann(U)$ denotes the annihilator of $U$ in $V^*$.  It follows directly by construction that $L$ is isotropic.  Since $\dim \Ann(U) 
= \codim U$, it follows that $L$ has the maximal dimension, and so defines a Dirac 
structure.
\end{ex}

\begin{ex} Let $b : V \to V^*$ be a linear map such that the associated element 
of $V^* \otimes V^*$, also denoted $b$, is skew-symmetric.  Then $\grapho(b) 
\subset V \oplus V^*$ is isotropic since for $X, Y \in V$,
\begin{align*}
\IP{X + b(X, \cdot), Y + b(Y, \cdot)} = \tfrac{1}{2} \left( b(X,Y) + b(Y,X) 
\right) = 0.
\end{align*}
Since the $\dim(\grapho(b)) = \dim V$, it follows that $\grapho(b)$ is a Dirac 
structure.
\end{ex}

\begin{ex} \label{e:Diracgraphs} The two previous examples can be unified by the following 
construction, which yields all possible Dirac structures.  Fix $U \subset V$ 
any subspace, fix $\phi \in \Lambda^2 U^*$, and let
\begin{align*}
L(U,\phi) = \left\{ X + \xi \in U \oplus V^*\ |\ \xi_{|U} = \phi(X, \cdot) 
\right\}.
\end{align*}
Analogous to the case of $\grapho(b)$ above, we observe that for $X + \xi, Y + \eta \in L(U,\phi)$,
\begin{align*}
\IP{X + \xi, Y + \eta} = \tfrac{1}{2} \left( \xi(Y) + \eta(X) \right) = 
\tfrac{1}{2} \left( \phi(X,Y) + \phi(Y,X) \right) = 0.
\end{align*}
Hence $L(U,\phi)$ is isotropic.  Note that for each $X$, the $\xi$ such that $X + \xi \in L(U, \phi)$ are 
prescribed on $U$, but arbitrary otherwise, and so the space of such $\xi$ has dimension equal 
to $\codim (U)$.  Thus $\dim(L) = \dim(U) + \codim(U) = \dim(V)$, and so $L$ is 
maximal.
\end{ex}

\begin{prop} \label{p:isotropicasgraph} Every maximal isotropic subspace $L \subset V \oplus V^*$ is of 
the form $L(U, \phi)$.
\begin{proof} Fix $L$ a maximal isotropic and let $U = \pi_V L$, where $\pi_V$ 
denotes the canonical projection onto $V$.  Since $L$ is isotropic, it follows 
that $L \cap V^* = \Ann(U)$.  Since in general one has a canonical 
identification $U^* = V^* / \Ann(U)$, we may define a map $\phi : U \to U^*$ via
\begin{align*}
\phi(e) = \pi_{V^*} \left( \pi_V^{-1}(e) \cap L \right) \in V^* / \Ann(U).
\end{align*}
To check that this is well-defined, fix $e + \eta, e + \mu \in L$, so that $\pi_V(e + \eta) = \pi_V(e + \mu) = e$.  It follows that $\eta - \mu \in L \cap V^*$, thus $\pi_{V^*}(\eta - \mu) \in \Ann(U)$, finishing the claim.  Also the map is skew-symmetric with respect to the neutral inner product since, using that $L$ is isotropic,
\begin{align*}
& \IP{ \pi_{V^*} (X + \eta), \pi_V(Y + \mu)} + \IP{ \pi_{V^*} (Y + \mu), \pi_V(X + \eta)}\\
&\ \quad = \eta(Y) + \mu(X) = \IP{X + \eta, Y + \mu} = 0.
\end{align*}
The tensor $\phi$ is canonically identified with a section of $\Lambda^2 U^*$, and it follows from the construction that $L = L(U, \phi)$.
\end{proof}
\end{prop}

We give next our characterization of generalized complex structures in terms of complex Dirac structures on $(V\oplus V^*) \otimes \mathbb C$.

\begin{prop} \label{p:DiracGCS}  A generalized complex structure on $V$ is
equivalent to a choice of maximal isotropic complex subspace $L \subset (V
\oplus V^*) \otimes \mathbb C$ satisfying $L \cap \bar{L} = \{0\}$.
\begin{proof} Given $\JJ$ a generalized complex structure, let $L$ be the
$\i$-eigenspace of the induced operator on $(V \oplus V^*) \otimes \mathbb C$. 
To show that $L$ is isotropic, fix $x,y \in L$ and use the orthogonality
property of $\JJ$ to obtain
\begin{align*}
 \IP{x,y} = \IP{\JJ x, \JJ y} = \IP{ \i x, \i y} = - \IP{x,y}.
\end{align*}
Since $\JJ$ is a real endomorphism, it follows immediately that $\bar{L}$ is the
$-\i$-eigenspace for $\JJ$, implying moreover that $L \cap \bar{L} = \{0\}$, and
that $L$ is maximal.  Conversely, given such an $L$ one defines $\JJ$ as
multiplication by $\i$ on $L$ and $-\i$ on $\bar{L}$.  An elementary calculation
shows that this defines a real endomorphism, which by construction certainly
satisfies $\JJ^2 = - \Id$.  Moreover, to check orthogonality, since $L$
is isotropic and an eigenspace for $\JJ$ it suffices to check orthogonality for
$x \in L, y \in \bar{L}$, in which case
\begin{align*}
\IP{\JJ x, \JJ y} = \IP{ \i x, - \i y} = \IP{x,y},
\end{align*}
as required.
\end{proof}
\end{prop}

\subsection{Spinor formulation}\label{ssec:spinors}

Generalized complex structures can also be elegantly described in terms of spinors. This point of view will also be useful in dealing with integrability issues for generalized complex structures on manifolds. We fix a vector space $V$ and consider the pairing $\IP{,}$ on $V \oplus V^*$ given by \eqref{Diracpairings}.

\begin{defn} Let $\CL(V \oplus V^*)$ denote the Clifford algebra defined by the relation
\begin{align*}
v \cdot v = \IP{v,v}, \qquad v \in V \oplus V^*.
\end{align*}
Furthermore define
\begin{align*}
\Spin(V \oplus V^*) = \left\{ v_1 \dots v_k\ |\ v_i \in V \oplus V^*,\ \IP{v_i, v_i} = \pm 1, \ k \mbox{ even} \right\}.
\end{align*}
The action of $V \oplus V^*$ on the exterior algebra $\Lambda^*V^*$, given by
\begin{align}\label{eq:Clifproduct}
(X + \xi) \cdot \rho = i_X \rho + \xi \wedge \rho,
\end{align}
extends to a natural action of $\CL(V \oplus V^*)$ on $\Lambda^* V^*$.
\end{defn}

\begin{lemma} The operation \eqref{eq:Clifproduct} defines an algebra representation of $\CL(V \oplus V^*)$ on $\Lambda^* V^*$.
\begin{proof} We compute
\begin{align*}
(X + \xi) \cdot (X + \xi) \cdot \rho =&\ i_X (i_X \rho + \xi \wedge \rho) + \xi \wedge (i_X \rho + \xi \wedge \rho)\\
=&\ \left(i_X \xi \right) \rho\\
=&\ \IP{X + \xi, X + \xi} \rho,
\end{align*}
as required.
\end{proof}
\end{lemma}

As a spin representation, the exterior algebra decomposes into a direct sum
$$
\Lambda^* V^* = \Lambda^{\mbox{\tiny{even}}}V^* \oplus \Lambda^{\mbox{\tiny{odd}}}V^*
$$
corresponding to irreducible representations.  Note that we have the double cover homomorphism 
$$
\psi \colon \Spin(V \oplus V^*) \to \operatorname{SO}(V \oplus V^*)
$$
defined by 
$$
\psi(x)(v) = x \cdot v \cdot x^{-1},
$$
for $x \in \Spin(V \oplus V^*)$ and $v \in V \oplus V^*$.  Using this map, a Lie algebra element of the form 
\begin{equation}\label{eq:zetaLie}
\zeta = \left(
 \begin{matrix}
  A & 0\\
  B & - A^*
 \end{matrix} 
\right) \in \mathfrak{so}(V \oplus V^*)
\end{equation}
acts on $\Lambda^* V^*$ via
\begin{align*}
\zeta \cdot \rho = - B \wedge \rho - A^* \rho + \tfrac{1}{2}\tr A \rho.
\end{align*}
Here, $A \in \End(V)$ and $B \in \Lambda^2V^*$, and we have used the natural identification 
$$
\mathfrak{so}(V \oplus V^*) = \Lambda^2(V \oplus V^*) = \End(V) \oplus \Lambda^2 V^* \oplus \Lambda^2 V.
$$
Observe that elements of the form \eqref{eq:zetaLie} are precisely those which preserve the flag \eqref{eq:linearsurjection2}, and will appear later as infinitesimal Courant algebroid automorphisms (see Proposition \ref{p:Courantsymmetries}). 

Exponentiating $\zeta$ above we obtain spin group elements corresponding to $B$-field transformations, given by
$$
e^B \cdot \rho = e^{-B} \wedge \rho = (1 - B + \tfrac{1}{2}B \wedge B + \ldots) \wedge \rho.
$$
We also obtain spin group elements corresponding to linear automorphisms $A \in GL^+(V) \subset GL(V)$ with positive determinant,
$$
A \cdot \rho = \sqrt{\det A} (A^{-1})^* \rho.
$$
This indicates that, as a $GL^+(V)$ representation, the spinors decompose as
$$
S = \Lambda^* T^* \otimes (\det V )^{1/2}.
$$
We turn next to the study of linear Dirac structures on $V$ in terms of spinors.

\begin{defn} Given $\rho \in \Lambda^* V^*$, define the associated \emph{null space} by
\begin{align*}
L_{\rho} := \left\{ v \in V \oplus V^*\ |\ v \cdot \rho = 0 \right\}.
\end{align*}
Note that $L_{\rho}$ is isotropic, since for $v_1, v_2 \in L_{\rho}$, we have
\begin{align*}
\IP{v_1, v_2} \rho = \tfrac{1}{2} \left(v_1 v_2 + v_2 v_1 \right) \cdot \rho = 0,
\end{align*}
thus $\IP{v_1, v_2} = 0$.  The spinor $\rho$ is called \emph{pure} if $L_{\rho}$ is maximal, that is, has dimension equal to $\dim V$.
\end{defn}

The next result shows that every maximal isotropic subspace $L \subset V \oplus V^*$ is represented by a unique line $K_L \subset \Lambda^* V^*$ of pure spinors.

\begin{prop}\label{p:spinorline}
Every maximal isotropic subspace $L = L(U, \phi)$ is determined by a pure spinor line $K_L \subset \Lambda^* V^*$. If $\theta_1, \ldots, \theta_k$ is a basis of $\Ann(U)$ and $B \in \Lambda^2V^*$ is such that $i^*B = - \phi$, where $i \colon U \to V$ is the inclusion, then the pure spinor line $K_L$
representing $L(U,\phi)$ is generated by
$$
\rho = e^{B}\theta_1 \wedge \ldots \wedge \theta_k.
$$
The integer $k$ is called the type of the maximal isotropic.
\begin{proof} 
By Proposition \ref{p:isotropicasgraph}, any maximal isotropic $L = L(U,\phi)$ may be expressed as the $B$-field transform of $L(U,0)$ for $B$ as in the statement. It is not difficult to see that the pure spinor line with null space $L(U,0)$ is precisely $\det \Ann (U)$. Thus, the statement follows from
$$
K_L = K_{e^{-B}L(U,0)} = e^{B}K_{L(U,0)} = e^B \det \Ann (U).
$$
We refer to (\cite{Chevalley} Chapter III) for further details.
\end{proof}
\end{prop}

Consider now a generalized complex structure $\JJ$ on $V$ and the associated complex Dirac structure (see Proposition \ref{p:DiracGCS})
$$
L \subset (V \oplus V^*) \otimes \mathbb{C}.
$$ 
Regarding $\Lambda^* V^* \otimes \mathbb{C}$ as a representation of the complex Clifford algebra of $V \oplus V^*$ via \eqref{eq:Clifproduct}, the analogue of Proposition \ref{p:spinorline} applies, and we obtain a complex pure spinor line $K_L \subset \Lambda^* V^* \otimes \mathbb{C}$ generated by
$$
\rho = e^{B + \sqrt{-1} \omega} \theta_1 \wedge \ldots \wedge \theta_k
$$
for $B,\omega \in \Lambda^2 V^*$ and $\theta_1, \ldots, \theta_k$ a basis of $L \cap (T^* \otimes \mathbb{C})$. To provide a convenient characterization of the condition 
$$
L \cap \overline{L} = \{0\}
$$
in Proposition \ref{p:DiracGCS}, let us recall the definition of the Mukai pairing,
$$
(,) \colon \Lambda^* V^* \otimes \Lambda^* V^* \to \det V^*.
$$
Given $\alpha, \beta \in \Lambda^* V^*$ this is defined by
$$
(\alpha,\beta) = (\alpha^T \wedge \beta)_{\mbox{\tiny{top}}},
$$
where $\alpha \to \alpha^T$ is the antiautomorphism of the Clifford algebra determined by the tensor map $v_1 \otimes \ldots \otimes v_k \to v_k \otimes \ldots \otimes v_1$ and $()_{\mbox{\tiny{top}}}$ indicates taking the top degree component of the form. The Mukai pairing extends $\mathbb{C}$-linearly to $\Lambda^* V^* \otimes \mathbb{C}$ and we have the following result, which we state without proof (see \cite{Chevalley}).

\begin{lemma}\label{l:Mukai}
Maximal isotropics $L,L' \subset  (V \oplus V^*) \otimes \mathbb{C}$ satisfy $L \cap L' = \{0\}$ if and only if their pure spinor representatives $\rho, \rho'$ satisfy
$$
(\rho , \rho') \neq 0.
$$
\end{lemma}

The following characterization of linear generalized complex structures is now a direct consequence of Proposition \ref{p:DiracGCS} and Lemma \ref{l:Mukai}. As a matter of fact, this was the original point of view on generalized Calabi-Yau manifolds taking by Hitchin in his seminal paper \cite{HitchinGCY}.

\begin{prop} \label{p:spinorGCS}  
A generalized complex structure on $V$ is equivalent to a choice of complex pure spinor line $K \subset \Lambda^* V^* \otimes \mathbb{C}$ such that any pure spinor representative $\rho$ satisfies
$$
(\rho , \overline{\rho}) \neq 0.
$$
\end{prop}

\section{Generalized complex structures on manifolds}\label{s:GCGmanifold}

Having defined the appropriate linear algebraic structure in the previous 
section, we now extend this definition to manifolds following a common two-step 
process.  First we define ``almost generalized complex structures'' as smooth
sections of $\End(T \oplus T^*)$ satisfying the appropriate algebraic 
conditions.  We then define a natural notion of 
integrability for such structures, and discuss how classical integrability 
conditions for complex and symplectic structures fit into this picture.
For the most recent developments on the subject we refer to \cite{bailey2019neighbourhood,Bailey_2017,Cavalcanti_2017} and references therein.

\subsection{Almost generalized complex structures} \label{ss:aGCs}

\begin{defn} Given $M^{2n}$ a smooth manifold, an \emph{almost generalized 
complex structure} on $M$ is a section $\JJ \in
\gG(\End(T \oplus T^*))$ such that $\JJ_p$ defines a linear generalized complex 
structure on
$T_p \oplus T^*_p$ for all $p \in M$.  More generally, given an exact Courant algebroid $E$ over $M$, an \emph{almost generalized complex structure} on $E$ is a section $\JJ \in \gG(\End E)$ such that $\JJ_p$ defines a linear generalized complex structure on $E_p$ for all $p \in M$.
\end{defn}

Following Lemma \ref{l:gencompstructuregroup}, we can regard an almost generalized complex structure on $M$ as a reduction of the $O(2n,2n)$-bundle of frames of the generalized tangent bundle to the maximal compact $U(n,n)$. Alternatively, by Proposition \ref{p:DiracGCS}, $\JJ$ is equivalent to a maximal isotropic subbundle $L \subset (T \oplus T^*) \otimes \mathbb C$ satisfying $L \cap \bar{L} = \{0\}$. Yet another point of view on this structure is provided by spinors. By the discussion in \S \ref{ssec:spinors}, the spinor bundle  for $T \oplus T^*$ corresponds to
$$
S = \Lambda^*T^* \otimes |\det T|^{1/2},
$$
where $|\det T|^{1/2}$ denotes the line bundle of half-densities on $M$. Since $|\det T|^{1/2}$ is trivializable, it will be more convenient to work with the twisted spinor bundle
$$
\mathbb{S} = \Lambda^*T^*
$$
endowed with the Clifford action \eqref{eq:Clifproduct}.  By the previous discussion, and also Proposition \ref{p:spinorGCS}, we obtain the following result giving different formulations of the notion of almost generalized complex structure.

\begin{prop} \label{p:GCequivalences} Given $M^{2n}$ a smooth manifold, a choice
of almost generalized complex structure $\JJ$ is equivalent to
\begin{itemize}
 \item A reduction of structure for the $O(2n,2n)$-bundle $(T \oplus T^*,
\IP{,})$ to the group $U(n,n)$.
 \item A maximal isotropic subbundle $L \subset (T \oplus T^*) \otimes \mathbb
C$ satisfying $L \cap \bar{L} = \{0\}$.
 \item A pure spinor line sub-bundle $K \subset \Lambda^{*} T^* \otimes \mathbb
C$ such that any pure spinor representative $\rho$ satisfies
$$
(\rho , \overline{\rho}) \neq 0 \in \det T^*.
$$
\end{itemize} 
\end{prop} 

The previous result generalizes the fundamental fact that a complex structure can be thought of as a tensor field, a reduction of structure groups, or in terms of certain subbundles. Notice that $(\rho , \overline{\rho}) \in \det T^*$ defines an orientation independent of the
choice of $\rho$, giving a global orientation on the manifold.

\begin{ex} \label{e:complexaGC} Continuing the discussion in Example \ref{e:compGC}, suppose $M^{2n}$ 
admits an almost complex structure $J$, and then define an almost generalized 
complex structure on $M$ via
\begin{align*}
\JJ_J =&\ \left(
\begin{matrix}
- J & 0\\
0 & J^*
\end{matrix}
\right).
\end{align*}
It is clear by definition that the $\i$-eigenbundle takes the form
\begin{align*}
L = T^{0,1} \oplus T^*_{1,0}.
\end{align*}
The pure spinor line in Proposition \ref{p:GCequivalences} corresponds to $\Lambda^{n,0}$ and is generated, locally, by a choice of non-vanishing $(n,0)$-form.
\end{ex}

\begin{ex} \label{e:symplecticaGC} Continuing the discussion in Example \ref{e:sympGC}, suppose $M$ 
admits an almost symplectic form $\gw$, and then define an almost generalized 
complex structure on $M$ via
\begin{align*}
\JJ_{\gw} =&\ \left(
\begin{matrix}
0 & - \gw^{-1}\\
\gw & 0
\end{matrix}
\right).
\end{align*}
Given $X \in T \otimes \mathbb C$, note that
\begin{align*}
\JJ_{\gw} \left( 
\begin{matrix}
X\\
- \i \gw(X)
\end{matrix}
\right) = \left( \begin{matrix}
\i X\\
\gw(X)
\end{matrix} \right),
\end{align*}
and hence such vectors lie in the $\i$-eigenspace of $\JJ_{\gw}$.  Since the  complex dimension of the space of such vectors is $n$, it follows that such 
vectors determine the entire $\i$-eigenspace $L$, i.e.
$$
L = e^{-\sqrt{-1}\omega}T.
$$
Observe that $L$ is given by a B-field transformation of the tangent bundle, which is a Dirac structure with pure spinor line generated by $1 \in (\Lambda^* T^*) \otimes \mathbb{C}$. Thus, by naturality of the Clifford action with respect to orthogonal transformations it follows that the pure spinor line in Proposition \ref{p:GCequivalences} is generated in this case by the global section
$$
e^{-\sqrt{-1}\omega} \cdot 1 = e^{\sqrt{-1}\omega} \in (\Lambda^* T^*) \otimes \mathbb{C}.
$$
\end{ex}

\subsection{Integrability} \label{ss:integrable}

We now turn to the question of integrability of almost generalized complex 
structures.  Recall that the classical definition of integrability of an almost 
complex structure $J$ asks for the $\i$-eigenbundle to be closed under the Lie 
bracket.  For integrability of generalized complex structures we will ask for involutivity of the $\i$-eigenbundle under the Dorfman bracket.  As this eigenbundle is an example of an almost Dirac structure by Proposition \ref{p:DiracGCS}, we briefly work in greater generality and record fundamental definitions and results concerning
Dirac structures on exact Courant algebroids.  Our exposition follows \cite{CourantDirac,GualtieriThesis} closely.

\begin{defn} \label{d:almostDirac} 
Let $E$ be an exact Courant algebroid over a smooth manifold $M$. An \emph{almost Dirac structure} on $E$ is a 
subbundle $L \subset E$ which is maximally isotropic with respect to the inner product $\IP{,}$.
We say that $L$ is an \emph{integrable Dirac structure} if $L$ is closed under the Dorfman bracket, i.e.
\begin{align*}
[L,L] \subset L.
\end{align*}
\end{defn}

\begin{rmk}
It is important to note that, since $L$ is isotropic with respect to the symmetric pairing, it follows from (\ref{f:CDrelation}) that the Courant and Dorfman brackets agree on sections of $L$.  Thus it is equivalent to ask that $L$ is closed under the Courant bracket.  To save on notation we will refer to brackets of sections of $L$ as $[e_1, e_2]$, where this can be equivalently chosen as either the Courant or Dorfman bracket.
\end{rmk}

This integrability condition admits a tensorial characterization, which helps in
verification.  To begin we make a definition.

\begin{defn} Given $L$ an almost Dirac structure on $E$, define a map
$\mathbf{T}_L : \gG \left( L^{\otimes 3} \right) \to \mathbb R$ via
\begin{align*}
\mathbf{T}_L (e_1,e_2,e_3) :=&\ \IP{ [e_1,e_2],e_3}.
\end{align*}
\end{defn}

\begin{lemma} \label{l:tensorDirac} Given $L$ an almost Dirac structure on $E$, 
$\mathbf{T}_L \in \gG
\left( \Lambda^2 L^* \otimes L^* \right)$.
\begin{proof} We must check that the map is $C^{\infty}(M)$-linear.  It is
obviously $C^{\infty}(M)$-linear in the third position, and due to the skew-symmetry in the first two indices it suffices to check $C^{\infty}$-linearity in
either.  To that end fix $\{e_i\} \in \gG(L)$ and $f \in C^{\infty}(M)$, and
compute using Lemma \ref{l:Courant},
\begin{align*}
\mathbf{T}_L (e_1, f e_2, e_3) =&\ \IP{ [e_1, f e_2], e_3}\\
=&\ \IP{ f [e_1, e_2] + (Xf) e_2 - \IP{e_1,e_2} df, e_3}\\
=&\ f T_L (e_1, e_2, e_3),
\end{align*}
where the last line follows using that $L$ is isotropic.  The lemma follows.
\end{proof}
\end{lemma}

\begin{prop} \label{p:tensorDiracinteg} 
Given $L$ an almost Dirac structure, 
$L$ is integrable if and only
if $\mathbf{T}_L = 0$.
\begin{proof} If $L$ is closed under the Courant bracket, then since $L$ is
isotropic it follows directly that $\mathbf{T}_L = 0$.  Conversely, suppose
$\mathbf{T}_L = 0$.  Then given any $e_1, e_2 \in \gG(L)$, it follows by
hypothesis that $[e_1,e_2] \in L^{\perp}$.  But since $L$ is maximal isotropic,
it follows that $L^{\perp} = L$, and so $L$ is closed under the bracket.
\end{proof}
\end{prop}

Furthermore, we can characterize the integrability of Dirac structures in terms of the description in Proposition \ref{p:isotropicasgraph}.

\begin{prop} \label{p:integrabilityofgraphs} 
Given $M$ a smooth manifold and $H \in \Lambda^3 T^*$, $dH = 0$, consider the exact Courant algebroid $(T\oplus T^*, \IP{,}, [,]_H)$. A Dirac structure $L(U, \phi)$ is integrable if and only if
\begin{enumerate}
\item $U$ is closed under the Lie bracket
\item $i_Y i_X (H + d\phi) = 0$ for all $X,Y \in \gG(U)$.
\end{enumerate}
\begin{proof} Following the main computation of Proposition \ref{p:bfieldbracket}, we obtain that
\begin{align*}
[X + i_X \phi, Y + i_X \phi] = e^{\phi} [X,Y] + i_Y i_X \left( H + d \phi \right).
\end{align*}
The result follows in a straightforward way from this equation.
\end{proof}
\end{prop}

With these preliminaries in hand we can give the definition of integrability of an almost generalized complex structure. Following the previous discussion, we do this in the generality of arbitrary exact Courant algebroids.

\begin{defn}\label{d:GCSabstract}
Let $E$ be an exact Courant algebroid over a smooth manifold $M$.  An almost generalized complex structure $\JJ$ on $E$ is \emph{integrable} if the $\i$-eigenbundle $L \subset E \otimes \mathbb{C}$ is an integrable Dirac structure.  An integrable almost generalized complex structure $\JJ$ on $E$ will be called a generalized complex structure.
\end{defn}

In particular, an almost generalized complex structure on a smooth manifold is \emph{integrable} if the $\i$-eigenbundle $L \subset (T \oplus T^*) \otimes \mathbb{C}$ is an integrable Dirac structure for the standard Dorfman bracket. It is not difficult to see that the analogue of Proposition \ref{p:GCequivalences} holds for almost generalized complex structures on exact Courant algebroids. We leave the details as an \textbf{exercise}.

\begin{ex} \label{e:complexGC} Continuing the discussion in Example \ref{e:complexaGC}, suppose $M$ 
admits an almost complex structure $J$, and consider the almost generalized complex structure $\JJ_J$. The $\i$-eigenbundle takes the form
\begin{align*}
L = T^{0,1} \oplus T^*_{1,0},
\end{align*}
and it directly follows from this that $L$ is closed under Courant 
bracket if and only if $T^{0,1}$ is closed under the Lie bracket, which is in 
turn equivalent to the classical integrability of $J$.
\end{ex}

\begin{ex} \label{e:symplecticGC} Continuing the discussion in Example \ref{e:symplecticaGC}, suppose $M$ 
admits an almost symplectic form $\gw$, and consider the almost generalized 
complex structure $\JJ_{\gw}$. Fix smooth complex vector field $X, Y \in T \otimes 
\mathbb C$ and observe, using an obvious extension of the proof of Proposition 
\ref{p:bfieldbracket} to the case of complex coefficients,
\begin{align*}
[X - \i \gw(X), Y - \i \gw(Y)] =&\ [ e^{- \i \gw}(X), e^{- \i \gw}(Y)] \\
=&\ e^{-\i 
\gw} [X,Y] + \i i_Y i_X d \gw.
\end{align*}
It follows that $L$ is integrable if and only if $d \gw = 0$.
\end{ex}

It is obvious from Definition \ref{d:GCSabstract} that if $\JJ$ is an integrable generalized complex structure on an exact Courant algebroid $E$, then conjugation by a Courant algebroid automorphism provides a new generalized complex structure. We state this for emphasis:

\begin{lemma} \label{l:BfieldactionGC} 
Let $\JJ$ be a generalized complex structure on an exact Courant algebroid $E$. Then, for any Courant algebroid automorphism $F \colon E \to E$, $F \circ \JJ \circ F^{-1}$ is also a generalized 
complex structure.
\end{lemma}

\begin{ex} \label{e:BfieldactionGC} 
Let $\JJ$ be an integrable generalized 
complex structure on $M$ and suppose $B \in \Lambda^2 T^*$, $d B = 0$.  Then 
$\JJ_B := e^{B} \JJ e^{-B}$ is a generalized complex structure by Proposition \ref{p:gendiff} and Lemma \ref{l:BfieldactionGC}.
\end{ex}

The prior class of examples together indicate the naturality of the circle of 
definitions yielding the concept of generalized complex structure, as it 
beautifully encodes complex and symplectic structures 
both algebraically and in terms of integrability.  To further indicate the 
manner in which generalized complex structures unify complex and symplectic 
structure, we give an example of a one-parameter family of generalized complex
structures connecting one of complex type to one of symplectic type.

\begin{ex} Fix $(M^{4n}, g, I, J, K)$ a hyperK\"ahler manifold.  Using the 
notation of the previous examples, consider $\JJ_I$ and $\JJ_{\gw_J}$, where 
$\gw_J$ denotes the K\"ahler form associated to $(g,J)$.  These are 
generalized complex structures by the previous examples. Now consider the one-parameter family
\begin{align*}
\JJ_t = \sin t \JJ_I + \cos t \JJ_{\gw_J}, \qquad t \in [0,\tfrac{\pi}{2}].
\end{align*}
From the quaternion relation $IJ = - JI$ one can check directly that the 
endomorphisms $\JJ_I$ and $\JJ_{\gw_J}$ anticommute, after which it directly 
follows that $\JJ_t^2 = - \Id$.  It remains to check the integrability of 
$\JJ_t$.  To that end let $B_t = (\tan t) \gw_K$, which is a closed form defined 
for $t \in [0,\tfrac{\pi}{2})$.  A calculation using the hyperK\"ahler 
conditions shows that
\begin{align*}
\JJ_t =&\ e^{-B} \left(
\begin{matrix}
0 & - (\sec t \gw_J)^{-1}\\
\sec t \gw_J & 0
\end{matrix} \right) e^B,
\end{align*}
thus for $t < \tfrac{\pi}{2}$, $\JJ_t$ is expressed as a $B$-field 
transformation of an integrable generalized complex structure, and is hence 
integrable by Lemma \ref{l:BfieldactionGC}.
\end{ex}

Following the line of thought in Proposition \ref{p:GCequivalences}, we close this section with the characterization of integrability for almost generalized complex structures in terms of spinors. The preliminaries for this result amount to expressing the Dorfman bracket as a \emph{derived bracket}. Let us consider the general setup of an exact Courant algebroid $E$ over a smooth manifold $M$. Consider the Clifford bundle $CL(E)$ given by the relation
$$
v\cdot v = \langle v,v \rangle.
$$
For a choice of isotropic splitting $\sigma \colon T \to E$, inducing an isomorphism $E \cong T \oplus T^*$ with twisted bracket $[,]_H$ for some closed three-form $H$, the bundle of polyforms 
$$
\mathbb{S} = \Lambda^* T^*
$$ 
is an irreducible Clifford module for $CL(E)$, via the Clifford action
\begin{equation*}
    (X+\xi)\cdot \rho =i_{X}\rho+\xi\wedge\rho.
\end{equation*}
For a different choice of isotropic splitting $\sigma'$, we can write $\sigma' = e^B \circ \sigma$ for $B \in \Gamma(\Lambda^2T^*)$, and one has
\begin{equation}\label{eq:spinoraut}
(e^{-B}(X+\xi))\cdot e^B\wedge \rho  = e^B \wedge ((X+\xi)\cdot \rho).
\end{equation}
Thus, the definition of $\mathbb{S}$ is independent of the choice of $\sigma$. Consider the twisted de Rham differential 
$$
d_H \colon \Gamma(\mathbb{S}) \to \Gamma(\mathbb{S})
$$
acting on the space of polyforms by
\begin{equation}\label{eq:twisteddiff}
d_H \rho = d\rho - H \wedge \rho.
\end{equation}
Using formula \eqref{eq:spinoraut}, it is an \textbf{exercise} to check that $d_H$ is independent of the choice of isotropic splitting, defining a canonical map 
$$
\slashed d_0 \colon \Gamma(\mathbb{S}) \to \Gamma(\mathbb{S}).
$$
Via the decomposition of the exterior algebra as a spin representation, the spinor bundle $\mathbb{S}$ inherits a $\mathbb{Z}_2$-grading
$$
\mathbb{S} = \Lambda^{\mbox{\tiny{even}}}T^* \oplus \Lambda^{\mbox{\tiny{odd}}}T^*.
$$
With respect to this grading, the operator $\slashed d_0$ has odd degree, as does the operator given by Clifford multiplication $\alpha \to v \cdot \alpha$ for any element $v \in E$. Recall the definition of the super-commutator of operators on a $\mathbb{Z}_2$-graded vector space
$$
[L_1,L_2] = L_1L_2 - (-1)^{\deg L_1 \deg L_2}L_2 L_1.
$$
Similar to how the usual de Rham differential is related to the Lie bracket, we have the following formula for the Dorfman bracket.

\begin{lemma}\label{l:derivedbracket}
For any $v_1,v_2 \in \Gamma(E)$ and $\rho \in \Gamma(\mathbb{S})$, we have
$$
[[\slashed d_0,v_1 \cdot],v_2\cdot] \rho = [v_1,v_2] \cdot \rho.
$$
\end{lemma}

\begin{proof}
Choose an isotropic splitting $\sigma \colon T \to E$ and identify $E \cong (T \oplus T^*,\IP{,},[,]_H)$ and $\slashed d_0 = d_H$. Then, for example, setting $v_1 = X \in T$ and $v_2 = Y \in T$ we have
\begin{align*}
[[\slashed d_0,v_1 \cdot],v_2\cdot] \rho =&\  [d_H,v_1 \cdot](v_2\cdot \rho) - v_2 \cdot ([d_H,v_1 \cdot]\rho)\\
=&\ d_H(v_1 \cdot v_2 \cdot \rho) + v_1 \cdot (d_H(v_2 \cdot \rho)) - v_2 \cdot (d_H(v_1 \cdot \rho)) - v_2 \cdot v_1 \cdot d_H \rho\\
=&\ d(i_X i_Y \rho) + i_X d( i_Y \rho) - i_Y d( i_X \rho) - i_Y i_X d\rho\\
&\ - H \wedge i_X i_Y \rho - i_X (H \wedge( i_Y \rho) + i_Y (H \wedge ( i_X \rho)) + i_Y i_X (H \wedge \rho)\\
=&\ L_X (i_Y \rho)  - i_Y (L_X \rho)\\
&\ - i_XH \wedge i_Y \rho + i_Y H \wedge i_X \rho - H \wedge i_Y i_X \rho  + i_Y i_X (H \wedge \rho)\\
 =&\ i_{[X,Y]}\rho + (i_Y i_X H) \wedge \rho\\
 =&\ [v_1,v_2]_H \cdot \rho
\end{align*} 
We leave the remaining cases as an \textbf{exercise}.
\end{proof}

We are ready now to characterize the integrability of almost generalized complex structures in terms of the associated complex spinor line. We will use the third characterization of these objects from Proposition \ref{p:GCequivalences}.

\begin{prop} \label{p:GCintdH} 
Let $\JJ$ be an almost generalized complex structure on an exact Courant algebroid $E$, with pure spinor line  $K \subset \Lambda^{*} T^* \otimes \mathbb C$. Then, $\JJ$ is integrable if and only if for any local trivialization $\rho$ of $K$ there exists a local section $v \in E$ such that
\begin{equation}\label{eq:GCintdH}
\slashed d_0 \rho = v \cdot \rho.
\end{equation}
\begin{proof}
Choose an isotropic splitting $\sigma \colon T \to E$ and identify $E \cong (T \oplus T^*,\IP{,},[,]_H)$ and $\slashed d_0 = d_H$. Let $L \subset E \otimes \mathbb{C}$ be the complex Dirac structure associated to $\JJ$, as in Proposition \ref{p:GCequivalences}. Given $v_1,v_2 \in L$, applying Lemma \ref{l:derivedbracket}, we calculate
\begin{align*}
[v_1,v_2]  \cdot \rho =  - v_2 \cdot v_1 \cdot d_H \rho = (v_1 \cdot v_2 - 2 \IP{v_1,v_2}) \cdot d_H \rho = v_1 \cdot v_2 \cdot d_H \rho.
\end{align*} 
Assuming now \eqref{eq:GCintdH}, we have 
$$
[v_1,v_2]  \cdot \rho = v_1 \cdot v_2 \cdot v \cdot \rho = v_1\cdot (2 \IP{v,v_2} - v \cdot v_2 ) \cdot \rho = 0,
$$
and therefore $[v_1,v_2] \in L_\rho = L$. Conversely, if $[L,L] \subset L$, then 
\begin{equation}\label{eq:F1dH}
v_1 \cdot v_2 \cdot d_H \rho = 0, \qquad \textrm{ for any } v_1,v_2 \in L.
\end{equation}
Observe that the Dirac structure $L$ induces a filtration
$$
K_L = F_0 \subset F_1 \subset \ldots F_{2n} = \Lambda^* T^* \otimes \mathbb{C}
$$
where $n = \dim M$ and $F_k$ is the subbundle annihilated by products of $k+1$ sections of $L$. The subbundle $F_1$ decomposes in even and odd degree parts as
$$
F_1 = F_0 \oplus E \cdot F_0,
$$
and since $d_H$ is of odd degree, we see from \eqref{eq:F1dH} that $d_H \rho \in E \cdot K_L$ as claimed.
\end{proof}
\end{prop}

\subsection{Associated Poisson structures}

Generalized complex geometry turns out to have many connections to Poisson geometry.  As a first glimpse of this here we will recall the fundamental notions of Poisson geometry and one way in which they appear in generalized complex geometry.  In \S \ref{s:GKG} below we will see further interactions in the case of generalized K\"ahler geometry.

\begin{defn} \label{d:Poissondef} Given $M^n$ a smooth manifold, a \emph{Poisson bracket} on $M$ is a map
\begin{align*}
\left\{ \cdot, \cdot \right\} : C^{\infty}(M) \times C^{\infty}(M) \to C^{\infty}(M)
\end{align*}
which satisfies
\begin{enumerate}
\item $\mathbb R$-linearity: $\{a \phi, b \psi \} = ab \{\phi, \psi\}$,
\item skew-symmetry: $\{\phi,\psi\} = - \{\psi, \phi\}$,
\item Leibniz rule: $\{\phi \psi,\eta\} = \phi \{\psi,\eta\} + \psi \{\phi,\eta\}$,
\item Jacobi identity: $\{\phi, \{\psi, \eta\}\} + \{\psi, \{\eta,\phi\}\} + \{\eta,\{\phi,\psi\}\} = 0$.
\end{enumerate}
\end{defn}

\begin{ex} The most fundamental example of a Poisson bracket, and moreover the reason for defining this structure in the first place, comes from classical mechanics.  In particular on $\mathbb R^{2n}$ with variables $p_i, q^i$, one defines
\begin{align*}
\{\phi, \psi\} = \sum_{i=1}^n \left( \frac{\del \phi}{\del p_i} \frac{\del \psi}{\del q^i} - \frac{\del \phi}{\del q^i} \frac{\del \psi}{\del p_i} \right).
\end{align*}
\end{ex}

\begin{rmk} By properties $1$ and $2$ one can show that for a given function $f$, the map $\{f, \cdot\}$ is a derivation of $C^{\infty}(M)$, i.e. a vector field, and that furthermore a Poisson bracket determines a section
\begin{align*}
P \in \Lambda^2 T,
\end{align*}
which recovers the Poisson bracket via
\begin{align*}
\{\phi,\psi\} = P(d \phi, d \psi).
\end{align*}
This allows to define the rank of the Poisson structure at a point $x \in M$, as the rank of the induced map $P_x \colon T^*_x \to T_x$. In particular, for the example above one obtains
\begin{align*}
P = \sum_{i=1}^n \frac{\del}{\del p_i} \wedge \frac{\del}{\del q^i}
\end{align*}
which has constant rank $2n$. Using $P$, for any smooth function $\phi$ we can obtain the associated Hamiltonian vector field via
\begin{align*}
X_{\phi} = P(d \phi).
\end{align*}
Using these different descriptions we can obtain equivalent descriptions of the integrability condition.  In view of this lemma below we will also refer to the tensor $P$ as a Poisson structure.

\end{rmk}

\begin{lemma} \label{l:Poissonintegrable} Given a smooth manifold $M$, let $P \in \Lambda^2 T$ with associated bracket $\{\phi, \psi\} := P(d \phi, d \psi)$.  Then the following are equivalent:
\begin{enumerate}
\item For any torsion-free connection $\N$, one has
\begin{align*}
P^{l i} \N_l P^{jk} + P^{lj} \N_l P^{ki} + P^{lk} \N_l P^{ij} = 0.
\end{align*}
\item $\{, \}$ satisfies the Jacobi identity.
\item For any $\phi, \psi$, one has
\begin{align*}
X_{\{\phi, \psi\}} = [X_{\phi}, X_{\psi}].
\end{align*}
\end{enumerate}
\begin{proof} \textbf{Exercise}.
\end{proof}
\end{lemma}

\begin{prop} \label{p:realPoisson}  Given $M^{2n}$ and $\JJ$ a generalized complex structure on $M$, define $P \in \End(T^*, T) \cong T \otimes T$ via
\begin{align*}
P(\xi) = \pi \JJ \xi,
\end{align*}
where $\pi: T \oplus T^* \to T$ is the anchor map.  Then $P$ is a Poisson structure.
\begin{proof} First we check that $P$ is skew-symmetric, following from
\begin{align*}
\IP{\eta, P \xi} = \IP{\eta, \pi \JJ \xi} = \IP{\eta, \JJ \xi} = - \IP{\JJ \eta,\xi} = - \IP{\pi \JJ \eta,\xi} = - \IP{P \eta,\xi}.
\end{align*}
Next we check the integrability.  To begin we give an alternate description of $P$.  Given ${\phi} \in C^{\infty}(M)$ observe that we can uniquely express
\begin{align*}
d{\phi} = X + \xi + \bar{X} + \bar{\xi},
\end{align*}
where $X + \xi \in L$.  Observe that $X + \bar{X} = 0$.  With this we define
\begin{align*}
X_{\phi} = 2 \i X = -2 \mbox{Im}(X), \qquad \xi_{\phi} = -2 \mbox{Im}(\xi).
\end{align*}
Observe furthermore using the eigenspace decomposition that
\begin{align*}
\JJ (d\phi) =&\ \JJ \left( X + \xi + \bar{X} + \bar{\xi} \right)\\
=&\ \i \left(X + \xi \right) - \i \left( \bar{X} + \bar{\xi} \right)\\
=&\ X_{\phi} + \xi_{\phi}.
\end{align*}
It thus follows that the Poisson bracket associated to $P$ is determined by the vector fields $X_{\phi}$.  In particular, for $\phi, \psi \in C^{\infty}(M)$ we set
\begin{align*}
\{ \phi, \psi \} = X_{\phi} \psi.
\end{align*}
To prove the integrability, we will show that
\begin{align} \label{f:realP30}
X_{\{\phi,\psi\}} = [X_{\phi}, X_{\psi}],
\end{align}
from which the Jacobi identity follows using Lemma \ref{l:Poissonintegrable}.  We express
\begin{align*}
d \phi = X + \xi + \bar{X} + \bar{\xi}, \qquad d \psi = Y + \eta + \bar{Y} + \bar{\eta},
\end{align*}
as above, with $X_{\phi} = 2 \i X,\ X_{\psi} = 2 \i Y$.  Using that $L$ is involutive we see that the bracket simplifies as
\begin{align*}
Z + \mu := 2 \i [X + \xi, Y + \eta] = \frac{1}{2 \i} [X_{\phi}, X_{\psi}] + L_{X_{\phi}} \eta - i_{X_{\psi}} d \xi \in L,
\end{align*}
where the answer lies in $L$ since $L$ is involutive.  The vector field component $Z$ is pure imaginary, hence $Z + \bar{Z} = 0$, whereas
\begin{align*}
\mu + \bar{\mu} =&\ L_{X_{\phi}} \left(\eta + \bar{\eta} \right) - i_{X_{\psi}} \left(d \xi + d \bar{\xi} \right)\\
=&\ L_{X_{\phi}} d \psi - i_{X_{\psi}} d d \phi\\
=&\ d i_{X_{\phi}} d \psi\\
=&\ d \{\phi,\psi\}.
\end{align*}
Hence $(Z, \mu)$ is the canonical section of $L$ associated to the smooth function $\{\phi,\psi\}$, and so \ref{f:realP30} follows.
\end{proof}
\end{prop}

Building on this observation, we can construct a new class of examples by explicitly combining complex and Poisson structures.

\begin{ex} \label{e:PoissonGCS}  Suppose $(M^{2n}, J)$ is a complex manifold, and let $\gs \in \Lambda^{2,0} T$ denote a holomorphic Poisson structure.  Let $Q$ denote the real part of $\gs$, i.e. $\gs = JQ + \i Q$, and let
\begin{align*}
\JJ_{Q} = \left(\begin{matrix}
-J & Q\\
0 & J^*
\end{matrix}
\right).
\end{align*}
Using integrability of $J$ and the Poisson structure, it is possible to show that this defines an integrable generalized complex structure.  This example is elucidated in \cite[Proposition 5]{GualtieriBranes}, where it is moreover shown that this is a kind of canonical form for a certain type of generalized complex structure.
\end{ex}

\section{Courant algebroids and pluriclosed metrics} \label{s:GCSKT}

In this section we recall the basic notion of pluriclosed Hermitian metrics (also known as strong K\"ahler with torsion or SKT metrics).  We interpret them in the language of generalized geometry and prove some basic results on their structure.  
We adopt here the point of view taken in \cite{garciafern2018canonical,garciafern2020gauge,GualtieriGKG} (cf. also \cite{HullGCGpq}). Further results on SKT metrics and their relationship to generalized geometry appear in \cite{CavalSKT}.

\subsection{Pluriclosed metrics}

We start by recalling the classical definition of pluriclosed metrics, and some basic facts on their structure.

\begin{defn} Let $(M^{2n}, J)$ be a complex manifold.  A Riemannian
metric $g$ is called \emph{Hermitian} if
\begin{align*}
g(JX, JY) = g(X, Y)
\end{align*}
for all $X, Y \in TM$.  Associated to a Hermitian metric is the
\emph{K\"ahler form}
\begin{align*}
\omega(X, Y) = g(J X, Y).
\end{align*}
It follows from the Hermitian condition that $\omega \in
\Lambda^{1,1}_{\mathbb R}$.  The triple $(M^{2n}, g, J)$ is a \emph{Hermitian manifold}.
\end{defn}

When dealing with several complex structures on the same smooth manifold, we will often use the notation $\omega_J = g( \cdot,J \cdot)$.

\begin{defn} \label{d:SKTclassic} 
Let $(M^{2n}, J)$ be a complex manifold. A Hermitian metric $g$ on $M$ is called \emph{pluriclosed} if
$$
dd^c \omega = 0.
$$
Here $d^c = \i (\delb - \del)$ is the conjugate differential, and in particular $d^c \gw = - d \gw(J, J, J)$.
We say that the metric is \emph{K\"ahler} if furthermore $d \gw = 0$.
\end{defn}

We have a local description of pluriclosed metrics similar to the local $d d^c$-lemma for K\"ahler metrics.

\begin{lemma} \label{l:PClocalddbar} Let $(M^{2n}, J)$ be a complex manifold and suppose $\gw \in \Lambda_{\mathbb R}^{1,1}$ satisfies
$dd^c \gw = 0$.  Given $p \in M$, for any local complex coordinates defined near $p$ there exists $\ga \in \Lambda^{1,0}$ defined in this coordinate chart so that
\begin{align*}
\gw = \delb \ga + \del \bga.
\end{align*}

\begin{proof} We give a sketch ignoring the relevant domains under consideration.  The pluriclosed condition tells us that $\del \gw \in \Lambda^{2,1}$ is holomorphic, and thus in a local coordinate chart one can apply the $\delb$-Poincar\'e lemma to obtain $\mu \in \Lambda^{2,0}$ such that $\del \gw = \delb \mu$.  It follows by a straightforward computation that $d \left( \gw - \mu - \bar{\mu} \right) = 0$.  As $\gw - \mu - \bar{\mu}$ is real, it follows from the $d$-Poincar\'e lemma that there exists $a \in \Lambda^1_{\mathbb R}$ such that $\gw - \mu - \bar{\mu} = d a$.  Expressing $a = \ga + \bar{\ga}$, it follows easily that $\ga \in \Lambda^{1,0}$ is the required potential.
\end{proof}
\end{lemma}

\begin{ex} \label{e:semisimplePC} Let $K$ denote a compact semisimple Lie group with $\dim K = 2n$.  Let $g$ denote the unique bi-invariant metric on $K$. By a classical result of Samelson \cite{Samelson}, we can choose a complex structure on the corresponding Lie algebra $\mathfrak {k}$ compatible with $g$, and we obtain a left-invariant complex structure $J$.  We claim that $(g, J)$ is pluriclosed with
\begin{align*}
-d^c\omega (X,Y,Z) = H(X, Y, Z) = g([X, Y], Z).
\end{align*}
Using the invariance properties and the integrability of $J$ we compute
\begin{align*}
- d^c \gw (X,Y,Z)=&\ d \gw \left( J X, J Y, J Z \right)\\
=&\ - \sum_{\gs(X,Y,Z)} \gw \left( [J X, J Y], J Z \right)\\
=&\ - \sum_{\gs(X,Y,Z)} g \left( [J X, J Y],  Z \right)\\
=&\ - \sum_{\gs(X,Y,Z)} g \left( J [J X, Y] + J [X, J Y] + [X,Y],  Z \right)\\
=&\ -3 H(X,Y,Z) + 2 \sum_{\gs(X,Y,Z)} g \left( [J X, J Y], Z \right)\\
=&\ -3 H(X,Y,Z) + 2 d^c \gw(X,Y,Z).
\end{align*}
Rearranging yields $-d^c \gw = H$. It is left as an \textbf{exercise} to show $d H = 0$.
\end{ex}

\begin{ex} \label{e:HopfSKT}
Consider the compact four-dimensional manifold $M = SU(2) \times U(1)$ as in Example \ref{ex:GRicciflatHopf}. Using the same notation, for every non-zero constant $x\in \mathbb{R}$ we define a left-invariant complex structure on $M$ by 
\begin{equation*}
J_x e_4 = x e_1, \qquad J_x e_2 = e_3.
\end{equation*}
We leave as an \textbf{exercise} to check that $J_x$ is integrable. Then, the metric $g = g_{k,x}$ defined in Example \ref{ex:GRicciflatHopf} is Hermitian and the associated K\"ahler form is
$$
\omega_{k,x} = - kx e^{14} + k e^{23}.
$$ 
Furthermore, we have that
$$
d_x^c\omega_{k,x} = - d\omega_{k,x}(J_x,J_x,J_x) = kx e^{234}(J_x,J_x,J_x) = - k e^{123}
$$
and therefore $g$ is pluriclosed. We will find this family of pluriclosed structures in a different disguise in Theorem \ref{t:Hopfrigidity}.
\end{ex}

A fundamental result of Gauduchon shows that all compact complex surfaces admit pluriclosed metrics:

\begin{thm} \label{t:Gaudthm} (\cite{Gauduchon1form}) Let $(M^{2n}, g, J)$ be a
Hermitian manifold.  There exists a unique $u \in C^{\infty}(M)$ such that
\begin{align*}
\int_M u dV = 0, \qquad \del \delb \left[ \left( e^{2u} \omega \right)^{\wedge n-1} \right]= 0.
\end{align*}
\end{thm}
\noindent More specifically, in dimension $n=2$ this theorem says that on compact complex surfaces every Hermitian metric has a conformally related metric which is pluriclosed, and this metric is unique up to scale in the conformal class.  In higher dimensions pluriclosed metrics need not exist.  In \cite{FinoParton} examples of compact complex nilmanifolds with invariant complex structure are given which do not admit compatible pluriclosed metrics.

\subsection{Courant algebroids and liftings}

The relation between pluriclosed metrics and generalized geometry has typically appeared in relation to generalized K\"ahler structures \cite{GualtieriGKG} (see also  \cite{CavalSKT}). Following \cite{garciafern2020gauge}, here we shall adopt a different viewpoint which uncovers gauge-theoretical aspects of pluriclosed metrics in interaction with Courant algebroids.  Connecting to the classical perspective, we are interested in \emph{pluriclosed structures} on a manifold endowed with closed three-form, as given in the following definition.  

\begin{defn} \label{d:SKT} Given $M^{2n}$ a smooth manifold and $H_0 \in \Lambda^3 T^*, d H_0 = 0$, a \emph{pluriclosed structure} is a triple $(g, b, J)$ consisting of a Riemannian metric $g,$ a two-form $b$, and an integrable complex structure $J$ such that
\begin{enumerate}
\item The metric $g$ is compatible with $J$,
\item One has
\begin{align*}
- d^c \gw = H_0 + db.
\end{align*}
\end{enumerate}
\end{defn}

Of course, the choice of closed three-form $H_0$ determines an exact Courant algebroid $E$ over $M$. Our goal in this section is to describe pluriclosed structures in terms of the geometry of $E$. We fix an exact Courant algebroid $E$ over a smooth manifold $M$ endowed with an integrable complex structure $J$. 

\begin{defn}\label{def:lifting}
A lifting of $T^{0,1}$ to $E$ is an isotropic, involutive subbundle $\ell \subset E \otimes \mathbb{C}$ mapping isomorphically to $T^{0,1}$ under the $\mathbb{C}$-linear extension of the anchor map $\pi \colon E \otimes \mathbb{C} \to T \otimes \mathbb{C}$.
\end{defn}

By a choice of isotropic splitting on $E$, we can give an explicit characterization of liftings.

\begin{lemma}\label{lemma:liftings}
Given $H_0 \in \Lambda^3 T^*, d H_0 = 0$, consider the exact Courant algebroid $(T\oplus T^*, \IP{,}, [,]_{H_0})$. Then, a lifting of $T^{0,1}$ is equivalent to $\gamma \in \Omega^{1,1 + 0,2}$ satisfying 
\begin{equation}\label{eq:liftingcond}
(H_0 + d \gamma)^{1,2 + 0,3} = 0.
\end{equation}
More precisely, given $\gamma$ satisfying \eqref{eq:liftingcond}, the lifting is
	\begin{equation}\label{eq:L}
	\ell = \{e^{\gamma}(X^{0,1}), \; X^{0,1} \in T^{0,1} \},
	\end{equation}
and, conversely, any lifting is uniquely expressed in this way.
\end{lemma}
\begin{proof}
An isotropic subbundle $\ell \subset  T\oplus T^* \otimes \mathbb{C}$ mapping isomorphically to $T^{0,1} X$ under $\pi$ is necessarily of the form \eqref{eq:L} for a suitable $\gamma \in \Lambda^{1,1 + 0,2}$. By Proposition \ref{p:bfieldbracket} we have
	\begin{equation*}
[e^\gamma X^{0,1} ,e^\gamma Y^{0,1} ]_{H_0} = e^\gamma [X^{0,1},Y^{0,1}] + i_{Y^{0,1}}i_{X^{0,1}}(H_0 + d\gamma).
	\end{equation*}
Then, $\ell$ is involutive if and only if \eqref{eq:liftingcond} holds.
\end{proof}

As we will see in \S \ref{ss:holCourant}, a lifting relates to the complex Courant algebroid $E \otimes \mathbb{C}$ as a Dolbeault operator relates to a smooth complex vector bundle, in the sense that it will enable us to construct a Courant algebroid in the holomorphic category out of $E \otimes \mathbb{C}$. With this idea in mind, and taking into account that we have a canonical \emph{real structure}
$$
E \subset E \otimes \mathbb{C},
$$
it is natural to look for a \emph{Chern correspondence}. This has been obtained in \cite{garciafern2020gauge} in a much more general setup, and we summarise here the result for the case of our interest.

\begin{defn}\label{def:unitarylift}
A \emph{horizontal lift} of $T$ to $E$ is given by a subbundle $W \subset E$ such that
$$
\operatorname{rk} W = \dim M, \qquad \textrm{and} \qquad W \cap T^* = \{0\}.
$$
By definition, a horizontal lift $W \subset E$ is equivalent to a (possibly degenerate) real symmetric $2$-tensor $g$ on $\underline{X}$ and an isotropic splitting $\sigma \colon T \to E$ such that
\begin{equation}\label{eq:Wlift}
W = \{\sigma(V) + g(V)\ |\  V \in T\}.
\end{equation}
In particular, it induces a closed three-form $H$ and an isomorphism $E \cong (T\oplus T^*, \IP{,}, [,]_{H})$.
\end{defn}

\begin{lemma}\label{lem:Cherncorr}
Any lifting $\ell \subset E \otimes \mathbb{C}$ of $T^{0,1}$ determines a unique horizontal lift $W \subset E$ such that
\begin{equation}\label{eq:L=LW}
\ell = \{e \in W \otimes \mathbb{C} \; | \; \pi(e) \in T^{0,1} \}.
\end{equation}
Furthermore, such $W$ determines uniquely a two-form $\omega = g(J \cdot,\cdot) \in \Lambda_{\mathbb{R}}^{1,1}$, and $H \in \Lambda^3T^*$, $dH = 0$, such that
$$
H = - d^c \omega.
$$
Therefore, in particular, $dd^c \omega = 0$.
\end{lemma}

\begin{proof}
We choose an isotropic splitting $\sigma_0 \colon T \to E$ inducing an isomorphism $E \cong (T\oplus T^*, \IP{,}, [,]_{H_0})$ for some closed three-form $H_0$. By Lemma \ref{lemma:liftings}, $\ell$ is uniquely determined by $\gamma \in \Lambda^{1,1 + 0,2}$ via \eqref{eq:L}. On the other hand, \eqref{eq:Wlift} implies that there exists a uniquely defined real two-form $b \in \Lambda^2 T^*$ such that
$$
W = e^b \{X + g(X)\ |\  X \in T\}.
$$
The isotropic condition for \eqref{eq:L=LW} implies that $g$ is a symmetric tensor of type $(1,1)$. Denote the associated $(1,1)$ form by
$$
\omega = g(J \cdot,\cdot) \in \Lambda^{1,1}_{\mathbb{R}}.
$$
Then, condition \eqref{eq:L=LW} implies
\begin{align*}
e^\gamma(T^{0,1}) &= e^{b + \sqrt{-1}\omega}(T^{0,1}),
\end{align*}
and therefore
$$
b^{1,1} + \sqrt{-1}\omega^{1,1} = \gamma^{1,1}, \qquad b^{0,2} = \gamma^{0,2}.
$$
Take now $\sigma_0$ to be the isotropic splitting determined by $W$. Then, $b = 0$ and $H_0 = H$, and hence $\gamma = \sqrt{-1}\omega$, and thus (see \eqref{eq:liftingcond})
$$
H^{0,3 + 1,2} = - \sqrt{-1} \; \overline{\partial} \omega.
$$
Since $\omega$ and $H$ are real, this yields the result.
\end{proof}

The previous lemma can be regarded as an analogue of the Chern correspondence for exact Courant algebroids. In particular, the horizontal subspace $W \subset E$, and hence the associated real $(1,1)$-form $\omega$, can be thought of as a connection-like object. The pluriclosed condition on $\omega$ can be regarded as a \emph{Bianchi identity} for $W$. The interpretation of two-forms as connections is reminiscent of higher gauge theory (see \cite{garciafern2020gauge} for further insights on this relation).  We close this subsection with a characterization of pluriclosed structures in terms of Courant algebroids.

\begin{thm} \label{t:SKTequivalence} 
Let $(M,J)$ be a complex manifold, and $H_0 \in \Lambda^3 T^*, d H_0 = 0$. Then, a pluriclosed structure of the form $(g, b, J)$ is equivalent to a lifting 
$$
\ell \subset (T \oplus T^*) \otimes \mathbb{C}
$$ 
of $T^{0,1}$ for the twisted Courant algebroid $(T\oplus T^*, \IP{,}, [,]_{H_0})$, such that the real symmetric two-tensor obtained via Lemma \ref{lem:Cherncorr} is positive definite.
\begin{proof} 
Given a lifting $\ell$ as in the statement, by Lemma \ref{lemma:liftings} and the proof of Lemma \ref{lem:Cherncorr} we have that $g$ is a Hermitian metric on $M$ and
$$
(H_0)^{0,3 + 1,2} + \overline{\partial}(b^{1,1} + \sqrt{-1}\omega) + d b^{0,2} = 0.
$$
Since $\omega$, $b$, and $H_0$ are real, this yields $- d^c \omega = H_0 + db$. The converse is left as an \textbf{exercise}.
\end{proof}
\end{thm}

\subsection{Metrics on holomorphic Courant algebroids}\label{ss:holCourant}

Motivated by the Chern correspondence in Lemma \ref{lem:Cherncorr}, we show next how a lifting of $T^{0,1}$ determines a Courant algebroid in the holomorphic category, following \cite{GualtieriGKG}. We will then go on and show that a pluriclosed structure yields a natural notion of Hermitian metric on these objects, following \cite{garciafern2018canonical, StreetsPCFBI}.

\begin{defn}\label{d:CAhol} 
Let $(M,J)$ be a complex manifold. A \emph{holomorphic Courant algebroid} is a holomorphic vector bundle $\mathcal{E} \to M$ together with a nondegenerate holomorphic symmetric bilinear form $\IP{,}$, a bracket  $[,]$ on holomorphic sections of $\mathcal{E}$, and a holomorphic bundle map $\pi : \mathcal{E} \to T^{1,0}$ satisfying the axioms in Definition \ref{d:CA} (in the holomorphic category). We will say that $\mathcal{E}$ is \emph{exact} if it fits into an exact sequence of holomorphic vector bundles
\begin{align*}
0 \longrightarrow T^*_{1,0} \overset{\pi^*}{\longrightarrow} \mathcal{E} \overset{\pi}{\longrightarrow} T^{1,0} \longrightarrow 0.
\end{align*}
\end{defn}

A complete classification of exact holomorphic Courant algebroids was obtained by Gualtieri in \cite{GualtieriGKG}. Let us summarize the result which we will use.

\begin{thm} \label{t:holCourant} 
Let $(M,J)$ be a complex manifold. Denote by $\Lambda^{2,0}_{cl}$ the sheaf of closed $(2,0)$-forms on $M$. Then, the set of isomorphism classes of exact holomorphic Courant algebroids on $M$ is bijective to the first \v Cech cohomology $H^1(\Lambda^{2,0}_{cl})$. Furthermore, there is a vector space isomorphism
\begin{equation}\label{eq:kcohomologyleqk}
H^1(\Lambda^{2,0}_{cl}) \cong \frac{\Ker \; d \colon \Lambda^{3,0 + 2,1}  \to \Lambda^{4,0 + 3,1 + 2,2}}{ \operatorname{Im} \; d \colon \Lambda^{2,0} \to \Lambda^{3,0 + 2,1}}.
\end{equation}
\end{thm}

The sheaf $\Lambda^{2,0}_{cl}$ in the previous classification result plays an important role in various developments in mathematical physics (see \cite{garciafern2018holomorphic} and references therein). A one-cocycle for this sheaf determines local data for gluing the model $T^{1,0} \oplus T^*_{1,0}$ by means of holomorphic $B$-field transformations. On the other hand, a representative of a cohomology class in the right hand side of \eqref{eq:kcohomologyleqk} yields a convenient global description of a holomorphic Courant algebroid.

\begin{defn}\label{def:Q0}
Let $(M,J)$ be a complex manifold. Given $H \in \Lambda^{3,0 + 2,1}$, $d H = 0$, we denote by 
	\begin{equation*}\label{eq:Qexpb}
	\mathcal{E}_{H} = T^{1,0} \oplus T^*_{1,0}
	\end{equation*}
	the exact holomorphic Courant algebroid with Dolbeault operator
	$$
	\overline{\partial} (X + \xi)  = \overline{\partial} X + \overline{\partial} \xi - i_{X}H^{2,1},
	$$ 
	symmetric bilinear form
	$$
	\IP{X + \xi, X + \xi}  = \xi(X),
	$$
	bracket given by
	\begin{equation*}
	[X + \xi, Y + \eta]_H   = [X,Y] + i_X \partial \eta + \partial (\eta(Y)) + i_Yi_X H^{3,0},
	\end{equation*}
	and anchor map $\pi(X + \xi) = X$.  We leave as an \textbf{exercise} to check that $\mathcal{E}_{H}$ so defined satisfies the axioms in Definition \ref{d:CAhol}.
\end{defn}

Given now an exact Courant algebroid $E$ over $M$ and a lifting $\ell \subset (T \oplus T^*) \otimes \mathbb{C}$ of $T^{0,1}$ (see Definition \ref{def:lifting}), it is a formality to associate a holomorphic Courant algebroid, namely,
$$
\ell \to \mathcal{E}_{\ell} := \mathcal{E}_{-2\sqrt{-1}\partial \omega}
$$
where $\omega$ is determined by $\ell$ as in Lemma \ref{lem:Cherncorr}. More invariantly, $\mathcal{E}_{\ell}$ can be obtained from $\ell$ via a natural reduction procedure applied to $E \otimes \mathbb{C}$, which explains the choice of normalization for $\partial \omega$ (see \cite{GualtieriGKG} for details). This illustrates further the fact that $\ell$ should be regarded as a Dolbeault operator on $E \otimes \mathbb{C}$, while $W$ in Lemma \ref{lem:Cherncorr} plays the role of its Chern connection.

Taking now the holomorphic point of view, we introduce the notion of a metric on $\mathcal E$ (cf. \cite{garciafern2018canonical,garciafern2020gauge}).

\begin{defn}\label{def:metricQ}
Let $(M,J)$ be a complex manifold endowed with an exact holomorphic Courant algebroid $\mathcal{E}$. A \emph{metric} on $\mathcal{E}$ is given by a triple $(E,\ell,\varphi)$, where
\begin{enumerate}

\item $E$ is an exact Courant algebroid over $M$,

\item $\ell \subset E \otimes \mathbb{C}$ is a lifting of $T^{0,1}$ with $\omega > 0$,

\item $\varphi \colon \mathcal{E}_\ell \to \mathcal{E}$ is an isomorphism of holomorphic Courant algebroids inducing the identity on $M$.

\end{enumerate}
\end{defn}

We next unravel the previous definition in terms of the model in Definition \ref{def:Q0}. The following result illustrates the important role played by $(2,0)$-forms in the theory of pluriclosed structures.

\begin{lemma} \label{l:SKTequivalence} 
Let $(M,J)$ be a complex manifold, and $H_0 \in \Lambda^{3,0 + 2,1}$, $dH_0 = 0$. Then, there is one to one correspondence between the set of metrics on $\mathcal{E}_{H_0}$ and
	\begin{equation*}
	 \Big\{\omega + \beta \; | \; \omega >0 , \; d \beta = H_0 + 2\sqrt{-1}\partial \omega  \Big\} \subset \Lambda^{1,1}_{\mathbb{R}} \oplus \Lambda^{2,0}.
\end{equation*}

\begin{proof}
By Lemma \ref{lem:Cherncorr}, the pair $(E,\ell)$ determines $\omega > 0$ and an isomorphism $E \cong (T\oplus T^*, \IP{,}, [,]_{-d^c\omega})$. By definition, $\mathcal{E}_{\ell} = \mathcal{E}_{-2\sqrt{-1}\partial \omega}$ and, by Theorem \ref{t:holCourant}, the isomorphism $\varphi \colon \mathcal{E}_\ell \to \mathcal{E}_{H_0}$ corresponds to
$$
\varphi = e^{-\beta} \colon T^{1,0} \oplus T^*_{1,0} \to T^{1,0} \oplus T^*_{1,0}
$$
for $\beta \in \Lambda^{2,0}$ satisfying $d \beta = H_0 + 2\sqrt{-1}\partial \omega$. We refer to \cite{garciafern2020gauge} for further details.
\end{proof}
\end{lemma}

We close this section indicating how the fairly abstract Definition \ref{def:metricQ} produces a Hermitian metric on the holomorphic vector bundle underlying $\mathcal{E}$, in the classical sense.  Let $\mathcal{E}$ be an exact Courant algebroid endowed with a metric $(E,\ell,\varphi)$. By Lemma \ref{lem:Cherncorr} the pair $(E,\ell)$ determines a pluriclosed Hermitian metric $g$. Define a Hermitian metric $G$ on $\mathcal{E}_{\ell} = \mathcal{E}_{-2\sqrt{-1}\partial \omega}$
by
\begin{align*}
G = \left( 
\begin{matrix}
g_{i \bj}  & 0 \\
0  & g^{\bl k}
\end{matrix}
\right).
\end{align*}
Via the isomorphism $\varphi \colon \mathcal{E}_\ell \to \mathcal{E}$ this construction yields a natural Hermitian metric $(\varphi^{-1})^*G$ on $\mathcal{E}$, which by direct computations is
\begin{align*}
(\varphi^{-1})^*G = \left( 
\begin{matrix}
g_{i \bj} + \gb_{i k} \bgb_{\bj \bl} g^{\bl k} & \gb_{ip} g^{\bl p}\\
 \bgb_{\bj \bp} g^{\bp k}  & g^{\bl k}
\end{matrix}
\right).
\end{align*}
This metric has an associated Chern connection, and associated Chern curvature.  As we will discuss in \S \ref{s:HRUP}, there a striking relationship between the curvature of the Chern connection associated to $G$ and the curvature of the Bismut connection associated to the underlying pluriclosed structure, an idea which is briefly discussed in \cite[Theorem 2.9]{Bismut}. This will play an important role in the higher regularity for the pluriclosed flow in Chapter \ref{c:GFCG} (cf. \cite{StreetsPCFBI, JordanStreets}).

\section{Generalized K\"ahler geometry} \label{s:GKG}

In this section we will give two equivalent definitions of generalized K\"ahler structure.  This concept first arose in mathematical physics literature \cite{GHR}, expressed in terms of classical Hermitian geometry and not exploiting the structures of generalized geometry.  Later Gualtieri \cite{GualtieriThesis} rederived this definition motivated by natural reductions of the structure group of $T \oplus T^*$.  Further works appeared in mathematical physics \cite{MP6, MP5, MP4, MP3, MP1, MP2} which influence our discussion below.

\subsection{Definitions and equivalences}

We will begin with the classical formulation of generalized K\"ahler geometry in terms of biHermitian geometry, a condition not obviously motivated from the point of view of classical complex and K\"ahler geometry.

\begin{defn} \label{d:GKBiherm} Given $M^{2n}$ a smooth manifold and $H_0 \in \Lambda^3 T^*$, $d H_0 = 0$, a 
\emph{generalized K\"ahler structure} is a quadruple $(g,b,I,J)$ consisting of 
a 
Riemannian metric $g$, $b \in \Lambda^2 T^*$, and two integrable complex 
structures $I$, $J$, such that
\begin{enumerate}
\item The metric $g$ is compatible with $I$ and $J$.
\item One has
\begin{align*}
- d^c_I \gw_I = H_0 + db = d^c_J \gw_J.
\end{align*}
\end{enumerate}
As the tensor $b$ can be absorbed by redefining $H_0$, at times we will refer to a generalized K\"ahler structure only using the data $(g, I, J)$.
\end{defn}

\begin{rmk} We note that the definition allows for the possibility that $[H_0] 
\neq 0$.  In some literature, including \cite{GualtieriThesis}, this is 
referred to 
as a \emph{twisted} generalized K\"ahler structure, an extra distinction we 
will 
not make.
\end{rmk}

We next give the definition of Gualtieri \cite{GualtieriThesis}, expressed in terms of generalized complex geometry, and motivated from the point of view of reduction of structure groups.  On a $2n$-manifold, a choice of complex structure gives a reduction of the structure group to $GL(n,\mathbb C)$, whereas a K\"ahler metric is a reduction to the maximal compact subgroup $U(n)$.  Similarly, as explained in Lemma \ref{l:gencompstructuregroup}, the presence of a single generalized complex structure, together with the neutral inner product, reduces the structure group of $E$ to $U(n,n)$.  By seeking a further reduction to the maximal compact subgroup $U(n) \times U(n)$, it is possible to derive the following definition.

\begin{defn} \label{d:GKCourant}  Let $E$ be an exact Courant 
algebroid over a smooth manifold $M^{2n}$.  A \emph{generalized K\"ahler structure} is a pair $(\JJ_1, 
\JJ_2)$ of generalized complex 
structures such that
\begin{enumerate}
\item $[\JJ_1, \JJ_2] = 0$,
\item $\GG = - \JJ_1 \JJ_2$ defines a generalized metric.
\end{enumerate}
\end{defn}

The fundamental theorem of Gualtieri shows that Definitions 
\ref{d:GKBiherm} and \ref{d:GKCourant} are equivalent.  Before showing this we record some further basic structure associated to generalized K\"ahler structures as per Definition \ref{d:GKCourant}.  In particular, given $\JJ_1, \JJ_2$ a generalized K\"ahler structure, the generalized metric $\GG = - \JJ_1 \JJ_2$ determines an isotropic splitting of $E$, and therefore an isomorphism $E \cong(T \oplus T^*, [,]_{H}, \IP{,},\pi)$ for a uniquely determined $H \in \Lambda^3T^*$, $d H = 0$. We also obtain two maximal isotropic subspaces $L_1, L_2$, the $\i$-eigenbundles of $\JJ_1, \JJ_2$, such that
\begin{align*}
(T \oplus T^*) \otimes \mathbb C = L_1 \oplus \bar{L}_1 = L_2 \oplus \bar{L}_2.
\end{align*}
Furthermore, as $\JJ_1$ and $\JJ_2$ commute, $L_1$ itself must admit a decomposition into the $\pm \i$-eigenspaces for $\JJ_2$, yielding
\begin{align*}
(T \oplus T^*) \otimes \mathbb C = L_1^+ \oplus L_1^- \oplus \bar{L}_1^+ \oplus \bar{L}_1^-.
\end{align*}
As intersections of isotropic, integrable subspaces, each of the four summands above is itself isotropic and integrable. We are now ready to prove the equivalence.

\begin{thm}[\cite{GualtieriThesis}] \label{t:GKequivalence}
Given $M^{2n}$ a smooth manifold and $H_0 \in \Lambda^3 T^*$, $d H_0 = 0$, a quadruple $(g,b,I,J)$ as in Definition \ref{d:GKBiherm} determines a generalized K\"ahler structure $(\JJ_1,\JJ_2)$ on $(T \oplus T^*, [,]_{H_0}, \IP{,},\pi)$ such that
\begin{align} \label{f:Gualtierimap}
\JJ_{1,2} = \tfrac{1}{2} \left(
\begin{matrix}
1 & 0\\
b & 1
\end{matrix} \right) \left(
\begin{matrix}
I \pm J & - (\gw_I^{-1} \mp \gw_J^{-1})\\
\gw_I \mp \gw_J & - \left(I^* \pm J^*\right)
\end{matrix} \right)
\left(
\begin{matrix}
1 & 0\\
-b & 1
\end{matrix} \right).
\end{align}
Conversely, given a generalized K\"ahler structure $(\JJ_1,\JJ_2)$ on an exact Courant algebroid $E$, it determines, up to a choice of isotropic splitting with associated $H_0 \in \Lambda^3 T^*$, $d H_0 = 0$, a quadruple $(g,b,I,J)$ as in Definition \ref{d:GKBiherm} such that \eqref{f:Gualtierimap} holds.
\begin{proof}
To begin we define the relevant maps connecting the two definitions, then prove that the integrability conditions are the same. We begin by mapping a generalized K\"ahler pair $(\JJ_1, \JJ_2)$ on an exact Courant algebroid identified with $(T \oplus T^*, [,]_{H_0}, \IP{,},\pi)$, to a biHermitian structure $(g, b, I, J)$.  Let $\GG = - \JJ_1 \JJ_2$ denote the generalized metric associated to this pair.  As described in \S \ref{s:genmetric}, associated to $\GG$ is an eigenspace decomposition $E = V_+ \oplus V_-$ of $E$.  Observe that since $V_+$ is positive definite for the neutral inner product, whereas $T$ is a null space, it follows that 
\begin{align*}
\pi : V_{\pm} \to T
\end{align*}
are isomorphisms, and so we may interpret structures on $V_{\pm}$ as structures on $T$ using this map.

In particular, the neutral symmetric inner product defines an inner product $g$ on $T$, which is positive definite by construction.  Similarly the neutral skew-symmetric inner product defines a two-form $b$.  This is a manifestation of the fact that $V_{\pm}$ are the graphs of $b \pm g$ as in Proposition \ref{p:genmetricreduction}.  To construct the complex structures, first note that since $[\JJ_1, \JJ_2] = 0$ we get a fourfold decomposition of $E$ according to the eigenspaces of the $\JJ_i$, and from this it follows easily that $V_{\pm}$ are preserved by each $\JJ_i$. We can transport $\JJ_1$ from both $V_{\pm}$ to obtain the two almost complex structures $I$ and $J$.  Since $\JJ_1$ is orthogonal with respect to the neutral inner product, it follows that $I$ and $J$ are compatible with $g$.  

To finish the proof we must check the integrability conditions, in particular that $I$ and $J$ are integrable and that $d^c_I \gw_I = H_0 + db = - d^c_J \gw_J$.  To check the integrability of $I$, we must check that the $\i$-eigenbundle $T^{1,0}_{I}$ inside $T \otimes \mathbb C$ is closed under the Lie bracket.  By definition $I$ is obtained by the action of $\JJ_1$ on $V_+$, using the isomorphism obtained by $\pi$.  It follows that the preimage of $T^{1,0}_I$ in $V_+$ must be $L_1 \cap V_+ \otimes \mathbb{C} = L_1^+$.  As remarked above $L_1^+$ is integrable, which implies $T^{1,0}_I$ is closed under Lie bracket.  The proof for integrability of $J$ is similar.  Lastly we check the conditions on the torsion.  In particular, using the explicit descriptions of the associated metric, one can compute an explicit description of the two Dirac structures $L_1^{\pm} \subset V_\pm \otimes \mathbb{C}$, specifically that
\begin{gather} \label{f:GK20}
\begin{split}
L_1^+ =&\ \{ X + (b - \i \gw_I) X |\ X \in T^{1,0}_I \},\\
L_1^- =&\ \{ X + (b + \i \gw_J) X |\ X \in T^{1,0}_J \}.
\end{split}
\end{gather}
Having described these Dirac structures explicitly as graphs, by Proposition \ref{p:integrabilityofgraphs}, we see that since $L_1^{\pm}$ are both integrable, we obtain the equations (cf. Lemma \ref{lemma:liftings})
\begin{align*}
i_Y i_X \left( H_0 + d (b - \i \gw_I) \right) =&\ 0, \qquad \mbox{ for all } X, Y \in T^{1,0}_I\\
i_Y i_X \left( H_0 + d (b + \i \gw_J) \right) =&\ 0, \qquad \mbox{ for all } X, Y \in T^{1,0}_J
\end{align*}
Considering the $(2,1)$ piece of the first equation above yields $(H_0 + db)^{2,1} = \i \del_I \gw_I$.  Conjugating and using that $b$ is real yields $(H_0 + db)^{1,2} = - \i \delb_I \gw_I$, and hence $H_0 + db - d^c_I \gw_I = 0$, and similarly we can obtain $H_0 + db + d^c_J \gw_J = 0$.  We conclude $d^c_I \gw_I = H_0 + db = - d^c_J \gw_J$, as required (cf. Theorem \ref{t:SKTequivalence}).

To reverse this construction, fix $(g,b,I,J)$ as above.  Using $g$ and $b$ we determine a generalized metric $\GG$ via equation (\ref{Gformula}).  With the metric in hand, we have access to the various projection operators onto $V_{\pm}$, which allows us to rebuild the $\JJ_i$ by hand from the way we extracted $I$ and $J$ above.  In particular, one must have
\begin{align*}
\JJ_1 =&\ \pi_{|V_+}^{-1} I \pi \pi_+ + \pi_{|V_-}^{-1} J \pi \pi_-\\
\JJ_2 =&\ \pi_{|V_+}^{-1} I \pi \pi_+ - \pi_{|V_-}^{-1} J \pi \pi_-.
\end{align*}
Unwinding these definitions is a lengthy computation left as an \textbf{exercise}.  The answer is expressed cleanly in terms of the K\"ahler forms, defined by
\begin{align*}
\gw_I(X,Y) = g(I X, Y), \qquad \gw_J(X,Y) = g(J X, Y).
\end{align*}
Using these one obtains \eqref{f:Gualtierimap}.

To finish the proof we must establish the integrability of $\JJ_{1,2}$, in other words the Courant involutivity of $L_1, L_2$, and it suffices to check $L_1$.  To begin we note that, as the algebraic aspects of the correspondence have been established, it follows that the Dirac structures $L_1^{\pm} \subset V_\pm \otimes \mathbb{C}$ take the form of (\ref{f:GK20}).  Moreover, each of $L_1^{\pm}$ is also integrable by Proposition \ref{p:integrabilityofgraphs} as above.  It remains to show that the bracket of a pair of sections taken one each from $L_1^{\pm}$ remains in $L_1^+ \oplus L_1^-$.  To that end we fix $X + (b + g) X \in \gG(L_1^+)$, $Y + (b-g) Y \in \gG(L_1^-)$, and compute
\begin{gather*}
\begin{split}
[X + (b + g) X, Y + (b-g)Y] =&\ [X,Y] - L_X(gY) - L_Y(gX)\\
&\ + d(g(X,Y)) + b([X,Y]) + i_Y i_X H.
\end{split}
\end{gather*}
To understand the Lie bracket term we will use the two natural Bismut connections associated to the generalized K\"ahler structure, namely
\begin{align*}
\IP{\N^I_X Y, Z} = \IP{\N_X Y, Z} - \tfrac{1}{2} d^c_I \gw_I (X,Y,Z),\\
\IP{\N^J_X Y, Z} = \IP{\N_X Y, Z} - \tfrac{1}{2} d^c_J \gw_J (X,Y,Z).
\end{align*}
These are connections compatible with $I$ and $J$ respectively, equivalent to the connections $\N^{\pm}$ of Definition \ref{d:Bismutconn}, whose properties are discussed more fully in \S \ref{s:CTT}.  Using these connections and letting $H = - d^c_I \gw_I = d^c_J \gw_J$, we can express
\begin{align*}
[X,Y] =&\ \N_X Y - \N_Y X\\
=&\ \left(\N_X^J + \tfrac{1}{2} g^{-1} H_X \right) Y - \left( \N_Y^I - \tfrac{1}{2} H_Y \right)X\\
=&\ \N^J_X Y - \N^I_Y X.
\end{align*}
A further computation shows that one has 
\begin{align*}
- L_X(gY) - L_Y(gX) + d(g(X,Y)) = - g(\N_X Y + \N_Y X).
\end{align*}
Combining these observations now yields
\begin{gather*}
\begin{split}
[X + (b + g) X,& Y + (b-g)Y]\\
=&\ \N_X^J Y - \N_X^I Y - g(\N_X Y + \N_Y X) + b([X,Y]) + i_Y i_X H\\
=&\ \N^J_X Y - \N_Y^I X + b(\N^J_X Y - \N^I_Y X) - g(\N^J_X Y + \N^I_Y X)\\
=&\ - \left( \N_Y^I X + (b + g) \N^I_Y X \right) + \left( \N^J_X Y + (b - g) \N^J_X Y \right).
\end{split}
\end{gather*}
Noting that, by construction, $X \in T^{1,0}_I, Y \in T^{1,0}_J$, it follows that $\N^I_Y X \in T^{1,0}_I$ and $\N^J_X Y \in T^{1,0}_J$, and thus the expression above is a section of $L_1^+ \oplus L_1^-$, finishing the proof.
\end{proof}
\end{thm}

\begin{rmk}
As is clear from the proof of Theorem \ref{t:GKequivalence}, given a generalized K\"ahler structure the subspace $\bar{L}_1^+$ (resp. $\bar{L}_1^-$) is a lifting of $T^{0,1}_I$ (resp. $T^{0,1}_J$) in the sense of Definition \ref{def:lifting}. This point of view was taken in \cite{GualtieriGKG} to analyze generalized K\"ahler structures in terms of Morita equivalence for holomorphic Dirac structures, and has recently been exploited in \cite{bischoff} (cf. Remark \ref{r:scalarreduction}).
\end{rmk}

\subsection{Poisson structure} \label{ss:Poisson}

A remarkable fact about generalized K\"ahler structures is that there is a further refinement of the Poisson tensors associated to the generalized complex structures.  These play a strong role in determining the 
local geometry of a GK structure, and are moreover crucial to constructing certain natural
deformation classes of GK structure (cf. \cite{GotoDef,HitchindelPezzo}).  The Poisson tensor $\gs$ below was originally discovered in \cite{AGG}, \cite{PontecorvoCS} in the four-dimensional case, and \cite{HitchinPoisson} in general dimensions.

\begin{defn} \label{d:GKPoissondef} Given $(M^{2n}, g, I, J)$ a generalized 
K\"ahler manifold, let
\begin{align*}
 \gs := \tfrac{1}{2} [I,J] g^{-1}.
\end{align*}
\end{defn}
Already from Proposition \ref{p:realPoisson} we know that a generalized K\"ahler structure will come equipped with Poisson tensors arising from the associated generalized complex structures, and one arrives at the tensor $\gs$ above by considering the type decomposition of these tensors according to the complex structures $I$ and $J$.  In particular, this tensor arises as the $(2,0)-(0,2)$ projection of the given Poisson tensors, with respect to either $I$ or $J$:
\begin{align*}
\gs = I \pi_{2,0 + 0,2}^I \pi \JJ_1 = J \pi_{2,0 + 0,2}^J \pi \JJ_2,
\end{align*}
as follows from a short computation using the Gualtieri map (\ref{f:Gualtierimap}).  To establish that $\gs$ is indeed a Poisson structure, we first need to establish that it is the real part of a holomorphic $(2,0)$ tensor with respect to both $I$ and $J$.

\begin{lemma} \label{l:projections} Let $(V^{2n}, J)$ be a real vector space with almost complex structure $J$.  Given $H \in \Lambda^3 V^*$,
\begin{align*}
(\pi_{3,0 + 0,3} H)_{ijk} =&\ \tfrac{1}{4} \left( H_{ijk} - J_j^p J_k^q H_{ipq} - J_i^p J_k^q H_{p j q} - J_i^p J_j^q H_{pq k} \right)\\
(\pi_{2,1 + 1,2} H)_{ijk} =&\ \tfrac{3}{4} H_{ijk} + \tfrac{1}{4} \left( J_j^p J_k^q H_{ipq} + J_i^p J_k^q H_{p j q} + J_i^p J_j^q H_{pq k} \right).
\end{align*}
\begin{proof} \textbf{Exercise}.
\end{proof}
\end{lemma}

\begin{prop} \label{p:holoPoisson} Given $(M^{2n}, g, I, J)$ a generalized 
K\"ahler manifold, one has that $\gs$ is a Poisson tensor, and furthermore
\begin{align*}
 \gs \in \Lambda^{2,0 + 0,2}_I(TM) \cap \Lambda^{2,0 + 0,2}_J(TM), \qquad
\delb_I \gs^{2,0}_I = 0, \qquad \delb_J \gs^{2,0}_J = 0.
\end{align*}
\begin{proof} It already follows as discussed above that $\gs$ has the claimed algebraic structure, but we check this directly using the definition of $\gs$.  Given $\ga,\gb \in T^*M$, using that $I$ anticommutes with the commutator $[I,J]$, and that
$g$ is compatible with $I$ we obtain
\begin{align*}
 \gs(I \ga, I \gb) =&\ (I \gb) ( [I,J] g^{-1} (I \ga))\\
 =&\ \gb (I [I,J] g^{-1} I \ga )\\
 =&\ - \gb ([I,J] I g^{-1} I \ga)\\
 =&\ - \gb( [I,J] g^{-1} \ga)\\
 =&\ \gs(\ga,\gb).
\end{align*}
Thus $\gs \in \Lambda^{2,0 + 0,2}_I(TM)$, and an identical argument yields $\gs
\in \Lambda^{2,0 + 0,2}_J(TM)$.

To check the holomorphicity, we first note that the $(2,0)$ projection can be computed as
\begin{align*}
\gs^{2,0}_I(\ga,\gb) = \gs(\ga,\gb) + \i \gs(I\ga, \gb).
\end{align*}
Thus to compute $\delb_I \gs^{2,0}_I$ it suffices to take the $I$-Chern derivative in the direction of $X + \i I X$.  We again refer to \S \ref{s:CTT} for fundamental properties of the relevant Chern connections associated to $(g, I)$ and $(g, J)$.  Taking the real part of this it suffices to check that
\begin{gather} \label{f:holoPoisson10}
\begin{split}
0 =&\ g \left( \left( \N^{C,I}_X [I,J] + I \N^{C,I}_{I X} [I,J] \right) Y, Z \right)\\
=&\ g \left( \left( I \N^{C,I}_X J - (\N^{C,I}_X J) I + I \left(I \N^{C,I}_{I X} J - (\N^{C,I}_{IX} J) I \right) \right) Y, Z \right)\\
=&\ - g \left( (\N_X^{C,I} J)Y , IZ \right) - g \left( (\N_X^{C,I} J) IY, Z \right)\\
&\  - g \left( (\N^{C,I}_{IX} J) Y, Z \right) + g \left( (\N^{C,I}_{IX} J) IY, IZ \right).
\end{split}
\end{gather}
Now we note, using (\ref{f:ChernviaB}),
\begin{gather*}
\begin{split}
& \IP{(\N^{C,I}_X J)(Y), Z}\\
&\quad = \IP{(\N^{C,I} - \N^{C,J})_X J)(Y), Z}\\
&\quad = \IP{(\N^{C,I} - \N^{C,J})_X)(J Y) - J (\N^{C,I} - \N^{C,J})_X)(Y), Z}\\
&\quad = \IP{(\N^{C,I} - \N^{C,J})_X)(J Y) + (\N^{C,I} - \N^{C,J})_X)(Y), J Z}\\
&\quad = \tfrac{1}{2} \left( H(X,IJY, IZ) - H(X,Y, JZ) + H(X,IY,IJZ) - H(X,JY,Z) \right).
\end{split}
\end{gather*}
Plugging this into (\ref{f:holoPoisson10}) means that we must check
\begin{align*}
0 =&\ - \left( H(X,IJY, IIZ) - H(X,Y, JIZ) + H(X,IY,IJIZ) - H(X,JY,IZ) \right)\\
&\ -  \left( H(X,IJIY, IZ) - H(X,IY, JZ) + H(X,IIY,IJZ) - H(X,JIY,Z) \right)\\
&\ -  \left( H(IX,IJY, IZ) - H(IX,Y, JZ) + H(IX,IY,IJZ) - H(IX,JY,Z) \right)\\
&\ +  \left( H(IX,IJIY, IIZ) - H(IX,IY, JIZ) + H(IX,IIY,IJIZ) - H(IX,JIY,IZ) \right)\\
=&\  H(X,IJY, Z) + H(X,Y, JIZ) - H(X,IY,IJIZ) + H(X,JY,IZ)\\
&\ -  H(X,IJIY, IZ) + H(X,IY, JZ) + H(X,Y,IJZ) + H(X,JIY,Z)\\
&\ -  H(IX,IJY, IZ) + H(IX,Y, JZ) - H(IX,IY,IJZ) + H(IX,JY,Z)\\
&\ - H(IX,IJIY, Z) - H(IX,IY, JIZ) - H(IX,Y,IJIZ) - H(IX,JIY,IZ)\\
=&\ \sum_{i=1}^{16} A_i.
\end{align*}
Using that $H$ is of type $(2,1) + (1,2)$ (cf. Lemma \ref{l:projections}), one observes that $A_1 + A_4 + A_9 + A_{12} = 0$, $A_2 + A_3 + A_{14} + A_{15} = 0$, $A_5 + A_8 + A_{13} + A_{16} = 0$, $A_6 + A_7 + A_{10} + A_{11} = 0$, finishing the claim.

Now we can show that $\gs$ is a Poisson tensor.  First note that $P = \pi \JJ_1$ is a Poisson tensor, and we have the type decomposition
\begin{align*}
P = P^{2,0}_I + P^{1,1}_I + P^{0,2}_I = \gs^{2,0}_I I + P^{1,1}_I + \gs^{0,2}_I I.
\end{align*}
Since $P$ is Poisson, item (1) of Lemma \ref{l:Poissonintegrable} holds for $P$, choosing the Levi-Civita connection of an arbitrarily chosen metric.   The $(3,0)$ piece of that equation has terms of the form $\gs^{2,0}_I \star \N \gs^{2,0}_I$ and also of the form $P_I^{1,1} \star \N \gs^{2,0}$.  Since $\gs_I^{2,0}$ is $I$-holomorphic, the terms of the form $P_I^{1,1} \star \N \gs^{2,0}$ vanish, thus it follows that item (1) of Lemma \ref{l:Poissonintegrable} holds for $\gs_I^{2,0}$.  Thus $\gs_I^{2,0}$ is Poisson, as is $\gs$.  A similar argument holds for $\gs_J^{2,0}$.
\end{proof}
\end{prop}

\subsection{Examples} \label{ss:GKexamples}

We give here some fundamental examples of generalized K\"ahler structure which will guide the discussion to follow.

\begin{ex} \label{e:KasGK} If $(M^{2n}, g, J)$ is K\"ahler, it is clear that this also 
represents a generalized K\"ahler structure with $H_0 = 0, b = 0$ and $I = \pm J$.  Note from (\ref{f:Gualtierimap}) one has that one of the associated generalized complex structures comes from the complex structure $I = J$ as in Example \ref{e:complexGC}, whereas the other is associated to the K\"ahler form $\gw$ as in Example \ref{e:symplecticGC}.  Also, for this structure one has $\gs = 0$.
\end{ex}

\begin{ex} \label{e:splitGK} Let $(M_i^{2n_i}, g_i, K_i), i = 1,2$ be two 
K\"ahler manifolds.  Let $M = M_1 \times M_2$, $g = \pi_1^* g_1 + \pi_2^* g_2$. 
 Note that each $K_i$ defines in a natural way an endomorphism of $T M_i 
\subset TM$, and thus we may define
\begin{align*}
I = \left(\begin{matrix}
K_1 & 0\\
0 & K_2
\end{matrix} \right), \qquad J = \left(\begin{matrix}
K_1 & 0\\
0 & - K_2
\end{matrix} \right).
\end{align*}
It is easy to check that these define integrable complex structures on $M$, 
which moreover are both compatible with $g$.  One easily computes that $d \gw_I 
= d \pi_1^* \gw_1 + d \pi_2^* \gw_2 = 0$ and $d \gw_J = d \pi_1^* \gw_1 - d 
\pi_2^* \gw_2 = 0$ (note the sign change due to the sign change in the definition 
of $J$), and so this defines a generalized K\"ahler structure.  Here again the 
underlying pairs $(g,I)$ and $(g,J)$ are K\"ahler, but we will show in \S 
\ref{ss:GKVPS} below how to deform this example in general to produce 
non-K\"ahler structures.
\end{ex}

\begin{ex} \label{e:HKGK} Let $(M^{4n}, g, I, J)$ be a hyperK\"ahler manifold.  
Setting $H_0 = 
0$, it follows from the properties of hyperK\"ahler structures that $(g,0,I,J)$ 
defines a generalized K\"ahler structure.  Note that even though the underlying Hermitian structures are K\"ahler, the associated generalized complex structures are not the standard ones as in Example \ref{e:KasGK}.  Furthermore, observe that in this setting one has 
using the quaternionic relations
\begin{align*}
\gs = \tfrac{1}{2} [I,J] g^{-1} = K g^{-1} = \gw_K^{-1}.
\end{align*}
In particular, the Poisson structure $\gs$ defines a nondegenerate pairing on 
$T^*$.  In this sense this example and Example \ref{e:splitGK} represent the 
two 
extremes of generalized K\"ahler geometry, and their local structure and 
geometry vary considerably, as we will see in the remainder of this chapter.
\end{ex}

\begin{ex} \label{e:semisimpleGK} Let $K$ denote a compact semisimple Lie group.  Let $g$ denote the unique bi-invariant metric on $K$.  Fixing any complex structure on the corresponding Lie algebra $\mathfrak {k}$ compatible with $g$, we obtain left- and right-invariant complex structures $I_L$ and $I_R$.  We claim that $(g, I_L, I_R)$ is a generalized K\"ahler structure with
\begin{align*}
H(X, Y, Z) = g([X, Y], Z).
\end{align*}
As computed in Example \ref{e:semisimplePC}, we know that $d^c_{I_L} \gw_{I_L} = H$.  Using that the right Lie algebra is anti-isomorphic to the left Lie algebra, it follows that $d^c_{I_R} \gw_{I_R} = - H$, finishing the proof.
\end{ex}

\subsection{Generalized K\"ahler structures with vanishing Poisson structure}\label{ss:GKVPS}

In this subsection we explore further the local structure of generalized 
K\"ahler manifolds in the case $\gs = 0$, which is equivalent to $[I,J] = 0$.  We set up and prove a basic structure theorem  (cf. \cite{ApostolovGualtieri} Theorem 4)
which yields a splitting of the tangent bundle in this setting, which in a very 
rough sense indicates that generalized K\"ahler geometry with $\gs = 0$ behaves locally like a twisted 
product of K\"ahler structures.

\begin{thm} \label{t:STBstructure} Let $(M^{2n},g, I,
J)$ be a generalized
K\"ahler manifold with
\begin{align*}
[I, J] = 0.
\end{align*}
Let $\Pi = I J$.  Then the $\pm
1$-eigenspaces of $\Pi$ are $g$-orthogonal $I$- and $J$-holomorphic foliations on
whose leaves $g$ restricts to a K\"ahler metric.
\begin{proof} Let $T_{\pm}$ denote the eigenspace decomposition of $\Pi$, whose eigenvalues are $\pm 1$ as it is a real endomorphism which squares to the identity.  We note that $T_{\pm} = \ker (\Pi \mp \Id) = \ker (I \pm J)$.  But $\ker (I \pm J)$ is the image of $I \mp J$, thus $T_{\pm}$ coincide with the images of the associated Poisson tensors
\begin{align*}
\gs_1 = \pi \JJ_1 \pi_* = (I - J) g^{-1}, \qquad \gs_2 = \pi \JJ_2 \pi_* = (I + J) g^{-1},
\end{align*}
which are thus integrable.  Since $g$ is compatible with both $I$ and $J$, the projection $\Pi$ is orthogonal, and hence $T_{\pm}$ are orthogonal.  We refer to the restrictions of $\gw_I$ to these subbundles $\gw_I^{\pm}$, similarly $g^{\pm}, \gw_J^{\pm}$.

The complex structures preserve $T_{\pm}$, and thus we obtain a refinement $T_+ \otimes \mathbb C = A \oplus \bar{A},\ T_- \otimes \mathbb C = B \oplus B$, where
\begin{align*}
A = T_I^{1,0} \cap T_J^{0,1}, \qquad B = T_I^{1,0} \cap T_J^{1,0}.
\end{align*}
As intersections of integrable distributions, $A$ and $B$ are both integrable.  To show $A$ is a holomorphic subbundle we show it is preserved by the $\delb$ operator of $I$.  Using the decompositions above, given $X \in \gG(T^{0,1}_I)$ and $Z \in \gG(A)$,
\begin{align*}
\delb_X Z = [X,Z]^{1,0} = [X,Z]_A + [X,Z]_B = [X,Z]_A + [X_{\bar{A}},Z]_B + [X_{\bar{B}}, Z]_B.
\end{align*}
Since $X_{\bar{A}},Z \in \gG(T_+)$ and $T_+$ is integrable, the term $[X_{\bar{A}}, Z]_B$ vanishes.  Also, since $A \oplus \bar{B} = T_J^{0,1}$ is again involutive, it follows that the term $[X_{\bar{B}}, Z]$ vanishes, finishing the claim that $A$ is a holomorphic subbundle.

Lastly, since $I = J$ along the leaves of $T_-$, it follows that $d^c_I  = d^c_J$, and $\gw_I^- = \gw_J^-$, thus $d^c_I \gw_I^- = d^c_J \gw_J^-$, and hence the restriction of $\gw_I^-$ to the leaves of $T_-$ is closed.   Similarly, along the leaves of $T_+$ we have $I = - J$, and hence $d^c_I = - d^c_J$, while $\gw_I^+ = - \gw_J^+$.  It follows that $d^c_I \gw^+_I = d^c_J \gw^+_J$ and thus the restriction of $\gw_I^+$ to the leaves of $T_+$ is closed.
\end{proof}
\end{thm}

This theorem yields a splitting of the holomorphic tangent bundle 
$T^{1,0} = T^{1,0}_+ \oplus T^{1,0}_-$, yielding in turn a four-fold splitting
\begin{align*}
T_{\mathbb C} = T^{1,0}_+ \oplus T^{1,0}_- \oplus T^{0,1}_+ \oplus T^{0,1}_-.
\end{align*}
By duality we also obtain a splitting of the cotangent bundle, which we 
indicate as
\begin{align*}
T^*_{\mathbb C} = \Lambda^{1,0}_+ \oplus \Lambda^{1,0}_- \oplus \Lambda^{0,1}_+ 
\oplus \Lambda^{0,1}_-.
\end{align*}
Going further, this yields a splitting of all form spaces $\Lambda^{p,q}$.  In 
particular, we define
\begin{align*}
\Lambda^{p,q}_{r,s} := \left( \Lambda^{1,0}_+ \right)^{\wedge p} \wedge \left( 
\Lambda^{1,0}_- \right)^{\wedge r} \wedge \left( \Lambda^{0,1}_+ 
\right)^{\wedge q} \wedge \left( \Lambda^{0,1}_- \right)^{\wedge s}.
\end{align*}
Using the natural projection operators, it 
follows that there is a further refinement of the exterior derivative $d = \del + 
\delb$ into
\begin{gather*}
\begin{split}
d :&\ \Lambda^{p,q}_{r,s} \to \Lambda^{p+1,q}_{r,s} \oplus 
\Lambda^{p,q+1}_{r,s} \oplus \Lambda^{p,q}_{r+1,s} \oplus 
\Lambda^{p,q}_{r,s+1}\\
d =&\ \del_+ + \delb_+ + \del_- + \delb_-,
\end{split}
\end{gather*}
where
\begin{align*}
\del_+ =&\ \pi_{p+1,q,r,s} \circ d, \qquad \delb_+ = \pi_{p,q+1,r,s} \circ d,\\
\del_- =&\ \pi_{p,q,r+1,s} \circ d, \qquad \delb_- = \pi_{p,q,r,s+1} \circ d.
\end{align*}

A basic example of these structures occurs by taking products of K\"ahler manifolds and changing the orientations of some factors, as in Example \ref{e:splitGK}.  More generally we can take quotients of such examples to yield nontrivial, i.e. non-K\"ahler, examples.  The first such occurs on Hopf surfaces.

\begin{defn} \label{d:Hopf} By definition, a \emph{Hopf surface} is a compact complex surface whose universal cover is biholomorphic to $\mathbb C^2 \backslash \{0\}$.  The surface is \emph{primary} if $\pi_1(M) = \mathbb Z$, and is otherwise \emph{secondary}.  For primary Hopf surfaces the fundamental group is generated by
\begin{align*}
\gg(z_1,z_2) = (\ga z_1, \gb z_2 + \gl z_1^m)
\end{align*}
where $0 < \brs{\ga} \leq \brs{\gb} < 1$, and
\begin{align*}
\left(\ga - \gb^m \right) \gl = 0.
\end{align*}
The map $\gg$ is obviously a biholomorphism, and thus one obtains a well-defined complex structure on the quotient, which by an \textbf{exercise} is always diffeomorphic to $S^3 \times S^1$.  The surface is \emph{class $1$} if $\gl = 0$, and is \emph{class $0$} otherwise.  Lastly, we say that the Hopf surface is \emph{diagonal} if $\brs{\ga} = \brs{\gb}$.

For diagonal Hopf surfaces, a certain Hermitian metric, called the Boothby/Hopf metric, will play a key role.  Let
\begin{align} \label{f:Hopfmetric}
g(X,Y)_{(z_1,z_2)} = \frac{g_{E}(X,Y)}{\brs{z_1}^2 + \brs{z_2}^2},
\end{align}
where $g_{E}$ is the standard Euclidean metric on $\mathbb C^2$.  This metric is Hermitian with respect to the standard complex structure on $\mathbb C^2$.  It is pluriclosed, and its curvature and torsion properties are further explored in Proposition \ref{p:Hopfmetricprop} below.
\end{defn}

\begin{ex} \label{e:Hopfsplit} Consider a diagonal Hopf surface as described above.  First note that, if we set
\begin{align*}
\textbf{J} = \left(\begin{matrix} 0 & -1\\ 1 & 0 \end{matrix}\right),
\end{align*}
then the standard 
complex structure $I$ on $\mathbb C^2 \backslash \{0\}$ can be expressed with 
respect to the standard coordinate basis $\{\del_{x_1}, \del_{y_1}, \del_{x_2}, 
\del_{y_2} \}$  as the block diagonal matrix
\begin{align*}
I =&\ \left( \begin{matrix}
\textbf{J} & 0\\
0 & \textbf{J}
\end{matrix} \right).
\end{align*}
We now define a second complex structure by changing the orientation of the 
$z_2$ plane, analogously to Example \ref{e:splitGK}.  In particular, let
\begin{align*}
J =&\ \left( \begin{matrix}
\textbf{J} & 0\\
0 & -\textbf{J}
\end{matrix} \right).
\end{align*}
This is easily checked to be integrable, and moreover the map $\gg(z_1,z_2) = 
(\ga z_1, \gb z_2)$, the generator of the group of Deck transformations for the 
Hopf surface, is certainly holomorphic for the complex structure $J$ as well 
as $I$.  It follows that both $I$ and $J$ descend to the quotient surface.  

We claim that if $g$ denotes the metric of (\ref{f:Hopfmetric}), then the triple 
$(g,I,J)$ is generalized K\"ahler. The compatibility of $g$ with $I$ and 
$J$ is immediate.  Note that $I$ and $J$ commute, and so already we obtain the decomposition $d = \del_+ + \delb_+ + \del_- + \delb_-$, and we observe that
\begin{align*}
d^c_I =&\ \i (\delb_I - \del_I) = \i (\delb_- + \delb_+ - \del_- - \del_+),\\
d^c_J =&\ \i (\delb_J - \del_J) = \i (\delb_- - \delb_+ - \del_- + \del_+).
\end{align*}
Thus we compute
\begin{align*}
d^c_I \gw_I =&\ \i (\delb_- + \delb_+ - \del_- - \del_+) \left( \frac{1}{\brs{z_1}^2 + \brs{z_2}^2} \left( dz_1 \wedge d \bz_1 + dz_2 \wedge d \bz_2 \right) \right)\\
=&\ \frac{\i}{(\brs{z_1}^2 + \brs{z_2}^2)^2} \left( (z_2 d \bz_2 - \bz_2 d z_2) dz_1 \wedge d \bz_1 + ( z_1 d \bz_1 - \bz_1 dz_1) \wedge dz_2 \wedge d \bz_2 \right)
\end{align*}
and
\begin{align*}
d^c_J \gw_J =&\ \i (\delb_- - \delb_+ - \del_- + \del_+) \left( \frac{1}{\brs{z_1}^2 + \brs{z_2}^2} \left( dz_1 \wedge d \bz_1 - dz_2 \wedge d \bz_2 \right) \right)\\
=&\ \frac{\i}{(\brs{z_1}^2 + \brs{z_2}^2)^2} \left( (\bz_2 dz_2 - z_2 d\bz_2) \wedge dz_1 \wedge d \bz_1 - (z_1 d \bz_1 - \bz_1 d z_1) \wedge dz_2 \wedge d \bz_2 \right)\\
=&\ - d^c_I \gw_I,
\end{align*}
verifying that $d^c_I \gw_I = - d^c_J \gw_J$.  We leave the computation that $d d^c_I \gw_I = 0$ to Proposition \ref{p:Hopfmetricprop}.  This example is notable since while every generalized K\"ahler manifold with $[I, J] = 0$ admits an integrable splitting of the tangent bundle, the manifold itself need not split as a product according to these leaves.  Note that the $S^1$ leaves are the projection of radial lines through the origin, the tangent space to which does not lie in either $T_{\pm}$.
\end{ex}

A key structural feature of K\"ahler metrics is the $\del\delb$-lemma, which says that locally every K\"ahler metric can be described as $\gw = \i \del \delb f$ for a potential function $f$.  It turns out that generalized K\"ahler metrics with vanishing Poisson tensor also admit local scalar potentials \cite{Streetsvisc}, and potential functions determine natural global deformations of generalized K\"ahler structures of this type \cite{ApostolovGualtieri}.

\begin{lemma} \label{l:GKlocalddbar} Let $(M^{2n}, g, I, J)$ be a generalized K\"ahler manifold with $[I,J] = 0$.  Given $p \in M$ there exist local complex coordinates $(z,w)$ such that $T_+ = \spn \{\frac{\del}{\del z^i}\}$, $T_- = \spn \{\frac{\del}{\del w^i}\}$ and a smooth function $f$ defined in this coordinate chart so that
\begin{align*}
\gw_I = \i \left( \del_+ \delb_+ - \del_- \delb_- \right) f.
\end{align*}
Furthermore, given $f \in C^{\infty}(M)$, suppose
\begin{align*}
\gw_I^f := \gw_I + \i \left( \del_+ \delb_+ - \del_- \delb_- \right) f > 0,
\end{align*}
and let $g^f := \gw_I^f(I, \cdot)$.  Then $(M^{2n}, g^f, I, J)$ is a generalized K\"ahler structure.

\begin{proof} Using Theorem \ref{t:STBstructure} we can choose local complex coordinates $(z, w)$ near $p$ as claimed in the statement.  To derive $f$ we first note that since $d_+ \gw_+ = 0$, on each $w \equiv
\mbox{const}$ plane we can apply the usual $\del\delb$-lemma to obtain a
function $\psi_+(z)$ such that $\i \del_+ \delb_+ \psi_+ = \gw_+$ on that plane.
Since $\gw_+$ is smooth, we can moreover choose these on each slice so that the
resulting function
$\psi_+(z,w)$ is smooth, and thus satisfies $\i \del_+ \delb_+ \psi_+ = \gw_+$ on all of
$U$. 
Arguing similarly on the $z \equiv \mbox{const}$ planes we obtain a function $\psi_-$ such that $\i \del_- \delb_-
\psi_- = \gw_-$ everywhere on $U$.  

Using that $\gw$ is
pluriclosed we obtain
\begin{align*}
0 =&\ \i \del_+ \delb_+ \gw_- + \i \del_- \delb_- \gw_+ = - \del_+ \delb_+
\del_- \delb_- \left(\psi_+ + \psi_- \right).
\end{align*}
We next claim that, as an element in the kernel of the operator $\del_+ \delb_+
\del_- \delb_-$, the function $\psi_+ + \psi_-$, can be expressed as
\begin{align} \label{f:localddbar10}
\psi_+ + \psi_- =&\ \gl_1(z,\bz,w) + \bar{\gl}_1(z,\bz,\bw) + \gl_2(w,\bw,z) +
\bar{\gl}_2(w,\bw,\bz).
\end{align}
To see this we first note that for $\phi$ such that $\del_+
\delb_+ \del_- \delb_- \phi = 0$, any component of $\del_- \delb_- \phi$ can be expressed as
the
real part of a $\delb_+$-holomorphic function, so $(\del_- \delb_- \phi)_{w_i
\bw_j}
 = \mu^{i\bj}(w,\bw,z) + \bar{\mu}^{i\bj}(w,\bw,\bar{z})$, where the
indices on the $\mu$ refer to the fact that each component of the
$\del_-\delb_-$-Hessian can be expressed this way.  It follows that $\gD_- \phi
:=
\i \phi_{,w_i \bw_i}$ is the real part of a $\delb_+$-holomorphic function. 
Applying
the Greens function on each $z$-slice it follows that $\phi$ can be expressed
as
the real part of a $\delb_+$-holomorphic function, up to the addition of an
arbitrary $\delb_-$-holomorphic function.  Thus (\ref{f:localddbar10}) follows.

We claim that $f = \psi_+ - \gl_2 - \bar{\gl}_2$ is the required potential
function.  In particular, since $\i \del_+ \delb_+ \left( \gl_2 + \bar{\gl}_2
\right) = 0$ it follows that $\i \del_+ \delb_+ f = \gw_+$.  Also, we compute
using (\ref{f:localddbar10}),
\begin{align*}
- \i \del_- \delb_- f =&\ - \i \del_- \delb_- \left( - \psi_- + \gl_1 +
\bar{\gl}_1 \right)\\
=&\ \i \del_- \delb_- \psi_-\\
=&\ \gw_-.
\end{align*}
The first claim follows.

To show the second claim we first observe that
\begin{align*}
\gw_J^f := \gw_J + \i \left( \del_+ \delb_+ + \del_- \delb_- \right) f.
\end{align*}
It is clear by construction that $d_{\pm} (\gw^f_I)_{\pm} = d_{\pm} (\gw^f_J)_{\pm} = 0$.  Also, since $I = J$ along the leaves of $T_-$, we have $(d^c_I)_- = (d^c_J)_-$, while $(\gw^f_I)_+ = - (\gw^f_J)_+$, so 
\begin{align*}
d^c_I (\gw_I^f)_+ = (d^c_I)_- (\gw_I^f)_+ = - (d^c_J)_- (\gw_J^f)_+ = - d^c_J (\gw_J^f)_+.
\end{align*}
A similar computation shows that $d^c_I (\gw^f_I)_- = - d^c_J (\gw^f_J)_-$.  It is elementary to check $d d^c_I \gw^f_I = 0$, finishing the proof.
\end{proof}
\end{lemma}

\begin{question} The natural global analogue of this scalar reduction is still open.  Given $(M^{2n}, g, I, J)$ a compact generalized K\"ahler manifold such that $[I,J] = 0$, suppose one has $\phi^1,\phi^2 \in \Lambda^2 T^*$ such that
\begin{enumerate}
\item $\phi^i \in \Lambda^{1,1}_I \cap \Lambda^{1,1}_J$,
\item $d_+ \phi^i_+ = 0 = d_- \phi^i_-$,
\item $d d^c_I \phi^i = 0$,
\item $[\phi^1] = [\phi^2]$,
\end{enumerate}
where the final equality is an equation of Aeppli cohomology classes (cf. Definition \ref{d:pluriclosedcone}) on $(M^{2n}, I)$ (equivalently $J$).  Conjecturally there exists $f \in C^{\infty}(M)$ such that
\begin{align*}
\phi_1 = \phi_2 + \i \left(\del_+ \delb_+ - \del_- \delb_- \right) f.
\end{align*}
\end{question}

\begin{rmk} \label{r:scalarreduction} Lemma \ref{l:GKlocalddbar} shows that a generalized K\"ahler metric is determined by a local potential function in the special case $[I, J] = 0$.  These local potentials conjecturally exist in full generality, with suggestive results appearing in \cite{bischoff, MP2} dealing with most cases.
\end{rmk}

\subsection{Generalized K\"ahler manifolds with nondegenerate Poisson structure} \label{s:GKND}

The other extreme case of generalized K\"ahler geometry occurs when the Poisson 
structure $\gs$ defines a nondegenerate pairing on $T^*$.  In this setting we 
can consider the inverse tensor $\Omega = \gs^{-1}$.  We make this definition for an arbitrary generalized K\"ahler structure, where one interprets it as existing on the locus where $\gs$ is invertible.  

\begin{defn} \label{d:Omegadef} Let $(M^{2n}, g, I, J)$ be a generalized K\"ahler structure, with associated Poisson tensor $\gs$ as in Definition \ref{d:GKPoissondef}.  We set
\begin{align*}
\Omega = \gs^{-1},
\end{align*}
wherever this is well-defined.  Observe as an immediate consequence of Proposition \ref{p:holoPoisson} that
\begin{align*}
\Omega \in \Lambda^{2,0 + 0,2}_I T^* \cap \Lambda^{2,0 + 0,2}_J T^*, \qquad \delb_I \Omega^{2,0} = 0, \qquad \delb_J \Omega_J^{2,0} = 0.
\end{align*}
If $M$ is a manifold where $\Omega$ is globally defined, then the pairing $\gs$ is everywhere nondegenerate, and we call the generalized K\"ahler structure itself \emph{nondegenerate}.
\end{defn}

As it turns out there are further natural symplectic structures associated to a generalized K\"ahler structure of this type.

\begin{lemma} \label{l:NDGC} Let $(M^{2n}, g, I, J)$ be a nondegenerate generalized K\"ahler structure.  Then 
\begin{align*}
F_{\pm} = g \left( I \pm J \right)^{-1}
\end{align*}
define symplectic forms on $M$.  Furthermore, one has $d^c_I \gw_I = db = - d^c_J \gw_J$ where
\begin{align*}
b = F_+ \left(I - J \right).
\end{align*}
\begin{proof} Since $\gs = \tfrac{1}{2} [I, J] g^{-1} $ is nondegenerate, the endomorphism $[I, J] = \tfrac{1}{2} (I - J)(I + J)$ is invertible, and thus so are $I \pm J$.  By Proposition \ref{p:realPoisson} the generalized complex structures $\JJ_{1/2}$ induce Poisson tensors, which by (\ref{f:Gualtierimap}) are given by $(I \pm J) g^{-1}$.  These are also nondegenerate, and thus the tensors $F_{\pm} = g \left( I \pm J \right)^{-1}$ as in the statement are well-defined symplectic forms.  Now set $b = F_+ \left(I - J \right)$.  The claimed equation $d^c_I \gw_I = d b$ follows by differentiating $\gw_I+ bI = - 2 F_+$ and using that $d F_+ = 0$, with $d^c_J \gw_J = - db$ following similarly.
\end{proof}
\end{lemma}

We next construct a natural class of deformations using
$\Omega$-Hamiltonian diffeomorphisms first appearing in \cite{AGG} (cf. \cite{BischoffThesis,bischoff,GotoDef,gualtieri2018generalized,HitchindelPezzo} for further developments), indicating a fundamental difference 
between generalized K\"ahler and classical K\"ahler geometry, namely that the 
basic deformations occur in a nonlinear space.  The starting point is to reduce the construction of nondegenerate generalized K\"ahler structures purely in terms of the holomorphic symplectic structures.  In fact every nondegenerate generalized K\"ahler structure
arises from the description below.

\begin{prop} \label{p:ndGKconst}  Suppose $(I, \Omega_I)$ and $(J,
\Omega_J)$ are two holomorphic symplectic structures on $M$ satisfying
\begin{enumerate}
\item ${\rm Re}(\Omega_I)= {\rm Re}(\Omega_J)$,
\item $\pi_{1,1}^{I} \left( -{\rm
Im}(\Omega_J) \right) $ is positive definite.
\end{enumerate}
Then, setting
\begin{align*}
g = - 2 {\rm Im}(\Omega_J ) I,
\end{align*}
we have that $(g, I, J)$ is a nondegenerate generalized K\"ahler
structure with $\Omega = {\rm Re}(\Omega_I)$.

\begin{proof} Since the real parts of $\Omega_I$ and $\Omega_J$ agree, we can express
\begin{equation*}
\Omega_I = \Omega +\i I \Omega, \qquad \Omega_J = \Omega + \i J \Omega.
\end{equation*}
where $\Omega$ is a real symplectic form on $M$ satisfying $\Omega \in \Lambda^{2,0 + 0,2}_I \cap \Lambda^{2,0 + 0,2}_J$.  Let
\begin{align*}
F := 2\big({\rm Im}(\Omega_I) - {\rm Im}(\Omega_J) \big)= 2(I\Omega - J\Omega).
\end{align*}
This is a real symplectic form which, by condition (2), tames the
complex structure $I$. We thus obtain
\begin{equation}\label{I-decomp}
F(X , IY)= g(X,Y) + b(X,Y),
\end{equation} 
as the decomposition of $- IF$ into its symmetric part, which is positive, and a skew-symmetric part $b$.
Now observe, using that $\Omega$ is type $(2,0) + (0,2)$ with respect to $I$, that
\begin{align*}
F I = 2(I\Omega I - J \Omega I )= 2(\Omega + JI\Omega) = JF.
\end{align*}
It is an \textbf{exercise} to show using this that $g$ is $J$-invariant, and furthermore
\begin{align}\label{J-decomp}
F(X, JY) = g(X,Y) - b(X,Y).
\end{align}
Having constructed the metric and shown that it is biHermitian, we turn to the integrability condition.

First, let $\gw_I=  \pi_{1,1}^I F$ and $\gw_J= \pi_{1,1}^J F$ denote the
K\"ahler forms of $(g, I)$ and  $(g, J)$, respectively.  Note that we can re-express 
\eqref{I-decomp} and \eqref{J-decomp} as
\begin{equation*}
F = \gw_I + Ib = \gw_J -Jb.
\end{equation*}
Since $F$ is closed we obtain
\begin{align*}
d\gw_I = - dIb \in \Lambda^{2,1 + 1,2}_I , \qquad d\gw_J = d Jb \in \Lambda^{2,1 + 1,2}_J.
\end{align*}
Carrying out the pairing with the complex structure and using the type decomposition, it follows that
\begin{align*}
d^c_I \gw_I = db, \qquad d^c_J \gw_J = - db,
\end{align*}
as required.
\end{proof}
\end{prop}

With this description in place we are now ready to exhibit a natural class of
deformations of nondegenerate generalized K\"ahler structures.

\begin{prop} \label{p:nondegvariations}  Let $(M^{2n}, g, I, J)$ be a nondegenerate
generalized K\"ahler
manifold.  Let $f_t \in C^{\infty}(M)$ be a one-parameter family of smooth functions $M$, and let $X_{f_t}$ be
the one-parameter family of
$\Omega$-Hamiltonian vector fields associated to $f_t$, i.e.
\begin{align*}
 df_t =&\ -X_{f_t} \lrcorner\ \Omega.
\end{align*}
Let $\phi_t$ be the one-parameter family of diffeomorphisms of $M$ generated by
$X_f$.  Then for all $t$ such that
\begin{align*}
\pi_{1,1}^{I} \left( -{\rm Im} (\Omega_{\phi_t^*J}) \right) > 0,
\end{align*}
the triple $(I, \phi_t^*
J, \Omega)$ are the complex structures and symplectic structure associated to a
unique nondegenerate generalized
K\"ahler structure.
 \begin{proof} For the given nondegenerate generalized K\"ahler structure, we first construct the $(2,0)$ pieces of $\Omega$ via
 \begin{equation*}
\Omega_I = \Omega +\i I \Omega, \qquad \Omega_J = \Omega +\i J \Omega.
\end{equation*}
For $\phi_t$ as in the statement, let 
$(\Omega_J)_t = \phi_t^*(\Omega_J)$.  The pair $(\phi_t^* J, \phi_t^* \Omega_J)$ is certainly holomorphic symplectic, and since $\phi_t$ is $\Omega$-Hamiltonian, it follows that
 \begin{align*}
  (\Omega_J)_t =&\ \phi_t^* (\Omega + \i J \Omega) = \Omega + \i (\phi_t^* J)
\Omega.
 \end{align*}
Thus ${\rm Re}((\Omega_J)_t) = {\rm Re}(\Omega_I)$ for all $t$, and the result
follows from Proposition \ref{p:ndGKconst}.
\end{proof}
\end{prop}

\begin{rmk}
\begin{enumerate}
\item Observe that while the complex structure $J$ is acted on by a diffeomorphism, the metric $g_t$ will be expressed as per Proposition \ref{p:ndGKconst}) via
\begin{align*}
g_t = - 2 {\rm Im}(\Omega_{\phi_t^* J} ) I,
\end{align*}
which is not a diffeomorphism pullback, and so these deformations are in general nontrivial for $g$.
\item This proposition is directly analogous to the $\del\delb$-lemma in K\"ahler geometry, whereby a single scalar function naturally determines a deformation of K\"ahler structure.  However, as remarked above, the scalar potentials drive a one-parameter family of Hamiltonian diffeomorphisms, and so we do not in general have an explicit description of $g_t$ in terms of $f_t$ alone, rather it is determined by the whole path $f_s,\ 0 \leq s \leq t$.
\item  Even when starting at a hyperK\"ahler structure, as long as the functions $f$ are not constant the deformation arising from Proposition \ref{p:nondegvariations} will yield structures which are not hyperK\"ahler.  This can be checked by showing that the angle function $p = -\tfrac{1}{4n} \tr (I J_t)$ becomes nonconstant, ruling out that the structure is hyperK\"ahler.
\end{enumerate}
\end{rmk}

\subsection{Generalized K\"ahler structures with Poisson 
tensor of mixed rank} \label{s:GKMR}

In general, the rank of the associated Poisson tensor $\gs$ can vary across the manifold.  Our first example of this is a different generalized K\"ahler structure on diagonal Hopf surfaces, where the Poisson tensor $\gs$ is nondegenerate away from a set of real codimension $2$.  

\begin{ex} \label{e:evenGKHopf} Let $(g, I)$ denote the metric and complex structure as in Example \ref{e:Hopfsplit}.  Whereas in that example the complex structure $J$ was easily obtained by changing orientation on the $z_2$ plane, here the second complex structure $J$ is quite different.  Consider the complex $1$-forms on $M$,
\begin{align*}
\mu_1 =&\ \bz_1 dz_1 + z_2 d \bz_2,\\
\mu_2 =&\ \bz_1 dz_2 - z_2 d \bz_1.
\end{align*}
The complex structure $J$ is defined by declaring $T^{0,1}_J = \ker \mu_1 \cap \ker \mu_2$.  This is indeed integrable, as direct computations show that $d \mu_1 \wedge \mu_1 \wedge \mu_2 = d \mu_2 \wedge \mu_1 \wedge \mu_2 = 0$.  The plane spanned by $\mu_1, \mu_2$ is easily checked to be isotropic with respect to $g^{-1}$, hence $g$ is compatible with $J$.  Further direct computations yield
\begin{align*}
\gw_J =&\ \frac{\i}{(\brs{z_1}^2 + \brs{z_2}^2)^2} \left( \mu_1 \wedge \bmu_1 + \mu_2 \wedge \bmu_2 \right).
\end{align*}
Using this the integrability condition $d^c \gw_I = - d^c \gw_J$ can be checked by a tedious but straightforward computation.

We can interpret the standard complex structure $I$ analogously via $T^{0,1}_I = \ker dz_1 \cap \ker dz_2$.  This makes clear that on the elliptic curve $\{z_2 = 0\}$, one has $I = J$, whereas along the elliptic curve $\{z_1 = 0\}$, one has $I = - J$.  It follows that $\gs$ has rank $4$ away from these two elliptic curves, where it has rank $0$.  Furthermore, observe that the Hopf metric is compatible with the standard hypercomplex structure on $\mathbb C^2$, but that the complex structure $J$ is not part of this hypercomplex structure.
\end{ex}

By adapting the methods of the previous subsection we can produce deformations of this structure by taking the vanishing locus of $\gs$ into account explicitly.

\begin{prop} \label{p:mixedrankdeformation} Let $(M^{4n}, g, I, J)$ be a 
generalized K\"ahler
manifold such that $\gs$ is generically nondegenerate, and let
\begin{align*}
K = \{ p \in M\ |\ \rank \gs < 2n \}.
\end{align*}
Let $f_t$ be a one-parameter family of smooth functions on $M \backslash K$, and let $X_{f_t}$ be the one-parameter family of $\Omega$-Hamiltonian vector fields associated to $f_t$, i.e.
\begin{align*}
d f_t = - X_{f_t} \hook \Omega.
\end{align*}
Suppose 
\begin{enumerate}
\item $X_{f_t}$ extends to a smooth one-parameter family of vector fields $\til{X}_{f_t}$ defined on $M$
\item $d d^c_J f_t$ extends to a smooth one-parameter family of tensor fields $W_t$ defined on $M$.
\end{enumerate}
Let $\phi_t$ be the one-parameter family of diffeomorphisms generated by $\til{X}_{f_t}$.  Then for all $t$ such that
\begin{align*}
 \pi_{1,1}^{I} \left( - {\rm Im} (\Omega_{\phi_t^*J}) \right) > 0,
\end{align*}
the triple $(I, \phi_t^*
J, \Omega)$ are the complex structures and symplectic structure associated to a
unique nondegenerate generalized
K\"ahler structure defined on $M \backslash K$, which extends smoothly to a generalized K\"ahler structure on $M$.

\begin{proof}
To show that the metric is well-defined across $K$ we must compute the variation of the tensor $F_t = 2 ({\rm Im} (\Omega_I) - {\rm Im}(\phi_t^* \Omega_J))$.  Using the definitions above and the Cartan formula we see
\begin{align*}
\dt F_t =&\ - 2 L_{X_{f_t}} (J \Omega) = -2 d d^c_J f_t = -2 W_t.
\end{align*}
The above computation a priori only makes sense away from $K$, but since the right hand side is defined smoothly across $K$ we take the equation to be the definition of the variation of $g$ and $b$, using that $F I = g + b$ away from $K$.  The preservation of the generalized K\"ahler conditions holds on the complement of $K$ by Proposition \ref{p:nondegvariations}.  Since this set is dense, it follows that the resulting structure is generalized K\"ahler, as claimed.
\end{proof}
\end{prop}

\begin{rmk} \label{r:HitchinCP2} In \cite{HitchindelPezzo}, Hitchin constructed nontrivial generalized K\"ahler structures on del Pezzo surfaces.  In particular, given $(M^4, J)$ del Pezzo with K\"ahler metric $g$, one chooses a holomorphic Poisson tensor $\gs_J$, and then the construction follows as in Proposition \ref{p:mixedrankdeformation}, deforming the (K\"ahler) generalized K\"ahler structure $(g, J, J)$ using the function $f = \log \brs{\brs{\gs_J^{-1}}}^2$.  One has to check the behavior of $f$ near the vanishing locus of $\gs_J$ to show that the associated diffeomorphisms $\phi_t$ are well-defined.  Using that $c_1 > 0$ it follows that $J$, $\phi_t^* J$ and
\begin{align*}
\gw_J := \pi_{1,1}^{J} \left( - {\rm Im} \left( \Omega_{\phi_t^* J} \right) \right) > 0
\end{align*}
define a non-K\"ahler generalized K\"ahler structure.
\end{rmk}

\chapter{Canonical Metrics in Generalized Complex Geometry} \label{c:CMGCG}

Having defined the basic objects and deformations in pluriclosed and 
generalized K\"ahler geometry, we now turn to the question of motivating and defining notions of canonical 
pluriclosed and generalized K\"ahler structures.  We begin by taking the classical point of view, defining connections associated to a given Hermitian manifold, then analyzing their torsion and curvature.  We use this to define an Einstein-type equation for certain Hermitian metrics, and then show that these structures define generalized Ricci solitons.  We end the chapter by giving examples and classification results for these structures.

\section{Connections, Torsion, and Curvature} \label{s:CTT}
\subsection{Hermitian connections}

\begin{defn} \label{d:Hermconn} Let $(M^{2n}, g, J)$ be a Hermitian manifold.  We say that a connection $\N$ on $TM$ is \emph{Hermitian} if it is compatible with both $g$ and $J$, i.e.
\begin{align*}
\N g \equiv 0, \qquad \N J \equiv 0.
\end{align*}
\end{defn}

For a K\"ahler manifold, owing to the integrability of $J$ and the fact that $d \gw = 0$, it is an \textbf{exercise} to show that the Levi-Civita connection is Hermitian.  When these integrability conditions are weakened however, this will no longer be the case, and one must modify the Levi-Civita connection to obtain one which preserves both $g$ and $J$.  While many Hermitian connections are present in general (cf. \cite{Gauduchonconn}), we will focus on the two most relevant to generalized Ricci flow.

\begin{defn} \label{d:Chernconn} Let $(M^{2n}, g, J)$ be a Hermitian manifold.  Let $\gw(X, Y) = g(J X, Y)$ define the K\"ahler form, and let $\N$ denote the Levi-Civita connection associated to $g$.  The \emph{Chern connection} is defined by
\begin{align*}
\IP{\N^C_X Y, Z} = \IP{\N_X Y, Z} - \tfrac{1}{2} d \gw (JX, Y, Z).
\end{align*}
The \emph{Bismut connection} is defined by
\begin{align*}
\IP{\N^B_X Y, Z} = \IP{\N_X Y, Z} - \tfrac{1}{2} d^c \gw(X,Y,Z).
\end{align*}
Also, it is useful to express the Chern connection in terms of the Bismut torsion.  Using that $d^c \gw(X,Y,Z) = - d \gw(JX,JY,JZ)$ we obtain
\begin{align} \label{f:ChernviaB}
\IP{\N^C_X Y, Z} = \IP{\N_X Y, Z} + \tfrac{1}{2} d^c \gw (X, JY, JZ).
\end{align}
\end{defn}
These two connections admit characterizations in terms of natural conditions on the torsion of possible Hermitian connections.

\begin{lemma} \label{l:BCconn}  Let $(M^{2n}, g, J)$ be a Hermitian manifold.
\begin{enumerate}
\item The Chern connection is the unique Hermitian connection such that
\begin{align*}
T^{1,1} \equiv 0,
\end{align*}
that is, the $(1,1)$ piece of the torsion, interpreted as a section of $\Lambda^2 T^* \otimes T$, vanishes.
\item The Bismut connection is the unique Hermitian connection with torsion tensor $T$ satisfying
\begin{align*}
T g \in \Lambda^3 T^*.
\end{align*}
\end{enumerate}
\begin{proof} \textbf{Exercise}.
\end{proof}
\end{lemma}

In what follows below we will write
\begin{align*}
T = \del \gw \in \Lambda^{2,1}
\end{align*}
for the $(2,1)$ piece of the torsion tensor of the Chern connection, while
\begin{align*}
H = - d^c \gw \in \Lambda^3
\end{align*}
denotes the torsion of the Bismut connection.

\begin{rmk} \label{r:GKBismutremark} Note that of course the connection $\N^B$ is the same as the Bismut connection $\N^+$ of Definition \ref{d:Bismutconn}, with $H = - d^c \gw$.  We keep the notation $\N^B$ in the context of Hermitian geometry to emphasize the distinction between it and the Chern connection, and also since a $\pm$ duality appears in relation to many other constructions as well.  Furthermore, in the context of generalized K\"ahler geometry, we see that the two Bismut connections $\N^{\pm}$ associated to $H = - d^c_I \gw_I = d^c_J \gw_J$ are in fact the Bismut connections associated to the two pluriclosed structures $(g, I)$ and $(g, J)$.  Here again we will adopt specialized notation and name the associated Bismut connections $\N^I$ and $\N^J$, i.e.
\begin{align*}
\IP{\N^I_X Y, Z} = \IP{\N_X Y, Z} - \tfrac{1}{2} d^c_I \gw_I (X,Y,Z) = \IP{\N_X Y, Z} + \tfrac{1}{2} H(X,Y,Z),\\
\IP{\N^J_X Y, Z} = \IP{\N_X Y, Z} - \tfrac{1}{2} d^c_J \gw_J (X,Y,Z) = \IP{\N_X Y, Z} - \tfrac{1}{2} H(X,Y,Z).
\end{align*}
These connections already appeared in Chapter \ref{c:GCG}.
\end{rmk}

\subsection{The Lee form}

In generalized geometry, understanding the torsion is of central importance, and in the setting of complex geometry we have a natural contraction of the torsion tensor, classically called the Lee form.  This tensor plays a central role in understanding the structure of pluriclosed flow and generalized K\"ahler-Ricci flow, entering as a modification of the underlying divergence operator.

\begin{defn} \label{d:Leeform} Let $(M^{2n}, g, J)$ be a Hermitian manifold.  The associated \emph{Lee form} is defined by
\begin{align*}
\theta = - d^* \gw \circ J.
\end{align*}
\end{defn}

\begin{lemma} \label{l:Leeformidentities} Let $(M^{2n}, g, J)$ be a Hermitian manifold.  Then
\begin{enumerate}
\item $\theta(X) = \tfrac{1}{2} H(J X,e_i, J e_i)$,
\item If $n=2$, then $\theta = \star H$.
\end{enumerate}
\begin{proof} We first compute in local coordinates, using that $J$ is parallel with respect to the Bismut connection,
\begin{align*}
\N_i \gw_{jk} =&\ \N_i (g_{jl} J_k^l)\\
=&\ (\N_i - \N^B_i) J_k^l g_{jl}\\
=&\ \tfrac{1}{2} \left( H_{ik}^p J_p^l - H_{ip}^l J_k^p \right) g_{jl}.
\end{align*}
Contracting with respect to the indices $i$ and $j$ and using $- d^* \gw = \tr_g \N \gw$, yields
\begin{align*}
- d^* \gw(X) = \tfrac{1}{2} H(e_i, J e_i, X),
\end{align*}
from which item (1) follows.  Item (2) is left as an \textbf{exercise}.
\end{proof}
\end{lemma}

\subsection{Curvature formulas}

As we will see below, the curvatures of the Levi-Civita, Bismut, and Chern connections are all relevant to understanding canonical metrics in Hermitian geometry, and relationships between them are key.  We record some relevant formulas here.

\begin{defn} \label{d:curvatures} Let $(M^{2n}, g, J)$ be a Hermitian manifold.  We let $R^{B,C}$ denote the $(3,1)$-curvature tensors of the Bismut and Chern connection, respectively.  That is,
\begin{align*}
R^{B,C}(X,Y)Z =&\ \N^{B,C}_X \left(\N^{B,C}_Y Z \right) - \N^{B,C}_Y \left(\N^{B,C}_X Z \right) - \N^{B,C}_{[X,Y]} Z.
\end{align*}
We furthermore adopt the standard notation for the $(4,0)$-curvature tensor,
\begin{align*}
R^{B,C}(X,Y,Z,W) =&\ \IP{R^{B,C}(X,Y)Z,W}.
\end{align*}
It is an elementary \textbf{exercise} to show that 
\begin{align*}
R^B \in \Lambda^2 \otimes \Lambda^{1,1}, \qquad R^C \in \Lambda^{1,1} \otimes \Lambda^{1,1},
\end{align*}
where $R^C$ is type $(1,1)$ in the first indices due to the torsion having no $(1,1)$ component.
\end{defn}

As with the Riemannian curvature tensor, there is only one natural trace (up to sign), yielding the Ricci tensor.  We have already derived a formula for the Bismut Ricci curvature tensor in Proposition \ref{p:Bismutcurvature}.  In the complex setting, due to the complex structure many different traces are possible, and several are relevant to the sequel.

\begin{defn} \label{d:Riccidef} Let $(M^{2n}, g, J)$ be a Hermitian manifold.  Let
\begin{align*}
\rho_{B,C}(X,Y) := \tfrac{1}{2} \IP{ R^{B,C}(X,Y) J e_i, e_i},\\
S_{B,C}(X,Y) := \tfrac{1}{2} \IP{ R^{B,C}(J e_i, e_i) X, Y},
\end{align*}
where $\{e_i\}$ is any orthonormal basis for the tangent space at a given point.  The tensors $\rho_{B,C}$ are referred to as the \emph{(first) Bismut/Chern Ricci curvatures}, and $S_{B,C}$ are the \emph{(second) Bismut/Chern Ricci curvatures}.  Furthermore we define the associated \emph{Bismut/Chern scalar curvatures}
\begin{align*}
s_{B,C} =&\ \rho_{B,C} (J e_i, e_i) = S_{B,C}(J e_i, e_i).
\end{align*}
\end{defn}

The tensors $\rho_{B,C}$ admit a different interpretation which makes clearer their topological significance.  Since the  connections $\N^{B,C}$ preserve $J$, it is clear that they induce connections on $\Lambda^{n,0} T$, the anticanonical line bundle associated to $(M^{2n}, J)$.  The two forms $\rho_{B,C}$ are the curvature tensors of these connections, and as such are closed and determine representatives of $\pi c_1(M, J)$.  Furthermore, as noted above, for K\"ahler manifolds both the Chern and Bismut connections agree with the Levi-Civita connection, and moreover the tensors $\rho$ and $S$ are equal to each other, and satisfy
\begin{align*}
\rho(X,Y) = S(X,Y) = \Rc(J X, Y) = - \Rc(X, JY).
\end{align*}
In the general non-K\"ahler setting the tensors $\rho_{B,C}$ and $S_{B,C}$ are all different, but useful for different purposes.  In the lemmas below we derive relationships between these different Ricci-type tensors needed for the sequel (cf. \cite{Ivanovstring} for related formulas, although with different conventions).

\begin{lemma} \label{l:BismutRicci} Let $(M^{2n}, g, J)$ be a pluriclosed 
structure.  Then
\begin{gather} \label{f:BismutRicci}
 \begin{split}
  \rho_B(X,Y) =&\ d d^* \gw(X, Y) + \rho_C(X,Y),\\
  \rho_B(X,Y) =&\ - \Rc^B(X,JY) - \N^B_X \theta (JY).
 \end{split}
\end{gather}
\begin{proof} To compute the first item of (\ref{f:BismutRicci}) it suffices to compute the induced connection on the canonical bundle.  To that end we compute, using (\ref{f:ChernviaB}) and Lemma \ref{l:Leeformidentities},
\begin{align*}
\IP{(\N^B - \N^C)_X J e_i, e_i} =&\ - \tfrac{1}{2} d^c \gw(X, J e_i, e_i) - \tfrac{1}{2} d^c \gw (X, J J e_i, J e_i)\\
=&\ - d^c \gw(X, J e_i, e_i)\\
=&\ 2 d^* \gw(X).
\end{align*}
Using this and the definitions of $\rho_{B,C}$ gives the first item of (\ref{f:BismutRicci}).  For the second item of (\ref{f:BismutRicci}) we first use the Bianchi identity for $R^B$ with $H = - d^c \gw$, specifically combining (\ref{f:Bianchi1}) and (\ref{f:Bianchi3}), to obtain
\begin{gather*}
\begin{split}
\sum_{\gs(X,Y,Z)} & R^B(X,Y,Z,U)\\
=&\ \sum_{\gs(X,Y,Z)} \left\{ H(H(X,Y),Z,U) + (\N^B_X 
H)(Y,Z,U) \right\}\\
=&\ \sum_{\gs(X,Y,Z)} \left\{ H(H(X,Y),Z,U) - 2 g(H(X,Y),H(Z,U)) \right\} + (\N^B_U H)(X,Y,Z)\\
=&\ (\N^B_U H)(X,Y,Z) - \sum_{\gs(X,Y,Z)} g(H(X,Y),H(Z,U)).
\end{split}
\end{gather*}
Adding back $\tfrac{1}{2}$ times (\ref{f:Bianchi3}) and using $d H = 0$ yields
\begin{gather} \label{f:BismutRicci10}
\sum_{\gs(X,Y,Z)} R^B(X,Y,Z,U) = \tfrac{1}{2} \left(\sum_{\gs(X,Y,Z)} (\N^B_X H)(Y,Z,U) + (\N_U^B H)(X,Y,Z) \right).
\end{gather}
Tracing the left hand side of this over $Z$ and $U$ with respect to $J$ and using that $R^B$ takes values in $(1,1)$ forms, we obtain
\begin{align*}
R^B(X,Y,e_i, J e_i)&  + R^B(Y,e_i,X,J e_i) + R^B(e_i,X,Y,J e_i)\\
=&\ - 2 \rho_B(X,Y) - R^B(Y,e_i,JX,e_i) - R^B(e_i,X,JY,e_i)\\
=&\ - 2 \rho_B(X,Y) + \Rc^B(Y,JX) - \Rc^B(X,JY).
\end{align*}
Also, the right hand side of (\ref{f:BismutRicci10}), when traced over $Z$ and $U$ with respect to $J$ yields
\begin{gather*}
\begin{split}
\tfrac{1}{2} & \left( (\N_X^B H)(Y,e_i,J e_i) + (\N_Y^B H)(e_i, X, J e_i) + (\N^B_{e_i} H)(X,Y,J e_i) + (\N^B_{J e_i} H)(X,Y,e_i) \right)\\
=&\ \tfrac{1}{2} (\N_X^B H)(Y,e_i,J e_i) + \tfrac{1}{2} (\N_Y^B H)(e_i, X, J e_i)\\
=&\ \N^B_X \theta(JY) - \N^B_Y \theta(JX),
\end{split}
\end{gather*}
 where the last line follows from Lemma \ref{l:Leeformidentities}.  Putting these computations together yields
\begin{gather} \label{f:BismutRicci20}
 \begin{split}
- 2 \rho_B(X,Y) + \Rc^B(Y,JX) - \Rc^B(X,JY) = - \N_Y^B \theta (JX) + \N_X^B \theta (JY).
 \end{split}
 \end{gather}
 On the other hand we can contract (\ref{Bismutcurvature}) to obtain
 \begin{gather} \label{f:BismutRicci30}
 \begin{split}
 \Rc^B(Y, JX) + \Rc^B(X, JY) =&\ \Rm^B(e_i, Y, JX, e_i) + \Rm^B(e_i, X, JY, e_i)\\
 =&\ - \Rm^B(e_i, Y, X, J e_i) - \Rm^B(e_i, X, Y, J e_i)\\
 =&\ \Rm^B(X, J e_i, e_i, Y) - \Rm^B(e_i, Y, X, J e_i)\\
 =&\ \tfrac{1}{2} \N^B_X H(J e_i, e_i, Y) - \tfrac{1}{2} \N^B_{J e_i} H(X,e_i, Y)\\
 &\ - \tfrac{1}{2} \N^B_{e_i} H(Y,X,J e_i) + \tfrac{1}{2} \N^B_Y H(e_i, X, J e_i)\\
 =&\ - \N^B_X \theta(JY) - \N^B_Y \theta(JX).
 \end{split}
 \end{gather}
 Combining (\ref{f:BismutRicci20}) and (\ref{f:BismutRicci30}) gives the second item of (\ref{f:BismutRicci}).
\end{proof}
\end{lemma}

\begin{prop} \label{p:BismutRicci} Let $(M^{2n}, g, J)$ be a pluriclosed 
structure.  Let $d^{\N^B}$ denote the exterior derivative associated to the Bismut connection, and let 
\begin{gather} \label{f:Qdef}
Q(X,Y) = \IP{X \hook T, Y \hook T},
\end{gather}
where $T$ denote the torsion of the Chern connection.  Then
\begin{gather} \label{f:11BismutRicci}
 \begin{split}
  \rho^{1,1}_B =&\ S_C - Q,\\
  \rho^{1,1}_B(\cdot, J \cdot) =&\ \Rc - \tfrac{1}{4} H^2 + \tfrac{1}{2} L_{\theta^{\sharp}} g,\\
  \rho_B^{2,0 + 0,2}(\cdot, J \cdot) =&\ - \tfrac{1}{2} d^* H + \tfrac{1}{2} d^{\N^B} \theta.
 \end{split}
\end{gather}

\begin{proof} The first item is left as an \textbf{exercise}.  For the second item we compute using Lemma \ref{l:BismutRicci},
\begin{align*}
\rho^{1,1}_B(X, J Y) =&\ \tfrac{1}{2} \left( \rho_B(X, J Y) - \rho_B(J X, Y) \right)\\
=&\ \tfrac{1}{2} \left( \Rc^B(X,Y) + \N^B_{X} \theta (Y) + \Rc^B(J X,J Y) + \N^B_{JX} \theta(J Y) \right).
\end{align*}
Using line (\ref{f:BismutRicci30}) with $X$ replaced by $JX$ yields
\begin{align*}
\Rc^B(JX, JY) =&\ \Rc^B(Y,X) - \N^B_{JX} \theta(JY) + \N_Y^B \theta(X).
\end{align*}
Plugging this in above yields
\begin{align*}
\rho_B^{1,1}(X, J Y) =&\ \tfrac{1}{2} \left( \Rc^B(X, Y) + \Rc^B(Y, X) + \N_X^B \theta(Y) + \N_Y^B \theta(X) \right)\\
=&\ \left(\Rc - \tfrac{1}{4} H^2 + \tfrac{1}{2} L_{\theta^{\sharp}} g \right)(X,Y),
\end{align*}
where the last line follows using Proposition \ref{p:Bismutcurvature} and the skew-symmetry of the torsion of $\N^B$.

To show the last claim we again use Lemma \ref{l:BismutRicci} and line (\ref{f:BismutRicci30}) and compute
\begin{align*}
\rho_B^{2,0 + 0,2}(X, Y) =&\ \tfrac{1}{2} \left( \rho^B(X,Y) - \rho^B(JX, JY) \right)\\
=&\ - \tfrac{1}{2} \left( \Rc^B(X, JY) + \N_X^B \theta(JY) + \Rc^B(JX, Y) + \N^B_{JX} \theta(Y) \right)\\
=&\ \tfrac{1}{2} \left( \Rc^B(Y, JX) - \Rc^B(JX, Y) + \N_Y^B \theta(JX) - \N^B_{JX} \theta(Y) \right)\\
=&\ \tfrac{1}{2} d^* H(JX, Y) - \tfrac{1}{2} d^{\N^B} \theta(JX, Y),
\end{align*}
where the last line follows using Proposition \ref{p:Bismutcurvature}.  The final claimed formula follows easily from this.
\end{proof}
\end{prop}

\section{Canonical metrics in complex geometry}

In this section we will give a define a notion of canonical pluriclosed metric using the curvature of the associated Bismut connection.  Recall that in Chapter \ref{c:GCC} we began by investigating the curvature of the Bismut connection, then determining how the Bismut Ricci curvature fits into generalized geometry, eventually arriving at the generalized Einstein-Hilbert functional which has generalized Einstein structures as critical points.  We proceed by using the Bismut curvature associated to a pluriclosed structure to define a kind of Einstein-type condition we call Bismut Hermitian-Einstein.  We will quickly see that these structures are in fact generalized Ricci solitons.  The relationship of this condition to generalized complex geometry will become clear through the rest of this chapter and the next.  In particular, in the setting of generalized K\"ahler geometry we will see the relationship between the Bismut Ricci curvature and the spinor formulation of the associated generalized complex structures.  Also in Chapter \ref{c:GFCG} (cf. Remark \ref{r:BRFasHYM}) we will see that Bismut Hermitian-Einstein structures naturally define an associated Hermitian metric on $T^{1,0} \oplus T^*_{1,0}$ which is Hermitian Yang-Mills.

\subsection{Bismut Hermitian-Einstein metrics}

\begin{defn} \label{d:HermBismutRicci} Let $(M^{2n}, g, J)$ be a pluriclosed 
structure.  We say that it is \emph{Bismut Hermitian-Einstein} if 
\begin{align*}
\rho_B = 0.
\end{align*}
\end{defn}

\begin{rmk} Bismut Hermitian-Einstein structures have been studied before in mathematics and physics literature (cf. for instance \cite{Ivanovstring} and references therein).  We resist the terminology `Bismut Ricci flat' as this is reserved for the vanishing of the Ricci tensor associated to the Bismut connection $\N^+$ as in Chapter \ref{c:GCC}.  Furthermore, as we will see in Proposition \ref{p:einsteinequiv} below these structures are actually \emph{solitons}.

More generally it is possible to allow for an Einstein constant in this definition, and consider either the equation $\rho_B = \gl \gw$ or $\rho_B^{1,1} = \gl \gw$.  The first automatically implies $\gw$ is K\"ahler if $\gl \neq 0$ by taking exterior derivative.  No known solutions to the second equation are known which are not already K\"ahler, and in fact rigidity results are known (cf. \cite[Proposition 3.5]{Streetssolitons}), in line with Proposition \ref{p:GRSprop} above.  Even if one considers the more general equation 
$\rho_B^{1,1} = f \gw$ for a smooth function $f$, then at least on a compact 
manifold the function $f$ is forced to be constant.  To see this, apply 
$\gw^{n-2} \wedge \left( \i \del \delb \right)$ to both sides of the equation 
and use that $\i \del \delb \rho_B^{1,1} = 0$ to obtain a strictly elliptic 
linear equation for $f$, which forces $f$ to be constant by the maximum 
principle.  For these reasons we focus here only on the equation $\rho_B = 0$.
\end{rmk}

\begin{lemma} Let $(M^{2n}, J)$ be a complex manifold.  The Bismut Hermitian-Einstein equation is a strictly elliptic equation for a pluriclosed metric $g$.
\begin{proof} Using the first equation of Proposition \ref{p:BismutRicci}, we can rewrite the $(1,1)$ piece of the Bismut Hermitian-Einstein equation as
\begin{align*}
S_C - Q = 0.
\end{align*}
Since $Q$ is a first-order differential operator in $g$, and $\gw$ is zeroth order in $g$, we only need to consider the quantity $S_C$.  In local complex coordinates the Chern connection has Christoffel symbols
\begin{align*}
\gG_{ij}^k =&\ g^{\bl k} \del_i g_{j \bl},
\end{align*}
and then one directly computes the curvature
\begin{align*}
(S_C)_{i \bj} =&\ \i g^{\bl k} R^C_{k \bl i \bj} = - \i g^{\bl k} \left( g_{\bj m}\del_{\bl} \gG_{k i}^m \right) = - \i g^{\bl k} \del_{\bl} \del_k g_{i \bj} + \del g^{\star2},
\end{align*}
which is manifestly a strictly elliptic operator for $g$, and the lemma follows.
\end{proof}
\end{lemma}

The next proposition indicates the key linkage between the Bismut Hermitian-Einstein condition and the generalized Ricci flow, namely that Bismut Hermitian-Einstein structures are automatically steady generalized Ricci solitons.  

\begin{prop} \label{p:einsteinequiv} Let $(M^{2n}, g, J)$ be Bismut Hermitian-Einstein.  Then $(g, H)$ is a steady generalized Ricci soliton with $X = \theta^{\sharp}, B = d \theta - i_{\theta^{\sharp}} H$.
\begin{proof} Pairing the $(1,1)$ part of the Bismut Hermitian-Einstein equation with $J$ and using the second equation of (\ref{f:11BismutRicci}) gives
\begin{align*}
\Rc - \tfrac{1}{4} H^2 + \tfrac{1}{2} L_{\theta^{\sharp}} g = 0,
\end{align*}
which is the metric component of the steady generalized Ricci soliton equation with vector field $X = \theta^{\sharp}$. Using the vanishing of $\rho_B^{2,0 + 0,2}$ and the third equation of (\ref{f:11BismutRicci}) we obtain
\begin{align*}
0 =&\  d^* H -  d^{\N^B} \theta =  d^* H -  d \theta +  i_{\theta^{\sharp}} H.
\end{align*}
Thus we have verified the equations (\ref{f:solitonexp}) of Proposition \ref{p:GRSprop} for $X = \theta^{\sharp}$ and $B = d \theta - i_{\theta^{\sharp}} H$, as claimed.
\end{proof}
\end{prop}

\subsection{Bismut Ricci curvature in generalized K\"ahler geometry}

In this subsection we derive explicit formulas for the Bismut Ricci curvature of a generalized K\"ahler manifold.  A central role is played by a natural scalar quantity associated to a generalized K\"ahler structure:

\begin{defn} Let $(E, [,], \IP{,})$ be an exact Courant algebroid.  Suppose $(\JJ_1, \JJ_2)$ is a generalized K\"ahler structure defined by pure spinors $\psi_1, \psi_2$.  The \emph{Bismut Ricci potential} is
\begin{align*}
\Phi = \log \frac{ (\psi_1, \bar{\psi}_1)}{(\psi_2, \bar{\psi}_2)},
\end{align*}
where $(,)$ is the Mukai pairing (cf. \S \ref{ssec:spinors}).
\end{defn}

We note that Gualtieri defined a special class of generalized K\"ahler structure called `generalized Calabi-Yau metric geometry' \cite[Definition 6.40]{GualtieriThesis} which in our notation is expressed
\begin{align*}
\Phi \equiv \mbox{constant}.
\end{align*}
As we will see below, in general this quantity determines both the Lee forms and the Bismut Ricci curvature of a given generalized K\"ahler structure.  Furthermore, in Chapter \ref{c:GFCG} we will show that in a natural sense this scalar quantity determines the generalized K\"ahler-Ricci flow lines.

To begin our computations, we recall the fundamental transgression formula, whose proof is left as an \textbf{exercise}.

\begin{prop} \label{p:Cherntrans}  Given $(M^{2n}, J)$ a complex manifold, $E \to M$ a holomorphic vector bundle, and Hermitian metrics $h_1, h_2$ on $E$ with associated Chern curvature tensors $\Omega_1, \Omega_2$, one has
\begin{align*}
\tr \left(\Omega_1 - \Omega_2 \right) = - \i \del \delb \log \frac{\det h_1}{\det h_2}.
\end{align*}
\end{prop}

\noindent We first the scalar reduction for the K\"ahler-Einstein equation, and how it fits into generalized K\"ahler geometry.

\begin{rmk} Suppose $(M^{2n}, g, J)$ is a K\"ahler manifold.  As noted above in this case the Bismut and Chern connections coincide.  Using the transgression formula, in local complex coordinates we can express
\begin{align*}
\rho_C =&\ - \i \del \delb \log \det g.
\end{align*}
If $M$ admits a holomorphic volume form $\Omega \in \Lambda^{n,0}$, one instead obtains the global formula
\begin{align} \label{f:Chernricci}
\rho_C =&\ - \i \del \delb \log \frac{ \gw^n}{\Omega \wedge \bar{\Omega}}.
\end{align}
In case $M$ is compact, by the maximum principle we see that $\rho_C \equiv 0$ if and only if
\begin{align*}
\frac{\gw^n}{\Omega \wedge \bar{\Omega}} \equiv \mbox{const}.
\end{align*}
This equation admits a natural interpretation in terms of generalized K\"ahler geometry.  First, we can set $I = J$ and then it follows directly that $(g, I, J)$ is generalized K\"ahler.  Recall that the generalized complex structures associated to a K\"ahler structure interpreted this way are the structures $\JJ_J, \JJ_{\gw}$ from Examples \ref{e:complexGC}, \ref{e:symplecticGC}, naturally induced by the two global spinors $\Omega$ and $e^{\i \gw}$.  Furthermore, the scalar equation above can be expressed as equality of the two associated spinor norms, up to scaling.  That is, we see that $\rho_C \equiv 0$ if and only if
\begin{align*}
\Phi = \log \frac{\gw^n}{\Omega \wedge \bar{\Omega}} = \log \frac{\left(\gw, \bar{\gw} \right)}{\left(\Omega, \bar{\Omega} \right)} \equiv \mbox{const}.
\end{align*}
Moreover, we obtain the global expression
\begin{align*}
\rho_B^I = \rho_C = - \tfrac{1}{2} d J d \Phi,
\end{align*}
with the same equation holding with the roles of $I$ and $J$ exchanged.  This same formula will determine the Bismut Ricci tensor in general, as we will see in examples below.
\end{rmk}

We now turn to the case of commuting generalized K\"ahler structures, adopting the notation of \S \ref{ss:GKVPS}.  Before deriving the formula for the Bismut Ricci curvature we make an observation about the associated Lee forms which we need in the sequel.

\begin{lemma} \label{l:CGKLeeform} Let $(M^{2n}, g, I, J)$ be a generalized K\"ahler manifold such that $[I,J] = 0$.  Letting $\theta_I, \theta_J$ denote the Lee forms with respect to $(g, I)$ and $(g, J)$, then one has
\begin{align*}
\theta_I = \theta_J.
\end{align*}
\begin{proof} First we choose adapted complex coordinates for the complex structure $I$ as in Lemma \ref{l:GKlocalddbar}.  These coordinates can be furthermore chosen so that $\del_{z_i} \gw^+_{j \bk} = \del_{w_i} \gw^-_{j \bk} = 0$.  It follows using a computation of the Christoffel symbols in these coordinates that
\begin{align*}
(\theta_I)_{\bk} = \left(I d^* \gw_I \right)_{\bk} =&\ \i \left(g^{\bz_j z_i} \del_{\bz_j} (\gw_I^+)_{z_i \bk} + g^{\bw_j w_i} \del_{\bw_j} (\gw_I^-)_{w_i \bar{k}} \right).
\end{align*}
To compute $\theta_J$, we use the same coordinates, and note that $\gw_J^+ = - \gw_I^+, \gw_J^- = \gw_I^-$, whereas the action of $J$ agrees with that of $I$ on $T_-$, and differs by a sign on $T_+$.  It follows that
\begin{align*} 
(\theta_J)_{\bk} =&\ \i \left( - g^{\bz_j z_i} \del_{\bz_j} (\gw_J^+)_{z_i \bk} + g^{\bw_j w_i} \del_{\bw_j} (\gw_J^-)_{w_i \bar{k}} \right) = (\theta_I)_{\bk}.
\end{align*}
\end{proof}
\end{lemma}

\begin{defn} \label{d:splitChernclasses} Let $(M^{2n}, g, I, J)$ be a generalized K\"ahler manifold with $[I,J] = 0$.  Define the \emph{split canonical bundles} by
\begin{align*}
K_{\pm} = \Lambda^{\rank T_{\pm}} (\Lambda_{\pm}^{1,0}).
\end{align*}
As holomorphic line bundles over $M$, these have first Chern classes with respect to the complex structure $I$, and we refer to these as $c_1^{\pm} := c_1(K_{\pm}, I)$.  Given a Hermitian metric $g$ on $M$, we obtain induced metrics on $K_{\pm}$, and therefore representatives of $c_1^{\pm}$.  We refer to these curvature representatives as
\begin{align*}
\rho^{\pm} = \rho_C(g, K_{\pm}, I).
\end{align*}
These curvature tensors are real two-forms of type $(1,1)$ with respect to $I$.  It will be useful for us to refer to the different components of these curvature operators with respect to the decomposition induced by $T = T_+ \oplus T_-$.  In particular we set
\begin{align*}
\rho^{+}_{\pm} = \pi_{\Lambda^{1,1} T^*_{\pm}} \rho^+, \qquad \rho^{-}_{\pm} = \pi_{\Lambda^{1,1} T^*_{\pm}} \rho^-.
\end{align*}
\end{defn}

\begin{prop} \label{p:CGKBismutcurvature} Let $(M^{2n}, g, I, J)$ be a generalized K\"ahler manifold with $[I,J] = 0$. 
Then
\begin{align} \label{f:CGKBismut10}
(\rho_B^I)^{1,1} = \rho_+^+ - \rho^+_- + \rho_-^- - \rho^-_+.
\end{align}
Furthermore, given $h = h_+ \oplus h_-$ a Hermitian metric, define
\begin{align*}
P(h) := \rho_+(h_+) - \rho_-(h_+) + \rho_-(h_-) - \rho_+(h_-),
\end{align*}
where $\rho(h_{\pm})$ denotes the representative of $c_1(K_{\pm}, I)$ as above, and $\rho^{\pm}(h_{\pm})$ denote the projections as in Definition \ref{d:splitChernclasses}.  Given $\til{h} = \til{h}_{+} \oplus \til{h}_-$ another Hermitian metric, one has
\begin{align} \label{f:commutingtransgression}
P(h) - P(\til{h}) = - \i \left( \del_+ \delb_+ - \del_- \delb_- \right) \log \frac{(\gw^{h}_+)^{\wedge k} \wedge (\gw^{\til{h}}_-)^l}{(\gw^{\til{h}}_+)^{\wedge k} \wedge (\gw^{h}_-)^{\wedge l}}.
\end{align}
\begin{proof} The proof of (\ref{f:CGKBismut10}) is left as a fairly lengthy but straightforward \textbf{exercise}.  The simplest way is to derive a general formula for $\rho_B^{1,1}$ in arbitrary complex coordinates, choose split coordinates as per Lemma \ref{l:GKlocalddbar}, and then compare against the usual formula for the curvature of a determinant line bundle.  See \cite[Proposition 3.2]{StreetsSTB} for details.

To check (\ref{f:commutingtransgression}), we use the definitions as well as the transgression formula for the first Chern class (Proposition \ref{p:Cherntrans}) to compute
\begin{align*}
P(h) - P(\til{h}) =&\ \left\{ \rho_+(h_+) - \rho_+(\til{h}_+) \right\} - \left\{ \rho_-(h_+) - \rho_-(\til{h}_+) \right\}\\
&\ + \left\{ \rho_-(h_-) - \rho_-(\til{h}_-) \right\} - \left\{ \rho_+(h_-) - \rho_+(\til{h}_-) \right\}\\
=&\ \pi_{\Lambda^{1,1}_{T^*_+}} \left\{ - \i \del \delb \log \frac{(\gw_+^h)^{\wedge k}}{(\gw^{\til{h}}_+)^{\wedge k} }\right\} - \pi_{\Lambda^{1,1}_{T^*_-}} \left\{ - \i \del \delb \log \frac{(\gw_+^h)^{\wedge k}}{(\gw^{\til{h}}_+)^{\wedge k} } \right\}\\
&\ + \pi_{\Lambda^{1,1}_{T^*_-}} \left\{ - \i \del \delb \log \frac{(\gw_-^h)^{\wedge l}}{(\gw^{\til{h}}_-)^{\wedge l} } \right\} - \pi_{\Lambda^{1,1}_{T^*_+}} \left\{ - \i \del \delb \log \frac{(\gw_-^h)^{\wedge l}}{(\gw^{\til{h}}_-)^{\wedge l} } \right\}\\
=&\ - \i \left( \del_+ \delb_+ - \del_- \delb_- \right) \log \frac{(\gw^{h}_+)^{\wedge k} \wedge (\gw^{\til{h}}_-)^l}{(\gw^{\til{h}}_+)^{\wedge k} \wedge (\gw^{h}_-)^{\wedge l}},
\end{align*}
as claimed.
\end{proof}
\end{prop}

\begin{rmk} This curvature formula reveals an interesting local description of the Bismut Ricci curvature in adapted local complex coordinates.  In particular, implicit in the computations underlying Proposition \ref{p:CGKBismutcurvature} is the local coordinate description
\begin{align*}
(\rho_B^I)^{1,1} = - \i \left(\del_+ \delb_+ - \del_- \delb_- \right) \log \frac{\det \gw_+}{\det \gw_-},
\end{align*}
where $\gw_{\pm}$ are the restrictions of the K\"ahler form $\gw_I$ to $\Lambda^{1,1} T^*_{\pm}$ as above.  This  is an interesting extension of the classical description of the Ricci curvature of a K\"ahler manifold as $-\i \del \delb \log \det \gw$.  
\end{rmk}

The formula can in fact be globalized in case the split canonical bundles admit global sections.  We do this next, summarizing the computations above.

\begin{prop} \label{p:BRcomm} Let $(M^{2n}, g, I, J)$ be a generalized K\"ahler manifold with $[I,J] = 0$.  Suppose there exist global nonvanishing holomorphic sections
\begin{align*}
\Omega_{\pm} \in \Lambda^{\rank T_{\pm}} \left(\Lambda^{1,0}_{\pm} \right).
\end{align*}
Then one has
\begin{align*}
\Phi =&\ \log \frac{\gw_-^{\wedge l} \wedge \Omega_+ \wedge \bar{\Omega}_+}{\gw_+^{\wedge k} \wedge \Omega_- \wedge \bar{\Omega}_-},
\end{align*}
and
\begin{align*}
(\theta_I - \theta_J)^{\sharp} =&\ \gs d \Phi,\\
\rho_B^I =&\ - \tfrac{1}{2} d J d \Phi.
\end{align*}
\begin{proof} A computation using (\ref{f:Gualtierimap}) shows that the two associated generalized complex structures are generated by the pure spinors
\begin{align*}
\psi_1 = e^{\i \gw_-} \wedge \Omega_+, \qquad \psi_2 = e^{- \i \gw_+} \wedge \Omega_-.
\end{align*}
It follows that $\Phi$ takes the claimed form.  We note that by Lemma \ref{l:CGKLeeform}, we have $\theta_J - \theta_I = 0$, and also $\gs = 0$ by definition, so the claimed torsion formula follows immediately.  (We have chosen to express the result of Lemma \ref{l:CGKLeeform} in this seemingly overcomplicated way to draw a similarity to the nondegenerate case below).  Furthermore, by Proposition \ref{p:CGKBismutcurvature} it follows that
\begin{align*}
(\rho_B^I)^{1,1} =&\ - \i \left( \del_+ \delb_+ - \del_- \delb_- \right) \log \frac{\gw_-^{\wedge l} \wedge \Omega_+ \wedge \bar{\Omega}_+}{\gw_+^{\wedge k} \wedge \Omega_- \wedge \bar{\Omega}_-}\\
=&\ - \tfrac{1}{2} \left( d J d \Phi \right)^{1,1}_I.
\end{align*}
Checking the $(2,0) + (0,2)$ piece of the Bismut Ricci curvature formula is left as an \textbf{exercise}.
\end{proof}
\end{prop}

Next we address the nondegenerate setting, where it turns out that once again the quantity $\Phi$ determines several key formulas relevant to the Lee forms and Bismut Ricci curvature, and in fact will determine the generalized K\"ahler-Ricci flow lines in this setting (cf. \S \ref{s:GKRF}).

\begin{prop} \label{p:NDRiccipotential} Let $(M^{2n}, g, I, J)$ be a nondegenerate generalized K\"ahler manifold.  Then one has
\begin{align*}
\Phi =&\ \log \frac{\det (I + J)}{\det (I - J)}.
\end{align*}
and
\begin{align*}
\left(\theta_I - \theta_J\right)^{\sharp} =&\ \gs d \Phi,\\
\rho_B^I =&\ - \tfrac{1}{2} d J d \Phi.
\end{align*}
\begin{proof} 
We recall as in Lemma \ref{l:NDGC} that for a nondegenerate generalized K\"ahler structure both $I \pm J$ are invertible, and the Poisson tensors $g^{-1}(I \pm J)$ associated to the generalized K\"ahler structures (cf. Proposition \ref{p:realPoisson}) are also nondegenerate, yielding symplectic forms
\begin{align*}
F_{\pm} = - g(I \pm J)^{-1}.
\end{align*}
Furthermore, it is an \textbf{exercise} to check that the spinors defining the associated generalized complex structures are $e^{\i F_{\pm}}$, thus yielding the claimed formula for $\Phi$.  The torsion and Bismut Ricci curvature formulas are direct but lengthy computations and we leave the details to \cite{ASnondeg} (n.b. the conventions for $\gs$ there differ from those here).
\end{proof}
\end{prop}

By comparing the discusion of the Ricci curvature in the K\"ahler setting, and the results of Proposition \ref{p:BRcomm} and  \ref{p:NDRiccipotential}, the general formula is now apparent.  We give the statement here, the proof of which can be obtained by a combination of the different cases above.  This can be thought of as a generalization of the classical transgression formula for the Ricci curvature on a K\"ahler manifold to the generalized K\"ahler setting.

\begin{prop} \label{p:BisRiccifinal} Let $(E, \IP{,}, [,])$ be an exact Courant algebroid, and suppose $(\JJ_1, \JJ_2)$ is a generalized K\"ahler structure such that each $\JJ_i$ is globally defined by a $\slashed d_0$-closed pure spinor $\psi_i$.  Then
\begin{align*}
\left(\theta_I - \theta_J\right)^{\sharp} =&\ \gs d \Phi,\\
\rho_B^I =&\ - \tfrac{1}{2} d J d \Phi.
\end{align*}
\end{prop}

\section{Examples and rigidity results}

We begin this section by explicitly describing the geometry of the Hopf/Boothby metric (\ref{f:Hopfmetric}), previously also encountered in Example \ref{ex:GRicciflatHopf} .  We will show that it is in fact isometric to 
$g_{S^3} \oplus g_{S^1}$, the product of canonical spherical metrics on $S^3$ 
and $S^1$, and is Bismut Hermitian-Einstein.  Remarkably, in complex dimension $2$ this is the only compact non-K\"ahler 
example of a Bismut Hermitian-Einstein metric, as we will prove in Theorem \ref{t:Hopfrigidity} below (cf. \cite{GauduchonIvanov, GauduchonWeyl}).  Even more remarkably, this metric is generalized K\"ahler, in two different ways!  As exhibited in Example \ref{e:Hopfsplit}, the Hopf metric is compatible with the complex structure determined by changing the orientation of one of the complex planes on $\mathbb C^2 \backslash \{0\}$, yielding a generalized K\"ahler structure.  Also, in Example \ref{e:evenGKHopf} we obtained another complex structure compatible with the Hopf metric which yields a generalized K\"ahler structure nondegenerate away from the two elliptic curves $z_1 = 0, z_2 = 0$.

\begin{prop} \label{p:Hopfmetricprop} Consider the metric $g$ defined on $\mathbb C^2 
\backslash \{0\}$ via
\begin{align} \label{f:Hopfmetric2}
g(X,Y)_{(z_1,z_2)} = \frac{g_{E}(X,Y)}{\brs{z_1}^2 + \brs{z_2}^2},
\end{align}
where $g_{E}$ is the standard Euclidean metric on $\mathbb C^2$.  Let 
$\gg(z_1,z_2) = (\ga z_1, \gb z_2)$, $\brs{\ga} = \brs{\gb} > 1$ be the 
generator of the group of Deck transformations defining a diagonal Hopf surface 
$(S^3 \times S^1, J)$ as in Definition \ref{d:Hopf}.  Then $g$ is 
$\gg$-invariant, and so descends to the quotient to define a metric, also 
denoted $g$, such that
\begin{enumerate}
\item $\i \del \delb \gw = 0$,
\item $\N \theta = 0$,
\item $\Rc^B = 0$,
\item $\rho_B = 0$.
\end{enumerate}
\begin{proof} The invariance under $\gg$ is easily computed using $\brs{\ga} = 
\brs{\gb}$.  Next we can compute
\begin{align*}
\gw =&\ \frac{\i}{\brs{z_1}^2 + \brs{z_2}^2} \left( dz_1 \wedge d \bz_1 + dz_2 
\wedge d \bz_2 \right),
\end{align*}
hence
\begin{align*}
d \gw =&\ - \frac{\i}{\left(\brs{z_1}^2 + \brs{z_2}^2 \right)^2} d \left( 
\brs{z_1}^2 + \brs{z_2}^2 \right) \wedge \left( dz_1 \wedge d \bz_1 + dz_2 
\wedge d \bz_2 \right)\\
=&\ - \frac{\i}{\left(\brs{z_1}^2 + \brs{z_2}^2 \right)^2} \left[ \left( \bz_2 
dz_2 + z_2 d \bz_2 \right) \wedge dz_1 \wedge d \bz_1 + \left(\bz_1 dz_1 + z_1 
d \bz_1 \right) \wedge dz_2 \wedge d \bz_2 \right]
\end{align*}
Thus one obtains
\begin{align*}
d^c \gw =&\ - d \gw (J_1, J_1, J_1)\\
=&\ -\frac{1}{\left(\brs{z_1}^2 + \brs{z_2}^2 \right)^2} \left[ \left(\bz_2 
dz_2 - z_2 d \bz_2 \right) \wedge dz_1 \wedge d \bz_1 + \left( \bz_1 dz_1 - z_1 
d \bz_1 \right) \wedge dz_2 \wedge d \bz_2 \right].
\end{align*}
Thus lastly
\begin{align*}
d d^c \gw =&\ \frac{2}{\left(\brs{z_1}^2 + \brs{z_2}^2 \right)^3} d \left( 
\brs{z_1}^2 + \brs{z_2}^2 \right) \wedge\\
&\ \qquad \left[ \left(\bz_2 dz_2 - z_2 d \bz_2 
\right) \wedge dz_1 \wedge d \bz_1 + \left( \bz_1 dz_1 - z_1 d \bz_1 \right) 
\wedge dz_2 \wedge d \bz_2 \right]\\
&\ + \frac{4}{\left( \brs{z_1}^2 + \brs{z_2}^2 \right)^2} \left[ dz_1 \wedge d 
\bz_1 \wedge d z_2 \wedge d \bz_2 \right]\\
=&\ 0,
\end{align*}
yielding property (1).

If we let $\pi$ denote the radial projection of $\mathbb C^2$ onto $S^3$, then these formulas also exhibit that $H = -d^c \gw = 2 \pi^* dV_{g_{S^3}}$, i.e. twice the standard volume form on $S^3$.  Observing that the metric on the universal cover $\mathbb C^2 \backslash \{0\}$ is isometric to the standard cylinder $(S^3 \times \mathbb R, g_{S^3} \oplus dt^2)$, it follows that $H$ is parallel.  Furthermore, using Lemma \ref{l:Leeformidentities} it follows that $\theta = \star H = dt$, which is also directly computable from the formula for $d^c \gw$ above.  Thus $\theta$ is also parallel.  As in Example \ref{e:S3S1E}, it also folllows that $\Rc = 2 g_{S^3}$, while $H^2 = 8 g_{S^3}$.  Referring to Proposition \ref{p:Bismutcurvature} we see that $\Rc^B = 0$, and it then follows from Proposition \ref{p:BismutRicci} that $\rho_B = 0$, finishing the proof.
\end{proof}
\end{prop}

\begin{thm} \label{t:Hopfrigidity} Let $(M^4, g, J)$ be a compact complex manifold  with puriclosed metric $g$ such that $\rho_B^{1,1} \equiv 0$.  Then $\rho_B \equiv 0$, and $(M^4, g, J)$ is 
biholomorphically isometric to either
\begin{enumerate}
\item A Calabi-Yau surface $(M^4, J)$ with K\"ahler-Ricci flat metric $g$
\item A quotient of a diagonal Hopf surface $(S^3 \times S^1, J)$, with Hopf metric $g$.
\end{enumerate}
\begin{proof} Using the curvature identities of Proposition 
\ref{p:BismutRicci}, we see that the condition $\rho_B^{1,1} \equiv 0$ implies 
the equations
\begin{gather} \label{f:HopfRig10}
\begin{split}
0=&\ \Rc - \tfrac{1}{4} H^2  + \tfrac{1}{2} L_{\theta^{\sharp}} g,\\
0=&\ d d^* H + \tfrac{1}{2} L_{\theta^{\sharp}} H.
\end{split}
\end{gather}
However, since the real dimension is $4$ we may apply Lemma \ref{l:Leeformidentities} to conclude that $\star \theta = H$ and thus $\theta^{\sharp} \hook H = 0$.  Thus $d d^* H = 0$, and by pairing 
with $H$ and integrating by parts we conclude $d^* H = 0$.  Thus $H$ is 
harmonic, and using again that $\star \theta = H$ it follows that $\theta$ is 
harmonic.  Using the Bochner formula for $1$-forms and integrating by parts we conclude that
\begin{gather} \label{f:Hopfrigidity10}
\begin{split}
0 =&\ \int_M \IP{\gD_d \theta, \theta} dV_g\\
=&\ \int_M \IP{\N^* \N \theta + \Rc(\theta,\cdot ), \theta} dV_g\\
=&\ \int_M \left[ \brs{\N \theta}^2 + \left( \tfrac{1}{4} H^2 - 
\tfrac{1}{2} L_{\theta^{\sharp}} g \right)(\theta,\theta)  \right] dV_g\\
=&\ \int_M \left[ \brs{\N \theta}^2 - \tfrac{1}{2} L_{\theta^{\sharp}} g (\theta,\theta)  
\right] dV_g.
\end{split}
\end{gather}
However we further observe, using that the Lee form is divergence-free since the metric is pluriclosed,
\begin{align*}
\int_M L_{\theta^{\sharp}} g (\theta, \theta) dV_g =&\ \int_M \left(\N_i 
\theta_j + \N_j \theta_i \right) \theta^i \theta^j dV_g\\
=&\ 2 \int_M \N_i \theta_j \theta^j \theta^i dV_g\\
=&\ \int_M \N_i \brs{\theta}^2 \theta^i dV_g\\
=&\ - \int_M \brs{\theta}^2 \N_i \theta^i dV_g\\
=&\ 0.
\end{align*}
Plugging this into (\ref{f:Hopfrigidity10}) we conclude that $\N \theta = 0$.  
If $\theta$ vanishes, then the metric is K\"ahler, and Ricci flat, yielding 
case (1).  Assuming $\theta \neq 0$, we have found a nontrivial parallel vector 
field.  By the de Rham decomposition theorem, it follows that the universal 
cover with the pullback Hermitian structure splits 
isometrically as a product $N \times \mathbb R$.  It follows that $H = \star \theta$ is a multiple of the volume form on $N$, and so examining the curvature equation of \ref{f:HopfRig10}, it follows that the slice has constant positive Ricci curvature, and is thus isometric to a multiple of the round $3$-sphere.  It follows that the original structure $(M^4, g)$ is isometric to a quotient of $(S^3 \times S^1, g_{S^3} \oplus \gl g_{S^1})$ for an appropriate constant $\gl$.

In particular, it follows that $(M^4, J)$ is a Hopf surface.  To obtain the precise complex structure, we note that, as a Hopf surface, it is first of all covered by a primary Hopf surface, and the universal cover is biholomorphic to $\mathbb C^2 \backslash \{0\}$.  We lift the metric, which is now cylindrical, to this space.  The unit vector tangent to the $S^1$ factor lifts to a global coordinate vector field $\frac{\del}{\del t}$.  Setting $r = e^t$, we obtain the radial vector field $r\frac{\del}{\del r}$, as well as an associated flat metric $g_0 = r^{-2} g$.  Furthermore, the generator of the fundamental group takes the form
\begin{align*}
\gg(s, p) = (as, \psi p),
\end{align*}
where $(s,p) \in \mathbb R \times S^3$, and $\psi \in SO(4)$, as the map must preserve the cylindrical metric.  Since $\gg$ is also a biholomorphism, it follows that in fact $\psi \in U(2)$.  The map $\psi$ will be conjugate to an element of the maximal torus in $U(2)$, in other words is conjugate to a map
\begin{align*}
(z_1, z_2) \to (e^{i \mu_1} z_1, e^{i \mu_2} z_2).
\end{align*}
By conjugating as required we assume $\psi$ itself takes this form.  It follows that, expressed in complex coordinates,
\begin{align*}
\gg(z_1,z_2) = (\ga z_1, \gb z_2), \qquad \ga = a e^{i \mu_1},\ \gb = a e^{i\mu_2},\ \brs{\ga} = \brs{\gb},
\end{align*}
as required.
\end{proof}
\end{thm}

\begin{rmk} \label{r:Perelmanrigidity} It is possible to obtain 
this rigidity result in part using the energy functional of \S \ref{s:GRFasgradient}.  In 
particular, using Proposition \ref{p:einsteinequiv} we know that $(g,H)$ defines a 
generalized Ricci soliton with associated vector field $\theta^{\sharp}$.  
Since $M$ is compact it follows from Corollary \ref{c:steadysolitongradient} 
that $(g,H)$ is in fact a steady gradient soliton.  We have thus obtained two 
expressions for $\Rc^B$, and comparing them yields $L_{\theta^{\sharp} - \N f} 
g \equiv 0$.  Since the Lee form is divergence-free one can trace this to 
obtain $\gD f = 0$, hence $f$ is a constant.  It follows that $\theta^{\sharp}$ 
is a Killing field, meaning the symmetric part of $\N \theta$ vanishes.  But $d 
\theta = 0$ as well, hence $\N \theta = 0$, after which the argument proceeds 
as above.
\end{rmk}

\begin{rmk} \label{r:steadysolitons} It was recently shown \cite{Streetssolitons} that the non-diagonal, class $1$, Hopf surfaces admit pluriclosed metrics which are steady generalized Ricci solitons in the sense of Definition \ref{d:GRSdef1}.  Note that by Theorem \ref{t:Hopfrigidity} these surfaces cannot admit Bismut Hermitian-Einstein metrics.  These solitons are conjecturally the global attractors for the pluriclosed flow on these surfaces.  Also in \cite{UstinovskiyStreets} it was shown that in fact these solitons admit two distinct compatible generalized K\"ahler structures, one with vanishing Poisson tensor and another with generically nondegenerate Poisson tensor.  These metrics form a continuous family extending the examples on diagonal Hopf surfaces discussed above.
\end{rmk}

It is possible to show rigidity results in higher dimensions with additional 
topological hypotheses.  We begin with a general rigidity theorem which is a simplified version of a result of Gauduchon (cf. \cite{Gauduchonfibres}).  First we recall the definition of Bott-Chern cohomology.

\begin{defn} \label{d:BottChern} Let $(M^{2n}, J)$ be a complex manifold.  Define the \emph{real $(1,1)$ Bott-Chern cohomology} via
\begin{align*}
H^{1,1}_{BC, \mathbb R} := \frac{\left\{ \Ker d : 
\Lambda^{1,1}_{\mathbb R} \to \Lambda^{3}_{\mathbb R}\right\}}{\left\{\i \del \delb f |\ f \in C^{\infty} \right\}}.
\end{align*}
\end{defn}

\begin{prop} \label{p:Gaudrigidity} Let $(M^{2n}, g, J)$ be a compact Hermitian 
manifold such that $s_C \geq 0$, with $s_C > 0$ at some point.  Then $c_1(M, J) 
\neq 0$ as an element of Bott-Chern cohomology.
\begin{proof} Let $\til{\gw} = e^{2u} \gw$ be a Gauduchon metric conformally 
related to $\gw$, which exists by Theorem \ref{t:Gaudthm}.  If $c_1(M, J) = 0$ in Bott-Chern cohomology, 
then there exists $f \in C^{\infty}(M)$ such that $\rho_C(g) = \i \del \delb 
f$.  It then follows that
\begin{align*}
0 =&\ \int_M \i \del \delb f \wedge \til{\gw}^{n-1} = \int_M e^{2(n-1)u} 
\rho_C(g) \wedge \gw^{n-1} = \int_M e^{2(n-1)u} s_C \gw^n > 0,
\end{align*}
a contradiction.
\end{proof}
\end{prop}

\begin{rmk} We note for instance that, since the Hopf metric (\ref{f:Hopfmetric}) has $s_C > 0$ everywhere, this proposition in particular implies that $c_1(M, J)$ is nonzero in Bott-Chern cohomology, while of course it is zero in de Rham cohomology, which vanishes entirely.
\end{rmk}

\begin{prop} \label{p:Kahlerrigidity} Let $(M^{2n}, g, J)$ be a compact 
Bismut Hermitian-Einstein manifold.  If $c_1(M, J) = 0$ in Bott-Chern cohomology, then $g$ is 
K\"ahler Calabi-Yau.
\begin{proof} Using Proposition \ref{p:BismutRicci}, we see that a Bismut Hermitian-Einstein 
manifold satisfies
\begin{align*}
0 =&\ \rho_B^{1,1} = S_C - Q.
\end{align*}
Taking the trace then yields $0 = s_C - \brs{T}^2$, and hence $s_C \geq 
\brs{T}^2 \geq 0$.  Since we have assumed $c_1(M, J) = 0$, it follows directly 
from Proposition \ref{p:Gaudrigidity} that $s_C \equiv 0$, and hence $\brs{T}^2 
\equiv 0$.  The metric $g$ is hence K\"ahler Calabi-Yau, as required.
\end{proof}
\end{prop}

\begin{rmk} In Proposition \ref{p:Kahlerrigiditysteadysoliton} below we use the strong maximum principle to improve Proposition \ref{p:Kahlerrigidity} to cover steady solitons. 
\end{rmk}

\begin{rmk} In the nondegenerate generalized K\"ahler setting, the tensor $(\Omega_I^{2,0})^{\wedge n}$ is a nonvanishing holomorphic section of the canonical bundle, and so $c_1(M, J) = 0$.  Thus in this setting any Hermitian generalized Einstein metric is automatically K\"ahler Calabi-Yau (cf. \cite{ASnondeg}).  In fact, due to the presence of the holomorphic symplectic form $\Omega_I^{2,0}$, it follows from the Beauville-Bogomolov-Yau theorem \cite{Beauville, Bogomolov, YauCC} that in fact $(M^{4n}, g, I)$ is part of a hyperK\"ahler structure.
\end{rmk}

\begin{ex} \label{e:CalabiEckmann} A higher-dimensional, simply-connected example of a Bismut Hermitian-Einstein structure is given by a Calabi-Eckmann manifold.  For the general construction one considers $\mathbb C^n \backslash \{0\} \times \mathbb C^m \backslash \{0\}$, and consider the action of $\mathbb C$ given by
\begin{align*}
\gg(z, w) =  (e^{\gg} z, e^{\ga \gg} w),
\end{align*}
where $\ga \in \mathbb C \backslash \mathbb R$ is a fixed non-real complex number.  This action is free and proper, and the quotient space is diffeomorphic to $S^{2n-1} \times S^{2m-1}$.  

We consider the special case of $n = m = 2, \ga = \i$.  Thus the manifold is $S^{3} \times S^3$, and is the total space of a $T^2$ fibration over $\mathbb {CP}^1 \times \mathbb {CP}^1$.  This is the product of the standard Hopf fibrations $\pi_i : S^3 \to \mathbb{CP}^1$.  Let $\xi_i$ denote the canonical vector fields associated to this fibration on the two factors, with $\mu_i$ the associated canonical connections satisfying $d \mu_i = \pi_i^* \gw_{FS}$.  Note that the complex structure further satisfies $J \xi_1 = \xi_2$.  We consider the K\"ahler form associated to the product of round metrics, namely
\begin{align*}
\gw = \pi_1^* \gw_{FS} + \pi_2^* \gw_{FS} + \mu_1 \wedge \mu_2.
\end{align*}
It follows that
\begin{align*}
d \gw =&\ d \mu_1 \wedge \mu_2 - \mu_1 \wedge d \mu_2 = \pi_1^* \gw_{FS} \wedge \mu_2 - \mu_1 \wedge \pi_2^* \gw_{FS},
\end{align*}
hence
\begin{align*}
H = - d^c \gw =&\ d \gw(J, J, J) = \pi_1^* \gw_{FS} \wedge \mu_1 - \pi_2^* \gw_{FS} \wedge \mu_2
\end{align*}
It is clear then that the geometric structure is that of a product of two copies of the $S^3$ factor of the Hopf metric, and furthermore that $d H = 0$.  It now follows by computations similar to those in Proposition \ref{p:Hopfmetricprop} that
\begin{align*}
\Rc - \tfrac{1}{4} H^2 = 0.
\end{align*}
Furthermore, it follows from Lemma \ref{l:Leeformidentities} that $\theta = -\tfrac{1}{2} (\mu_1 + \mu_2)$.  Thus  $\theta^{\sharp}$ is a Killing field which further preserves $H$, so by Proposition \ref{p:BismutRicci} it follows that $\rho_B^{1,1} = 0$.  We leave as an \textbf{exercise} to further check that $\rho_B^{2,0 + 0,2} = 0$.
\end{ex}

\begin{question} Do Calabi-Eckmann manifolds for different choices of $n, m$, and $\ga$ admit Bismut Hermitian-Einstein metrics?  Do they admit steady generalized Ricci solitons?
\end{question}

\begin{question} How rigid are higher dimensional non-K\"ahler Bismut Hermitian-Einstein metrics?  In complex dimension $n=3$, is it possible to classify solutions with nonzero, (anti)self-dual torsion?
\end{question}

\begin{question} 
A natural class of Bismut Hermitian-Einstein manifolds has been introduced in \cite[Section 2.3]{garciafern2018canonical}. These equations impose the stronger conditions that the holonomy of the Bismut connections reduces to the special unitary group and $d \theta = 0$, and admit a natural interpretation in generalized geometry in terms of the operators $D^+_-$ and $\slashed D^+$ introduced in Chapter \ref{c:GCC}. A complete classification of this geometry on complex surfaces has been obtained in \cite[Proposition 2.11]{garciafern2018canonical}, but the higher-dimensional case is still open. It would be interesting to obtain a classification of solutions of these stronger equations in three complex dimensions, as a first step towards the classification of non-K\"ahler Bismut Hermitian-Einstein three-folds.
\end{question}

A fundamental result of Samelson \cite{Samelson} shows that every compact Lie group of even dimension admits left-invariant complex structures which are compatible with a bi-invariant metric.

\begin{prop} \label{p:BismutflatSamelson} Let $G$ be an even-dimensional connected Lie group with a bi-invariant metric $g$ and left-invariant complex structure $J$ compatible with $g$.  Then $(G, g, J)$ is Bismut flat.
\begin{proof} As noted in Proposition \ref{p:Bismutflat}, the Levi-Civita connection of a bi-invariant metric takes the form
\begin{align*}
\N_X Y = \tfrac{1}{2} [X, Y].
\end{align*}
On the other hand, the computation of Example \ref{e:semisimplePC} shows that for a left-invariant complex structure, one has $H = - d^c \gw = g([X,Y],Z)$.  Thus $\N^J$ is flat by Proposition \ref{p:Bismutflat}.
\end{proof}
\end{prop}

\begin{defn} \label{d:samelson} A \emph{Samelson space} is a Hermitian manifold $(G, g, J)$ where $G$ is a simply-connected Lie group with $g$ a bi-invariant metric, and $J$ a left-invariant complex structure on $G$ compatible with $g$.
\end{defn}

As it turns out, Samelson spaces are the only possible simply-connected Hermitian Bismut-flat manifolds, and moreover the possible fundamental groups of compact Samelson spaces are quite restricted.

\begin{defn} \label{d:localsamelson} Let $(G_1 = G_0 \times \mathbb R^k, g, J)$ be a Samelson space.  Given $\rho : \mathbb Z^k \to \Isom(G)$ a homomorphism, define a free and properly discontinuous action of $\gG_{\rho} \cong \mathbb Z^k$ on $G_1$ via
\begin{align*}
\gg(g, x) = (\rho(\gg) g, x + \gg).
\end{align*}
The manifold $M = G_1 / \gG_{\rho}$ inherits a complex structure and a Bismut flat Hermitian metric, and is called a \emph{local Samelson space}.
\end{defn}

\begin{thm} \label{t:Zhengclass} (\cite{ZhengBismutflat}) Let $(M^n, g, J)$ be a compact Bismut flat Hermitian manifold.  Then there exists a finite unbranched cover $M'$ of $M$ which is a local Samelson space.
\end{thm}

\begin{rmk} Theorem \ref{t:Zhengclass} can be used to give a precise description of all compact Bismut flat Hermitian manifolds of complex dimension $n=2,3$.  In particular, in complex dimension $n=2$ one recovers the standard Hopf surfaces as described in Theorem \ref{t:Hopfrigidity}.  Of course the result of Theorem \ref{t:Hopfrigidity} is stronger, giving a classification among the a priori larger class of Bismut-Ricci flat metrics.  In complex dimension $3$, the universal cover is either a product of $\left(\mathbb C^2 \backslash \{0\} \right) \times \mathbb C$, with metric a product of the Hopf metric (\ref{f:Hopfmetric}) and the flat metric, or a central Calabi-Eckmann manifold, defined as $SU(2) \times SU(2)$ with left-invariant complex structure and bi-invariant metric.
\end{rmk}

\begin{question} What further restrictions apply if we ask that a Bismut flat pluriclosed metric is also part of a generalized K\"ahler structure?  Note that as observed in \S \ref{ss:GKexamples}, the Hopf metric is part of a generalized K\"ahler structure in two distinct ways.  Furthermore, the Calabi-Eckmann metrics are compatible with a commuting-type generalized K\"ahler structure, with the second complex structure given by changing the orientation on the $T^2$ fiber.
\end{question}

\begin{question} Note that the final stage of the proof of Theorem \ref{t:Hopfrigidity} identifies the possible complex structures compatible with a particular flat Bismut connection.  It is natural to ask this question in higher dimensions: given an even dimensional Bismut flat structure $(M^{2n}, g, H)$, can we describe all complex structures $J$ (necessarily left-invariant) which are compatible with $g$, and satisfy $\N^B J \equiv 0$?
\end{question}

\chapter{Generalized Ricci Flow in Complex Geometry} \label{c:GFCG}

In Chapter \ref{c:GCG} we introduced the fundamental concepts of generalized complex geometry, including
pluriclosed and generalized K\"ahler metrics.  Furthermore, Chapter 
\ref{c:CMGCG} yielded a definition of canonical metric in pluriclosed geometry 
related to the generalized Einstein condition.  In this chapter we link 
together these ideas with the generalized Ricci flow by showing that, when 
coupled to evolution equations for the associated complex structures, the 
generalized Ricci flow preserves pluriclosed and generalized K\"ahler geometry. 

The presence of the extra structure coming from complex geometry renders these 
flows decidedly more tractable than the full generalized Ricci flow.  To begin we show how to reduce the pluriclosed flow to a flow of $(1,0)$-forms, and give a scalar reduction for the generalized K\"ahler-Ricci flow in certain settings, generalizing the classical reduction of the K\"ahler-Ricci flow to a parabolic complex Monge-Amp\`ere equation.  Building on this we show smooth regularity of pluriclosed flow in the presence of uniform parabolicity bounds, generalizing Calabi/Yau's $C^3$ estimate for the Monge-Amp\`ere equation.  Using these estimates we end the chapter by showing global existence and convergence results for the pluriclosed flow and generalized K\"ahler-Ricci flow and discuss their implications.

\section{K\"ahler-Ricci flow}

The natural starting point for discussing geometric flows in complex geometry is K\"ahler-Ricci flow.  Though we will not spend much time discussing the deep, detailed theory of K\"ahler-Ricci flow here, it is instructive to review the rudimentary aspects of the theory as a way of setting the stage for the more general flows we discuss below.  We refer the reader to \cite{KRFbook, SongWeinkove} for more detailed accountings of the fundamental theory of K\"ahler-Ricci flow.

\subsection{Basic properties}

\begin{defn} \label{d:KRF} Let $(M^{2n}, J)$ be a complex manifold.  A one 
parameter family of K\"ahler metrics $g_t$ is a solution of 
\emph{K\"ahler-Ricci flow} if the associated K\"ahler forms $\gw_t$ satisfy
\begin{align} \label{f:KRF}
\dt \gw =&\ - \rho_C.
\end{align}
Noting that $\rho_C$ is a closed $(1,1)$ form, one can expect the equation above to preserve the K\"ahler conditions.  Before proving this we observe the evolution equation for the underlying Riemannian metric.
\end{defn}

\begin{lemma} \label{l:KRFequiv} Let $(M^{2n}, g_t, J)$ be a solution to 
K\"ahler-Ricci flow.  Then the metric $g_t$ satisfies
\begin{gather*}
\begin{split}
\dt g =&\ - \Rc.
\end{split}
\end{gather*}
\begin{proof} Using Lemma \ref{l:BismutRicci}, applied here to a K\"ahler metric so that $\rho_C = \rho_B$ and $\theta = 0$, we obtain
\begin{align*}
\dt g(X,Y) =&\ \dt \gw(X, J Y) = - \rho_C(X, J Y) = - \Rc(X,Y),
\end{align*}
as required.
\end{proof}
\end{lemma}

We next establish the short-time existence of solutions to K\"ahler-Ricci flow.  Note that this does not follow from Lemma \ref{l:KRFequiv} and the short-time existence of solutions to Ricci flow, as we do not yet know that the Ricci flow should preserve the K\"ahler conditions, which is rather a consequence of proving the short-time existence of solutions to (\ref{f:KRF}).

\begin{prop} \label{p:KRFste} Given $(M^{2n}, J)$ a compact complex manifold and $g_0$ a K\"ahler metric, there exists a unique maximal solution $g_t$ of K\"ahler-Ricci flow with initial condition $g_0$ on $[0,T)$ for some $T \in \mathbb R \cup \{\infty\}$.  In particular, the solution to Ricci flow with initial condition $g_0$ is a solution to K\"ahler-Ricci flow.
\begin{proof} We can compute in local complex coordinates, for a K\"ahler metric $g$,
\begin{align*}
\rho_C(\gw)_{i \bj} =&\ g^{\bq p} \del_{\bq} \del_p \gw_{i \bj} + \del \gw \star \delb \gw.
\end{align*}
It follows that equation (\ref{f:KRF}), restricted to the open cone of positive tensors in the linear space of closed $(1,1)$-forms, is strictly parabolic.  Thus we can apply the standard theory for strictly parabolic equations to conclude the existence of a maximal solution $\gw_t$ defined on $[0,T)$.  By Lemma \ref{l:KRFequiv}, the associated family of Riemannian metrics $g_t$ is a solution of Ricci flow with initial condition $g_0$, which is unique by Theorem \ref{t:STE}.
\end{proof}
\end{prop}

\subsection{K\"ahler cone and formal existence time}

To better understand the qualitative picture of the long time existence 
and singularity formation of K\"ahler-Ricci flow, we study the flow at the 
formal level of cohomology.  We note that K\"ahler metrics will define classes in Bott-Chern cohomology as in Definition \ref{d:BottChern}.  Within this cohomology space lies the K\"ahler cone, which is the set of 
classes in $H^{1,1}$ represented by K\"ahler metrics.

\begin{defn} \label{d:Kahlercone} Let $(M^{2n}, J)$ be a complex manifold.  Define the 
\emph{K\"ahler cone} by
\begin{align*}
\KK := \left\{ [\psi] \in H^{1,1}_{BC, \mathbb R} \ |\ \exists\ \gw \in [\psi],\ \gw > 0 
\right\}.
\end{align*}
\end{defn}

\begin{ex} \label{e:kahlerconesurfaces} Let $(M^2, J)$ be a compact Riemann 
surface.  Any metric $g$ compatible with $J$ is K\"ahler as the K\"ahler form 
$\gw$ is automatically closed since $d \gw \in \Lambda^3(M) \cong \{0\}$.  The 
form $\gw$ induces an orientation on $M$, and it follows that $H^{1,1}_{BC,\mathbb R} \cong \mathbb R$.  It is clear by construction that $\gl [\gw]$ admits the 
K\"ahler metric $\gl \gw$ for $\gl \geq 0$.  Alternatively, for $\gl < 0$, any 
form $\eta = \gl \gw + \i \del \delb f$ must satisfy 
\begin{align*}
\int_M \eta = \int_M \left( \gl \gw + \i \del \delb f \right) = \gl \int_M \gw 
= \gl \Vol(g) \leq 0.
\end{align*}
But if $\eta$ was the K\"ahler form associated to a metric $h$ one would have 
$\int_M \eta = \Vol(h) > 0$.  It follows that the K\"ahler cone is precisely 
the half-line $\{ \gl [\gw]\ |\ \gl > 0 \}$.
\end{ex}

The structure of the K\"ahler cone plays a fundamental role in understanding 
the singularity formation of the K\"ahler-Ricci flow, and first of all gives an upper bound for the possible smooth existence time of the flow.

\begin{lemma} \label{l:KflowtimeUB} Let $(M^{2n}, g_0, J)$ be a complex 
manifold 
with K\"ahler metric.  Let
\begin{align*}
\tau^*(g_0) := \sup \{ t \geq 0\ |\ [\gw_0] - t c_1 \in \KK \},
\end{align*}
and let $T$ denote the maximal smooth existence time for the K\"ahler-Ricci flow
with initial condition $g_0$.  Then $T \leq \tau^*(g_0)$.
\begin{proof} Let $\gw_t$ denote the one-parameter family of K\"ahler forms 
evolving by K\"ahler-Ricci flow with initial condition $\gw_0$.  It follows 
easily that $[\gw_t] = [\gw_0] - t c_1$.  If the flow 
existed smoothly for some time $t > \tau^*(g_0)$, then in particular there exists a 
smooth positive definite metric in $[\gw_0] - t c_1$, contradicting the 
definition of $\tau^*(g_0)$.
\end{proof}
\end{lemma}

In view of this lemma, the natural question arises which is, given a K\"ahler 
metric $g_0$, is $\tau^*(g_0)$ the maximal existence time of the flow?  In 
other words, is the flow smooth until the K\"ahler class reaches the boundary 
of the K\"ahler cone?  The answer is yes, and this is proved in Theorem 
\ref{t:TianZhang} below.

\section{Pluriclosed flow}

Given the rich geometric and analytic structure of K\"ahler-Ricci flow, one can ask if it can be extended to a natural flow of Hermitian, non-K\"ahler metrics.  Though the world of Hermitian, non-K\"ahler geometry is large and diverse, in keeping with the thrust of this text we will focus our attention on the parts of the subject most closely linked to generalized complex geometry.

\subsection{Basic properties}

\begin{defn} \label{d:PCF} Let $(M^{2n}, J)$ be a complex manifold, and fix $H_0 \in \Lambda^3 T^*, d H_0 = 0$.  A one 
parameter family of pluriclosed metrics $(g_t, b_t)$ is a solution of 
\emph{pluriclosed flow} if the associated K\"ahler forms $\gw_t$ and $(2,0)$-forms $\gb_t = \i b^{2,0}$ satisfy
\begin{align} \label{f:augPCF}
\dt \gw =&\ - \rho_B^{1,1}, \qquad \dt \gb = - \rho_B^{2,0}.
\end{align}
\end{defn}
The evolution equation for $\gw$ first appeared in \cite{PCF}.  We note that the equation for $\gb_t$ is not strictly speaking necessary to determine the metric or torsion, since $H = - d^c \gw$ and $\rho_B$ is closed.  For this reason we will at times refer to a family of K\"ahler forms $\gw_t$ alone as a solution to pluriclosed flow. Nonetheless, the equation for $\gb$ plays an important role in deriving estimates and determining the relationship to generalized geometry. Indeed, as we have different ways of describing pluriclosed structures in terms of the underlying Riemannian metric, the K\"ahler form, or objects in generalized geometry, so pluriclosed flow has different faces when expressed in terms of these underlying objects, each revealing different important analytic and geometric properties of the flow.  In particular, pluriclosed flow turns out to be a gauge-modified version of generalized Ricci flow (up to a reparameterization of time by a factor of $2$). Another interesting aspect of the flow \eqref{f:augPCF} is that, if we fix $\tilde H_0 \in \Lambda^{3,0 + 2,1}$, $d\tilde H_0 = 0$, then the condition 
$$
\partial \omega_t + d \beta_t = \tilde H_0
$$ 
is preserved along the flow. Thus, pluriclosed flow can be regarded as a natural flow for metrics on a fixed holomorphic Courant algebroid, as defined in \S\ref{ss:holCourant}. 

\begin{prop} \label{p:PCFequiv} Let $(M^{2n}, g_t, J)$ be a solution to 
pluriclosed flow.  Then
\begin{enumerate}
\item The associated K\"ahler forms $\gw_t$ satisfy
\begin{gather*}
\dt \omega = - \del \del^*_{\omega} \omega - \delb \delb^*_{\omega} \omega - \rho_C(\gw).
\end{gather*}
\item The associated K\"ahler forms $\gw_t$ satisfy
\begin{gather} \label{f:HMPCF}
\dt \gw = - S_C + Q,
\end{gather}
where $Q$ is as in (\ref{f:Qdef}).
\item The associated pairs $(g_t, b_t)$ satisfy
\begin{gather} \label{f:RMPCF}
\begin{split}
\dt g =&\ - \Rc + \tfrac{1}{4} H^2 - \tfrac{1}{2} L_{\theta^{\sharp}} g,\\
\dt b =&\ - \tfrac{1}{2} d^* H + \tfrac{1}{2} d \theta - \tfrac{1}{2}  i_{\theta^{\sharp}} H.
\end{split}
\end{gather}
In particular, pluriclosed flow is a solution of $(- \theta^{\sharp}, d \theta)$-gauge-fixed generalized Ricci flow (cf. \ref{f:gaugefixedgrf}).

\item The associated generalized metrics $\GG = \GG(g_t,b_t)$ on $(T\oplus T^*, \IP{,}, [,]_{H_0},\pi)$ satisfy
$$
\GG^{-1}\dt \GG \circ \pi_-  = - \gRc^+(\GG,\divop) \circ \pi_-,
$$
where $\divop = \divop^{\GG} - 2 \IP{\theta,}$.

\end{enumerate}
\begin{proof}
The first equivalence follows directly by projecting the first equation from Lemma \ref{l:BismutRicci} onto the $(1,1)$ component.  The second formulation follows directly from Proposition \ref{p:BismutRicci}.  For the third formulation, the evolution equation for the metric follows directly from Proposition \ref{p:BismutRicci}, where the evolution equation for $b$ follows from the third equation of Proposition \ref{p:BismutRicci} and the computations of Proposition \ref{p:einsteinequiv}. As for the last formulation, it follows from (3) and direct application of Proposition \ref{p:GFGRFinGG+}. 
\end{proof}
\end{prop}

\begin{cor} \label{c:PCFasGRF} Let $(M^{2n}, g_t, J)$ be a solution to 
pluriclosed flow, and let $\phi_t$ denote the one-parameter family of diffeomorphisms generated by $\tfrac{1}{2} \theta^{\sharp}_t$.  Then for all $t$ such that $g_t$ and $\phi_t$ are defined, let
\begin{align*}
g'_t = \phi_t^* g_t, \qquad H'_t = \phi_t^* H_t, \qquad J'_t = \phi_t^* J
\end{align*}
Then $(g'_t, H'_t)$ is a solution of generalized Ricci flow, and $J'_t$ satisfies
\begin{align*}
\dt J' = \tfrac{1}{2} L_{\theta^{\sharp}} J' = \tfrac{1}{2} \gD J' +\Rm' \star H' + H' \star H'.
\end{align*}
\begin{proof} The claim that $(g'_t, H'_t)$ is a solution of generalized Ricci flow is a direct consequence of (\ref{f:RMPCF}) and the naturality properties of Lie derivatives.  The claim that $\dt J' = \tfrac{1}{2} L_{\theta^{\sharp}} J'$ also follows by definition.  The final formula is an \textbf{exercise} using properties of the Lee form on complex manifolds with pluriclosed metric.
\end{proof}
\end{cor}

\begin{rmk} The fact that the Lie derivative operator $L_{\theta^{\sharp}} J$ is actually a heat operator in disguise is an interesting feature of pluriclosed flow, which plays an essential role in the generalized K\"ahler-Ricci flow to come.  An explicit formula for the lower-order terms $\Rm \star H$ and $H \star H$ was computed in \cite[Proposition 3.1]{GKRF}, although it is quite unwieldy.  It would be interesting to derive estimates for the pluriclosed flow directly from this formula.
\end{rmk}

\begin{rmk} \label{r:PCFGRF} In view of Corollary \ref{c:PCFasGRF}, it is natural at times to express pluriclosed flow explicitly as a solution to generalized Ricci flow.  The system of equations indicated there is
\begin{gather} \label{f:PCFasGRF}
\begin{split}
 \dt g =&\ - \Rc + \tfrac{1}{4} H^2,\\
 \dt b =&\ - \tfrac{1}{2} d^*_g H,\\
 \dt J=&\ \tfrac{1}{2} L_{\theta_J^{\sharp}} J,\\
\end{split}
\end{gather}
At times we will refer to solutions to this system also as ``pluriclosed flow,'' even though it is related to the original definition by a nontrivial gauge transformation.
\end{rmk}

\begin{rmk} \label{r:HCFrmk} Item (2) of Proposition \ref{p:PCFequiv} expresses pluriclosed flow in terms of the second Ricci curvature of the Chern connection, denoted $S_C$ in Definition \ref{d:Riccidef}, and a certain quadratic expression $Q$ in the torsion of the Chern connection.  Thus the pluriclosed flow fits into a general family of geometric flows of Hermitian metrics introduced in \cite{HCF}, which take the form
\begin{align*}
\dt \gw =&\ - S_C + Q,
\end{align*}
where $Q$ is an \emph{arbitrary} quadratic expression in the torsion.  This general family of equations is referred to as \emph{Hermitian curvature flow}.  Fundamental analytic properties such as short-time existence, smoothing estimates, and stability results were shown for this family of flows in \cite{HCF}.  Furthermore, a certain expression for $Q$ was identified from the point of view of seeking a fixed-point equation for the flow which arises as the Euler-Lagrange equation for a Hilbert-type functional in the context of Hermitian geometry.

Given the rich variety within Hermitian geometry, it seems natural to expect that different choices of $Q$ may be better suited for different questions.  In particular, recent work of Ustinovskiy \cite{UstinovskiyHCF} identifies a particular choice of $Q$ such that the resulting flow preserves nonnegative holomorphic bisectional curvature, yielding a natural extension of the classic Frankel conjecture.  This flow is also related to recent work on the Hull-Strominger system \cite{anomaly}.  Also Lee \cite{LeeHCF} has shown that with $Q = 0$ the flow can be used to show that compact complex manifolds admitting a metric of non-positive bisectional curvature and nonpositive Chern-Ricci curvature which is negative at one point have ample canonical bundle.  In a different direction, flowing a Hermitian metric by the Chern-Ricci curvature directly has been explored in \cite{Gill, TW1}.
\end{rmk}

\subsection{Short-time existence}

The proposition below establishes short-time existence for pluriclosed flow for smooth initial data on compact manifolds.

\begin{prop} \label{p:PCFpreserved} Given $M$ a smooth compact manifold and $(g, 
H,
J)$ a 
pluriclosed structure, there exists a unique maximal solution $(g_t,H_t,J)$ to pluriclosed flow with this initial condition on $[0,T)$ for some $T \in \mathbb R \cup \{\infty\}$.  Also, if $g_0$ is K\"ahler then $H_t \equiv 0$, and $g_t$ is the unique solution to K\"ahler-Ricci flow with initial condition $g$.
\begin{proof}
Using Proposition \ref{p:PCFequiv} and the expression for the Chern curvature in complex coordinates we see that the pluriclosed flow can be expressed as 
\begin{gather} \label{p:PCFpres10}
\begin{split}
\dt \gw_{i \bj} =&\ - (S_C)_{i \bj} + Q_{i \bj}\\
=&\ g^{\bq p} \del_p \del_{\bq} \gw_{i \bj} + \del \gw \star \delb \gw.
\end{split}
\end{gather}
This is a strictly parabolic equation on the open cone of positive tensors in the linear space of pluriclosed $(1,1)$-forms, and so by the theory of strictly parabolic PDE we can obtain unique short-time solution on compact manifolds.  If $g_0$ is K\"ahler, then let $\til{g}_t$ denote the unique solution of K\"ahler-Ricci flow with initial data $g_0$, whose existence is guaranteed by Proposition \ref{p:KRFste}.  As the metrics $\til{g}_t$ are all K\"ahler, it follows easily by Lemma \ref{l:BismutRicci} that $\rho_C(\til{g}_t) = \rho_B^{1,1}(\til{g}_t)$, and thus $\til{g}_t$ is a solution of pluriclosed flow as well.  Since solutions to pluriclosed flow are unique on compact manifolds, it follows that $g_t = \til{g}_t$ is the solution to K\"ahler-Ricci flow.
\end{proof}
\end{prop}

\subsection{A positive cone and conjectural existence for pluriclosed flow}

Given the short-time existence of solutions to pluriclosed flow, we are again faced with the basic question of, what is the 
maximal smooth existence time, and what happens to the metric at that time?  As 
these flows are special cases of generalized Ricci flow, the discussion of \S 
\ref{s:MET} certainly applies here, showing that the Riemannian curvature must blow up at a singular time.  However due to the extra rigidity arising 
from the presence of complex structures, one expects sharper statements.  In this section we formulate 
a conjecture on the maximal existence time for pluriclosed flow.  Later in this chapter we verify these conjectures in various 
special cases.

\begin{defn} \label{d:pluriclosedcone} Let $(M^{2n}, J)$ be a complex manifold.  Define the 
\emph{real $(1,1)$-Aeppli cohomology} via
\begin{align*}
H^{1,1}_{A, \mathbb R} := \frac{\left\{ \Ker \i \del\delb : 
\Lambda^{1,1}_{\mathbb R} \to \Lambda^{2,2}_{\mathbb R}\right\}}{\left\{\del 
\bga + \delb \ga\ |\ \ga \in \Lambda^{1,0} \right\}}.
\end{align*}
Furthermore, define the 
\emph{$(1,1)$-Aeppli positive cone} via
\begin{align*}
\PP^{1,1}_{A} := \left\{ [\psi] \in H^{1,1}_{A, \mathbb R}\ |\ 
\exists\ \gw \in [\psi],\ \gw > 0 \right\}.
\end{align*}
In other words, this cone consists precisely of the $(1,1)$-Aeppli cohomology classes which contain pluriclosed metrics.
\end{defn}

\begin{rmk} \label{r:cohomologident} We observe here that in complex manifolds 
there is a natural map
\begin{align*}
\iota : H^{1,1}_{BC, \mathbb R} \to H^{1,1}_{A, \mathbb R},
\end{align*}
In particular we can apply this map to 
the first Chern class, yielding a well-defined class in $H^{1,1}_{A, \mathbb R}$.  In what follows we will not make a notational distinction between the 
first Chern class thought of in these two different contexts, as we will mean 
the element of $H^{1,1}_{A, \mathbb R}$ unless otherwise specified.
\end{rmk}

\begin{lemma} \label{l:flowclassevol} Let $(M^{2n}, J)$ be a complex manifold, 
and suppose $\gw_t$ is a one-parameter family of metrics satisfying the 
pluriclosed flow equation.  Then $[\gw_t] = [\gw_0] - t c_1$.
\begin{proof} Fix a background Hermitian metric with K\"ahler form $\eta$, and note that $[\gw_0 - t \rho_C(\eta)] = [\gw_0] - t c_1$.  
We use the transgression formula of Proposition \ref{p:Cherntrans} to observe that
\begin{align*}
\gw_t =&\ \gw_0 + \int_0^t \frac{\del}{\del s} \gw ds\\
=&\ \gw_0 + \int_0^t \left( - \del \del^* \gw - \delb \delb^* \gw - \rho_C(\gw) 
\right) ds\\
=&\ \gw_0 + \int_0^t \left( - \del \del^* \gw - \delb \delb^* \gw + \i \del \delb 
\log \frac{\gw_t^n}{\eta^n} - \rho_C(\eta) \right) ds\\
=&\ \gw_0 - t \rho_C(\eta) + \del \bgb + \delb \gb,
\end{align*}
where
\begin{align*}
\gb := \int_0^t \left( - \delb^* \gw - \frac{\i}{2} \del \log 
\frac{\gw^n}{\eta^n} \right) ds.
\end{align*}
Hence $[\gw_t] = [\gw_0 - t c_1(\eta)] = [\gw_0] - t c_1$, as claimed.
\end{proof}
\end{lemma}

\begin{lemma} \label{l:flowtimeUB} Let $(M^{2n}, g_0, J)$ be a complex manifold 
with pluriclosed metric.  Let
\begin{align*}
\tau^*(g_0) := \sup \{ t \geq 0\ |\ [\gw_0] - t c_1 \in \PP^{1,1}_{A} \},
\end{align*}
and let $T$ denote the maximal smooth existence time for the pluriclosed flow 
with initial condition $g_0$.  Then $T \leq \tau^*(g_0)$.
\begin{proof} The proof is directly analogous to that of Lemma \ref{l:KflowtimeUB}.
\end{proof}
\end{lemma}

What follows is the main conjecture guiding the study of pluriclosed flow.  While Lemma \ref{l:flowtimeUB} indicates the elementary fact that the maximal existence time for the flow can be no larger than $\tau^*(g_0)$, Conjecture \ref{c:mainflowconj} indicates that the flow is actually smooth up to time $\tau^*(g_0)$, i.e. that it equals the maximal existence time.

\begin{conj}\label{c:mainflowconj} Let $(M^{2n}, g_0, J)$ be a compact complex manifold with pluriclosed metric.  The unique maximal smooth solution to pluriclosed flow exists on $[0,\tau^*(g_0))$.
\end{conj}

This conjecture first appeared in \cite{PCFReg}, inspired by a theorem of Tian-Zhang (Theorem \ref{t:TianZhang} below).  We will give a proof of this theorem as well as proofs of Conjecture \ref{c:mainflowconj} in various special cases below.  As a guide for the nature of singularity formation it is useful to have a characterization of the necessary and sufficient 
conditions for cohomology classes to lie in the appropriate positive cone.  In the remainder of this section we will record a theorem characterizing this cone in the case of complex surfaces.

\begin{lemma} \label{l:conekernel} Let $(M^4, \omega, J)$ be a compact complex
surface with pluriclosed metric, and let
\begin{align*}
B^{1,1}_{\mathbb R} = \{ \mu \in \Lambda^{1,1}_{\mathbb R}\ |\ \exists\ a \in \Lambda^1, \mu = d a \}.
\end{align*}
There is an exact sequence
\begin{align*}
0 &\rightarrow \i \del \delb \Lambda^0_{\mathbb R} \rightarrow B^{1,1}_{\mathbb
R} \rightarrow \mathbb R,
\end{align*}
where the final map above is given by the $L^2$ inner product with $\omega$.
\begin{proof} Exactness in the first two positions is clear.  To check 
exactness in the third place, fix $\mu \in
B^{1,1}_{\mathbb R}$ satisfying
\begin{align*}
0 = \int_M \IP{\mu, \gw} dV.
\end{align*}
It follows from the maximum principle that the kernel of the $L^2$ adjoint of 
$\tr_{\omega} \i \del
\delb$ consists of only the constant functions.  Thus $\tr_{\omega} \mu$ is 
orthogonal to the cokernel of $\tr_{\gw} \i \del \delb$, so by the Fredholm 
alternative we can find $u$ such that
\begin{align*}
\tr_{\gw} \i \del \delb u = \tr_{\omega} \mu.
\end{align*}
Thus $\i \del \delb u - \mu$ is exact, and also anti-self-dual since its inner
product with $\omega$ vanishes.  Thus it is harmonic, and exact, and hence 
vanishes, yielding $\mu = \i \del \delb u$, as required.
\end{proof}
\end{lemma}

In view of this lemma, we define
\begin{align*}
\Gamma = \frac{B^{1,1}_{\mathbb R}}{i \del \delb \Lambda^0_{\mathbb R}}.
\end{align*}
This is identified with a subspace of $\mathbb R$, via the $L^2$ inner product 
with $\omega$ as in Lemma \ref{l:conekernel}.  Note that if $(M^4, J)$ is 
K\"ahler then the $\del\delb$-lemma holds and so $\gG = \{0\}$.   Let
$\gamma_0$ denote a positive generator of $\Gamma$, which is well-defined since the space of
pluriclosed metrics on $M$ is connected (cf. \cite{Telemancone} for further detail).
This form $\gg_0$ plays a key role in the characterization of the positive cone 
$\PP^{1,1}_{A}$ in the next theorem.

\begin{thm} (\cite{PCFReg} Theorem 5.6, cf. \cite{BuchdahlK, Buchdahl2, Lamari}) \label{t:conecharacterization} Let $(M^4, J)$ be a
complex non-K\"ahler surface.
 Suppose $\phi \in
\Lambda^{1,1}$ is pluriclosed.  Then $[\phi] \in \mathcal P^{1,1}_{A}$ if
and
only if
\begin{enumerate}
\item{$\int_M \phi \wedge \gamma_0 > 0$}
\item{$\int_D \phi > 0 \mbox{ for every effective divisor with negative self
intersection}$.}
\end{enumerate}
\end{thm}

\begin{question} \label{q:PCFconeques} Is there a natural characterization 
of $\PP^{1,1}_{A}$ in higher dimensions?
\end{question}

\section{Generalized K\"ahler-Ricci flow} \label{s:GKRF}

According to Definition \ref{d:GKBiherm}, a generalized K\"ahler structure can be interpreted as a metric which is pluriclosed with respect to two distinct complex structures, satisfying the further relation that $d^c_I \gw_I = - d^c_J \gw_J$.  Given this it is natural to ask whether pluriclosed flow will preserve the generalized K\"ahler condition.  The answer is yes, but with a crucial and subtle caveat, namely that the complex structures must evolve as well, a feature not common to geometric evolution in complex geometry.  Given that we have seen natural deformations of generalized K\"ahler structure using certain Hamiltonian diffeomorphisms acting on one of the complex structures in \S \ref{s:GKND}, this is actually quite natural in this setting.  We will first describe the biHermitian formulation of the resulting flow, called generalized K\"ahler-Ricci flow, and give the proof that it does indeed preserve generalized K\"ahler geometry (cf. \cite{GKRF}).  Following that we will describe the flow in terms of generalized metrics and complex structures.

\subsection{BiHermitian formulation} \label{ss:GKRFbiherm}

\begin{defn} Let $M$ be a smooth manifold, we say that a one-parameter family 
of generalized K\"ahler structures $(g_t, H_t, I_t, J_t)$ is a solution of 
\emph{generalized K\"ahler-Ricci flow} if
\begin{gather} \label{f:GKRF1}
 \begin{split}
 \dt g =&\ - \Rc + \tfrac{1}{4} H^2,\\
 \dt b =&\ - \tfrac{1}{2} d^*_g H,\\
 \dt I =&\ \tfrac{1}{2}  L_{\theta_I^{\sharp}} I,\\
 \dt J =&\ \tfrac{1}{2} L_{\theta_J^{\sharp}} J.  
 \end{split}
\end{gather}
Observe that this system of equations is in a certain sense two copies of the pluriclosed flow system (\ref{f:PCFasGRF}) running simultaneously.  Indeed this was how it was discovered, as we will indicate in Theorem \ref{t:GKpreserved} below.  This formulation, while it has a pleasing symmetry to it, has the disadvantage that \emph{both} complex structures are evolving.  When it comes time to estimate solutions of this system, it will be to our advantage to freeze one of these complex structures by a gauge transformation.  Of course it is not possible to fix \emph{both}, since $I$ and $J$ are evolving by (in general) distinct gauge transformations.  Gauge-fixing the flow to freeze $I$ and comparing against Proposition \ref{p:PCFequiv} yields
\begin{gather} \label{f:GKRFIfixed}
\begin{split}
 \dt g =&\ - \Rc + \tfrac{1}{4} H^2 - \tfrac{1}{2} L_{\theta_I^{\sharp}} g,\\
\dt b =&\ - \tfrac{1}{2} d^*_g H + \tfrac{1}{2} d \theta_I - \tfrac{1}{2} i_{\theta_I^{\sharp}} H,\\
 \dt J =&\ \tfrac{1}{2} L_{\left(\theta_J^{\sharp} - \theta_I^{\sharp}\right)} J.  
\end{split}
\qquad \longleftrightarrow \qquad
\begin{split}
 \dt \gw_I =&\ - (\rho_B^I)^{1,1}\\
  \dt \beta_I =&\ - (\rho_B^I)^{2,0}\\
 \dt J =&\ \tfrac{1}{2} L_{\left(\theta_J^{\sharp} - \theta_I^{\sharp}\right)} J.  
\end{split}
\end{gather}
where $\gb_I = \i b^{2,0}_I$. We will refer to this system as generalized K\"ahler-Ricci flow \emph{in the $I$-fixed gauge}.  Of course one can also study the flow in the $J$-fixed gauge as well.
\end{defn}

\subsection{Short-time existence}

\begin{thm} \label{t:GKpreserved} Given $M$ a smooth compact manifold and $(g, 
H,
I,J)$ a 
generalized K\"ahler structure, there exists a unique maximal solution $(g_t,H_t,I_t,J_t)$ to generalized K\"ahler-Ricci flow with this initial condition on $[0,T)$ for some $T \in \mathbb R \cup \{\infty\}$.  Moreover, the pair $(g_t,H_t)$ is the unique solution to generalized Ricci flow with initial condition $(g,H)$.
\begin{proof} We know from the definition of generalized K\"ahler structure that $(g,H,I)$ is a pluriclosed structure.  Thus by Proposition \ref{p:PCFpreserved} we obtain a solution $(g_t, H_t, I)$ to pluriclosed flow on a maximal time interval $[0,T)$.  This can be gauge-fixed as in Remark \ref{r:PCFGRF} to obtain $(g_t, H_t, I_t)$ a solution to pluriclosed flow in the form of (\ref{f:PCFasGRF}).  Similarly, the triple $(g,-H,J)$ is a pluriclosed structure, and so we obtain $(\til{g}_t, \til{H}_t, J_t)$ a solution to gauge-fixed pluriclosed flow, where $(\til{g}_t, \til{H}_t)$ is the unique solution to generalized Ricci flow with initial condition $(g,-H)$.  We use the notation $\til{g}, \til{H}$ to emphasize that we do not yet know the relationship between $(g_t,H_t)$ and $(\til{g}_t,\til{H}_t)$.

To link these two families, we make a trivial yet crucial observation: if $(g_t, H_t)$ is a solution to generalized Ricci flow, so is $(g_t,-H_t)$.  This follows since the evolution equation for $H$ is linear in $H$, whereas the evolution equation for $g$ is quadratic in $H$.  In other words,
\begin{align*}
\dt g =&\ - \Rc + \tfrac{1}{4} H^2 = - \Rc + \tfrac{1}{4} (-H)^2,\\
\dt (-H) =&\ - \tfrac{1}{2} \left( \gD_d H \right) = \tfrac{1}{2} \gD_d (-H).
\end{align*}
Therefore, $(\til{g}_t, - \til{H}_t)$ is the unique solution to generalized Ricci flow with initial condition $(g,H)$, and thus $(\til{g}_t, - \til{H}_t) = (g_t,H_t)$.  Examining the construction it follows that $(g_t,H_t,I_t,J_t)$ is a one-parameter family of generalized K\"ahler structures which satisfies (\ref{f:GKRF1}).
\end{proof}
\end{thm}

\subsection{Generalized geometry formulation}

\begin{defn} \label{d:GKRFGG} Let $(E, [,],\IP{,}) \to M$ be an exact Courant algebroid.  We say that 
a one-parameter family of generalized K\"ahler structures $(\JJ^1_t, 
\JJ^2_t)$ is a solution of \emph{generalized K\"ahler-Ricci flow} if
\begin{gather} \label{f:GKRF2}
\begin{split}
\dt \JJ^1 =&\ [\JJ_1, e^{-\rho_B^I} \JJ_1],\\
\dt \JJ^2 =&\ [\JJ_2, e^{-\rho_B^I} \JJ_2],
\end{split}
\end{gather}
where $\rho_B^I$ denotes the Bismut-Ricci form of the associated pluriclosed structure $(g, I)$.
\end{defn}

We must show that this formulation of generalized K\"ahler-Ricci flow agrees with that arising in equations (\ref{f:GKRFIfixed}).  The key points are a variational formula and a buildup of fundamental curvature identities for generalized K\"ahler manifolds.

\begin{lemma} \label{l:GKvariation} Given $(E,[,],\IP{,}) \to M$ an exact Courant algebroid, fix $K_t \in \Lambda^2 T^*$ a one-parameter family of two-forms.  A one-parameter family of generalized K\"ahler structures $(\JJ_1^t, \JJ_2^t)$ satisfies
\begin{align*}
\dt \JJ_i =&\ [ \JJ_i, e^K \JJ_i],
\end{align*}
if and only if the associated one-parameter family of biHermitian structures $(g_t, b_t, I_t, J_t)$ satisfies
\begin{gather*}
\begin{split}
\dot{g} =-\tfrac{1}{2}[K, I],&\ \qquad \dot{b} =-\tfrac{1}{2} \left(KI + IK \right),\\
\dot{\gw}_I =-\tfrac{1}{2}[K,I]I, &\ \qquad \dot{\gw}_J =-\tfrac{1}{2}\left(KIJ + IJK \right), \\
\dot{I} =0, &\ \qquad \dot{J} =\tfrac{1}{2}[I,J]g^{-1}K.
\end{split}
\end{gather*}
\begin{proof} This follows by differentiating the explicit expressions for $\JJ_i$ in equation (\ref{f:Gualtierimap}), and then comparing against the formulas for $[\JJ_i, e^K \JJ_i]$.  We leave the details as an \textbf{exercise} (cf. \cite{gibson2020deformation}).
\end{proof}
\end{lemma}

\begin{lemma} \label{l:GKcurvaturetype} Let $(M^{2n}, g, I, J)$ be a generalized K\"ahler structure.  Let $\N^I, \N^J$ denote the Bismut connections associated to the pluriclosed structures $(g, I)$ and $(g, J)$, with associated curvature tensors $R^I, R^J$.  Then
\begin{align*}
R^I \in \Lambda^{1,1}_J \otimes \Lambda^{1,1}_I, \qquad R^J \in \Lambda^{1,1}_I \otimes \Lambda^{1,1}_J.
\end{align*}
\begin{proof} We give the proof for $R^I$, the case of $R^J$ being directly analogous.  Since $\N^I$ preserves $I$ it is easy to show that $R^I$ is of type $(1,1)$ with respect to $I$ in the last two indices.  Using Remark \ref{r:GKBismutremark}, we know that if we set $H = - d^c_I \gw_I$ then $\N^I = \N^+, \N^J = \N^-$ in the notation of Definition \ref{d:Bismutconn}.  Thus using Proposition \ref{p:Bismutpair} we see
\begin{align*}
R^I(JX,JY,Z,W) =&\ R^+(JX,JY,Z,W)\\
=&\ R^-(Z,W,JX,JY)\\
=&\ R^-(Z,W,X,Y)\\
=&\ R^+(X,Y,Z,W)\\
=&\ R^I(X,Y,Z,W),
\end{align*}
as required.
\end{proof}
\end{lemma}

\begin{lemma} \label{l:GKBismutcommutator} Let $(M^{2n}, g, I, J)$ be a generalized K\"ahler structure.  Then
\begin{align*}
\rho_B^I(X, [I,J] Y) =&\ g \left( \left( L_{\theta_J^{\sharp} - \theta_I^{\sharp}} J \right) X, Y \right).
\end{align*}
\begin{proof} As an \textbf{exercise} one can show that for a Hermitian manifold $(M^{2n}, g, J)$ one has
\begin{align} \label{f:Liederiv}
g( (L_X J)Y, Z) =&\ g( (\N_X J)Y - \N_{JY} X + J \N_Y X, Z),
\end{align}
where $\N$ denotes the Levi-Civita connection.  We will compute the symmetric and skew parts of the identity separately.
First note, using that $\rho_B^I\in \Lambda^{1,1}_J$ by Lemma \ref{l:GKcurvaturetype},
\begin{align*}
\rho_B^I([I,J]X, JY) &+ \rho_B^I([I,J]JY, X)\\
=&\ \rho_B^I( IJ X - JI X, JY) - \rho_B^I (X, IJJ Y - JIJ Y)\\
=&\ \rho_B^I(IJ X, JY) - \rho_B^I(I X, Y) + \rho_B^I(X, IY) - \rho_B^I(JX, IJY)\\
=&\ - \rho_B^I(JY, IJX) + \rho_B^I(Y, IX) + \rho_B^I(X, IY) - \rho_B^I(JX, IJY).
\end{align*}
Applying Lemma \ref{l:BismutRicci} we further compute
\begin{align*}
\rho_B^I(JY, IJX)& - \rho_B^I(Y, IX) - \rho_B^I(X, IY) + \rho_B^I(JX, IJY)\\
=&\ \Rc^I(JY, JX) + \N^I_{JY} \theta^I(JX) - \Rc^I(Y, X) - \N^I_Y \theta^I(X)\\
&\ - \Rc^I(X, Y) - \N^I_X \theta^I(Y) + \Rc^I(JX, JY) + \N^I_{JX} \theta^I(JY)\\
=&\ \Rc^J(JY, JX) + \Rc^J(JX, JY) - \Rc^J(X,Y) - \Rc^J(Y,X)\\
&\ + \N^I_{JY} \theta^I(JX) - \N^I_Y \theta^I(X) - \N^I_X \theta^I(Y) + \N^I_{JX} \theta^I(JY)\\
=&\ - \rho_B^J(JY, X) - \rho_B^J(JX, Y) - \rho_B^J(X,JY) - \rho_B^J(Y, JX)\\
&\ -\N^J_{JY} \theta^J(JX) - \N^J_{JX} \theta^J(JY) + \N^J_X \theta^J(Y) + \N^J_Y \theta^J(X) \\
&\ + \N^I_{JY} \theta^I(JX) - \N^I_Y \theta^I(X) - \N^I_X \theta^I(Y) + \N^I_{JX} \theta^I(JY)\\
=&\  \N_X(\theta^J - \theta^I)(Y) + \N_Y (\theta^J - \theta^I)(X) \\
&\ - \N_{JX} \left(\theta^J - \theta^I \right) (JY) - \N_{JY} \left(\theta^J - \theta^I \right)(JX)\\
=&\ g \left( \left(L_{\theta_J^{\sharp} - \theta_I^{\sharp}} J \right) X, JY \right) + g\left( \left(L_{\theta_J^{\sharp} - \theta_I^{\sharp}} J \right) JY, X \right).
\end{align*}
where the last line follows by comparing against (\ref{f:Liederiv}).

To address the skew-symmetric piece we first compute, again using that $\rho_B^I \in \Lambda^{1,1}_J$,
\begin{align*}
\rho_B^I & \left( [I,J] X, Y \right) - \rho_B^I \left( [I,J]Y, X \right)\\
=&\ \rho_B^I((IJ - JI) X, Y) + \rho_B^I(X,(IJ - JI)Y)\\
=&\ \rho_B^I(IJ X, Y) + \rho_B^I(IX, JY) + \rho_B^I(X, IJ Y) + \rho_B^I(JX, IY)\\
=&\ 2 \left(\left( \rho_B^I \right)^{2,0 + 0,2}(JX, IY) + \left( \rho_B^I \right)^{2,0 + 0,2}(IX, JY) \right).
\end{align*}
Now we compute using Lemma \ref{l:BismutRicci}, Proposition \ref{p:BismutRicci} and equation (\ref{f:Liederiv}),
\begin{align*}
2 & \left(\left( \rho_B^I \right)^{2,0 + 0,2}(JX, IY) + \left( \rho_B^I \right)^{2,0 + 0,2}(IX, JY) \right)\\
=&\ - d^* H_I(JX, Y) + d^* H_I(JY,X) + d^{\N^I} \theta^I(JX,Y) - d^{\N^I} \theta^I (JY,X)\\
=&\ d^* H_J (JX,Y) - d^* H_J(JY, X) + d^{\N^I} \theta^I(JX, Y) - d^{\N^I} \theta^I (JY,X)\\
=&\ - 2 (\rho_B^J)^{2,0 + 0,2} (X,Y) + d^{\N^J} \theta^J(JX, Y) + 2 (\rho_B^J)^{2,0 + 0,2}(Y,X) - d^{\N^J} \theta^J(JY, X)\\
&\ + d^{\N^I} \theta^I(JX, Y) - d^{\N^I} \theta^I (JY, X)\\
=&\  - d\theta^J(JX, Y) + d \theta^J(JY, X) + d\theta^I(JX, Y) - d \theta^I (JY, X)\\
=&\ g \left( \left( L_{\theta_J^{\sharp} - \theta_I^{\sharp}} J \right) X, Y \right) - g \left( \left( L_{\theta_J^{\sharp} - \theta_I^{\sharp}} J \right) Y, X \right).
\end{align*}
Adding together the two pieces above and rearranging yields the result.
\end{proof}
\end{lemma}

\begin{prop} \label{p:GKflowequiv} Let $(E, [,],\IP{,}) \to M$ be an exact Courant algebroid.  The family 
$(\JJ^1_t, 
\JJ^2_t)$ is a solution of generalized K\"ahler-Ricci flow as in (\ref{f:GKRF2}) if and only if the 
associated family of biHermitian structures $(g_t, H_t, I_t, J_t)$ is a solution of generalized K\"ahler-Ricci flow in the $I$-fixed gauge as in (\ref{f:GKRFIfixed}).
\begin{proof} We first assume that equations (\ref{f:GKRF2}) hold.  Using Lemma \ref{l:GKvariation} with $K = - \rho_B^I$ one notes that $I$ is fixed and one easily derives the evolution equation $\dot{\gw}_I = - (\rho_B^I)^{1,1}$.  By Lemma \ref{l:GKBismutcommutator} it follows that
\begin{align*}
\dot{J} = - \tfrac{1}{2} [I,J] g^{-1} \rho_B^I = L_{\tfrac{1}{2} \left(\theta_J^{\sharp} - \theta_I^{\sharp}\right) } J.
\end{align*}
as required.  Thus we have verified the evolution equations of (\ref{f:GKRFIfixed}).  Note that the induced evolution equation for $b$ is not strictly necessary in defining the generalized K\"ahler-Ricci flow from the biHermitian point of view.  Nonetheless this torsion potential (cf. \S \ref{s:TPEE}) plays an important role in obtaining estimates for even pluriclosed flow.  Assuming equations (\ref{f:GKRFIfixed}) and further imposing $\dot{b} = - \{\rho_B^I, I\}$, equations (\ref{f:GKRF2}) follow in a similar way using Lemma \ref{l:GKvariation}.
\end{proof}
\end{prop}

\subsection{Evolution of the Poisson tensor}

As described in \S \ref{ss:Poisson}, associated to every generalized K\"ahler structure we have the real Poisson structure $\gs$.  A basic question is how this tensor evolves along a solution to generalized K\"ahler-Ricci flow.  Since the $I$-$(2,0)$ piece of $\gs$ is $I$-holomorphic by Proposition \ref{p:holoPoisson}, the flow in the $I$-fixed gauge can move $\sigma$ at most within the finite dimensional space of such holomorphic Poisson structures.  As it turns out, it remains completely fixed, a consequence of the curvature identities of the previous subsection.

\begin{prop} \label{p:Poissonfixed} Let $(M^{2n}, g_t, I, J_t)$ be a solution to generalized K\"ahler-Ricci flow in the $I$-fixed gauge.  Let $\gs_t = \tfrac{1}{2} g^{-1}_t [I, J_t]$ be the one-parameter family of associated Poisson structures.  Then
\begin{align*}
\gs_t \equiv \gs_0.
\end{align*}
\begin{proof} Using the evolution equations in the $I$-fixed gauge (\ref{f:GKRFIfixed}) we obtain
\begin{align*}
\dt \gs =&\ \dt \tfrac{1}{2} [I, J] g^{-1}\\
=&\ \tfrac{1}{2} \left([I, L_{\tfrac{1}{2} \left(\theta_J^{\sharp} - \theta_I^{\sharp} \right)} J] g^{-1} + [I, J] g^{-1} (\rho_B^I)^{1,1} (\cdot, I \cdot) g^{-1} \right)\\
=&\ 0,
\end{align*}
where the last line follows using Lemma \ref{l:GKBismutcommutator}.
\end{proof}
\end{prop}

\begin{rmk} Note that the tensor $\gs$ scales as $g^{-1}$, thus $\gs$ will change along different normalizations of the flow.  In particular, considering the normalized flows as in Definition \ref{d:nonsingular}, it follows easily from Proposition \ref{p:Poissonfixed} that the Poisson tensors $\gs_t$ will satisfy
\begin{align*}
\gs_t = e^{- \gL t} \gs_0.
\end{align*}
For instance, on $\mathbb CP^2$ we expect the normalized pluriclosed flow (with $\gL = 1$) to exist globally and converge to the Fubini-Study metric, and this is verified in some cases.  Thus for the generalized K\"ahler structures on $\mathbb CP^2$ with nontrivial Poisson tensor described in Remark \ref{r:HitchinCP2}, the normalized flow causes $\gs$ to exponentially decay to zero, yielding K\"ahler structure for any smooth limit.
\end{rmk}

\subsection{Positive cone in commuting generalized K\"ahler geometry}

In this subsection we define a cohomological cone for generalized
K\"ahler metrics with vanishing Poisson tensor, extending the idea of the K\"ahler cone in Definition \ref{d:Kahlercone} and specializing the positive cone in the pluriclosed setting in Definition \ref{d:pluriclosedcone}.  To begin we record a preliminary definition which captures all the properties of being a generalized K\"ahler metric in the commuting setting, aside from the positivity condition.

\begin{defn} \label{d:formallyGK}  Let $(M^{2n}, g, I, J)$ be a generalized K\"ahler manifold with $[I,J] = 0$, and let $\Pi = IJ$.  Given $\phi_I \in \Lambda^{1,1}_{I,\mathbb R}$, let 
\begin{align*}
\phi_J = - \phi_I(\Pi \cdot, \cdot) \in \Lambda^{1,1}_{J,\mathbb R}.
\end{align*}
We say that $\phi_I$ is
\emph{formally generalized K\"ahler} if
\begin{gather*}
\begin{split}
d^c_{I} \phi_I =&\ - d^c_{J} \phi_J,\\
d d^c_{I} \phi_I =&\ 0.
\end{split}
\end{gather*}
\end{defn}

It is a basic \textbf{exercise} using the computations of \S \ref{ss:GKVPS} to show that, for a smooth function $f$, the tensor
\begin{align*}
\i \left( \del_+ \delb_+ - \del_- \delb_- \right) f
\end{align*}
is formally generalized K\"ahler.  Using these facts we can give the notion of positive cone in this setting.

\begin{defn} \label{d:CGKcone}  Let $(M^{2n}, g, I, J)$ be a generalized K\"ahler manifold
such that $[I,J] = 0$.  Let
\begin{align*}
 \mathcal H := \frac{ \left\{ \phi_I \in \Lambda^{1,1}_{I, \mathbb R}\ |  \
\phi_I
\mbox{ is formally generalized K\"ahler} \right\}}{ \left\{ \i \left( \del_+ \delb_+ - \del_- \delb_- \right) f \ |\ f \in C^{\infty} \right\}}.
\end{align*}
Furthermore, define the \emph{generalized K\"ahler cone} by
\begin{align*}
 \mathcal K := \{ [\phi] \in \mathcal H \ | \ \exists\ \gw \in [\phi], \gw > 0
\}.
\end{align*}
\end{defn}

\begin{defn} Let $(M^{2n}, g_0, I, J)$ be a generalized K\"ahler manifold
satisfying $[I,J] = 0$.  Let $h_{\pm}$ denote Hermitian metrics on
$T_{\pm}^{\mathbb C} M$, and let $P = P(h_{\pm})$ as in Proposition \ref{p:CGKBismutcurvature}.  Then set
\begin{align*}
 \tau^*(g_0) := \sup \{ t > 0 \ | \ [\gw_0- t P] \in \mathcal K \}.
\end{align*}
\end{defn}

Akin to Lemma \ref{l:flowtimeUB}, certainly $\tau^*$ represents the maximum possible existence time
for a solution
to pluriclosed flow in this setting.  In analogy with the Theorem of Tian-Zhang
\cite{TianZhang} for K\"ahler-Ricci flow, and as a specialization of the
corresponding general conjecture for pluriclosed flow \cite{PCFReg} we make the
conjecture that $\tau^*$ does indeed represent the smooth maximal existence
time.

\begin{conj} \label{c:CGKconeconj} Let $(M^{2n}, g_0, I, J)$ be a generalized
K\"ahler manifold
satisfying $[I,J] = 0$.  Then
the solution to generalized K\"ahler-Ricci flow with initial condition $g_0$ exists on
$[0,\tau^*(g_0))$.
\end{conj}

\begin{question} Is it possible to give a natural characterization of the generalized K\"ahler cone $\KK$ in terms of cohomological conditions and subvarieties?
\end{question}

\begin{question} \label{q:blowupGKques} Recently, in \cite{Finoblowups} the notion of blowup is extended to certain cases in the context of pluriclosed metrics.  Also, in \cite{Gualtieriblowups, JoeyGKblowups}, the idea of blowup has been extended to generalized complex and generalized K\"ahler manifolds.  A natural but broad question to ask is, what is the relationship of these blowups to pluriclosed flow and generalized K\"ahler-Ricci flow?  Do they inevitably produce finite time singularities of the flow?
\end{question}

\section{Reduced flows} \label{s:reduced}

A fundamental observation in the analysis of K\"ahler-Ricci flow is the 
reduction of the flow to a parabolic partial differential equation for a 
potential function.  In particular, as we show in Proposition \ref{p:KRFscalar} 
below, on any compact time interval inside $[0,\tau^*(g_0))$, the flow is determined 
by a function evolving by a kind of parabolic Monge-Amp\`ere equation.  As pluriclosed metrics are determined locally by a $(1,0)$-tensor (cf. Lemma \ref{l:PClocalddbar}), we cannot expect a scalar reduction for pluriclosed flow, but rather a reduction to a flow of $(1,0)$-forms, and this is carried out in Proposition \ref{p:PCFoneform} (cf. \cite{PCFReg}, with refinements in \cite{StreetsPCFBI}).  For generalized K\"ahler-Ricci flow, in Chapter \ref{c:GCG} we saw examples of variations of generalized K\"ahler metrics determined by one function, and so we can again hope for a reduction of the flow to a scalar PDE.  This is carried out below in some settings, although the nature of the scalar reduction varies considerably depending on the rank of the Poisson structure underlying the solution (cf. \cite{ASnondeg}, \cite{StreetsPCFBI}, \cite{StreetsND}, \cite{StreetsSTB}).

\subsection{Scalar reduction of K\"ahler-Ricci flow}

\begin{prop} \label{p:KRFscalar} Let $(M^{2n}, g_0, J)$ be a compact K\"ahler 
manifold, and fix $\tau < \tau^*(g_0)$.  There exists a volume form $\Omega$ 
and a one-parameter family $\gw_t$ of K\"ahler metrics, $0 \leq t \leq \tau$, 
such that $\gw_t \in [\gw_0] - t c_1$, and if $u_t \in C^{\infty}(M)$ satisfies
\begin{align} \label{f:KRFPMA}
\dt u = \log \frac{\left(\gw_t + \i \del \delb 
u\right)^n}{\Omega}, \qquad u(0) = 0,
\end{align}
then the one-parameter family of metrics $\gw^u_t := \gw_t + \i \del \delb u$ 
is the unique solution to K\"ahler-Ricci flow with initial condition $\gw_0$.
\begin{proof} Fix some background volume form $\til{\Omega}$.  Since $\tau < 
\tau^*$ by hypothesis, we may fix $f \in C^{\infty}(M)$
such that $\gw_0 - \tau \rho_C(\til{\Omega}) + \i \del \delb f > 0$.  Setting 
$\Omega = e^{- f/\tau} \til{\Omega}$, this now reads $\gw_0 - \tau \rho_C(\Omega) > 0$.  Now 
set $\gw_t = \gw_0 - t \rho_C(\Omega)$, which certainly satisfies $\gw_t \in 
[\gw_0] - t c_1$.  Now suppose $u_t$ satisfies (\ref{f:KRFPMA}).  By 
construction it is clear that $\gw_0^u = \gw_0$.  Furthermore we compute using the transgression formula
\begin{align*}
\frac{\del}{\del t} \gw_t^u =&\ \dt \left( \gw_t + \i \del \delb u \right)\\
=&\ - \rho_C(\Omega) + \i \del \delb \log \frac{\left(\gw_t + \i \del \delb 
u\right)^n}{\Omega}\\
=&\ - \rho_C( \gw_t^u ),
\end{align*}
as required.
\end{proof}
\end{prop}

\subsection{$1$-form reduction of pluriclosed flow}

\begin{prop} \label{p:PCFoneform} Let $(M^{2n}, g_0, J)$ be a compact 
manifold with pluriclosed metric, and fix $\tau < \tau^*(g_0)$.  There exists a volume form 
$\Omega$, $\mu \in \Lambda^{1,0}$, and a one-parameter family ${\gw}_t$ of pluriclosed metrics, $0 \leq t 
\leq \tau$, satisfying ${\gw}_t \in [\gw_0] - t c_1$, such that if $\ga_t \in 
\Lambda^{1,0}$ satisfies
\begin{align} \label{f:PCFreduction}
\frac{\del \ga}{\del t} = - \delb^*_{\gw^{\ga}_t} \gw^{\ga}_t - \tfrac{\i}{2} \del \log 
\frac{\left(\gw_t^{\ga} \right)^n}{\Omega} - \frac{\mu}{\tau}, \qquad \ga(0) = 
0,
\end{align}
where
\begin{align*}
\gw^{\ga}_t := {\gw}_t + (\delb \ga + \del \bga),
\end{align*}
then the one-parameter family of metrics $\gw^{\ga}_t$ is the unique solution to pluriclosed flow with initial condition $\gw_0$.
\begin{proof} Fix some background volume form $\Omega$.  Since $\tau < \tau^*$ 
by hypothesis, we may fix $\mu \in \Lambda^{1,0}$
such that $\gw_{\tau} := \gw_0 - \tau \rho_C(\Omega) + \left( \delb \mu + \del 
\bmu \right)  > 0$.  Now set $\gw_t = \frac{t}{\tau} \gw_{\tau} + \frac{\tau - 
t}{\tau} \gw_{0} $, which certainly satisfies $\gw_t \in [\gw_0] - t c_1$, 
interpreted as an equation of Aeppli cohomology classes.  Now suppose $\ga_t$ 
satisfies (\ref{f:PCFreduction}).  By construction it is clear that 
$\gw_0^{\ga} = \gw_0$.  Furthermore we compute
\begin{align*}
\frac{\del}{\del t} \gw_t^{\ga} =&\ \dt \left( \gw_t + \delb \ga + \del \bga 
\right)\\
=&\ - \rho_C(\Omega) + \frac{1}{\tau} (\delb \mu + \del \bmu)+ \delb \left( 
- \delb^*_{\gw_t^{\ga}} \gw_t^{\ga} - \tfrac{\i}{2} \del \log \frac{ \left( 
\gw_t^{\ga} \right)^n}{\Omega} - \frac{\mu}{\tau} \right)\\
&\ + \del \left( - \del^*_{\gw_t^{\ga}} \gw_t^{\ga} + \tfrac{\i}{2} \delb \log 
\frac{ \left( \gw_t^{\ga} \right)^n}{\Omega} - \frac{\bmu}{\tau} \right)\\
=&\ - \del \del^*_{\gw_t^{\ga}} \gw_t^{\ga} - \delb \delb^*_{\gw_t^{\ga}} \gw_t^{\ga} 
+ \i \del \delb \log \frac{\left( \gw_t^{\ga} \right)^n}{\Omega} - 
\rho_C(\Omega)\\
=&\ - \del \del^*_{\gw_t^{\ga}} \gw_t^{\ga} - \delb \delb^*_{\gw_t^{\ga}} \gw_t^{\ga} 
- \rho_C(\gw_t^{\ga})\\
=&\ - \rho_B^{1,1}(\gw_t^{\ga}),
\end{align*}
as required.
\end{proof}
\end{prop}

\begin{rmk} \label{r:torspot} We note that the local potential $\ga$ in Proposition \ref{p:PCFoneform} also recovers the flow for the $(2,0)$-form $\gb$.  In particular, if the initial condition satisfies
\begin{equation}\label{eq:misterious}
\partial \omega_0 + d\beta_0 = \tilde H_0,
\end{equation}
for $\tilde H_0 \in \Lambda^{3,0 + 2,1}$, $d \tilde H_0 = 0$, and we set $\gb_t = \beta_0 + \del \left(\ga + t \frac{\mu}{\tau} \right)$ then condition \eqref{eq:misterious} is preserved along the flow and furthermore
\begin{align*}
\dt \gb =&\ \del \left( \dt \ga + \frac{\mu}{\tau} \right) = - \del \delb^*_{g_t^{\ga}} \gw_t^{\ga} = - \rho_B^{2,0},
\end{align*}
as required. Thus, the one-form reduction in Proposition \ref{p:PCFoneform} can be regarded as a natural flow for metrics on a fixed holomorphic Courant algebroid, as defined in \S\ref{ss:holCourant} (cf. \cite{garciafern2018canonical}). 
\end{rmk}

\subsection{Scalar reduction of generalized K\"ahler-Ricci flow in commuting case}

\begin{prop} \label{p:GKRFcommuting} Let $(M^{2n}, g_0, I, J)$ be a compact 
generalized K\"ahler manifold satisfying $[I,J] = 0$.  Given $\tau < \tau^*(g_0)$, there exist partial volume forms $\Omega_{\pm}$ and a one-parameter family of generalized K\"ahler metrics $\gw_t$ such that if $u_t \in C^{\infty}(M)$ satisfies
\begin{gather} \label{f:GKRFreduction}
\begin{split}
  \dt u =&\ \log \frac{\left( (\gw_+)_t + \i \del_+ \delb_+ u \right)^{\rank T_+} \wedge \Omega_-}{\Omega_+ \wedge \left( (\gw_-)_t - \i \del_- \delb_- u \right)^{\rank T_-}},\\
 u(0) =&\ 0,
\end{split}
\end{gather}
then the one-parameter family of metrics
\begin{align*}
\gw^u_t := \gw_t + \i \left( \del_+ \delb_+ - \del_- \delb_- \right) u
\end{align*}
is the unique solution to generalized K\"ahler-Ricci flow with initial condition $\gw_0$.

\begin{proof} Since $\tau < \tau^*(g_0)$, if we fix $\til{h}$ a Hermitian metric, then by definition
\begin{align*}
[\gw_0 - \tau P(\til{h})] \in \mathcal K,
\end{align*}
thus there exists $f \in C^{\infty}$ such that 
\begin{align*}
\gw_0 - \tau P(\til{h}) + \i \left( \del_+ \delb_+ - \del_- \delb_- \right) f > 0.
\end{align*}
Using the transgression formula for $P$ (equation \ref{f:commutingtransgression}), if we set $h = e^{f} \til{h}_+ \oplus \til{h}_-$, then this formula reads
\begin{align*}
\gw_0 - \tau P(h) > 0.
\end{align*}
It follows that $\gw_t := \gw_0 - t P(h)$ is positive for all $0 \leq t \leq \tau$.  Define $\Omega_{\pm} = (\gw_{\pm}^h)^{\rank T_{\pm}}$, and then we can compute for $u$ satisfying (\ref{f:GKRFreduction}),
\begin{align*}
 \dt \gw_t^u =&\ - P(h) + \i \left( \del_+ \delb_+ - \del_- \delb_- \right) \log \frac{\left( (\gw_+)_t + \i \del_+ \delb_+ u \right)^k \wedge \Omega_-}{\Omega_+ \wedge \left( (\gw_-)_t - \i \del_- \delb_- u \right)^l},\\
 =&\ - P(\gw_t^u)\\
=&\ - \rho_B^{1,1}(\gw^u_t),
\end{align*}
as required.
\end{proof}
\end{prop}

\begin{rmk} This proposition indicates that in the commuting setting, the GRKF in the $I$-fixed gauge will fix the complex structure $J$ as well (see also Lemma \ref{l:CGKLeeform}).  By contrast, we will see in \S \ref{ss:SRND} below that, in the nondegenerate setting, the evolution of the complex structures determines the entire structure.
\end{rmk}

\subsection{Scalar reduction of generalized K\"ahler-Ricci flow in nondegenerate
case} \label{ss:SRND}

As explained in Proposition \ref{p:nondegvariations}, a natural class of deformations of generalized K\"ahler structure in the nondegenerate setting arises via families of $\Omega$-Hamiltonian diffeomorphisms.  The proposition below shows that the generalized K\"ahler-Ricci flow in the nondegenerate setting is a deformation of this kind, with the family of diffeomorphisms driven by the Bismut-Ricci potential.

\begin{prop}\label{p:Hamiltonian-flow} Suppose $(g_t, I, J_t)$ is a smooth
solution of the generalized K\"ahler-Ricci flow in the $I$-fixed gauge with $\gs_0$ nondegenerate, and let
\begin{align*}
\Phi_t = \log \frac{\det (I + J_t)}{\det (I - J_t)}.
\end{align*}
Let $\phi_t$ denote the flow of the $\Omega$-Hamiltonian
vector field
\begin{align*}
X_{t}:= \gs d \Phi_t.
\end{align*}
Then the induced family of generalized K\"ahler structures $(g_{\phi_t},
I, J_{\phi_t})$ obtained via Proposition \ref{p:nondegvariations} coincides
with
$(g_t, I, J_t)$.
\end{prop}
\begin{proof} By Proposition \ref{p:NDRiccipotential}, $X_t =
(\theta^t_J-\theta^t_I)^{\sharp}$,  where $\theta^t_I$ and
$\theta^t_J$ are the Lee forms along the GKRF and $\sharp$ denotes $g_t^{-1}$. 
It thus follows that
$$\dt (J_{\phi_t}  - J_t)= 0,$$
showing that  $J_{\phi^*_t}  =J_t$ as they equal $J$ at $t=0$. Also, as $\phi_t$ are $\Omega$-Hamiltonian, it follows that
\begin{align*}
{\rm Im} ( \phi_t^*\Omega_{J}) = {\rm Im} ((\phi_t^*J)\Omega) ={\rm Im}( J_t\Omega).
\end{align*}
Taking $(1,1)$-part with respect to $I$ gives $g_t = g_{\phi_t}$, as required.
\end{proof}

\begin{rmk} In many ways Proposition \ref{p:Hamiltonian-flow} serves as a natural analogue of the scalar reduction of K\"ahler-Ricci flow and generalized K\"ahler-Ricci flow in the case $[I, J] = 0$.  We have shown that the flow is determined at each time by a single scalar function, but it remains to be seen if the geometry can be described globally by a single function (cf. Remark \ref{r:scalarreduction}).
\end{rmk}

\section{Torsion potential evolution equations} \label{s:TPEE}

A key challenge in understanding solutions to the pluriclosed flow is to control the torsion tensor.  As the $1$-form potential $\ga$ for the flow derived in Proposition \ref{p:PCFoneform} determines the metric tensor, it of course also determines the torsion tensors as well.  Specifically, as explained in Remark \ref{r:torspot}, the tensor $\del \ga$ serves as a holomorphic potential for the torsion of the Chern connection,
\begin{align*}
T = \del \gw,
\end{align*}
and is instrumental in bounding the torsion tensor.  Somewhat miraculously, this torsion potential satisfies a very clean evolution equation, derived in Proposition \ref{p:torsionpotentialev} below.  Here and in the remainder of this chapter, to save on notation, unless otherwise specified all connections will be Chern connections associated to a given time-dependent Hermitian metric.

\begin{lemma} \label{l:01formgenev} Let $(M^{2n},
g_t, J)$ be a solution to
pluriclosed flow, and
suppose $\gb_t \in \Lambda^{p,0}$ is a one-parameter family satisfying
\begin{align*}
 \dt \gb =&\ {\gD}^C \gb + \mu,
\end{align*}
where $\mu_t \in \Lambda^{p,0}$.  Then
\begin{align*}
\left( \dt - \gD^C \right) \brs{\gb}^2 =&\ - \brs{\N \gb}^2 - \brs{\bar{\N} \gb}^2 -
p \IP{Q, \tr_g \left(\gb \otimes \bar{\gb} \right)} + 2 \Re \IP{\gb,\bmu}.
\end{align*}
\begin{proof} We directly compute using the given evolution equation for $\gb$ and the pluriclosed flow equation (\ref{f:HMPCF}),
\begin{gather} \label{l:01form10}
\begin{split}
\dt \brs{\gb}^2 =&\ \dt \left( g^{\bj_i i_i} \dots g^{\bj_p i_p} \gb_{i_1 \dots i_p} \gb_{i_1 \dots i_p} \bgb_{\bj_i \dots \bj_p} \right)\\
=&\ - p \IP{ \frac{\del g}{\del t}, \tr_g \left( \gb \otimes \bgb \right)} + \IP{ \gD^C \gb + \mu, \bgb} + \IP{\gb, \bar{\gD}^C \bgb + \bmu}\\
=&\ p \IP{ S_C - Q, \tr_g \left( \gb \otimes \bgb \right)} + \IP{ \gD^C \gb , \bgb} + \IP{\gb, \bar{\gD}^C \bgb} + 2 \Re \IP{\gb, \bmu}.
\end{split}
\end{gather}
A standard computation using the Leibniz rule shows that
\begin{align*}
\gD^C \brs{\gb}^2 =&\ \IP{ \gD^C \gb , \bgb} + \IP{\gb, {\gD}^C \bgb} + \brs{\N \gb}^2 + \brs{\bar{\N} \gb}^2.
\end{align*}
Also, using that $\gb$ has pure type $(p,0)$, a computation using the definition of the Chern curvature tensor shows that
\begin{align*}
\left[ \left(\gD^C - \bar{\gD}^C \right) \gb \right]_{i_1 \dots i_p} =&\ g^{\bs r} \left( \N_r \N_{\bs} - \N_{\bs} \N_r \right) \gb_{i_1 \dots i_p}\\
=&\ \sum_{k=1}^p g^{\bs r} \Omega_{r \bs i_k}^t \gb_{i_1 \dots i_{k-1} t i_{k+1} \dots i_p}\\
=&\ \sum_{k=1}^p (S_C)_{i_k}^t \gb_{i_1 \dots i_{k-1} t i_{k+1} \dots i_p}.
\end{align*}
It follows that
\begin{align*}
\IP{\gb, \bar{\gD}^C \bar{\gb}} = \IP{\gb, \gD^C \bar{\gb}} - p \IP{S_C, \tr_g (\gb \otimes \gb)}.
\end{align*}
Plugging these computations into (\ref{l:01form10}) yields the result.
\end{proof}
\end{lemma}

\begin{prop} \label{p:torsionpotentialev} Let $(M^{2n}, g_t, J)$ be a solution to
pluriclosed flow.  Fix background data $\hat{g}_t, \Omega, \mu$ as in Proposition \ref{p:PCFoneform} and a solution
$\ga_t$ to (\ref{f:PCFreduction}).  Then
\begin{align*}
\left( \dt - \gD^C \right) \del \ga =&\ - \tr_{g^{\ga}} \N^{g^{\ga}} T_{\hat{g}}
+ \del \mu,\\
\left( \dt - \gD^C \right) \brs{\del \ga}^2 =&\ - \brs{\N \del
\ga}^2 - \brs{\bar{\N} \del \ga}^2 - 2 \IP{Q, \tr\del \ga \otimes
\delb \bar{\ga}}\\
&\  - 2 \Re \IP{\tr_{g^{\ga}} \N^{g^{\ga}} T_{\hat{g}} + \del \mu,
\delb \bga}.
\end{align*}
\begin{proof} The first equation follows directly from a lengthy computation using (\ref{f:PCFreduction}), see \cite{StreetsPCFBI} for details.  The
second equation follows from the first and Lemma \ref{l:01formgenev}.
\end{proof}
\end{prop}

\begin{prop} \label{p:specialtorspot} Let $(M^{2n}, g_t, J)$ be a solution to
pluriclosed flow.  Fix background data $\hat{g}_t, \Omega, \mu = 0$ as in Proposition \ref{p:PCFoneform} and a solution
$\ga_t$ to (\ref{f:PCFreduction}).  Suppose furthermore that there exists $\eta \in \Lambda^{2,0}$ such that
\begin{align} \label{f:torsioncoh}
\del \hat{\gw}_t = \del \hat{\gw}_0 = \delb \eta.
\end{align}
Let $\phi = \del \ga - \eta$.  Then
\begin{gather} \label{f:specialtorsev}
\begin{split}
\left(\dt - \gD^C \right) \phi =&\ 0,\\
\left(\dt - \gD^C \right) \brs{\phi}^2 =&\ - \brs{\N \phi}^2 -
\brs{T_{g_t}}^2
- 2 \IP{Q, \phi \otimes \bar{\phi}}.
\end{split}
\end{gather}
\begin{proof} We observe that
\begin{align*} 
\left(\tr_{g_{\ga}} \N^{g_{\ga}} T_{\hat{g}} \right)_{ij} =&\ g_{\ga}^{\bq p} 
\N_p \del
\hat{\gw}_{ij
\bq}= g_{\ga}^{\bq p} \N_p \N_{\bq} \eta_{ij} = \gD_{g_{\ga}} \eta_{ij}.
\end{align*}
Thus using Proposition \ref{p:torsionpotentialev} and the assumption that $\mu =
0$ we obtain
\begin{align*}
\dt \phi =&\ \dt \del \ga = \gD_{g_t} \del \ga - \tr_g \N^g T_{\hat{g}} =
\gD_{g_t} \left( \del \ga - \eta \right) = \gD_{g_t} \phi.
\end{align*}
This yields the first claimed equation, and then the second follows from Lemma
\ref{l:01formgenev} and the fact that
\begin{align*}
\bar{\N} \phi =&\ \delb \phi = \delb \left( \del \ga - \eta \right) = -
T_{g_{t}}.
\end{align*}
\end{proof}
\end{prop}

\begin{rmk} \label{r:torsionremark} We observe that, in the setting of Proposition \ref{p:specialtorspot}, the hypothesis of (\ref{f:torsioncoh}) will automatically be satisfied if the initial torsion $\del \gw_0$ has vanishing \v Cech cohomology class in  $H^1(\Lambda^{2,0}_{cl})$, so that the pluriclosed metrics are naturally defined on the trivial holomorphic Courant algebroid (cf. Theorem \ref{t:holCourant}).  One special case occurs when $(M^{2n}, J)$ satisfies the $\i\del\delb$ lemma.  In this case we observe that $\del \hat{\gw}_0$ represents the zero cohomology class in $H^{1,2}_{\delb}$, and thus by the $\i\del\delb$ lemma, it represents the zero cohomology class in the Bott-Chern cohomology $H^{1,2}_{BC}$.  Thus there exists $\xi \in \Lambda^{1,0}$ such that $\del \hat{\gw}_0 = \i \del \delb \xi = \delb \left(- \i \del \xi \right)$.  It is easy to see from the construction of Proposition \ref{p:PCFoneform} that since $\mu = 0$, we have $\del \hat{\gw}_t = \del \hat{\gw}_0 = \delb \left(- \i \del \xi \right)$.  It furthermore follows in this setting that $\brs{\phi}^2$ serves as a torsion-bounding subsolution in the sense of \S \ref{s:shrinker}.
\end{rmk}

\section{Higher regularity from uniform parabolicity} \label{s:HRUP}

In Theorem \ref{t:Rmblowup} we showed that the Riemann curvature tensor must blow up at any finite time singularity of generalized Ricci flow.  The key point is that in the presence of a bound on the Riemann curvature one can obtain the higher-order estimates of all derivatives of curvature from Theorem \ref{t:smoothing}.  For the case of pluriclosed flow it is possible to significantly strengthen the characterization of finite time singularities using the special structure of generalized complex geometry.  

We first recall that in the setting of K\"ahler-Ricci flow we can reduce to the parabolic complex Monge-Amp\`ere equation using Proposition \ref{p:KRFscalar}.  Assuming uniform equivalence of the evolving metric with a background, one has a strictly parabolic, fully nonlinear, convex 
PDE, with coefficients in $L^{\infty}$.  If these coefficients were slightly 
more regular, i.e. in $C^{\ga}$, then a standard bootstrapping argument using 
Schauder estimates would yield $C^{\infty}$ 
regularity of the flow.
Thus improving the coefficients from $L^{\infty}$ to $C^{\ga}$ becomes a 
crucial step.  These $C^3$ estimates for the potential function follow directly by applying the maximum principle 
to a precisely chosen geometric quantity, namely the difference between the evolving Chern connection and a background (cf. \cite{Calabiaffine,YauCC}), which gives the 
needed regularity.  This estimate can also follow in this setting following the classical works of Evans-Krylov \cite{EvansC2a, Krylov}.

One faces a similar challenge in the setting of pluriclosed flow, where one seeks regularity of the flow assuming uniform equivalence of the metric tensor with some background.  Comparing against equation (\ref{p:PCFpres10}), one then sees that the 
metric satisfies a uniformly parabolic equation with $L^{\infty}$ coefficients 
and quadratic first-order nonlinearity.  By a combination of rescaling and 
Schauder estimates, it is possible to show that a $C^{\ga}$ estimate for the 
metric would suffice to obtain full $C^{\infty}$ estimates.  Thus obtaining an 
a priori $C^{\ga}$ estimate for the
metric in the presence of upper and lower bounds on the metric becomes a 
crucial barrier.  It is hard to initially guess what the right quantity to study is to obtain this estimate, as the pluriclosed flow is a system of equations, and for parabolic systems this regularity is known in general to be false.  Moreover, we note that even in the restricted setting of generalized K\"ahler-Ricci flow, the scalar reduction in (\ref{f:GKRFreduction}) is a \emph{nonconvex} scalar PDE, and thus the classical methods will not apply.

Nonetheless, using ideas from generalized geometry, it is possible to obtain this regularity.  As mentioned above, for K\"ahler-Ricci flow the appropriate quantity to study is the difference of Chern connections of the flowing metric and a background.  For the pluriclosed flow, it turns out that there is a clean evolution equation for the difference of Chern connections associated to certain generalized Hermitian metrics on $T^{1,0} \oplus T^*_{1,0}$ defined using the underlying Hermitian metric on $T^{1,0}$ and an associated torsion potential.  We refer the reader to (\cite{StreetsSTB, StreetsWarren, StreetsPCFBI, JordanStreets}) to see the conceptual buildup leading to the relatively simple formulation below.  We first record the relevant regularity statement expressed using classical objects, then an equivalent version expressed using generalized metrics which is the version we will prove.

\begin{defn} \label{d:EKfdef1} Fix a complex manifold $(M^{2n}, J)$, and let $g$, $\til{g}$ denote Hermitian metrics.  Let
\begin{align*}
\gU(g,\til{g}) := \N^C_g - \N^C_{\til{g}}
\end{align*}
denote the difference of the Chern connections associated to $g$ and $\til{g}$.  Furthermore, let
\begin{align*}
f_k(g,\til{g}) := \sum_{j=0}^k \brs{\N_{g}^{j} \gU(g,\til{g})}^{\frac{2}{1+j}}.
\end{align*}
The quantity $f_k$ is a natural measure of the $(k+1)$-st derivatives of the metric
which scales as the inverse of the metric.
\end{defn}

\begin{thm} \label{t:EKthm2}  Let $(M^{2n}, J)$ be a compact complex manifold.
 Suppose $g_t$ is a solution to the pluriclosed flow on $[0,\tau)$, $\tau \leq
1$, with
$\ga_t$ a solution to the $(\hat{g}_t,\Omega,\mu)$-reduced flow (see Proposition 
\ref{p:PCFoneform}).  Suppose there
exist constants $\gl,\gL$ such that
\begin{align*}
\gl g_0 \leq g_t \leq \gL g_0, \qquad \brs{\del \ga}^2 \leq \gL.
\end{align*}
Given $k \in \mathbb N$ there exists a constant $C =
C(n,k,g_0,\hat{g},\Omega,\mu,\gl,\gL)$ such that
\begin{align*}
\sup_{M \times \{t\}} t f_k(g_t,h) \leq C. 
\end{align*}
\end{thm}

As discussed above, this theorem is proved using a certain generalized metric on $T^{1,0} \oplus T^*_{1,0}$.  In particular, given $(M^{2n}, J)$ a complex manifold, fix $g$  a pluriclosed metric and $\gb \in \Lambda^{2,0}$.  Using these we define an associated generalized metric (cf. \S \ref{s:GCSKT})
\begin{align} \label{f:genHerm}
G = \left( 
\begin{matrix}
g_{i \bj} + \gb_{i k} \bgb_{\bj \bl} g^{\bl k} & \gb_{ip} g^{\bl p}\\
 \bgb_{\bj \bp} g^{\bp k}  & g^{\bl k}
\end{matrix}
\right).
\end{align}
This is a Hermitian metric on $T^{1,0} \oplus T^*_{1,0}$, and thus comes equipped with a Chern connection, which we will denote $\N^C_G$.  Our estimates are phrased naturally in terms of these connections, so we make a definition analogous to Definition \ref{d:EKfdef1}:

\begin{defn} \label{d:EKfdef} Fix a complex manifold $(M^{2n}, J)$, let $g$ denote a Hermitian metric, and let $G$, $\til{G}$ denote Hermitian metrics on $T^{1,0} \oplus T^*_{1,0}$.  Let
\begin{align*}
\gU(G,\til{G}) := \N^C_G - \N^C_{\til{G}} \in \gG \left( T^{1,0} \otimes \End(T^{1,0} \oplus T^*_{1,0})\right)
\end{align*}
denote the difference of the Chern connections associated to $G$ and $\til{G}$.  This is 
a natural tensor which measures the first derivatives of $G$ with respect to a 
background choice $\til{G}$. Furthermore, let
\begin{align*}
f_k(G,\til{G}) := \sum_{j=0}^k \brs{\N_{g,G}^{j} \gU(G,\til{G})}^{\frac{2}{1+j}}_{g,G}.
\end{align*}
The connection $\N_{g,G}$ denotes the Chern connection on $T^{1,0} \otimes \End(T^{1,0} \oplus T^*_{1,0})$ induced naturally by the Chern connection of $g$ on $T^{1,0}$ and the Chern connnection associated to $G$ on $\End(T^{1,0} \oplus T^*_{1,0})$.  The quantity $f_k$ is a natural measure of the $(k+1)$-st derivatives of the generalized
metric
which scales as the inverse of the metric.
\end{defn}

\begin{thm} \label{t:EKthm1} Let $(M^{2n}, J)$ be a compact complex manifold, and suppose $(\gw_t, \gb_t)$ is a solution of pluriclosed flow (\ref{f:augPCF}) on $[0,\tau)$, $\tau \leq
1$.  Let $G_t$ denote the one-parameter family of associated generalized Hermitian metrics, fix $\til{G}$ another generalized Hermitian metric, and suppose
\begin{align*}
\gL^{-1} \til{G} \leq G \leq \gL \til{G}.
\end{align*}
Given $k \in \mathbb N$ there exists $C = C(n,k,\gL,\til{G})$ such that
\begin{align*}
\sup_{M \times \{t\}} t f_k(G, \til{G}) \leq&\ C.
\end{align*}
\end{thm}

For the proof, which comprises the remainder of this section, we will assume that our solution to pluriclosed flow satisfies
\begin{equation*}
\del \gw_0 + \delb \gb_0 = 0,
\end{equation*}
which implies that this equation holds at all times as discussed above.  This is not necessary for the proof but simplifies the exposition (cf. Remark \ref{r:torsionremark}).  The proof Theorem \ref{t:EKthm1} relies principally on a very clean evolution equation for $G$ along pluriclosed flow.  To state this we note that given $G$ as in (\ref{f:genHerm}), we have an associated curvature tensor $\Omega^G \in \Lambda^{1,1} \otimes \End(T^{1,0} \oplus T^*_{1,0})$.  Akin to Definition \ref{d:Riccidef}, we define the tensor
\begin{align*}
S^G = \tr_{\gw} \Omega^G \in \End (T^{1,0} \otimes T^*_{1,0}).
\end{align*}
Amazingly, while the evolution equations for $g$ and $\gb$ along pluriclosed flow are determined by the curvature of the Bismut connection, the evolution equation for the metric $G$ is determined by the curvature of its associated Chern connection in a very simple way:

\begin{prop} \label{p:bigGev} Given $(M^{2n}, J)$ a complex manifold, suppose $(\gw_t, \gb_t)$ is a solution of pluriclosed flow (\ref{f:augPCF}).  Let $G_t$ denote the one-parameter family of associated generalized metrics as in (\ref{f:genHerm}).  Then
\begin{align*}
G^{-1} \dt G =&\ - S^G.
\end{align*}
\begin{proof} 
We leave this lengthy but straightforward computation as an \textbf{exercise}, see \cite{JordanStreets} for details. The principal fact is that, by the proof of Lemma \ref{l:SKTequivalence}, $e^{2\sqrt{-1}\beta}$ can be regarded as defining an homomorphism of holomorphic Courant algebroids
$$
e^{2\sqrt{-1}\beta} \colon \mathcal{E}_{-2\sqrt{-1}\partial\omega} \to T^{1,0} \otimes T^*_{1,0},
$$ 
where the right hand-side is equipped with the trivial structure. Thus, for the calculation of the Chern connection of $G$ we can assume $\gb = 0$, provided that we work with the twisted holomorphic structure in Definition \ref{def:Q0}. A lengthy calculation then shows that
\begin{align*}
S^G G =&\ \left(
\begin{matrix}
\rho_B^{1,1} & \rho_B^{2,0} g^{-1}\\
g^{-1} \rho_B^{0,2} & - g^{-1} \rho_B^{1,1} g^{-1}
\end{matrix} \right).
\end{align*}
\end{proof}
\end{prop}

\begin{rmk} \label{r:BRFasHYM} The tensor $S^G$ as defined above is the same as that appearing in the Hermitian-Yang-Mills equations \cite{KobayashiH}.  In that setting one chooses a background Hermitian metric $\gw$ on a given complex manifold, then solves for a Hermitian metric $h$ on an associated holomorphic vector bundle $E$ such that $S_{\gw,h} = \tr_{\gw} \Omega^h = \gl \Id_E$, where $\Omega^h$ is the curvature of the canonically associated Chern connection.  In our setting both metrics $\gw$ and $h = G$ are determined by the same underlying pluriclosed metric on the base manifold $M$.  Nonetheless, in this way the Bismut Hermitian-Einstein condition can be interpreted as $S^G = 0$.  The existence of Hermitian-Yang-Mills metrics is associated to the conditions of slope stability \cite{DonaldsonHYM, UYau}, and one wonders what role these obstructions can play in understanding the existence and uniqueness of Bismut Hermitian-Einstein metrics.
\end{rmk}

\begin{lemma} \label{l:Htrevol} Given $(M^{2n}, J)$ and $(\gw_t, \gb_t)$ a solution to pluriclosed flow (\ref{f:augPCF}) with associated generalized metric $G_t$ and background generalized metric $\til{G}$, one has
	\begin{align*}
	\left(\dt-\gD^C \right)\tr_G \widetilde{G}=&\ -|\gU|_{g^{-1},G^{-1},\widetilde{G}}^2+g^{\bn m} G^{\bB A}\widetilde{\Omega}_{m \bn A \bB}.
	\end{align*}
\end{lemma}
\begin{proof}
	Using Proposition \ref{p:bigGev} we compute
\begin{align*}
\dt \tr_G \widetilde{G} = \dt (G^{\bB A}\til{G}_{A\bB}) = G^{\bC A}G^{\bB D}S_{D\bC}\til{G}_{A\bB}.
\end{align*}
	Furthermore,
\begin{align*}
		\gD\tr_{G}\til{G} =&\ g^{\bn m} \N_m\N_{\bn} (G^{\bB A}\til{G}_{A\bB})\\
	 =&\ -g^{\bn m} G^{\bB A} \N_m(\til{G}_{A\bC}\bgU_{\bn \bB}^{\bC})\\
	  =&\ g^{\bn m} G^{\bB A} \left(\til{G}_{D\bC}\gU_{m A}^{D}\bgU_{\bn \bB}^{\bC}+\til{G}_{A\bC}(\Omega_{\bn m \bB}^{\bC} - \widetilde{\Omega}_{\bn m \bB}^{\bC}) \right)\\
=&\ |\gU|_{g^{-1},G^{-1},\til{G}}^2 + G^{\bC A}G^{\bB D}S_{D\bC}\til{G}_{A\bB} -g^{\bn m} G^{\bB A} \widetilde{\Omega}_{m \bn A \bB}.
\end{align*}
Combining the two equations above gives the result.
\end{proof}

\begin{prop} \label{p:gUevol} Given $(M^{2n}, J)$ and $(\gw_t, \gb_t)$ a solution to pluriclosed flow (\ref{f:augPCF}), with associated generalized metric $G_t$ and background generalized metric $\til{G}$, one has
	\begin{align*}
\left(\dt - \gD^C \right)|\gU|^2 =-|\N\gU|^2-|\bar{\N}\gU+T\cdot\gU|^2 + T\star\gU \star \til{\Omega}+ \gU \star \bgU \star \til{\Omega} + \gU \star \til{\N}\til{\Omega}.
	\end{align*}
	\begin{proof} We work in local complex coordinates, and adopt the convention that lower-case letters will be used for indices ranging over $T^{1,0}$ and capital letters will be used to represent indices ranging over all $T^{1,0}\oplus T^*_{1,0}$.  A general calculation for the variation of the Chern connection associated to a Hermitian metric yields
		\begin{align*}
		\dt \gU_{iA}^B = \N_i \dt G_{A}^{B}.
		\end{align*}
		Using Proposition \ref{p:bigGev} we obtain
		\begin{equation} \label{f:C3calc20}
		\dt \gU_{i A}^{B} = - \N_i S_{A}^{B}.
		\end{equation}
		Using this, we compute
		\begin{gather} \label{f:C3calc25}
		\begin{split}
		\gD^C \gU_{iA}^B =&\ g^{\bk l}\N_{l}\N_{\bk} \gU_{iA}^B\\
		=&\ g^{\bk l}\N_l \left(\Omega_{\bk i A}^B-\widetilde{\Omega}_{\bk i A}^B\right)\\
		=&\ g^{\bk l}\N_i\Omega_{\bk l A}^B+g^{\bk l}T^p_{il}\Omega_{\bk p A}^B +g^{\bk l}\N_l\widetilde{\Omega}_{i \bk A}^B\\
		=&\ \partial_t \gU_{iA}^B-g^{\bk l}T^p_{il}\Omega_{p \bk A}^B +g^{\bk l}\widetilde{\N}_l\widetilde{\Omega}_{i \bk A}^B \\
		&\ \qquad +g^{\bk l}\left(\gU_{l i}^q \widetilde{\Omega}_{q\bk A}^B +\gU_{l A}^D \widetilde{\Omega}_{i\bk D}^B - \gU_{l D}^B \widetilde{\Omega}_{i\bk A}^D \right).\\
		\end{split}
		\end{gather}
		Next we observe the commutation formula
		\begin{gather} \label{f:C3calc30}
		\begin{split}
		\bar{\gD}^C \bar{\gU}_{\bj \bA}^{\bB} =&\ g^{\bl k} \N_{\bl} \N_k \bgU_{\bj \bA}^{\bB}\\
		=&\ g^{\bl k} \left[ \N_k \N_{\bl} \bgU_{\bj \bA}^{\bB} - (\Omega^g)_{k \bl \bj}^{\bq} \bar{\gU}_{\bq \bA}^{\bB} - \Omega_{k \bl \bA}^{\bC} \bar{\gU}_{\bj \bC}^{\bB} + \Omega_{k \bl \bC}^{\bB} \bar{\gU}_{\bj \bA}^{\bC} \right]\\
		=&\ \gD^C \bar{\gU}_{\bj \bA}^{\bB} - (S^g)_{\bj}^{\bq} \bgU_{\bq \bA}^{\bB} - S_{\bA}^{\bC} \bar{\gU}_{\bj \bC}^{\bB} +S_{\bC}^{\bB} \bar{\gU}_{\bj \bA}^{\bC}.
		\end{split}
		\end{gather}
		Combining (\ref{f:C3calc25}), (\ref{f:C3calc30}), the evolution equations of Proposition \ref{p:bigGev} and the pluriclosed flow equation  we obtain
		\begin{align*}
		\dt \brs{\gU}^2_{g^{-1},G^{-1},G} =&\ \dt \left[ g^{\bj i} G^{\bC A} G_{B \bD} \gU_{i A}^{B} \bgU_{\bj \bC}^{\bD} \right]\\
		=&\ - g^{\bj k} \left( - S^g_{k \bl} + Q_{k \bl} \right) g^{\bl i} G^{\bC A} G_{B \bD} \gU_{i A}^{B} \bgU_{\bj \bC}^{\bD} - g^{\bj i} G^{\bC M} \left( - S^G_{M \bL} \right) G^{\bL A} \gU_{i A}^{B} \bgU_{\bj \bC}^{\bD}\\
		&\ + g^{\bj i} G^{\bC A} \left( - S^G_{B \bD} \right) \gU_{i A}^{B} \bgU_{\bj \bC}^{\bD} + g^{\bj i} G^{\bC A} G_{B \bD} \left(\gD^C \gU_{iA}^B + g^{\bk l}(T^g)_{il}^p\Omega_{p\bk A}^B\right.\\
		&\ \left. -g^{\bk l}\widetilde{\N}_l\widetilde{\Omega}_{i\bk A}^B-g^{\bk l}\left[\gU_{l i}^q \widetilde{\Omega}_{q\bk A}^B +\gU_{l A}^M \widetilde{\Omega}_{i\bk M}^B - \gU_{l M}^B \widetilde{\Omega}_{i\bk A}^M \right]\right) \bgU_{\bj \bC}^{\bD}\\
		&\ + g^{\bj i} G^{\bC A} G_{B \bD} \gU_{i A}^{B} \left(\bar{\gD}^C \bgU_{\bj \bC}^{\bD} - g^{\bk l}(\bar{T}^g)_{\bk\bj}^{\bq}\Omega_{l\bq \bC}^{\bD}+g^{\bk l}\widetilde{\N}_{\bk}\widetilde{\Omega}_{l \bj \bC}^{\bD}\right.\\
		&\ \left.-g^{\bk l}\left[\bgU_{\bk \bj}^{\bq} \widetilde{\Omega}_{\bq l \bC}^{\bD} +\bgU_{\bk \bC}^{\bL} \widetilde{\Omega}_{\bj l \bL}^{\bD} - \bgU_{\bk \bL}^{\bD} \widetilde{\Omega}_{\bj l \bC}^{\bL} \right] \right). \\
		\end{align*}
		We furthermore compute
		\begin{align*}
		\gD^C |\gU|^2_{g^{-1},G^{-1},G}&= g^{\bk l}g^{\bj i}G^{\bC A}G_{B \bD}\N_{l}\N_{\bk} (\gU_{i A}^{B} \bgU_{\bj \bC}^{\bD})\\
		&= \langle\gD^C \gU,\bar{\gU} \rangle + \langle \gU, \gD^C \bar{\gU}\rangle + |\N \gU|^2+ |\bar{\N}\gU|^2 \\
		&=(S^g)_{\bj}^{\bq}\bar{\gU}_{\bq \bC}^{\bD}\gU_{iA}^{B} g^{\bj i} G^{\bC A} G_{B \bD} +S_{\bC}^{\bar{\Lambda}}\bar{\gU}_{\bj \bar{\Lambda}}^{\bD}\gU_{iA}^B g^{\bj i}G^{\bC A}G_{B\bD}\\
		&\quad - S_{\bar{\Lambda}}^{\bD}\bar{\gU}_{\bj \bC}^{\bar{\Lambda}}\gU_{iA}^B g^{\bj i}G^{\bC A}G_{B \bD} + \langle\gD^C \gU,\bar{\gU} \rangle + \langle \gU, \bar{\gD}^C \bar{\gU} \rangle+ |\N \gU|^2+ |\bar{\N}\gU|^2.
		\end{align*}
		Combining the two equations above yields
		\begin{align*}
		\left(\dt - \gD^C \right)|\gU|^2 =&\ -|\N\gU|^2-|\bar{\N}\gU|^2\\
		&\ + g^{\bj i}g^{\bk l}G^{\bC A}G_{B\bD}\left( -Q_{i\bk}\gU_{l A}^{B}\bgU_{\bj \bC}^{\bD} +T_{il}^p \bar{\N} \gU_{p\bk A}^B\bgU_{\bj \bC}^{\bD}-\bar{T}_{\bk\bj}^{\bq} \bar{\N} \gU_{l\bq\bC}^{\bD}\gU_{iA}^{B} \right)\\
		&\ + T\star \gU \star \til{\Omega}+ \gU \star \bgU \star \til{\Omega} + \gU \star \til{\N}\til{\Omega}.
		\end{align*}
		Then, observing that the second through fifth terms above form a perfect square we arrive at
		\begin{align*}
		\left(\dt - \gD^C \right)|\gU|^2 =-|\N\gU|^2-|\bar{\N}\gU+T\cdot\gU|^2 + T\star \gU \star \til{\Omega}+ \gU \star \bgU \star \til{\Omega} + \gU \star \til{\N}\til{\Omega},
		\end{align*}
		as claimed.
	\end{proof}
\end{prop}

\begin{proof}[Proof of Theorem \ref{t:EKthm1}] To begin we prove the case $k=1$.  Fix a constant $A > 0$ and let
\begin{align*}
\Phi = t \brs{\gU}^2 + A \tr_G \til{G}.
\end{align*}
Combining Proposition \ref{p:gUevol} and Lemma \ref{l:Htrevol} we obtain
\begin{align*}
\left(\dt - \gD^C \right) \Phi =&\ t \left(-|\N\gU|^2-|\bar{\N}\gU+T\cdot\gU|^2 + T*\gU*\til{\Omega}+ \gU *\bgU * \til{\Omega} + \gU * \til{\N}\til{\Omega} \right)\\
&\ + \brs{\gU}^2 + A \left( -|\gU|_{g^{-1},G^{-1},\til{G}}^2+g^{\bn m} G^{\bB A}\til{\Omega}_{m \bn A \bB} \right)\\
\leq&\ \brs{\gU}^2 \left(1 + t C - A \right) + C A\\
\leq&\ C A,
\end{align*}
where the last line follows using that $t \leq 1$ and choosing $A = C + 1$.  Applying the maximum principle we obtain
\begin{align*}
\sup_{M \times \{t\}} \Phi \leq C(1 + t) \leq C.
\end{align*}
Since $\Phi \geq \brs{\gU}^2$, the case $k=1$ follows.

Having established the case $k=1$, there are two ways to finish the general proof.  Beginning with (\ref{f:C3calc20}) one can derive natural heat equations for all derivatives $\N^j \Upsilon$, and derive smoothing estimates for these quantities in a manner similar to Theorem \ref{t:smoothing}.  Alternatively one can use the scale invariance of the quantity $tf_k$ and apply a rescaling argument as carried out in  \cite{JordanStreets}.  We leave the details to the reader.
\end{proof}

\section{Metric evolution equations}

Theorem \ref{t:EKthm2} shows that solutions to pluriclosed flow can be extended smoothly as long as the equation remains uniformly parabolic.  To measure this parabolicity requires obtaining uniform estimates on the metric tensor along the flow.
In this section we derive reaction-diffusion equations for various quantities which measure the metric evolving by pluriclosed flow against a background metric.  We will apply the maximum principle to these evolution equations in the next section to establish global existence results for the pluriclosed flow.

\begin{lemma} \label{l:volumeformev} Let $(M^{2n}, g_t, J)$ be a solution to
pluriclosed flow, and let $h$ denote another Hermitian metric on $(M, J)$.  Then
\begin{align*}
 \left( \frac{\del}{\del t} - \gD^C \right) \log \frac{\det g}{\det h} =&\
\brs{T}^2 - \tr_g \rho_C(h).
\end{align*}
\begin{proof} We compute in local complex coordinates.  Here and below we will use the expression for the tensor $S_C$ in terms of local complex coordinates, which follows from a short computation,
\begin{align} \label{pcfcoords}
(S_C)_{i \bj} =&\ - g^{\bq p} g_{i \bj,p\bq} + g^{\bq p}
g^{\bs r} g_{i \bs,p} g_{r \bj,\bq}.
\end{align}
Using this and equation (\ref{f:HMPCF} for pluriclosed flow we obtain
\begin{align*}
 \dt \log \frac{\det g}{\det h} =&\ g^{\bj i} \left( \dt g \right)_{i \bj}\\
 =&\ g^{\bj i} \left[ g^{\bq p} g_{i \bj,p\bq} - g^{\bq p}
g^{\bs r} g_{i \bs,p} g_{r \bj,\bq} + Q_{i \bj} \right]\\
 =&\ g^{\bj i} \left[ g^{\bq p} g_{i \bj,p\bq} - g^{\bq p}
g^{\bs r} g_{i \bs,p} g_{r \bj,\bq} \right] + \brs{T}^2.
\end{align*}
Also
\begin{align*}
 \gD^C \log \frac{\det g}{\det h} =&\ g^{\bq p} \left[ \log \frac{\det g}{\det h}
\right]_{,p\bq}\\
 =&\ g^{\bq p} \left[ g^{\bj i} g_{i \bj,p} - h^{\bj i} h_{i \bj,p}
\right]_{,\bq}\\
 =&\ g^{\bq p} \left[ g^{\bj i} g_{i\bj,p\bq} - g^{\bj k} g_{k \bl,\bq} g^{\bl
i} g_{i \bj,p} - h^{\bj i} h_{i \bj,p\bq} + h^{\bj k} h_{k \bl,\bq} h^{\bl i}
h_{i \bj,p} \right].
\end{align*}
Combining the above calculations and comparing against the transgression formula of Proposition \ref{p:Cherntrans} yields
\begin{align*}
 \left(\dt - \gD^C \right) \log \frac{\det g}{\det h} =&\ \brs{T}^2 - \tr_g
\rho_C(h),
\end{align*}
as required.
\end{proof}
\end{lemma}

\begin{lemma} \label{l:invtraceev} Let $(M^{2n}, g_t, J)$ be a solution to
pluriclosed flow, and let $h$ denote another Hermitian metric on $(M, J)$.  Then
\begin{align*}
 \left( \frac{\del}{\del t} - \gD^C \right) \tr_{g} h =&\ -
\brs{\gU(g,h)}^2_{g^{-1},g^{-1},h} - \IP{h,Q}_g + \Omega_h(g^{-1},g^{-1}).
\end{align*}
\begin{proof} We directly compute
\begin{align*}
 \dt \tr_g h =&\ \dt g^{\bj i} h_{i \bj}\\
 =&\ - g^{\bj k} \left( \dt g_{k \bl} \right) g^{\bl i} h_{i \bj}\\
 =&\ - g^{\bj k} g^{\bl i} h_{i \bj} \left[ g^{\bq p} g_{k \bl,p\bq} - g^{\bq p}
g^{\bs r} g_{k \bs,p} g_{r \bl,\bq} + Q_{k \bl} \right].
\end{align*}
On the other hand
\begin{align*}
 \gD^C \tr_g h =&\ g^{\bq p} \left[ g^{\bj i} h_{i \bj} \right]_{,p\bq}\\
 =&\ g^{\bq p} \left[ - g^{\bj k} g_{k \bl,p} g^{\bl i} h_{i \bj} + g^{\bj i}
h_{i\bj,p} \right]_{,\bq}\\
 =&\ g^{\bq p} \left[ g^{\bj r} g_{r \bs,\bq} g^{\bs k} g_{k \bl,p} g^{\bl i}
h_{i \bj} - g^{\bj k} g_{k \bl,p\bq} g^{\bl i} h_{i \bj} + g^{\bj k} g_{k \bl,p}
g^{\bl r} g_{r\bs,\bq} g^{\bs i} h_{i \bj} \right.\\
&\ \qquad \left. - g^{\bj k} g_{k \bl,p} g^{\bl i} h_{i \bj,\bq} - g^{\bj k}
g_{k \bl,\bq} g^{\bl i} h_{i\bj,p} + g^{\bj i} h_{i \bj,p\bq} \right].
 \end{align*}
Combining the above calculations yields
\begin{align*}
 \left( \dt - \gD^C \right) \tr_g h =&\ - g^{\bq p} g^{\bj r} g^{\bs k} g^{\bl
i}h_{i \bj} g_{r \bs,\bq} g_{k \bl,p}\\
&\ + g^{\bq p} \left[g^{\bj k} g_{k \bl,p} g^{\bl i} h_{i \bj,\bq} + g^{\bj k}
g_{k \bl,\bq} g^{\bl i} h_{i\bj,p} - g^{\bj i} h_{i \bj,p\bq} \right]
- \IP{h,Q}_g\\
=&\ - \brs{\gU(g,h)}^2_{g^{-1},g^{-1},h} - \IP{h,Q}_g + \Omega_h(g^{-1},g^{-1}).
\end{align*}
\end{proof}
\end{lemma}

\begin{lemma} \label{l:loginvtraceev} Let $(M^{2n}, g_t, J)$ be a solution to
pluriclosed flow, and let $h$ denote another Hermitian metric on $(M, J)$.  Then
\begin{align*}
 \left( \frac{\del}{\del t} - \gD^C \right) \log \tr_{g} h \leq&\ - \frac{\IP{h,Q}_g}{\tr_g h} + K \tr_g h,
\end{align*}
where $K = \sup K(h)$, i.e. the supremum of all holomorphic bisectional curvatures.
\begin{proof} We first observe the general fact that
\begin{align*}
\left(\dt - \gD^C \right) \log f =&\ f^{-1} \left(\dt - \gD^C \right) f + \frac{\brs{\N f}^2}{f^2}.
\end{align*}
Using this together with Lemma \ref{l:invtraceev} yields
\begin{gather*}
\begin{split}
\left(\dt - \gD^C \right)& \log \tr_g h\\
=&\ \frac{1}{\tr_g h} \left(-
\brs{\gU(g,h)}^2_{g^{-1},g^{-1},h} - \IP{h,Q}_g + \Omega_h(g^{-1},g^{-1}) \right)\\
&\ + \frac{\brs{\N \tr_g h}^2}{(\tr_g h)^2}\\
=&\ T_1 + T_2 + T_3 + T_4,
\end{split}
\end{gather*}
where we have labeled the four terms appearing on the right hand side.  To estimate the term $T_4$, we first pick special coordinates at some given point $p \in M$.  In particular, we can choose complex coordinates for which $h_{i \bj} = \gd_{i}^j$ and $g$ is diagonalized.  We furthermore observe then that, at the point $p$,
\begin{align*}
\N_i \tr_g h =&\ \del_i \left( g^{\bl k} h_{k \bl} \right)\\
=&\ - g^{\bl p} \del_i g_{p \bq} g^{\bq k} h_{k \bl} + g^{\bl k} \del_i h_{k \bl}\\
=&\ - \gU(g,h)_{ip}^k (g^{-1} h)_k^p.
\end{align*}
Using this and the Cauchy-Schwarz inequality we can estimate
\begin{align*}
T_4 =&\ \frac{\brs{\N \tr_g h}^2}{\tr_g h}\\
=&\ \left( \sum_i g^{i \bi} \right)^{-1}
g^{\bj j} \N_j \tr_g h \N_{\bj} \tr_g h\\
 =&\ \left( \sum_i g^{i \bi} \right)^{-1}  \sum_j g^{\bj j} \left[ \sum_{k} g^{\bk k} \gU_{j k}^k \sum_{l} g^{\bl l} \gU_{\bj \bl}^{\bl} \right]\\
 =&\ \left( \sum_i g^{i \bi} \right)^{-1}  \sum_j \left[ \sum_{k} \left[ (g^{\bj
j})^{\frac{1}{2}} (g^{\bk k})^{\frac{1}{2}} \gU_{jk}^k \right] (g^{\bk
k})^{\frac{1}{2}} \sum_{l} \left[ (g^{\bj j})^{\frac{1}{2}} (g^{l \bl})^{\frac{1}{2}} \gU_{\bj
\bl}^{\bl} \right] (g^{l \bl})^{\frac{1}{2}} \right]\\
 \leq&\ \left( \sum_i g^{i \bi} \right)^{-1} \left( \sum_{j,k} g^{\bj j} g^{\bk
k} \gU_{j k}^k \gU_{\bj \bk}^{\bk} \right)^{\frac{1}{2}} \left( \sum_k g^{\bk
k} \right)^{\frac{1}{2}} \left( \sum_{j,l} g^{\bj j} g^{\bl l} \gU_{j l}^l
\gU_{\bj \bl}^{\bl} \right)^{\frac{1}{2}} \left( \sum_l g^{\bl l}
\right)^{\frac{1}{2}}\\
 \leq&\ \brs{\gU(g,h)}^2_{g^{-1},h^{-1},g} = T_1.
 \end{align*}
We also estimate, using the definition of $A$,
\begin{align*}
T_3 =&\ \frac{1}{\tr_g h} \Omega_h(g^{-1}, g^{-1})\\
=&\ \frac{1}{\tr_g h} (\Omega_h)_{i\bj k \bl} g^{\bj i} g^{\bl k}\\
\leq&\ K \frac{1}{\tr_g h} h_{i \bj} h_{k \bl} g^{\bj i} g^{\bl k}\\
=&\ K \tr_g h.
\end{align*}
Combining the above estimates, the lemma follows.
\end{proof}
\end{lemma}

\begin{lemma} \label{l:traceev} Let $(M^{2n}, g_t, J)$ be a
solution to
pluriclosed flow.  Then
\begin{align*}
\left(\dt - \gD^C \right) \tr_h g =&\ - \brs{\gU(g,h)}^2_{g^{-1},h^{-1},g} + \tr_h
Q - \Omega_h(h^{-1} g h^{-1}, g^{-1}).
\end{align*}
\begin{proof} Using (\ref{pcfcoords}) we obtain
\begin{align*}
 \dt \tr_h g =&\ \dt h^{\bj i} g_{i \bj}\\
 =&\ h^{\bj i} \left[ g^{\bq p} g_{i
\bj,p\bq} - g^{\bq p}
g^{\bs r} g_{i \bs,p} g_{r \bj,\bq} + Q_{i \bj} \right]\\
=&\ h^{\bj i} \left[ g^{\bq p} g_{i
\bj,p\bq} - g^{\bq p}
g^{\bs r} g_{i \bs,p} g_{r \bj,\bq} \right] + \tr_h Q.
\end{align*}
On the other hand
\begin{align*}
 \gD^C \tr_h g =&\ g^{\bq p} \left[ h^{\bj i} g_{i \bj} \right]_{,p\bq}\\
 =&\ g^{\bq p} \left[ - h^{\bj k} h_{k \bl,p} h^{\bl i} g_{i \bj} + h^{\bj i}
g_{i\bj,p} \right]_{,\bq}\\
 =&\ g^{\bq p} \left[ h^{\bj r} h_{r \bs,\bq} h^{\bs k} h_{k \bl,p} h^{\bl i}
g_{i \bj} - h^{\bj k} h_{k \bl,p\bq} h^{\bl i} g_{i \bj} + h^{\bj k} h_{k \bl,p}
h^{\bl r} h_{r\bs,\bq} h^{\bs i} g_{i \bj} \right.\\
&\ \qquad \left. - h^{\bj k} h_{k \bl,p} h^{\bl i} g_{i \bj,\bq} - h^{\bj k}
h_{k \bl,\bq} h^{\bl i} g_{i\bj,p} + h^{\bj i} g_{i \bj,p\bq} \right].
 \end{align*}
Combining the above calculations yields
\begin{align*}
 \left( \dt - \gD^C \right) \tr_h g =&\ - h^{\bj i} g^{\bq p}
g^{\bs r} g_{i \bs,p} g_{r \bj,\bq} + \tr_h Q\\
&\ - g^{\bq p} \left[ h^{\bj r} h_{r \bs,\bq} h^{\bs k} h_{k \bl,p} h^{\bl i}
g_{i \bj} - h^{\bj k} h_{k \bl,p\bq} h^{\bl i} g_{i \bj} + h^{\bj k} h_{k \bl,p}
h^{\bl r} h_{r\bs,\bq} h^{\bs i} g_{i \bj} \right.\\
&\ \qquad \left. - h^{\bj k} h_{k \bl,p} h^{\bl i} g_{i \bj,\bq} - h^{\bj k}
h_{k \bl,\bq} h^{\bl i} g_{i\bj,p} \right]\\
=&\ - \brs{\gU(g,h)}^2_{g^{-1},h^{-1},g} + \tr_h Q - g^{\bq p}
(\Omega^h)_{p \bq}^{\bl k} g_{k \bl},
\end{align*}
as required.
\end{proof}
\end{lemma}

\begin{lemma} \label{l:logtraceev} Let $(M^{2n}, g_t, J)$ be a solution to
pluriclosed flow, and let $h$ denote another Hermitian metric on $(M, J)$.  Then
\begin{align*}
 \left( \frac{\del}{\del t} - \gD^C \right) \log \tr_{h} g \leq&\  \frac{\tr_h Q}{\tr_h g} + K \tr_g h,
\end{align*}
where $K = - \inf K(h)$, i.e. the negative of the infimum of all holomorphic bisectional curvatures.
\begin{proof} The proof follows that of Lemma \ref{l:loginvtraceev} closely, building off of Lemma \ref{l:traceev}.
\end{proof}
\end{lemma}

\section{Sharp existence and convergence results} \label{s:sharpexistence}

In this section we combine the results of the previous sections to establish global existence and convergence results for the K\"ahler-Ricci flow, pluriclosed flow, and generalized K\"ahler-Ricci flow, and discuss their consequences.

\subsection{A general regularity result}

We prove a technical tool underpinning all of the global existence results to follow for the pluriclosed flow.  In particular, we show that for times $\tau < \tau^*(g_0)$, a lower bound on the metric suffices to obtain full $C^{\infty}$ regularity of the flow.  The proof builds on the evolution equations of the previous sections as well as the sharp higher order regularity estimate of Theorem \ref{t:EKthm2}.

\begin{prop} \label{p:lbregular} Let $(M^{2n}, g_0, J)$ be a compact complex manifold with pluriclosed metric.  Fix $\tau < \tau^*(g_0)$.  Suppose the solution to pluriclosed flow with initial condition $g_0$ exists on $[0,\tau)$ and furthermore satisfies
\begin{align*}
g_t \geq \gd g_0
\end{align*}
for some constant $\gd$, at all points $(p,t) \in M \times [0,\tau)$.  Then $g_t$ has uniform $C^{\infty}$ estimates on $[0,\tau)$, and the solution to pluriclosed flow extends smoothly past $\tau$.
\begin{proof} We choose $\hat{g}_t, \Omega$, $\mu$, and a solution $\ga_t$ to (\ref{f:PCFreduction}) as in Proposition \ref{p:PCFoneform}.  For simplicity we assume that these background objects satisfy the hypotheses of Proposition \ref{p:specialtorspot}, i.e $\mu = 0$ and $\del \hat{\gw}_t = \del \hat{\gw}_0 = \delb \eta$, and we set $\phi = \del \ga - \eta$ .  The general case requires estimating a few more background terms but is effectively the same (cf. \cite{StreetsPCFBI} Theorem 1.8).  We estimate, using Lemma \ref{l:volumeformev}, Proposition \ref{p:specialtorspot} and the assumed lower bound for $g_t$, 
\begin{align*}
\left(\dt - \gD^C \right) \log \frac{\det g}{\det h} + \brs{\phi}^2 \leq&\ \tr_g \rho_C(h) - \brs{\N \phi}^2  - 2 \IP{Q, \phi \otimes \phi} \leq C \gd^{-1}.
\end{align*}
By applying the maximum principle we obtain a time-dependent upper bound for the metric as well as for $\brs{\phi}^2$ and hence $\brs{\del \ga}^2$.  Applying Theorem \ref{t:EKthm2} we obtain uniform $C^{\infty}$ bounds on $[0,\tau)$ and the result follows.
\end{proof}
\end{prop}

\subsection{K\"ahler-Ricci flow}

\begin{thm} \label{t:TianZhang}  (\cite{TianZhang}) Let $(M^{2n}, \omega_0, J)$ 
be a compact K\"ahler
manifold.  Then the solution to K\"ahler-Ricci flow with initial condition $\omega_0$
exists smoothly on $[0, \tau^*(g_0))$.
\begin{proof} Fix any time $\tau < \tau^*$, and choose $\Omega$, a family of background metrics $\gw_t$, and $u_t$ a solution of the reduced flow (\ref{f:KRFPMA}), i.e.
\begin{align} \label{t:TZ10}
\dt u = \log \frac{\left(\gw_t + \i \del \delb 
u\right)^n}{\Omega}, \qquad u(0) = 0,
\end{align}
  as guaranteed by Proposition \ref{p:KRFscalar}.  It suffices to show that the one-parameter family of metrics $\gw_t^u = \gw_t + \i \del\delb u$ remains smooth on the time interval $[0,\tau]$, as the uniqueness of solutions to K\"ahler-Ricci flow then guarantees that the maximal solution must exist on $[0,\tau^*(g_0))$.

First, by a direct application of the maximum principle to (\ref{t:TZ10}) we obtain a uniform upper bound on $\dt u$ at maximum points of $u$, and thus a uniform upper bound for $u$ on $[0,\tau]$.  Arguing similarly we obtain a lower bound as well.  Next we obtain a uniform upper and lower bound for the volume form.  Let
\begin{align*}
\Phi_1 := \log \frac{(\gw_t^u)^n}{\Omega} - A u,
\end{align*}
where $A$ is a constant to be determined.  Combining Lemma \ref{l:volumeformev} and (\ref{t:TZ10}) we compute, bounding the background curvature $\rho_C(\Omega)$, using the a priori estimate for $u$, and choosing $A$ sufficiently large,
\begin{align*}
\left(\dt - \gD^C \right) \Phi =&\ - \tr_{\gw_t^u} \rho_C(\Omega) - A \left( \dt u - \tr_{\gw_t^u} \left(\gw_t^u - \gw_t \right) \right)\\
\leq&\ B \tr_{\gw_t^u} \gw_t - A \left( \Phi_1 + A u - n + \tr_{\gw_t^u} \gw_t \right)\\
\leq&\ \left(B - A \right) \tr_{\gw_t^u} \gw_t - A \Phi_1 + C A^2 + An\\
\leq&\ - A \Phi_1 + C A^2 + An.
\end{align*}
A direct application of the maximum principle yields an a priori upper bound for $\Phi_1$, and hence for $\log \frac{(\gw_t^u)^n}{\Omega}$.  A directly analogous argument yields the lower bound.  We now obtain the lower bound for the metric.  Fix a background metric $h$ and let
\begin{align*}
\Phi_2 =&\ \log \tr_{\gw_t^u} \gw_h - A u.
\end{align*}
Applying Lemma \ref{l:invtraceev},  (\ref{t:TZ10}), and the a priori estimate for $\dt u$,
\begin{align*}
\left(\dt - \gD^C \right) \Phi_2 \leq&\ K \tr_{\gw_t^u} {\gw_h} - A \left(\dt u - \tr_{\gw_t^u} \left(\gw_t^u - \gw_t \right) \right)\\
\leq&\ K \tr_{\gw_t^u} \gw_h + A  C - A \gd \tr_{\gw_t^u} \gw_h\\
\leq&\ \left(K - A \gd \right) \tr_{\gw_t^u} \gw_h + A C\\
\leq&\ - \frac{A \gd}{2} \tr_{\gw_t^u} \gw_h + A,
\end{align*}
where the constant $\gd$ is such that $\gw_t \geq \gd \gw_h$, and $A$ is chosen sufficiently large to obtain the final inequality.  By the maximum principle, at a sufficiently large maximum for $\Phi_2$ we obtain an a priori upper bound for  $\tr_{\gw^u_t} \gw_h$.  It follows that $\Phi_2$ has an a priori upper bound and thus $\tr_{\gw_t^u} \gw_h$ is bounded above, giving an a priori lower bound for the metric $\gw_t^u$.  The smooth existence now follows from Proposition \ref{p:lbregular}, finishing the proof.
\end{proof}
\end{thm}

Note that Theorem \ref{t:TianZhang} is confirmation of Conjecture \ref{c:mainflowconj} in the case the initial metric is K\"ahler, and was the inspiration for making Conjecture \ref{c:mainflowconj} in the first place.

\subsection{Pluriclosed flow}

\begin{thm} \label{t:PCFonNCB} Let $(M^{2n}, J)$ be a compact complex manifold admitting a Hermitian metric $h$ with nonpositive holomorphic bisectional curvature.  Given $g_0$ a pluriclosed metric on $M$, the solution to pluriclosed flow with initial condition $g_0$ exists on $[0,\infty)$.
\begin{proof} We first note that the formal existence time of the flow is infinite.  In particular, using that $h$ has nonpositive holomorphic bisectional curvature, it follows that
\begin{align*}
\rho_C(h)_{i \bj} =&\ h^{\bl k} \Omega_{i \bj k \bl}
\end{align*}
is also nonpositive.  Thus for any $\tau \geq 0$ we have $\gw_0 - \tau \rho_C(h) \geq 0$, and hence $[\gw_0] - t c_1 \in \mathcal P$, thus the formal existence time is infinite, as claimed.  

Using that $h$ has nonpositive holomorphic bisectional curvature, and that $Q \geq 0$, we observe using Lemma \ref{l:invtraceev} that
\begin{align*}
\left(\dt - \gD^C \right) \tr_g h =&\ - \brs{\gU(g,h)}^2_{g^{-1},g^{-1},h} - \IP{h,Q}_g + \Omega_h(g^{-1},g^{-1})\\
\leq&\ 0.
\end{align*}
From the maximum principle we conclude that for all times $t$ such that the flow is defined, we have
\begin{align*}
\sup_{M \times \{t\}} \tr_g h \leq \sup_{M \times \{0\}} \tr_g h.
\end{align*}
Having established a uniform lower bound for the metric for all times $t$, the result follows from Proposition \ref{p:lbregular}
\end{proof}
\end{thm}

In general, the convergence behavior at infinity in the setting of Theorem \ref{t:PCFonNCB} can be fairly complicated.  In the case $(M^{2n}, J)$ admits a K\"ahler-Einstein metric of negative scalar curvature, we should expect the normalized flow to converge to this K\"ahler-Einstein metric.  Similar results are shown in the case the background metric has nonpositive curvature operator in \cite{LeeStreets}.  The flow in these settings can also exhibit collapsing behavior at infinity.  For instance, if $\Sigma$ denotes a Riemann surface of genus $g \geq 2$, consider the product manifold $T^2 \times \Sigma$ with the product complex structure.  The product of a flat metric on $T^2$ with a hyperbolic metric on $\Sigma$ has nonpositive holomorphic bisectional curvature.  The normalized pluriclosed flow (actually K\"ahler-Ricci flow) with this initial data will remain a product metric, keeping the metric on $\Sigma$ fixed while homothetically shrinking the metric on $T^2$ at an exponential rate.  In some cases the hypothesis of nonpositive curvature can be weakened to a topological condition, and global existence and convergence results can be obtained (cf. \cite{StreetsGG}).

In the case $(M^{2n}, J)$ is a torus, there is a Hermitian metric with vanishing Chern curvature, and the solution to (unnormalized) pluriclosed flow will converge exponentially to a flat metric, as we prove in Theorem \ref{t:PCFontori} below.  
First we need a rigidity result for steady solitons generalizing Proposition \ref{p:Kahlerrigidity}.

\begin{prop} \label{p:Kahlerrigiditysteadysoliton} Let $(M^{2n}, g, J)$ be a compact complex manifold with 
pluriclosed metric $g$ which is a steady generalized Ricci soliton.  If $c_1(M, J) \leq 0$ in Bott-Chern cohomology, then $g$ is 
K\"ahler Calabi-Yau.
\begin{proof} As $c_1(M, J) \leq 0$ in Bott-Chern cohomology, we may choose a Hermitian metric $h$ such that $\rho_C(h) \leq 0$.  As $g$ is a steady soliton, the pluriclosed flow will evolve by diffeomorphism pullback by a vector field $X$.  Using Lemma \ref{l:volumeformev} we then obtain
\begin{align*}
X \cdot \log \frac{\det g}{\det h} = \dt \log \frac{\det g}{\det h} = \gD^C \log \frac{\det g}{\det h} + \brs{T}^2 - \rho_C(h) \geq \gD^C \log \frac{\det g}{\det h} + \brs{T}^2.
\end{align*}
Applying the strong maximum principle (Proposition \ref{p:strongmax}), it follows that $\log \frac{\det g}{\det h}$ is constant, and hence $\brs{T}^2 = 0$, and so $g$ is K\"ahler.  Compact steady Ricci solitons are Einstein by Proposition \ref{p:compactRiccisolitons}, hence $g$ is K\"ahler Calabi-Yau, as claimed.
\end{proof}
\end{prop}

\begin{thm} \label{t:PCFontori} Given $\gw_0$ a pluriclosed metric on a complex torus $(T^{2n}, J)$, the solution to pluriclosed flow with initial condition $\gw_0$ exists on $[0,\infty)$, and converges exponentially to a flat K\"ahler metric.
\begin{proof} The global existence of the flow, as well as an a priori lower bound for the metric, follows from Theorem \ref{t:PCFonNCB}.  We claim that there exists a uniform a priori upper bound for the metric for all times $t$.  To see this we can choose $h$ to be the given flat metric, set $\hat{\gw}_t = \hat{\gw}_0$, $\mu = 0$, and obtain a solution $\ga_t$ to the reduced pluriclosed flow equation (\ref{f:PCFreduction}) using Proposition \ref{p:PCFoneform}.  Furthermore, since $T^{2n}$ is K\"ahler, the $\i \del \delb$-lemma holds, and thus following Remark \ref{r:torsionremark}, we can verify the hypotheses of Proposition \ref{p:specialtorspot}, and obtain $\phi = \del \ga - \eta$ satisfying (\ref{f:specialtorsev}).  Setting $\Phi = \log \frac{\det g}{\det h} + \brs{\phi}^2$, combining Lemma \ref{l:volumeformev} with (\ref{f:specialtorsev}) yields
\begin{align*}
\left(\dt - \gD^C \right) \Phi \leq&\ - \brs{\N \phi}^2 - 2 \IP{Q, \phi \otimes \bar{\phi}} \leq 0.
\end{align*}
Thus by the maximum principle we conclude
\begin{align*}
\sup_{M \times \{t\}} \Phi \leq \sup_{M \times \{0\}} \Phi.
\end{align*}
It follows that the determinant of $g$ is bounded above, and thus since there is an a priori lower bound, it follows that there is an a priori upper bound for the metric.  It follows that we can apply Theorem \ref{t:EKthm2} to obtain uniform $C^{\infty}$ estimates for the metric at all times as well.

Given these estimates, for any sequence $\{t_j\} \to \infty$ there exists a subsequence, also denoted $\{t_j\}$, such that $\{g_{t_j}\}$ converges in $C^{\infty}$ to a smooth limiting metric $g_{\infty}$.  Corollary \ref{c:genllmonotonicity} implies that $\gl(g_t, H_t)$ is monotonically increasing and moreover bounded above due to the $C^{\infty}$ control of $g$, thus there exists $\bar{\gl}$ such that $\lim_{t \to \infty} \gl(g_t, H_t) = \bar{\gl}$.  Choose a sequence $\{\til{t}_j\}$ such that $\bar{\gl} - \gl(g_{\til{t}_j}, H_{\til{t}_j}) \leq j^{-1}$.  Let $f_{\til{t}_j + 1}$ denote the unique minimizer of $\FF(g_{\til{t}_j + 1}, H_{\til{t}_j + 1}, \cdot)$, and let $u = e^{-f_t}$ denote the unique solution to the conjugate heat equation on $[\til{t}_j, \til{t}_j + 1]$ with $u_{\til{t}_j + 1} = e^{-f_{\til{t}_j + 1}}$.  It follows from Proposition \ref{p:genFFmonotonicity} that
\begin{align*}
j^{-1} \geq&\ \int_{\til{t}_j}^{\til{t}_j + 1} \frac{d}{dt} \FF(g_t, H_t, f_t) dt\\
=&\ \int_{\til{t}_j}^{\til{t}_j + 1} \int_M \left[ \brs{\Rc - \tfrac{1}{4} H^2 
+ \N^2 f}^2 + \tfrac{1}{4} \brs{d^* H + \N f \lrcorner H}^2 \right] e^{-f} dV dt.
\end{align*}
It follows that there exists a sequence of times $t_j \in [\til{t}_j, \til{t}_j + 1]$ such that
\begin{align} \label{t:PCFontori10}
\lim_{j \to \infty} \int_M \left[ \brs{\Rc - \tfrac{1}{4} H^2 
+ \N^2 f}_{g_{t_j}}^2 + \tfrac{1}{4} \brs{d^* H + \N f \lrcorner H}^2_{g_{t_j}} \right] e^{-f_{t_j}} dV_{g_{t_j}} = 0.
\end{align}
Applying the argument above, the sequence $\{g_{t_j}\}$ admits a further subsequence which converges in $C^{\infty}$ to a limiting metric $g_{\infty}$.  Furthermore, the subsequence can be chosen so that the functions $\{f_{t_j}\}$, which admit $C^{\infty}$ estimates by elliptic regularity, also converge to a limit $f_{\infty}$.  By (\ref{t:PCFontori10}), the limit $(g_{\infty}, H_{\infty}, f_{\infty})$ is a steady generalized Ricci soliton.  Since $c_1(T^{2n}, J) = 0$ in Bott-Chern cohomology, it follows from Proposition \ref{p:Kahlerrigiditysteadysoliton} that the limit $g_{\infty}$ is actually K\"ahler-Einstein, with $f_{\infty}$ constant.  As all K\"ahler-Einstein metrics on $T^{2n}$ are flat, we have established subsequential convergence to a flat K\"ahler metric.  To finish the proof of exponential convergence of the entire flow we appeal to a general stability result for pluriclosed flow \cite[Theorem 1.2]{HCF} near K\"ahler, Ricci-flat metrics, which in this setting implies that the pluriclosed flow starting sufficiently close to a flat metric on a torus will converge exponentially to a flat metric, as required.
\end{proof}
\end{thm}

\subsection{Generalized K\"ahler-Ricci flow: Commuting case}

We first prove a sharp local existence theorem confirming Conjecture \ref{c:CGKconeconj} for the generalized K\"ahler-Ricci flow in the commuting case with some topological hypotheses on the bundle splitting in place.  As described in \cite{StreetsSTB}, this result covers almost all cases of generalized K\"ahler-Ricci flow in the commuting case on complex surfaces, including general type, properly elliptic, bi-elliptic, some ruled surfaces, and Inoue surfaces.  The statement below was shown for the case $n=2$ in \cite{StreetsSTB}, using various techniques special to that case.

\begin{thm} \label{t:commutinglowrank} Let $(M^{2n}, g_0, I, J)$ be a compact generalized K\"ahler manifold satisfying $[I,J] = 0$.  Suppose
\begin{enumerate}
\item $\rank T_-^{1,0} = 1$,
\item $c_1(T_-^{1,0}) \leq 0$.
\end{enumerate}
Then the solution to generalized K\"ahler-Ricci flow with initial condition $g_0$ exists on $[0, \tau^*(g_0))$.
\begin{proof} Choose a time $\tau < \tau^*(g_0)$, and choose background partial volume forms $\Omega_{\pm} = \det (\gw_h)_{\pm}$, a family of background metrics $\hat{\gw}_t$, and $u_t$ a solution of the reduced equation (\ref{f:GKRFreduction}) by Proposition \ref{p:GKRFcommuting}.  Note that, by applying the maximum principle to (\ref{f:GKRFreduction}), it follows that there exists a uniform constant $C$ such that
\begin{align*}
\inf_{M \times \{t\}} u \geq \inf_{M \times \{0\}} u - t C.
\end{align*}
In particular, on the compact time interval $[0,\tau]$ there is a uniform lower bound for $u$.

Next we obtain a lower bound for the metric.  The first step is to estimate the metric on $T_-$.  A general computation for the generalized K\"ahler-Ricci flow in this setting, closely related to Lemma \ref{l:volumeformev}, shows that the partial volume forms satisfy
\begin{align} \label{f:clrk10}
\left(\dt - \gD^C \right) \log \frac{\det g_{\pm}}{\det h_{\pm}} = \tfrac{1}{2} \brs{T}^2 - \tr_g \rho_C(h_{\pm}),
\end{align}
where $h_{\pm}$ are background Hermitian metrics on $T_{\pm}$.  As $c_1(T_-^{1,0}) \leq 0$, we can choose $h_-$ such that $\rho_C(h_-) \leq 0$.  For such $h_{-}$ it follows from the maximum principle applied to (\ref{f:clrk10}) that
\begin{align} \label{f:clrk15}
\inf_{M \times \{t\}} \frac{\det g_-}{\det h_-} \geq \inf_{M \times \{0\}} \frac{\det g_-}{\det h_-}.
\end{align}
Since the rank of $T_-^{1,0}$ is $1$, we can rephrase this estimate as
\begin{align*}
\sup_{M \times \{t\}} \tr_{g_-}  h_- \leq \sup_{M \times \{0\}} \tr_{g_-} h_-.
\end{align*}
Next we obtain a lower bound for $\frac{\del u}{\del t}$.  Using (\ref{f:clrk10}) we compute
\begin{gather} \label{f:clrk20}
\begin{split}
\left(\dt - \gD^C \right) \frac{\del u}{\del t} =&\ \left(\dt - \gD^C \right) \left[ \log \frac{\det g_{+}}{\det h_{+}} - \log \frac{\det h_{-}}{\det g_{-}} \right]\\
=&\ \tr_g \left( \rho_C(h_-) - \rho_C(h_+) \right)\\
=&\ - \tr_{g_+} P(h) + \tr_{g_-} P(h).
\end{split}
\end{gather}
Now let $\Phi_1 = (\tau - t) \frac{\del u}{\del t} + u$.  Before computing an evolution equation for $\Phi$ we observe the basic fact that
\begin{gather} \label{f:clrk25}
\begin{split}
\gD^C u =&\ \tr_g \i \del \delb u\\
=&\ \tr_{\gw_+} \i \del_+ \delb_+ u + \tr_{\gw_-} \i \del_- \delb_- u\\
=&\ \tr_{\gw_+} \left( \gw_t - \hat{\gw}_t \right)_+ + \tr_{\gw_-} \left( \hat{\gw}_t - \gw_t^u \right)_-\\
=&\ k - 1 - \tr_{g_+} \hat{g}_+ + \tr_{g_-} \hat{g}_-,
\end{split}
\end{gather}
where $k$ is the rank of $T_+$.  Using this and (\ref{f:clrk20}) we obtain
\begin{gather}
\begin{split}
\left(\dt - \gD^C \right) \Phi_1 =&\ \left(\tau - t \right) \left(\dt - \gD^C \right) \frac{\del u}{\del t} - \gD^C u\\
=&\ \tr_{g_+} \left(  \hat{g}_+ - \left(\tau - t \right) P(h) \right) - \tr_{g_-} \left( \hat{g}_- - \left(\tau - t \right) P(h) \right)\\
&\ \qquad + k-1\\
=&\ \tr_{g_+} (\hat{g}_{\tau})_+ - \tr_{g_-} (\hat{g}_{\tau})_- + k - 1\\
\geq&\ -C,
\end{split}
\end{gather}
where the final inequality uses the uniform lower bound on $g_-$ of (\ref{f:clrk15}).  It follows from the maximum principle that there exists a uniform constant $C$ such that
\begin{align} \label{f:clrk30}
\inf_{M \times \{t\}} \Phi_1 \geq&\ \inf_{M \times \{0\}} \Phi_1 - t C.
\end{align}
Together with the uniform estimate for $u$, this yields a uniform lower bound for $\frac{\del u}{\del t}$ on $[0,\tau - \ge]$ for any $\ge > 0$.  

Finally we are ready to establish the lower bound for $g$.  Fix a large constant $A$ to be determined, and let $\Phi_2 = \log \tr_g h - A u$.  Using Lemma \ref{l:invtraceev} and \ref{f:clrk25}, we estimate
\begin{align*}
\left(\dt - \gD^C \right) \Phi_2 \leq&\ K \tr_g h - A \dt u + A \left(k - 1 - \tr_{g_+} \hat{g}_+ + \tr_{g_-} \hat{g}_- \right)\\
\leq&\ \left(C - A \right) \tr_{g_+} \hat{g}_+ + C A\\
\leq&\ C A,
\end{align*}
where the second line uses the equivalence of $\hat{g}$ and $h$, the lower bound for $\frac{\del u}{\del t}$, and the lower bound for $g_-$, and the final line follows by choosing $A$ sufficiently large.  It follows from the maximum principle that
\begin{align*}
\sup_{M \times \{t\}} \tr_g h \leq&\ \sup_{M \times \{0\}} \tr_g h + C t.
\end{align*}
Thus we have established a uniform lower bound for the metric on all time intervals of the form $[0,\tau - \ge]$.  Proposition \ref{p:lbregular} implies that the solution has uniform estimates and hence exists on such intervals.  Letting $\tau \to \tau^*(g_0)$ and $\ge \to 0$, we obtain that the solution exists on $[0,\tau^*(g_0))$, as required.
\end{proof}
\end{thm}

The proof of Theorem \ref{t:commutinglowrank} makes clear a key difficulty in establishing a full confirmation of Conjecture \ref{c:CGKconeconj}, namely the fact that the resulting scalar PDE of (\ref{f:GKRFreduction}) is not \emph{convex}, which makes directly applying the classical Pogorelov method impossible.  This is the reason for the topological hypotheses of Theorem \ref{t:commutinglowrank}, which give some a priori partial control over the metric, after which the Pogorelov method can be applied.  We note that a global viscosity solution in this setting was established in \cite{Streetsvisc}.  We next establish convergence on tori.

\begin{cor} \label{c:commutingontorus} Given $(T^{2n}, g_0, I, J)$ a generalized K\"ahler structure on the torus satisfying $[I,J] = 0$,  the solution to generalized K\"ahler-Ricci flow with initial condition $g_0$ exists on $[0,\infty)$, and converges exponentially to a flat metric.  In particular, if $g_{E}$ denotes a flat metric compatible with $I$ and $J$, given any function $u_0$ such that $\gw_E + \i \left( \del_+ \delb_+ - \del_- \delb_- \right) u > 0$, the solution to
\begin{gather} \label{f:GKRFreductiontori}
\begin{split}
  \dt u =&\ \log \frac{\left( \gw^E_+ + \i \del_+ \delb_+ u \right)^k \wedge (\gw^E_-)^l}{(\gw^E_+)^k \wedge \left( \gw^E_- - \i \del_- \delb_- u \right)^l},\\
 u(0) =&\ u_0,
\end{split}
\end{gather}
exists on $[0,\infty)$ and converges to $u_{\infty} = const$.
\begin{proof} Since the solution to generalized K\"ahler-Ricci flow in the $I$-fixed gauge is the same as the solution to pluriclosed flow (cf. \S \ref{ss:GKRFbiherm}), the claims of long time existence and convergence to a flat metric follow directly from Theorem \ref{t:PCFontori}.  Using that $I_{|T_{\pm}} = \mp J_{|T_{\pm}}$, it is easy to see that there exist flat metrics compatible with both $I$ and $J$.  Choosing one such, $g_E$, and using this as the choice of background metric in Proposition \ref{p:GKRFcommuting}, we see that equation (\ref{f:GKRFreductiontori}) is the resulting scalar reduction of generalized K\"ahler-Ricci flow, and so the claims of global existence and convergence follow.
\end{proof}
\end{cor}

\subsection{Generalized K\"ahler-Ricci flow: Nondegenerate case}

\begin{cor} \label{c:nondegeneratetorus} Given $(T^{2n}, g_0, I, J)$ a generalized K\"ahler structure on the torus such that $[I,J]$ is nondegenerate, the solution to generalized K\"ahler-Ricci flow with initial condition $g_0$ in the $I$-fixed gauge exists on $[0,\infty)$, and converges exponentially to a triple $(g_{\infty}, I, J_{\infty})$ such that $g_{\infty}$ is flat and $(I,J_{\infty})$ are part of a hyperK\"ahler sphere of complex structures compatible with $g_{\infty}$.
\begin{proof} Since the solution to generalized K\"ahler-Ricci flow in the $I$-fixed gauge is the same as the solution to pluriclosed flow (cf. \S \ref{ss:GKRFbiherm}), the claims of long time existence and convergence to a flat K\"ahler metric follow directly from Theorem \ref{t:PCFontori}.  By hypothesis, the initial Poisson tensor $\gs = \tfrac{1}{2} g^{-1} [I,J]$ is nondegenerate, and we denote its inverse by $\Omega$.  By Proposition \ref{p:Poissonfixed} we know that $\gs$, and hence $\Omega$, remain fixed along the flow, and so we can express $\Omega = 2 g_{\infty} [I,J_{\infty}]^{-1}$.  This form is of type $(2,0)-(0,2)$ with respect to both $I$ and $J_{\infty}$, with the respective $(2,0)$ parts each holomorphic with respect to the corresponding complex structure by Proposition \ref{p:holoPoisson}.  By the Beauville-Bogomolov-Yau decomposition \cite{Beauville, Bogomolov, YauCC}, the data $(g_{\infty}, I, J_{\infty})$ are part of a hyperK\"ahler structure.
\end{proof}
\end{cor}

This corollary shows the power of generalized K\"ahler-Ricci flow to prove classification results for problems in complex and symplectic geometry.  On the one hand, we have shown contractibility of the space of nondegenerate generalized K\"ahler structures on tori to the space of structures $(g, I, J)$ with $g$ flat, and $I$ and $J$ part of a compatible hyperK\"ahler structure.  Using the interpretation of generalized K\"ahler-Ricci flow in this setting in terms of positive $\Omega$-Hamiltonian diffeomorphisms in Proposition \ref{p:Hamiltonian-flow}, we have determined the global topology of this nonlinear space as well.  In principle one expects generalized K\"ahler-Ricci flow to be able to yield further results on the structure of natural spaces arising in symplectic geometry.

Despite the long time existence and convergence results described above, the global existence and convergence properties of pluriclosed flow and generalized K\"ahler-Ricci flow are still largely open.  The main existence Conjecture \ref{c:mainflowconj} has only been established in special cases, the key missing point being to establish uniform parabolicity of the equation in the positive cone.  At the conjectural singular time $\tau^*(g_0)$ we expect a wide variety of behavior including convergence, blowdown of complex subvarities, collapse to a lower dimensional space, etc.  In the setting of complex surfaces, a conjectural framework for the convergence properties of pluriclosed flow was described in \cite{PCF, PCFReg, SG}.  In principle the limiting structure of pluriclosed flow can be used to detect the complex structure of the underlying complex manifold, resulting in a kind of geometrization conjecture for complex surfaces.

\chapter{T-duality}\label{c:Tdual}

In this chapter we present a method to relate solutions of the generalized Ricci flow equation based on a natural symmetry of generalized geometry known as T-duality. Quite remarkably, these pairs of T-dual solutions of the flow may live in topologically distinct manifolds and typically have very different geometric features. The idea has its origins in the string theory literature \cite{Buscher,KIKKAWA1984357,Rocek1992}.  The main result is the preservation of the generalized Ricci flow under T-duality, and builds on a fundamental result by Cavalcanti and Gualtieri \cite{GualtieriTdual}, which states that \emph{topological T-duality} for principal torus bundles--as defined by Bouwknegt, Evslin and Mathai \cite{BEM}--induces a Courant algebroid isomorphism.  One consequence of our discussion will be a mathematical explanation of the \emph{dilaton shift} in string theory, as originally observed by Buscher \cite{Buscher}.  Furthermore, we illustrate examples of global existence and convergence of generalized Ricci flows which are T-dual to solutions of classical Ricci flow.

\section{Topological T-duality}\label{s:TTdual}

The goal of this section is to introduce some basic aspects of \emph{topological T-duality} for principal torus bundles, as defined in \cite{BEM}.  First, let us briefly summarize the definition following \cite{GualtieriTdual}. Let $T^k$ be a $k$-dimensional torus acting freely and properly on a smooth manifold $M$, so that $M$ is a principal $T^k$-bundle over the smooth manifold $B := M/T^k$. We denote by $\mathfrak{t}$ the Lie algebra of $T^k$. We endow $M$ with a choice of $T^k$-invariant cohomology class
$$
\tau \in H^3(M,\mathbb{R})^{T^k}.
$$
If $B$ happens to be non-compact, we can replace $ H^3(M,\mathbb{R})$ by the cohomology with compact support on $M$. Consider the fiber product $\overline M = M \times_B \hat M$ and the diagram
\begin{equation*}
  \xymatrix{
 & \ar[ld]_{q} \overline M \ar[rd]^{\hat{q}} & \\
 M \ar[rd]_{p} &  & \hat{M} \ar[ld]^{\hat{p}} \\
  & B & \\
  }
\end{equation*}

\begin{defn}\label{def:toptdual}
We say that two pairs $(M,\tau)$ and $(\hat M,\hat \tau)$ as above are \emph{T-dual} if there exist invariant representatives $H \in \tau$ and $\hat H \in \hat \tau$ such that
\begin{equation}\label{eq:relationTdual}
\hat{q}^* \hat H - q^*H  = d \overline{B},
\end{equation}
where $\overline{B} \in \Gamma(\Lambda^2 T^*\overline M)$ is a $T^k \times \hat T^k$-invariant $2$-form such that
\begin{equation*}
\overline B \colon \Ker d q \otimes \Ker d \hat{q} \to \mathbb{R}
\end{equation*}
is non-degenerate. 
\end{defn}

\begin{rmk}\label{rem:unimodular}
In the literature, one typically considers $\bar{B}$ to be \emph{unimodular}, in the sense that it takes integer values for invariant period-$1$ elements of $\Gamma(\Ker d q)$ and $\Gamma(\Ker d \hat{q})$. This condition is useful for interpreting the T-dual bundles as fibrations of mutually dual tori, but we will not use it explicitly in our discussion.
\end{rmk}

We present next the most basic example of topological T-duality, given by the case where the base of our torus bundles is a point.

\begin{ex}\label{ex:torus}
Consider the case where the base is a point $B = \{\bullet\}$ and $k = 1$, so that $M \cong S^1$. We regard $M$ as the quotient
$$
M = \mathbb{R}/\langle \alpha \rangle,
$$
where $\langle \alpha \rangle$ denotes the additive subgroup generated by $0 \neq \alpha \in \mathbb{R}$. Consider the dual torus $\hat M = \mathbb{R}^*/\langle \alpha \rangle^*$, where 
$$
\langle \alpha \rangle^{*}=\{u: \mathbb{R} \rightarrow \mathbb{R} \hspace{1mm} | \hspace{1mm} u(\alpha)\in \mathbb{Z}\}=\{\alpha^{-1}u \hspace{1mm} | \hspace{1mm} u\in \mathbb{Z}^{*}\}.
$$
Identifying $\mathbb{R} \cong \mathbb{R}^*$ we regard 
$$
\hat M = \mathbb{R}/\langle \alpha^{-1} \rangle.
$$
For a choice of coordinate $x$ on $\mathbb{R}$, we can take $ \overline{B} = dx \wedge d \hat x$ in \eqref{eq:relationTdual}, which realizes $M$ and $\hat M$ (endowed with the trivial de Rham class) as topological T-dual circles. Observe that, by construction, $\overline{B}$ is unimodular in the sense of Remark \ref{rem:unimodular}.
\end{ex}

To obtain more interesting examples, which illustrate the topological nature of Definition \ref{def:toptdual}, we derive further structure of T-dual pairs. Given such a pair, we can choose special representatives of the given de Rham classes which still satisfy \eqref{eq:relationTdual}. To see this, consider the natural exact sequence
\begin{equation}\label{eq:2.1.3}
    0 \rightarrow B \times \mathfrak{t} \rightarrow \frac{TM}{T^{k}} \rightarrow TB \rightarrow 0.
\end{equation}
This induces a filtration 
\begin{equation}\label{eq:filtration}
    \Gamma(\Lambda^\bullet T^* B)\cong \mathcal{F}^{0}\subset \mathcal{F}^{1} \subset \dots \subset \mathcal{F}^{\bullet}= \Gamma(\Lambda^\bullet T^*M)^{T^{k}}
\end{equation}
where $\mathcal{F}^{i}=\mathrm{Ann}(\wedge^{i+1} \mathfrak{t})$. The choice of a connection $\theta$ on $p \colon M \to B$ induces an isomorphism $TM/T^k \cong TB \oplus \mathfrak{t}$ and gives a splitting for the exact sequence \eqref{eq:2.1.3}, inducing isomorphisms
\begin{equation}\label{eq:difformssplit}
    \Gamma(\Lambda^m T^*M)^{T^{k}} \cong \bigoplus_{j=0}^{m}\Gamma(\Lambda^j T^*B \otimes \wedge^{m-j}\mathfrak{t}^{*})^{T^{k}}.
\end{equation}

\begin{rmk}\label{rem:notationtheta}
Throughout this section, $\theta$ will denote a connection on a principal torus bundle.  Notice that in previous chapters we reserved this notation for the Lee form of a Hermitian structure, but as we will not use the Lee form in this chapter this should not lead to any confusion.
\end{rmk}

\begin{lemma}\label{lem:Hmodel}
Let $(M,\tau)$ and $(\hat M,\hat \tau)$ be T-dual. Then, there exist representatives $H \in \tau$ and $\hat H \in \hat \tau$ in $\mathcal{F}^{1}$ of $M$ and $\hat M$, respectively, which satisfy the conditions in Definition \ref{def:toptdual}.
\end{lemma}

\begin{proof}
Using the filtrations \eqref{eq:filtration} applied to $q$ and $\hat q$, we can write the two-form $\overline{B}$ in Definition \ref{def:toptdual} as
$$
\overline{B} = q^*B + \hat q^* \hat B + \overline{B}_0,
$$
where $B \in \Gamma(\Lambda^2 T^*M)^{T^k}$, $\hat B \in \Gamma(\Lambda^2T^*\hat M)^{\hat T^k}$ and $\overline{B}_0 \in \Gamma(\Lambda^2T^*\overline M)^{T^k \times \hat T^k}$ satisfies
$$
\overline B_{|\Ker d q \otimes \Ker d \hat{q}} =  {\overline{B}_0}_{|\Ker d q \otimes \Ker d \hat{q}}
$$
and furthermore $\overline{B}_0$ is in $\mathcal{F}^1$ for both fibrations $q$ and $\hat q$. Hence, we have that $\overline{B}_0$ is non-degenerate on $\Ker d q \otimes \Ker d \hat{q}$ and
$$
d \overline{B}_0 = \hat q^*(\hat H + d \hat B) - q^*(H - d B).
$$
Using the last equation and the fact that $\overline{B}_0$ is in $\mathcal{F}^1$ for the fibration $q$, we leave as an \textbf{exercise} to see that $H - d B$ is in $\mathcal{F}^1$ for the fibration $p$, and similarly for $\hat H + d \hat B$.
\end{proof}

The condition $H \in \mathcal{F}^1$ in the previous result means that, for a choice of connection $\theta$ on $M$, we can write (see \eqref{eq:difformssplit})
\begin{equation}\label{eq:HF1}
H = p^*h - \langle \hat c \wedge \theta \rangle
\end{equation}
for suitable $\hat c \in \Gamma(\Lambda^2 T^*B \otimes \mathfrak{t}^*)$ satisfying $d \hat c =0$ and $h \in \Gamma(\Lambda^3T^*B)$. Here, $ \langle  , \rangle$ denotes the natural duality pairing between $\mathfrak{t}$ and $\mathfrak{t}^*$. With this observation in hand, we can present a result about the existence of T-duals due to Bouwknegt, Hannabuss, and Mathai \cite{Bouwknegt_2004}.

\begin{prop}\label{prop:BHM}
Given a principal $T^k$-bundle $M$ over a base manifold $B$ and a closed three-form $H \in \mathcal{F}^1 \cap\Gamma(\Lambda^3T^*M)^{T^k}$ representing an integral cohomology class $[H] \in H^3(M,\mathbb{Z})$, there exists a T-dual for $(M,[H])$.
\end{prop}

\begin{proof}
Choose a connection $\theta$ on $p \colon M \to B$ so that $H$ is given by \eqref{eq:HF1}. The condition $[H] \in H^3(M,\mathbb{Z})$ implies that the closed $\mathfrak{t}^*$-valued two-form $\hat c$ represents an integral cohomology class in $H^2(B,\mathfrak{t}^*)$. Consider the dual torus $\hat T^k$ with Lie algebra $\mathfrak{t}^*$ and let $\hat M$ be a principal torus bundle over $B$ with Chern class $[\hat c]$. Let $\hat \theta$ be a connection on $\hat p \colon \hat M \to B$ such that $d \hat \theta = \hat c$, and define
$$
\hat H = \hat{p}^*h - \langle c \wedge \hat \theta \rangle
$$
where $c = d\theta$. We leave as an \textbf{exercise} to check that $d \hat H = 0$ and that $(M,[H])$ and $(\hat M,[\hat H])$ are T-dual, taking $\overline{B} = - \langle \theta \wedge \hat \theta \rangle$ in \eqref{eq:relationTdual}.
\end{proof}

We are ready to present two basic examples of topological T-duality, for a positive dimensional base (cf. Example \ref{ex:torus}). The first illustrates the topology change on the manifold $M$ under T-duality.

\begin{ex} \label{e:TdualHopf1}
Consider the Hopf fibration $p\colon M = S^3 \to B = S^2$ with $H = 0$. Choose a connection $\theta$ such that $d \theta = \omega_{S^2}$ is the standard symplectic structure on $S^2$ (suitably normalized). We claim that $(M,0)$ is T-dual to $(S^2 \times S^1,[\hat H])$, where $\hat H = - \hat \theta \wedge \omega_{S^2}$, where $\hat \theta$ is the line element on $S^1$. To see this, take $ \overline{B} = - q^*\theta \wedge \hat q^*\hat \theta $ and calculate
$$
d \overline{B} = - q^*d\theta \wedge \hat q^*\hat \theta = - q^*p^* \omega_{S^2} \wedge \hat q^*\hat \theta = \hat q^* \hat H. 
$$
\end{ex}

The next example shows that T-duality is sensitive to the choice of cohomology class of the three-form $H$, as is implicit in the proof of Proposition \ref{prop:BHM}.

\begin{ex} \label{e:TdualHopf2}
Consider again the Hopf fibration in the previous example, but regarding $M$ as the compact Lie group $M = SU(2)$. Consider generators for the Lie algebra $\mathfrak{su}(2) = \langle e_1,e_2,e_3 \rangle$ as in Example \ref{ex:GRicciflatHopf} and the Cartan three-form
$$
H = - e^{123}. 
$$
We claim that $(M,[H])$ is self T-dual. Notice that $e_1$ generates the circle action on the fibers of $p \colon M \to S^2$ and hence $e^1$ is a connection on $M$ with $de^1 = e^{23} = p^*\omega_{S^2}$. Consider
$$
\overline{B} = - q^* e^1 \wedge \hat q^* \hat e^1
$$
where $\hat e^1$ is a left-invariant one-form dual to the generator of the circle in $\hat M = SU(2)$. Then,
$$
d \overline{B} = - q^*p^*\omega_{S^2}\wedge \hat q^* \hat e^1 + q^* e^1 \wedge \hat q^* \hat p^*\omega_{S^2} = \hat p^* \hat H - p^*H,
$$
and hence $\overline{B}$ satisfies the conditions in \ref{def:toptdual}. More generally, following the previous argument one can prove that any semi-simple Lie group equipped with the cohomology class of its Cartan 3-form is self T-dual (see \cite[Example 2.4]{GualtieriTdual}).
\end{ex}

\section{T-duality and Courant algebroids}

Having described the most basic features about T-duality in the previous section, we now recall the relevant implications of the T-duality relation for the theory of generalized geometry \cite{GualtieriTdual}. For this, let us first interpret the data in Definition \ref{def:toptdual} in the language of Courant algebroids.

\begin{defn}
An \emph{equivariant Courant algebroid} is given by a triple $(E,M,T^k)$, where $T^k$ is a $k$-dimensional torus acting freely and properly on a smooth manifold $M$, so that $M$ is a principal $T^k$-bundle over a base $M/T^k = B$, and $E$ is an exact Courant algebroid over $M$ with a $T^k$-action, lifting the $T^k$-action on $M$ and preserving the Courant algebroid structure.
\end{defn}

Effectively, given an equivariant Courant algebroid $(E,M,T^k)$, since $T^k$ is compact we  can choose an equivariant isotropic splitting $\sigma \colon T \to E$ and hence the previous definition reduces to having the twisted Courant algebroid $(T \oplus T^*,\IP{,},[,]_H)$ over $M$ with a $T^k$-action on $T \oplus T^*$ and a $T^k$-invariant three-form $H \in \Gamma(\Lambda^3T^*M)^{T^k}$. Thus, any triple $(E,M,T^k)$ determines a $T^k$-invariant \v Severa class
$$
\tau = [E] \in H^3(M,\mathbb{R})^{T^k},
$$ 
leading to a pair $(M,\tau)$ as in Definition \ref{def:toptdual}. Conversely, given a pair $(M,\tau)$ as in Definition \ref{def:toptdual} we can consider the equivariant Courant algebroid $(T \oplus T^*,\IP{,},[,]_H)$ over $M$, endowed with the canonical $T^k$-action on $T \oplus T^*$, for a choice of $T^k$-invariant element $H \in \tau$. 

Our goal here is to explain a theorem due to Cavalcanti and Gualtieri \cite{GualtieriTdual}, which states that for any T-dual pairs one can associate an isomorphism of equivariant Courant algebroids, upon reduction to the base $B$. The preliminaries for this result are contained in \S \ref{s:GCGmanifold}, where the Dorfman bracket on $E$ is expressed as a \emph{derived bracket} using the twisted de Rham differential \eqref{eq:twisteddiff}. In the situation of our interest, assume that $(T \oplus T^*,\IP{,},[,]_H)$ is equivariant over a $T^k$-bundle $p \colon  M \to B$. Consider the twisted de Rham complex of $T^k$-invariant differential forms $(\Gamma(\mathbb{S})^{T^k},d_H)$ given by the $\mathbb{Z}_2$-graded vector space 
$$
\Gamma(\mathbb{S})^{T^k} = \Gamma(\Lambda^{\mbox{\tiny{even}}}T^*M)^{T^k} \oplus \Gamma(\Lambda^{\mbox{\tiny{odd}}}T^*M)^{T^k}
$$
with differential $d_H$ in \eqref{eq:twisteddiff}. The following result is due to Bouwknegt, Evslin, and Mathai \cite{BEM}.

\begin{thm}[\cite{BEM}]\label{thm:BEM}
Let $(M, \tau)$ and $(\hat M, \hat \tau)$ be T-dual. Let $H\in \tau$ and $\hat H \in \hat \tau$ be representatives as in Lemma \ref{lem:Hmodel} with $\hat q^* \hat H - q^* H  = d \overline{B}$. Then there is an isomorphism of differential complexes
\begin{gather*}
\begin{split}
\tau : (\Gamma(\mathbb{S})^{T^k},d_H) \to (\Gamma(\hat{\mathbb{S}})^{T^k},d_{\hat H}), \qquad
\tau(\rho) = \int_{T^k} e^{\overline{B}} \wedge q^*\rho,
\end{split}
\end{gather*}
where the integration is along the
fibers of $\hat q \colon \overline{M} \to \hat M$.
\end{thm}

\begin{proof}
First, let us make more explicit the definition of $\tau$. Notice that $e^{\overline{B}} \wedge q^*\rho$ is linear in $\rho$, and hence it suffices to define the fiber-wise integral above for a degree-$m$ element $\rho \in \Gamma(\Lambda^mT^* M)^{T^k}$. For this choice of $\rho$, consider a connection $\theta$ on $M$ and using \eqref{eq:difformssplit} we decompose
$$
\rho = \sum_{j=0}^m \rho_j \in \bigoplus_{j=0}^{m}\Gamma(\Lambda^j T^*M \otimes \wedge^{m-j}\mathfrak{t}^{*})^{T^{k}}.
$$ 
Then, $\tau$ is defined taking the component of $e^{\overline{B}} \wedge q^*\rho $ which has degree $k$ along the fiber of $\hat p$ and integrating, that is,
$$
\tau(\rho) = \sum_{j = 0}^m \frac{1}{(k-m+j)!}\int_{T^k} \overline{B}^{k-m+j} \wedge q^*\rho_j.
$$
To write this expression, we have used that $H$ and $\hat H $ are as in Lemma \ref{lem:Hmodel}, and therefore $\overline{B}$ belongs to $\mathcal{F}^1$ in the filtration \eqref{eq:filtration} for $p \colon M \to B$. We prove next that $\tau$ is a morphism of differential complexes:
\begin{align*}
\mathrm{d}_{\hat H}(\tau(\rho)) & =\int_{T^{k}} \mathrm{d}_{\hat H}(e^{\overline{B}}\wedge \rho)\\
& = \int_{T^{k}} (\hat H - H)\wedge e^{\overline{B}} \wedge \rho + e^{\overline{B}}\wedge \mathrm{d}\rho - \hat H\wedge e^{\overline{B}}\wedge \rho\\
& =\int_{T^{k}} e^{\overline{B}}\wedge \mathrm{d}\rho - H\wedge e^{\overline{B}}\wedge \rho\\
& = \tau (\mathrm{d}_{H}\rho)
 \end{align*}
where, for simplicity, we have omitted pull-backs in the formula above. It remains to show that $\tau$ so defined is an isomorphism. For this, one can work on a local coordinate patch $U$ on the base manifold $B$, so that $\overline{M}_{|U} = U \times T^k \times \hat T^k$, and take a global frame for $T^*T^k$ and $T^*\hat T^k$. We leave the details as an \textbf{exercise}.
\end{proof}

We are ready for the main result of this section, which states that T-duality induces an isomorphism of equivariant Courant algebroids, upon reduction to the base. Let $(E,M,T^k)$ be an equivariant Courant algebroid with base $B$. Consider the vector bundle 
$$
E/T^k \to B,
$$
whose sheaf of sections is given by the invariant sections of $E$, that is, $\Gamma(E/T^k) = \Gamma(E)^{T^k}$. We can endow $E/T^k$ with a natural Courant algebroid $(E/T^k,\IP{,},[,],\pi_{E/T^k})$ as in Definition \ref{d:CA}, with anchor $\pi_{E/T^k} \colon E/T^k \to TB$ and pairing and Dorfman bracket given by the restriction of the neutral pairing and the Dorfman bracket on $E$ to $\Gamma(E)^{T^k}$. Observe that $E/T^k$ has the same rank as $E$, and hence $E/T^k$ is not an exact Courant algebroid over $B$ (see Definition \ref{d:CAex}). We will call $(E/T^k,\IP{,},[,],\pi_{E/T^k})$ the \emph{simple reduction} of $E$ by $T^k$. 

\begin{thm}[\cite{GualtieriTdual}]\label{th:Tduality}
Let $(E,M,T^k)$ and $(\hat E, \hat M, \hat T^k)$ be equivariant exact Courant algebroids over the same base $M/T^k = B = \hat M/\hat T^k$. Assume that $(M,[E])$ is T-dual to $(\hat M,[\hat E])$. Then, there exists a Courant algebroid isomorphism
$$
\psi \colon E/T^k \to \hat E/\hat T^k
$$
which exchanges the twisted differentials in Lemma \ref{l:derivedbracket} and which is compatible with Clifford multiplication, that is, 
$$
\tau(\slashed d_0 \rho) = \hat{\slashed d_0}(\tau \rho), \quad \textrm{and} \quad \tau(e \cdot \rho) = \psi(e) \cdot \tau(\rho)
$$ 
for all $e \in \Gamma(E)^{T^k}$ and $\rho \in \Gamma(\mathbb{S})^{T^k}$.
\end{thm}

\begin{proof}
We choose equivariant isotropic splittings of $E$ and $\hat E$, and $\overline{B}$ as in Definition \ref{def:toptdual}. Given $X + \xi \in \Gamma(TM \oplus T^*M)^{T^k}$, choose the unique lift $\overline{X}$ of $X$ to $\Gamma(T\bar{M})^{T^k \times \hat T^k}$ such that
\begin{equation}\label{eq:liftvector}
q^* \xi(Y) - \overline{B}(\overline{X}, Y) = 0 \quad \mbox{ for all } Y \in \mathfrak t.
\end{equation}
Due to this condition the form $q^* \xi - \overline{B}(\overline{X}, \cdot)$ is basic for the
bundle determined by $\hat{q}$, and can therefore be pushed forward to
$\hat{M}$.  We define a map
\begin{align*}
\psi(X + \xi) = \hat{q}_*(\overline{X} + q^* \xi - \overline{B}(\overline{X}, \cdot)).
\end{align*}
Let us see next that $\psi$ is independent of the choice of isotropic splittings. For a different choice of isotropic splittings we have
$$
H' = H + dB, \qquad \hat H' = \hat H + d \hat B,
$$
and hence 
$$
\hat q^* \hat H' - q^*H'  = d (\overline{B} + \hat q^* \hat B - q^*B) = d \overline{B}'
$$
for $\overline{B}' = \overline{B} - q^*B + \hat q^* \hat B$. Moreover, the action of $q^* B$ and $\hat q^* \hat B$ on $\mathfrak t \otimes \hat{\mathfrak
t}$ is trivial. Hence when lifting vectors to the configuration
space, using either $\overline{B}$ or $\overline{B}'$ yields the same result, as claimed.

We check next that $\psi$ preserves the Courant structure. First, $\psi$ is an isometry by
\begin{align*}
\IP{\psi(X+\xi), \psi(X+\xi)} & = \IP{\hat{q}_*(\overline{X} + q^* \xi - \overline{B}(\overline{X}, \cdot)),\hat{q}_*(\overline{X} + q^* \xi - \overline{B}(\overline{X}, \cdot))}\\
& = \hat{q}_*(q^* \xi - \overline{B}(\overline{X}, \cdot))(\hat{q}_*(\overline{X}))\\
& = (q^* \xi - \overline{B}(\overline{X}, \cdot))(\overline{X})\\
& = \xi(X).
\end{align*}
To see that $\psi$ preserves the Dorfman bracket, we prove first that it is compatible with Clifford multiplication
$$
\tau(e \cdot \rho) = \psi(e) \cdot \tau(\rho).
$$
For $e = X+ \xi \in \Gamma(TM \oplus T^*M)^{T^k}$, we have
\begin{align*}
\psi(e)\cdot \tau(\rho)& = \psi(e)\cdot \int_{T^{k}} e^{\overline B} \wedge \rho\\
& = \int_{T^{k}} (\overline{X} + q^* \xi - \overline{B}(\overline{X}, \cdot)) \cdot (e^{\overline B} \wedge \rho)\\
& = \int_{T^{k}} i_{\overline{X}} \overline{B} \wedge e^{\overline{B}}\wedge\rho + e^{\overline{B}}\wedge i_{q_{*}\overline{X}}\rho + (q^{*}\xi-i_{\overline{X}}\overline{B})\wedge e^{\overline{B}} \wedge \rho\\
& =\int_{T^{k}} e^{\overline{B}}\wedge (i_{X}\rho+\xi\wedge\rho)\\
& = \tau(e \cdot \rho).
\end{align*}
Applying now Theorem \ref{thm:BEM} and Lemma \ref{l:derivedbracket}, for any $a,b \in \Gamma(TM \oplus T^*M)^{T^k}$, we have
\begin{align*}
\psi([a,b])\cdot \tau(\rho)& = \tau([v_{1},v_{2}]\cdot \rho)\\
& = \tau([[\slashed d_0,a \cdot],b \cdot](\rho))\\
& = [[\hat{\slashed d}_0,\psi(a) \cdot],\psi(b) \cdot](\tau(\rho)) \\
& = [\psi(a),\psi(b)]\cdot \tau(\rho).
\end{align*}
The fact that $\psi$ preserves the anchor map is left as an \textbf{exercise}.
\end{proof}

\begin{rmk}\label{rem:nonuniquepsi}
Given T-dual pairs $(M,[E])$ and $(\hat M,[\hat E])$, the isomorphism $\psi$ in Theorem \ref{th:Tduality} is not unique, as it depends on the choice of two-form $\overline{B}$ as in Definition \ref{def:toptdual}. For instance, there is the freedom to modify $\overline{B}$ by adding a closed two-form on $M \times_B \hat M$. The change of $\psi$ can be seen explicitly in Example \ref{ex:GRicciflatHopfTdual} by introducing an extra parameter multiplying the second summand in the definition of $\overline{B}$.
\end{rmk}

\section{Geometric T-duality}

In this section we explain the fundamental interplay between T-duality, the generalized Ricci tensor, and the generalized Ricci flow.  For this, we need to introduce a notion of duality for pairs $(\GG,\divop)$, given by a generalized metric and a divergence operator.  Given an equivariant Courant algebroid $(E,M,T^k)$ and a $T^k$-invariant generalized metric $\GG$ on $E$, we can push-forward $V_+$ along the bundle projection $p \colon M \to B$
$$
p_*V_+ \subset E/T^k
$$
to obtain a generalized metric $p_*\GG$ on the simple reduction. This way we obtain an identification between $T^k$-invariant generalized metrics on $E$ and generalized metrics on $E/T^k$. Similarly, given a $T^k$-invariant divergence operator $\divop \colon \Gamma(E) \to C^\infty(M)$ it induces a first-order differential operator 
$$
p_*\divop \colon \Gamma(E)^{T^k} = \Gamma(E/T^k) \to C^\infty(M)^{T^k} \cong C^\infty(B).
$$
The $\pi_E$-Leibniz rule for $\divop$ restricted to $p^*C^\infty(B)$ (see \eqref{eq:Leibnizdiv}) is precisely the $\pi_{E/T^k}$-Leibniz rule for $p_*\divop$, which thus defines a divergence operator on $E/T^k$.

\begin{defn}\label{def:dualpairs}
Let $(E,M,T^k)$ and $(\hat E, \hat M, \hat T^k)$ be equivariant exact Courant algebroids over the same base $M/T^k = B = \hat M/\hat T^k$. Assume that $(M,[E])$ is T-dual to $(\hat M,[\hat E])$ and fix an isomorphism $\psi \colon E/T^k \to \hat E / \hat T^k$ as in Theorem \ref{th:Tduality}. We will say that a $\hat T^k$-invariant pair $(\hat \GG,\hat{\divop})$ on $\hat E$ is the \emph{T-dual} via $\psi$ of a $T^k$-invariant pair $(\GG,\divop)$ on $E$ if
$$
\psi \circ p_* \GG = \hat p_* \hat \GG \circ \psi \qquad \textrm{and} \qquad p_* \divop = \psi^* \hat p_* \hat{\divop},
$$
where $p \colon M \to B$ and $\hat p \colon \hat M \to B$ are the bundle projections.
\end{defn}

The notion of T-dual pairs depends upon the choice of isomorphism $\psi$, which is by no means unique (see Remark \ref{rem:nonuniquepsi}). Nonetheless, we will often abuse notation and say that $(\hat \GG,\hat{\divop})$ is \emph{the T-dual} of $(\GG,\divop)$, assuming that some background isomorphism has been chosen.

In order to prove that the generalized Ricci tensor in Definition \ref{d:generalizedRicci} is naturally exchanged under T-duality, we need to understand the relation between $T^k$-invariant generalized connections on $E$ and generalized connections on $E/T^k$ (see \S \ref{s:connection}). Denote by $\cD_E$ the space of generalized connections on $E$, by $\cD_E^{T^k} \subset \cD_E$ the space of $T^k$-invariant generalized connections, and by $\cD_{E/T^k}$ the space of generalized connections on the simple reduction $E/T^k$.

\begin{lemma}\label{lem:Gconnec}
There is a canonical identification $\cD_E^{T^k} = \cD_{E/T^k}$.
\end{lemma}
\begin{proof}
The space $\cD_{E/T^k}$ is non-empty, since we can construct a generalized connection $\check D$ out of an orthogonal connection $\nabla^{E/T^k}$ on $E/T^k$ by 
$$
\check D_{a}b = \nabla^{E/T^k}_{\pi_{E/T^k}(a)}b.
$$
Similarly, the space $\cD_E^{T^k}$ is non-empty, since $T^k$ is compact and hence we can construct $D \in \cD_E^{T^k}$ by averaging an arbitrary element in $\cD_E$. The affine spaces $\cD_E^{T^k}$ and $\cD_{E/T^k}$ are both modeled on $\Gamma(E^* \otimes \mathfrak{o}(E))^{T^k}$, via the identification $\Gamma(E/T^k) = \Gamma(E)^{T^k}$, and therefore it is enough to prove that a generalized connection on $E/T^k$ induces an invariant connection on $E$. For this, we note that $E$ is locally generated by $\Gamma(E)^{T^k}$, and therefore given $\check D \in \cD_{E/T^k}$ we can define a generalized connection $D$ on $E$ by extending its action from $\Gamma(E)^{T^k}$ to $\Gamma(E)$ imposing the Leibniz rule with respect to the anchor $\pi_E$. The generalized connection $D$ is compatible with the pairing, as for $f \in C^\infty(M)$ and $a,b,c \in \Gamma(E)^T$ we have
\begin{align*}
\pi_E(a)\la f b, c \ra & = df(\pi_E(a)) \IP{ b, c } + f \pi_{E/T^k}(a)(\IP{ b, c })\\
& = \la df(\pi_E(a)) + f \check D_{a}b,c\ra + f\la b,\check D_{a}c \ra\\
& = \la D_{a}(fb),c\ra + \la b,D_{a}c \ra.
\end{align*}
For the first equality we have used that $\IP{b,c}$ is the pull-back of a function on $B$ and consequently
$$
\pi_E(a)\IP{b,c} = \pi_{E/T^k}(a)\IP{b,c}.
$$
\end{proof}

The next result shows that the generalized Ricci tensor in Definition \ref{d:generalizedRicci} is naturally exchanged under T-duality.

\begin{prop}\label{prop:duality}
Let $(E,M,T^k)$ and $(\hat E, \hat M, \hat T^k)$ be equivariant exact Courant algebroids over the same base $M/T^k = B = \hat M/\hat T^k$. Assume that $(M,[E])$ is T-dual to $(\hat M,[\hat E])$ and let $(\hat \GG,\hat{div})$ on $\hat E$ be the T-dual of an invariant pair $(\GG,div)$ on $E$. Then
\begin{equation}\label{eq:GRiccidual}
p_* \gRc(\GG,\divop)\circ \psi = \psi \circ \hat p_* \gRc(\hat \GG,\hat{\divop}),
\end{equation}
where $\gRc$ is as in Definition \ref{d:generalizedRicci}. In particular, generalized Einstein pairs are exchanged under T-duality (see Definition \ref{d:generalizedEinstein}).
\end{prop}

\begin{proof}
By definition of the generalized Ricci tensor (see Definition \ref{d:generalizedRicci}) it suffices to prove that
\begin{equation*}
p_* \gRc^\pm(\GG,\divop) = \psi^*\hat p_* \gRc^\pm(\hat \GG,\hat{\divop}),
\end{equation*}
where $\gRc^\pm$ are as in Definition \ref{def:Ricci}. Given a $T^k$-invariant pair $(\GG,\divop)$, let $D \in \cD^0(\GG,\divop)$ (see Lemma  \ref{lem:Ddiv}). Averaging over $T^k$ if necessary, we obtain $D \in \cD_E^{T^k}$. Using the isomorphism $\psi$ in Theorem \ref{th:Tduality} we have that $\psi_* p_*D$ is a generalized connection on $\hat E/\hat T^k$, which induces a generalized connection $\hat D$ on $\hat E$ by Lemma \ref{lem:Gconnec}. By hypothesis $(\GG,\divop)$ and $(\hat \GG,\hat{\divop})$ are T-dual, and hence
$$
\hat p_* \hat D \hat \GG = (\psi p_*D \psi^{-1} )(\psi p_*\GG \psi^{-1}) = \psi p_*D\GG \psi^{-1} = 0,
$$
and also
$$
\hat p_* \divop_{\hat{D}}(a) = \tr \hat D a = \tr \psi p_*D \psi^{-1}(a) = \tr D (\psi^{-1}a) = (\psi^{-1})^*\divop(a) = \hat p_* \hat \divop (a)
$$
for any $a \in \Gamma(\hat E/\hat T^k)$. Thus, it follows that $\hat D \in \cD^0(\hat \GG,\hat \divop)$, which implies (see Definition \ref{def:Ricci})
$$
p_* \gRc^\pm(\GG,\divop) = p_* \gRc^\pm_D = \psi^*\hat p_* \gRc^\pm_{\hat D} = \psi^*\hat p_* \gRc^\pm(\hat \GG,\hat{\divop}).
$$
\end{proof}

\begin{rmk}
A proof of the fact that generalized Ricci flat metrics are exchanged under T-duality has been obtained by several authors and with different methods \cite{BarHek,GF19,StreetsTdual}, some including an extension to more general Courant algebroids. A more general result, for Poisson-Lie T-duality, in the sense of Klim\v cik and \v Severa \cite{Klim_k_1995}, was proved in \cite{SeveraValach1}.
\end{rmk}

Our next goal is to understand the behavior of the scalar curvatures $\mathcal{S}^\pm$ in \eqref{eq:scalardef} under T-duality. The reason why we do not consider here the generalized scalar curvature $\mathcal{S}$ in Definition \ref{def:scalar} is that, in principle, this would require showing that T-duality preserves the notion of closed pair in Definition \ref{def:closed}. This turns out to be a delicate question, related to a phenomenon known in the literature as \emph{dilaton shift}, which will be discussed in the next section.

\begin{prop}\label{prop:dualityScalar}
Let $(E,M,T^k)$ and $(\hat E, \hat M, \hat T^k)$ be equivariant exact Courant algebroids over the same base $M/T^k = B = \hat M/\hat T^k$. Assume that $(M,[E])$ is T-dual to $(\hat M,[\hat E])$ and let $(\hat \GG,\hat{\divop})$ on $\hat E$ be the T-dual of an invariant pair $(\GG,\divop)$ on $E$. Then
\begin{equation}\label{eq:Gscalardual}
\mathcal{S}^\pm(\GG,\divop) = \mathcal{S}^\pm(\hat \GG,\hat{\divop}),
\end{equation}
where $\mathcal{S}^\pm$ is as in equation \eqref{eq:scalardef}. In particular, solutions of the equation $\mathcal{S}^\pm = 0$ are exchanged under T-duality.
\end{prop}

\begin{proof}
Given a point on the base manifold $B$, we can consider a local spinor bundle $S_\pm$ for the bundle $V_\pm/T^k$. Via the duality isomorphism $\psi$ in Theorem \ref{th:Tduality} we can identify $S_\pm$ with a local spinor bundle for $\hat V_\pm /\hat T^k$. Let $D \in \cD^0(\GG,\divop)^{T^k}$ and consider $\hat D \in \cD^0(\hat \GG,\hat \divop)^{\hat T^k}$ the induced generalized connection on $\hat E$ via $\psi$, as in the proof of Proposition \ref{prop:duality}. Then, \eqref{eq:Gscalardual} follows by Definition \ref{eq:scalardef}, using $D$ and $\hat D$ to evaluate the scalar curvatures $\mathcal{S}^\pm$ and applying Lemma \ref{l:mixedfixed} and Lemma \ref{lem:dDpm}.
\end{proof}

\begin{rmk}
A proof of the fact that solutions of the equations
\begin{equation*}
\gRc^+ = 0, \qquad \mathcal{S}^+ = 0
\end{equation*}
are exchanged under T-duality was first given in \cite{Jur_o_2018,SeveraValach2} for the case of exact pairs. Notice that when $(\GG,\divop)$ is a compatible pair, so is the T-dual $(\hat \GG,\hat{\divop})$, and in this case the equation $\gRc^+ = 0$ is equivalent to $\gRc = 0$ (see Lemma \ref{lem:Ricciskew}).
\end{rmk}

We state next our main result of this section, which follows as a consequence of the proof of Proposition \ref{prop:duality}.

\begin{thm}\label{th:dualityGRicci}
Let $(E,M,T^k)$ and $(\hat E, \hat M, \hat T^k)$ be equivariant exact Courant algebroids over the same base $M/T^k = B = \hat M/\hat T^k$. Assume that $(M,[E])$ is T-dual to $(\hat M,[\hat E])$. Then, if $(\GG_t,\divop_t)$ is a $T^k$-invariant solution of the flow
\begin{equation}\label{eq:GRicciflowTdual}
\GG^{-1}\dt \GG = -2 \gRc(\GG_t, \divop_t)
\end{equation}
and $(\hat \GG_t,\hat{\divop}_t)$ is the family of $\hat T^k$-invariant T-dual pairs, then $(\hat \GG_t,\hat{\divop}_t)$ is also a solution of \eqref{eq:GRicciflowTdual}. In particular, solutions of gauge-fixed generalized Ricci flow, and generalized Einstein pairs, are exchanged under T-duality.
\end{thm}

\begin{proof}
The proof is a direct consequence of Definition \ref{def:dualpairs} and equation \eqref{eq:GRiccidual}.
\end{proof}

\begin{rmk}
Preservation of the flow \eqref{eq:GRicciflowTdual} under T-duality, as defined in the present work, was proved in \cite{GF19} following \cite{StreetsTdual}. A more general result for Poisson-Lie T-duality was proved in \cite{SeveraValach1}.
\end{rmk}

It is important to notice that, in the strictest sense, the generalized Ricci flow \eqref{f:GRF} is not necessarily preserved under T-duality, since it is formulated in terms of the generalized Ricci tensor for the special class of pairs $(\GG,\divop^\GG)$. This is due to the fact that the Riemannian divergence $\divop^\GG$ of a generalized metric is defined in terms of the anchor map $\pi_E \colon E \to TM$ (see Definition \ref{d:GRiemnniandiv}), which is not preserved under T-duality. Instead, a solution of \eqref{f:GRF} is typically exchanged with a solution of the gauge-fixed generalized Ricci flow, corresponding to an exact pair $(\GG,\divop)$ (see Definition \ref{def:exact}). This is again related to the \emph{dilaton shift}, which will be discussed in the next section.

\begin{rmk}
As observed in \cite{GualtieriTdual}, pluriclosed metrics and generalized K\"ahler structures are preserved under T-duality. Relying on Theorem \ref{th:dualityGRicci}, we expect that flow lines for pluriclosed flow and generalized K\"ahler Ricci flow are also preserved by the T-duality isomorphism.
\end{rmk}

\section{Buscher rules and the dilaton shift}\label{s:Buscherrules}

In the present section we describe more explicitly the T-dual of a pair $(\GG,\divop)$, as considered in Definition \ref{def:dualpairs}. The specific local formulae for calculating the T-dual are the famous ``Buscher rules'', discovered in \cite{Buscher}, which we will express here in terms of global tensors. Our discussion will illustrate further the simplicity of Definition \ref{def:dualpairs}, and the value of adopting the viewpoint of Courant algebroids.

For the sake of clarity, we will assume that the torus $T^k$ is one-dimensional. The discussion can be easily generalized to the case of higher dimensional tori. We set the notation $T^1 = S^1$. Let $(E,M,S^1)$ and $(\hat E, \hat M, \hat S^1)$ be equivariant exact Courant algebroids over the same base $M/S^1 = B = \hat M/\hat S^1$. Assume that $(M,[E])$ is T-dual to $(\hat M,[\hat E])$. Let $H\in [E]$ and $\hat H \in [\hat E]$ be representatives as in Lemma \ref{lem:Hmodel} with $\hat q^* \hat H - q^* H = d \overline{B}$. Then, we can choose equivariant isotropic splittings 
$$
\sigma \colon TM \to E, \qquad \sigma \colon T\hat M \to \hat E
$$ 
such that $E$ is given by the twisted Courant algebroid $(TM \oplus T^*M,\IP{,},[,]_H)$ over $M$, and similarly for $\hat{E}$. Recall from Proposition \ref{p:genmetricreduction} that an $S^1$-invariant generalized metric $\GG$ on $E$ is determined by an $S^1$-invariant pair $(g, b)$ of metric $g$ and two-form potential $b$ on $M$. Upon a choice of connection $\theta$ on $p \colon M \to B$, the pair $(g,b)$ can be expressed as
\begin{gather} \label{generalmetric}
 \begin{split}
 g =&\ g_0 \theta \odot \theta + g_1 \odot \theta + g_2\\
 b =&\ b_1 \wedge \theta + b_2  
 \end{split}
\end{gather}
where $b_i$ are basic forms of degree $i$ and $g_i$ are symmetric tensors of degree $i$. By Lemma \ref{lem:Hmodel}, we can assume that
$$
\overline{B} = - q^*\theta \wedge \hat q^* \hat \theta 
$$
for a suitable connection $\hat \theta$ on $\hat p \colon \hat M \to B$.

With this explicit decomposition of the data $(g, b)$ in hand we can now give the Buscher rules, which determine their relationship under T-duality.  We give two versions of this relationship, first expressed as an explicit formula for the associated pair $(\hat{g}, \hat{b})$, then expressed as a mapping of the invariant decomposition of $(g, b)$.

\begin{lemma}\label{lem:buscher} 
Given $\GG$ an $S^1$-invariant generalized metric on $(TM \oplus T^*M,\IP{,},[,]_H)$, consider the T-dual generalized metric $\hat{\GG}$ on $(T\hat M \oplus T^*\hat M,\IP{,},[,]_{\hat H})$, as in Definition \ref{def:dualpairs}. Then, if the pair $(g,b)$ associated to $\GG$ is given by (\ref{generalmetric}), then
the pair $(\hat{g}, \hat{b})$ determined by $\hat \GG$ takes the form
\begin{gather}\label{bus}
 \begin{split}
  \hat{g} =&\ \frac{1}{g_0}\hat{\theta} \odot\hat{\theta} - \frac{b_1}{g_0}
\odot \hat{\theta} + g_2 + \frac{b_1 \odot b_1 - g_1 \odot g_1}{g_0}\\
 \hat{b} =&\ - \frac{g_1}{g_0} \wedge \hat{\theta} + b_2 + \frac{g_1 \wedge
b_1}{g_0}.
 \end{split}
\end{gather}
\end{lemma}

\begin{proof}
We denote by $e_\theta$ the canonical vector field of $\theta$, given by the unique $S^1$-invariant vertical  vector field such that $\theta(e_\theta) = 1$. Let $X + \xi \in \Gamma(TM \oplus T^*M)^{S^1}$. Then, we can write
\begin{equation*}
\begin{split}
    X& = \theta^\perp Y+ f_X e_\theta,\\
    \xi & = p^*\zeta+ f_\xi\theta,
\end{split}
\end{equation*}
where $Y + \zeta \in \Gamma(TB \oplus T^*B)$, and $\theta^\perp Y$ denotes the horizontal lift with respect to $\theta$, and $f_X,f_\xi \in C^\infty(B)$. Going back to the definition of $\psi$ in the proof of Theorem \ref{th:Tduality}, the vector $\overline{X}$ on $\overline{M}$ satisfying \eqref{eq:liftvector} is given by
$$
\overline{X} = \hat{\theta}^\perp X + f_\xi e_{\hat \theta} ,
$$
and therefore
$$
\psi(X + \xi) = \hat \theta^\perp Y + f_\xi e_{\hat \theta} + \hat p^* \zeta + f_X \hat \theta = \hat X + \hat \xi.
$$
Using now that
\begin{align*}
e^b(X + g(X)) =&\ \theta^\perp Y + f_X e_\theta +(f_X g_{1} + i_{Y} g_{2}-f_X b_{1} + i_{Y}b_{2})\\
&\ +(f_X g_{0}+ g_{1}(Y) + b_{1}(Y))\theta,
\end{align*}
it is not difficult to see that
$$
\psi(e^b(X + g(X))) = e^{\hat b}(\hat X + \hat g(\hat X)),
$$
where $(\hat g, \hat b)$ are as in \eqref{bus}. The explicit calculation is left as an \textbf{exercise}.
\end{proof}

For the calculations to come later, it will be useful to give the T-duality relationship explicitly in terms of the canonical
decomposition of an $S^1$-invariant pair $(g, b)$ on a principal bundle which we recall next.

\begin{lemma} \label{metricbflddecomp} Given $M$ a principal $S^1$ bundle over $B$ with canonical vector field $e_{\theta}$, an $S^1$-invariant metric on $M$ is uniquely
determined by a metric $h$ on $B$, a smooth function $\phi \in C^{\infty}(B)$, and a principal connection $\theta$.  More
precisely, $g$ may be expressed
\begin{align*}
g = \phi \theta \odot \theta + p^*h,
\end{align*}
where
\begin{align*}
\phi =&\ g(\dth, \dth), \qquad \theta = \frac{g(\dth, \cdot)}{g(\dth,
\dth)}, \qquad h(\cdot, \cdot) = g( \pi^{\theta} \cdot, \pi^{\theta} \cdot),
\end{align*}
and $\pi^{\theta}$ is the horizontal projection determined by $\theta$,
i.e.
\begin{align*}
\pi^{\theta}(X) = X - \theta(X) \dth.
\end{align*}
Also, an $S^1$-invariant two-form $b$ is uniquely determined by basic forms $\eta, \mu$.  More precisely, $b$ may be expressed
\begin{align*}
b = \theta \wedge \eta + \mu,
\end{align*}
where
\begin{align*}
\eta = e_{\theta} \hook b, \qquad \mu = b - \theta \wedge \eta.
\end{align*}
\begin{proof} \textbf{Exercise}.
\end{proof}
\end{lemma}

The following result provides a version of the T-duality relationship which is adapted to the canonical connections determined by the $S^1$-invariant metrics.

\begin{prop}\label{Tdualrules} 
Given $\GG$ an $S^1$-invariant generalized metric on $(TM \oplus T^*M,\IP{,},[,]_H)$, consider the T-dual generalized metric $\hat{\GG}$ on $(T\hat M \oplus T^*\hat M,\IP{,},[,]_{\hat H})$, as in Definition \ref{def:dualpairs}. Let $(g,b)$ and $(\hat{g}, \hat{b})$ be the pairs associated to $\GG$ and $\hat \GG$, respectively. Let $\theta_g, \phi_g, h_g, \eta_g, \mu_g$ denote the decomposition determined by $(g,b)$ via Lemma \ref{metricbflddecomp}, with $\theta_{\hat{g}}$, etc. associated to $(\hat{g}, \hat{b})$.  Then
\begin{align*}
\phi_{\hat{g}} =&\ \frac{1}{\phi_g}\\
\theta_{\hat{g}} =&\ \hat{\theta} + \eta_g\\
h_{\hat{g}} =&\ h_g\\
\eta_{\hat{g}} =&\ \theta_g - \theta\\
\mu_{\hat{g}} =&\ \mu_g - \eta_g \wedge \eta_{\hat{g}}.
\end{align*}
\begin{proof} First we compute
\begin{align*}
\theta_g = \theta + \frac{g_1}{g_0}.
\end{align*}
Then we obtain
\begin{align*}
\eta_g = \dth \hook b =&\ - b_1.
\end{align*}
Then we may express
\begin{align*}
\mu_g =&\ b - \theta_g \wedge \eta_g\\
=&\ b_1 \wedge \theta - \left( \theta + \frac{g_1}{g_0} \right) \wedge (-b_1) +
b_2\\
=&\ b_2 + \frac{g_1}{g_0} \wedge b_1.
\end{align*}
Furthermore we obtain
\begin{align*}
\hat{\theta}_{\hat{g}} = \hat{\theta} - b_1 = \hat{\theta} + \eta_g
\end{align*}
Then, according to the Buscher rules,
\begin{align*}
\eta_{\hat{g}} =&\ e_{\hat \theta} \hook \hat{b} = \frac{g_1}{g_0} = \theta_g - \theta.
\end{align*}
Then we obtain
\begin{align*}
\mu_{\hat{g}} =&\ \hat{b} - \hat{\theta}_{\hat{g}} \wedge \eta_{\hat{g}}\\
=&\ - \frac{g_1}{g_0} \wedge \hat{\theta} + b_2 + \frac{g_1 \wedge b_1}{g_0} -
\left( \hat{\theta} + \eta_g \right) \wedge \left( \theta_g - \theta \right)\\
=&\ (\hat{\theta} - b_1) \wedge \frac{g_1}{g_0} + b_2 - \hat{\theta}_{\hat{g}}
\wedge (\frac{g_1}{g_0})\\
=&\ b_2\\
=&\ \mu_g - \frac{g_1}{g_0} \wedge b_1\\
=&\ \mu_g - \eta_{g} \wedge \eta_{\hat{g}}.
\end{align*}
\end{proof}
\end{prop}

Having described explicitly the change of a generalized metric in terms of the Buscher rules in Lemma \ref{lem:buscher} and Proposition \ref{Tdualrules}, we turn next to the interplay between divergences and T-duality. Our discussion provides a theoretical framework to understand the so-called \emph{dilaton shift} in string theory. To provide a more invariant description, we shall work in the generality of Definition \ref{def:toptdual} and Definition \ref{def:dualpairs}, for arbitrary torus bundles. Let $(E,M,T^k)$ and $(\hat E, \hat M, \hat T^k)$ be equivariant exact Courant algebroids over the same base $M/T^k = B = \hat M/\hat T^k$. Assume that $(M,[E])$ is T-dual to $(\hat M,[\hat E])$ and consider T-dual pairs $(\GG,\divop)$ and $(\hat \GG,\hat{\divop})$. As in Proposition \ref{Tdualrules}, the $T^k$-invariant metric $g$ on $M$ induced by $\GG$ is equivalent to a triple $(g_V,\theta_g,h_g)$, where $h_g$ is a metric on $B$, $\theta$ is a principal connection on $M \to B$, and $g_V$ is a family of metrics on the fibers of $M$, that is,
$$
g = g_V(\theta_g , \theta_g) + p^*h_g.
$$
Using the Riemannian divergence $\divop^\GG$ of $\GG$ (see Definition \ref{d:GRiemnniandiv}), we can express
\begin{equation*}
\divop(a) = \divop^\GG(a) - \la e,a\ra
\end{equation*}
for a suitable $e \in \Gamma(E)^{T^k}$. We denote by $\nu \in \Gamma(|\det VM^*|)$ the density induced by $g_V$ on the vertical bundle $VM \subset TM$. Identifying 
$$
VM \cong M \times \mathfrak{t},
$$ 
we regard $\nu = |\det g_V|$ as an equivariant map $\nu \colon M \to |\det \mathfrak{t}^*|$ and,
using that $T^k$ is unimodular, we obtain a well-defined map
$$
\nu \colon B \to |\det \mathfrak{t}^*|.
$$
Observe that the following expression defines a well-defined one-form on the base manifold
$$
d \log \nu := \nu^{-1} d\nu \in \Gamma(T^*B).
$$
Similarly, we denote by $(\hat e, \hat g,h_{\hat g},\theta_{\hat g},\hat g_V, \hat \nu)$ the corresponding data associated to $(\hat \GG,\hat{\divop})$. By Proposition \ref{Tdualrules}, we have an equality $h = h_g = h_{\hat g}$ for the induced metric on the base manifold $B$.

\begin{prop}[Dilaton shift]\label{prop:dilatonshift}
Let $(E,M,T^k)$ and $(\hat E, \hat M, \hat T^k)$ be equivariant exact Courant algebroids over the same base $M/T^k = B = \hat M/\hat T^k$. Assume that $(M,[E])$ is T-dual to $(\hat M,[\hat E])$ and consider T-dual pairs $(\GG,\divop)$ and $(\hat \GG,\hat{\divop})$. Then,
\begin{equation}\label{eq:dilatonshift}
\hat e = \psi(e) + 2 d\log (\hat \nu/\nu),
\end{equation}
where $d\log (\hat \nu/\nu) : = d\log \hat \nu - d\log \nu$ is regarded as a section of $\hat E/\hat T^k$ using the anchor map. 
\end{prop}
\begin{proof}
We denote $\underline{\pi} = \pi_{E/T^k}$. The statement follows from the formula
$$
\divop(a) = \frac{L_{\underline{\pi}(a)}\mu^{h}}{\mu^h} + 2\la \nu^{-1} d\nu,a\ra - \la e,a\ra,
$$
for any $a \in \Gamma(E)^{T^k}$, since the metric $h$ on $B$ induced by $g$ and $\hat g$ coincide and $\psi$ is the identity along the kernel of the anchor map $\underline{\pi}$. To prove this formula, we calculate
\begin{align*}
\frac{L_{\pi_E(a)}\mu^g}{\mu^g} =  \frac{L_{\underline{\pi}(e')}\mu^h}{\mu^h} + \nu^{-1} d\nu(\underline{\pi}(e')),
\end{align*} 
using the natural decomposition $\mu^g =  \mu^h \otimes \nu$.
\end{proof}

Note that if $e = \varphi$ and $\hat e = \hat \varphi$ for invariant $1$-forms $\varphi$ on $M$ and $\hat \varphi$ on $\hat M$, respectively, then \eqref{eq:dilatonshift} is equivalent to
\begin{equation}\label{eq:dilatonexp}
\hat \varphi = \varphi - 2d \log \Bigg{(}\frac{|\det g_V|}{| \det \hat g_{V}|}\Bigg{)} = \varphi - 4d \log |\det g_V|.
\end{equation}
where we have used Proposition \ref{Tdualrules} for
$$
| \det \hat g_V| = |\det g_V|^{-1}.
$$
Assuming that $\varphi$ and $\hat \varphi$ are exact, with the normalization $\varphi = 8d\phi$ and $\hat \varphi = 8d\hat\phi$, we obtain
$$
\hat \phi = \phi - \frac{1}{2}\log |\det g_V|.
$$
This equation was encountered by physicists in their computations of the dual Riemannian metric and dilaton field for T-dual sigma-models \cite{Buscher} and carries the name of \emph{dilaton shift}. Formula \eqref{eq:dilatonshift} was first obtained in \cite{GF19}, and later generalized to the context of Poisson-Lie T-duality in \cite{Jur_o_2018}. It would be interesting to compare our formula \eqref{eq:dilatonshift} with \cite[Eq. (3.16)]{de_la_Ossa_1993}, in the context of non-abelian duality in physics.

\section{Einstein-Hilbert action}\label{s:TdualGKRflow}

In this section we derive the preservation of the Einstein-Hilbert action \ref{d:genEHdiv} under T-duality.  The main point is to complete our analysis of the interplay between T-duality and the generalized scalar curvature initiated in Proposition \ref{prop:dualityScalar}. We will use in an essential way the dilaton shift formula \eqref{eq:dilatonshift}.

\begin{prop}\label{prop:dualityScalar2}
Let $(E,M,T^k)$ and $(\hat E, \hat M, \hat T^k)$ be equivariant exact Courant algebroids over the same base $M/T^k = B = \hat M/\hat T^k$. Assume that $(M,[E])$ is T-dual to $(\hat M,[\hat E])$ and let $(\hat \GG,\hat{\divop})$ on $\hat E$ be the T-dual of an invariant pair $(\GG,\divop)$ on $E$, in the sense of Definition \ref{def:dualpairs}. Then, $(\GG,\divop)$ is exact in the sense of Definition \ref{def:exact} if and only if $(\hat \GG,\hat{\divop})$ is exact. Furthermore, if this is the case, then
\begin{equation*}
\mathcal{S}(\GG,\divop) = \mathcal{S}(\hat \GG,\hat{\divop}),
\end{equation*}
where $\mathcal{S}$ is the generalized scalar curvature in Definition \ref{eq:scalardef}. In particular, generalized scalar flat metrics with exact divergence are exchanged under T-duality.
\end{prop}

\begin{proof}
If $(\GG,\divop)$ is exact, $e = d\phi$ for $\phi \in C^\infty(M)$. But then $\phi$ must be $T^k$-invariant and consequently $e$ is basic. By the dilaton shift formula \eqref{eq:dilatonshift}
$$
\hat e = \psi(e) = d\phi.
$$
Therefore, $(\GG,\divop)$ is exact if and only if $(\hat \GG,\hat \divop)$ is exact.  The last part of the statement follows from Proposition \ref{prop:dualityScalar}.
\end{proof}

\begin{rmk}
In general, the T-dual of a closed pair, in the sense of Definition \ref{def:closed}, may not be closed. A counterexample is given in Example \ref{ex:GRicciflatHopfTdual}. Nonetheless, if $(\GG,\divop)$ is closed then it must be a compatible pair, and hence the T-dual $(\hat \GG, \hat \divop)$ is also compatible. This suggests an extension of the notion of generalized scalar curvature in Definition \ref{def:scalar} to the case of compatible pairs.
\end{rmk}

Observe that by combining Proposition \ref{prop:dualityScalar2}, Proposition \ref{prop:duality}, and Theorem \ref{t:genEHvar}, it follows that critical points of the generalized Einstein-Hilbert functional \eqref{GHFdiv}, given by the solutions of \eqref{eq:RicciflatScalarflat}, are exchanged under T-duality. This indicates that the generalized Einstein-Hilbert functional is itself preserved by T-duality, and this is the content of our last result. Notice from \S \ref{s:EHF} that the divergence of a $T^k$-invariant exact pair $(\GG,\divop)$ on $M$ is of the form 
$$
\divop = \divop^\mu
$$
for some $T^k$-invariant volume form $\mu$ on $M$.

\begin{thm}\label{th:dualityGEH}
Let $(E,M,T^k)$ and $(\hat E, \hat M, \hat T^k)$ be equivariant exact Courant algebroids over the same base $M/T^k = B = \hat M/\hat T^k$. Assume that $(M,[E])$ is T-dual to $(\hat M,[\hat E])$ and let $(\GG,\divop^\mu)$ and $(\hat \GG,\divop^{\hat \mu})$ be exact T-dual invariant pairs on $E$ and $\hat E$, respectively, such that 
$$
\int_M \mu = \int_{\hat M} \hat \mu.
$$
Then,
\begin{equation}\label{eq:EHTdual}
\mathcal{EH}(\GG,\divop^\mu) = \mathcal{EH}(\hat \GG,\divop^{\hat \mu}).
\end{equation}
\end{thm}

\begin{proof}
By hypothesis there exist smooth functions $f,\hat f \in C^\infty(B)$ such that
$$
\mu = e^{-f}\mu^g, \qquad \hat \mu = e^{-\hat f}\mu^{\hat g}.
$$
Using the decompositions $\mu^g = \mu^{h_g} \otimes \nu$ and $\mu^{\hat g} = \mu^{h_{\hat g}} \otimes \hat \nu$ as in the proof of Proposition \ref{prop:dilatonshift}, it follows from \eqref{eq:dilatonshift} that
\begin{equation}\label{eq:dilatonshiftbis}
d\hat f =  df + d \log (\hat \nu / \nu ).
\end{equation}
Define functions $t, \hat t \in C^{\infty}(B)$ by integration along the fibers of $p$ and $\hat p$, that is,
$$
t(b) = \int_{p^{-1}(b)}\nu, \qquad \hat t(b) = \int_{\hat p^{-1}(b)}\hat \nu.
$$
Then, \eqref{eq:dilatonshiftbis} now reads
$$
e^{-f} t \mu^{h} = C e^{-\hat f} \hat t \mu^{h}
$$
and hence
$$
\int_M \mu = \int_B e^{-f} t \mu^{h} = C \int_B e^{-\hat f} \hat t \mu^{h} = C \int_B e^{-\hat f} \hat t \mu^{h} = C \int_{\hat M} \hat \mu.
$$
The normalization $\int_M \mu = \int_{\hat M} \hat \mu $ implies $C = 1$, and applying Proposition \ref{prop:dualityScalar2} we conclude
$$
\mathcal{EH}(\GG,\divop) = \int_B \mathcal{S}(\GG,\divop^{\mu}) e^{-f} t \mu^{h} = \int_B \mathcal{S}(\hat \GG,\divop^{\hat \mu}) e^{-\hat f} \hat t \mu^{h} = \mathcal{EH}(\hat \GG,\hat \divop),
$$
as claimed.
\end{proof}

The equality \eqref{eq:EHTdual} gives geometric content to the dilaton shift equation \eqref{eq:dilatonshift}, and provides a conceptual explanation of why critical points of the generalized Einstein-Hilbert functional \eqref{GHFdiv}, given by generalized Ricci flat and scalar flat metrics with exact divergence \eqref{eq:RicciflatScalarflat}, are exchanged under T-duality. From the point of view of the analysis in Chapter \ref{energychapter}, Theorem \ref{th:dualityGEH} implies that the energy functional and the invariant of generalized metrics $\lambda(\GG)$ in Definition \ref{d:Penergy} are compatible with topological T-duality. 

\begin{rmk}
In the context of type II string theory (see Remark \ref{rem:TypeIIscalar}), Theorem \ref{th:dualityGEH} shows that the low-energy string effective action functional \eqref{stringaction} is invariant under T-duality, as observed originally in \cite{Bergshoeff_1995}. In this physics setup, T-duality is regarded as a symmetry of the $\Sigma$-model action \eqref{eq:S_GGWZW}, while the functional \eqref{stringaction} is obtained from \eqref{eq:S_GGWZW} via a complicated Feynman diagram expansion. Thus, Theorem \ref{th:dualityGEH} can be regarded as a `quantum effect' of T-duality.
\end{rmk}

\section{Examples}\label{s:examplesTdual}

We start with examples of T-dual generalized Ricci flat and scalar flat metrics. The first example is classical, and shows that T-duality recovers `topological mirror symmetry' for \emph{semi-flat} Calabi-Yau metrics.

\begin{ex}\label{ex:CYsemiflat}
Consider the non-compact manifold $M=D\times T^{k}$, where $D\subset \mathbb{R}^{k}$ is the unit disk. We wish to endow $M$ with a $T^k$-invariant Ricci flat metric and apply T-duality to it. For this, we take coordinates $(x^{1},\dots x^{k})$ on $D$ and $(y^{1},\cdots,y^{k})$ on $T^{k}$, so that $\{dy^{1},\ldots,dy^{k}\}$ is a global frame for $T^*T^k$. By a theorem by Cheng and Yau \cite{ChengYau}, given a positive constant $C$ there exists a unique solution
$$
\phi \colon D \to \mathbb{R}
$$
to the real Monge-Amp\`ere equation
$$
\mathrm{det}\Bigg(\frac{\partial^{2}\phi}{\partial x^{i}\partial x^{j}}\Bigg) = C,
$$
satisfying the boundary condition and convexity property
$$
\phi_{|\partial D}=0, \qquad \Bigg(\frac{\partial^{2}\phi}{\partial x^{i}\partial x^{j}}\Bigg)>0.
$$
We define a $T^k$-invariant metric on $M$ by
\begin{equation*}
    g=\sum_{i,j} \phi_{ij}(dx^{i}\otimes dx^{j}+ dy^{i}\otimes dy^{j}),
\end{equation*}
where $\phi_{ij}= \frac{\partial^{2}\phi}{\partial x^{i}\partial x^{j}}$. Taking complex coordinates $z^j = x^j + \i y^j$ it turns out that $g$ is K\"ahler, with K\"ahler form
$$
\omega = \frac{\i}{2} \sum_{i,j}\phi_{ij} dz^i \wedge d\overline{z}^j,
$$
and hence its Chern Ricci form is
$$
\rho_C = - \i \overline{\partial} \partial \log \det (\phi_{ij}) = 0.
$$
Therefore, $g$ is Ricci flat. Observe that $g$ is flat along the $k$-dimensional torus fibers, and $\dim M = 2k$, and hence $g$ receives the name of \emph{semi-flat} Calabi-Yau metric on $M$ \cite{Leung}.

Consider now the $T^k$-invariant standard exact Courant algebroid $TM \oplus T^*M$ over $M$ with vanishing three-form $H = 0$. Take $\hat T^k$ another $k$-dimensional torus with coordinates $(\hat y^1,\ldots, \hat y^k)$ and define $\hat M = D \times \hat T^k$. Taking the closed two-form
$$
\overline{B} = \sum_j dy^j \wedge d \hat y^j
$$
on $\overline{M} = D \times T^k \times \hat T^k$, it follows directly from Definition \ref{def:toptdual} that $(M,0)$ is T-dual to $(\hat M,0)$. Arguing as in the proof of Lemma \ref{lem:buscher}, the isomorphism $\psi$ in the proof of Theorem \ref{th:Tduality} is given by
$$
\psi\Bigg(\frac{\partial}{\partial x^j}\Bigg) = \frac{\partial}{\partial x^j}, \quad \psi(dx^j) = dx^j, \quad \psi\Bigg(\frac{\partial}{\partial y^j}\Bigg) = -d \hat y^j, \quad  \psi(dy^j) = -\frac{\partial}{\partial \hat y^j}.
$$
Define the generalized metric $\GG$ on $TM \oplus T^*M$ by 
$$
V_+ = \{X + g(X) \;|\; X \in TM\} \subset TM \oplus T^*M,
$$
and notice that
$$
V_+ = \Bigg\langle \frac{\partial}{\partial x_j} + \phi_{jk} dx^k, \frac{\partial}{\partial y_j} + \phi_{jk} dy^k  \Bigg\rangle.
$$
Therefore, the T-dual generalized metric $\hat \GG$ is determined by
\begin{align*}
\hat V_+ = \psi(V_+) =  \Bigg\langle \frac{\partial}{\partial x_j} + \phi_{jk} dx^k, \frac{\partial}{\partial \hat y_j} + \phi^{jk} d\hat y^k  \Bigg\rangle \subset T\hat M \oplus T^*\hat M
\end{align*}
where $(\phi^{jk}) = (\phi_{jk})^{-1}$. From this, $\hat V_+ = \{X + \hat g(X) \;|\; X \in T\hat M\}$ where
$$
\hat g = \sum_{i,j} \phi_{ij}dx^{i}\otimes dx^{j} + \phi^{ij} d\hat y^{i}\otimes d\hat y^{j}.
$$
Consider next the pair $(\GG,\divop^\GG)$, where $\GG$ is the generalized metric above and $\divop^\GG$ is the associated Riemannian divergence (see Definition \ref{d:GRiemnniandiv}). By Propositions \ref{prop:Ricciexplicit} and \ref{prop:scalarpmexplicit}, this pair is trivially Ricci flat and scalar flat, and hence so is the corresponding T-dual pair $(\hat \GG,\hat \divop)$ by Propositions \ref{prop:duality} and  \ref{prop:dualityScalar2}. We already have a formula for $\hat \GG$, so let us calculate $\hat \divop$ applying the dilaton shift formula in Proposition \ref{prop:dilatonshift}. For this, notice that
\begin{align*}
\nu & = \det (\phi_{ij})^{1/2} dy_1\wedge \ldots dy_k = C^{1/2} dy_1\wedge \ldots \wedge dy_k,\\
\hat \nu & = \det (\phi^{ij})^{1/2} d\hat y_1\wedge \ldots\wedge  d\hat y_k = C^{-1/2} d\hat y_1\wedge \ldots \wedge d\hat y_k,
\end{align*}
and therefore by \eqref{eq:dilatonshift}
$$
\hat e = -2d\log C = 0.
$$
By Proposition \ref{prop:Ricciexplicit}, we conclude that $\hat{g}$ is a Ricci flat metric in the standard sense. 

To see this more explicitly, we apply the Legendre transform to the base coordinates, defining new coordinates $\hat x^j$ on the disk $D$ by
$$
\frac{\partial \hat x_{i}}{\partial x^{j}}=\phi_{ij}.
$$
In these new coordinates, the T-dual metric reads
$$
\hat g = \sum_{i,j} \phi^{ij}(d\hat x^{i}\otimes d\hat x^{j}+ d\hat y^{i}\otimes d\hat y^{j}),
$$
and taking complex coordinates $\hat z^j = \hat x^j + \i \hat y^j$ we conclude as before that $\hat g$ is a semi-flat Calabi-Yau metric on $\hat M$.
\end{ex}

Our second example shows the effect of T-duality on the generalized Ricci flat and scalar flat metrics on $S^3 \times S^1$ considered in Example \ref{ex:GRicciflatHopf}.

\begin{ex}\label{ex:GRicciflatHopfTdual}
Consider the compact four-dimensional manifold $M = SU(2) \times U(1)$ in Example \ref{ex:GRicciflatHopf}. We regard $M$ as a principal $T^2$-bundle over $S^2 \cong U(1)\backslash SU(2)$, via the natural left action of the abelian subgroup
$$
T^2 = U(1) \times U(1) \subset SU(2) \times U(1).
$$ 
Consider the $T^2$-equivariant Courant algebroid $(TM \oplus T^*M,\IP{,},[,]_H)$ over $M$, where 
$$
H = - k e^{123}
$$
for $k \in \mathbb{R}$. 

We claim that $(M,[H])$ is self-T-dual. To see this, we regard the correspondence space $\overline{M}$ in Definition \ref{def:toptdual} inside the Lie group
$$
\iota \colon \overline{M} \hookrightarrow M \times \hat M,
$$
where $\hat M$ denotes another copy of $M$ endowed with the three-form $\hat H = - k \hat{e}^{123}$. Then, define $\overline{B}$ as the pull-back to $\overline{M}$ of a left-invariant two-form 
\begin{equation*}
\overline{B} = - \iota^*(k e^1\wedge \hat e^{1}+ e^4 \wedge \hat e^4).
\end{equation*}
Then, we have
$$
d \overline{B} = - \iota^*(k e^{23}\wedge \hat e^1 - k e^1 \wedge \hat e^{23}) = - \iota^*(k \hat{e}^{123} - k e^{123}) = \hat{p}^{*}\hat{H}- p^{*}H
$$
where we used that $\iota^* de^{1}= \iota^*e^{23}=\iota^* \hat{e}^{23} = \iota^* d\hat{e}^{1}$, since $e^{23}$ and $\hat e^{23}$ are basic.

Consider now the Ricci flat and scalar flat closed pair $(\GG,\divop)$ defined in Example \ref{ex:GRicciflatHopf}. By Propositions \ref{prop:duality} and \ref{prop:dualityScalar} the corresponding T-dual pair $(\hat \GG,\hat \divop)$ is also Ricci flat, and has vanishing scalar curvatures $\mathcal{S}^\pm = 0$. Let us calculate the T-dual pair explicitly. A direct calculation as in the proof of Lemma \ref{lem:buscher} shows that the isomorphism $\psi$ in Theorem \ref{th:Tduality} satisfies
$$
\psi(e_j) = \hat e_j, \qquad \psi(e^j) = \hat e^j, \textrm{ for } j= 2,3.
$$
$$
\psi(e_1) = k \hat e^1, \quad \psi(e_4) = \hat e^4, \quad \psi(e^1) = \frac{1}{k} \hat e_1, \quad \psi(e^4) = \hat e_4.
$$
By definition of $\GG$ we have
$$
V_+ = \langle e_2 + k e^2, e_3 + k e^3, e_1 + k e^1, e_4 + k x^2 e^4  \rangle \subset TM \oplus T^*M,
$$
and therefore the T-dual generalized metric $\hat \GG$ is
$$
\hat V_+ := \psi(V_+) = \langle \hat e_2 +k \hat e^2, \hat e_3 + k \hat e^3, \hat e_1 +k \hat e^1, k x^2 \hat e_4 + \hat e^4  \rangle \subset T\hat M \oplus T^* \hat M.
$$
Notice that
$$
\hat V_+ := \{X + \hat g(X) \;|\; X \in T\hat M\},
$$
where
$$
\hat g = k(\hat e^1 \otimes \hat e^1 + \hat e^2 \otimes \hat e^2 + \hat e^3 \otimes \hat e^3 + (kx)^{-2} e^4 \otimes e^4).
$$
Finally, we calculate the T-dual divergence 
$$
\hat \divop = \divop^{\hat \GG} - \IP{\hat e,\cdot}
$$ 
applying the dilaton shift formula in Proposition \ref{prop:dilatonshift}. Note that
\begin{align*}
\nu = kx e^{14}, \qquad  \hat \nu = x^{-1} \hat e^{14},
\end{align*}
and therefore $\nu$ and $\hat \nu$ are constant on $B = S^2$. By \eqref{eq:dilatonshift} we conclude
$$
\hat e = \psi(\varphi)  = \psi(-x e^4) = - x \hat e_4.
$$
Observe that $(\hat \GG,\hat \divop)$ is a compatible pair, but it is no longer closed. Similarly as in Example \ref{ex:GRicciflatHopf}, one can check now explicitly that the pair $(\hat \GG,\hat \divop)$ is Ricci flat, and has vanishing scalar curvatures $\mathcal{S}^\pm = 0$. We leave this last part as an \textbf{exercise} for the reader.
\end{ex}

We finish this section with two explicit examples of T-dual solutions of the generalized Ricci flow.

\begin{ex} Consider $M \cong S^3$ with 
the Hopf fibration $S^1 \to S^3 \to S^2$, and let $\theta$ denote the connection
one-form on $S^3$ satisfying 
$d \theta = \gw_{S^2}$, where $\gw_{S^2}$ denotes the standard area form on
$S^2$, and furthermore let $H = 0$.  Next let $\hat{M} \cong S^2 \times S^1$,
and consider the trivial
fibration $S^1 \to S^1 \times S^2 \to S^2$.  Let $\hat{\theta}$ denote the
pullback of the canonical line element on $S^1$ to $\hat{M}$, so that $d \hat{\theta} = 0$.  Furthermore let $\hat{H} =
- \hat{\theta} \wedge \gw_{S^2}$, which is closed.
As observed in Example \ref{e:TdualHopf1},
\begin{align*}
p^* H - \hat{p}^* \hat{H} = \hat{p}^* \left( \gw_{S^2} \wedge \hat{\theta}
\right) = d p^*\theta \wedge \hat{\theta} = d \left( p^* \theta \wedge \hat{p}^*
\hat{\theta} \right),
\end{align*}
hence $(M, H, \theta)$ and $(\hat{M}, \hat{H}, \hat{\theta})$ are topologically
T-dual.  Let $g_{S^2}$ denote the round metric on $S^2$ and consider an
$S^1$-invariant metric on $S^3$ of the form
\begin{align*}
g = K \theta \otimes \theta + L g_{S^2}.
\end{align*}
Observe that by applying Proposition \ref{Tdualrules} we obtain that $(g, 0)$ is
T-dual to $(\hat{g}, \hat{b})$ with
\begin{gather} \label{Hopfdual}
\begin{split}
\hat{g} =&\ K^{-1} \hat{\theta} \otimes \hat{\theta} + L g_{S^2},\\
\hat{b} =&\ 0.
\end{split}
\end{gather}
The solution to generalized Ricci flow with initial condition $(g, 0)$ on $S^3$ has $b \equiv 0$ for all time, and takes the form
\begin{align*}
\dot{K} =&\ - \frac{K^2}{L^2}, \qquad \dot{L} = - 2 + \frac{K}{L}.
\end{align*}
Expressing the T-dual data as $\hat{g} = \hat{K} \hat{\theta} \otimes
\hat{\theta} + \hat{L} g_{S^2}$ and using (\ref{Hopfdual}) we obtain the
evolution equation for $\hat{g}$ as
\begin{align*}
\dot{\hat{K}} =&\ \frac{1}{\hat{L}^2}, \qquad \dot{\hat{L}} = -2 + \frac{1}{\hat{K} \hat{L}},
\end{align*}
which one can check is a solution to generalized Ricci flow.  Observe that $S^3$ shrinks to a round point under the
flow, whereas on $S^2 \times S^1$ the $S^2$ shrinks to a point while the $S^1$ fiber
blows up.
\end{ex}

\begin{ex} Let $M \cong S^{2n+1}$ and consider the Hopf
fibration 
$S^1 \to S^{2n+1} \to \mathbb C \mathbb P^n$, and let $\theta$ denote the
connection
one-form on $S^{2n+1}$ satisfying $d \theta = \gw_{FS}$, where $\gw_{FS}$ is the
K\"ahler
form of the Fubini-Study metric on $\mathbb C \mathbb P^n$, and furthermore let
$H = 0$.  Next let $\hat{M} \cong \mathbb C \mathbb P^n \times S^1$, and
consider the trivial
fibration $S^1 \to S^1 \times \mathbb C \mathbb P^n \to \mathbb C \mathbb P^n$. 
Let $\hat{\theta}$ denote the pullback of the canonical line element on $S^1$ to
$\hat{M}$, and let $\hat{H} = - \hat{\theta} \wedge \gw_{S^2}$.  As in the
previous example one easily checks that $(M, H, \theta)$ and $(\hat{M}, \hat{H},
\hat{\theta})$ are topologically T-dual.

Now let $g_0$ denote any metric on $S^{2n+1}$ with positive curvature operator.  Choose $H_0 = 0$, and let $(g_t, b_t)$ denote the solution to generalized Ricci flow with initial condition $(g_0, 0)$.  By Proposition \ref{p:Hzeropres} we know that $H_t = 0$ for all time and $(g_t, b_t) = (g_t, 0)$, where $g_t$ is the unique solution to Ricci flow with initial condition $g_0$.  By the theorem of Bohm-Wilking \cite{Wilking}, we have that $g_t$ exists on some finite time interval
$[0, T)$, and converges to a round point as $t \to T$.  It follows from Theorem \ref{th:dualityGRicci} and 
Proposition \ref{Tdualrules} that the dual solution $(\hat{g}_t, \hat{b}_t)$ to generalized Ricci flow
also exists on a finite time interval, asymptotically converging to a solution which homothetically shrinks the $\mathbb C \mathbb P^n$ base and expanding the $S^1$ fiber, analogously to the previous example.  In fact, any of the Ricci flow `sphere theorems' in odd dimension (for instance \cite{BS}, \cite{Hamilton3folds}) can be applied to invariant metrics satisfying curvature positivity conditions to generate similar examples.
\end{ex}

%\bibliographystyle{amsplain}
%\bibliography{FSGRFBib}

\providecommand{\bysame}{\leavevmode\hbox to3em{\hrulefill}\thinspace}
\providecommand{\MR}{\relax\ifhmode\unskip\space\fi MR }
\providecommand{\MRhref}[2]{%
  \href{http://www.ams.org/mathscinet-getitem?mr=#1}{#2}
}
\providecommand{\href}[2]{#2}

\end{document}